\documentclass{cambridge7A}

\newif\ifsolutions
\newif\ifcolor
\newif\iffinal
\newif\ifarxiv
\newif\ifbook
\newif\ifwork

\arxivtrue

\ifarxiv
	\finaltrue
	\colortrue
	\solutionsfalse
\fi

\ifbook
	\finaltrue
	\colorfalse
	\solutionsfalse
\fi

\ifwork
	\colortrue
	\solutionstrue
	\finalfalse
\fi

\ifcolor
	\newcommand{\emphcolor}{blue}
\else
	\newcommand{\emphcolor}{black}
\fi

\usepackage{makeidx}
\makeindex 

\usepackage{fontawesome}
\usepackage[T1]{fontenc}
\usepackage{xcolor}

\ifcolor
\definecolor{dkblue}{cmyk}{1,.54,.04,.19} 
\definecolor{dkred}{rgb}{0.7,0,0}
\else
\definecolor{dkblue}{rgb}{0,0,0} 
\definecolor{dkred}{rgb}{0,0,0}
\fi

\iffinal 
\usepackage[disable,textsize=tiny]{todonotes}
\else
\usepackage[textsize=tiny]{todonotes} 
\fi

\newcommand{\todot}[2][]{\todo[noinline,size=\scriptsize,color=red!20!white,#1]{Tor: #2}}

\usepackage{pgfplots}
\pgfplotsset{compat=1.11}
\usepackage{pgfplotstable}
\usepackage[normalem]{ulem}
\usepackage{colortbl}
\usepackage{wrapfig}
\usepackage{floatrow}
\usepackage{amsmath}
\usepackage{setspace}
\usepackage{pifont}
\usepackage{mathtools}
\usepackage{listings}

\ifcolor
\else
\fi
\ifcolor
\newcommand*{\Comment}[1]{\hfill\makebox[3.0cm][l]{\texttt{\color{orange}\# #1}}}
\else
\newcommand*{\Comment}[1]{\hfill\makebox[3.0cm][l]{\texttt{\color{black}\# #1}}}
\fi

\usepackage{microtype}
\usepackage{tikz}
\usetikzlibrary{fadings}
\usetikzlibrary{shapes.geometric, positioning, arrows.meta}
\usepackage{dsfont}
\usepackage{amssymb}
\usepackage[plain,chapter]{algorithm}
\usepackage{bm}
\usepackage[labelfont=bf, position=below,skip=0pt,labelsep=quad]{caption}
\usepackage{subcaption}
\usepackage{graphicx}
\usepackage[inline]{enumitem}
\usepackage{placeins}
\usepackage{longtable}
\usepackage{etoolbox}
\usepackage{mdframed}
\usepackage{tablefootnote}
\mdfsetup{skipabove=0pt,skipbelow=0pt,backgroundcolor=black!15!white}
\usepackage{stackrel}

\definecolor{mycolor}{rgb}{0.122, 0.435, 0.698}
\colorlet{grayone}{black!15!white}
\colorlet{graytwo}{black!30!white}
\colorlet{graythree}{black!45!white}

\newcounter{explanacounter}
\newcounter{explanrcounter}
\makeatletter
\pretocmd{\start@align}{\setcounter{explanacounter}{0}}{}{}
\pretocmd{\end@align}{\setcounter{explanacounter}{0}}{}{}
\makeatother

\newcounter{tempctr}

\newcommand{\showalpha}[1]{%
\setcounter{tempctr}{#1}%
\alph{tempctr}%
}

\newcommand{\explana}{\setcounter{explanrcounter}{0}\refstepcounter{explanacounter}\stackrel{\mathclap{\text{\tiny \texttt{(\alph{explanacounter})}}}}}
\newcommand{\explanr}[1]{%
\if\relax\detokenize{#1}\relax%
\refstepcounter{explanrcounter}\texttt{(\alph{explanrcounter})}%
\else%
\texttt{(\showalpha{#1})}%
\fi%
}

\newmdenv[innerlinewidth=0.5pt, roundcorner=4pt,linecolor=mycolor,innerleftmargin=6pt,
innerrightmargin=6pt,innertopmargin=6pt,innerbottommargin=6pt]{mybox}

\usepackage{mathrsfs}
\usepackage{calrsfs}
\usepackage{calrsfs}
\DeclareMathAlphabet{\pazocal}{OMS}{zplm}{m}{n}

\newcounter{rowcntr}[table]
\renewcommand{\therowcntr}{\thetable.\texttt{\alph{rowcntr}}}

\newcolumntype{N}{>{\refstepcounter{rowcntr}\therowcntr}c}

\AtBeginEnvironment{tabular}{\setcounter{rowcntr}{0}}

\setlist[enumerate,1]{label={\normalfont\texttt{(\arabic*)}},itemsep=1pt,wide,labelwidth=!,labelindent=0pt,topsep=1pt}
\setlist[enumerate,2]{label={\normalfont\texttt{(\roman*)}},itemsep=0pt,wide,labelwidth=!,labelindent=0pt,leftmargin=30pt}
\setlist[itemize,1]{label={$\circ$}}

\newlist{itemizeinner}{itemize}{1}
\setlist[itemizeinner]{label={$\circ$},leftmargin=25pt,topsep=1pt}

\newlist{enumeraterom}{enumerate}{1}
\setlist[enumeraterom]{label={\normalfont\texttt{(\roman*)}},
                       align=left,
                       leftmargin=*,
                       labelwidth=0.8em,
                       labelsep=0em}

\newlist{enumeratenotes}{enumerate}{1}
\setlist[enumeratenotes]{ref={\texttt{\thechapter.\roman*}},label={\normalfont\textbf{\uline{\texttt{\thechapter.\roman*}}}},itemsep=5pt,labelsep=10pt,wide,labelwidth=!,labelindent=0pt}

\let\eps\varepsilon

\usepackage{amsthm}
\usepackage{thmtools} 

\newenvironment{algcontents}{
\begin{mdframed}[innerleftmargin=-0.1cm]
}{
\end{mdframed}}

\theoremstyle{plain}
\newtheorem{theorem}{Theorem}

\newtheorem{lemma}[theorem]{Lemma}
\newtheorem{proposition}[theorem]{Proposition}
\newtheorem{corollary}[theorem]{Corollary}

\theoremstyle{definition}
\newtheorem{definition}[theorem]{Definition}
\newtheorem{assumption}[theorem]{Assumption}

\newtheorem{remark}[theorem]{Remark}
\newtheorem{fact}[theorem]{Fact}

\newtheorem{exer}[theorem]{Exercise}
\theoremstyle{remark}

\newcommand{\solution}[1]{
\ifsolutions
\paragraph{Solution}
\noindent{}#1
\hfill \faCheckCircle
\par
\vspace{0.2cm}
\fi
}

\usepackage{newtxtext}
\usepackage{newtxmath}

\usepackage[hyperindex]{hyperref} 

\hypersetup{
  bookmarks=true,         
    unicode=false,          
		hypertexnames=false,
    pdftoolbar=true,        
    pdfmenubar=true,        
    pdffitwindow=false,     
    pdfstartview={FitH},    
    pdftitle={Bandit Convex Optimisation},    
    pdfauthor={Tor Lattimore},
    pdfsubject={Bandits},   
    pdfcreator={pdflatex},   
    pdfproducer={Producer}, 
    pdfkeywords={bandits} {zeroth-order convex optimisation} {machine learning}, 
    pdfnewwindow=true,      
    colorlinks=true,        
    linkcolor=dkblue,       
    citecolor=dkblue,       
    filecolor=dkblue,       
    urlcolor=dkblue,        
    bookmarksdepth=2,
}

\usepackage[capitalize]{cleveref}
\crefname{figure}{Figure}{Figures} 
\crefname{exer}{Exercise}{Exercises}
\crefname{assumption}{Assumption}{Assumptions}
\crefname{equation}{}{} 
\crefname{page}{page}{pages}
\crefname{rowcntr}{Table}{Tables} 
\crefname{fact}{Fact}{Facts}

\numberwithin{theorem}{chapter}
\numberwithin{lemma}{chapter}
\numberwithin{proposition}{chapter}
\numberwithin{corollary}{chapter}
\numberwithin{claim}{chapter}
\numberwithin{remark}{chapter}
\numberwithin{fact}{chapter}
\numberwithin{exer}{chapter}
\numberwithin{assumption}{chapter}
\numberwithin{definition}{chapter}

\usepackage{inconsolata}
\usepackage{natbib}
\usepackage{titlesec}

\renewcommand{\varkappa}{\mathscr K}

\titleformat{\paragraph}
  {\normalfont\normalsize\centering\it}
  {\theparagraph}
  {1em}
  {}
\titlespacing*{\paragraph}
  {0pt}{1.5ex plus 1ex minus .2ex}{0.8ex plus .2ex}

\ifarxiv
\newcommand{\copynotice}{\footnote{This material will be published by Cambridge University Press as Bandit Convex Optimisation by Tor Lattimore . This prepublication version is free to view and download for personal use only. Not for re-distribution, re-sale, or use in derivative works. \copyright\ Tor Lattimore, 2025}}
\else
\newcommand{\copynotice}{}
\fi

\newcommand{\stepsection}[1]{\paragraph{#1}}
\newcommand{\Proofskippy}{Proof\, {\normalfont (\skippy)}}
\newcommand{\ceil}[1]{\left\lceil #1 \right\rceil}
\newcommand{\floor}[1]{\left\lfloor #1 \right\rfloor}
\newcommand{\cl}{\operatorname{cl}}
\newcommand{\pip}{\pi_{\text{\tiny{$\wedge$}}}}

\newcommand{\R}{\mathbb R}
\newcommand{\argmin}{\operatornamewithlimits{arg\,min}}

\newcommand{\explan}[1]{\stackrel{\mathclap{\text{\tiny \texttt{#1}}}}}
\newcommand{\explanw}[1]{\stackrel{\text{\tiny \texttt{#1}}}}

\newcommand{\ip}[1]{\left \langle #1 \right \rangle}
\newcommand{\sip}[1]{\langle #1\rangle}
\newcommand{\bip}[1]{\left\langle #1 \right\rangle}
\newcommand{\ext}{e}
\newcommand{\Reg}{\textrm{\normalfont Reg}}
\newcommand{\sReg}{\textrm{\normalfont sReg}}
\newcommand{\eReg}{\textrm{\normalfont eReg}}
\newcommand{\hqReg}{\widehat{\textrm{\normalfont qReg}}}
\newcommand{\qReg}{\textrm{\normalfont qReg}}
\newcommand{\gReg}{\textrm{\normalfont gReg}}
\newcommand{\hsReg}{\widehat{\textrm{\normalfont sReg}}}
\newcommand{\beReg}{\textrm{\normalfont eR$\overline{\textrm{eg}}$}}
\newcommand{\bReg}{\textrm{\normalfont bReg}}
\newcommand{\hReg}{\widehat{\textrm{\normalfont Reg}}}

\newcommand{\Breg}{\operatorname{D}}

\newcommand{\aff}{\operatorname{aff}}
\newcommand{\rank}{\operatorname{rank}}
\newcommand{\sphere}{\mathbb{S}^{d-1}}
\newcommand{\logs}{L}
\newcommand{\im}{\operatorname{im}}
\newcommand{\ball}{\mathbb{B}^d}
\newcommand{\laspan}{\operatorname{span}}
\newcommand{\norm}[1]{\left \Vert  #1 \right \Vert}
\newcommand{\snorm}[1]{\Vert  #1 \Vert}

\newcommand{\E}{\mathbb E}
\newcommand{\skippy}{\includegraphics[height=0.65em]{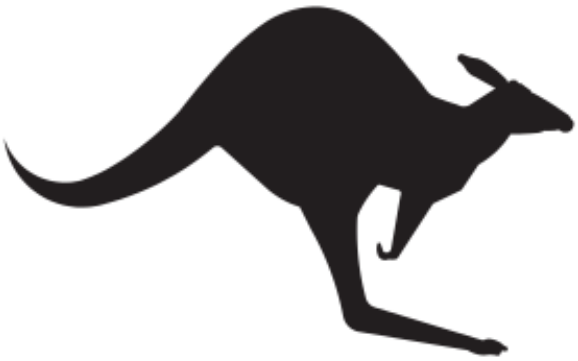}}
\newcommand{\ri}{\operatorname{ri}}

\robustify{\skippy}

\newcommand{\cE}{\mathcal E}

\newcommand{\KL}{\operatorname{KL}}

\newcommand{\cF}{\mathscr F}

\newcommand{\cG}{\mathscr G}
\newcommand{\cC}{\mathcal C}
\newcommand{\cCS}{\mathcal C_{\mathbb S}}

\newcommand{\sG}{\mathscr G}
\newcommand{\sF}{\mathscr F}

\newcommand{\sP}{\mathscr P}
\newcommand{\sU}{\mathscr U}
\newcommand{\sS}{\mathscr S}

\newcommand{\sA}{\mathscr A}

\newcommand{\cU}{\mathcal U}
\newcommand{\cS}{\mathscr S}

\newcommand{\cnf}{\operatorname{cnf}}
\newcommand{\ISO}{\operatorname{\normalfont\textsc{iso}}}
\newcommand{\JOHN}{\operatorname{\normalfont\textsc{john}}}

\newcommand{\MVEE}{\operatorname{\normalfont\textsc{mvee}}}

\newcommand{\GRAD}{\operatorname{\normalfont\textsc{grad}}}
\newcommand{\CVX}{\operatorname{\normalfont\textsc{cvx}}}
\newcommand{\PROJ}{\operatorname{\normalfont\textsc{proj}}}

\newcommand{\SEP}{\operatorname{\normalfont\textsc{sep}}}
\newcommand{\LIN}{\operatorname{\normalfont\textsc{lin}}}
\newcommand{\SAMP}{\operatorname{\normalfont\textsc{samp}}}
\newcommand{\MEM}{\operatorname{\normalfont\textsc{mem}}}
\newcommand{\MVIE}{\operatorname{\normalfont\textsc{mvie}}}
\newcommand{\CUT}{\hyperref[alg:cut:cut]{\operatorname{\normalfont\textsc{cut}}}}
\newcommand{\BISECT}{\hyperref[alg:bisection]{\operatorname{\normalfont\textsc{bisect}}}}
\newcommand{\ESTIMATEG}{\hyperref[alg:cut:estimate-g]{\operatorname{\normalfont\textsc{estimate-gradient}}}}
\newcommand{\BAI}{\hyperref[alg:cut:bai]{\operatorname{\normalfont\textsc{bai}}}}
\newcommand{\SAMPLE}{\hyperref[alg:cut:oracle]{\ensuremath{\varkappa}}}
\newcommand{\INF}{\hyperref[alg:cut:inf]{\operatorname{\normalfont\textsc{inf}}}}
\newcommand{\psd}{\mathbb S^d_{+}}
\newcommand{\pd}{\mathbb S^d_{++}}

\newcommand{\cN}{\mathcal N}
\newcommand{\sB}{\mathscr B}

\newcommand{\tr}{\operatorname{tr}}

\newcommand{\zeros}{ \bm 0}
\newcommand{\ones}{ \bm 1}
\newcommand{\bbP}{\mathbb P}

\newcommand{\plin}{\tiny{\texttt{lin}}}
\newcommand{\pquad}{\tiny{\texttt{quad}}}
\newcommand{\pb}{{\tiny{\texttt{b}}}}
\newcommand{\pl}{{\tiny{\texttt{l}}}}
\newcommand{\psm}{{\tiny{\texttt{sm}}}}
\newcommand{\psc}{\tiny{\texttt{sc}}}
\newcommand{\pu}{\tiny{\texttt{i}}}
\newcommand{\ps}{{\tiny{\texttt{s}}}}

\newcommand{\lip}{\operatorname{lip}}
\newcommand{\vol}{\operatorname{vol}}
\newcommand{\polylog}{\operatorname{polylog}}

\newcommand{\interior}{\operatorname{int}}

\newcommand{\conv}{\operatorname{conv}}
\newcommand{\diam}{\operatorname{diam}}
\newcommand{\poly}{\operatorname{poly}}
\newcommand{\sind}{\bm{1}}
\renewcommand{\d}[1]{\operatorname{d}\!#1}
\newcommand{\id}{\mathds{1}}

\newcommand{\Ymax}{Y_{\max}}

\newcommand{\dom}{\operatorname{dom}}

\newcommand{\eventNoise}{{\normalfont\textsc{e1}}}
\newcommand{\eventPi}{{\normalfont \textsc{e2}}}
\newcommand{\eventAzuma}{{\normalfont\textsc{e0}}}
\newcommand{\eventGauss}{{\normalfont\textsc{e3}}}
\newcommand{\eventY}{{\normalfont\textsc{e4}}}
\newcommand{\eventQ}{{\normalfont\textsc{e5}}}
\newcommand{\eventS}{{\normalfont\textsc{e6}}}

\renewcommand{\emptyset}{\varnothing}

\newcommand{\idxdef}[1]{\index{#1|textbf}}
\usepackage{comment}
\newcommand{\commentAlt}[1]{\ignorespaces}
\newcommand{\commentLongAlt}[1]{\ignorespaces}

\begin{document}

\ifbook

\else
	\title{Bandit Convex Optimisation}
	\author{\large Tor Lattimore \\[1cm] November 2025}
\fi

\bookkeywords{Bandit convex optimisation; zeroth-order optimisation}
\frontmatter

\definecolor{bgcolor}{rgb}{0.0941, 0.2353, 0.4118} 

\ifbook

\begin{tikzpicture}[overlay,remember picture]
\node at (current page.center) {
\includegraphics[width=21cm]{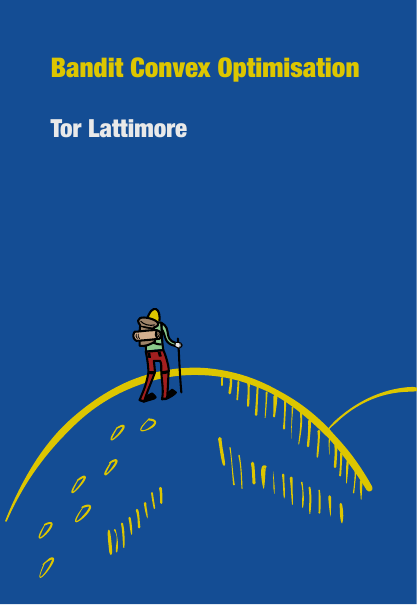}
};
\end{tikzpicture}

\fi

\maketitle
\tableofcontents

\chapter*{Preface}

This book is about zeroth-order convex optimisation; that is, approximately solving 
\begin{align*}
\argmin_{x \in K} f(x)\,,
\end{align*}
where $f \colon K \to \R$ is convex, $K \subset \R^d$ is a convex body and the learning system can only observe noisy function values
and not gradients or higher-order derivatives.
The focus is on finite-time minimax bounds on the regret or sample complexity that is standard in the multi-armed bandit literature.
The book covers all the algorithmic ideas, including gradient descent, barrier methods, cutting plane methods, exponential weights, 
information-theoretic arguments and second-order methods. 
The penultimate chapter is devoted to bandit submodular minimisation and its relation to bandit convex optimisation via the Lov\'asz extension.
The book is more or less self-contained, though a little background in optimisation and online learning will go a long way.
The content is almost entirely theoretical.   

\paragraph{Acknowledgements}
This work would not have been possible without my many wonderful collaborators, especially 
Alireza Bakhtiari, Hidde Fokkema, Andr\'as Gy\"orgy, Dirk van der Hoeven, Jack Mayo
and Csaba Szepesv\'ari.
For at least the second time I am in debt to Marcus Hutter, who read almost the entire book and made a huge number of thoughtful suggestions.
Finally, thank you Rosina for your endless love and support; and Phoebe for your love and smiles.

\mainmatter

\chapter[Introduction and Problem Statement]{Introduction and Problem Statement\copynotice}\label{chap:intro}

This book is about approximately solving problems of the form
\begin{align}
\argmin_{x \in K} f(x)\,,
\label{eq:int:opt}
\end{align}
where $K$ is a convex body in $\R^d$ (convex, compact and nonempty interior) and $f \colon K \to \R$ is a convex function.
Problems of this kind are ubiquitous in machine learning, operations research, economics and beyond.
The difficulty of this problem depends on many factors; for example, the dimension and smoothness properties of $f$.
Most important, however, is the representation of the function $f$ and constraint set $K$.\index{constraint set} 
Our focus is on zeroth-order stochastic optimisation, where you can query $f$ at any $x \in K$ and observe $y = f(x) + \text{noise}$.
By contrast, the vast majority of the literature on mathematical programming allows both the value of $f$ and its derivatives to be computed at any $x \in K$, either exactly
or with additive noise.
This is not a book about applications, but for the sake of inspiration and motivation we list a few situations where zeroth-order optimisation is a natural fit.
\begin{itemize}
\item \textit{Real-world experiments:} A chef wants to optimise the temperature and baking time when baking a souffl\'e. The constraint set $K$ is some reasonable subset of
the possible temperature/time pairings and $f(x)$ is the expected negative quality of the finished product. Noise arises here from exogenous factors, such as unintentional variation
in recipe preparation.
\item \textit{Adversarial attacks in machine learning:} A company releases an image recognition system. Can you find an image that looks to the human eye like a stop sign but is classified by
the system as something else? Unless the company has released the code and weights, you can only interact with a black-box function $C$ that accepts images as input and 
returns a classification, possibly at some cost.
A simple idea is to take an image $x_\circ$ and let $K$ be a set of images that are visually indistinguishable from $x_\circ$. Then approximately solve the following optimisation
problem:
\begin{align*}
\argmin_{x \in K} S(C(x), C(x_\circ))\,,
\end{align*}
where $S(C(x), C(y))$ is some measure of similarity between the classifications $C(x)$ and $C(y)$. That is, the problem is to find an image $x$ that looks similar to 
$x_\circ$ but with $C(x)$ maximally different from $C(x_{\circ})$. 
\item \textit{Reinforcement learning and control:} There are many ways to do reinforcement learning. Suppose that $\pi_\theta$ is a policy parameterised by $\theta \in K$ and
$f(\theta)$ is the (expected) loss when implementing policy $\pi_\theta$. Then solving \cref{eq:int:opt} corresponds to finding the optimal policy in $\{\pi_\theta \colon \theta \in K\}$.
\item \textit{Hyperparameter tuning:} Most machine learning systems have a range of parameters; for example, the width, depth and structure of a neural network or learning rates (schedules)
for the training algorithm. You may want to automate the process of finding the best architecture and training parameters. The loss in this case could be the performance of the resulting system
after training. So computing $f(x)$ requires running an entire training process, which is enormously expensive and possibly random. Generally speaking $f$ cannot be differentiated.
\item 
\textit{Dynamic pricing:} In dynamic pricing a retailer interacts sequentially with an environment. Customers arrive in the system and the retailer suggests a price $X_t \in K \subset \R$.
The loss $f(X_t)$ is the (expected) negative profit. When the price is too high the customer will not purchase and some loss is incurred (operating costs). On the other hand,
when the price is too low there is also a loss. The function $f$ is not known \textit{a priori} to the retailer and variability in customers introduces noise in the observations.  
\end{itemize}

\begin{wrapfigure}{r}{5cm}
\includegraphics[width=5cm]{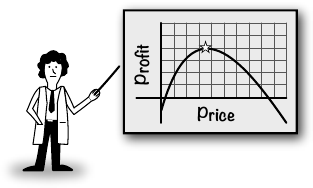}
\end{wrapfigure}
\noindent The last application reveals another consideration.
Approximately minimising the loss $f$ as in \cref{eq:int:opt} is not the only possible objective. In dynamic pricing every query to the loss function (decision) entails 
an actual cost (or profit). In this situation a more natural objective is to minimise the cumulative loss over all interactions. Most of this book is about this objective,
which is most often seen in the literature on multi-armed bandits.
As we shall soon see, an algorithm that approximately minimises the cumulative loss can be used to solve the optimisation problem \cref{eq:int:opt}, but the reverse does not hold.
Note that many practical problems are not convex. Nevertheless, enormous practical experience with first-order methods suggests this is less of a problem 
than it seems.
The same may be true in the zeroth-order setting.

Many of the proposed applications in this book are aspirational rather than truly practical. 
Automatic hyperparameter tuning is used everywhere, but the algorithms employed are usually based on
other methods, such as Gaussian process bandits.\index{bandit!Gaussian process} Presumably the reason for this is because convex bandit algorithms are not quite ready for the big stage yet.
The simple methods converge slowly, while the complicated methods are challenging to implement and tune. Hopefully in the coming years new and better algorithms will be developed.

\section{Methods and Challenges}
In spite of its simplicity, bandit (and zeroth-order) convex optimisation is still not fully understood;
and certainly much less so than first-order optimisation where the gradients of $f$ are available.
Nearly every method for bandit convex optimisation is based on estimating the gradient (and maybe higher-order derivatives) of
a surrogate loss and combining this with a classical gradient-based optimisation method.
This approach leads to a pernicious bias/variance trade-off, which is particularly exacerbated 
when the loss function is non-smooth.
Unsurprisingly, then, the primary focus of the classical literature on stochastic zeroth-order optimisation is on smooth settings.

This book explores the full range of methods, including the most classical online projected gradient descent (\cref{chap:sgd}), 
cutting plane methods (\cref{chap:ellipsoid}) and more
exotic methods based on information theory and minimax duality (\cref{chap:exp}).
Buried somewhere in nearly every approach is some kind of bias/variance trade-off.
Optimising for the regret also introduces complications. The learner cannot freely query the loss at any point they might
like without consideration for the loss suffered. This is known as the exploration/exploitation dilemma for which
reinforcement learning is famous. We will see, in fact, that in certain settings optimising for the regret is (provably) harder
than merely finding a near-optimiser of the loss. 
Another theme that raises its head in many chapters is the need to balance statistical and computational efficiency, both of the 
theoretical kind (poly time or not) and the practical kind. At present, those algorithms with the best statistical guarantees
do not have known efficient implementations.

So, let us begin. The next section explains what prerequisites you may find useful. After that, the rest of the chapter is devoted
to the formal problem settings studied later.
\cref{chap:overview} covers the history and includes large reference tables of known results.
\cref{chap:regularity} covers the necessary mathematical tools. Nearly every chapter thereafter is based on what can be achieved
with some algorithmic idea.

\section{Prerequisites}
Most readers will benefit from a reasonable knowledge of online learning \citep{Ces06,Haz16,Ora19} and
bandits \citep{BC12,Sli18,LS20book}.
We use some theory from interior point methods,\index{interior point methods} which you could refresh by reading the lecture notes by \cite{nemirovski96}.  
None of this is essential, however, if you are prepared to take a few results on faith.
Similarly, we use a few simple results from concentration of measure. Our reference was the book by \cite{Ver18} but \cite{BLM13} also
covers the material needed. We do use martingale\index{martingale} versions of these results, which sadly do not appear in these books but are more or less trivial extensions.
The standard reference for convex analysis is the book by \cite{Roc70} to which we occasionally refer. No deep results are needed, however.
Similarly, we make use of certain elementary results in convex geometry. Our reference for these is mostly the book by \cite{ASG15}.
The symbol ($\skippy$) on a proof, section or chapter means that you could (should?) skip this part on your first pass.
The book contains a number of exercises, which are assigned a difficulty based on a star rating.
Problems with one \faStar\ are straightforward, requiring limited ingenuity or mathematical sophistication.
Problems with two \faStar\faStar\ are moderately difficult and may require considerable ingenuity.
Problems with three \faStar\faStar\faStar\ will probably take several days of work.
Problems marked with \faBook\ may require extensive literature searches and problems marked with \faQuestion\ have not been solved yet.
All difficulty assessments are subjective estimates only.

\section{Bandit Convex Optimisation}
Let $K \subset \R^d$ be convex and $f_1,\ldots,f_n \colon K \to \R$ be an unknown sequence of convex functions. 
A learner interacts with the environment over $n$ rounds. In round $t$ the learner chooses an action $X_t \in K$. \label{page:iterates} They 
then observe a noisy loss $Y_t = f_t(X_t) + \eps_t$, where $(\eps_t)_{t=1}^n$ is a sequence of noise random variables. \label{page:responses} 
The precise conditions on the noise are given in \cref{eq:noise} below, but for now you could think of the noise as a sequence of independent standard Gaussian random 
variables.
The learner's decision $X_t$ is allowed to depend on an exogenous source of randomness and the data observed already, which 
is $X_1, Y_1,\ldots,X_{t-1},Y_{t-1}$. 
The main performance metric in this book is the regret,\index{regret|textbf} which is 
\begin{align*}
\Reg_n = \sup_{x \in K} \sum_{t=1}^n (f_t(X_t) - f_t(x))\,.
\end{align*}
The regret is a random variable with the randomness coming from both the noise and the learner's decisions.
Normally the regret will be bounded in expectation or with high probability, depending on what is most convenient.
Of course, the regret also depends on the loss functions. In general we will argue that our algorithms have small regret for any convex losses within some class.
Stronger assumptions (smaller classes) lead to stronger results and/or simpler and/or more efficient algorithms.
The following definition and notation are sufficient for our purposes.

\begin{definition}\label{def:class}
Let $\cF$ be the space of convex functions from $K$ to $\R$, and
with $\norm{\cdot}$ the standard euclidean norm define the following properties of a function $f \in \cF$:
\begin{enumerate}
\item[\texttt{Prop (b)}] $f$ is bounded: $f(x) \in [0,1]$ for all $x \in K$.
\item[\texttt{Prop (l)}] $f$ is Lipschitz: $f(x) - f(y) \leq \norm{x - y}$ for all $x, y \in K$. \index{Lipschitz|textbf}
\item[\texttt{Prop (sm)}] $f$ is $\beta$-smooth: $f(x) - \frac{\beta}{2} \norm{x}^2$ is concave on $K$. \index{smooth} 
\item[\texttt{Prop (sc)}] $f$ is $\alpha$-strongly convex: $f(x) - \frac{\alpha}{2} \norm{x}^2$ is convex on $K$. \index{strongly convex} \label{page:smooth}
\item[\texttt{Prop (lin)}] $f$ is linear: $f(x) = x^\top b + c$.
\item[\texttt{Prop (quad)}] $f$ is quadratic: $f(x) = x^\top A x + x^\top b + c$.
\end{enumerate}
\end{definition}

We use the property symbols to define subsets of $\cF$. For example: 
\begin{itemize}
\item $\cF_\pb = \{f \in \cF \colon f(x)  \in[0,1] \text{ for all } x \in K\}$.
\item $\cF_{\pb,\psm,\psc}$ is the set of bounded convex functions that are smooth and strongly convex. 
\end{itemize}
When smoothness and strong convexity are involved, our bounds will depend on the parameters $\alpha > 0$ and $0 \leq \beta < \infty$, which we assume are known constants.

\subsection{Constraint Set}
The set $K$ is called the constraint set and its geometry also plays a role in the hardness of bandit convex optimisation.
We make the following assumption throughout the entire book:

\begin{assumption}
The constraint set $K \subset \R^d$ is a convex body, which means that: \index{constraint set}\index{convex body|textbf}
\begin{enumerate}
\item $K$ is convex; 
\item $K$ has a nonempty interior; and
\item $K$ is compact.
\end{enumerate}
\end{assumption}

Only convexity and the boundedness part of compactness are really essential for most results. 
Every nonempty convex set has a nonempty interior when embedded in a suitable affine subset\index{affine!subset} of $\R^d$.
Properties of $K$ that influence the regret of various algorithms include its diameter and how well-rounded it is (\cref{sec:reg:rounding}).
Mathematically speaking, nothing more is needed from the constraint set. Computationally the representation of $K$ is very important.
Standard options are as a polytope,\index{polytope} the convex hull\index{convex hull} of a point cloud or as given by a separation/membership oracle. Matters involving computation and the constraint set
are discussed in \cref{sec:reg:compute}.

\begin{remark}\label{rem:unconstrained}
You may be interested to know what happens when $K$ is not bounded; for example, when $K = \R^d$. We discuss this and other possible assumptions on the constraint set 
in Note~\ref{note:K} at the end of the chapter.\index{setting!unconstrained}
\end{remark}

\subsection{Noise}
Our assumption on the noise is that the sequence $(\eps_t)_{t=1}^n$ is conditionally subgaussian.\index{subgaussian} By this we mean: \label{page:noise}\index{noise|textbf}

\begin{assumption}\label{ass:noise} 
The noise random variables $\eps_1,\ldots,\eps_n$ are conditionally subgaussian:
\begin{align}
\E\left[\eps_t \mid X_1,Y_1,\ldots,X_{t-1},Y_{t-1},X_t\right] &= 0\,; \text{ and} \nonumber \\
\E\left[\exp(\eps_t^2) \mid X_1,Y_1,\ldots,X_{t-1},Y_{t-1},X_t\right] &\leq 2\,.
\label{eq:noise}
\end{align}
\end{assumption}

\cref{ass:noise} is considered global and will not be referred to subsequently.
Note that \cref{eq:noise} together with the fact that $x^2 \leq \exp(x^2) - 1$ for all $x$ guarantees that the conditional variance of $\eps_t$ satisfies \index{variance}
\begin{align}
\E[\eps_t^2|X_1,Y_1,\ldots,X_{t-1},Y_{t-1},X_t] \leq 1\,.
\label{eq:noise-var}
\end{align}
This definition of subgaussianity is based on the Orlicz norm\index{Orlicz norm} definitions. We give a brief summary in Appendix~\ref{app:conc}, or you can read the wonderful book
by \cite{Ver18}.
Sometimes we work in the noise-free setting where $\eps_t = 0$.
The assumption of subgaussian noise is rather standard in the bandit literature. The Gaussian distribution $\cN(0, 3/8)$ is subgaussian, and so is any distribution
that is suitably bounded almost surely. All that is really needed is a suitable concentration\index{concentration} of measure phenomenon, 
and hence all results in this book could be generalised
to considerably larger classes of noise distribution without too much effort. We keep things simple, however.

\subsection{Adversarial and Stochastic Convex Bandits}\label{sec:settings}

We have already outlined some of the assumptions on the function class to which the losses belong. 
The other major classification is whether or not the problem is adversarial or stochastic.

\subsubsection*{Adversarial Bandit Convex Optimisation}
In the adversarial setting\index{setting!adversarial} the most common assumption is that the noise $\eps_t = 0$ while the functions $f_1,\ldots,f_n$ are chosen
in an arbitrary way by the adversary.
Sometimes the adversary is allowed to choose $f_t$ at the same time as the learner chooses $X_t$, in which case we say the adversary is non-oblivious.
Perhaps more commonly, however, the adversary is obliged to choose all loss functions $f_1,\ldots,f_n$ before the interaction starts. Adversaries of this kind
are called oblivious. For our purposes it is convenient to allow nonzero noise even in the adversarial case. This is sometimes essential in applications; for example,
in bandit submodular minimisation (\cref{chap:submod}).\index{bandit!submodular minimisation} And besides, it feels natural that the adversarial setting should generalise the stochastic one.

\begin{remark}
In the adversarial setting you might be tempted to combine the noise and losses by defining the loss function to be $f_t + \eps_t$.
But this would prevent the noise distribution from depending on the action of the learner, which is permitted by \cref{ass:noise}
and is essential in certain applications, such as submodular bandits (\cref{chap:submod}).
\index{noise}
\end{remark}

\subsubsection*{Stochastic Bandit Convex Optimisation}
The stochastic setting\index{setting!stochastic} is more classical. The loss function is now constant over time: $f_t = f$ for all rounds $t$ and unknown $f$. 
The standard performance metric in bandit problems is the regret, but in the stochastic setting it also makes sense to consider the simple regret.
At the end of the interaction the learner is expected to output one last point $\widehat X \in K$ and the simple regret is
\begin{align*}
\sReg_n = f(\widehat X) - \inf_{x \in K} f(x) \,.
\end{align*}
Thanks to convexity, there is a straightforward reduction from cumulative regret to simple regret: \index{regret!simple}
simply let $\widehat X = \frac{1}{n} \sum_{t=1}^n X_t$. Then by convexity,
\begin{align}
\sReg_n \leq \frac{1}{n} \Reg_n\,.
\label{eq:simple}
\end{align}
Another standard measure of performance in the stochastic setting is the sample complexity, which is the number of \idxdef{sample complexity} 
interactions needed before the simple regret is at most $\eps > 0$ with high probability.
The following fact provides a conversion from a high probability bound on the cumulative regret to sample complexity:

\begin{fact}\label{fact:conversion}
Suppose that with probability at least $1 - \delta$ the regret of some algorithm is bounded by $\Reg_n \leq R(n)$. Then, by the above conversion, the sample complexity can be bounded by 
the smallest $n$ such that $R(n)/n \leq \eps$. Concretely, if $R(n) = A n^{p}$ for $p \in (0,1)$ and $A > 0$, then the sample complexity is at most
$1 + (A/\eps)^{1/(1-p)}$.
\end{fact}

There is also a simple reduction from a bound on the expected regret to a (high probability) sample complexity bound.
The idea is to run the regret-minimising algorithm $k_{\max} = O(\log(1/\delta))$ times with horizon $n$ large enough that $\E[\frac{1}{n}\Reg_n] \leq \frac{\eps}{4}$.
By Markov's inequality and \cref{eq:simple}, with constant probability the average iterate over any run 
is nearly optimal, which means that one of the $k_{\max}$ average iterates is nearly optimal with high probability.
Lastly, a best arm identification procedure is applied to select a near minimiser among the $k_{\max}$ candidates. \index{best arm identification}
The full procedure and its analysis follow.

\begin{algorithm}[h!]
\begin{algcontents}
\begin{lstlisting}
args: base algorithm $\textsc{alg}$, $R(\cdot)$, $\eps > 0$, $\delta \in (0,1)$
let $n = \min\{s \colon R(s)/s \leq \eps/4\}$
for $k = 1$ to $k_{\max} \triangleq \ceil{\frac{\log(2/\delta)}{\log(2)}}$:
  run $\textsc{alg}$ over $n$ rounds
  observe iterates $X_1,\ldots,X_n$
  let $\widehat x_k = \frac{1}{n} \sum_{s=1}^n X_s$
return $\BAI(\eps/2,\delta/2, \widehat x_1,\ldots,\widehat x_{k_{\max}})$ &\Comment{\cref{alg:cut:bai}}&
\end{lstlisting}
\caption{A master algorithm for obtaining high probability sample complexity bounds from a base algorithm with expected regret.
See \cref{sec:cut:bai} for a detailed explanation of the $\BAI$ subroutine (\cref{alg:cut:bai}).
}
\label{alg:reduction}
\end{algcontents}
\end{algorithm}

\FloatBarrier

\begin{proposition}\label{prop:conversion}
Suppose that $\E[\Reg_n] \leq R(n)$. Then
\begin{enumerate}
\item \cref{alg:reduction} queries the loss at most 
\begin{align*}
O\left(\left(\min\left\{n \colon \frac{R(n)}{n} \leq \frac{\eps}{4}\right\} + \frac{\log(1/\delta)}{\eps^2}\right) \log(1/\delta)\right)
\end{align*}
times; and
\item With probability at least $1 - \delta$, \cref{alg:reduction} returns an $\widehat X \in K$ such that 
\begin{align*}
f(\widehat X) \leq \inf_{x \in K} f(x) + \eps \,.
\end{align*}
\end{enumerate}
\end{proposition}

\begin{proof}
Let $f_\star = \inf_{x \in K} f(x)$.
By \cref{eq:simple}, $\E[f(\widehat x_k)] - f_\star \leq \frac{\eps}{4}$. 
Since $f(x) - f_\star \geq 0$ for all $x \in K$, by Markov's inequality,
\begin{align*}
\bbP\left(f(\widehat x_k) - f_\star \geq \frac{\eps}{2}\right) \leq \frac{1}{2} \,.
\end{align*}
Hence, by the definition of $k_{\max}$, with probability at least $1 - \delta/2$ there exists a $k \in \{1,\ldots,k_{\max}\}$ such that
\begin{align*}
f(\widehat x_k) - f_\star \leq \frac{\eps}{2} \,.
\end{align*}
Finally, referring to \cref{sec:cut:bai}, by \cref{thm:cut:bai} the call to \cref{alg:cut:bai} uses 
\begin{align*}
O\left(\frac{k_{\max}}{\eps^2} \log\left(\frac{k_{\max}}{\delta}\right)\right)
\end{align*}
queries to the loss function and with probability at least $1 - \delta/2$ returns an $\widehat X \in \{\widehat x_1,\ldots,\widehat x_{k_{\max}}\}$ with
\begin{align*}
f(\widehat X) \leq \min_{1 \leq k \leq k_{\max}} f(\widehat x_k) + \frac{\eps}{2} \,.
\end{align*}
Combining everything with a union bound completes the proof.
\end{proof}

Our focus for the remainder of the book is primarily on the cumulative regret, but we occasionally highlight the sample complexity of algorithms in order to compare to the literature.
The arguments above show that bounds on the cumulative regret imply bounds on the simple regret and sample complexity. The converse is not true.

\subsubsection*{Regret is Random}
You should note that $\Reg_n$ and $\sReg_n$ are random variables with the randomness arising from both the algorithm and the noise.
Most of our results either control $\E[\Reg_n]$ or prove that $\Reg_n$ is bounded by such-and-such with high probability.
Bounds that hold with high probability are generally preferred since they can be integrated to obtain bounds in expectation.
But we will not be too dogmatic about this. Indeed, we mostly prove bounds in expectation to avoid tedious concentration\index{concentration} of measure calculations. As far as we know,
these always work out if you try hard enough.

\section{Notation}

A full list of notation is available in Appendix~\ref{app:notation}.

\subsubsection*{Norms}
The norm $\norm{\cdot}$ is the euclidean norm for vectors and the spectral norm for matrices.
For positive definite $A$, $\ip{x, y}_A = x^\top A y$ and $\norm{x}^2_A = \ip{x, x}_A$.
Given a random variable $X$, $\norm{X}_{\psi_k} = \inf\{t > 0 \colon \E[\exp(|X/t|^k)] \leq 2\}$ for $k \in \{1,2\}$ are the Orlicz norms.
Remember, $X$ is subgaussian if $\norm{X}_{\psi_2} < \infty$ and subexponential if $\norm{X}_{\psi_1} < \infty$.\index{subgaussian}\index{subexponential}
You can read more about the Orlicz norms in Appendix~\ref{app:conc}.

\subsubsection*{Sets}
$\R$ and $\mathbb Z$ are the sets of real values and integers.
The restriction of the reals to the (strictly) positive real line are $\R_+ = [0, \infty)$ and $\R_{++} = (0, \infty)$.
The euclidean ball and sphere of radius $r$ are denoted by $\ball_r = \{x \in \R^d \colon \norm{x} \leq r\}$ and $\sphere_r = \{x \in \R^d \colon \norm{x} = r\}$.
Hopefully the latter will not be confused with the space of positive (semi-)definite matrices on $\R^d$, which we denote by ($\psd$) $\pd$. 
Given $x \in \R^d$ and $\zeros \neq \eta \in \R^d$ we let $H(x, \eta) = \{y \colon \ip{y - x, \eta} \leq 0\}$, which is a closed half-space passing through $x$ with outwards-facing normal $\eta$.
\idxdef{half-space}
Given $x \in \R^d$ and positive definite matrix $A$, $E(x, A) = \{y \in \R^d \colon \norm{x - y}_{A^{-1}} \leq 1\}$, which is an ellipsoid centred
at $x$. \idxdef{ellipsoid}
When $E = E(x, A)$ is an ellipsoid, $E(r) = E(x, \sqrt{r} A)$ denotes the same ellipsoid scaled by a factor of $r$. 
The space of probability measures on $K \subset \R^d$ is $\Delta(K)$ where we always take the Borel $\sigma$-algebra $\sB(K)$.
Given a natural number $m$, $\Delta_m = \{p \in \R^m_+ \colon \norm{p}_1 = 1\}$ and $\Delta_m^+ = \Delta_m \cap \R_{++}^m$.
For $x, y \in \R^d$ we let $[x,y] = \{(1 - \lambda) x + \lambda y \colon \lambda \in [0,1]\}$, which is the chord connecting $x$ and $y$.

\subsubsection*{Elementary Notation}
We define arithmetic operations on sets in the Minkowski fashion. Concretely,
given $A, B \subset \R^d$ the Minkowski sum is $A + B = \{x + y \colon x \in A, y \in B\}$. \idxdef{Minkowski sum}
When $u \in \R$ we let $u A = \{u x \colon x \in A \}$ and of course $-A = (-1)A$ and $A - A = A + (-A)$. Occasionally for $x \in \R^d$ we abbreviate $\{x\} + A = x + A$.
The closure of $A$ is $\cl(A)$, its interior is $\interior(A) = \{x \in A \colon \exists \epsilon > 0, x + \ball_\eps \subset A\}$
and its boundary is $\partial A = \cl(A) \setminus \interior(A)$.\label{page:boundary}\label{page:interior}
When $A$ is convex, the relative interior of $A$ is $\ri(A) = \{x \in A \colon \exists \eps > 0, (x + \ball_\eps) \cap \aff(A) \subset A\}$
where $\aff(A)$ is the affine hull of $A$. \label{page:ri}
The polar \idxdef{polar} of $A$ is $A^\circ = \{u \colon \sup_{x \in A} \ip{x, u} \leq 1\}$.
We use $\id$ for the identity matrix and $\zeros$ for the zero matrix or zero vector.
Dimensions and types will always be self-evident from the context.
The euclidean projection onto $K$ is $\Pi_K(x) = \argmin_{y \in K} \norm{x - y}$.\index{projection}
Suppose that $f \colon \R^d \supset A \to \R$ is differentiable at $x \in A$; then we write $f'(x)$ for its gradient and $f''(x)$ for its Hessian.\index{differentiable}
When $f$ is convex we write $\partial f(x)$ for the set of subderivatives of $f$ at $x$.
More generally, $Df(x)[h]$ is the directional derivative of $f$ at $x$ in the direction $h$. Higher-order directional derivatives are 
denoted by $D^kf(x)[h_1,\ldots,h_k]$.
Note that for convex $f$, $Df(x)[h]$ is defined for all $x \in \interior(\dom(f))$ and $h \in \R^d$ but the mapping $h \mapsto Df(x)[h]$ need not be linear, 
as the convex function $\R \colon x \mapsto |x|$ shows.
Densities are always with respect to the Lebesgue measure.
The diameter of a nonempty set $K$ is \index{diameter}
\begin{align*}
\diam(K) &= \sup_{x, y \in K} \norm{x - y} \,. \label{page:diam} 
\end{align*}
The Lipschitz constant of a function $f \colon K \to \R$ is 
\begin{align*}
\lip_K(f) = \sup\left\{\frac{f(x) - f(y)}{\norm{x - y}} \colon x, y \in K, x \neq y\right\} \,.
\end{align*}
When $f \colon \R^d \to \R \cup \{\infty\}$ is convex we write $\lip(f)$ to mean $\lip_{\dom(f)}(f)$ where $\dom(f) = \{x \colon f(x) < \infty\}$.\index{domain|textbf}
Suppose that $A, B \in \psd$. Then $A \preceq B$ if $B - A \in \psd$ and $A \succeq B$ if $A - B \in \psd$.
$A \prec B$ and $A \succ B$ are defined similarly but with $\psd$ replaced by $\pd$.

\subsubsection*{Probability Spaces}
We will not formally define the probability space on which the essential random variables $X_1,Y_1,\ldots,X_n,Y_n$ live. You can see
how this should be done in the book by \cite{LS20book}. In general $\bbP$ is the probability measure on some space carrying these random variables
and we let $\sF_t = \sigma(X_1,Y_1,\ldots,X_t,Y_t)$ be the $\sigma$-algebra generated by the first $t$ rounds of interaction. \label{page:filtration}\index{filtration}
Events are often defined as $\{\text{condition}\}$. For example, $\{f(X_t) > \eps\}$ is the event that $f(X_t) > \eps$.
We abbreviate $\bbP_t(\cdot) = \bbP(\cdot | \sF_t)$ and $\E_t[\cdot] = \E[\cdot|\sF_t]$.\label{page:cond-measure}\label{page:cond-expect}%
The multivariate Gaussian distribution with mean $\mu$ and covariance $\Sigma$ is $\cN(\mu, \Sigma)$.
Given $A \subset \R^d$ we let $\cU(A)$ be the uniform probability measure on $A$,\index{uniform measure} which can be defined in multiple ways.
We only need it when $A$ is finite or $A \in \{\ball_r, \sphere_r\}$, where the definition is obvious.

\subsubsection*{Regret}
Recall the regret is defined by
\begin{align*}
\Reg_n = \sup_{x \in K} \sum_{t=1}^n (f_t(X_t) - f_t(x)) \,.
\end{align*}
The $\sup$ cannot always be replaced by a $\max$ because even for compact $K$ the loss functions may not be continuous. For example, when
$K = [0,1]$ and $f \colon K \to [0,1]$ is defined by $f(0) = 1$ and $f(x) = x$ for $x > 0$, then $f$ is convex and does not have a minimiser on $K$.
We occasionally need the regret relative to a specific $x \in K$, which is
\begin{align*}
\Reg_n(x) = \sum_{t=1}^n (f_t(X_t) - f_t(x))\,.
\end{align*}

\section{Notes}

\begin{enumeratenotes}
\item \label{note:intro:smooth} The function classes outlined in \cref{def:class} are by no means the only ones considered.
For example, the definition of smoothness can be generalised beyond the second-order smoothness. There are multiple ways to do this, but
one is to call a function $f \colon \R^d \to \R$ smooth of order $p \in [2,\infty)$ on $\R^d$ if for all $x, y \in \R^d$,
\begin{align*}
\norm{D^qf(x) - D^qf(y)} \leq \beta \norm{x - y}^{p - q} \,,
\end{align*}
where $q$ is the largest integer strictly smaller than $p$ and the norm of the left-hand side is the operator norm.
When $p = 2$ and $f$ is convex this definition is equivalent to what appears in \cref{def:class} \citep[Theorem 2.1.5]{nesterov2018lectures}.
Zeroth-order convex optimisation has been studied for highly smooth functions with slightly varying definitions of smoothness by 
\cite{polyak1990optimal,bach2016highly,akhavan2020exploiting,akhavan2024gradient} and others.
This line of work is briefly discussed in the notes of \cref{chap:ftrl}.

\item \label{note:K} The focus of this book is on a particular kind of \textit{constrained} optimisation where the learner must play in a convex body $K$, which by definition is compact, convex 
and has a nonempty interior. As mentioned, only boundedness and convexity are really important for the results in this book. 
But assuming boundedness discards many interesting settings.
Most fundamentally, it prohibits the \textit{unconstrained} setting where $K = \R^d$. Similarly, $K$ cannot be a half-space or other unbounded subset of $\R^d$.\index{setting!unconstrained}
Practically speaking, unbounded domains can often be reduced to bounded ones with prior knowledge that the minimiser must lie in a bounded convex set.
But such reductions do not always play well with the other assumptions; for example, that the loss is bounded on $K$.

\item \label{note:improper} Another situation related to the constraint set that arises often is that the losses are defined on all of $\R^d$ and the learner can query any point in $\R^d$ but is
only evaluated against the best point in $K$. This setting is sometimes referred to as the \textit{improper} setting.\index{setting!improper} The definition ensures that the improper setting is easier
than both the constrained setting (the learner has more power) and the unconstrained setting (the adversary has less power).

\item Other performance criteria also exist in the literature. For example, both the distance to the minimiser and the gradient of the loss are potentially sensible measures
of quality. Neither of these is considered in this book.

\item We assume throughout that the noise either vanishes completely or is $1$-subgaussian: $\E_{t-1}[\exp(\eps_t^2)|X_t] \leq 2$ and $\E_{t-1}[\eps_t] = 0$.
At the same time, we often assume that the range of the losses is in $[0,1]$.
This means the scale of the noise is the same as the scale of the losses and makes it hard to know how the range or variance of the noise might affect
the regret bound if either of these quantities were scaled individually. One analysis that decouples these quantities is given in \cref{sec:sc:stoch}; occasionally
we give pointers to the literature or suggest exercises/open problems. 

\item More exotic interaction protocols have also been investigated. \index{setting!two-point} \label{note:two-point} 
Consider for a moment adversarial setting without noise.
A number of authors have explored what changes if the learner is allowed to choose
\textit{two} points $X_{t,1}, X_{t,2} \in K$ and observes $f_t(X_{t,1})$ and $f_t(X_{t,2})$.
One might believe that such a modification would have only a mild effect but this is not at all the case. 
Having access to two evaluations makes the bandit setup behave more like the full information setting \citep{ADX10,nesterov2017random,duchi2015optimal}.
Essentially the learner can compute directional derivatives to arbitrary precision, which is not possible in the standard setting.
Note that in the stochastic setting\index{setting!stochastic} the noise is what makes it impossible to compute directional derivatives, 
while in the adversarial setting there need not be any relation from one loss to the next.

\item At the end of the day, the number of settings and possible assumptions is \textit{enormous}. We have naturally selected a subset most aligned with our interest and expertise.
Many of the algorithms and analysis presented can be generalised straightforwardly to other settings. But sometimes the conditions and modifications are subtle.
As much as possible we try to give pointers to the literature, but no doubt much has been missed.

\end{enumeratenotes}

\chapter[Overview of Methods and History]{Overview of Methods and History\copynotice}\label{chap:overview}

This chapter briefly outlines the key algorithmic ideas and history of bandit convex optimisation.
There follow in \cref{sec:lower} and \cref{sec:table} summary tables of known lower and upper bounds for the various settings
studied in this book. 

\section{Methods for Bandit Convex Optimisation}
Methods for bandit convex optimisation can be characterised into five classes:
\begin{itemize}
\item \textit{Cutting plane methods} are important theoretical tools for linear programming\index{linear programming} and non-smooth convex optimisation. \index{cutting plane methods}
The high-level idea is to iteratively cut away pieces of $K$ that have large volume while ensuring that the minimiser stays inside the active set.
Cutting plane methods are the geometric version of elimination algorithms for bandits and consequentially  
are typically analysed in the stochastic setting.\index{setting!stochastic} 
At least three works have adapted these ideas to stochastic convex bandits. \cite{AFHKR11} and \cite{LG21a} both use the ellipsoid method, while \cite{carpentier2024simple}
uses the centre of gravity method.\index{centre of gravity method}\index{ellipsoid method}
A simple bisection algorithm is discussed in \cref{chap:bisection} 
for one-dimensional convex bandits while in \cref{chap:ellipsoid} the approach is generalised to higher dimensions using the centre of gravity method, the ellipsoid method or the method
of the inscribed ellipsoid.

\item \textit{Gradient descent} is the fundamental algorithm for (convex) optimisation and a large proportion of algorithms for convex bandits \index{gradient descent} 
use it as a building block \citep[and more]{Kle04,FK05,Sah11,HL14}. 
At a high level the idea is to estimate gradients of a smoothed version of the loss and use these in gradient descent in place of the real unknown gradients.
We explore this idea in depth in \cref{chap:sgd,chap:ftrl}.

\item \textit{Newton's method} is a second-order method that uses curvature\index{curvature} information as well as the gradient.  
One of the challenges in bandit convex optimisation is that algorithms achieving optimal regret need to behave
in a way that depends on the curvature. Second-order methods that estimate the Hessian of the actual loss or a surrogate 
have been used for bandit convex optimisation \citep{SRN21,LG23,suggala2024second,LFMV24} and are the topic of \cref{chap:ons}.

\item \textit{Continuous exponential weights} is a powerful algorithm for full information\index{setting!full information} online learning and has been used for convex bandits by \cite{BEL16}, who \index{continuous exponential weights}
combined it with the surrogate loss function described in \cref{chap:opt} along with many tricks to construct the first polynomial-time algorithm for bandit
convex optimisation in the adversarial setting with $O(\sqrt{n})$ regret and without any assumptions beyond boundedness.\index{setting!adversarial}
Their algorithm is more complex than one might like
and is not discussed here except for the special case when $d = 1$ where many details simplify and the approach 
yields a reasonably practical algorithm. More details are in \cref{chap:exp}.

\item \textit{Information-directed sampling} is a principled Bayesian algorithm for sequential decision making \citep{RV14}. \index{information-directed sampling}
\cite{BDKP15} showed how to use information-directed sampling to bound the Bayesian regret\index{regret!Bayesian} for one-dimensional convex bandits and then applied minimax
duality\index{duality!minimax} to argue that the minimax Bayesian regret is the same as the adversarial regret.
This idea was later extended by \cite{BE18} and \cite{Lat20-cvx}.
Although these methods still yield the best bounds currently known for the adversarial setting, they are entirely non-constructive\index{non-constructive} thanks to the application of
minimax duality. We explain how these ideas relate to continuous exponential weights and mirror descent in \cref{chap:exp}.

\end{itemize}

\section{History}
Bandit convex optimisation is a relative newcomer in the bandit literature, with the earliest work by \cite{Kle04} and \cite{FK05}, 
both of whom use gradient-based methods in combination with gradient estimates of the smoothed losses (explained in \cref{chap:sgd}).
At least for losses in $\cF_{\pb,\pl}$ they showed that the regret is at most $O(n^{3/4})$. 
While these works seem to be the first to have considered the regret criterion for zeroth-order convex optimisation, their algorithms 
strongly resemble those algorithms designed for zeroth-order stochastic approximation, which we briefly summarise in \cref{sec:his:sa}.

\cite{ADX10} showed that by assuming strong convexity and smoothness the regret of these algorithms could be improved to $\tilde O(\sqrt{n})$
in the improper setting where the learner is allowed to query outside $K$ (see Note~\ref{note:improper}).\index{setting!improper}
The big question was whether or not $\tilde O(\sqrt{n})$ regret is possible without assuming smoothness and strong convexity.
A resolution in the stochastic setting\index{setting!stochastic} was provided by \cite{AFHK13}, who used the ellipsoid method\index{ellipsoid method} in combination with the pyramid construction of
\citet{NY83}, which is classically used for deterministic zeroth-order optimisation.\index{setting!deterministic}
They established $\tilde O(d^{16}\sqrt{n})$ regret while assuming only that the loss is bounded and Lipschitz. 
Because their algorithm is essentially an elimination method, the idea does not generalise to the adversarial setting, where the minimiser may appear to be
in one location for a long time before moving elsewhere. 

Meanwhile, back in the adversarial setting \cite{HL14} assumed strong convexity and smoothness to prove that 
a version of follow-the-regularised-leader\index{follow-the-regularised-leader} achieves $\tilde O(\sqrt{n})$ regret without the assumption that the learner can play outside the constraint set,\index{constraint set} thus
improving the results of \cite{ADX10}. The observation is that the increased variance\index{variance} of certain estimators when the learner is playing close to the boundary
can be mitigated by additional regularisation at the boundary using a self-concordant barrier (\cref{chap:ftrl}).\index{self-concordant barrier}

One fundamental question remained, which is whether or not
$\tilde O(\sqrt{n})$ regret was possible in the adversarial setting without strong convexity or smoothness.
The first breakthrough in this regard came when \cite{BDKP15} proved that $\tilde O(\sqrt{n})$ regret \textit{is} possible in the adversarial setting with no assumptions
beyond convexity and boundedness, but only when $d = 1$.
Strikingly, their analysis was entirely non-constructive, with the argument relying on minimax duality to relate the Bayesian regret to the adversarial
regret and on information-theoretic means to bound the Bayesian regret \citep{RV14}.

\cite{BE18} subsequently extended the information-theoretic tools to $d > 1$, showing for the first time that $\poly(d) \sqrt{n}$ regret is possible in
the adversarial setting. Later, \cite{Lat20-cvx} refined these arguments to prove that the minimax regret for adversarial bandit convex optimisation
with no assumptions beyond boundedness and convexity is at most $d^{2.5} \sqrt{n}$. This remains the best result known in the adversarial setting with losses in $\cF_\pb$.
One last chapter in the information-theoretic story is a duality\index{duality!information ratio} between the information-theoretic means and classical approaches based on 
mirror descent. \cite{lattimore2021mirror} have shown that any bound obtainable with the information-theoretic machinery of \cite{RV14} can also be obtained
using mirror descent. Their argument is still non-constructive\index{non-constructive} since the mirror descent algorithm needs to solve an infinite-dimensional convex optimisation problem.
Nevertheless, we believe this is a promising area for further exploration (\cref{chap:exp}).

The most obvious remaining challenge was to find an efficient algorithm with $\tilde O(\sqrt{n})$ regret for the adversarial setting and losses in $\cF_\pb$. 
An interesting step in this direction was given by \cite{HL16} who proposed an algorithm with $\tilde O(\sqrt{n})$ regret 
but super-exponential dependence on the dimension. Their algorithm has a running time of $O(\log(n)^{\poly(d)})$.

Finally, \cite{BEL16} constructed an algorithm based on continuous exponential weights for which the regret in the adversarial setting with losses 
in $\cF_\pb$ is bounded by $\tilde O(d^{10.5} \sqrt{n})$. Furthermore, the algorithm can be implemented in polynomial time.
Although a theoretical breakthrough, this algorithm has several serious limitations. For one, the dimension-dependence is so large
that in practically all normal situations one of the earliest algorithms would have better regret. Furthermore, although the algorithm can be
implemented in polynomial time, it relies on approximate log-concave sampling and approximate convex optimisation in every round. Practically
speaking the algorithm is nearly impossible to implement. The exception is when $d = 1$ where many aspects of the algorithm simplify (\cref{chap:exp}).

The remaining challenge at this point was (and still is) to improve the practicality of the algorithms and reduce the dimension-dependence in the regret.
\cite{LG21a} used the ellipsoid method\index{ellipsoid method} in the stochastic setting\index{setting!stochastic} in combination with the surrogate loss introduced by \cite{BEL16} to show
that $\tilde O(d^{4.5} \sqrt{n})$ regret is possible in that setting with a semi-practical algorithm. Recently \cite{LG23} showed that $O(d^{1.5}\sqrt{n})$ regret
is possible in the improper\index{setting!improper} stochastic setting when the loss is Lipschitz and $K = \ball_1$, a result that was extended to the constrained setting without
the Lipschitz assumption by \cite{LFMV24}. This last algorithm is detailed in \cref{chap:ons}.

\section{Classical Stochastic Optimisation Methods}\index{non-convex}\label{sec:his:sa}
This book has a particular focus on constrained optimisation with finite-time bounds on the 
regret as a measure of performance. 
You will surely not be surprised to know that stochastic zeroth-order optimisation (often called stochastic approximation) has a long history. 
One of the earliest works is by \cite{KW52},\footnote{This paper
is just five pages long. You should go and read it right now and come back.}
who studied the one-dimensional problem and constructed an iterative algorithm based on the Robbins--Monro algorithm 
for root finding \citep{RM51}. 
The same idea was generalised to the multidimensional setting by \cite{Blu54}.

The stochastic approximation literature is now enormous.
A short and readable summary of the earlier developments is by \cite{spall1994developments} while the modern literature is
covered in detail by \cite{prashanth2025gradient}.
The algorithms developed by this community have a similar flavour to the gradient-based methods used in bandit convex optimisation.
Concretely, they propose iterative schemes based on gradient descent or Newton's method with a huge array of methods
for estimating the gradient and Hessian using zeroth-order oracles.

The fundamental difference between stochastic approximation and bandit convex optimisation is in the assumptions and analysis. 
The stochastic optimisation community has largely avoided assuming global convexity
but is quite satisfied to make assumptions on smoothness and prove asymptotic convergence bounds, often to a local minimum.
Because of this, the results have a different flavour and can be hard to compare.
For example, \cite{KW52} prove their scheme converges to the minimiser of $f$ in probability under mild assumptions on the noise and assuming that
\begin{enumerate}
\item $f$ is unimodal\index{unimodal} and has a global minimiser $x_\star \in \R$;  
\item $f$ is bounded and Lipschitz on an interval containing $x_\star$; and
\item $f$ is not arbitrarily flat away from the minimiser. Formally, there exists a function $\varrho \colon (0,\infty) \to (0, \infty)$ such that
\begin{align*}
|x - x_\star| \geq \eps \text{ implies } \inf_{\delta \in (0,\eps/2)}  \frac{|f(x + \delta) - f(x - \delta)|}{\delta} > \varrho(\eps)\,. 
\end{align*}
\end{enumerate}
Convergence in probability is rather a weak notion and later work has strengthened this considerably, 
for example by proving asymptotic normality \citep{spall1992multivariate}.

There are also considerable differences in methods of analysis. The standard method in the stochastic approximation community is to 
view the iterates of a gradient-based method as an approximation of a differential equation, while in bandit convex optimisation the workhorse
is the theory of online convex optimisation (i.e. online learning).
At the moment there seems to be a regrettable divide between the bandit convex optimisation literature and the classical stochastic optimisation literature.
The two communities focus on different aspects of related problems, with the former intent on finite-time bounds 
in the convex setting with limited assumptions
beyond global convexity. The latter lean more towards asymptotic analysis with more local assumptions.
It seems there is some scope for unification, which we do not address at all here.

\subsubsection*{Dependent Noise Model}
In many works on zeroth-order stochastic optimisation 
there is some unknown convex function $f \colon K \to \R$ to be minimised. The learner has oracle access to some
function $F \colon K \times \Omega \to \R$ and a probability measure $\rho$ on a measurable space $(\Omega, \cG)$ such that
$\int_\Omega F(x, \xi) \d{\rho}(\xi) = f(x)$ for all $x \in K$.
Equivalently, $f(x) = \E[F(x, \xi)]$ when $\xi$ has law $\rho$.
The learner can sample freely from $\rho$ and query $F$ at any point $x \in K$ and $\xi \in \Omega$.
Structural assumptions are then made on $F$, $f$ or both. For example, \cite{nesterov2017random} assume that $f$ is convex and $x \mapsto F(x, \xi)$
is Lipschitz $\rho$-almost surely.
The big difference relative to the setting we study is that the learner can query $x\mapsto F(x, \xi)$ at multiple points with the same $\xi \in \Omega$.
Our stochastic setting can more or less be modelled in this setting but with the only assumption on $F$ being that for all $x \in K$ the random variable
$F(x, \xi)$ has well-behaved moments. 
Our setting is slightly more generic because we allow the noise to depend on the history as well as the decision.
Whether or not you want to make continuity/Lipschitz/smoothness assumptions on $F$ depends on how your problem is modelled.
Here are two real-world examples.

\begin{itemize}[itemsep=2pt]
\item
You are crafting a new fizzy beverage and are deciding the sugar content. A focus group has been arranged and with each person you can give a few samples and
obtain their scores. You want to find the amount of sugar that maximises the expected score over the entire population. 
This problem fits the stochastic optimisation viewpoint because you can trial multiple recipes with each person in your focus group.
Connecting this formally to the notation, $\Omega$ is the space of potential customers and $\rho$ is some reasonable distribution over customers. 
$K$ is the space of parameters for the fizzy beverage (the amount of sugar) and $F(x, \xi)$ is the loss suffered by person $\xi$ on 
the beverage with sugar content $x$.

\item\tikz[baseline={([yshift={-0.23cm}]a.north)}]{\node[inner sep=0pt] (a) {
\begin{minipage}[t]{\linewidth}
\begin{wrapfigure}{r}{4cm}
\includegraphics[width=4cm]{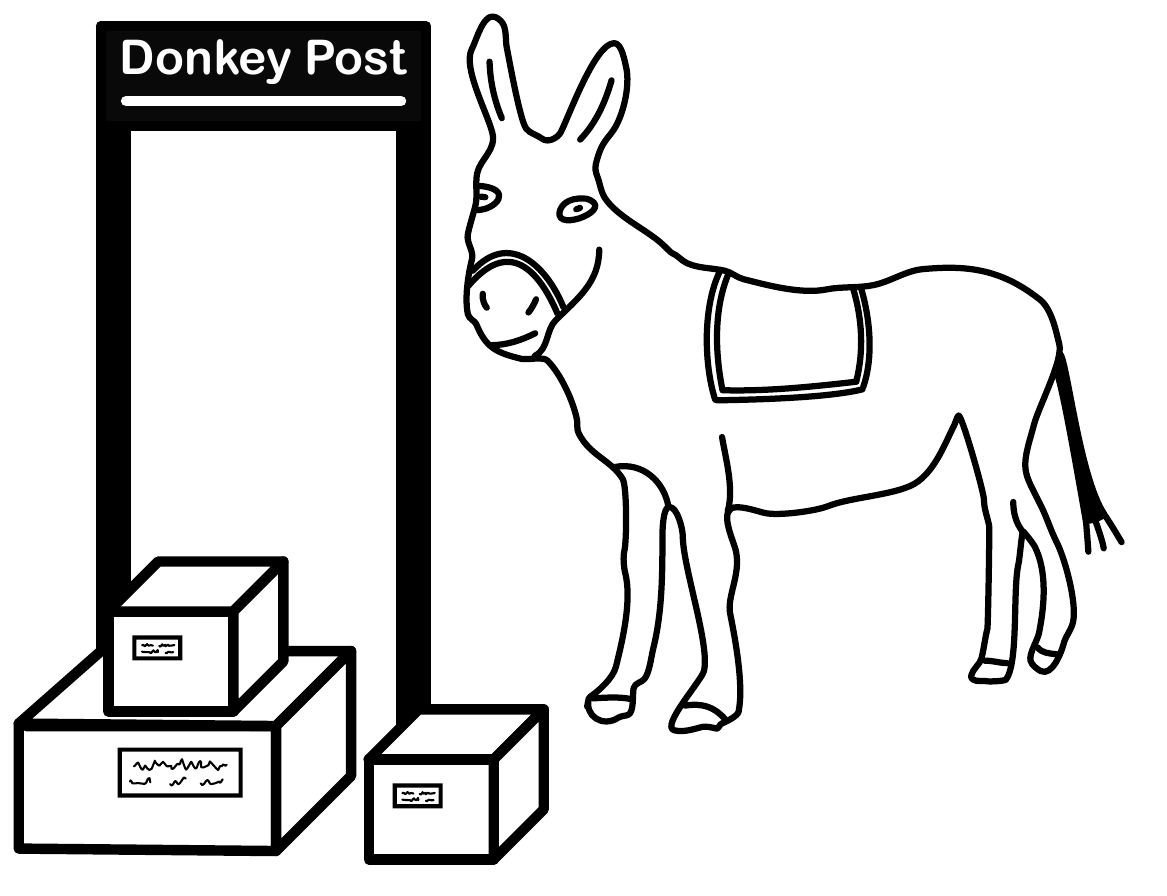}
\end{wrapfigure}
You operate a postal service using donkeys to transport mail between Sheffield and Hathersage. Donkeys are stoic creatures and do not give away how tired they are.
Every day you decide how much to load your donkey. Overload and they might have a nap along the way but obviously you want to transport as much post as possible.
The success of a journey is a function of how much mail was delivered and how long it took.
You'll get a telegraph with this information at the end of the day.
This problem is best modelled using the bandit framework because the tiredness of the donkey varies from day to day unpredictably and you only
get one try per day.
Formally in this setting $K$ is the set of possible loads for the donkey and $\rho$ is a distribution on unobservable states.
\end{minipage}
};}
\end{itemize}

There is too much material to do justice to this literature here. Some influential papers are by 
\cite{ghadimi2013stochastic} and \cite{NJL09}.

\section{Lower Bounds Summarised}\label{sec:lower}
The best lower bound known when the losses are in $\cF_\pb$ 
is that the minimax regret is at least $\Omega(d \sqrt{n})$.
Interestingly, the lower bound was established using linear losses where the upper bound is $\tilde O(d \sqrt{n})$. 
Can it really be that the hardest examples in the enormous non-parametric class of bounded convex functions $\cF_\pb$ lie in the tiny subset of linear functions?
Our intuition from the full information setting\index{setting!full information} says it could be like this. Curvature always helps in the full information setting.
We discuss in \cref{chap:ons} why, in bandit convex optimisation, curvature both helps and hinders in a complicated way.\index{curvature}
There also exist lower bounds for more structured classes of convex losses, which are summarised in \cref{tab:lower}.

\subsubsection*{Reading the Table} 
In the class column we use an \texttt{(s)} superscript to indicate that the lower bound holds in the stochastic setting
and \texttt{(i)} to show that it holds in the improper setting explained in Note~\ref{note:improper}. In the improper setting, Lipschitzness and
boundedness only hold inside $K$ while smoothness and strong convexity hold everywhere.\index{setting!improper}

\begin{table}[h!]
\small
\caption{Summary of lower bounds.
}\label{tab:lower}

\begin{subtable}{\textwidth}
\renewcommand{\arraystretch}{1.4}
\begin{tabular}{|p{2.8cm}p{2.2cm}p{3.2cm}p{1.05cm}|}
\hline 
\textsc{author} & \textsc{regret} & \textsc{class} & $K$ \\ \hline
\citealt{DHK08}$^1$ & $\Omega(d \sqrt{n})$ & $\cF^{\ps}_{\pb,\plin}$ & $\prod_{k=1}^{d/2} \mathbb S^1_1$  \\
\citealt{Sha13} & $\Omega(d \sqrt{n})$ & $\cF^{\ps,\pu}_{\pb,\pl,\psc,\psm}$, $\alpha = \frac{1}{2}, \beta = \frac{7}{2}$ & $\ball_1$ \\
\citealt{Sha13} & $\Omega(d \sqrt{n})$ & $\cF^{\ps}_{\pb,\pl,\pquad,\psc,\psm}$, $\alpha = \beta = 1$ & $\ball_1$ \\
\citealt{akhavan2024gradient}$^2$ & $\Omega(\frac{d \sqrt{n}}{\max(1,\alpha)})$ & $\cF^{\ps,\pu}_{\pb,\pl,\psc,\psm}$, $\beta = O(\alpha)$ & $\ball_1$ \\ \hline
\end{tabular}
\caption{Lower bounds on the regret.
$^1$The product of spheres is not convex. As \cite{shamir2015complexity} argues, in the linear setting lower bounds for non-convex sets
imply lower bounds on $\conv(K)$ via a simple reduction. 
$^2$These authors also prove a more general result for function classes with more smoothness and their results
hold for a large class of noise models.}

\end{subtable}
\begin{subtable}{\textwidth}
\renewcommand{\arraystretch}{1.4}
\begin{tabular}{|p{2.8cm}p{2.2cm}p{3.2cm}p{1.05cm}|}
\hline
\textsc{author} & \textsc{simple regret} & \textsc{class} & $K$ \\ \hline
\citealt{Sha13} & $\Omega(d / \sqrt{n})$ & $\cF^{\ps,\pu}_{\pb,\pl,\psc,\psm}$, $\alpha = \frac{1}{2}, \beta = \frac{7}{2}$ & $\ball_1$ \\
\citealt{akhavan2024gradient}$^2$ & $\Omega\left(\frac{d}{\max(1,\alpha) \sqrt{n}}\right)$ & $\cF^{\ps,\pu}_{\pb,\pl,\psc,\psm}$, $\beta = O(\alpha)$ & $\ball_1$ \\
\hline
\end{tabular}
\caption{Lower bounds on the simple regret. $^2$See the footnote for table above.}
\end{subtable}
\end{table}

\FloatBarrier

Let us make some comments on the lower bounds and how they relate to each other:
\begin{itemize}
\item 
The unconstrained (Note~\ref{note:K}) and constrained settings are both harder than the improper setting (Note~\ref{note:improper}).\index{setting!unconstrained}
Hence, lower bounds that hold in the latter also hold in the unconstrained/constrained settings.
\item 
The reduction in \cref{eq:simple} shows that lower bounds on the simple regret imply lower bounds on the cumulative regret.
\item
Lower bounds that are proven with Gaussian noise can be generalised to lower bounds with bounded noise at the price of at most logarithmic factors, which follows
by a scaling and truncation argument \citep{shamir2015complexity}.  
\item 
The adversarial setting as defined here is strictly harder than the stochastic setting, which means that all the bounds in \cref{tab:lower} also apply to the adversarial setting.
Many authors focus on the adversarial setting without noise and where the losses are assumed to be bounded.
The scaling and truncation argument by \cite{shamir2015complexity} shows that lower bounds proven with Gaussian noise also apply in this case except for logarithmic factors.

\item 
As far as we know, no one has written a minimax simple regret lower bound for linear losses. At least when $K = [-1,1]^d$, the technique for bounding the cumulative
regret by \cite{LS20book} also yields a bound on the simple regret of $\Omega(d/\sqrt{n})$.
\item 
The quadratic case with smoothness and strong convexity ($\alpha, \beta = \Theta(1)$) is quite interesting. 
The lower bound on the cumulative regret is $\Omega(d \sqrt{n})$, which is matched
by upper bounds in the unconstrained and improper settings \citep{akhavan2020exploiting} and nearly so in the constrained setting \citep{HL14}.
In the unconstrained and improper settings, however, the upper bound on the simple regret improves to $O(d^2/n)$.
This shows that the reduction in \cref{eq:simple} is not guaranteed to be tight.
\item
Many of the algorithms in this book are based on combining gradient descent with noisy gradient estimates of some surrogate loss function.
\cite{HPGySz16:BCO} explore the limitations of this argument. Their idea is to modify the information available to the learner.
Rather than observing the loss directly, the learner observes a noisy gradient estimate from an oracle that satisfies certain conditions
on its bias\index{bias} and variance.\index{variance} This allows the authors to prove a lower bound\index{lower bound} in terms of the bias and variance of the oracle that holds for any algorithm.
The main application is to argue that any analysis using the spherical smoothing estimates explained in \cref{chap:sgd} either cannot 
achieve $O(\sqrt{n})$ regret or must use some more fine-grained properties of the specific estimator than its bias and variance alone.
\end{itemize}

\section{Upper Bounds Summarised}\label{sec:table}

\newcommand{\optsc}{\textsc{opt}}
\newcommand{\svd}{\textsc{svd}}
\newcommand{\cog}{\textsc{cog}}
\newcommand{\ellipsoid}{\textsc{ellipsoid}}
\newcommand{\inscribed}{\textsc{inscribed}}
\newcommand{\low}{\textsc{l\"ow}}
\newcommand{\iso}{\textsc{iso}}

\cref{tab:upper} summarises the past and current situation. 
Those bounds that depend on $n$ are regret bounds while those that depend on $\eps$ are sample complexity bounds.\index{sample complexity}
Remember, you can use \cref{fact:conversion} or \cref{prop:conversion} to convert a regret bound into a sample complexity bound.
The superscript \texttt{(s)} in the function classes indicate whether or not the work only considers the
stochastic setting, while the superscript \texttt{(i)} is used when the algorithm needs to query outside $K$ (the improper setting of Note~\ref{note:improper}).\index{setting!improper} 
In some cases the subscript $d = 1$ in the function class indicates that the algorithm assumes the dimension is one.
The quantity $\vartheta$ is the parameter associated with a self-concordant barrier on $K$ (see \cref{chap:ftrl}) and $D$ is an abbreviation for the diameter $D = \diam(K)$.
The \textsc{compute} column gives the per-round complexity of each algorithm, which is $\infty$ in the few cases that the regret bound was established non-constructively.
Some algorithms need to position $K$ into L\"owner's position or isotropic position, marked in the \textsc{compute} column by \low{} and \iso{}, respectively.
The \textsc{compute} column also indicates whether the algorithm uses one of the classical cutting plane methods: ellipsoid method (\ellipsoid), centre of gravity method (\cog)
or the method of inscribed ellipsoid (\inscribed{}).
The symbol $\Pi$ refers to the complexity of a euclidean projection onto $K$, which depends on how $K$ is represented, and \textsc{svd} refers to the complexity of
a singular value decomposition, which is generally $O(d^3)$.\index{singular value decomposition}
Those algorithms that use $O(1)$ computation per round all query the same point for many rounds in a row. In many implementations this could take $O(d)$ computation, 
since storing/copying the iterate generally has this complexity. We offload this aspect to the oracle computing the loss function.  
The last column refers to the chapter where we analyse the relevant algorithm, if applicable.

\begin{table}[h!]
\small
\caption{Summary of upper bounds}\label{tab:upper}
\setlength\tabcolsep{1pt}
\renewcommand{\arraystretch}{1.3}
\begin{longtable}{|lp{2.7cm}lll|}
\hline
\textsc{author} & \textsc{regret/comp} & \textsc{class} & \textsc{compute} & \textsc{ch} \\ \hline
\cite{FK05} & \hyperref[thm:sgd]{$O(d^{1/2} D^{1/2}n^{3/4})$} & $\sF_{\pb,\pl}$ & $O(d), \Pi$  & \ref{chap:sgd} \\ 
\cite{FK05} & $O(d n^{5/6})$ & $\sF_\pb$ & $O(d), \Pi, \ISO$  & \textendash \\
This book & \hyperref[thm:ftrl:basic]{$\tilde O(d^{1/2} \vartheta^{1/4} n^{3/4})$} & $\sF_\pb$ & $O(d^2), \optsc, \svd$  & \ref{chap:ftrl}\\
\cite{ADX10}$^1$ & $\tilde O(d \sqrt{\beta n / \alpha})$ & $\sF^{\pu}_{\pb,\psm,\psc}$ & $O(d), \Pi$ & \textendash \\
This book & \hyperref[thm:sgd:sm-sc]{$O(d \sqrt{\beta n / \alpha})$} & $\sF^{\pu}_{\pb,\psm,\psc}$ & $O(d), \Pi$  & \ref{chap:sgd} \\
\cite{Sah11} & \hyperref[thm:ftrl:smooth]{$\tilde O([\vartheta \beta]^{1/3} [D d n]^{2/3})$} & $\sF_{\pb,\psm}$ & $O(d^2), \optsc, \svd$ & \ref{chap:ftrl} \\
\cite{BCK12} & \hyperref[thm:linear]{$\tilde O(d \sqrt{n})$} & $\sF_{\pb,\plin}$ & $O(\exp(d))$ & \ref{chap:lin} \\
\cite{AFHK13} & \hyperref[thm:bisection]{$\tilde O(\sqrt{n})$} & $\sF_{\pl,d=1}^\ps$ & $O(1)$ & \ref{chap:bisection} \\ 
\cite{AFHK13} & $\tilde O(d^{16} \sqrt{n})$ & $\sF_{\pb,\pl}^\ps$ & $O(1)$, \ellipsoid & \textendash \\
\cite{HL14} & \hyperref[thm:ftrl:sc-smooth]{$\tilde O(d \sqrt{(\vartheta + \beta/\alpha)n)}$} & $\sF_{\pb,\psm,\psc}$ & $O(d^2), \optsc, \svd$  & \ref{chap:ftrl} \\ 
\cite{BLN15} & $\tilde O(d^{7.5} / \eps^2)$ & $\sF^\ps$ & $O(1)$ & \textendash \\
\cite{BDKP15} & $\tilde O(\sqrt{n})$ & $\sF_{\pb,d=1}$ & $\infty$ & \textendash \\ 
\cite{HL16} & $\tilde O(2^{(d^4)} \sqrt{n})$ & $\sF_\pb$ & $O(\log(n)^{\poly(d)})$ & \textendash \\
\cite{BEL16} & $\tilde O(d^{10.5} \sqrt{n})$ & $\sF_\pb$ & $\poly(d, n)$ & \textendash \\
\cite{BEL16} & \hyperref[thm:exp-1d]{$\tilde O(\sqrt{n})$} & $\sF_{\pb,d=1}$ & $O(\sqrt{n})$ & \ref{chap:exp} \\ 
\cite{BEL18} & $\tilde O(d^{18} \sqrt{n})$ & $\sF_\pb$ & $\infty$ & \textendash \\
\cite{akhavan2020exploiting}$^2$ & $O(d \sqrt{\beta n / \alpha})$ & $\sF_{\pl,\psm,\psc}^{\ps,\pu}$ & $O(d)$ & \textendash \\
\cite{Lat20-cvx} & $\tilde O(d^{2.5} \sqrt{n})$ & $\sF_\pb$ & $\infty$ & \textendash \\ 
\cite{Ito20}$^3$ & $\tilde O(d \sqrt{\beta n / \alpha})$ & $\sF_{\pb,\psm,\psc}$ & $\poly(d)$ & \textendash \\ 
\cite{Ito20} & $\tilde O(d^{1.5} \sqrt{\beta n/\alpha})$ & $\sF_{\pb,\psm,\psc}$ & $\poly(d)$ & \textendash \\
\cite{SRN21} & $\tilde O(d^{16} \sqrt{n})$ & $\sF_{\pb,\pquad}$ & $\poly(d)$ & \textendash  \\
\cite{LG21a} & $\tilde O(d^{4.5} \sqrt{n})$ & $\sF_\pb^\ps$ & $O(d^2)$, \ellipsoid & \textendash \\
\cite{LG23} & $\tilde O(d^{1.5} \sqrt{n})$ & $\sF_{\pl}^{\ps,\pu}$ & $O(d^2), \svd$  & \textendash \\
\cite{LFMV24}$^4$ & \hyperref[thm:ons:bandit]{$\tilde O(d^{1.5} \sqrt{n})$} & $\sF_\pb^\ps$ & $O(d^2), \Pi, \svd$, \low &  \ref{chap:ons} \\
\cite{LFMV24} & \hyperref[thm:ons:bandit]{$\tilde O(d^{2} \sqrt{n})$} & $\sF_\pb^\ps$ & $O(d^2), \Pi, \svd$, \iso & \ref{chap:ons}\\
\cite{LFMV24} & \hyperref[thm:ons:bandit]{$\tilde O(d^{2.5} \sqrt{n})$} & $\sF_\pb$ & $\poly(d,n)$, \iso & \ref{chap:ons-adv}\\
\cite{carpentier2024simple} & \hyperref[thm:cut:cog]{$\tilde O(d^4/\eps^2)$} & $\sF_\pb^\ps$ & $O(d^2)$, \cog & \ref{chap:ellipsoid} \\
This book & \hyperref[sec:cut:ellipsoid]{$\tilde O(d^5 / \eps^2)$} & $\sF_\pb^\ps$ & $O(d^2)$, \ellipsoid & \ref{chap:ellipsoid} \\
This book & \hyperref[thm:cut:inscribed]{$\tilde O(d^4 / \eps^2)$} & $\sF_\pb^\ps$ & $O(d^2)$, \inscribed & \ref{chap:ellipsoid} \\
\hline
\end{longtable} 
\vspace{0.2cm}
\begin{minipage}{\textwidth}
$^1$The proof of this result is only sketched and both the theorem statement and proof contain minor issues.
These are corrected in \cref{sec:gd:sm-sc} where we take the opportunity to incorporate an idea by
\cite{akhavan2020exploiting} to remove the logarithmic factor. \\
$^2$These authors also show how to exploit the case where the noise is small as well as higher-order smoothness. \\
$^3$This result holds when the minimiser lies deep inside $K$. \\
$^4$This also holds if $K$ is symmetric and either in John's position or its polar is isotropic. \\
\end{minipage}
\end{table}

\FloatBarrier

\section{Notes}

\begin{enumeratenotes}
\item There are some books on zeroth-order optimisation \citep[for example]{larson2019derivative,conn2009introduction}. These works focus most of their attention
on noise-free settings and without a special focus on convexity. 
\cite{NY83} is a more theoretically focused book with one chapter on zeroth-order methods.
\cite{BPP12} and \cite{prashanth2025gradient} take the stochastic optimisation viewpoint and focus primarily on 
gradient-based methods in both convex and non-convex settings.
There is also a short survey by \cite{liu2020primer}.

\item Speaking of non-convexity,\index{non-convex} zeroth-order methods are also analysed in non-convex settings. Sometimes the objective is still to find the global minimum, but
for many non-convex problems this cannot be done efficiently. In such cases one often tries to find a point $x \in K$ such that $\norm{f'(x)}$ is small or even a local minimum.
We only study convex problems here. A recent reference for the non-convex case is the work by \cite{balasubramanian2022zeroth}.
\item There are esoteric settings that are quite interesting and may suit some applications. For example, \cite{bach2016highly} study
a problem where the learner chooses two actions in each round. The learner receives information for only the first action but is evaluated based on the quality
of the second. They also study higher levels of smoothness than we consider here.

\item Online learning has for a long time made considerable effort to prove adaptive bounds\index{adaptive} that yield stronger results when the loss functions are somehow nice
or show that the learner adapts to non-stationary environments.\index{non-stationary} 
Such results have also been commonplace in the standard bandit literature and are starting to appear in the convex bandit literature as well
\citep{zhao2021bandit,LZZ22,wang2023adaptivity,LBZ25}
\item 
We have not talked much about the efforts focused on sample complexity or simple regret for the stochastic setting.\index{setting!stochastic}
\cite{jamieson2012query} consider functions in $\cF_{\psm,\psc}$ and $K = \R^d$ and prove a sample complexity bound of $O(\frac{d^{3}}{\eps^2})$ for an algorithm
based on coordinate descent with polynomial dependence on the smoothness and strong convexity parameters hidden.
\cite{BLN15} use an algorithm based on simulated annealing to prove a sample complexity bound of $O(\frac{d^{7.5}}{\eps^2})$ \index{sample complexity}
for losses in $\cF_\pb$. In its current form their algorithm is not suitable for regret minimisation though this minor deficiency may be correctable. 
Another thing to mention about that work is that the algorithm is robust in the sense that it can (approximately) find minima of functions that are only approximately convex.
Slightly earlier, \cite{liang2014zeroth} also used a method based on random walks but obtained a worse rate of $O(\frac{d^{14}}{\eps^2})$.
\end{enumeratenotes}

\chapter[Mathematical Tools]{Mathematical Tools (\skippy)\copynotice}\label{chap:regularity}

The purpose of this chapter is to introduce the necessary tools from optimisation, convex geometry and convex analysis.
You can safely skip this chapter, referring back as needed.
The main concepts introduced are as follows:
\begin{itemize}
\item Convex bodies, the Minkowski and support functions and basic theory of polarity.
\item Basic properties associated with smoothness and strong convexity.
\item The near-Lipschitzness of convex functions and the implications of this for the location of near-minimisers of convex functions.
\item Rounding procedures for convex bodies, including the classical John's and isotropic positions of convex bodies.
\item Smoothing operators and mechanisms for extending the domain of a convex function $f \colon K \to \R$ to all of $\R^d$.
\item Methods for computing various operations on convex bodies, such as projection and optimisation.
\end{itemize}

\section{Convex Bodies}\label{sec:regularity:minkowski}\index{convex body|(}
A convex set $K \subset \R^d$ is a convex body if it is compact and has a nonempty interior.
The latter corresponds to the existence of an $x \in \R^d$ and $\eps > 0$ such that $x + \ball_\eps \subset K$.
The Minkowski functional\index{Minkowski functional|textbf}\label{page:mink} of $K$ is the function $\pi \colon \R^d \to \R$ defined by
\begin{align*}
\pi(x) = \inf\left\{t > 0 \colon x \in t K\right\}\,. 
\end{align*}
A top-down illustration is provided in \cref{fig:mink}.
Another way to visualise the Minkowski functional is via the suspension cone, which is the set \index{suspension cone|textbf}
\begin{align*}
S(K) = \{(x, y) \colon x \in \R^d, y \in \R, \pi(x) \leq y, y \geq 0\} \,.
\end{align*}
The set $S(K)$ is a cone with tip $(\zeros, 0)$ and $\{x \colon (x,1) \in S\} = K$.
In most of our applications $\zeros \in \interior(K)$ and $K$ is a convex body (hence closed). In this case $\sind_K(x) = \sind(\pi(x) \leq 1)$.
Moreover, when $K$ is a symmetric convex body, then $\pi$ is a norm and $K$ is its unit ball.
The support function is \label{page:support} \index{support function|textbf}
\begin{align*}
h(u) = \sup_{x \in K} \ip{u, x}  \,.
\end{align*}
The support function is defined so that for any $\R^d \ni u \neq \zeros$, the hyperplane $\{x \colon \ip{u, x} = h(u)\}$ is a supporting hyperplane of $K$ (\cref{fig:support}).
Of course the Minkowski functional and support function both depend on $K$ as well as $x$. 
When necessary we explicitly write $\pi_K$ or $h_K$ but in general the set will be $K$ and is omitted from the notation.

\begin{figure}[H]
\centering
\begin{subfigure}{0.48\textwidth}
\begin{tikzpicture}[thick]
\draw[fill=grayone] plot [smooth cycle] coordinates {(0,0) (1,1.5) (2,1.9) (3,0) (3,-1) (1.4,-1.5)};
\draw[fill=black] (1.5,-0.5) circle (1pt);
\node[anchor=north] at (1.5,-0.5) {$\zeros$};
\draw (1.5, -0.5) -- (-0.5,0.166);
\draw[fill=black] (-0.5,0.166) circle (1pt);
\draw[fill=black] (0,0) circle (1pt);
\node[anchor=east] at (-0.5,0.166) {$x$};
\node[anchor=north,xshift=-5pt] at (0,0) {$\frac{x}{\pi(x)}$};
\draw[fill=black] (2.25,-0.25) circle (1pt);
\draw[fill=black] (3,0) circle (1pt);
\node[anchor=north] at (2.25,-0.25) {$y$};
\node[anchor=west] at (3, 0) {$\frac{y}{\pi(y)}$};
\draw (3,0) -- (1.5,-0.5);
\node at (1.5, 0) {$K$};
\end{tikzpicture}
\caption{Minkowski functional}\label{fig:mink}
\commentAlt{Figure~\ref{fig:mink}: A convex body in 2 dimensions showing that x divided by the Minkowski functional of x likes on the boundary.}
\end{subfigure}
\hspace{0.02\textwidth}
\begin{subfigure}{0.48\textwidth}
\begin{tikzpicture}[thick]
\draw[fill=grayone] plot [smooth cycle] coordinates {(0,0) (1,1.5) (2,1.9) (3,0) (3,-1) (1.4,-1.5)};
\draw (0,-1) -- (0,1);
\draw[-latex] (0,0) -- (-0.5, 0);
\node[anchor=east] at (-0.5,0) {$u$};
\node[anchor=north,xshift=-13pt] at (0, -1) {$\{x \colon \ip{u, x} = h(u)\}$};
\node at (1.5, 0) {$K$};
\end{tikzpicture}

\caption{Support function}\label{fig:support}
\commentAlt{Figure~\ref{fig:support}: A convex body in 2 dimensions showing that the supporting hyperplane with normal u has intercept equal to the support function of u.}
\end{subfigure}
\caption{The Minkowski and support functions.}
\end{figure}
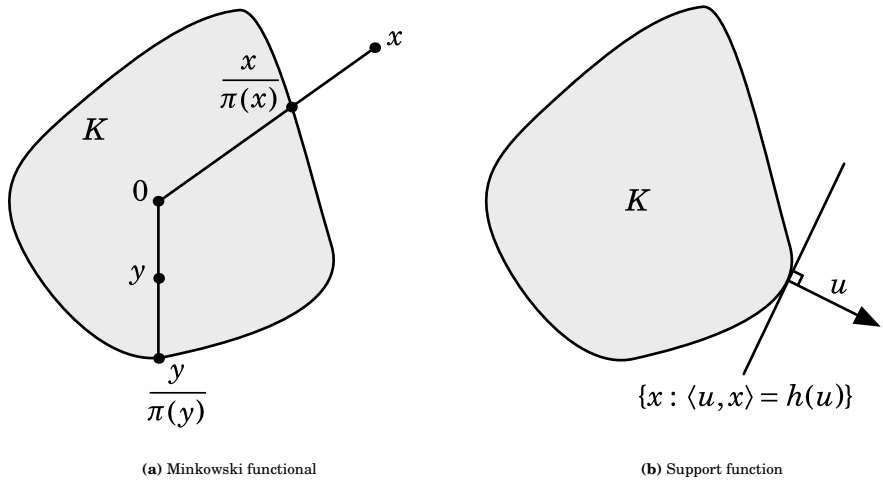
\FloatBarrier

The polar is $K^\circ = \{u \in \R^d \colon h(u) \leq 1\}$.\label{page:polar} \index{polar}
The polar is not easy to visualise and for our purposes an in-depth understanding of this concept is rarely needed. Readers looking for more intuition
and theory should read the classic text by \citet[\S14]{Roc70}.

\begin{proposition}\label{prop:polar-dual}
The following hold:
\begin{enumerate}
\item Given convex bodies $K \subset J$, polarity reverses inclusion: $J^\circ \subset K^\circ$. \label{prop:polar-dual:inc}
\item Given a convex body $K$ with $\zeros \in \interior(K)$, the polar $K^\circ$ is a convex body with $\zeros \in \interior(K^\circ)$.
\item The polar body of a nonempty ball $K = \ball_r$ is $K^\circ = \ball_{1/r}$. \label{prop:polar:ball}
\item For a symmetric convex body $K$, $\norm{\cdot}_K \triangleq \pi_K(\cdot)$ is a norm. \label{page:K-norm}
\item For a symmetric convex body $K$, the dual of $\norm{\cdot}_K$ is $\norm{\cdot}_{K^\circ}$:
\begin{align*}
\norm{u}_{K^\circ} = \max \{\ip{u, x} \colon \norm{x}_K \leq 1, x \in \R^d\} \,.
\end{align*}
\end{enumerate}
\end{proposition}

\begin{exer}
\faStar\quad
Prove \cref{prop:polar-dual}.
\end{exer}

\solution{
\begin{enumerate}
\item Suppose that $K \subset J$. 
Obviously $h_K(u) \leq h_J(u)$.
Suppose that $u \in J^\circ$. Then $1 \geq h_J(u) \geq h_K(u)$, which implies that $u \in K^\circ$.
\item Let $K$ be a convex body with $\zeros \in \interior(K)$.
Hence there exists constant $0 < r < R < \infty$ such that $\ball_r \subset K \subset \ball_R$,
which by parts \ref{prop:polar-dual:inc} and \ref{prop:polar:ball}  means that $\ball_{1/R} \subset K^\circ \subset \ball_{1/r}$.
Hence $\zeros \in \interior(K^\circ)$ and $K^\circ$ is bounded.
Convexity of $K^\circ$ follows because $h_K$ is convex. 
That $K^\circ$ is compact follows from continuity of $h_K$.
\item This follows immediately from the fact that $h_K(u) = r \norm{u}$.
\item Since $K$ is a convex body it has a nonempty interior.
Hence there exists an $r > 0$ and $x \in \R^d$ such that $x + \ball_r \subset K$.
But since $K$ is symmetric, $-x + \ball_r \subset K$ as well and by convexity $\ball_r \subset K$. 
Therefore $\pi(x) = 0$ if and only $x = \zeros$.
That $\pi(ax) = |a| \pi(x)$ is immediate from symmetry. Subadditivity follows from \cref{lem:reg:mink}.
\item 
The Minkowski functional of $K^\circ$ is
\begin{align*}
\pi_{K^\circ}(u) 
&= \inf\{t > 0 \colon u/t \in K^\circ\} \\
&= \inf\{t > 0 \colon h_K(u/t) \leq 1\} \\
&= \inf\{t > 0 \colon h_K(u) \leq t \} \\
&= h_K(u) \,. 
\end{align*}
Since $K = \{x \in \R^d \colon \pi(x) \leq 1\}$,
\begin{align*}
\norm{u}_{K^\circ} = h_K(u) = \max_{x \in K} \ip{x, u} = \max\{\ip{u,x} \colon \norm{x}_K \leq 1, x \in \R^d\} \,.
\end{align*}
\end{enumerate}
}

The Minkowski functional has many properties:

\begin{lemma}\label{lem:reg:mink}
Let $K$ be a convex body with $\zeros \in \interior(K)$ and $\pi$ the associated Minkowski functional.
The following hold:
\begin{enumerate}
\item $\pi(\alpha x) = \alpha \pi(x)$ for all $\alpha > 0$. \label{lem:reg:mink:hom}
\item $\pi$ is convex. \label{lem:reg:mink:cvx}
\item $\pi(x + y) \leq \pi(x) + \pi(y)$ for all $x, y \in \R^d$. \label{lem:reg:mink:sub}
\item $x/\pi(x) \in \partial K$ whenever $\pi(x) > 0$. \label{lem:reg:mink:proj}
\item $\pi$ is the support function of the polar body: $\pi(x) = \sup_{u \in K^\circ} \ip{x, u}$. \label{lem:reg:mink:support}
\item $D\pi(x)[h] = \ip{u, h}$ for some $u \in K^\circ$. \label{lem:reg:mink:diff}
\item $\lip(\pi) \leq 1/r$ whenever $\ball_r \subset K$. \label{lem:reg:mink:lip} 
\end{enumerate}
\end{lemma}

\begin{proof}
Part~\ref{lem:reg:mink:hom} is immediate from the definitions.
For part~\ref{lem:reg:mink:cvx},
by definition, $x \in \pi(x) K$ and $y \in \pi(y) K$.
Hence, for any $\lambda \in [0,1]$, 
\begin{align*}
(1 - \lambda) x + \lambda y \in (1 - \lambda) \pi(x) K + \lambda \pi(y) K = ((1 - \lambda) \pi(x) + \lambda \pi(y)) K \,.
\end{align*}
Therefore $\pi((1 - \lambda) x + \lambda y) \leq (1 - \lambda)\pi(x) + \lambda \pi(y)$, which establishes convexity.
Part~\ref{lem:reg:mink:sub} follows from \ref{lem:reg:mink:hom} and \ref{lem:reg:mink:cvx} since
$\pi(x + y) = \pi((2x)/2 + (2y)/2) \leq \pi(2x)/2 + \pi(2y)/2 = \pi(x) + \pi(y)$.
For part~\ref{lem:reg:mink:proj}, let $x$ be such that $\pi(x) > 0$. That $x / \pi(x) \in K$ is immediate from the definition. 
Suppose that $y = x/\pi(x) \in \interior(K)$; then there exists an $\epsilon > 0$ such that $y + \ball_\epsilon \subset K$.
A simple calculation shows there exists a $\delta \in (0, \pi(x))$ such that $x/(\pi(x) - \delta) \in K$, which contradicts the definition of $\pi(x)$.
Part~\ref{lem:reg:mink:support} is given by \citet[Theorem 14.5]{Roc70}
and part~\ref{lem:reg:mink:diff} follows from part~\ref{lem:reg:mink:support} and \citet[Corollary 23.5.3]{Roc70}. 
Part~\ref{lem:reg:mink:lip} follows because the Minkowski functional is the support function of the polar body $K^\circ$ and  
polarity reverses inclusion, $K^\circ \subset \ball_{1/r}$. Finally, the subgradients\index{subgradient!of support function} of the support function are in $K^\circ$ and the result follows \citep[Corollary 23.5.3]{Roc70}.
\end{proof}

Convex functions $f \colon K \to [0,1]$ are often not well behaved near the boundary (see Section~\ref{sec:regularity:lip}).
For this reason we often shrink $K$ towards the origin.
Given $\eps > 0$, let
\begin{align*}
K_\eps = \{(1 - \eps) x \colon x \in K\} = \{x \in K \colon \pi(x) \leq 1 - \eps\}\,.
\end{align*}
\begin{lemma}\label{lem:reg:ball}
Suppose that $\ball_r \subset K$ and $x \in K_\eps$. Then $x + \ball_{r \eps} \subset K$.
\end{lemma}

\begin{proof}
Let $x \in K_\eps$. By the definition of the Minkowski functional\index{Minkowski functional} there exists a $y \in K$ such that
$x = (1 - \eps) y$. Since $K$ is convex and $\ball_r \subset K$, it follows that
\begin{align*}
K &\supset \eps\ball_r + (1-\eps) y = x + \ball_{\eps r} \,.
\qedhere
\end{align*}
\end{proof}

\index{convex body|)}

\section{Smoothness and Strong Convexity}\index{smooth}\index{strongly convex}

For your own understanding, you should solve the following:

\begin{exer}\label{ex:reg:sm-sc}
\faStar \quad
Suppose that $f \in \cF_{\psm,\psc}$ and $f$ is twice differentiable. Show that \index{differentiable}
\begin{align*}
\alpha \id \preceq f''(x) \preceq \beta \id \text{ for all } x \in \interior(K) \,.
\end{align*}
\end{exer}

Besides this, the only properties of smoothness and strong convexity that we need are as follows:

\begin{lemma}\label{lem:reg:sc}
If $f \in \cF$ is $\alpha$-strongly convex, then for all $y \in \interior(\dom(f))$, 
\begin{align*}
f(x) \geq f(y) + Df(y)[x - y] + \frac{\alpha}{2} \norm{x - y}^2 \,.
\end{align*}
\end{lemma}

\begin{proof}
Let $g(x) = f(x) - \frac{\alpha}{2} \norm{x}^2$, which by assumption is convex.
By convexity,
\begin{align*}
f(x) - \frac{\alpha}{2} \norm{x}^2 
&= g(x)  \\
&\geq g(y) + Dg(y)[x - y] \\
&= f(y) - \frac{\alpha}{2} \norm{y}^2 + Df(y)[x-y] - \alpha \ip{y, x - y}\,.
\end{align*}
Rearranging shows that
\begin{align*}
f(x) 
&\geq f(y) + Df(y)[x-y] + \frac{\alpha}{2} \norm{x}^2 - \frac{\alpha}{2} \norm{y}^2 - \alpha \ip{y, x - y} \\
&= f(y) + Df(y)[x-y] + \frac{\alpha}{2} \norm{x - y}^2 \,.
\qedhere
\end{align*}
\end{proof}

\begin{lemma}\label{lem:smooth}
If $f \in \cF$ is $\beta$-smooth and $X$ is a random variable supported in $K$ and $x = \E[X]$. Then,
\begin{align*}
\E[f(X) - f(x)] \leq \frac{\beta}{2} \E\left[\norm{X - x}^2\right] \,.
\end{align*}
\end{lemma}

\begin{proof}
Let $g(x) = f(x) - \frac{\beta}{2} \norm{x}^2$, which by assumption is concave. Then,
\begin{align*}
\E[f(X) - f(x)]
&= \E[g(X) - g(x)] + \frac{\beta}{2} \E\left[\norm{X}^2 - \norm{x}^2\right] \\
\tag*{since $g$ is concave}
&\leq \frac{\beta}{2} \E\left[\norm{X}^2 - \norm{x}^2\right] \\
\tag*{since $\E[X] = x$ \,.}
&= \frac{\beta}{2} \E\left[\norm{X - x}^2\right] 
\end{align*}
\end{proof}

\section{Scaling Properties}\label{sec:reg:scaling}

A class of problems is defined by the constraint set\index{constraint set} $K$ and the function class in which the losses lie (see \cref{def:class}) as well 
as constraints on the adversary (stochastic/non-stochastic) or the noise.
Regardless, we hope you agree that simply changing the scale of the coordinates should not affect the achievable regret.
The following proposition describes how the various constants change when the coordinates are scaled.

\begin{proposition}
Let $f \colon K \to [0,1]$ be convex and twice differentiable. \index{differentiable}
Define $g(y) = f(x/\gamma)$ and $J = \{\gamma x \colon x \in K\}$.
The following hold:
\begin{enumerate}
\item $g \colon J \to [0,1]$ is convex and twice differentiable.
\item $g'(y) = f'(x/\gamma) / \gamma$.
\item $g''(y) = f''(x/\gamma) / \gamma^2$.
\item $\diam(J) = \gamma \diam(K)$.
\end{enumerate}
\end{proposition}

From this we see that the product of the Lipschitz constant\index{Lipschitz} and the diameter\index{diameter} is invariant under scaling, as is the ratio of strong convexity and smoothness parameters.
You should always check that various results are compatible with these scaling results in the sense that the regret bound should be invariant to scale if the assumptions
permit scaling. 

\section{Convex Functions are Nearly Lipschitz}\label{sec:regularity:lip}\index{Lipschitz}
Let $f \colon K \to [0,1]$ be a convex function.
The example $K = [0,1]$ and $f(x) = 1 - \sqrt{x}$ shows that such functions are not always Lipschitz.
What is true is that $f$ must be Lipschitz on the interior of $K$ in some sense.
You should start by solving the following exercise:

\begin{exer}\label{ex:reg:lip}
\faStar \quad
Suppose that $f \colon \R^d \to \R \cup \{\infty\}$ is convex.
Show the following:
\begin{enumerate}
\item Suppose that $A \subset \interior(\dom(f))$. Then
\begin{align*}
\lip_A(f) \leq \sup_{x \in A} \sup_{\eta \in \sphere_1} Df(x)[\eta] \,.
\end{align*}
\item Suppose that $A$ is a bounded subset of $\R^d$ and $\dom(f) = \R^d$. Then
\begin{align*}
\lip(f) \leq \sup_{x \notin A} \sup_{\eta \in \sphere_1} Df(x)[\eta] \,.
\end{align*}
\end{enumerate}
\end{exer}

\solution{%
The first follows from the fundamental theorem of calculus and because convex functions are differentiable almost everywhere on their domain.
For the second, let $x \in A$ and $\eta \in \sphere_1$ and $g(t) = f(x + t \eta)$, which is chosen so that
$g'_+(t) = Df(x)[\eta]$ where $g'_+$ is the right-derivative of $g$.
But $g$ is convex and hence its right-derivative is non-decreasing. Moreover, since $A$ is bounded there exists a $t$ such that $x + t \eta \notin A$.
Therefore 
\begin{align*}
\sup_{x \in \R^d} \sup_{\eta \in \sphere_1} Df(x)[\eta] = \sup_{x \notin A} \sup_{\eta \in \sphere_1} Df(x)[\eta] \,.
\end{align*}
The result now follows from the first part.
}

\begin{proposition}\label{prop:lip}
Suppose that $f \in \cF_\pb$ and $r > 0$ and $x + \ball_r \subset K$. Then 
\begin{align*}
\max_{\eta \in \sphere_1} Df(x)[\eta] \leq \frac{1}{r} \,.
\end{align*}
\end{proposition}

\begin{proof}
The assumption that $f$ is convex and bounded in $[0,1]$ on $K$ shows that
for any $\eta \in \sphere_1$,
\begin{align*}
1 &\geq f(x+r \eta) \geq f(x) + D f(x)[r \eta] \geq Df(x)[r \eta] = r Df(x)[\eta]\,.
\end{align*}
Therefore $Df(x)[\eta] \leq \frac{1}{r}$.
\end{proof}

Combining \cref{prop:lip} and the solution to \cref{ex:reg:lip} yields the following:

\begin{corollary}\label{cor:lip}
Let $f \in \cF_\pb$ be convex and suppose that $A$ is convex and $A + \ball_r \subset K$. Then
$\lip_A(f) \leq \frac{1}{r}$.
\end{corollary}

\begin{corollary}\label{cor:lip2}
Suppose that $f \colon K \to [0,1]$ is convex and $\ball_r \subset K$ and $K_\eps = (1-\eps) K$. Then $\lip_{K_\eps}(f) \leq \frac{1}{\eps r}$. 
\end{corollary}

\section{Near-Optimality on the Interior}

The observation that convex functions are Lipschitz on a suitable subset of the interior of $K$ suggests that if we want to restrict our attention
to Lipschitz functions, then we might pretend that the domain of $f$ is not $K$ but rather a subset. This idea is only fruitful because bounded convex
functions are always nearly minimised somewhere in the interior, in the following sense.
Recall the definition of $K_\eps$ from \cref{sec:regularity:minkowski}.

\begin{proposition}\label{prop:shrink}
Let $K$ be a convex body with $\zeros \in \interior(K)$ and $\eps \in (0,1)$ and $f \in \cF_\pb$.
Then
\begin{align*}
\min_{y \in K_\eps} f(y) \leq \inf_{y \in K} f(y) + \eps \,.
\end{align*}
\end{proposition}

\begin{proof}
$K_\eps$ is a closed subset of $\interior(K)$, hence compact. Convex functions are continuous on the interior of their domain, which means 
that $f$ is continuous on $K_\eps \subset \interior(K)$ and hence has a minimiser.
Let $y \in K$. Then $z = (1 - \eps) y \in K_\eps$ and
by convexity, $f(z) \leq (1 - \eps) f(y) + \eps f(\zeros) \leq f(y) + \eps$.
Taking the infimum over all $y \in K$ completes the proof.
\end{proof}

\section{Classical Positions and Rounding}\label{sec:reg:rounding}
\index{convex body!positions of|(}
We make frequent use of certain classical positions of convex bodies. 
\begin{itemize}
\item A convex body $K$ is in John's position if $\ball_1$ is the ellipsoid of largest volume contained in $K$. \index{John's position}
\item A convex body $K$ is in L\"owner's position if $\ball_1$ is the ellipsoid of smallest volume that contains $K$. \index{L\"owner's position}
\item A convex body $K$ is in isotropic position if $\frac{1}{\vol(K)}\int_K xx^\top \d{x} = \id$ and $\int_K x \d{x} = \zeros$. \index{isotropic position}
\end{itemize}
The unifying characteristic of these positions is that for every convex body $K$ there exists an affine map\index{affine!map} $T \colon \R^d \to \R^d$ such
that the image $T(K)$ is in the relevant position:

\begin{theorem}\label{thm:positions}
Given $\text{X} \in \{\text{John's}, \text{L\"owner's}, \text{isotropic}\}$ and convex body $K$, there exists an affine transformation $T \colon \R^d \to \R^d$
such that $T(K)$ is in X position.
\end{theorem}

Some of our analysis depends on the constraint set $K$ being well rounded. By this we mean that \index{rounded|textbf}\index{constraint set}
\begin{align*}
\ball_r \subset K \subset \ball_R\,,
\end{align*}
where $R/r$ is not too large.
The following shows that convex bodies in John's position are well rounded:

\begin{theorem}\label{thm:john}
Suppose that $K$ is in John's position. Then
\begin{align*}
\ball_1 \subset K \subset \ball_d\,.
\end{align*}
\end{theorem}

\cref{thm:john} is an immediate consequence of John's theorem\index{John's theorem} \citep[Remark 2.1.17]{ASG15}.
Combining \cref{thm:john} with \cref{thm:positions} shows that for any convex body $K$ there exists an affine map $T$ such that
\begin{align}
\ball_1 \subset T(K) \subset \ball_d\,.
\label{eq:rounded}
\end{align}
This is less constructive than we would like because even when $K$
is represented by a separation oracle,\index{separation oracle} there is no known procedure for efficiently computing John's position.
In a moment we discuss how to algorithmically find an affine mapping $T$ such that \cref{eq:rounded} holds approximately. First though, we explain how such a mapping can be used. 
Any affine $T$ for which \cref{eq:rounded} holds must be invertible. 
A learning algorithm designed for rounded constraint sets can be used on arbitrary constraint sets by first
finding a $T$ such that \cref{eq:rounded} approximately holds. The learner is then instantiated with $T(K)$ as a constraint set and proposes actions $(X_t')_{t=1}^n$
with $X_t' \in T(K)$.
The response is $Y_t = f_t(T^{-1}(X_t')) + \eps_t$. Letting $g_t = f_t \circ T^{-1}$, we have
\begin{align*}
\Reg_n 
&= \sup_{x \in K} \sum_{t=1}^n (f_t(X_t) - f_t(x)) \\
&= \sup_{x' \in T(K)} \sum_{t=1}^n (f_t(T^{-1}(X_t')) - f_t(T^{-1}(x'))) \\
&= \sup_{x' \in T(K)} \sum_{t=1}^n (g_t(X_t') - g_t(x'))\,.
\end{align*}
The Lipschitz and smoothness properties of $g_t$ may be different from $f_t$, but if $f_t$ is bounded on $K$, then $g_t$ is similarly bounded on $T(K)$.
Therefore when assuming losses are in $\sF_\pb$ and you are indifferent to computation cost you can assume that $\ball_1 \subset K \subset \ball_{d}$.
Next we discuss what is possible using a computationally efficient algorithm.

\subsubsection*{Rounding Algorithms}
In order to implement the translation we need a procedure for finding $T$. Let $\nu_K$ be the uniform probability measure\index{uniform measure} on $K$ and
$\mu = \int_K x \d{\nu_K}(x)$ be the centre of mass and $\Sigma = \int_K (x - \mu)(x - \mu)^\top \d{\nu_K}(x)$ the moment of inertia of $K$. 
Let $\ISO_K(x) = \Sigma^{-1/2} (x - \mu)$ and $J = \ISO_K(K)$.\label{page:isotropic}
A simple calculation shows that $\int_J x \d{\nu_J}(x) = \zeros$ and $\int_J xx^\top \d{\nu_J}(x) = \id$. That is, $J$ is in isotropic position.

\begin{theorem}[Theorem 4.1, \citealt{kannan1995isoperimetric}]\label{thm:isotropic}
Suppose that $J$ is in isotropic position. Then
\begin{align*}
\ball_1 \subset J \subset \ball_{1+d}\,.
\end{align*}
\end{theorem}

\begin{remark}
Be careful. Our definition of isotropic position is standard in probability theory while in geometric analysis it is normal to say that $K$ is in isotropic position
if $\int_K x \d{x} = \zeros$, $\vol(K) = 1$ and $\int_K xx^\top \d{x} = L_K \id$ for some $L_K$. The recent resolution to the `slicing conjecture'
shows that $L_K$ is upper- and lower-bound by universal constant, which means that up to constant scaling factors the two definitions of isotropic position are the same
\citep{klartag2024affirmative}.
\index{slicing conjecture}
\end{remark}

Provided that $K$ is suitably represented, then there exist algorithms that find an affine map\index{affine!map} $T$ in polynomial time such that $J = T(K)$ is close enough to isotropic
position that $\ball_1 \subset K \subset \ball_{2d}$. The procedure is based on estimating the centre of mass and moment of inertia of $K$ using uniform samples
and estimating the corresponding affine map $T$ defined above \citep{lovasz2006simulated}.

\index{convex body!positions of|)}

\section{Extension}\label{sec:reg:extension} \index{extension|(}
Some of the algorithms presented in this book are only defined for unconstrained problems where $K = \R^d$.
Furthermore, techniques designed for handling constraints such as self-concordant barriers introduce complexity and dimension-dependent constants into the analysis.
One way to mitigate these problems is to use an algorithm designed for unconstrained bandit convex optimisation on an extension of the loss function(s).\index{setting!unconstrained}
An extension of a convex function $f \colon \R^d \to \R \cup\{\infty\}$ with $K = \dom(f)$ is another convex function $e \colon \R^d \to \R$ such that
\begin{align*}
e(x) = f(x) \text{ for all } x \in K\,.
\end{align*}
Sometimes no such extension exists. For example, the function defined by
\begin{align*}
f(x) = \begin{cases} 
1-\sqrt{x} & \text{if } x \geq 0 \\
\infty & \text{if } x < 0
\end{cases}
\end{align*}
cannot be extended to a convex function with domain $\R$.
When $f$ is Lipschitz on its domain then an extension to $\R^d$ is always possible.

\begin{proposition}\label{prop:extension}
Suppose that $f \colon \R^d \to \R \cup \{\infty\}$ is convex and $\lip(f) < \infty$. Then there exists a convex function $e \colon \R^d \to \R$
such that
\begin{enumerate}
\item $e(x) = f(x)$ for all $x \in \dom(f)$; and
\item $\lip(e) = \lip(f)$.
\end{enumerate}
\end{proposition}

\begin{proof}[\Proofskippy]
To keep things simple, let us assume that $\dom(f)$ has nonempty interior (but see \cref{ex:lip-extension}). 
The idea is to define $e$ as the supremum of all tangent hyperplanes to $f$ in $\interior(\dom(f))$. 
Define
\begin{align*}
e(x) = \sup_{y \in \interior(\dom(f))} (f(y) + Df(y)[x-y])\,.
\end{align*}
Note that no convex extension can take a smaller value than this by convexity, so this $e$ is the minimal extension.
We leave it as an exercise to establish the claimed properties of $e$.
\end{proof}

\begin{exer}(\skippy)\label{ex:lip-extension}
\faStar \faBook \quad
Prove \cref{prop:extension}. We suggest you start by assuming $\dom(f)$ has nonempty interior. 
In case $\dom(f)$ has no interior you should first extend $f$ to the affine hull\index{affine!hull} of the relative interior\index{relative interior} and then extend the extension to the whole space.
You may find it useful to use the fact that for $y \in \interior(\dom(f))$, $Df(y)[h] = \sup_{g \in \partial f(y)} \ip{g, h}$.
\end{exer}

The extension in \cref{prop:extension} has the limitation that it cannot be evaluated in an unbiased way with stochastic oracle access to $f$.
The reason is that to even  approximate $e$ at some point $x \notin \dom(f)$ you need to solve an optimisation problem that may require you to evaluate $f$ at many points in $K$.
The next proposition shows there exists an extension that can be evaluated at $x \notin \dom(f)$ using only a single evaluation of $f$.

\begin{proposition}\label{prop:reg:bandit-extension}
Suppose that $\ball_r \subset K$ and $f : \R^d \to \R \cup \{\infty\}$ is a convex function 
such that $K \subset \interior(\dom(f))$, $f(K) \subset [0,1]$ and $\lip(f) < \infty$.
Let $m \in \R_+$ be such that $Df(y)[y] - f(y) \leq m$ for all $y \in K$.
Let $\pi$ be the Minkowski functional\index{Minkowski functional} of $K$
and
\begin{align*}
e(x) = \max(1, \pi(x)) f\left(\frac{x}{\max(1, \pi(x))}\right) + m (\max(1, \pi(x)) - 1) \,.
\end{align*}
The function $e$ satisfies the following:
\begin{enumerate}
\item $e(x) = f(x)$ for all $x \in K$. \label{prop:reg:bandit-extension:ext}
\item $e$ is convex. \label{prop:reg:bandit-extension:cvx}
\item $\lip(e) \leq \frac{m}{r} + \frac{1}{r} + \lip(f)$. \label{prop:reg:bandit-extension:lip}
\item For all $x \notin K$, $e(x/\pi(x)) \leq e(x)$. \label{prop:reg:bandit-extension:min}
\end{enumerate}
\end{proposition}

\begin{remark}
The condition that $K \subset \interior(\dom(f))$ ensures that $Df(y)[y]$ is real-valued for all $y \in K$.
You could replace this condition with $\dom(f) = K$ and define $Df(y)[y]$ for $y \in \partial K$ via the extension in \cref{prop:extension}.
Note, it can happen in this case that $Df(y)[y] \neq -Df(y)[-y]$.
\end{remark}

Compared to the extension in \cref{prop:extension}, the extension above has the drawback that $\lip(e)$ can be much larger than $\lip(f)$.
More positively, however, the extension above can be evaluated at $x$ by computing $\pi(x)$
and evaluating $f$ at $x/\max(1, \pi(x))$, which means the extension can be evaluated at $x$ using a single query to $f$.

\begin{remark}
Let $x \in \partial K$.
By \cref{lem:reg:mink}\ref{lem:reg:mink:hom}, the Minkowski functional is homogeneous, which means that for $t \geq 1$
$e(tx) = t f(x) + m (t - 1)$ is a linear function.
So $e$ is defined outside of $K$ by glueing together rays emanating from points $x \in \partial K$.
The most challenging part of the proof of \cref{prop:reg:bandit-extension} is establishing convexity.
\end{remark}

\begin{proof}[Proof of \cref{prop:reg:bandit-extension}]
Abbreviate $\pip(x) = \max(1, \pi(x))$.
Recall that for $x \in \interior(\dom(f))$, $h \mapsto Df(x)[h]$ is convex and positively homogeneous, and hence subadditive. \todot{citation}
Part~\ref{prop:reg:bandit-extension:ext} follows immediately from the fact that for $x \in K$, $\pi(x) \leq 1$ and therefore $\pip(x) = 1$.
Moving to part~\ref{prop:reg:bandit-extension:cvx},
define $g(z, \lambda) = \lambda f(z/\lambda)$, which is called the perspective of $f$ and according to \citet[\S2.3.3]{BV04} is jointly convex on $\R^d \times (0,\infty)$.
\index{perspective}
Let $z \in \R^d$ and $\lambda \geq \pip(z)$. Then
\begin{align}
g(z, \lambda) 
&\explana\geq g(z, \pip(z)) + \frac{\d{g(z,\theta)}}{\d{\theta}}\Big|_{\theta = \pip(z)} \left(\lambda - \pip(z)\right) \nonumber \\
&\explana= g(z, \pip(z)) + \left(f\left(\frac{z}{\pip(z)}\right) + Df\left(\frac{z}{\pip(z)}\right)\left[\frac{-z}{\pip(z)}\right] \right) \left(\lambda - \pip(z)\right) \nonumber \\
&\explana\geq g(z, \pip(z)) - m \left(\lambda - \pip(z)\right)\,,
\label{eq:ext:1}
\end{align}
where in \explanr{} the derivative is the right-derivative and the inequality follows from convexity of $g$, 
\explanr{} follows by the chain rule and the definition of $g$ and 
\explanr{} by the assumptions on $m$ in the proposition statement so that with $w = z / \pip(z) \in K$ by subadditivity of $Df(w)[\cdot]$, 
$f(w) + Df(w)[-w] \geq f(w) - Df(w)[w] \geq -m$.
Let $x, y \in \R^d$ and $p \in (0,1)$ and $z = px + (1-p)y$.
By definition,
\begin{align*}
e(z) 
&= \pip(z) f\left(\frac{z}{\pip(z)}\right) + m(\pip(z) - 1) \\
&= g(z, \pip(z)) + m(\pip(z) - 1) \\
&\explana\leq g(z, p\pip(x) + (1-p)\pip(y)) + m\left[p \pip(x) + (1-p) \pip(y) - 1\right] \\
&\explana\leq p g(x, \pip(x)) + (1-p) g(y, \pip(y)) + m \left[p\pip(x) + (1-p)\pip(y) - 1\right] \\
&= pe(x) + (1-p)e(y) \,,
\end{align*}
where \explanr{} follows from \cref{eq:ext:1} with $\lambda = p\pip(x)+(1-p)\pip(y) \geq \pip(z)$ by convexity of $\pip$, and
\explanr{} follows from joint convexity of $g$ and because $z = px + (1-p)y$. 
Therefore $e$ is convex.
Next we prove part~\ref{prop:reg:bandit-extension:lip}. Let $h \in \sphere_1$ and $x \notin K$, which means that $\pip(x) = \pi(x) > 1$.
By \cref{lem:reg:mink}\ref{lem:reg:mink:diff}, $D\pi(x)[h] = \ip{\theta, h}$ for some $\theta \in K^\circ$. 
Because polarity reverses inclusion (\cref{prop:polar-dual}) and $\ball_r \subset K$, $K^\circ \subset \ball_{1/r}$ and therefore $\norm{\theta} \leq 1/r$.
Letting $w = x/\pi(x) \in K$,
\begin{align}
D e(x)[h]
&= \ip{\theta, h} \left(m  + f(w)\right) + Df(w) \left[h - \ip{\theta, h} w\right] \nonumber \\
&\leq \ip{\theta, h} \left(m + f(w) \right) + Df(w)\left[-\ip{\theta, h} w\right] + \lip(f) \,, 
\label{eq:ext:2}
\end{align}
where the inequality follows because $Df(w)[\cdot]$ is subadditive and $Df(w)[h] \leq \lip(f)$ since $w \in \interior(\dom(f))$.
When $\ip{\theta, h} \geq 0$; then by convexity
\begin{align*}
\ip{\theta, h} \left(m + f(w)\right) + Df(w)\left[-\ip{\theta, h} w\right] 
&= \ip{\theta, h} \left( m + f(w) + Df(w)[-w]\right) \\
&\leq \ip{\theta, h} \left(m + f(\zeros)\right) 
\leq \frac{m+1}{r} \,.
\end{align*}
Alternatively, if $\ip{\theta, h} < 0$, then by the assumption that $Df(w)[w] - f(w) \leq m$,
\begin{align*}
\ip{\theta, h} \left(m + f(w)\right) + Df(w)\left[-\ip{\theta, h} w\right] 
&= \ip{\theta, h} \left(m + f(w) - Df(w)[w]\right) \\
&\leq 0\,. 
\end{align*}
Combining the previous two displays with \cref{eq:ext:2} shows that $De(x)[h] \leq (m+1)/r + \lip(f)$ for all $h \in \sphere_1$ and $x \notin K$.
The claim now follows from \cref{ex:reg:lip} and the fact that $K$ is bounded.
For part~\ref{prop:reg:bandit-extension:min}, since $x \notin K$ we have
\begin{align*}
e(x) &= \pi(x) f\left(\frac{x}{\pi(x)}\right) + m (\pi(x) - 1) 
\geq f\left(\frac{x}{\pi(x)}\right) = e\left(\frac{x}{\pi(x)}\right)\,,
\end{align*}
where we used the fact that $x/\pi(x) \in K$ so that $e(x/\pi(x)) = f(x/\pi(x))$.
\end{proof}

The next proposition uses the results from \cref{sec:regularity:lip} to refine the Lipschitz constant of the extension in
\cref{prop:reg:bandit-extension} when $f$ is extended from a suitable subset of its domain.

\begin{proposition}\label{prop:reg:bandit-extension-eps}
Suppose that $f \in \cF_\pb$ and $\ball_r \subset K$.
Let $\pi$ be the Minkowski functional of $K$ and
$\eps \in (0,1)$, and let $\pip(x) = \max(1, \pi(x)/(1-\eps))$ and 
\begin{align*}
e(x) =   \pip(x) f\left(\frac{x}{\pip(x)}\right) + \frac{1-\eps}{\eps} \left(\pip(x) - 1\right) \,.
\end{align*}
Then, the following hold:
\begin{enumerate}
\item $e(x) = f(x)$ for all $x \in K_\eps = \{x \in K \colon \pi(x) \leq 1 - \eps\}$. \label{prop:reg:bandit-extension-eps:equal}
\item $e$ is convex. \label{prop:reg:bandit-extension-eps:cvx}
\item $\lip(e) \leq \frac{2}{\eps(1-\eps) r}$. \label{prop:reg:bandit-extension-eps:lip}
\item For all $x \notin K_\eps$, $e(x/\pip(x)) \leq e(x)$. \label{prop:reg:bandit-extension-eps:leq}
\end{enumerate}
\end{proposition}

\begin{proof}
Let $y \in K_\eps$ and $z = y/(1 - \eps) \in K$, which means that $z - y = \frac{\eps}{1-\eps} y$.
Combining this with the fact that $f \in \cF_{\pb}$ yields
\begin{align*}
1 
\tag*{$f \in \cF_{\pb}$}
&\geq f(z) \\
\tag*{$f$ convex}
&\geq f(y) + Df(y)[z-y] \\
\tag*{def.\ of $z$}
&= f(y) + \frac{\eps}{1-\eps} Df(y)[y] \\
\tag*{$f \in \cF_{\pb}$}
&\geq \frac{\eps}{1-\eps} \left(Df(y)[y] - f(y)\right) \,. 
\end{align*}
Rearranging shows that for all $y \in K_\eps$, $Df(y)[y] - f(y) \leq \frac{1-\eps}{\eps} \triangleq m$.
The claim now follows by applying \cref{prop:reg:bandit-extension} to $K_\eps$, which by definition has $B_{(1-\eps)r} \subset K_\eps$.
Hence, by \cref{cor:lip2}, $\lip_{K_\eps}(f) \leq \frac{1}{\eps r}$,
which by \cref{prop:reg:bandit-extension} means
the Lipschitz constant $\lip(e)$ is bounded by
\begin{align*}
\lip(e)
&\leq \frac{m+1}{(1-\eps) r} + \frac{1}{(1-\eps)r} + \lip_{K_\eps}(f) \\
&\leq \frac{\frac{1-\eps}{\eps} + 1}{(1-\eps)r} + \frac{1}{(1-\eps)r} + \frac{1}{r \eps} \\
&= \frac{2}{\eps(1-\eps) r} \,.
\qedhere 
\end{align*}
\end{proof}

\index{extension|)}

\section{Smoothing}\label{sec:reg:smooth} \index{smoothing}

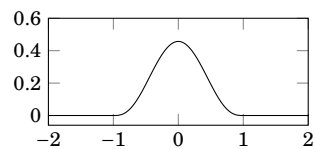
\begin{wrapfigure}[4]{r}{4cm}
\vspace{-0.5cm}
\scriptsize
\begin{tikzpicture}
\begin{axis}[xmin=-2,xmax=2,height=3cm,width=5cm,ymax=0.6]
\addplot[mark=none,smooth,samples=500] {48 / 105 * pow(max(0,1 - x * x), 3)};
\end{axis}
\end{tikzpicture}
\caption{$\phi$ in dimension one.}\label{fig:phi}
\commentAlt{A plot of phi in dimension 1 showing that it vanishes outside of the interval between minus one and one and looks like a smooth bump inside the interval.}
\end{wrapfigure}
Let $\phi \colon \R^d \to \R$ be the twice-differentiable function given by \index{differentiable}
\begin{align*}
\phi(x) &= \frac{1}{C} \left(1 - \norm{x}^2\right)^3 \sind_{\ball_1}(x) \quad \text{with}\\
C &= \int_{\ball_1} \left(1 - \norm{x}^2\right)^3 \d{x} \,.
\end{align*}
Note that $\phi$ is the density of a probability measure on $\R^d$ that is supported on $\ball_1$ (see \cref{fig:phi}).
Given $\eps > 0$, let 
\begin{align*}
\phi_\eps(x) = \eps^{-d} \phi(x/\eps)\,, 
\end{align*}
which by a change of measure is also a probability density, this time supported on $\ball_\eps$.

\newcommand{\closure}{\operatorname{cl}}
\begin{proposition}\label{prop:reg:smooth}
Suppose that $f \colon \R^d \to \R \cup \{\infty\}$ is convex and $\lip(f) < \infty$ and let $g = f * \phi_\eps$ with $*$ the convolution, which is defined
on $\dom(g) = \{x \in \R^d \colon x + \ball_\eps \subset \closure(\dom(f))\}$. 
Then the following hold:
\begin{enumerate}
\item $g$ is twice differentiable on $\interior(\dom(g))$. \label{prop:reg:smooth:diff}
\item $\lip(g) \leq \lip(f)$. \label{prop:reg:smooth:lip}
\item $g$ is smooth: $\norm{g''(x)} \leq \frac{(d+1)(d+6) \lip(f)}{\eps}$ for all $x \in \interior(\dom(g))$. \label{prop:reg:smooth:smooth}
\item $\max_{x \in \dom(g)} |f(x) - g(x)| \leq \eps \lip(f)$. \label{prop:reg:smooth:close}
\end{enumerate}
\end{proposition}

\begin{proof}
Part~\ref{prop:reg:smooth:diff} follows by writing out the definition, making a change of variables and exchanging limits and integrals. This is a good technical exercise.
Part~\ref{prop:reg:smooth:lip} is also left as an (easy) exercise.
For part~\ref{prop:reg:smooth:smooth}, the constant $C$ can be calculated by integrating in polar coordinates:
\begin{align*}
C 
&= \int_{\ball_1} \left(1 - \norm{x}^2\right)^3 \d{x} \\
\tag*{\cref{prop:radial-int}}
&= d \vol(\ball_1) \int_0^1 r^{d-1} \left(1 - r^2\right)^3 \d{r} \\
&= \frac{48\vol(\ball_1)}{(d+2)(d+4)(d+6)} \,.
\end{align*}
By convexity of the spectral norm and naive calculation,
\begin{align*}
\norm{g''(x)} 
&= \norm{\int_{\ball_\eps} f(x + u) \phi_\eps''(u) \d{u}} \\
&\explana= \norm{\int_{\ball_\eps} (f(x + u) - f(x)) \phi_\eps''(u) \d{u}} \\
&\explana\leq \eps \lip(f) \int_{\ball_\eps} \norm{\phi_\eps''(u)} \d{u} \\
&\explana= \frac{\lip(f)}{\eps} \int_{\ball_1} \norm{\phi''(u)} \d{u} \\
&\explana= \frac{\lip(f)}{C \eps} \int_{\ball_1} \norm{24 uu^\top (1 - \norm{u}^2) - 6 \id(1 - \norm{u}^2)^2} \d{u} \\
&\explana\leq \frac{\lip(f)}{C \eps} \int_{\ball_1} \left(\norm{24 uu^\top (1 - \norm{u}^2)} + 6 \norm{\id(1 - \norm{u}^2)^2}\right) \d{u} \\
&\explana= \frac{d \vol(\ball_1) \lip(f)}{C \eps} \int_0^1 r^{d-1} \left[24 r^2(1-r^2) + 6 (1 - r^2)^2\right] \d{r} \\
&= \frac{(d+1)(d+6)\lip(f)}{\eps}\,,
\end{align*}
where \explanr{} follows because $\int_{\ball_\eps} \phi_\eps''(u) \d{u} = \zeros$;
\explanr{} since $f$ is Lipschitz and the spectral norm is convex (or triangle inequality);
\explanr{} and \explanr{} by a change of measure and differentiating;
\explanr{} by the triangle inequality; and \explanr{} by \cref{prop:radial-int} and because $\snorm{uu^\top} = \norm{u}^2$.
For part~\ref{prop:reg:smooth:close}, since $f$ is Lipschitz,
\begin{align*}
\left|g(x) - f(x)\right| 
&= \left|\int_{\ball_\eps} \left(f(x+u) - f(x) \right) \phi_\eps(u) \d{u}\right|  \\
&\leq \lip(f) \int_{\ball_\eps} \norm{u} \phi_\eps(u) \d{u} \\
&\leq \eps \lip(f) \,.
\qedhere
\end{align*}
\end{proof}

\begin{exer}
\faStar \quad
Prove \cref{prop:reg:smooth}\ref{prop:reg:smooth:diff} and \ref{prop:reg:smooth:lip}.
\end{exer}

\section{Computation}
\label{sec:reg:compute}

There are a variety of standard operations that are components in (bandit) convex optimisation algorithms;
for example, projections and positioning a convex body $K$ into isotropic/John's position.

\paragraph{Standard operations}
We are interested in the following operations:
\begin{itemize}
\item $\MEM_K$ is the membership oracle: $\MEM_K(x) = \sind_K(x)$. \label{page:membership} 
\item $\SEP_K$ is a separation oracle: $\SEP_K(x) = \bot$ if $x \in K$ and otherwise $\SEP_K(x) = H$ for some half-space $H$ with $K \subset H$. \label{page:separation}\index{separation oracle}
\item $\LIN_K$ is the linear optimisation oracle: $\LIN_K(c) = \argmin_{x \in K} \ip{c, x}$. 
\item $\CVX_K$ is the function with $\CVX_K(f) = \argmin_{x \in K} f(x)$.
\item $\PROJ_{K,N}$ is the function $\PROJ_{K,N}(y) = \argmin_{x \in K} N(x - y)$ where $N$ is a norm. \index{projection}
\item $\ISO_K$ is the affine map\index{affine!map} such that $\{\ISO_K(x) \colon x \in K\}$ is istropic.
\item $\JOHN_K$ is the affine map such that $\{\JOHN_K(x) \colon x \in K\}$ is in John's position. \label{page:john}
\item $\SAMP_K$ is the oracle that returns a point sampled from the uniform distribution on $K$.\index{uniform measure}
\item $\GRAD_f$ returns the gradient of a function $f \colon \R^d \to \R$.
\end{itemize}
\cref{tab:comp} provides complexity bounds for computing one oracle from others. The bounds ignore logarithmic factors and are given in terms of
the number of arithmetic operations as well as calls to other oracles. For example, \cref{tab:comp} claims that linear optimisation can be computed 
in $\tilde O(d^3)$ arithmetic operations and $\tilde O(d)$ calls to a separation oracle for $K$.
More importantly, we are only claiming the relevant quantity can be computed \textit{approximately}. 
Moreover, the oracles used as inputs are permitted to be approximate as well. We badly want to avoid handling approximation errors for computations in this book.
We will assume exact computation in our analysis and leave it to you to carefully consider the approximation error if this concerns you. 
For some of the oracles it is not even obvious what metric should be used to define the approximation error. That too, we leave to you to figure out. Usually
the reference in \cref{tab:comp} contains what you need to know.

\begin{exer}
\faStar \quad Prove the complexity bound for all entries in \cref{tab:comp} without a reference. 
\end{exer}

\begin{table}[h!]
\caption{Computation costs for standard operations. In rows where $m$ appears we assume that $K = \{x \colon Ax \leq b\}$ with $A \in \R^{m \times d}$. In rows where $v$ appears we assume
that $K = \conv(x_1,\ldots,x_v)$. Since $K$ is assumed to be a convex body, $m = \Omega(d)$ and $v = \Omega(d)$.
} \label{tab:comp}
\fontsize{7.5}{8}\selectfont \centering \setlength\tabcolsep{5pt}
\renewcommand{\arraystretch}{1.9}
\begin{longtable}{|lll|}
\hline 
\textsc{op.} & \textsc{complexity} & \textsc{reference} \\ \hline
$\MEM_K$ & $m d$ & \textendash \\
$\MEM_K$ & $v^{2.37}$$\dagger$ & \citealt{jiang2020faster} \\
$\MEM_{K^\circ}$ & $1 + \LIN_K$  & \textendash \\
$\SEP_K$ & $md$ & \textendash \\
$\SEP_K$ & $d \MEM_K$ & \citealt{lee2018efficient} \\
$\LIN_K$ & $d^3 + d^2 \MEM_K$ & \citealt{lee2015faster}, \citealt{lee2018efficient} \\
$\LIN_K$ & $d^3 + d \SEP_K$ &  \citealt{lee2015faster} \\
$\LIN_K$ & $m^{2.37}$$\dagger$ & \citealt{jiang2020faster} \\
$\LIN_K$ & $vd$ & \textendash \\
$\CVX_K$ & $d^3 + d \GRAD_f + d^2 \MEM_K$ &  \citealt{lee2015faster}, \citealt{lee2018efficient} \\
$\CVX_K$ &  $d^3 + d \GRAD_f + d \SEP_K$ &  \citealt{lee2015faster} \\
$\CVX_K$ & $d^3 + d \GRAD_f + md^2$ & \citealt{lee2015faster} \\
$\PROJ_{K,N}$ & $d^3 + d \GRAD_N + d^2 \MEM_K$ & \textendash \\
$\PROJ_{K,N}$ & $d^3 + d \GRAD_N + d \SEP_K$ & \textendash \\
$\PROJ_{K,N}$ & $d^3 + d \GRAD_N + md^2$ & \textendash \\
$\ISO_K$ & $d^4 + d^4 \MEM_K$ &  \citealt{lovasz2006simulated} \\
$\ISO_K$ & $d^3 + d \SAMP_K$ &  \citealt{lovasz2006simulated} \\   
$\JOHN_K$ & $m^{3.5}$  & \citealt{KT93}  \\
$\JOHN_{K^\circ}$ & $v^{3.5}$ & \citealt{KT93} \\
\hline
\end{longtable}

\vspace{0.1cm}
\flushleft $\dagger$ these bounds would improve to $m^{2.055}$ and $v^{2.055}$ if matrix multiplication algorithms improved to their theoretical limits \citep{jiang2020faster}. 
\end{table}

\section{Notes}

\begin{enumeratenotes}
\item Versions of some or all the properties used here have been exploited in a similar fashion by \cite{FK05,BEL16,Lat20-cvx} and others.

\item Occasionally it would be convenient to be able to extend $\beta$-smooth functions while preserving $\beta$-smoothness to all of $\R^d$.
Curiously, this is not always possible \citep{drori2018properties}.

\item The extension in \cref{prop:reg:bandit-extension} is due to \cite{LFMV24}. A related extension was proposed by \index{extension}
\cite{mhammedi2022efficient}, who also use an extension based on the `projection' $x/\pi(x)$ but assume knowledge of the gradient of $f$ at this point. 

\end{enumeratenotes}

\chapter[Bisection in One Dimension]{Bisection in One Dimension\copynotice}\label{chap:bisection}

We start with a simple but instructive algorithm for the one-dimensional stochastic setting.\index{setting!stochastic}
The next assumption is considered global throughout the chapter:

\begin{assumption}\label{ass:bisect}
The following hold:
\begin{enumerate}
\item $d = 1$ and $K$ is a nonempty interval; 
\item the setting is stochastic: $f_t = f$ for all $t$; and
\item the loss function is Lipschitz: $f \in \cF_\pl$.
\end{enumerate}
\end{assumption} 

Like many algorithms for convex bandits, the bisection method is based on a classical technique for deterministic convex optimisation.
The algorithm in this chapter only works in the stochastic one-dimensional setting but has the advantages that it can be implemented trivially and is nearly
minimax optimal. The ideas are also quite instructive and highlight some of the challenges when moving from deterministic to noisy zeroth-order optimisation.
The main theoretical result is a proof that under \cref{ass:bisect} the regret of \cref{alg:bisection:full} is bounded with high probability by $\tilde O(\sqrt{n})$.

\section{Bisection Method without Noise} \index{bisection method!deterministic|(}
We start by considering the noise-free setting, which illustrates the main idea.
The bisection method for deterministic zeroth-order convex optimisation is very simple.

\begin{algorithm}[h!]
\begin{algcontents}
\begin{lstlisting}
let $K_1 = K$
for $k = 1$ to $\infty$:
  let $x = \min K_k$ and $y = \max K_k$ 
  let $x_0 = \frac{2}{3} x + \frac{1}{3} y$, $x_1 = \frac{1}{3} x + \frac{2}{3} y$
  if $f(x_1) \geq f(x_0)$: then $K_{k+1} = [x,x_1]$
  else: $K_{k+1} = [x_0, y]$
\end{lstlisting}
\caption{Bisection method without noise}\label{alg:bisection:det}
\end{algcontents}
\end{algorithm}

\begin{theorem}
Let $(K_k)_{k=1}^\infty$ be the sequence of sets produced by \cref{alg:bisection:det}.
Then
\begin{align*}
\max_{x \in K_k} f(x) \leq \min_{y \in K} f(y) + \left(\frac{2}{3}\right)^{k-1} \vol(K) \text{ for all } k \geq 1 \,,
\end{align*}
where $\vol(K)$ is the width of the interval $K$.
\end{theorem}

\begin{proof}
Suppose that $x \in K_k$ and $x \notin K_{k+1}$. By convexity you immediately have that
$f(x) \geq \min_{y \in K_{k+1}} f(y)$. Therefore by induction, $\min_{x \in K_k} f(x) = \min_{x \in K} f(x)$ for all $k$.
By construction of the algorithm, $\vol(K_k) = (2/3)^{k-1} \vol(K)$. Since $f$ is Lipschitz by assumption, it follows that
\begin{align*}
\max_{x \in K_k} f(x) 
&\leq \max_{x \in K_k} \min_{y \in K_k} \left(f(y) + |x - y|\right)  \\
&\leq \min_{y \in K_k} f(y) + \vol(K_k) \\
&= \min_{y \in K} f(y) + \left(\frac{2}{3}\right)^{k-1} \vol(K)\,,
\end{align*}
which completes the proof.
\end{proof}
\index{bisection method!deterministic|)}

\section{Bisection Method with Noise}
\index{bisection method!stochastic|(}

\begin{wrapfigure}[12]{r}{4cm}
\vspace{-0.5cm}
\includegraphics[width=3.6cm]{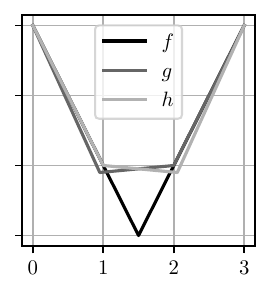}
\caption{Three convex functions}\label{fig:bisection}
\commentAlt{A plot of three piecewise linear convex functions on the interval between zero and three. 
All three have value 3 at the ends of the interval. Function f passes through one and a half, zero. Function g passes through the point nine, point nine and two, one. Function h
passes through one, one and two point one, point nine.}
\end{wrapfigure}
The generalisation of the bisection method to noisy optimisation is surprisingly subtle.
While \cref{alg:bisection:det} divides the current interval into three blocks, in the noisy setting it turns out that four blocks are necessary.
The situation is best illustrated by the example in \cref{fig:bisection}.
Suppose you have noisy (and therefore only approximate) estimates of the loss at all of $x \in \{0,1,2,3\}$. Notice
how all three convex functions $f, g$ and $h$ have very similar values at these points but the minimiser could be in any of $(0,1)$, $(1,2)$ or $(2,3)$.
Hence it will take many samples to identify which function is the truth.
Even worse, if the real function is $f$, then you are paying considerable regret while trying to identify the region where the minimiser lies.
The example illustrates the problem of exploring efficiently. A good exploration strategy will ensure that if the regret is large, then the information
gain about the identity/location of a near-minimiser is also large. The exploration strategy in \cref{fig:bisection} is not good.
The example also illustrates the challenges of generalising methods designed for deterministic zeroth-order optimisation to stochastic zeroth-order optimisation.
Fundamentally the problem is one of stability. \cref{alg:bisection:det} is not a stable algorithm because small perturbations of its observations can
dramatically change its behaviour.

We decompose the bisection method for stochastic convex optimisation into two algorithms.
The first accepts as input an interval and interacts with the loss for a number of rounds. Eventually it outputs
a new interval such that with high probability all of the following hold:
\begin{itemize}
\item The minimiser of the loss is contained in the output interval.
\item The new interval is three quarters as large as the input interval.
\item The regret suffered during the interaction is controlled.
\end{itemize}

\begin{algorithm}[h!]
\begin{algcontents}
\begin{lstlisting}
def $\BISECT$($K = [x,y]$, $n$, $\delta \in (0,1)$):
  $x_0 = \frac{3}{4} x + \frac{1}{4} y$, $x_1 = \frac{1}{2} x + \frac{1}{2} y$, $x_2 = \frac{1}{4} x + \frac{3}{4} y$
  for $t = 1$ to $n$:
    $c_t = \sqrt{\frac{24}{t} \log\left(\frac{4n}{3\delta}\right)}$
    let $X_t = x_{t \bmod 3}$ and observe $Y_t = f(X_t) + \eps_t$
    if $t \equiv 0 \mod 3$:
      let $\hat f_t(x_k) = \frac{3}{t} \sum_{u=1}^t \sind(u \equiv k\! \mod 3) Y_u$ with $k \in \{0,1,2\}$
      if $\hat f_t(x_2) - \hat f_t(x_1) \geq c_t$: return $[x, x_2]$
      if $\hat f_t(x_0) - \hat f_t(x_1) \geq c_t$: return $[x_0, y]$
  return $[x, y]$
\end{lstlisting}
\caption{Bisection episode}\label{alg:bisection}
\end{algcontents}
\end{algorithm}

\FloatBarrier

\begin{proposition}\label{prop:bisection}
Let $[z,w]$ be the interval returned by \cref{alg:bisection} with inputs $K$, $n$ and $\delta \in (0,1)$, and let
\begin{align*}
\Delta = \frac{1}{3}\left[f(x_0) + f(x_1) + f(x_2)\right] - \min_{x \in K} f(x)\,.
\end{align*}
Suppose that $\Delta > 0$. Then, with probability at least $1 - \delta$ the following both hold:
\begin{enumerate}
\item The interval $[z,w]$ returned by the algorithm satisfies  \label{prop:bisection:opt}
\begin{align*}
\min_{x \in [z,w]} f(x) = \min_{x \in K} f(x).
\end{align*}
\item The number of queries to the zeroth-order oracle is at most \label{prop:bisection:n} 
\begin{align*}
3 + \frac{384}{\Delta^2} \log\left(\frac{4n}{3\delta}\right) \,.
\end{align*}
\end{enumerate}
\end{proposition}

\begin{proof}
By convexity, $\max(f(x_0), f(x_2)) \geq f(x_1)$. Assume without loss of generality for the remainder of the proof that $f(x_2) \geq f(x_0)$ and
let $\theta = f(x_2) - f(x_1)$.
Let $G = G_{01} \cap G_{21}$ with
\begin{align*}
G_{01} &= \bigcap_{t \in I} \left\{\left|\hat f_t(x_0) - \hat f_t(x_1) - f(x_0) + f(x_1)\right| \leq c_t\right\} \quad \text{ and }\\
G_{21} &= \bigcap_{t \in I} \left\{\left|\hat f_t(x_2) - \hat f_t(x_1) - f(x_2) + f(x_1)\right| \leq c_t\right\} \,,
\end{align*}
where $I = \{1 \leq t \leq n \colon t \equiv 0 \mod 3\}$.
These are the events that $\hat f_t(x_0) - \hat f_t(x_1)$ is a reasonable approximation of $f(x_0) - f(x_1)$ for all rounds $t \in I$ and similarly for $\hat f_t(x_2) - \hat f_t(x_1)$.

\begin{exer}
\faStar \quad
Use \cref{thm:hoeffding} and a union bound to show that $\bbP(G) \geq 1 - \delta$. \index{Hoeffding's inequality}
\end{exer}

\solution{
Let $t \in I$.
By definition, $|\hat f_t(x_0) - \hat f_t(x_1) - f(x_0) + f(x_1)| = \frac{3 |S_t|}{t}$ where
\begin{align*}
S_t 
&= \sum_{u=1}^t \epsilon_u (\sind(u \equiv 0 \mod 3) - \sind(u \equiv 1 \mod 3)) \,,
\end{align*}
which is a sum of $2t/3$ independent subgaussian random variables.
By \cref{thm:hoeffding},
\begin{align*}
\bbP\left(|S_t| \geq \sqrt{\frac{8t}{3} \log\left(\frac{4n}{3\delta}\right)}\right) \leq \frac{3\delta}{2n} \,.
\end{align*}
Therefore
\begin{align*}
\bbP\left(\left|\hat f_t(x_0) - \hat f_t(x_1) - f(x_0) + f(x_1)\right| \geq \sqrt{\frac{24 \log\left(\frac{4n}{3\delta}\right)}{t}}\right) \leq \frac{3\delta}{2n} \,.
\end{align*}
A union bound over rounds $t \in I$ shows that $\bbP(G_{01}) \geq 1 - \delta/2$. The same argument shows that $\bbP(G_{21}) \geq 1 - \delta/2$ and a union bound completes the proof.
}

Suppose now that $G$ holds. We claim that $\theta \geq \frac{1}{2} \Delta$.
To reduce clutter, assume without loss of generality that $f(x_\star) = 0$. Suppose we can show that $f(x_0) + f(x_1) + f(x_2) \leq 6 \theta$.
Then
\begin{align*}
\Delta = \frac{1}{3}[f(x_0) + f(x_1) + f(x_2)] - f(x_\star) \leq 2 \theta\,,
\end{align*}
which shows that $\theta \geq \frac{1}{2} \Delta$ as required.
Proving that $f(x_0) + f(x_1) + f(x_2) \leq 6 \theta$ is a tedious case-based analysis depending on the interval containing $x_\star$.
To begin, convexity of $f$ and the assumption that $f(x_2) \geq f(x_0)$ implies that $x_\star$ can be chosen in $[x, x_2]$.
\begin{enumerate}
\item[\texttt{\uline{Case 1}:}] $[x_\star \in [x, x_0]]$. Let $\lambda \in [0,1]$ be such that $x_1 = \lambda x_2 + (1-\lambda) x_\star$.
Then $x_1 = \lambda x_2 + (1 - \lambda) x_\star \geq \lambda x_2 + (1 - \lambda) x$ and therefore $\lambda \leq (x_1 - x)/(x_2 - x) = \frac{2}{3}$.
By convexity of $f$, 
\begin{align*}
f(x_1) 
\tag*{Definition of $\lambda$}
&= f(\lambda x_2 + (1 - \lambda) x_\star) \\
\tag*{Convexity of $f$}
&\leq \lambda f(x_2) + (1 - \lambda) f(x_\star) \\
\tag*{Since $f(x_\star) = 0$}
&= \lambda f(x_2) \\
\tag*{Since $\lambda \leq \frac{2}{3}$}
&\leq \frac{2}{3} f(x_2) \\
\tag*{Definition of $\theta$}
&= \frac{2}{3} f(x_1) + \frac{2}{3} \theta \,.
\end{align*}
Rearranging shows that $f(x_1) \leq 2 \theta$.
Similarly, $f(x_0) \leq \frac{1}{2} f(x_1) \leq \theta$.
Finally, by definition, $f(x_2) = f(x_1) + \theta \leq 3 \theta$. Summing the bounds we have $f(x_0) + f(x_1) + f(x_2) \leq \theta + 2 \theta + 3 \theta = 6\theta$.
\item[\texttt{\uline{Case 2}:}] $[x_\star \in [x_0, x_1]]$. The argument follows a similar pattern. 
Let $\lambda \in [0,1]$ be such that $x_1 = \lambda x_2 + (1 - \lambda) x_\star$. 
Then $x_1 = \lambda x_2 + (1 - \lambda) x_\star \geq \lambda x_2 + (1 - \lambda) x_0$ and hence $\lambda \leq (x_1 - x_0)/(x_2 - x_0) = \frac{1}{2}$.
By convexity, $f(x_1) \leq \frac{1}{2} f(x_2) = \frac{1}{2} f(x_1) + \frac{1}{2} \theta$ and hence $f(x_1) \leq \theta$.
As before, $f(x_2) = f(x_1) + \theta \leq 2 \theta$, and by assumption $f(x_0) \leq f(x_2) \leq 2 \theta$.
Summing shows that $f(x_0) + f(x_1) + f(x_2) \leq 2\theta + \theta + 2\theta = 5\theta$.

\item[\texttt{\uline{Case 3}:}] $[x_\star \in [x_1, x_2]]$. Let $\lambda \in [0,1]$ be such that $x_1 = \lambda x_0 + (1 - \lambda) x_\star$. Then
$x_1 = \lambda x_0 + (1 - \lambda) x_\star \leq \lambda x_0 + (1 - \lambda) x_2$. Therefore $\lambda \leq (x_2 - x_1)/(x_2 - x_0) = \frac{1}{2}$.
Hence $f(x_1) \leq \frac{1}{2} f(x_0) \leq \frac{1}{2} f(x_2) = \frac{1}{2} f(x_1) + \frac{1}{2}\theta$ and so $f(x_1) \leq \theta$.
As before, $f(x_0) \leq f(x_2) = f(x_1) + \theta \leq 2 \theta$, which also shows $f(x_2) \leq 2 \theta$.
Therefore $f(x_0) + f(x_1) + f(x_2) \leq 5 \theta$.
\end{enumerate}
We are now in a position to establish the claims of the theorem, starting with
part~\ref{prop:bisection:opt}. By assumption $f(x_2) \geq f(x_1)$ and hence $x_\star$ cannot be in $[x_2, y]$. The algorithm cannot do any wrong if 
$x_\star \in [x_0, x_2]$. Suppose that $x_\star \in [x, x_0]$. By convexity $f(x_0) \leq f(x_1)$ and hence on $G$,
\begin{align*}
\hat f_t(x_0) - \hat f_t(x_1) < f(x_0) - f(x_1) + c_t \leq c_t\,,
\end{align*}
which means the algorithm does not return $[x_0, y]$.
For part~\ref{prop:bisection:n}, suppose that $c_t \leq \frac{1}{2} \theta$. Then, on event $G$,
\begin{align*}
\hat f_t(x_2) - \hat f_t(x_1)
\geq f(x_2) - f(x_1) - c_t
= \theta - c_t
\geq c_t\,,
\end{align*}
which means the algorithm halts.
Since $\theta \geq \frac{1}{2} \Delta$, it follows that on $G$ the algorithm halts once $t$ is a multiple of three and 
\begin{align*}
\frac{1}{4} \Delta \geq c_t = \sqrt{\frac{24}{t} \log\left(\frac{4n}{3\delta}\right)} \,.
\end{align*}
Solving shows the algorithm halts after at most
\begin{align*}
3 + \frac{384}{\Delta^2} \log\left(\frac{4n}{3\delta}\right)
\end{align*}
queries to the loss function.
\end{proof}

\begin{exer}
\faStar \faStar \faQuestionCircle \quad
Find a slick proof to replace the ugly case-by-case analysis in the proof of \cref{prop:bisection}.
\end{exer}

The main algorithm runs \cref{alg:bisection} iteratively on a shrinking interval and decreasing confidence parameter $\delta$ to ensure that with high probability
the returned interval contains a minimiser of $f$.

\begin{algorithm}[h!]
\begin{algcontents}
\begin{lstlisting}
args: $K = [x, y]$, $n$, $\delta \in (0,1)$
let $K_1 = [x, y]$ and $k_{\max} = 1 + \ceil{\log(n) / \log(4/3)}$. 
for $k = 1$ to $\infty$:
  let $t$ be the current round
  if $t$ = $n + 1$: exit
  $K_{k+1} = \BISECT\left(K_k, n - t + 1, \frac{\delta}{k_{\max}}\right)$ &\Comment{\cref{alg:bisection}}&
\end{lstlisting}
\caption{Bisection method}\label{alg:bisection:full}
\end{algcontents}
\end{algorithm}

\FloatBarrier

The main theorem of this chapter is the following theorem bounding the regret of \cref{alg:bisection:full}.

\begin{theorem}\label{thm:bisection}
Under \cref{ass:bisect},
with probability at least $1 - \delta$, the regret of \cref{alg:bisection:full} is bounded by
\begin{align*}
\Reg_n &= O\left(\vol(K) + \sqrt{n \log\left(\frac{n}{\delta}\right) \log(n)}\right) \,.
\end{align*}
\end{theorem}

\begin{proof}
\cref{alg:bisection:full} runs until the time horizon is reached, repeatedly calling \cref{alg:bisection}.
Let $n_k$ be the number of queries made to the loss function in episode $k$ of \cref{alg:bisection:full} and before
the time horizon $n$ has been reached. Hence $\sum_{k=1}^\infty n_k = n$. 
Let $K_k = [x_k, y_k]$ and
\begin{align*}
\Delta_k = \frac{1}{3} \left[f\left(\frac{1}{4} x_k + \frac{3}{4} y_k\right) + f\left(\frac{1}{2} x_k + \frac{1}{2} y_k\right) + f\left(\frac{3}{4} x_k + \frac{1}{4} y_k\right)\right] - f(x_\star) \,.
\end{align*}
Since $f$ is Lipschitz, provided that $x_\star \in K_k$ it holds that 
\begin{align}
\Delta_k \leq \max_{x \in K_k} (f(x) - f(x_\star)) \leq \vol(K_k) = \left(\frac{3}{4}\right)^{k-1} \vol(K) \,.
\label{eq:bisection:delta}
\end{align}
A union bound and \cref{prop:bisection} show that with probability at least $1 - \delta$ every call made to \cref{alg:bisection} in iterations $k \leq k_{\max}$ either
ends with the horizon being reached or returns a new interval containing the optimum after at most $n_k$ queries with
\begin{align}
n_k \leq 3 + \frac{384}{\Delta_k^2} \log\left(\frac{4n k_{\max}}{3\delta}\right) \,.
\label{eq:bisec:nk}
\end{align}
Assume this good event occurs.
We claim that
\begin{align}
\Reg_n
&=\sum_{t=1}^n \left(f(X_t) - f(x_\star)\right) 
\leq 2\vol(K) + \sum_{k=1}^\infty \sind(k < k_{\max}) n_k \Delta_k \,.
\label{eq:bisect:bound}
\end{align}
There are two cases. 
\begin{enumerate}
\item[\texttt{\uline{Case 1}:}] $[n_{k_{\max}} > 0]$. 
This means that all calls to $\BISECT$ in episodes $k < k_{\max}$ resulted in a smaller interval being returned. Hence, for $k < k_{\max}$, $n_k$ is a multiple of $3$ 
and the regret is $n_k \Delta_k$. The regret in the remaining episodes is bounded by $n (3/4)^{k_{\max} - 1} \leq \vol(K)$ by \cref{eq:bisection:delta}.
\item[\texttt{\uline{Case 2}:}] $[n_{k_{\max}} = 0]$. 
In this case the final episode ends before a smaller interval could be returned. Since for this $k$ it may not hold that $n_k$ is a multiple of $3$, we naively bound
the regret by $n_k \Delta_k + 2 \vol(K)$.
\end{enumerate}
Together the two cases establish \cref{eq:bisect:bound}.
By \cref{eq:bisec:nk}, for $k < k_{\max}$,
\begin{align*}
n_k \Delta_k \leq 6 \Delta_k + \sqrt{768 n_k \log\left(\frac{4n k_{\max}}{3\delta}\right)} \,.
\end{align*}
Hence, by \cref{eq:bisect:bound},
\begin{align*}
\Reg_n 
&\leq 2 \vol(K) +  \sum_{k=1}^\infty \sind(k < k_{\max}) \left[6 \Delta_k + \sqrt{768 n_k \log\left(\frac{4n k_{\max}}{3\delta}\right)}\right] \\
&\leq 2 \vol(K) + 6 \sum_{k=1}^\infty \left(\frac{3}{4}\right)^{k-1} \vol(K) + \sqrt{768 k_{\max} \sum_{k=1}^\infty n_k \log\left(\frac{4n k_{\max}}{3\delta}\right)} \\
&\leq 26 \vol(K) + \sqrt{768 k_{\max} n \log\left(\frac{4n k_{\max}}{3\delta}\right)} \\
&= O\left(\vol(K) + \sqrt{n \log(n/\delta) \log(n)} \right) \,,
\end{align*}
where we used Cauchy--Schwarz along with the fact that $\sum_{k=1}^\infty n_k = n$ and the formula for the geometric sum.
\end{proof}

\index{bisection method!stochastic|)}

\section{Notes}

\begin{enumeratenotes}
\item \cref{alg:bisection} is due to \cite{AFHKR11}. The basic principle behind the bisection method is that the volume of $K_k$ is guaranteed to decrease rapidly
with the number of iterations. Generalising this method to higher dimensions is rather non-trivial. \cite{AFHKR11} and \cite{LG21a} both used algorithms based on the
ellipsoid method \index{ellipsoid method}
while \cite{carpentier2024simple} used the centre of gravity method. These methods are covered in \cref{chap:ellipsoid}.
\item \cref{alg:bisection} works with no assumptions on $f$ beyond convexity and Lipschitzness and ensures $O(\sqrt{n} \log(n))$ regret in the stochastic setting.\index{setting!stochastic}
The algorithm is distinct from all others in this book because its regret depends only very weakly on the range of the loss function.
This is what one should expect from algorithms in the stochastic setting where the magnitude of the noise rather than the losses should determine the regret, as it does
for finite-armed bandits.\index{noise}\index{bandit!finite-armed}
\item There are various ways to refine \cref{alg:bisection:det} for the deterministic case that better exploit convexity \citep{orseau2023line}.
These ideas have not yet been exploited in the noisy (bandit) setting. Bisection-based methods for the deterministic setting\index{setting!deterministic} 
seem fast and require just $O(\log(1/\eps))$ queries to
the zeroth-order oracle to find an $\eps$-optimal point.
Remarkably, for suitably well-behaved functions Newton's method is exponentially faster with sample complexity $O(\log \log(1/\eps))$.\index{Newton's method}
\item \cref{alg:bisection:det} can be improved with a simple modification.
The algorithm evaluates $f$ at two points in each iteration, but by reusing data from previous iterations you can implement almost the same algorithm
using only one evaluation in each iteration. The algorithm is called the golden section search, and is usually analysed for unimodal\index{unimodal} function minimisation \index{golden section search}
\citep{kiefer1953sequential}. \label{note:bisection:reuse}

\item There is an interesting generalisation of the bisection search to a different model where the noise is a bit more adversarial \citep{BCCP22,BCCP22b}.
In the setting of this chapter these algorithms have about the same regret as \cref{alg:bisection:full} but are careful to reuse data as we hinted
at in Note~\ref{note:bisection:reuse} above.
\item The confidence intervals\index{confidence interval} used by \cref{alg:bisection} are generally quite conservative. In practice you may prefer to use standard statistical tests.
This will generally improve performance (quite dramatically). The price is that in extreme cases your confidence intervals will be invalid and the algorithm could suffer linear regret.
\end{enumeratenotes}

\chapter[Online Gradient Descent]{Online Gradient Descent\copynotice}\label{chap:sgd}

In this chapter we introduce an idea that is ubiquitous in zeroth-order optimisation, which is to use a gradient-based algorithm but replace the true gradients
with estimated gradients of a smoothed loss.
Except for \cref{sec:gd:sm-sc},
we assume throughout this chapter that the constraint set contains a euclidean ball of unit radius, the losses are bounded, Lipschitz and there is no noise: 

\begin{assumption}\label{ass:sgd}
The following hold:
\begin{enumerate}
\item $\ball_1 \subset K$; 
\item the loss functions $(f_t)_{t=1}^n$ are in $\cF_{\pb,\pl}$; and
\item there is no noise: $\eps_t = 0$ for all $t$.
\end{enumerate}
\end{assumption}

\begin{wrapfigure}[5]{r}{4cm}
\vspace{-2.5cm}
\includegraphics[width=4cm]{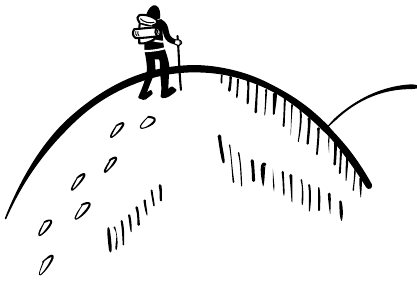}
\caption*{\textit{The long steady slog to the top. Hope the view is nice.}}
\end{wrapfigure}
In contrast to the previous chapter, the setting is now adversarial.
The most well-known optimisation algorithm is gradient descent. By default this algorithm needs access to the gradient of the loss.
We start by explaining the standard analysis of online gradient descent and then introduce spherical smoothing. These ideas are then combined to yield a simple algorithm
and analysis.
At the end of the chapter we explain how smoothness and strong convexity can be exploited to improve the bound, but only in the improper
setting where the learner is allowed to play outside of $K$.

\section{Gradient Descent}\index{gradient descent|(}
Gradient descent incrementally computes a sequence of iterates $(x_t)_{t=1}^n$ with $x_{t+1}$ computed by taking a gradient step from $x_t$.
Let $\Pi_K(x) = \argmin_{y \in K} \norm{x - y}$ be the euclidean projection onto $K$.
An abstract version of gradient descent for bandit convex optimisation is given below.

\begin{algorithm}[h!]
\begin{algcontents}
\begin{lstlisting}
args: learning rate $\eta > 0$
initialise $x_1 \in K$
for $t = 1$ to $n$
  sample $X_t$ from some distribution based on $x_t$ 
  observe $Y_t = f_t(X_t)$
  compute gradient estimate $g_t$ using $x_t$, $X_t$ and $Y_t$
  update $x_{t+1} = \Pi_{K}(x_t - \eta g_t)$ 
\end{lstlisting}
\caption{Abstract gradient descent}\label{alg:sgd-basic}
\end{algcontents}
\end{algorithm}

\FloatBarrier

Importantly, the algorithm does not evaluate the loss function at $x_t$ but rather at some random point $X_t$ whose distribution has not
been specified yet. We have rather informally written that the (conditional) law of $X_t$ should be based on $x_t$, by which we mean that 
\begin{align*}
\bbP(X_t \in A | \sF_{t-1}) = \nu(A | x_t)
\end{align*}
for some probability kernel $\nu \colon \sB(K) \times K \to [0,1]$. The kernel $\nu$ determines how the algorithm explores. \index{probability kernel} 
The gradient estimate $g_t$ is usually not an estimate of $f'_t(x)$, which may not even exist. 
Instead it is an estimate of the gradient of some surrogate loss function $s_t$ that \index{surrogate loss}
is close to $f_t$.
We return to the problem of defining the exploration kernel, surrogate and gradient estimates momentarily. Before that we give some details about gradient descent.
The analysis of gradient descent\index{gradient descent} at our disposal from the online learning literature yields a bound on the regret relative to the 
linear losses defined by the gradient
estimates $g_t$. Specifically, we have the following theorem:

\begin{theorem}\label{thm:sgd:abstract}
Let $(x_t)_{t=1}^n$ be the iterates produced by \cref{alg:sgd-basic}. Then,
for any $x \in K$,
\begin{align*}
\hReg_n(x) \triangleq \sum_{t=1}^n \ip{g_t, x_t - x} \leq \frac{\diam(K)^2}{2 \eta} + \frac{\eta}{2} \sum_{t=1}^n \norm{g_t}^2 \,.
\end{align*}
\end{theorem}

\begin{remark}
In most applications of \cref{thm:sgd:abstract}, $g_t = f_t'(x_t)$. By convexity one has $f_t(x_t) - f_t(x) \leq \ip{g_t, x_t - x}$ and
\cref{thm:sgd:abstract} provides an upper bound on the regret of gradient descent with respect to the losses $(f_t)$.
As mentioned above, in the bandit setting the gradient $f_t'(x_t)$ is not available and we will let $g_t$ be an estimate of the gradient of a suitable surrogate.
\end{remark}

\begin{proof}
Let $x \in K$. 
The idea is quite simple. Suppose the instantaneous regret $r_t = \ip{g_t, x_t - x}$ is large. Then, provided that $\eta \norm{g_t}$ is not too large, 
the algorithm takes a step such that $\norm{x_{t+1} - x} \leq \norm{x_t - x}$. And indeed, the decrease can be written as a function of $r_t$.
Since the distance to $x$ is always non-negative, the cumulative change in distance to $x$ cannot be large.
This suggests a potential argument, which mathematically uses the squared norms as follows:
\begin{align*}
\frac{1}{2}\norm{x_{t+1} - x}^2
&=\frac{1}{2}\norm{\Pi_K(x_t - \eta g_t) - x}^2 \\
&\leq \frac{1}{2}\norm{x_t - x - \eta g_t}^2 \\ 
&= \frac{1}{2}\norm{x_t - x}^2 + \frac{\eta^2}{2} \norm{g_t}^2 - \eta \ip{g_t, x_t - x} \,,
\end{align*}
where in the inequality we used the fact that $\norm{z - \Pi_K(y)} \leq \norm{z - y}$ for all $z \in K$ and $y \in \R^d$.
Rearranging shows that
\begin{align*}
\hReg_n(x) &= \sum_{t=1}^n \ip{g_t, x_t - x} \\
&\leq \sum_{t=1}^n \left[\frac{\eta}{2} \norm{g_t}^2 + \frac{1}{2\eta} \norm{x_t - x}^2 - \frac{1}{2\eta} \norm{x_{t+1} - x}^2\right] \\
&\leq \frac{1}{2\eta} \norm{x_1 - x}^2 + \frac{\eta}{2} \sum_{t=1}^n \norm{g_t}^2 \\ 
&\leq \frac{\diam(K)^2}{2\eta} + \frac{\eta}{2} \sum_{t=1}^n \norm{g_t}^2 \,.
\qedhere
\end{align*}
\end{proof}

What conditions are needed on the gradients $(g_t)_{t=1}^n$ if we want to bound the actual regret in terms of $\hReg_n$?
Let $x_\star = \argmin_{x \in K} \sum_{t=1}^n f_t(x)$.
We have
\begin{align*}
\E[\Reg_n]
&= \E\left[\sum_{t=1}^n \left(f_t(X_t) - f_t(x_\star)\right)\right] \\
&= \E\left[\sum_{t=1}^n \left(\E_{t-1}[f_t(X_t)] - f_t(x_\star)\right)\right] \\
&\stackrel{(\dagger)}\lesssim \E\left[\sum_{t=1}^n \ip{\E_{t-1}[g_t], x_t - x_\star}\right] \\
&= \E\left[\sum_{t=1}^n \ip{g_t, x_t - x_\star}\right] 
= \E\left[\hReg_n(x_\star) \right]
\end{align*}
Can we ensure that $(\dagger)$ holds? 
Remember that $\bbP_{t-1}(X_t = \cdot) = \nu(\cdot|x_t)$ and we get to choose the kernel $\nu$ and the gradient estimator $g_t$.\index{probability kernel}
Since $x_\star$ is not known, the most natural objective is to try and select the kernel and gradient estimate in such a way that for all $x \in K$,
\begin{align*}
\E_{t-1}[f_t(X_t)] - f_t(x) \lesssim \ip{\E_{t-1}[g_t], x_t - x}\,.
\end{align*}
Furthermore, to bound $\E[\hReg_n]$ we need to bound $\E_{t-1}[\Vert g_t \Vert^2]$.
Summarising, a kernel $\nu$ and gradient estimate $g_t$ will yield a good regret bound if
\begin{enumerate}
\item $\E_{t-1}[f_t(X_t)] - f_t(x) \lesssim \ip{\E_{t-1}[g_t], x_t - x}$ for all $x \in K$; and \label{part:sgd:est}
\item $\E_{t-1}[\Vert g_t \Vert^2]$ is small. \label{part:sgd:var}
\end{enumerate}

\begin{remark}
If the learner has access to the gradient $g_t = f'_t(x_t)$, then $f_t(x_t) - f_t(x) \leq \ip{g_t, x_t - x}$ for all $x \in K$ by convexity
and $\norm{g_t}^2 \leq 1$ since $f_t \in \sF_{\pb,\pl}$ is Lipschitz. That is, \ref{part:sgd:est} and \ref{part:sgd:var} hold with $g_t = f_t'(x_t)$ and $X_t = x_t$.
\end{remark}

\index{gradient descent|)}

\section{Spherical Smoothing}\label{sec:gd:smooth}
Let $x \in K$ and $f \in \cF_{\pb,\pl}$. Our algorithm will play some action $X$ that is a random variable
and observe $Y = f(X)$.
We want a gradient estimator $g$ that is a function of $X$ and $Y$ such that
\begin{enumerate}
\item $\E[f(X)] - f(y) \lesssim \ip{\E[g], x - y}$ for all $y \in K$; and
\item $\E[\Vert g \Vert^2]$ is small.
\end{enumerate}
A simple and beautiful estimator is based on Stokes' theorem.\index{Stokes' theorem}
Let $r \in (0,1)$ be a precision parameter and define $s$ as the convolution between $f$ and a uniform distribution on $\ball_r$. That is,\index{uniform measure}
\begin{align*}
s(x) = \frac{1}{\vol(\ball_r)} \int_{\ball_r} f(x + u) \d{u} \,.
\end{align*}
Some examples are plotted in \cref{fig:spherical}.
The function $s$ is convex because it is the convolution of a convex function and a probability density.
We have to be careful about the domain of $s$. Because $f$ is only defined on $K$, the surrogate $s$ is only defined on \index{surrogate loss!spherical} 
\begin{align*}
\dom(s) = \{x \in K \colon x + \ball_r \subset K\} \,.
\end{align*}
By Stokes' theorem, the gradient of $s$ at $x \in \dom(s)$ is
\begin{align}
s'(x) = \frac{1}{\vol(\ball_r)} \int_{\ball_r} f'(x + u) \d{u} = \frac{d}{r} \frac{1}{\vol(\sphere_r)} \int_{\sphere_r} f(x + u)\frac{u}{r} \d{u}\,, 
\label{eq:sgd:stokes}
\end{align}
where we also used the fact from \cref{prop:vol}\ref{prop:vol:sphere} that $\vol(\sphere_r) / \vol(\ball_r) = \frac{d}{r}$.
Actually in the above display we took some liberties. What if $f$ is not differentiable? \index{differentiable}
\begin{exer}
\faStar \quad
Prove that the left-hand and right-hand sides of \cref{eq:sgd:stokes} holds even when $f$ is not differentiable.
\end{exer}

The right-hand side of \cref{eq:sgd:stokes} suggests a way of estimating $s'(x)$.
Let $U$ be uniformly distributed on $\sphere_r$ and $X = x + U$ and define the surrogate gradient estimate by
\begin{align*}
g = \frac{d Y U}{r^2}\,,
\end{align*}
which has expectation $\E[g] = s'(x)$.
How well does this estimator satisfy our criteria?
Suppose that $V$ has law $\cU(\ball_r)$. Then, since $f$ is Lipschitz and using \cref{prop:ball-mean} that $\E[\norm{V}] = \frac{rd}{d+1}$, for all $y \in \dom(s)$,
\begin{align*}
s(y) = \E[f(y + V)] \leq f(y) + \E[\norm{V}] = f(y) + \frac{r d}{d+1}\,.
\end{align*}
On the other hand, since $rV/\norm{V}$ has law $\cU(\sphere_r)$,
\begin{align*}
s(x) = \E[f(x + V)] \geq \E\left[f\left(x+\frac{rV}{\norm{V}}\right) - \norm{V - \frac{rV}{\norm{V}}}\right] = \E[f(X)] - \frac{r}{d+1} \,,
\end{align*}
where we used again that $\E[\norm{V}] = \frac{rd}{d+1}$.
Therefore, since $s$ is convex,
\begin{align}
\ip{\E[g], x - y} 
&= \ip{s'(x), x - y} \nonumber \\
&\geq s(x) - s(y) \nonumber \\
&\geq \E[f(X)] - f(y) - r \,.
\label{eq:sgd:g-diff}
\end{align}
This seems fairly promising. When $r$ is small, then \ref{part:sgd:est} above is indeed satisfied.
Moving now to \ref{part:sgd:var},
\begin{align}
\E[\norm{g}^2] = \frac{d^2}{r^2} \E[Y^2]
= \frac{d^2}{r^2} \E[f(X)^2]
\leq \frac{d^2}{r^2} \,,
\label{eq:sgd:norm}
\end{align}
where we used the fact that $\norm{U} = r$ and the assumption that $f \in \sF_\pb$ is bounded on $K$.
The situation is at a standoff. To satisfy \ref{part:sgd:est} we need $r$ to be fairly small, but then $\E[\norm{g}^2]$ will be quite large.
Nevertheless, enough has been done to make progress.

\begin{figure}
\centering
\includegraphics[height=5cm]{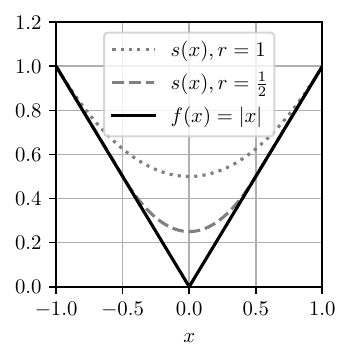}
\includegraphics[height=5cm]{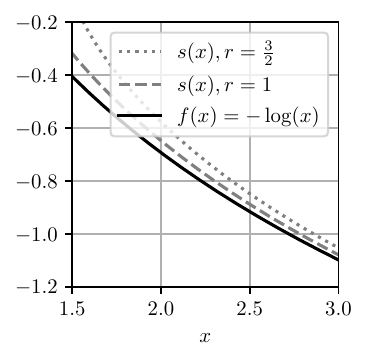}
\caption{The smoothed surrogates for different functions and precisions. Because of convexity the surrogate function is always an upper bound on
the original function. Notice how much better the approximation is for $-\log(x)$, which on the interval considered is much smoother than $|x|$.}
\label{fig:spherical}
\commentAlt{The left plot shows the absolute value loss function and two smoothed versions, both of which lie above the true loss.
The right plot shows the same thing with the logarithmic loss function where the approximation is better.}
\end{figure}

\section{Algorithm and Regret Analysis}

The surrogate and its gradient estimator can be cleanly inserted into online gradient descent to obtain the following
simple algorithm for bandit convex optimisation.

\begin{algorithm}[h!]
\begin{algcontents}
\begin{lstlisting}
args: learning rate $\eta > 0$ and precision $r \in (0,1)$
initialise $x_1 \in K_r = (1 - r)K$
for $t = 1$ to $n$
  sample $U_t$ uniformly from $\sphere_r$ and play $X_t = x_t + U_t$ 
  observe $Y_t = f_t(X_t)$
  compute gradient estimate $g_t = \frac{d Y_t U_t}{r^2}$ 
  update $x_{t+1} = \Pi_{K_r}(x_t - \eta g_t)$.
\end{lstlisting}
\caption{Bandit gradient descent}\label{alg:sgd}
\end{algcontents}
\end{algorithm}

\begin{theorem}\label{thm:sgd}
Suppose that
\begin{align*}
\eta &= \sqrt{\frac{1}{2}} \diam(K)^{\frac{3}{2}} d^{-\frac{1}{2}} n^{-\frac{3}{4}} & ~~~\text{and}~~~ & &
r &= \min\left(1, \sqrt{\frac{1}{2}} \diam(K)^{\frac{1}{2}} d^{\frac{1}{2}} n^{-\frac{1}{4}}\right) \,.
\end{align*}
Under \cref{ass:sgd} the expected regret of \cref{alg:sgd} is bounded by
\begin{align*}
\E[\Reg_n] \leq 
\sqrt{8} \diam(K)^{\frac{1}{2}} d^{\frac{1}{2}} n^{\frac{3}{4}} \,. 
\end{align*}
\end{theorem}

\begin{proof}
Suppose that $r = 1$. Then $\sqrt{1/2} \diam(K)^{\frac{1}{2}} d^{\frac{1}{2}} n^{-\frac{1}{4}} \geq 1$, which along with the assumption that
the losses are bounded in $[0,1]$ implies that
\begin{align*}
\E[\Reg_n] \leq n \leq \sqrt{1/2} \diam(K)^{\frac{1}{2}} d^{\frac{1}{2}} n^{\frac{3}{4}} 
\end{align*}
and the claim is proven.
For the remainder we assume that $r < 1$.
The surrogate in round $t$ is
\begin{align*}
s_t(x) = \frac{1}{\vol(\ball_r)} \int_{\ball_r} f_t(x+u) \d{u}\,.
\end{align*}
Note, by \cref{lem:reg:ball} and the assumption that $\ball_1 \subset K$, it follows that $K_r \subset \dom(s_t)$ for all $1 \leq t \leq n$.
By \cref{prop:shrink},
\begin{align*}
\min_{x \in K_r} \sum_{t=1}^n f_t(x) 
\leq rn + \min_{x \in K} \sum_{t=1}^n f_t(x) \,. 
\end{align*}
Therefore, letting $x_\star = \argmin_{x \in K_r} \sum_{t=1}^n f_t(x)$, 
\begin{align*}
\E[\Reg_n]
&= \max_{x \in K} \E\left[\sum_{t=1}^n (f_t(X_t) - f_t(x))\right] \\
&\leq rn + \E\left[\sum_{t=1}^n (f_t(X_t) - f_t(x_\star))\right] \\
\tag*{By \eqref{eq:sgd:g-diff}}
&\leq 2rn + \E\left[\sum_{t=1}^n \ip{g_t, x_t - x_\star}\right]\,.
\end{align*}
By \cref{thm:sgd} and \cref{eq:sgd:norm},
\begin{align*}
\E\left[\sum_{t=1}^n \ip{g_t, x_t - x_\star}\right] 
&\leq \frac{\diam(K)^2}{2 \eta} + \frac{\eta}{2}\E\left[\sum_{t=1}^n \norm{g_t}^2\right] 
\leq \frac{\diam(K)^2}{2 \eta} + \frac{\eta n d^2}{2r^2} \,. 
\end{align*}
Combining shows that
\begin{align*}
\E[\Reg_n] \leq \frac{\diam(K)^2}{2\eta} + \frac{\eta n d^2}{2r^2} + 2rn \,.
\end{align*}
The claim follows by substituting the constants.
\end{proof}

\section[Smoothness and Strong Convexity]{Smoothness and Strong Convexity (\skippy)}\label{sec:gd:sm-sc}
As a prelude to the next chapter, let us explore what happens when the losses $(f_t)_{t=1}^n$ are assumed to both smooth and strongly convex.
For reasons explained after the analysis in \cref{rem:gd:improper}, the argument that follows only works in the improper 
setting (see \ref{note:improper}), 
which motivates the use of self-concordance that appears in the next chapter.
The algorithm here is based on those by \cite{ADX10} and \cite{akhavan2020exploiting} with the differences explained 
in Note~\ref{note:gd:smooth}.
The operating assumptions in this section are as follows:

\begin{assumption}\label{ass:gd:sc-sm}
The losses $(f_t)_{t=1}^n$ satisfy the following:
\begin{enumerate}
\item $\dom(f_t) = \R^d$ and the setting is improper so that the learner can play outside of $K$.
\item $f_t$ is strongly convex and smooth: $x \mapsto f_t(x) - \frac{\alpha}{2} \norm{x}^2$ is convex 
and $x \mapsto f_t(x) - \frac{\beta}{2} \norm{x}^2$ is concave for some known constants $0 < \alpha \leq \beta$.
\item There exists a known constant $r > 0$ such that $f_t(x) \in [0,1]$ for all $x \in K + \ball_r$.
\item There is no noise: $\eps_t = 0$.
\end{enumerate}
\end{assumption}

The algorithm is the same as \cref{alg:sgd} but projects onto $K$ and uses decreasing learning rates and smoothing radii.

\begin{algorithm}[h!]
\begin{algcontents}
\begin{lstlisting}
args: learning rates $(\eta_t)_{t=1}^n$ and $(r_t)_{t=1}^n$ 
initialise $x_1 \in K$
for $t = 1$ to $n$
  sample $U_t$ uniformly from $\sphere_{r_t}$ and play $X_t = x_t + U_t$ 
  observe $Y_t = f_t(X_t)$
  compute gradient estimate $g_t = \frac{d Y_t U_t}{r_t^2}$ 
  update $x_{t+1} = \Pi_{K}(x_t - \eta_t g_t)$.
\end{lstlisting}
\caption{Bandit gradient descent with varying learning rates and smoothing radii}\label{alg:sgd-sm-sc}
\end{algcontents}
\end{algorithm}

Note there is no guarantee that $X_t \in K$. This is why we assumed the improper setting and that the losses $f_t$ are bounded on $K + \ball_r$.

\begin{theorem}\label{thm:sgd:sm-sc}
Suppose that
\begin{align*}
\eta_t = \frac{1}{t \alpha} \quad \text{and} \quad
r_t^2 = \min\left(r^2, d  \sqrt{\frac{1}{2 \alpha \beta t}}\right)\,.
\end{align*}
Then, under \cref{ass:gd:sc-sm}, the regret of \cref{alg:sgd-sm-sc} is at most
\begin{align*}
\E[\Reg_n] = O\left(d  \sqrt{\frac{\beta n}{\alpha}} + \frac{d^2}{\alpha r^2} \log\left(1 + \frac{d^2}{\alpha \beta r^4}\right)\right)\,.
\end{align*}
\end{theorem}

\begin{proof}
As a good test of your skills, we present this argument as a series of exercises.
To begin, we need an analysis of gradient descent that accommodates the changing learning rates:

\begin{exer}
Suppose that $(g_t)_{t=1}^n$ is an arbitrary sequence of vectors in $\R^d$ and $x_1 \in K$ and $x_{t+1} = \Pi_K(x_t - \eta_t g_t)$.
Prove that for any $x \in K$,
\begin{align*}
\sum_{t=1}^n \ip{g_t, x_t - x} \leq \sum_{t=1}^n \frac{\eta_t}{2} \norm{g_t}^2 + \frac{1}{2} \sum_{t=1}^n \frac{\norm{x_t - x}^2}{2}\left(\frac{1}{\eta_t} - \frac{1}{\eta_{t-1}}\right) \,,
\end{align*}
where we adopt the convention that $\eta_0 \triangleq \infty$.
\end{exer}

\solution{
Following the standard proof:
\begin{align*}
\frac{1}{2} \norm{x_{t+1} - x}^2
&= \frac{1}{2} \norm{\Pi_K(x_t - \eta_t g_t) - x}^2 \\
&\leq \frac{1}{2} \norm{x_t - x - \eta_t g_t}^2 \\
&\leq \frac{1}{2} \norm{x_t - x}^2 + \frac{\eta_t^2}{2} \norm{g_t}^2 - \eta_t \ip{g_t, x_t - x} \,.
\end{align*}
Rearranging and summing shows that
\begin{align*}
\sum_{t=1}^n \ip{g_t, x_t - x}
&\leq \sum_{t=1}^n \frac{\eta_t^2}{2} \norm{g_t}^2 + \sum_{t=1}^n \frac{\norm{x_t - x}^2 - \norm{x_{t+1} - x}^2}{2 \eta_t} \\ 
&\leq \sum_{t=1}^n \frac{\eta_t^2}{2} \norm{g_t}^2 +\frac{1}{2} \sum_{t=1}^n \norm{x_t-x}^2 \left(\frac{1}{\eta_t} - \frac{1}{\eta_{t-1}}\right)
\end{align*}
}

Moving on, let 
\begin{align*}
s_t(x) = \frac{1}{\vol(\ball_{r_t})} \int_{\ball_{r_t}} f_t(x+u) \d{u}\,,
\end{align*}
which is the same surrogate loss used in the proof of \cref{thm:sgd} but now with a smoothing radius of $r_t$ that depends on the round.
By the considerations in \cref{sec:gd:smooth}, $\E_{t-1}[g_t] = s_t'(x_t)$.

\begin{exer}\label{ex:gd:sc-sm}
Prove that $s_t$ is smooth and strongly convex:
\begin{enumerate}
\item $x \mapsto s_t(x) - \frac{\alpha}{2} \norm{x}^2$ is convex; and
\item $x \mapsto s_t(x) - \frac{\beta}{2} \norm{x}^2$ is concave.
\end{enumerate}
\end{exer}

The next step is to use smoothness to refine the comparison between the surrogate loss $s_t$ and the true loss $f_t$.

\begin{exer}\label{ex:gd:compare}
Show that $\E_{t-1}[f_t(X_t) - f_t(x)] \leq \E_{t-1}[s_t(x_t) - s_t(x)] + \beta r_t^2$.
\end{exer}

\solution{
By definition,
\begin{align*}
\E_{t-1}[f_t(X_t)] 
&\leq \E_{t-1}\left[f_t(x_t) + \sip{f_t'(x_t), X_t - x_t} + \frac{\beta}{2} \norm{X_t - x_t}^2\right] \\
&= f_t(x_t) + \E_{t-1}\left[\frac{\beta}{2} \norm{X_t - x_t}^2\right] \\
&= f_t(x_t) + \beta r_t^2 / 2 \\
&\leq s_t(x_t) + \beta r_t^2 / 2 \,.
\end{align*}
Similarly, letting $V_t$ be uniformly distributed on $\ball_{r_t}$ under $\bbP_{t-1}$,
\begin{align*}
s_t(x)
\leq \E_{t-1}[f_t(x + V_t)] 
&\leq f_t(x) + \beta r_t^2 / 2 \,.
\end{align*}
Combining the previous displays completes the exercise.
}

The exercise shows how smoothness improves the error between loss and surrogate to have a quadratic dependence on the smoothing radius.
Let us now show how strong convexity helps. By \cref{ex:gd:sc-sm} and \cref{ex:gd:compare},
\begin{align*}
\E[\Reg_n]
&= \E\left[\sum_{t=1}^n (f_t(X_t) - f_t(x_\star))\right] \\
&\leq \E\left[\sum_{t=1}^n (s_t(x_t) - s_t(x_\star))\right] + \beta \sum_{t=1}^n r_t^2 \\
&\leq \E\left[\sum_{t=1}^n \ip{s_t'(x_t), x_t - x} - \frac{\alpha}{2} \norm{x_t - x}^2\right] + \beta \sum_{t=1}^n r_t^2 \\
&= \E\left[\sum_{t=1}^n \ip{g_t, x_t - x} - \frac{\alpha}{2} \norm{x_t - x}^2 \right] + \beta \sum_{t=1}^n r_t^2 \\
&\leq \E\left[\sum_{t=1}^n \frac{\eta_t}{2} \norm{g_t}^2 + \frac{1}{2} \sum_{t=1}^n \norm{x_t - x}^2\left(\frac{1}{\eta_t} - \frac{1}{\eta_{t-1}} - \alpha\right)\right] + \beta \sum_{t=1}^n r_t^2 \,.
\end{align*}
Now you will see why the learning rate has the form it does. The expression with the reciprocol learning rates vanishes and one obtains
\begin{align} 
\E[\Reg_n]
&\leq \E\left[\sum_{t=1}^n \frac{\norm{g_t}^2}{2\alpha t}\right] + \beta \sum_{t=1}^n r_t^2 \nonumber \\
&= \sum_{t=1}^n \frac{d^2 \E[Y_t^2]}{2\alpha t r_t^2} + \beta \sum_{t=1}^n r_t^2 \,. 
\label{eq:gd:r}
\end{align}
The result follows by substituting the definition of $r_t$, and using \cref{ass:gd:sc-sm} to bound $Y_t^2 \leq 1$ and naive bounding.
\end{proof}

\begin{remark}\label{rem:gd:improper}
\cref{alg:sgd-sm-sc} might play outside of the constraint set.
You may wonder if we can employ the idea in \cref{alg:sgd} of projecting onto a subset of $K$.
The first problem is that we need $x_t + \ball_{r_t} \subset K$ but the tuning of $r_t$ in \cref{thm:sgd:sm-sc} has $r_1 = \Omega(1)$,
which means that $x_t$ must be really quite deep inside $K$. Examining, \cref{eq:gd:r} you could instead let
\begin{align*}
r_t^2 = d \sqrt{\frac{\log(n)}{\alpha \beta n}} \triangleq r^2 \,,
\end{align*}
which gives the same bound as \cref{thm:sgd:sm-sc} up to a $\sqrt{\log(n)}$ factor.
So now the condition is that $x_t + \ball_r \subset K$. But even with smoothness and strong convexity it can happen that
\begin{align*}
\min_{x \in K : x + \ball_r \subset K} f(x) 
\geq \min_{x \in K} f(x) + \Omega(r) \,.
\end{align*}
Hence the increased regret suffered by restricting $K$ can be as large as $n r = \Omega(n^{3/4})$ and there is no improvement
relative to \cref{alg:sgd}.
\end{remark}

\section{Notes}

\begin{enumeratenotes}
\item
\cref{thm:sgd:abstract} is due to \cite{Zin03}. \cref{alg:sgd} essentially appears in the independent works by \cite{FK05} and \cite{Kle04}.
The algorithm continues to work without Lipschitzness but the regret increases to $O(d n^{5/6})$ as explained by \cite{FK05}. 

\item
By \cref{prop:conversion}, the regret bound in \cref{thm:sgd} implies a bound on the sample complexity of $\tilde O(\diam(K)^2 d^2/\eps^4)$.\index{sample complexity}
As far as we know, the spherical smoothing estimator was introduced by \cite{NY83} who used it to prove essentially the same sample complexity as above modulo some minor
technical assumptions about the boundary. \cite{NY83} also noticed that smoothness increases
the performance of the spherical estimator, which we explain in \cref{chap:ftrl}. 

\item 
We did not say much about computation. The only complicated part is computing the projections, the hardness of which
depends on how $K$ is represented. 

\item
\cite{garber2022new} show there are alternative ways to keep the iterates inside the constraint set.
They assume that $\ball_\delta \subset K$ for some $\delta > 0$ and design gradient-descent-based algorithms for which the regret more or less matches
\cref{thm:sgd} and that need either $O(n)$ queries to a linear optimisation oracle or $O(n)$ queries to a separation oracle.\index{separation oracle}

\item Another way to avoid projections is to run gradient descent on the extension defined in \cref{prop:reg:bandit-extension}. This is the approach \index{extension}
we will take in \cref{chap:ons}. Yet another is to use self-concordant barriers as explained in \cref{chap:ftrl}, though this also comes at a computational cost. 

\item \label{note:gd:smooth}
\cref{alg:sgd-sm-sc} is based on \cite{ADX10} who study the adversarial setting and \cite{akhavan2020exploiting} who work in the stochastic
setting. Let us consider the similarities and differences.
\begin{itemizeinner}
\item Because \cite{ADX10} work in the adversarial setting, they used same single-point gradient estimate employed by \cref{alg:sgd-sm-sc}. 
On the other hand, \cite{akhavan2020exploiting} focus on the stochastic setting and consequentially estimate the gradient by querying
the loss at both $x_t$ and $X_t$ to reduce the second moment of the gradient estimate. An example of this idea in action 
is in \cref{sec:sc:stoch}. 
\item \cite{ADX10} used a constant value for $r_t$ rather than the decreasing value used here and by \cite{akhavan2020exploiting}.
This leads to an additional logarithmic factor appearing in the main term. Moreover, \cite{ADX10} assumed that $K$ was a euclidean ball,
though this assumption was not really used in the analysis.
\item \cite{ADX10} assumed the losses were bounded on $K$, possibly forgetting 
that the losses actually need to be bounded on the expansion $K + \ball_r$ that appears in \cref{ass:gd:sc-sm}, although you can derive
one from the other using smoothness at the cost of larger lower-order terms.
Details of this are omitted in their article.
Meanwhile, \cite{akhavan2020exploiting} do not assume boundedness of the loss, but rather that it is Lipschitz on $K$.
\end{itemizeinner}
\end{enumeratenotes}

\chapter[Self-Concordant Regularisation]{Self-Concordant Regularisation\copynotice}\label{chap:ftrl}

The algorithm based on gradient descent\index{gradient descent} in the previous chapter is simple and computationally efficient, at least provided the projection can be computed.
There are two limitations, however.
\begin{itemize}
\item We needed to assume the losses were Lipschitz and the regret depended polynomially on the diameter of the constraint set.
\item Exploiting smoothness and/or strong convexity in the constrained setting is not straightforward due to boundary effects, 
as we explained in \cref{rem:gd:improper}.
\end{itemize}
Both limitations will be removed using `follow-the-regularised-leader' and the beautiful
machinery of self-concordant barriers. \index{follow-the-regularised-leader}
With the exception of \cref{sec:sc:stoch}, it is assumed throughout this chapter that there is no noise and the losses are bound:

\begin{assumption}\label{ass:ftrl}
The following hold:
\begin{enumerate}
\item There is no noise: $\eps_t = 0$ for all $t$.
\item The losses are bounded: $f_t \in \cF_\pb$ for all $t$.
\end{enumerate}
\end{assumption}

Four new regret bounds are given in this chapter, all improving on what was shown in \cref{chap:sgd} in various ways.
The first removes the requirement that the loss is Lipschitz and eliminates entirely the dependence on the diameter of the constraint set.
The second shows how smoothness of the losses improves the quality of the surrogate loss and leads to a dependence on the horizon of $\tilde O(n^{2/3})$.
The highlight is showing that $\tilde O(\sqrt{n})$ regret is attained by a simple algorithm when the losses are 
assumed to be smooth and strongly convex. 
Without this assumption it is still possible to obtain $\tilde O(\sqrt{n})$ regret but with a more complicated algorithm 
and a much more sophisticated analysis (Chapters~\ref{chap:ons} and \ref{chap:ons-adv}).
Lastly, a little time is devoted to investigating the stochastic setting where additionally the variance of the noise is assumed to be smaller than the range of 
the losses (\cref{sec:sc:stoch}).

\section{Self-Concordant Barriers}\label{sec:ftrl:sc} \index{self-concordant barrier|(}

Self-concordance was introduced by \cite{nesterov1988polynomial} as part of the machinery of interior point methods\index{interior point methods} for linear programming.\index{linear programming}
A three-times-differentiable\index{differentiable} convex function $R \colon \interior(K) \to \R$ is a self-concordant barrier on $K$ if
\begin{itemize}
\item $|D^3R(x)[h,h,h]| \leq 2(D^2R(x)[h,h])^{3/2}$ for all $x \in \interior(K)$ and $h \in \R^d$.
\item $R$ is a barrier: $R(x_t) \to \infty$ whenever $x_t \to \partial K$.
\end{itemize}
It is called a $\vartheta$-self-concordant barrier if additionally
\begin{itemize}
\item $DR(x)[h] \leq \sqrt{\vartheta D^2 R(x)[h,h]}$ for all $x \in \interior(K)$ and $h \in \R^d$ where $\vartheta$ is a (hopefully small) positive real value.
\end{itemize}

The local norm at $x \in \interior(K)$ associated with $R$ is $\norm{h}_x \triangleq \norm{h}_{R''(x)}$ and its dual is $\norm{h}_{x\star} = \norm{h}_{R''(x)^{-1}}$.\index{local norm}
The Dikin ellipsoid of radius $r$ at $x$ is 
\begin{align*}
E^x_r = \{y \colon \norm{y - x}_x \leq r\} \,.
\end{align*}
We collect the following facts about $\vartheta$-self-concordant barriers:\index{Dikin ellipsoid} 

\begin{lemma}\label{lem:sc:ftrl}
Suppose that $R$ is a self-concordant barrier on $K$. The following hold:
\begin{enumerate}
\item The Dikin ellipsoid is contained in $K$: $E^x_1 \subset K$ for all $x \in \interior(K)$. \label{lem:sc:ftrl:dikin}
\item For all $x, y \in \interior(K)$, 
\begin{align*} 
R(y) \geq R(x) + \ip{R'(x), y - x} + \rho(-\norm{x - y}_x)
\end{align*}
with $\rho(s) = -\log(1-s) - s$. \label{lem:sc:ftrl:taylor}
\item $\tr(R''(x)^{-1}) \leq \frac{d \diam(K)^2}{4}$ for all $x \in \interior(K)$. \label{lem:sc:ftrl:tr}
\end{enumerate}
Suppose additionally that $R$ is a $\vartheta$-self-concordant barrier and is minimised at $\zeros \in K$; then
with $\pi$ the Minkowski functional of $K$,\index{Minkowski functional}
\begin{enumerate}
\setcounter{enumi}{3}
\item $R(x) \leq R(\zeros) - \vartheta \log\left(1 - \pi(x)\right)$ for all $x \in \interior(K)$. \label{lem:sc:ftrl:mink}
\end{enumerate}
\end{lemma}

For some intuition, part~\ref{lem:sc:ftrl:dikin} is illustrated in \cref{fig:sc} and \ref{lem:sc:ftrl:tr} is a consequence of this (see proof below).
Part~\ref{lem:sc:ftrl:taylor} is a kind of local strong convexity with respect to the norm $\norm{\cdot}_x$. Alternatively, you can view it as an explicit bound on the Taylor series expansion
of $R$ at $x$ by noting that $\rho(s) \sim s^2/2$ for $|s| = o(1)$.
Because $R$ is a barrier it explodes near the boundary of $K$. Part~\ref{lem:sc:ftrl:mink} says that this explosion is quite slow, remembering from \cref{sec:regularity:minkowski} 
that $\pi(x) < 1$ is equivalent to $x \in \interior(K)$.
Note that in \ref{lem:sc:ftrl:mink} we are using \cref{ass:sc}. Otherwise the Minkowski functional would need to be defined relative to the minimiser of $R$.

\begin{proof}[\Proofskippy]
Part~\ref{lem:sc:ftrl:dikin} appears above Equation (2.2) in the notes by \cite{nemirovski96}.
Part~\ref{lem:sc:ftrl:taylor} is Equation (2.4) in the same notes.
Part~\ref{lem:sc:ftrl:tr} follows from Part~\ref{lem:sc:ftrl:dikin}. To see why, let $\xi \in \sphere_1$ and notice that
$x \pm R''(x)^{-1/2} \xi \in E^x_1 \subset K$. 
Therefore 
\begin{align*}
\snorm{\xi}_{R''(x)^{-1}} = \frac{1}{2} \norm{\left(x + R''(x)^{-1/2} \xi\right) - \left(x - R''(x)^{-1/2}\xi\right)} \leq \frac{\diam(K)}{2} \,.
\end{align*}
The result follows because
\begin{align*}
\tr(R''(x)^{-1}) = \sum_{k=1}^d \norm{e_k}^2_{R''(x)^{-1}} \leq \frac{d \diam(K)^2}{4} 
\end{align*}
with $(e_k)_{k=1}^d$ the standard basis vectors.
Part~\ref{lem:sc:ftrl:mink} appears as Equation (3.7) in the notes by \cite{nemirovski96}.
\end{proof}

We will always assume that the coordinate system has been chosen so that $\zeros \in K$ 
and $R$ is minimised at $\zeros$:

\begin{assumption}\label{ass:sc}
$R$ is a $\vartheta$-self-concordant barrier on $K$ and the coordinates are chosen so that
$\argmin_{x \in \interior(K)} R(x) = \zeros$.
\end{assumption}

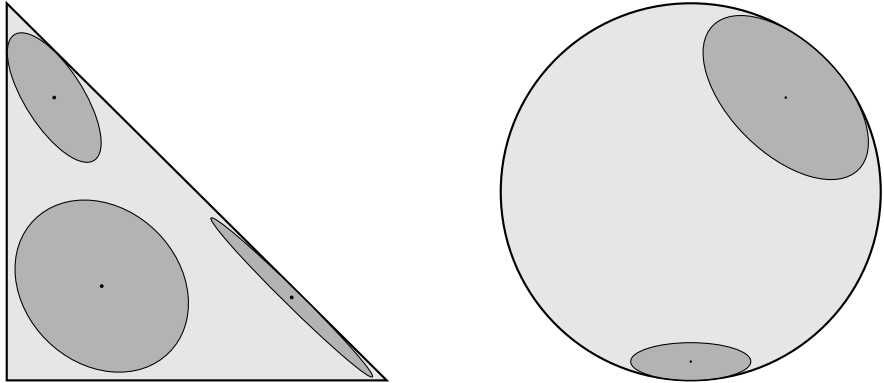
\begin{figure}[h!]
\begin{tikzpicture}[scale=4]
\draw[fill=grayone,thick] (0,0) -- (1cm,0) -- (0,1cm) -- cycle;

\begin{scope}
\pgftransformcm{0.227}{-0.0229}{-0.0229}{0.227}{\pgfpoint{0.25cm}{0.25cm}}
\draw[fill=graytwo] (0,0) circle (1);

\end{scope}

\begin{scope}
\pgftransformcm{0.111}{-0.054}{-0.054}{0.163}{\pgfpoint{0.125cm}{0.75cm}}
\draw[fill=graytwo] (0,0) circle (1);
\end{scope}

\begin{scope}
\pgftransformcm{0.161}{-0.139}{-0.139}{0.159}{\pgfpoint{0.75cm}{0.22cm}}
\draw[fill=graytwo] (0,0) circle (1);
\end{scope}

\draw[fill=black] (0.25cm,0.25cm) circle (0.1pt);
\draw[fill=black] (0.125cm,0.75cm) circle (0.1pt);
\draw[fill=black] (0.75cm,0.22cm) circle (0.1pt);

\begin{scope}
\pgftransformcm{0.5}{0}{0}{0.5}{\pgfpoint{1.8cm}{0.5cm}}
\draw[fill=grayone,thick] (0cm, 0cm) circle (1cm);
\begin{scope}
\pgftransformcm{0.417}{-0.124}{-0.124}{0.417}{\pgfpoint{0.5cm}{0.5cm}}
\draw[fill=graytwo] (0,0) circle (1);
\end{scope}
\begin{scope}
\pgftransformcm{0.316}{0.0}{0.0}{0.1}{\pgfpoint{0.0cm}{-0.9cm}}
\draw[fill=graytwo] (0,0) circle (1);
\end{scope}

\draw[fill=black] (0.5cm,0.5cm) circle (0.1pt);
\draw[fill=black] (0.0cm,-0.9cm) circle (0.1pt);

\end{scope}

\end{tikzpicture}
\caption{Dikin ellpsoids for a polytope and the ball using the barriers in Note~\ref{note:ftrl:logarithmic}.}
\label{fig:sc}
\commentAlt{Three Dikin ellipsoids for convex body K as a triangle and two when K is a unit ball. 
In all cases the Dikin ellipsoids nearly touch the boundary of K and are nearly the largest possible subject to symmetry constraints.}
\end{figure}

\begin{lemma}\label{lem:sc}
Suppose that $\Phi : \interior(K) \to \R$ is 
a self-concordant barrier on $K$, $x = \argmin_{z \in \interior(K)} \Phi(z)$ and $y \in \interior(K)$ is such that $\snorm{\Phi'(y)}_{y\star} \leq \frac{1}{2}$.
Then $\Phi(y) - \Phi(x) \leq \snorm{\Phi'(y)}^2_{y\star}$.
\end{lemma}

Note that $\Phi \neq R$ and we do \textit{not} assume that $\Phi$ is minimised at $\zeros$.

\begin{proof}[\Proofskippy]
Let $y \in \interior(K)$ be such that $\snorm{\Phi'(y)}_{y\star} \leq \frac{1}{2}$ and
abbreviate $g = \Phi'(y)$. Then,
\begin{align*}
\Phi(x) 
\tag*{\cref{lem:sc:ftrl}\ref{lem:sc:ftrl:taylor}}
&\geq \Phi(y) + \ip{g, x - y} + \rho(-\norm{x - y}_y)  \\
\tag*{Cauchy--Schwarz}
&\geq \Phi(y) - \norm{g}_{y\star} \norm{x - y}_y + \rho(-\norm{x - y}_y) \,. 
\end{align*}
Therefore,
\begin{align*}
\Phi(y) 
&\leq \Phi(x) + \norm{g}_{y\star} \norm{x - y}_y - \rho(-\norm{x - y}_y) \\
&\leq \Phi(x) + \max_{r \geq 0} \left[r \norm{g}_{y\star}- \rho(-r)\right] \\
&= \Phi(x) - \log\left[1 - \norm{g}_{y\star}\right] - \norm{g}_{y\star} \\
&\leq \Phi(x) + \norm{g}_{y\star}^2
\end{align*}
where the equality follows by substituting the definition of $\rho(s) = -\log(1-s) - s$ and using basic calculus and the assumption that $\norm{g}_{y \star} \leq 1/2$. 
The final inequality follows from the elementary and naive inequality $-\log(1-t) -t \leq t^2$ for $t \leq \frac{1}{2}$.
\end{proof}
\index{self-concordant barrier|)}

\section{Follow-the-Regularised-Leader}\label{sec:ftrl:ftrl}\index{follow-the-regularised-leader|(}
Follow-the-regularised-leader can be viewed as a generalisation of gradient descent,\index{gradient descent} which for bandits has the following abstract form.
Like gradient descent, follow-the-regularised-leader maintains a sequence of iterates $(x_t)_{t=1}^n$ in $K$
with $x_1 = \argmin_{x \in \interior(K)} R(x)$. 

\begin{algorithm}[h!]
\begin{algcontents}
\begin{lstlisting}
args: $\eta > 0$
for $t = 1$ to $n$
  compute $x_t = \argmin_{x \in \interior(K)}\left[R(x) + \sum_{u=1}^{t-1} \eta \ip{g_u, x}\right]$
  sample $X_t$ based on $x_t$ and observe $Y_t$
  compute gradient estimate $g_t$ using $x_t$, $X_t$ and $Y_t$
\end{lstlisting}
\caption{Follow-the-regularised-leader}\label{alg:ftrl}
\end{algcontents}
\end{algorithm}

\FloatBarrier

As with gradient descent, to make this an algorithm we need to decide on the conditional law of $X_t$ and what
to use for the gradient $g_t$. To get a handle on what is needed, we explain what is guaranteed on the regret relative
to the linear losses defined by $g_t$. 

\begin{theorem}\label{thm:ftrl}
Let $x \in \interior(K)$ and suppose that $\eta \norm{g_t}_{x_t \star} \leq 1/2$ for all $t$. Then for \cref{alg:ftrl}, 
\begin{align*}
\hReg_n(x) 
&\triangleq \sum_{t=1}^n \ip{g_t, x_t - x}
\leq \frac{\vartheta}{\eta} \log\left(\frac{1}{1 - \pi(x)}\right) + \eta \sum_{t=1}^n \norm{g_t}^2_{x_t \star} \,.
\end{align*}
\end{theorem}

The proof of \cref{thm:ftrl} is omitted because it follows from a more general result that we prove later (\cref{thm:ftrl2}).

\index{follow-the-regularised-leader|)}

\section{Optimistic Ellipsoidal Smoothing}\label{sec:ftrl:ellipsoid}
Let us momentarily drop the $t$ indices and let $x \in \interior(K)$ and $f \in \cF_\pb$.
We will introduce a new kind of smoothing. Let $\Sigma$ be positive definite and $E = \{z \in \R^d \colon \norm{x - z}_{\Sigma^{-1}} \leq 1\}$,
which is an ellipsoid centred at $x$. We will assume that $E \subset K$ and let \index{surrogate loss!ellipsoidal}
\begin{align}
s(y) = \frac{1}{\vol(E)} \int_{E} \left(2f\left(\textstyle{\frac{1}{2}} z + \textstyle{\frac{1}{2}} y\right) - f\left(z\right)\right) \d{z} \,.
\label{eq:ftrl:s}
\end{align}

\begin{remark}
Caution! The ellipsoid $E$ in the definition of $s(y)$ is centred at $x$.
\end{remark}

The surrogate loss $s$ behaves quite differently to the spherical smoothing used in \cref{chap:sgd}. 
Perhaps the most notable property is that $s$ is optimistic\index{optimistic} in the sense that $s(y) \leq f(y)$ for all $y \in K$ as we prove below.
The second is that the surrogate is not a good uniform approximation of the real loss, even when $\Sigma = r \id$ and the precision $r$ 
is very small. We want $g$ to be an estimate of $s'(x)$, which is
\begin{align*}
s'(x) 
&= \frac{1}{\vol(E)} \int_{E} f'(\textstyle{\frac{1}{2}} z + \textstyle{\frac{1}{2}} x) \d{z} \\
\tag*{Change of variables}
&= \frac{1}{\vol(\ball_1)} \int_{\ball_1} f'(x + \textstyle{\frac{1}{2}} \Sigma^{1/2} z) \d{z} \\
\tag*{Stokes' theorem}
&= \frac{2 \Sigma^{-1/2}}{\vol(\ball_1)} \int_{\sphere_1} f(x + \textstyle{\frac{1}{2}} \Sigma^{1/2} \xi) \xi \d{\xi} \\
\tag*{\cref{prop:vol}\ref{prop:vol:sphere}}
&= \frac{2 d \Sigma^{-1/2}}{\vol(\sphere_1)} \int_{\sphere_1} f(x + \textstyle{\frac{1}{2}} \Sigma^{1/2} \xi) \xi \d{\xi} \,. 
\end{align*}
Please note we have cheated a little here by assuming that $f$ is differentiable and applying Stokes' theorem. \index{Stokes' theorem}
Fortunately the equality still holds even without differentiability, which is a good exercise.

\begin{exer}
\faStar \quad
Prove the equality in the above display without assuming $f$ is differentiable.\index{differentiable}
\end{exer}

Let $\xi$ be uniformly distributed\index{uniform measure} on $\sphere_1$ and $X = x + \frac{1}{2} \Sigma^{1/2} \xi$.
Then, by the previous display,
\begin{align*}
s'(x) = 4d \Sigma^{-1} \E[f(X) (X - x)]\,.
\end{align*}
Therefore an unbiased estimator of $s'(x)$ is 
\begin{align*}
g = 4 d \Sigma^{-1} Y(X - x) \,.
\end{align*}
The above considerations yield the following lemma:

\begin{lemma} \label{lem:ellipsoid-smooth:unbiased}
$\E[g] = s'(x)$.
\end{lemma}

The next lemma explores the properties of $s$. 

\begin{lemma}\label{lem:ellipsoid-smooth}
Suppose $s$ is defined by \cref{eq:ftrl:s} and $E = \{y \colon \norm{x - y}_{\Sigma^{-1}} \leq 1\} \subset K$. 
The following hold:
\begin{enumerate}
\item $s$ is convex. \label{lem:ellipsoid-smooth:cvx}
\item $s(y) \leq f(y)$ for all $y \in K$. \label{lem:ellipsoid-smooth:opt}
\item If $r \in (0,1)$ and $\Sigma = r^2 R''(x)^{-1}$, then $\E\left[f(X) - s(x)\right] \leq \frac{r/2}{1-r}$.\label{lem:ellipsoid-smooth:lower} 
\item If $f$ is $\beta$-smooth, then $\E\left[f(X) - s(x)\right] \leq \frac{3\beta \tr(\Sigma)}{4d}$. \label{lem:ellipsoid-smooth:lower-smooth} \index{smooth}
\item If $f$ is $\alpha$-strongly convex, then $s$ is $\frac{\alpha}{2}$-strongly convex. \label{lem:ellipsoid-smooth:strong} \index{strongly convex}
\end{enumerate}
\end{lemma}

\begin{proof}
Part~\ref{lem:ellipsoid-smooth:cvx} follows immediately from convexity of $f$, noting that the second (negated) term in the definition of $s$ is 
constant as a function of $y$.
Part~\ref{lem:ellipsoid-smooth:opt} follows from convexity of $f$ as well:
\begin{align*}
s(y) &= \frac{1}{\vol(E)} \int_{E} \left(2f\left(\textstyle{\frac{1}{2}} z + \textstyle{\frac{1}{2}} y\right) - f\left(z\right)\right) \d{z} 
\leq \frac{1}{\vol(E)} \int_{E} f(y) \d{z} 
= f(y)\,.
\end{align*}
For part~\ref{lem:ellipsoid-smooth:lower}, let $\R \ni u \mapsto h_\xi(u) = f(x + u \Sigma^{1/2} \xi)$.
You now have to solve the following exercise, which follows directly from the definitions.

\begin{exer} \label{ex:ftrl:equal}
\faStar \quad
Let $\nu$ be sampled uniformly from $\ball_1$ and independent of $\xi$, which is uniformly sampled from $\sphere_1$. Show that
\begin{align*}
\E[f(X)] = \E[h_\xi(1/2)] \quad \text{and} \quad
s(x) = \E[2h_\xi(\norm{\nu}/2) - h_\xi(\norm{\nu})] \,.
\end{align*}
\end{exer}

\solution{The first equality is immediate.
For the second, note that $s(x) = \E[2f(x+\frac{1}{2}\Sigma^{1/2} \nu) - f(x + \Sigma^{1/2} \nu)]$.
Then use the fact that $\norm{\nu} \xi$ has the same law as $\nu$.
}

By definition $\Sigma = r^2 R''(x)^{-1}$ so that for $u \in [-1/r, 1/r]$, $x + u \Sigma^{1/2} \xi \in E^x_1 \subset K$.
Therefore $h$ is defined on $[-1/r, 1/r]$ and $h_\xi(u) \in [0,1]$ for all $u \in [-1/r,1/r]$.
Hence, by \cref{cor:lip}, 
\begin{align}
\lip_{[-1,1]}(h_\xi) \leq \frac{r}{1-r} \,.
\label{eq:ftrl:lip}
\end{align}
Let $\nu$ be uniformly distributed on $\ball_1$.
\begin{align*}
\E\left[f(X) - s(x)\right]
\tag*{By \cref{ex:ftrl:equal}}
&= \E\left[h_\xi(1/2) + h_\xi(\norm{\nu}) - 2 h_\xi(\norm{\nu}/2)\right] \\
&= \E\left[(h_\xi(1/2) - h_\xi(\norm{\nu}/2)) + (h_\xi(\norm{\nu}) - h_\xi(\norm{\nu}/2))\right] \\
\tag*{By (\ref{eq:ftrl:lip})}
&\leq \frac{r}{1 - r}\E\left[\left|\frac{1}{2} - \frac{\norm{\nu}}{2}\right| + \left|\norm{\nu} - \frac{\norm{\nu}}{2}\right|\right] \\
&= \frac{r/2}{1 - r}\,.
\end{align*}
For part~\ref{lem:ellipsoid-smooth:lower-smooth}, by convexity $\E[h_\xi(u)] \geq \E[h_\xi(0)]$ for all $u \in \R$.
Hence
\begin{align*}
\E\left[f(X) - s(x)\right]
\tag*{By \cref{ex:ftrl:equal}}
&= \E\left[h_\xi(1/2) + h_\xi(\norm{v}) - 2 h_\xi(\norm{v}/2)\right] \\
\tag*{convexity}
&\leq \E\left[h_\xi(1/2) + h_\xi(\norm{v}) - 2 h_\xi(0)\right] \\
\tag*{convexity}
&\leq \E\left[(1/2 + \norm{v}) (h_\xi(1) - h_\xi(0))\right] \\
\tag*{since $\E[\norm{\nu}] \leq 1$}
&\leq \frac{3}{2} \E\left[f(x + \Sigma^{1/2} \xi) - f(x)\right] \\
\tag*{By \cref{lem:smooth}}
&\leq \frac{3\beta}{4} \E\left[\snorm{\Sigma^{1/2} \xi}^2\right] \\
&= \frac{3\beta \tr(\Sigma)}{4d}\,.
\end{align*}
The last equality holds because 
$\E[\snorm{\Sigma^{1/2} \xi}^2] = \E[\tr(\xi \xi^\top \Sigma)] = \tr(\E[\xi \xi^\top] \Sigma)$
and $\E[\xi \xi^\top]=\frac{1}{d} \id$ by a symmetry argument. 
Part~\ref{lem:ellipsoid-smooth:strong} is left as a straightforward exercise. 
\end{proof}

\begin{exer}
\faStar \quad
Prove \cref{lem:ellipsoid-smooth}\ref{lem:ellipsoid-smooth:strong}.
\end{exer}

\section{Algorithms and Regret Analysis}
We start by studying an algorithm that relies on neither smoothness nor strong convexity.

\begin{algorithm}[h!]
\begin{algcontents}
\begin{lstlisting}
args: learning rate $\eta > 0$, $r \in (0,1)$ 
for $t = 1$ to $n$
  compute $x_t = \argmin_{x \in \interior(K)} \sum_{u=1}^{t-1} \eta \ip{g_u, x} + R(x)$ $\label{line:ftrl:basic:opt}$ 
  sample $\xi_t$ uniformly from $\sphere_1$
  play $X_t = x_t + \frac{r}{2} R''(x_t)^{-1/2} \xi_t$ and observe $Y_t$ $\label{line:ftrl:basic:svd}$
  compute gradient $g_t = \frac{4dY_t R''(x_t)(X_t - x_t)}{r^2}$
\end{lstlisting}
\caption{Follow-the-regularised-leader with ellipsoidal smoothing}\label{alg:ftrl:basic}
\end{algcontents}
\end{algorithm}

\FloatBarrier

\subsubsection*{Computation}
\cref{alg:ftrl:basic} needs to compute three non-trivial problems:
\begin{itemize}
\item The optimisation problem in Line~\ref{line:ftrl:basic:opt} is a self-concordant barrier minimisation problem. 
Note in round $t = 1$ we have $x_t = \zeros$ since we assumed that $R$ is minimised at $\zeros$.
In subsequent rounds $x_t$ can be approximated to extreme precision with $\tilde O(1)$ iterations of damped Newton method\index{Newton's method!damped} initialised at $x_{t-1}$ (\cref{ex:ftrl:newton}).
Hence the computation time is dominated by the evaluation of the Hessian of $R$ and a matrix inversion.
\item 
You can sample from a sphere in $O(d)$ time by sampling a $d$-dimensional standard Gaussian and renormalising. 
\item The matrix inverse square root in \cref{line:ftrl:basic:svd} can be computed via singular value decomposition, which has complexity $O(d^3)$.\index{singular value decomposition} 
\end{itemize}

\begin{exer}\label{ex:ftrl:newton}
\faStar \faStar \faBook \quad
Prove that $\tilde O(1)$ iterations of damped Newton is sufficient to approximate $x_t$ to extreme precision (quadratic rate).
You may find \cref{lem:sc} useful, along with the notes by \cite{nemirovski96}.
\end{exer}

The machinery developed in \cref{sec:ftrl:ellipsoid} combined with \cref{thm:ftrl} can be used to bound the regret of \cref{alg:ftrl:basic}.

\begin{theorem}\label{thm:ftrl:basic}
Suppose that
\begin{align*}
\eta &= (\vartheta \log(n))^{\frac{3}{4}} d^{-\frac{1}{2}} n^{-\frac{3}{4}} & ~~~\text{and}~~~ & &
r &= \min\left(1,\, 2 d^{\frac{1}{2}} n^{-\frac{1}{4}} (\vartheta \log(n))^{\frac{1}{4}}\right) \,.
\end{align*}
Under \cref{ass:ftrl} the expected regret of \cref{alg:ftrl:basic} is upper bounded by
\begin{align*}
\E[\Reg_n] \leq 1 + 4 (\vartheta \log(n))^{\frac{1}{4}} d^{\frac{1}{2}} n^{\frac{3}{4}} \,. 
\end{align*}
\end{theorem}

\begin{proof} 
By definition, $\norm{X_t - x_t}_{x_t} = \frac{r}{2} \leq \frac{1}{2}$ and therefore $X_t \in E^{x_t}_1 \subset K$ where the inclusion 
follows from \cref{lem:sc:ftrl}\ref{lem:sc:ftrl:dikin}. 
Hence, the algorithm always plays inside $K$.
Using the fact that the losses are in $\cF_\pb$ it holds automatically that $\Reg_n \leq n$ and when $r > \frac{1}{2}$ this already implies the
bound in the theorem. Suppose for the remainder that $r \leq \frac{1}{2}$.
Similarly, for the same reason we may suppose for the remainder that $n \geq 4 \vartheta \log(n)$.
Let
\begin{align*}
K_{1/n} = \{x \in K \colon \pi(x) \leq 1 - 1/n\}
\end{align*}
and $x_\star = \argmin_{x \in K_{1/n}} \sum_{t=1}^n f_t(x)$ with ties broken arbitrarily. 
Such a point is guaranteed to exist by 
\cref{prop:shrink}, which also shows that
\begin{align*}
\E[\Reg_n]
&\leq 1 + \E[\Reg_n(x_\star)]\,.
\end{align*}
Before using \cref{thm:ftrl} we need to confirm that $\eta \norm{g_t}_{x_t\star} \leq \frac{1}{2}$:
\begin{align*}
\eta \norm{g_t}_{x_t\star} 
&= \frac{4\eta d|Y_t|}{r^2} \norm{R''(x_t)(X_t - x_t)}_{x_t \star} 
= \frac{2\eta d|Y_t|}{r} \leq \frac{2 \eta d}{r} \leq \frac{1}{2}\,, 
\end{align*}
where in the final inequality we used the assumption that $n \geq 4 \vartheta \log(n)$.
Let $\Sigma_t = r^2 R''(x_t)^{-1}$ and $E_t = \{y \colon \norm{x_t - y}_{\Sigma_t^{-1}} \leq 1\} = E^{x_t}_r$.
The surrogate in round $t$ is
\begin{align*}
s_t(x) = \frac{1}{\vol(E_t)} \int_{E_t} \left(2f_t(\textstyle{\frac{1}{2}} y + \textstyle{\frac{1}{2}} x) - f_t(y)\right) \d{y}\,.
\end{align*}
Hence, by \cref{thm:ftrl} and the results in \cref{sec:ftrl:ellipsoid},
\begin{align*}
\E[\Reg_n]
&\leq 1 + \E\left[\sum_{t=1}^n f_t(X_t) - f_t(x_\star)\right] \\
\tag*{\cref{lem:ellipsoid-smooth}\ref{lem:ellipsoid-smooth:opt}\ref{lem:ellipsoid-smooth:lower}}
&\leq 1 + \frac{n r/2}{1-r} + \E\left[\sum_{t=1}^n s_t(x_t) - s_t(x_\star)\right] \\ 
\tag*{\cref{lem:ellipsoid-smooth}\ref{lem:ellipsoid-smooth:cvx}}
&\leq 1 + \frac{n r/2}{1-r} + \E\left[\sum_{t=1}^n \ip{s'_t(x_t), x_t - x_\star} \right] \\
\tag*{\cref{lem:ellipsoid-smooth:unbiased}}
&= 1 + \frac{nr/2}{1-r} + \E\left[\sum_{t=1}^n \ip{g_t, x_t - x_\star}\right] \\
\tag*{\cref{thm:ftrl}}
&\leq 1 + \frac{nr/2}{1-r} + \frac{\vartheta \log(n)}{\eta} + \E\left[\sum_{t=1}^n \eta \norm{g_t}^2_{x_t\star}\right] \\
&\leq 1 + nr + \frac{\vartheta \log(n)}{\eta} + \frac{4 \eta n d^2}{r^2}  \,,
\end{align*}
where the final inequality follows since $r \leq 1/2$ and $Y_t \in [0,1]$ and
\begin{align*}
\eta \norm{g_t}^2_{x_t\star}
&= \eta \norm{\frac{4d Y_t R''(x_t) (X_t - x_t)}{r^2}}^2_{x_t \star} 
\leq \frac{16\eta d^2}{r^4} \norm{X_t - x_t}_{R''(x_t)}^2 
= \frac{4 \eta d^2}{r^2}\,.
\end{align*}
The result follows by substituting the values of the constants.
\end{proof}

Notice how the dependence on the diameter that appeared in \cref{thm:sgd} has been replaced with a dependence on the self-concordance parameter $\vartheta$ and
logarithmic dependence on the horizon. This can be a significant improvement. For example, when $K$ is a ball, then the bound in \cref{thm:sgd} 
depends linearly on $\sqrt{\diam(K)}$ while with a suitable self-concordant barrier the regret in \cref{thm:ftrl:basic} replaces this with $\sqrt{\log(n)}$.
Essentially what is happening is that \cref{alg:ftrl:basic} moves faster deep in the interior where the losses are necessarily more Lipschitz,
whereas \cref{alg:sgd} does not adapt the amount of regularisation to the location of $x_t$.
For smooth functions the rate can be improved by using
\cref{lem:ellipsoid-smooth}\ref{lem:ellipsoid-smooth:lower-smooth}
instead of \cref{lem:ellipsoid-smooth}\ref{lem:ellipsoid-smooth:lower}.

\begin{theorem}\label{thm:ftrl:smooth}
Suppose the losses are in $\cF_{\pb,\psm}$, there is no noise and \index{smooth}
\begin{align*}
r^2 &= \min\left(4,\, 8 \cdot 2^{1/3} \cdot 3^{-2/3}  d^{\frac{2}{3}} (\vartheta \log(n))^{\frac{1}{3}} \beta^{-\frac{2}{3}} \diam(K)^{-\frac{4}{3}} n^{-\frac{1}{3}}\right) \quad \text{and} \\
\eta &= \frac{r}{2d} \sqrt{\frac{\vartheta \log(n)}{n}}\,. 
\end{align*}
Then the expected regret of \cref{alg:ftrl:basic} is upper bounded by
\begin{align*}
\E[\Reg_n] \leq 1 + 3d \sqrt{\vartheta n \log(n)} + \left(\frac{9}{2}\right)^{2/3} (\vartheta \beta \diam(K)^2 \log(n))^{\frac{1}{3}} d^{\frac{2}{3}} n^{\frac{2}{3}}\,.
\end{align*}
\end{theorem}

\begin{proof} 
Note the condition that $r^2 \leq 4$ is needed to ensure that $X_t \in K$. 
Repeat the argument in the proof of \cref{thm:ftrl:basic} but replace \cref{lem:ellipsoid-smooth}\ref{lem:ellipsoid-smooth:lower}
with \cref{lem:ellipsoid-smooth}\ref{lem:ellipsoid-smooth:lower-smooth}, which yields
\begin{align*}
\E[\Reg_n] 
&\leq 1 + \frac{\vartheta \log(n)}{\eta} + \frac{4 \eta nd^2}{r^2} + \frac{3\beta r^2}{4d} \sum_{t=1}^n \tr(R''(x_t)^{-1}) \\
&\leq 1 + \frac{\vartheta \log(n)}{\eta} + \frac{4 \eta nd^2}{r^2} + \frac{3\beta n r^2 \diam(K)^2}{16} \\ 
&= 1 + \frac{4 d}{r} \sqrt{n \vartheta \log(n)} + \frac{3 \beta n r^2 \diam(K)^2}{16}
\end{align*}
where in the second inequality we used \cref{lem:sc:ftrl}\ref{lem:sc:ftrl:tr} and in the equality the definition of $\eta$.
The result follows by substituting the definition of $r$ and using the fact that if $r^2 = 4$, then
\begin{align*}
8 \cdot 2^{1/3} \cdot 3^{-2/3} d^{2/3} (\vartheta \log(n))^{1/3} \beta^{-2/3} \diam(K)^{-4/3} n^{-1/3} \geq 4\,,
\end{align*}
which implies that $\frac{3 \beta n r^2 \diam(K)^2}{16} \leq d\sqrt{\vartheta n \log(n)}$.
\end{proof}

The diameter now appears in the bound, as it must. Otherwise you could scale the coordinates and make the regret vanish (\cref{sec:reg:scaling}).
There is no hope of removing the $\sqrt{n}$ term from \cref{thm:ftrl:smooth}, since when $\beta = 0$ the losses are linear and the lower bound\index{lower bound} 
for linear bandits\index{bandit!linear}
says the regret should be at least $\Omega(d \sqrt{n})$ \citep{DHK08}.

\section{Smoothness and Strong Convexity}\label{sec:ftrl:sm-sc}
With both strong convexity and smoothness
a version of follow-the-regularised-leader can achieve $O(\sqrt{n})$ regret.
The main modification of the algorithm is that the linear surrogate loss functions are replaced by quadratics. \index{surrogate loss}
For this a generalisation of \cref{thm:ftrl} is required. \index{strongly convex}\index{smooth}

\begin{theorem}\label{thm:ftrl2}
Suppose that $(\hat f_t)_{t=1}^n$ is a sequence of self-concordant functions from $K$ to $\R$ and let
\begin{align*}
x_t = \argmin_{x \in \interior(K)} \underbracket{\left(R(x) + \eta \sum_{u=1}^{t-1} \hat f_u(x)\right)}_{\Phi_{t-1}(x)} 
\quad \text{ and } \quad
\norm{\cdot}_{x_t\star} = \norm{\cdot}_{\Phi_t''(x_t)^{-1}} \,.
\end{align*}
Then, provided that $\eta \Vert \hat f'_t(x_t)\Vert_{x_t\star} \leq \frac{1}{2}$ for all $t$, for any $x \in \interior(K)$,
\begin{align*}
\hReg_n(x) = \sum_{t=1}^n \left(\hat f_t(x_t) - \hat f_t(x)\right) \leq
\frac{\vartheta}{\eta} \log\left(\frac{1}{1-\pi(x)}\right) +  \eta \sum_{t=1}^n \Vert \hat f_t'(x_t) \Vert^2_{x_t\star}\,,
\end{align*}
\end{theorem}

\cref{thm:ftrl} is recovered by choosing $\hat f_t(x) = \ip{g_t, x}$.

\begin{proof} 
By the definition of $\Phi_t$,
\begin{align*}
&\hReg_n(x) 
= \sum_{t=1}^n \left(\hat f_t(x_t) - \hat f_t(x)\right) \\
&= \frac{1}{\eta} \sum_{t=1}^n \left(\Phi_{t}(x_t) - \Phi_{t-1}(x_t)\right) - \frac{\Phi_{n}(x)}{\eta} + \frac{R(x)}{\eta} \\
&= \frac{1}{\eta}\sum_{t=1}^n \left(\Phi_{t}(x_t) - \Phi_{t}(x_{t+1})\right) + \frac{\Phi_{n}(x_{n+1})}{\eta} - \frac{\Phi_{n}(x)}{\eta} + \frac{R(x) - R(x_1)}{\eta} \\
\tag*{$\Phi_n(x_{n+1}) \leq \Phi_n(x)$}
&\leq \frac{1}{\eta} \sum_{t=1}^n \left(\Phi_{t}(x_t) - \Phi_{t}(x_{t+1})\right) + \frac{R(x) - R(x_1)}{\eta} \\
\tag*{\cref{lem:sc}}
&\leq \eta \sum_{t=1}^n \Vert \hat f'_t(x_t)\Vert^2_{x_t\star} + \frac{R(x) - R(x_1)}{\eta} \\
\tag*{\cref{lem:sc:ftrl}\ref{lem:sc:ftrl:mink}}
&\leq \frac{\vartheta}{\eta} \log\left(\frac{1}{1 - \pi(x)}\right) + \eta \sum_{t=1}^n \Vert \hat f_t(x_t)\Vert^2_{t \star}\,,
\end{align*}
where the application of \cref{lem:sc} relied on the assumption that $\eta \snorm{\hat f_t'(x_t)}_{t \star} \leq \frac{1}{2}$ and the 
fact that $x_t$ minimises $\Phi_{t-1}$ on $\interior(K)$, which implies that
$\Phi'_t(x_t) = \Phi'_{t-1}(x_t) + \hat f_t'(x_t) = \hat f_t'(x_t)$.
\end{proof}

The algorithm for smooth and strongly convex losses uses follow-the-regularised-leader with a self-concordant barrier and quadratic loss estimates.

\begin{algorithm}[h!]
\begin{algcontents}
\begin{lstlisting}
args: learning rate $\eta > 0$
for $t = 1$ to $n$
  let $x_t = \argmin_{x \in \interior(K)}\left[R(x) + \eta \sum_{u=1}^{t-1} \left(\ip{g_u, x} + \frac{\alpha}{4} \norm{x - x_u}^2\right)\right]$ 
  let $\Sigma_t^{-1} = R''(x_t) + \frac{\eta \alpha t}{2} \id$
  sample $\xi_t$ uniformly from $\sphere_1$ 
  play $X_t = x_t + \frac{1}{2} \Sigma_t^{1/2} \xi_t$ and observe $Y_t$
  compute gradient $g_t = 4dY_t\Sigma_t^{-1} (X_t - x_t)$
\end{lstlisting}
\caption{Follow-the-regularised-leader with ellipsoidal smoothing}
\label{alg:ftrl:basic-sc}
\end{algcontents}
\end{algorithm}

\FloatBarrier

Let us think a little about why Algorithm~\ref{alg:ftrl:basic-sc} makes sense. As usual, let $s_t$ be the surrogate as defined in Section~\ref{sec:ftrl:ellipsoid}, which
by \cref{lem:ellipsoid-smooth} is $\frac{\alpha}{4}$-strongly convex. The quantity $g_t$ is an unbiased estimator of $s_t'(x_t)$. So Algorithm~\ref{alg:ftrl:basic-sc} is playing
follow-the-regularised-leader with quadratic approximations of $s_t$.
The inverse covariance $\Sigma_t^{-1}$ is chosen to be the Hessian of the optimisation objective to find $x_t$.
From a technical perspective this makes sense because the covariance of the gradient estimator plays well with the dual norm in the last term in Theorem~\ref{thm:ftrl2}.
More intuitively, when the losses have high curvature,\index{curvature} then the algorithm needs to smooth on a smaller region, which corresponds to a larger inverse covariance.
This introduces additional variance\index{variance} in the gradient estimators, which is offset by the regularisation arising from strong convexity.

\begin{theorem}\label{thm:ftrl:sc-smooth}
Suppose the losses are in $\cF_{\pb,\psm,\psc}$, there is no noise and
\begin{align*}
\eta = \frac{1}{2d} \sqrt{\frac{\vartheta \log(n) + \frac{3\beta}{2\alpha}[1+\log(n)]}{n}} \,.
\end{align*}
Then the expected regret of \cref{alg:ftrl:basic-sc} is upper bounded by
\begin{align*}
\E[\Reg_n] \leq 1 + 4d\sqrt{n \left(\vartheta \log(n) + \frac{3\beta}{2\alpha}(1+\log(n))\right)}\,.
\end{align*}
\end{theorem}

\begin{proof} 
Assume that $n \geq 4(\vartheta \log(n) + \frac{2\beta}{\alpha}(1 + \log(n))$, since otherwise the regret bound holds trivially using the
assumption that the losses are bounded so that $\E[\Reg_n] \leq n$.
The same argument as in the proof of \cref{thm:ftrl:basic} shows that $X_t$ is in the Dikin ellipsoid associated with $R$ at $x_t$ and therefore\index{Dikin ellipsoid}
is in $K$.
Let $E_t = E(x_t, \Sigma_t)$ and
\begin{align*}
s_t = \frac{1}{\vol(E_t)} \int_{E_t} \left(2 f_t(\textstyle{\frac{1}{2}} x_t+\textstyle{\frac{1}{2}} z) - f_t(z)\right) \d{z}\,,
\end{align*}
which is the optimistic surrogate from \cref{sec:ftrl:ellipsoid}.
Like in \cref{thm:ftrl:basic}, let $K_{1/n} = \{x \in K : \pi(x) \leq 1 - 1/n\}$
and $x_\star = \argmin_{x \in K_{1/n}} \sum_{t=1}^n f_t(x)$, which by \cref{prop:shrink} and
\cref{lem:ellipsoid-smooth}\ref{lem:ellipsoid-smooth:opt}\ref{lem:ellipsoid-smooth:lower-smooth} means that
\begin{align*}
\E[\Reg_n] 
&\leq 1 + \E[\Reg_n(x_\star)] \\
&= 1 + \E\left[\sum_{t=1}^n (f_t(X_t) - f_t(x_\star))\right] \\
&\leq 1 + \E\left[\sum_{t=1}^n \left(s_t(x_t) - s_t(x_\star) + \frac{3\beta}{4d} \tr(\Sigma_t)\right)\right] \,.
\end{align*}
Next, let $\hat f_t(x) = \ip{g_t, x - x_t} + \frac{\alpha}{4} \norm{x - x_t}^2$.
By \cref{lem:ellipsoid-smooth}\ref{lem:ellipsoid-smooth:strong}, $s_t$ is $\frac{\alpha}{2}$-strongly convex
and therefore
\begin{align*}
\E\left[\sum_{t=1}^n (s_t(x_t) - s_t(x_\star))\right]
&\leq \E\left[\sum_{t=1}^n \left(\ip{\E_{t-1}[g_t], x_t - x_\star} - \frac{\alpha}{4} \norm{x_t - x_\star}^2\right)\right] \\ 
&= \E\left[\sum_{t=1}^n (\hat f_t(x_t) - \hat f_t(x_\star))\right] \\
\tag*{\cref{thm:ftrl2}}
&\leq \frac{\vartheta \log(n)}{\eta} + \eta \E\left[\sum_{t=1}^n \norm{g_t}^2_{\Sigma_t} \right] \\
&\leq \frac{\vartheta \log(n)}{\eta} + 4 \eta n d^2 \,. 
\end{align*}
The application of \cref{thm:ftrl2} relies on $\eta \norm{g_t}_{\Sigma_t} \leq \frac{1}{2}$, which follows from
the definition of $\eta$ and our assumption that $n$ is large enough.
Using the definition of $\Sigma_t$,
\begin{align*}
\frac{3\beta}{4d} \sum_{t=1}^n \tr(\Sigma_t)
&\leq \frac{3\beta}{2\alpha \eta} \sum_{t=1}^n \frac{1}{t} 
\leq \frac{3\beta}{2\alpha \eta} (1 + \log(n))\,.
\end{align*}
Combining everything shows that
\begin{align*}
\E[\Reg_n] 
&\leq 1 + 4 \eta n d^2 + \frac{1}{\eta} \left(\vartheta \log(n) + \frac{3\beta}{2\alpha} (1 + \log(n))\right) \,.
\end{align*}
The result follows by substituting the definition of $\eta$.
\end{proof}

\section[Stochastic Setting and Variance]{Stochastic Setting and Variance (\skippy)}\label{sec:sc:stoch}\index{variance} 
In most of this book it as assumed that the losses are bounded in $[0,1]$ and the noise is subgaussian.\index{subgaussian}
You may wonder what happens in the stochastic setting\index{setting!stochastic} if the variance of the noise is much smaller than the range of the loss function. Note, there is nothing substantive to be gained in the adversarial setting since the loss functions themselves can be noisy.
In this section we explain one way to handle this situation by slightly modifying \cref{alg:ftrl:basic} and showing how its regret depends on the variance
of the noise. The modification and analysis used here generalises to all the other results in this chapter and many beyond.
The operating assumption in this section is the following:

\begin{assumption}\label{ass:sgd-stoch}
The setting is stochastic: $f_t = f$ for all rounds with $f \in \sF_{\pb}$. The observed loss is $Y_t = f(X_t) + \eps_t$ where the noise $\eps_t$ satisfies
\begin{enumerate}
\item \textit{(zero mean):} $\E_{t-1}[\eps_t|X_t] = 0$;  
\item \textit{(boundedness):} $|\eps_t| \leq 1$ almost surely; and \label{ass:sgd-stoch:bounded}
\item \textit{(variance):} $\E_{t-1}[\eps_t^2|X_t] \leq \sigma^2$ for some known $\sigma > 0$.
\end{enumerate}
\end{assumption}

\begin{remark}
The boundedness assumption could be relaxed with minor modifications to the analysis if we instead assumed that $\eps_t$ 
was conditionally $\sigma$-subgaussian: $\E_{t-1}[\exp(\eps_t^2/\sigma^2)|X_t] \leq 2$.
Concretely, boundedness is only used in \cref{eq:ftrl:g-bound}. When the noise is subgaussian you need to bound the difference
between losses with high probability. The regret should be the same except possibly a logarithmic factor that vanishes as $n \to \infty$.
\end{remark}

\newcommand{\odd}{\operatorname{\textsc{odd}}}

Let $\odd(t)$ be the set of odd natural numbers less than or equal to $t$.
We assume for simplicity that the horizon $n$ is even so that $\odd(n) = \{1,3,\ldots,n-1\}$.

\begin{algorithm}[h!]
\begin{algcontents}
\begin{lstlisting}
args: learning rate $\eta > 0$, $r \in (0,1)$
for $t \in \odd(n)$:
  compute $x_t = \argmin_{x \in \interior(K)} \sum_{u \in \odd(t-1)} \eta \ip{g_u, x} + R(x)$ 
  sample $\xi_t$ uniformly from $\sphere_1$
  play $X_t = x_t$ and observe $Y_t$
  play $X_{t+1} = x_t + \frac{r}{2} R''(x_t)^{-1/2} \xi_t$ and observe $Y_{t+1}$ 
  compute gradient $g_t = \frac{4d(Y_{t+1} - Y_t) R''(x_t)(X_{t+1} - x_t)}{r^2}$
\end{lstlisting}
\caption{Follow-the-regularised-leader with ellipsoidal smoothing}\label{alg:ftrl:basic-stoch}
\end{algcontents}
\end{algorithm}

\begin{theorem}\label{thm:ftrl:basic-stoch}
Suppose that \cref{alg:ftrl:basic-stoch} is run with parameters
\begin{align*}
r &= \max\left(d \sqrt{\frac{\vartheta \log(n)}{n}},  d^{1/2} \sigma^{1/2} n^{-1/4} (\vartheta \log(n))^{1/4}\right) \quad \text{and}  \\
\eta &= \frac{1}{12d} \sqrt{\frac{\vartheta \log(n)}{n}} \min\left(\frac{r}{\sigma }, 1\right)  \,.
\end{align*}
Then, under \cref{ass:sgd-stoch}, the regret is bounded by
\begin{align*}
\E[\Reg_n] = O\left(d \sqrt{n \vartheta \log(n)} + \sigma^{1/2} d^{1/2} (\vartheta \log(n))^{1/4} n^{3/4} \right) \,.
\end{align*}
\end{theorem}

When $n$ is large, then the bound in \cref{thm:ftrl:basic-stoch} improves on the bound in \cref{thm:ftrl:basic} by a factor of $\sigma^{1/2}$.
Alternatively, if the noise vanishes, the rate improves to $\tilde O(n^{1/2})$.
The noise-free setting is quite special, since in this case there exist algorithms with much smaller sample complexity or regret
\citep{yudin1976informational,protasov1996algorithms}. Nevertheless, in intermediate regimes the improvement is non-negligible.

\begin{proof}
Without loss of generality assume that $r \leq 1/2$, since otherwise the claimed regret bound holds vacuously for any algorithm.

\begin{exer}\label{ex:ftrl:stoch-diff}
Suppose that $t \in \odd(n)$. Show that $|f(X_t) - f(X_{t+1})| \leq r/2$.
\end{exer}

\solution{%
By definition $X_t = x_t$ and $X_{t+1} = x_t + \frac{r}{2} R''(x_t)^{-1/2} \xi_t$ and $\xi_t \in \sphere_1$.
Let $h(u) = f(x_t + u R''(x_t)^{-1/2} \xi_t)$, which is well-defined on $[-1,1]$ since $x_t \pm R''(x_t)^{-1/2} \xi_t$ is
in the Dikin ellipsoid at $x_t$.
By assumption $f \in \cF_{\pb}$, which means that $h(u) \in [0,1]$ for all $u$ and of course $h$ is convex.
Therefore $h(r/2) \leq (1-r/2) h(0) + r/2 h(1) \leq h(0) + r/2$ and similarly $h(0) \leq \frac{r}{2+r} h(-1) + \frac{2}{2+r} h(r/2)
\leq h(r/2) + r/2$.
The result is completed by substituting the definition of $h(0) = f(x_t)$ and $h(r/2) = f(X_{t+1})$.
}

Suppose that $t \in \odd(n)$.
By definition,
\begin{align}
\eta \norm{g_t}_{x_t\star} 
&= \frac{4 \eta d |Y_{t+1} - Y_t|}{r^2} \norm{R''(x_t)(X_t - x_t)}_{x_t \star} \nonumber \\
&= \frac{2 \eta d |Y_{t+1} - Y_t|}{r} 
\explana\leq \frac{6 \eta d}{r} 
\explana\leq \frac{1}{2} \,,
\label{eq:ftrl:g-bound}
\end{align}
where \explanr{} follows from \cref{ass:sgd-stoch}\ref{ass:sgd-stoch:bounded} and the definitions to bound $|Y_{t+1} - Y_t| \leq 3$ and
\explanr{} from the definitions of $\eta$ and $r$.
Hence, repeating more or less exactly the proof of \cref{thm:ftrl:basic} shows that
\begin{align}
\E[\Reg_n] \leq 1 + nr + \frac{\vartheta \log(n/2)}{\eta} + \E\left[\sum_{t \in \odd(n)} \eta \norm{g_t}^2_{x_t \star}\right] \,.
\label{eq:sc:stoch-1}
\end{align}
Moreover, when $t \in \odd(n)$,
\begin{align}
\eta \norm{g_t}^2_{x_t\star}
&= \eta \norm{\frac{4 d (Y_{t+1} - Y_t) R''(x_t)(X_t - x_t)}{r^2}}^2_{x_t\star}  \nonumber \\
&\leq \frac{4 \eta d^2 (Y_{t+1} - Y_t)^2}{r^2} \,.
\label{eq:sc:stoch-2}
\end{align}
The expectation of $(Y_{t+1} - Y_t)^2$ is bounded by
\begin{align*}
\E[(Y_{t+1} - Y_t)^2]
&= \E[(f(X_{t+1}) + \eps_{t+1} - f(X_t) - \eps_t)^2] \\
&\explana\leq \E[(f(X_{t+1}) - f(x_t))^2] + \E[(\eps_t - \eps_{t+1})^2] \\
&\explana\leq \E[(f(X_{t+1}) - f(x_t))^2] + 2 \sigma^2 \\
&\explana\leq \frac{r^2}{4} + 2 \sigma^2 \,. 
\end{align*}
where
\explanr{} and \explanr{} follow from \cref{ass:sgd-stoch} and
\explanr{} since $|f(X_{t+1}) - f(x_t)| \leq r/2$ by \cref{ex:ftrl:stoch-diff}.
Therefore
\begin{align*}
\E\left[\sum_{t \in \odd(n)} (Y_{t+1} - Y_t)^2\right] &\leq n\sigma^2 + \frac{nr^2}{8} \,.
\end{align*}
Combining this with \cref{eq:sc:stoch-1} and \cref{eq:sc:stoch-2} shows that
\begin{align*}
\E[\Reg_n]
&\leq 1 + nr + \frac{\vartheta \log(n/2)}{\eta} + \frac{4 \eta d^2}{r^2} \E\left[\sum_{t \in \odd(n)} (Y_{t+1} - Y_t)^2\right] \\
&\leq 1 + nr + \frac{\vartheta \log(n/2)}{\eta} + \frac{4 \eta d^2}{r^2}\left(n \sigma^2 + \frac{nr^2}{8}\right) \,. 
\end{align*}
The claim now follows by substituting the constants and naive simplification.
\end{proof}

\begin{exer}
\faStar\faStar\faQuestion\quad
Explore the possiblity of using the technique developed here for other algorithms in this book.
\end{exer}

\section{Notes}

\begin{enumeratenotes}
\item The notion of self-concordance was introduced and refined by \cite{nesterov1988polynomial} and \cite{NN89}, applying it to interior point methods.\index{interior point methods}
The first application of self-concordance to bandits was by \cite{AHR08}, who studied linear bandits.\index{bandit!linear}
\cref{thm:ftrl:basic} seems to be new while \cref{thm:ftrl:smooth} is by \cite{Sah11}.
\cref{alg:ftrl:basic-sc} and \cref{thm:ftrl:sc-smooth} are due to \cite{HL14}. 
\cref{thm:ftrl:basic-stoch} is new but the idea is standard \citep[and many others]{akhavan2024gradient}.
The class of smooth and strongly convex losses already appeared in the work by \cite{polyak1990optimal}, who considered 
the stochastic unconstrained and improper settings. Their results show that the optimal simple regret in this case is $\Theta(n^{-1/2})$ 
with non-specified dependence on other quantities like the dimension. 
\cite{akhavan2020exploiting,akhavan2024contextual} consider the same setting but with explicit constants and in some settings nearly
matching lower bounds.

\item At no point in this chapter did we need Lipschitz losses. The analysis essentially exploits the fact that convex functions
cannot have large gradients except very close to the boundary, where the regularisation provided by the self-concordant barrier prevents the blowup in variance
from severely impacting the regret.
\item We have made several improvements to the statistical efficiency relative to the algorithm presented in \cref{chap:sgd}.
In exchange the algorithms are more complicated and computationally less efficient. Algorithms based on gradient descent run in $O(d)$ time per round except
those rounds where a projection is needed. Furthermore, even when the projection is needed it is with respect to the euclidean norm and likely to be extremely fast.
Meanwhile the algorithms in this chapter need a singular value decomposition to compute $X_t$, solve an optimisation problem to find $x_t$ and need oracle\index{singular value decomposition}
access to a $\vartheta$-self-concordant barrier. 
\item The reader interested in knowing more about ($\vartheta$-)self-concordant barriers is referred to the wonderful notes by \cite{nemirovski96}.
The most obvious question is whether or not these things even exist. Here are some examples: \index{self-concordant barrier}
\begin{itemizeinner}
\item When $K = \{x \colon \ip{a_i, x} \leq b_i, 1 \leq i \leq k\}$ is a polytope defined by $k$ half-spaces,\index{polytope} then $R(x) = -\sum_{i=1}^k \log(b_i - \ip{a_i,x})$ is
called the logarithmic barrier\index{self-concordant barrier!logarithmic} and is $k$-self-concordant. \label{note:ftrl:logarithmic}
\item When $K = \{x \colon \norm{x} \leq \rho\}$ is a ball, then $R(x) = -\log(\rho^2 - \norm{x}^2)$ is a $1$-self-concordant barrier on $K$.
\label{note:ftrl:sc-ball} 
\item For any convex body $K$ there exists a $\vartheta$-self-concordant barrier with $\vartheta \leq d$. Specifically, the entropic barrier \index{self-concordant barrier!entropic}
\citep{chewi2023entropic,bubeck2014entropic} and the universal barrier \citep{nesterov1994interior,lee2021universal} satisfy this.\index{self-concordant barrier!universal}
\end{itemizeinner}
\item The surrogate loss only appears in the analysis. Interestingly, \cite{HL14} and \cite{Sah11} analysed their algorithms using the surrogate \index{surrogate loss}
\begin{align*}
s_t(y) = \frac{1}{\vol(E^{x_t}_r)} \int_{E^{x_t}_r - x_t} f_t(y + u) \d{u}\,,
\end{align*}
which is the ellipsoidal analogue of the surrogate used in \cref{chap:sgd}.
Except for a constant factor this surrogate has the same gradient at $x_t$ as the surrogate we used, which means the resulting algorithms are the same.
The difficulty is that the surrogate above is not defined on all of $K$, which forces various contortions or assumptions in the analysis.

\item Even when $\beta = 0$, the regret upper bound of \cref{alg:ftrl:basic} is still $\Omega(\sqrt{n})$. Since
$\beta = 0$ corresponds to linear losses, the lower bounds for linear bandits (\cref{tab:lower}) show that this is not improvable.
Hence, no amount of smoothness by itself can improve the regret beyond the $\sqrt{n}$ barrier.
Combining higher-order smoothness (see Note~\ref{note:intro:smooth}) with strong convexity, however, does lead to improved regret
\citep{polyak1990optimal,akhavan2020exploiting,novitskii2021improved,akhavan2024gradient}. These works prove upper and lower bounds showing that the minimax simple regret is 
$\Theta(n^{(1-p)/p})$ in the unconstrained and improper settings and with slightly varying assumptions and dependence on the constants.\index{setting!improper}\index{setting!unconstrained} 
This is much better than $O(1/\sqrt{n})$ when $p \gg 2$. 
Of these, the most refined is by
\cite{akhavan2024gradient}, who prove an upper bound on the simple regret $O(\frac{1}{\alpha}(d^2/n)^{(p-1)/p})$ and a lower bound on the same
of $\Omega(\frac{d}{\alpha} n^{-(p-1)/p})$, which match when $p = 2$. Note that the correct dependence on the smoothness parameter $\beta$ has not yet been nailed down
and there are some mild conditions on the magnitude of the parameters.
The aforementioned works also study a variety of alternatives to strong convexity and more flexible noise models than what is assumed in this book.

\item \cref{thm:ftrl:basic,thm:ftrl:smooth} bound the regret for the same algorithm with different learning rates and smoothing
parameters. You should wonder if it is possible to obtain the best of both bounds with a single algorithm by adaptively\index{adaptive} tuning the learning rates. At present this
is not known as far as we are aware.

\end{enumeratenotes}

\chapter[Linear and Quadratic Bandits]{Linear and Quadratic Bandits\copynotice}\label{chap:lin}\index{bandit!linear}\index{bandit!quadratic}

Function classes like $\cF_\pb$ are non-parametric. In this chapter we shift gears by studying two important 
parametric classes: $\cF_{\pb,\plin}$ and $\cF_{\pb,\pquad}$. 
The main purpose of this chapter is to use the machinery designed for linear bandits to prove an upper bound on
the minimax regret for quadratic bandits.
On the positive side the approach is both elementary and instructive. More negatively, the resulting algorithm is not computationally efficient.
Before the algorithms and regret analysis we need three tools: covering numbers, optimal experimental design and the exponential weights algorithm.

\section{Covering Numbers}\index{covering number|(} 
Given $A, B \subset \R^d$, the external/internal covering numbers are defined by \label{page:cover} 
\begin{align*}
N(A, B) &= \min\left\{|\cC| \colon \cC \subset \R^d, A \subset \bigcup_{x \in \cC} (x + B)\right\} \quad \text{ and }\\
\bar N(A, B) &= \min\left\{|\cC| \colon \cC \subset A, A \subset \bigcup_{x \in \cC} (x + B)\right\} \,.
\end{align*}
Both are the smallest number of translates of $B$ needed to cover $A$, with the latter demanding that the `centres' are in $A$.
Obviously $N(A, B) \leq \bar N(A, B)$. The inequality can also be strict, as you will show in the following exercise.

\begin{exer}\label{ex:cover}
\faStar \quad
Suppose that $A, B, C \subset \R^d$ and $A \subset B$. 
Show the following:
\begin{enumerate}
\item $N(A, C) \leq \bar N(A, C)$ and give an example where $N(A, C) < \bar N(A, C)$. \label{ex:cover:non-monotone}
\item $N(A, C) \leq N(B, C)$ and give an example where $\bar N(A, C) > \bar N(B, C)$ . \label{ex:cover:monotone}
\item $\bar N(A, C - C) \leq N(A, C)$. \label{ex:cover:diff}
\end{enumerate}
\end{exer}

\solution{
\ref{ex:cover:non-monotone} is immediate, since any internal cover is also an external cover.
As an example where there is a difference, let $A = \{0,1\}$ and $C = [-1/2,1/2]$. 
Then $N(A, C) = 1$ but $\bar N(A, C) = 2$.
\ref{ex:cover:monotone} follows since any external cover of $B$ is also a cover of $A$. For the example, let $A = \{0,1\}$, $B = [0,1]$ and $C = [-1/2,1/2]$. 
Then $\bar N(A, C) = 2$ while $\bar N(B, C) = 1$.
For \ref{ex:cover:diff}, let $m = N(A, C)$ and $x_1,\ldots,x_m$ be such that $A \subset \cup_{k=1}^m (x_k + C)$.
Since $m$ is minimal, for each $k \in \{1,\ldots,m\}$ there exists a point $y_k$ with $y_k \in A$ and $y_k \in x_k + C$.
But note that $x_k + C \subset y_k + (C - C)$ and hence $A \subset \cup_{k=1}^m (y_k + (C - C)$ so that $\bar N(A, C - C) \leq N(A, C)$.
}

The next proposition follows from Fact 4.1.4 and Corollary 4.1.15 in the book by \cite{ASG15}.

\begin{proposition}\label{prop:cover}
Suppose that $A \subset \R^d$ is centrally symmetric, compact and convex. Then, for any $\eps \in (0,1)$,
\begin{align*}
\bar N(A, \eps A) \leq N(A, \textstyle{\frac{\eps}{2}} A) \leq \left(1 + \frac{4}{\eps}\right)^d \,.
\end{align*}
\end{proposition}

\begin{proposition}\label{prop:cover2}
Suppose that $K \subset \R^d$ is compact and $A = \conv(K - K)$. Then
\begin{align*}
\bar N(K, \eps A) \leq \left(1 + \frac{4}{\eps}\right)^d \,.
\end{align*}
\end{proposition}

\begin{proof}
Since $A$ is symmetric, $A - A = 2A$.
By your solution to \cref{ex:cover}\ref{ex:cover:diff} and \ref{ex:cover:monotone}, \cref{prop:cover}, and letting $x \in K$ be arbitrary, 
\begin{align*}
\bar N(K, \eps A) 
&\leq N(K, \textstyle{\frac{\eps}{2}} A) 
= N(K - \{x\}, \textstyle{\frac{\eps}{2}} A) 
\leq N(A, \textstyle{\frac{\eps}{2}} A) 
\leq \left(1 + \frac{4}{\eps}\right)^d \,.
\qedhere
\end{align*}
\end{proof}

\begin{proposition}\label{prop:cover3}
Suppose that $\eps \in (0,1)$ and $A \subset \ball_r$ with $r \geq \eps$. Then 
\begin{align*}
\bar N(A, \ball_\eps) \leq \left(1 + \frac{4r}{\eps}\right)^d \,.
\end{align*}
\end{proposition}

\begin{proof}
By \cref{ex:cover}\ref{ex:cover:diff}\ref{ex:cover:monotone} and \cref{prop:cover}, 
\begin{align*}
\bar N(A, \ball_\eps)
&\leq N(A, \ball_{\eps/2})
\leq N(\ball_r, \ball_{\eps/2})
= N(\ball_r, \textstyle{\frac{\eps}{2  r}} \ball_r)
\leq \left(1 + \frac{4r}{\eps}\right)^d \,.
\qedhere
\end{align*}
\end{proof}

\index{covering number|)}

\section{Optimal Design}\label{sec:linear:design} \index{optimal design}

Suppose that $A$ is a nonempty compact subset of $\R^d$ and $\theta \in \R^d$ is unknown.
A learner samples $X$ from some probability measure $\pi$ on $A$ and observes $Y = \ip{X, \theta}$.
How can this information be used to estimate $\theta$? A simple idea is to use importance-weighted least squares. 
Let $G_\pi = \int_A xx^\top \d{\pi}(x)$, which is called the design matrix.
Assume for a moment that $G_\pi$ is invertible and let
\begin{align*}
\hat \theta = G_\pi^{-1} X Y \,.
\end{align*}
A simple calculation shows that $\E[\hat \theta] = \theta$, which implies that $\E[\sip{\hat \theta, x}] = \ip{x, \theta}$ for all $x \in A$.
So $\sip{\hat \theta, x}$ is an unbiased estimator of $\ip{x, \theta}$. Assuming that $\ip{x, \theta} \in [0,1]$ for all $x \in A$, then the second moment
is bounded by
\begin{align*}
\E\left[\sip{\hat \theta, x}^2\right] 
= \E\left[Y^2 x^\top G_\pi^{-1} XX^\top G_\pi^{-1}x\right] 
\leq \E\left[x^\top G_\pi^{-1} XX^\top G_\pi^{-1}x\right] 
= \norm{x}^2_{G_\pi^{-1}} \,.
\end{align*}
The following theorem shows there exists a $\pi$ such that the right-hand side is at most $d$ for all $x \in A$.

\begin{theorem}[\citealt{KW60}]\label{thm:kw}
For any nonempty compact $A \subset \R^d$ with $\laspan(A) = \R^d$ there exists a probability measure $\pi$ supported on a subset of $A$ such that
$G_\pi = \int_A xx^\top \d{\pi}(x)$ is invertible and
\begin{align*}
\norm{x}^2_{G_\pi^{-1}} \leq d \text{ for all } x \in A \,.
\end{align*}
\end{theorem}

Remarkably the constant $d$ is the best achievable for \textit{any} compact $A$ with $\laspan(A) = \R^d$ in the sense that
\begin{align*}
\min_{\pi \in \Delta(A)} \max_{x \in A} \norm{x}^2_{G_\pi^{-1}} = d\,.
\end{align*}

The assumption that $\laspan(A) = \R^d$ was more or less only needed to ensure that $G_\pi$ is invertible.
Given a matrix $Q$ let $Q^+$ be the pseudoinverse (see \cref{sec:pinv}). \index{Moore--Penrose pseudoinverse}

\begin{theorem}\label{thm:kwx}
For any nonempty compact $A \subset \R^d$ there exists a probability measure $\pi$ supported on a subset of $A$ such that $A \subset \im(G_\pi^\top)$ and
\begin{align*}
\norm{x}^2_{G_\pi^+} \leq \dim(\laspan(A)) \text{ for all } x \in A\,,
\end{align*}
where $G_\pi = \int_A xx^\top d{\pi}(x)$.
\end{theorem}

The requirement in \cref{thm:kwx} that $A \subset \im(G_\pi^\top)$ is essential and corresponds to $G_\pi$ being invertible when restricted
to the subspace spanned by $A$. 

\begin{exer}
\faStar \quad Prove \cref{thm:kwx}.
\end{exer}

\solution{%
Let $L = \laspan(A)$ and $u_1,\ldots,u_m$ be an orthonormal basis for $L$ and $U \in \R^{d \times m}$ have columns $u_1,\ldots,u_m$,
which means that $U^\top U = \id$.
Let $B = \{U^\top a : a \in A\} \subset \R^m$ which has $\dim(\laspan(B)) = m$.
By \cref{thm:kw} there exists a probability measure $\rho$ supported on $B$ such that
for all $y \in B$, $\norm{y}^2_{G_\rho^{-1}} \leq m$.
Let $Y$ have law $\rho$ and $\pi$ be the law of $U Y$.
Then for any $x \in A$,
\begin{align*}
\norm{x}^2_{G_\pi^+} 
&= x^\top \left(\int_B U y y^\top U^\top \d{\rho}(y)\right)^+ x \\
&= (U^\top x)^\top U \left(U G_\rho U^\top \right)^+ U (U^\top x) \\ 
&= (U^\top x)^\top G_\rho^{-1} (Ux) \\
&\leq m \,.
\end{align*}
where we used the fact that $(U G_\rho U^\top)^+ = U G_\rho^+ U^\top = U G_\rho^{-1} U^\top$.
To see that $A \subset \im(G_\pi)$, let $x \in A \subset L$.
Therefore $x = Uy$ for some $y \in \R^m$.
Since $G_\rho$ is invertible, there exists a $w \in \R^m$ such that $y = G_\rho w = G_\rho U^\top U w$
and hence $x = U G_\rho U^\top Uw = G_\pi Uw$. Hence $x \in \im(G_\rho) = \im(G_\rho^\top)$.
}

\section{Exponential Weights}\index{exponential weights|textbf}
Let $\cC$ be a finite set and $\ell_1,\ldots,\ell_n$ a sequence of functions from $\cC \to \R$.
The set $\cC$ is sometimes referred to as the set of experts and $\ell_t(a)$ is the loss suffered by expert $a$ in round $t$.
A learner chooses a sequence of probability distributions $(q_t)_{t=1}^n$ in $\Delta(\cC)$ where $q_t$ can depend on $\ell_1,\ldots,\ell_{t-1}$.
Note that this is not a bandit setting. The entire loss function $\ell_t$ is observed after round $t$.
The learner's aim is to be competitive with the best expert in hindsight,
which is measured by the regret
\begin{align*}
\hReg_n = \max_{b \in \cC} \sum_{t=1}^n \left[\sum_{a \in \cC} q_t(a) \ell_t(a) - \ell_t(b)\right] \,.
\end{align*}
The quantity $\sum_{a \in \cC} q_t(a) \ell_t(a)$ is the average loss suffered by the learner 
if they follow the advice of expert $a$ with probability $q_t(a)$.
Given a learning rate $\eta > 0$, define a distribution $q_t$ on $\cC$ by
\begin{align}
q_t(a) = \frac{\exp\left(-\eta \sum_{u=1}^{t-1} \ell_u(a)\right)}{\sum_{b \in \cC} \exp\left(-\eta \sum_{u=1}^{t-1} \ell_u(b)\right)}\,,
\label{eq:lin:exp}
\end{align}
which is called the exponential weights distribution.
Staring at the definition you can see that $q_t$ puts more mass relatively speaking on experts for which the cumulative loss is smaller. 
That exponential weights has small regret is perhaps the most fundamental result in online learning, as illustrated by the many applications and implications \citep{Ces06}. 

\begin{theorem}\label{thm:exp-discrete}
Suppose that $\eta |\ell_t(a)| \leq 1$ for all $1 \leq t \leq n$, then the regret when $q_t$ is given by \cref{eq:lin:exp} is upper bounded by
\begin{align*}
\hReg_n \leq \frac{\log |\cC|}{\eta} + \eta \sum_{t=1}^n \sum_{a \in \cC} q_t(a) \ell_t(a)^2\,.
\end{align*}
\end{theorem}
\begin{remark}\label{rem:lin:cvx}
There is no convexity here but the bound in \cref{thm:exp-discrete} has some kind of symbolic resemblance to the bounds for gradient descent and 
follow-the-regularised-leader (for example, \cref{thm:sgd:abstract}).\index{gradient descent}
This is no accident. Exponential weights is equivalent to follow-the-regularised-leader on the convex space of probability measures $\Delta(\cC)$ with unnormalised negentropy regularisation.\index{unnormalised negative entropy}
The map $\Delta(\cC) \colon q \mapsto \sum_{a \in \cC} q(a) \ell_t(a)$ is linear and hence the tools from convex optimisation can be used.
\end{remark}

\begin{proof}[Proof of \cref{thm:exp-discrete}]
The following two inequalities provide crude explicit bounds on the series expansion of $\exp(\cdot)$:
\begin{align}
\exp(-x) &\leq 1 - x + x^2 \text{ for all } x \geq -1 \text{ ; and} \label{eq:lin:exp-ineq-1} \\
\log(1+x) &\leq x \text{ for all } x > -1 \,.\label{eq:lin:exp-ineq-2}
\end{align}
Let $b \in \cC$ and $D_t = \log(1/q_t(b))$. 
Note that $q_{t+1}$ is the probability distribution with 
\begin{align*}
q_{t+1}(b) \propto \exp\left(-\eta \sum_{u=1}^t \ell_u(b)\right) = \exp\left(-\eta \ell_t(b)\right) \underbracket{\exp\left(-\eta \sum_{u=1}^{t-1} \ell_u(b)\right)}_{\propto q_t} \,.
\end{align*}
Therefore
\begin{align*}
q_{t+1}(b) = \frac{q_t(b) \exp(-\eta \ell_t(b))}{\sum_{a \in \cC} q_t(a) \exp(-\eta \ell_t(a))} \,.
\end{align*}
Then,
\begin{align*}
D_{t+1} 
&=\log\left(\frac{1}{q_{t+1}(b)}\right) \\
&= \log\left(\frac{\sum_{a \in \cC} q_t(a) \exp\left(-\eta \ell_t(a)\right)}{q_t(b) \exp\left(-\eta \ell_t(b)\right)}\right) \\
&= D_t + \log\left(\sum_{a \in \cC} q_t(a) \exp\left(-\eta \ell_t(a)\right)\right) + \eta \ell_t(b) \\
\tag*{by (\ref{eq:lin:exp-ineq-1})}
&\leq D_t + \log\left(\sum_{a \in \cC} q_t(a) \left[1 - \eta \ell_t(a) + \eta^2 \ell_t(a)^2\right] \right) + \eta \ell_t(b) \\
&= D_t + \log\left(1 + \sum_{a \in \cC} q_t(a) \left[- \eta \ell_t(a) + \eta^2 \ell_t(a)^2\right] \right) + \eta \ell_t(b) \\
\tag*{by (\ref{eq:lin:exp-ineq-2})}
&\leq D_t - \eta \left[\sum_{a \in \cC} q_t(a) \ell_t(a) - \eta \ell_t(b)\right] + \eta^2 \sum_{a \in \cC} q_t(a) \ell_t(a)^2 \,.
\end{align*}
Rearranging and summing over $t$ and telescoping yields
\begin{align*}
\sum_{t=1}^n \left(\sum_{a \in \cC} q_t(a) \ell_t(a) - \ell_t(b)\right) 
&\leq \frac{1}{\eta} \log\left(\frac{q_{n+1}(b)}{q_1(b)}\right) + \eta \sum_{t=1}^n \sum_{a \in \cC} q_t(a) \ell_t(a)^2 \\
&\leq \frac{1}{\eta} \log\left(|\cC|\right) + \eta \sum_{t=1}^n \sum_{a \in \cC} q_t(a) \ell_t(a)^2\,,
\end{align*}
where in the final inequality we used the fact that $\log(q_{n+1}(b)) \leq 0$ and $q_1(b) = 1/|\cC|$.
Since the above calculations hold for any $b \in \cC$ we are free to take the maximum on the left-hand side, which yields the theorem.
\end{proof}

\begin{remark}\label{rem:exp-discrete}
When $\ell_t(a) \geq 0$ for all $t$ and $a \in \cC$, then the bound improves to
\begin{align*}
\max_{b \in \cC} \sum_{t=1}^n \left(\sum_{a \in \cC} q_t(a) \ell_t(a) - \ell_t(b)\right) \leq \frac{\log|\cC|}{\eta} + \frac{\eta}{2} \sum_{t=1}^n \sum_{a \in \cC} q_t(a) \ell_t(a)^2\,.
\end{align*}
The proof is the same except you may now use that $\exp(-x) \leq 1 - x + x^2/2$ for $x \geq 0$.
\end{remark}

\section{Continuous Exponential Weights}\label{sec:lin:cont} \index{continuous exponential weights|(}

The material in this section is not used by any algorithm in this book.
It is included because it played a fundamental role in one of the most influential papers on convex bandits \citep{BEL16} and
may be useful in future algorithms.
The exponential weights distribution in \cref{eq:lin:exp} is defined only for finite $\cC$.
This is sometimes desirable, as we discuss in Note~\ref{note:lin:cont}. Generally though, in applications to convex bandits you need $\cC$ to be a cover of $K$
and for this $|\cC|$ is exponentially large in the dimension. This is not such a problem from a sample efficiency perspective ($|\cC|$ appears in a logarithm in \cref{thm:exp-discrete}) 
but is a disaster computationally.
Continuous exponential weights is a beautiful alternative that sometimes leads to computational improvements. Suppose that $\vol(K) > 0$ and let $\ell_1,\ldots,\ell_n \colon K \to \R$
be a sequence of measurable functions.\index{measurable} The continuous exponential weights distribution is
\begin{align*}
q_t(x) = \frac{\exp\left(-\eta \sum_{s=1}^{t-1} \ell_s(x)\right)}{\int_K \exp\left(-\eta \sum_{s=1}^{t-1} \ell_s(y)\right) \d{y}}\,,
\end{align*}
which, provided it exists, is a density supported on $K$.
Given a probability density $p$ supported on $K$, let
\begin{align*}
\hReg_n(p) = \sum_{t=1}^n \int_K \ell_t(x) \left(q_t(x) - p(x)\right) \d{x} \,, 
\end{align*}
which is the regret of continuous exponential weights relative to the density $p$.

\begin{theorem}\label{thm:lin:cont}
Suppose that $\eta |\ell_t(x)| \leq 1$ for all $x \in K$ and $1 \leq t \leq n$.
Then, for any density $p$ on $K$,
\begin{align*}
\hReg_n(p) \leq \frac{1}{\eta} \int_K p(x) \log\left(\frac{p(x)}{q_1(x)}\right) \d{x} + \eta \sum_{t=1}^n \int_K q_t(x) \ell_t(x)^2 \d{x} \,.
\end{align*}
\end{theorem}

The first term on the right-hand side is the relative entropy\index{relative entropy} between $p$ and $q_1$, which is the uniform distribution on $K$.
\cref{rem:lin:cvx} applies here as well. There is no requirement that the losses $(\ell_t)$ or the constraint set $K$ are convex, though
the exponential weights distribution is log-concave if they are.
The regret relative to a distribution is not entirely satisfactory. Ideally you want to choose $p$ as a Dirac on the minimiser of $\sum_{t=1}^n \ell_t$, but this does not have
a density. Unsurprisingly the idea is to choose $p$ to be concentrated close to a minimiser. Exactly how you do this depends on the structure of the losses.
When the losses are bounded and $K$ is convex and bounded, then the situation is especially clean:

\begin{corollary}\label{cor:lin:cont}
Suppose that $K$ is a convex body, $\sum_{t=1}^n \ell_t$ is convex and $d \leq 2n$. Then,
under the same conditions as \cref{thm:lin:cont},
\begin{align*}
\hReg_n(x) 
&= \sum_{t=1}^n \left(\int_K \ell_t(y) q_t(y) \d{y} - \ell_t(x)\right) \\
&\leq \frac{d}{\eta} \left[1 + \log\left(\frac{2n}{d}\right)\right] + \eta \sum_{t=1}^n \int_K q_t(y) \ell_t(y)^2 \d{y}\,.
\end{align*}
\end{corollary}

\begin{exer}
\faStar \quad
Prove \cref{cor:lin:cont} by taking $p$ as the uniform distribution on $(1 - \eps) x + \eps K$ for suitable $\eps \in [0,1]$.
\end{exer}

\solution{
As suggested, let $p$ be the uniform distribution on $J = (1 - \eps) x + \eps K$, which has density
\begin{align*}
p(y) = \frac{\sind_{J}(y)}{\eps^d \vol(K)} \,.
\end{align*}
Furthermore, by convexity, for all $y \in J$ there exists a $z \in K$ such that
\begin{align*}
\sum_{t=1}^n \ell_t(y) 
&\leq \sum_{t=1}^n \ell_t(x) + \eps \sum_{t=1}^n (\ell_t(z) - \ell_t(x)) \\
&\leq \sum_{t=1}^n \ell_t(x) + \frac{2\eps n}{\eta} \,.
\end{align*}
Moreover,
\begin{align*}
\int_K p(y) \log\left(\frac{p(y)}{q_1(y)}\right) \d{y} = d \log\left(\frac{1}{\eps}\right) \,.
\end{align*}
Hence, by \cref{thm:lin:cont},
\begin{align*}
\hReg_n(x) \leq \frac{d}{\eta} \log\left(\frac{1}{\eps}\right) + \eta \sum_{t=1}^n \int q_t(x) \ell_t(x)^2 \d{x} + \frac{2\eps n}{\eta} \,.
\end{align*}
The bound is minimised by $\eps = d/(2n) \in [0,1]$ 
}

Computationally the continuous exponential weights distribution has some nice properties.  
Most notably, if $K$ is convex and the cumulative loss $\sum_{s=1}^{t-1} \ell_s$ is convex, then $q_t$ is log-concave and under mild additional assumptions can be sampled from
approximately in polynomial time \citep{chewi24}.

\index{continuous exponential weights|)}
\section{Linear Bandits}\label{sec:linear:linear} \index{bandit!linear|(}

For this section we assume there is no noise and that the losses are bounded, linear and homogeneous:

\begin{assumption}\label{ass:linear}
The following hold:
\begin{enumerate}
\item There is no noise: $\eps_t = 0$ for all $t$.
\item There exists a sequence $(\theta_t)_{t=1}^n \in \R^d$ such that $f_t = \ip{\cdot, \theta_t}$.
\item The losses are bounded: $(f_t) \in \cF_\pb$.
\end{enumerate}
\end{assumption}

\begin{remark}
Cautious readers may notice that the above assumptions do not correspond to $\cF_{\pb,\plin}$ because the representation has been chosen so that the losses are homogeneous.
In the notes we explain a simple way to reduce the inhomogeneous setting to the homogeneous one.
Note also that the global assumption that $K$ is convex is not actually used in this section, only that it is nonempty and compact.
\end{remark}

The plan is to use exponential weights\index{exponential weights} on a finite $\cC \subset K$ that is sufficiently large that the optimal action in $K$ can be
approximated by something in $\cC$. Let $A = \conv(K \cup (- K))$, which is a symmetric convex body.
Recall the definitions of $A^\circ$ and $\norm{\cdot}_A$ and $\norm{\cdot}_{A^\circ}$ in \cref{sec:regularity:minkowski}.
By the assumption that the losses are bounded, for any $t$,
$1 \geq \max_{x,y \in K} |\ip{x - y, \theta_t}| = \norm{\theta_t}_{A^\circ}$ and therefore
\begin{align*}
\theta_t \in \Theta = \left\{\theta \in \R^d \colon \norm{\theta}_{A^{\circ}} \leq 1\right\}\,.
\end{align*}
What we need from $\cC$ is that for all $\theta_t (\in \Theta)$ and $y \in K$ there exists an $x \in \cC$ such that $f_t(x) - f_t(y)$ is small. Precisely, we need \index{covering number}
\begin{align*}
\max_{y \in K} \min_{x \in \cC} \ip{x - y, \theta} \leq \frac{1}{n} = \eps\,.
\end{align*}
Suppose that $\norm{x - y}_A \leq \eps$; then by \cref{prop:polar-dual} $\ip{x - y, \theta} \leq \norm{x - y}_A \norm{\theta}_{A^\circ} \leq \eps$.
Hence, it suffices to choose $\cC \subset K$ such that $K \subset \bigcup_{x \in \cC} (x + \eps A)$. By \cref{prop:cover2}, such a cover exists with
\begin{align}
|\cC| \leq \left(1 + \frac{4}{\eps}\right)^d\,.
\label{eq:exp:linear-cover}
\end{align}
The algorithm for linear bandits plays actions in $\cC$ and uses importance-weighted least squares (in Line~\ref{line:alg:linear:importance}) to estimate $\ip{x, \theta_t}$ 
for all $x \in \cC$. 
The distribution proposed by exponential weights is mixed with a small amount of an optimal design on $\cC$, which is needed so that the estimates
are suitably bounded as required by \cref{thm:exp-discrete}.

\begin{algorithm}[h!]
\begin{algcontents}
\begin{lstlisting}
args: $\eta > 0$, $\gamma \in (0,1)$, $K$
find $\cC \subset K$ such that $\displaystyle \max_{y \in K} \min_{x \in \cC} \norm{x - y}_A \leq \frac{1}{n}$
find optimal design $\pi$ on $\cC$ $\text{(see \cref{thm:kwx})}$
for $t = 1$ to $n$
  let $q_t(x) = \frac{\exp\left(-\eta \sum_{u=1}^{t-1} \sip{x, \hat \theta_u}\right)}{\sum_{y \in \cC} \exp\left(-\eta \sum_{u=1}^{t-1} \sip{y, \hat \theta_u}\right)}$
  let $p_t = (1 - \gamma) q_t + \gamma \pi$
  sample $X_t$ from $p_t$ and observe $Y_t = f_t(X_t)$
  let $G_t = \sum_{a \in \cC} p_t(a) aa^\top$ and $\hat \theta_t = G_t^+ X_t Y_t$ $\label{line:alg:linear:importance}$
\end{lstlisting}
\caption{Exponential weights for linear bandits}
\label{alg:linear}
\end{algcontents}
\end{algorithm}

\FloatBarrier

As mentioned, \cref{alg:linear} does not need $K$ to be a convex body. And indeed, in many applications $K$ is a finite set (see Note~\ref{note:lin:cont}).
When $K$ is a convex body, you can use continuous exponential weights (\cref{sec:lin:cont}) instead of the discretisation.\index{continuous exponential weights} 

\begin{exer}
\faStar \quad
Replace the discrete exponential weights in \cref{alg:linear} with continuous exponential weights from \cref{sec:lin:cont} and 
adapt the proof of \cref{thm:linear} below to prove that for $n \geq 2d$ the regret of this algorithm is upper-bounded by
\begin{align*}
\E[\Reg_n] = O\left(d \sqrt{n \log(n/d)}\right) \,. 
\end{align*}
\end{exer}

\begin{theorem}\label{thm:linear}
Suppose that
\begin{align*}
\eta = \sqrt{\frac{\log |\cC|}{2nd}} \quad \text{and} \quad
\gamma = \eta d \,.
\end{align*}
Under \cref{ass:linear}, the regret of \cref{alg:linear} is bounded by
\begin{align*}
\E[\Reg_n] \leq 1 +  \sqrt{8nd \log |\cC|}\,. 
\end{align*}
\end{theorem}

Note that by \cref{eq:exp:linear-cover}, $\cC$ can be chosen so that $\log |\cC| \leq d \log(1 + 4n)$ and in this case one has $\E[\Reg_n] = O(d \sqrt{n \log(n)})$. 

\begin{proof} 
The algorithm is only well-defined if $\gamma \in [0,1]$. Suppose that
$\gamma \geq 1$; then 
\begin{align*}
1 \leq \eta d = \sqrt{\frac{d \log|\cC|}{2n}}\,.
\end{align*}
In this case, no matter how the actions are chosen, boundedness of the losses implies that
$\Reg_n \leq n \leq \sqrt{nd \log(|\cC|)/2}$. Suppose for the remainder that $\gamma \in (0,1)$.
Recall the definition of $G_\pi$ from \cref{sec:linear:design} and let 
\begin{align*}
G_{q_t} = \sum_{x \in \cC} q_t(x) xx^\top\,.
\end{align*}
With this notation, $G_t = (1 - \gamma) G_{q_t} + \gamma G_\pi$.
Of course $G_t \succeq (1 - \gamma) G_{q_t}$ and $G_t \succeq \gamma G_\pi$.
Since $q_t$ is strictly positive, $\ker(G_{q_t}) = \laspan(\cC)^\perp$ and by assumption $\ker(G_\pi) = \laspan(\cC)^\perp$.
Hence, by \cref{fact:pinv-lowner},
\begin{align}
G_t^+ \preceq \frac{1}{1 - \gamma} G_{q_t}^+ \qquad \text{ and } \qquad G_t^+ \preceq \frac{1}{\gamma} G_\pi^+ \,.
\label{eq:linear:lowner}
\end{align}
Next, for any $y \in \cC$,
\begin{align}
\ip{y, \E_{t-1}[\hat \theta_t]}
&= \ip{y, \sum_{x \in \cC} p_t(x) G_t^+ x \ip{x, \theta_t}} = \ip{y, G_t^+ G_t \theta_t} = \ip{y, \theta_t} \,,
\label{eq:linear:unbiased}
\end{align}
where in the final inequality we use the assumption that $G_\pi$ is an optimal design for $\cC$ (see \cref{thm:kwx}), 
so that $\cC \subset \im(G_\pi^\top) \subset \im(G_t^\top)$, and \cref{fact:pinv}.
Furthermore,
\begin{align}
\E_{t-1}\big[\sip{x, \hat \theta_t}^2 \big]
&= \E_{t-1}\big[Y_t^2 x^\top G_t^+ X_tX_t^\top G_t^+ x\big] 
\leq \norm{x}^2_{G_t^+ G_t G_t^+} 
= \norm{x}^2_{G_t^+} \,,
\label{eq:linear:quad}
\end{align}
where we used the fact that $Y_t^2 \leq 1$.
Therefore,
\begin{align}
\E_{t-1}\left[\sum_{x \in \cC} q_t(x) \sip{x, \hat \theta_t}^2\right]
\tag*{by \cref{eq:linear:quad}}
&\leq \sum_{x \in \cC} q_t(x) \norm{x}^2_{G_t^{+}} \nonumber \\
\tag*{by \cref{eq:linear:lowner}}
&\leq \frac{1}{1 - \gamma} \sum_{x \in \cC} q_t(x) \norm{x}^2_{G_{q_t}^{+}} \nonumber \\
&= \frac{1}{1-\gamma} \tr(G^{\vphantom{+}}_{q_t} G_{q_t}^+) \nonumber \\
&\leq \frac{d}{1 - \gamma} \,,
\label{eq:linear:squared}
\end{align}
where the last inequality follows from \cref{fact:pinv-trace}.
In order to apply \cref{thm:exp-discrete} the loss estimates need to be suitably bounded. 
This is where the exploration using experimental design comes into play. For any $x \in \cC$,
\begin{align*}
\eta |\sip{x, \hat \theta_t}|
&= \eta \left|x^\top G_t^+ X_t Y_t\right| 
\leq \eta \snorm{x}_{G_t^+} \snorm{X_t}_{G_t^+} 
\leq \frac{\eta d}{\gamma} = 1 \,,
\end{align*}
where we used \cref{eq:linear:lowner} and the fact that $\pi$ is an optimal design on $\cC$ so that
by \cref{thm:kw}, for any $x \in \cC$, $\norm{x}_{G_\pi^+} \leq \sqrt{d}$.
Define $x_\star = \argmin_{x \in \cC} \sum_{t=1}^n \ip{x, \theta_t}$.
Then the regret is bounded by
\begin{align}
\E[\Reg_n]
&= \max_{x \in K} \E\left[\sum_{t=1}^n \ip{X_t - x, \theta_t}\right] \nonumber \\
&= \max_{x \in K} \sum_{t=1}^n \ip{x_\star - x, \theta_t} + \E\left[\sum_{t=1}^n \ip{X_t - x_\star, \theta_t}\right] \,.
\label{eq:linear:decomp}
\end{align}
The first term in \cref{eq:linear:decomp} is bounded by
\begin{align*}
\max_{x \in K} \sum_{t=1}^n \ip{x_\star - x, \theta_t}
&=  \max_{x \in K} \min_{y \in \cC} \sum_{t=1}^n \ip{y - x, \theta_t} \\
&\leq \max_{x \in K} \min_{y \in \cC} \sum_{t=1}^n \norm{y - x}_A \norm{\theta_t}_{A^\circ} \\
&\leq \max_{x \in K} \min_{y \in \cC} n \norm{y - x}_A \\
&\leq 1\,,
\end{align*}
where we used Cauchy--Schwarz, 
the fact that $\theta_t \in \Theta = \{\theta : \norm{\theta}_{A^\circ} \leq 1\}$ and the definition of the cover $\cC$.
The second term in \cref{eq:linear:decomp} is bounded using \cref{thm:exp-discrete} by
\begin{align*}
&\E\left[\sum_{t=1}^n \sip{X_t - x_\star, \theta_t}\right] 
= \E\left[\sum_{t=1}^n \sum_{x \in \cC} p_t(x) \sip{x - x_\star, \theta_t}\right] \\
&\quad\leq n \gamma + (1 - \gamma) \E\left[\sum_{t=1}^n \sum_{x \in \cC} q_t(x) \sip{x - x_\star, \theta_t}\right] \\
\tag*{by \cref{eq:linear:unbiased}}
&\quad= n \gamma + (1 - \gamma) \E\left[\sum_{t=1}^n \sum_{x \in \cC} q_t(x) \ip{x - x_\star, \hat \theta_t}\right] \\
\tag*{by \cref{thm:exp-discrete}}
&\quad\leq n \gamma + (1 - \gamma)\left(\frac{\log|\cC|}{\eta} + \E\left[\sum_{t=1}^n \sum_{x \in \cC} q_t(x) \sip{x, \hat \theta_t}^2\right]\right) \\
\tag*{by \cref{eq:linear:squared}}
&\quad\leq n \gamma + \frac{\log|\cC|}{\eta} + \eta n d \,,
\end{align*}
where in the first inequality we used the fact that $p_t = (1 - \gamma) q_t + \gamma \pi$ and the assumption that $(f_t)$ are bounded so that
\begin{align*}
\sum_{x \in \cC} \pi(x) \ip{x - x_\star, \theta_t} = \sum_{x \in \cC} \pi(x) (f_t(x) - f_t(x_\star)) \leq 1\,.
\end{align*}
The result follows by substituting the constants.
\end{proof}
\index{bandit!linear|)}

\section{Quadratic Bandits}\index{bandit!quadratic|(}

Quadratic bandits seem much harder than linear bandits. But if you ignore the computation complexity then it turns out that quadratic bandits \textit{are} linear bandits.

\begin{assumption}\label{ass:quadratic}
There is no noise and
the loss functions are quadratic and bounded: $(f_t) \in \cF_{\pb,\pquad}$.
\end{assumption}

Quadratic bandits can be viewed as linear bandits by introducing a kernel feature map.
Let $d_2 = \frac{d^2 + 3d + 2}{2}$ and
define a function $\phi(x) \colon \R^d \to \R^{d_2}$ by 
\begin{align*}
\phi(x) = (1, x_1,\ldots,x_d,x_1^2, x_1 x_2, \ldots, x_1 x_d, x_2^2,x_2x_3, \ldots, x_{d-1}^2, x_{d-1}x_d, x_d^2) \,.
\end{align*}
So $\phi$ is the feature map associated with the polynomial kernel of degree 2.
You should check that any $f \in \cF_{\pb,\pquad}$ can be written as $f(x) = \ip{\phi(x), \theta}$ for some $\theta \in \R^{d_2}$.

\begin{lemma}\label{lem:cover} \index{covering number}
Let $K \subset \R^d$ be a convex body. There exists a cover $\cC \subset K$ such that
\begin{enumerate}
\item $\max_{x \in K} \min_{y \in \cC} |f(x) - f(y)| \leq \eps$ for all $f \in \cF_{\pb,\pquad}$; and
\item $|\cC| \leq \left(1 + \frac{24d^{3.5}}{\eps}\right)^d$.
\end{enumerate}
\end{lemma}

\begin{proof}
We provide the proof when
$\ball_1 \subset K \subset d \ball_1$ and leave the general case as an exercise.
Let $\cC$ be a cover of $K$ such 
\begin{align*}
\max_{x \in K} \min_{y \in \cC} \norm{x - y} \leq \frac{\eps}{6d^{2.5}}\,.
\end{align*}
By \cref{prop:cover3}, $\cC$ can be chosen so that
\begin{align*}
|\cC| \leq \left(1 + \frac{24d^{3.5}}{\eps}\right)^d \,.
\end{align*}
Let $f \in \cF_{\pb,\pquad}$ and let $\theta$ be such that $f(x) = \ip{\phi(x), \theta}$. Then, since $\ball_1 \subset K$ by assumption,
\begin{align*}
1 \geq \max_{x \in K} |f(x)|
\geq \max_{x \in \ball_1} |f(x)| 
\geq \frac{1}{2} \norm{\theta}_\infty\,,
\end{align*}
where the last inequality is left as an exercise:

\begin{exer}\label{ex:lin:theta-norm}
\faStar\quad
Show that $\max_{x \in \ball_1} |f(x)| \geq \frac{1}{2} \norm{\theta}_\infty$.
\end{exer}

\solution{Let $M = \max_{x \in \ball_1} |f(x)|$ and $N = \norm{\theta}_\infty$ and
let $f(x) = a + \ip{b,x} + x^\top C x$ where $C$ is symmetric and note that
\begin{align*}
N = \max\left(|a|, \norm{b}_\infty, C_{1,1},\ldots,C_{d,d},2C_{1,2},2C_{1,3},\ldots,2C_{d-1,d}\right) \,.
\end{align*}
By definition 
\begin{align*}
a = |f(\zeros)| \leq M \,.
\end{align*}
Let $i \in \{1,\ldots,d\}$. Then
\begin{align*}
2 |b_i| = |f(e_i) - f(-e_i)| \leq 2 M \,.
\end{align*}
Moreover, $2|C_{i,i}| = |f(e_i) + f(-e_i) - 2f(\zeros)| \leq 4 M$.
Next, let $i \neq j \in \{1,\ldots,d\}$ and $v_{\pm\pm} = \frac{1}{\sqrt{2}}(\pm e_i \pm e_j)$.
Then
\begin{align*}
f(v_{++}) &= a + \frac{b_i + b_j}{\sqrt{2}} + \frac{C_{ii}}{2} + \frac{C_{jj}}{2} + C_{ij} \\ 
f(v_{+-}) &= a + \frac{b_i - b_j}{\sqrt{2}} + \frac{C_{ii}}{2} + \frac{C_{jj}}{2} - C_{ij} \\  
f(v_{-+}) &= a + \frac{-b_i + b_j}{\sqrt{2}} + \frac{C_{ii}}{2} + \frac{C_{jj}}{2} - C_{ij} \\ 
f(v_{--}) &= a + \frac{-b_i - b_j}{\sqrt{2}} + \frac{C_{ii}}{2} + \frac{C_{jj}}{2} + C_{ij} \,. 
\end{align*}
Therefore
\begin{align*}
4 |C_{i,j}| &= |f(v_{++}) - f(v_{+-}) - f(v_{-+}) + f(v_{--})| \leq 4 M
\end{align*}
Therefore $2 |C_{i,j}| \leq 2 M$.
Combining all the parts shows that $M \geq N/2$ as required.
}

Let $x \in K$ be arbitrary and $y \in \cC$ be such that $\norm{x -y} \leq \frac{\eps}{6d^{2.5}}$. Then, 
\begin{align*}
|f(x) - f(y)|
= \left|\ip{\phi(x) - \phi(y), \theta}\right|
\leq 2\snorm{\phi(x) - \phi(y)}_1
\leq 6d^{2.5} \norm{x - y} 
\leq \eps
\end{align*}
where the first inequality follows from Cauchy--Schwarz and the fact that $\norm{\theta}_\infty \leq 4$ by Exercise~\ref{ex:lin:theta-norm}. 
The second follows because $x \in \ball_d$ and by using \cref{ex:lin:phi} below.
\end{proof}

\begin{exer}\label{ex:lin:phi}
\faStar \quad
Suppose that $x, y \in \ball_d$.
Show that 
\begin{align*}
\norm{\phi(x) - \phi(y)}_1 \leq 3d^{2.5} \norm{x - y}_\infty \,.
\end{align*}
\end{exer}

\solution{
Abbreviate $\Delta = \norm{x - y}_\infty$.
By definition,
\begin{align*}
\norm{\phi(x) - \phi(y)}_1
&= \sum_{i=1}^d |x_i - y_i| + \sum_{i=1}^d \sum_{j=1}^i |x_i x_j - y_i y_j| \\
&\leq d \Delta + \sum_{i=1}^d \sum_{j=1}^i |x_i x_j - y_i y_j| \\
&= d \Delta + \frac{1}{2} \sum_{i=1}^d \sum_{j=1}^i \left|(x_i - y_i)(x_j + y_j) + (x_i+y_i)(x_j - y_j)\right| \\
&\leq d \Delta + \frac{\Delta}{2} \sum_{i=1}^d \sum_{j=1}^i \left(|x_j + y_j| + |x_i + y_i|\right) \\
&\leq d \Delta + d \Delta \sum_{i=1}^d (|x_i| + |y_i|) \\
&\leq d \Delta + d^{1.5} \Delta (\norm{x} + \norm{y}) \\
&\leq 3 \Delta d^{2.5} \,.
\end{align*}
}

\begin{exer}
\faStar \quad
Prove \cref{lem:cover} for arbitrary convex bodies $K$. You may use the fact that for any convex body $K$ there exists an affine map\index{affine!map} $T$ such that $\ball_1 \subset TK \subset d \ball_1$, which follows from \cref{thm:john}. 
\end{exer}

We can now simply write the kernelised version of \cref{alg:linear}.

\begin{algorithm}[h!]
\begin{algcontents}
\begin{lstlisting}
args: $\eta > 0$, $\gamma  \in (0,1)$ and $K$
find $\cC \subset K$ satisfying conds. in Lemma $\ref{lem:cover}$, $\eps=\frac{1}{n}$
find optimal design $\pi$ on $\{\phi(a) : a \in \cC\}$
for $t = 1$ to $n$
  let $q_t(x) = \frac{\exp\left(-\eta \sum_{u=1}^{t-1} \ip{\phi(x), \hat \theta_u}\right)}{\sum_{x \in \cC} \exp\left(-\eta \sum_{u=1}^{t-1} \ip{\phi(x), \hat \theta_u}\right)}$
  let $p_t = (1 - \gamma) q_t + \gamma \pi$
  sample $X_t$ from $p_t$ and observe $Y_t = f_t(X_t)$
  let $G_t = \sum_{a \in \cC} p_t(a) \phi(a) \phi(a)^\top$ and $\hat \theta_t = G_t^+ X_t Y_t$
\end{lstlisting}
\caption{Exponential weights for quadratic bandits}\label{alg:quadratic}
\end{algcontents}
\end{algorithm}

\FloatBarrier

An immediate corollary of \cref{thm:linear} is the following
bound on the regret of \cref{alg:quadratic}:

\begin{theorem}\label{thm:quadratic}
Suppose that
\begin{align*}
\eta =  \sqrt{\frac{\log |\cC|}{2nd_2}} \quad \text{and} \quad
\gamma = d_2 \eta \,.
\end{align*}
Under \cref{ass:quadratic} the expected regret of \cref{alg:quadratic} is bounded by
\begin{align*}
\E[\Reg_n] \leq 1 + \sqrt{8 nd_2 \log |\cC|} = O\left(d^{1.5} \sqrt{n \log(n d)}\right) \,.
\end{align*}
\end{theorem}

Note that we did not use anywhere that $\cF_{\pb,\pquad}$ only included convex quadratics. Everything works for more general quadratic losses.
Even if we restrict our attention to convex quadratics, none of the algorithms we have presented so far can match this bound.
Actually no efficient algorithm is known matching this bound except for special $K$.

\index{bandit!quadratic|)}

\section{Notes}

\begin{enumeratenotes}

\item We promised to explain how to handle inhomogeneous linear losses.
The simple solution is to let $\phi(x) = (1, x)$ and run the algorithm for linear bandits on $\phi(K)$.

\item Even when $K$ is convex, (non-convex)\index{non-convex} quadratic programming\index{quadratic programming} is computationally hard. For example, 
when $K$ is a simplex\index{simplex} and $A$ is the adjacency matrix of an undirected graph $G$, then \index{adjacency matrix}
a theorem by \cite{motzkin1965maxima} says that
\begin{align*}
\frac{1}{2} \min_{x \in \Delta_d} (-x^\top A x) = \frac{1}{2}\left(\frac{1}{\omega(G)} - 1\right)\,,
\end{align*}
where $\omega(G)$ is the size of the largest clique in $G$. Since the clique decision problem is NP-complete \citep{karp1972reducibility}, there (probably) 
does not exist an efficient algorithm for minimising non-convex quadratic functions over the simplex.
\item By the previous note, \cref{alg:quadratic} cannot be implemented efficiently, since its analysis did not make use of convexity of the losses. 
Sadly, even if we restrict our
attention to convex quadratic loss functions and convex $K$, the kernel method does not seem amenable to efficient calculations. Why not? 
The problem is that if $K$ is a convex body, then the set $J = \{\phi(a) \colon a \in K\}$ is not convex. 

\item \cref{alg:linear} is by \cite{BCK12} and
\cref{alg:quadratic} is essentially due to \cite{chatterji2019online}. 
Note that the latter consider kernelised bandits\index{bandit!kernelised} more generally. You could replace the quadratic kernel with a higher-degree polynomial (or even infinite-dimensional feature map) 
and obtain regret bounds for a larger class of losses.

\item Note that $\cF_{\pb,\plin} = \cF_{\pb,\psm}$ when $\beta = 0$. This means that \cref{thm:ftrl:smooth} shows that \cref{alg:ftrl:basic} 
has regret $\E[\Reg_n] \leq 1 + 3d \sqrt{\vartheta n \log(n)}$. 
The algorithm is efficient when the learner has access to a $\vartheta$-self-concordant barrier.\index{self-concordant barrier}
This special case was analysed in the first application of self-concordance to bandits by \cite{AHR08}.

\item \label{note:lin:cont}
Continuous exponential weights is generally preferable computationally when the losses and constraint set\index{constraint set} are convex, at least when the dimension is not tiny.
In many applications of linear bandits $K$ is a moderately sized finite set. For example, $K$ might be a set of features associated with books to be recommended. 
In these cases \cref{alg:linear} is a good choice with $|\cC| = K$.
We will see an example of exponential weights for convex bandits in the next chapter and with $d = 1$ where either the continuous or discrete versions of exponential weights
could be employed, with the latter moderately more computationally efficient.

\item 
For some time it was open whether or not $d \sqrt{n \polylog(n, d)}$ regret is possible for linear bandits with a computationally efficient algorithm when $K$ 
is a convex body represented by a membership or separation oracle.\index{separation oracle}
This was resolved by \cite{HKM16}, who used continuous exponential weights but replaced the Kiefer--Wolfowitz distribution with a spanner that can be computed efficiently.\index{continuous exponential weights} 
The only reason not to include that algorithm here is that we are primarily focused on convex bandits and because the analysis is more sophisticated.
\end{enumeratenotes}

\chapter[Exponential Weights]{Exponential Weights\copynotice}\label{chap:exp}

We already saw an application of exponential weights\index{exponential weights} to linear and quadratic bandits in \cref{chap:lin}. The same abstract algorithm can also be used
for convex bandits but the situation is more complicated. 
Throughout this chapter we assume the losses are bounded and there is no noise:

\begin{assumption}\label{ass:exp}
The following hold:
\begin{enumerate}
\item The losses are in $\cF_\pb$. 
\item There is no noise, so that $Y_t = f_t(X_t)$.
\end{enumerate}
\end{assumption}

Two main topics will be covered here:
\begin{itemize}
\item a relatively practical algorithm for adversarial bandits when $d = 1$ with $O(\sqrt{n} \log(n))$ regret;
\item the connection via minimax duality\index{duality!minimax} between mirror descent and the information ratio \citep{RV14} 
and its application to convex bandits. 
\end{itemize}
The relation between the two topics is that both make use of an abstract version of exponential weights. 

\section{Exponential Weights for Convex Bandits}

The version of exponential weights\index{exponential weights} introduced in \cref{chap:lin} is designed for a finite action set.
As in that chapter, we will apply this algorithm to a discretisation of $K$, which is given in an abstract form in \cref{alg:exp:abstract} below.

\begin{algorithm}[h!]
\begin{algcontents}
\begin{lstlisting}
args: learning rate $\eta > 0$
let $\cC \subset K$ be finite
for $t = 1$ to $n$
  compute $q_t(x) = \frac{\exp\left(-\eta \sum_{u=1}^{t-1} \hat s_u(x)\right)}{\sum_{y \in \cC} \exp\left(-\eta \sum_{u=1}^{t-1} \hat s_u(y)\right)}$ for all $x \in \cC$
  find distribution $p_t$ as a function of $q_t$
  sample $X_t$ from $p_t$ and observe $Y_t = f_t(X_t)$
  compute $\hat s_t(x)~ \forall x \in \cC$ using $p_t, q_t, X_t$ and $Y_t$
\end{lstlisting}
\caption{Exponential weights for bandits}\label{alg:exp:abstract}
\end{algcontents}
\end{algorithm}

\FloatBarrier

As with gradient descent, to make this concrete we need a sampling distribution $p_t$ and a mechanism for estimating the loss function.
The extra complication here is that while for gradient descent we only needed to estimate a gradient, here
the algorithm needs to estimate an entire function from a single observation.
From a computational perspective there is a serious problem in that $\cC$ will need to be exponentially large in $d$.
A potential way to mitigate this is to use continuous exponential weights (\cref{sec:lin:cont}), which was the approach taken by \cite{BEL16} \index{continuous exponential weights} and (non-obviously) by the algorithm in \cref{chap:ons}.
The reason we avoid this here is twofold:
\begin{itemize}
\item For the one-dimensional algorithm it turns out that the discrete version is more computationally efficient
and simpler.
\item The minimax duality arguments are already computationally inefficient and using continuous exponential weights only serves to introduce
measure-theoretic challenges.
\end{itemize}

\section{Exponential Weights in One Dimension}\label{sec:exp:1d}
For this section assume that $d = 1$ and $K = [-1,1]$ is the interval, and let $\eps = 1/\sqrt{n}$ and 
\begin{align*}
\cC = \{k \eps \colon k \in \mathbb Z, k \eps \in K\}\,.
\end{align*}
Let $x_\star = \argmin_{x \in \cC} \sum_{t=1}^n f_t(x)$.
The plan is to apply \cref{thm:exp-discrete} to \cref{alg:exp:abstract} by carefully choosing the exploration distributions $(p_t)$ and
a mechanism for constructing the estimated surrogate losses $(\hat s_t)$.
Provided the estimated surrogates are non-negative and letting $s_t(x) = \E_{t-1}[\hat s_t(x)]$, then
\cref{thm:exp-discrete} shows that
\begin{align}
&\E\left[\sum_{t=1}^n \left(\sum_{x \in \cC} q_t(x) s_t(x) - s_t(x_\star)\right)\right]
=\E\left[\sum_{t=1}^n \left(\sum_{x \in \cC} q_t(x) \hat s_t(x) - \hat s_t(x_\star)\right)\right] \nonumber \\
&\qquad\qquad\qquad\qquad\leq \frac{\log |\cC|}{\eta} + \frac{\eta}{2} \E\left[\sum_{t=1}^n \sum_{x \in \cC} q_t(x) \hat s_t(x)^2\right] \,.
\label{eq:exp-1d-1}
\end{align}
\cref{alg:exp:abstract} samples $X_t$ from distribution $p_t$, which means the expected regret relative to $x_\star$ is
\begin{align}
\E[\Reg_n(x_\star)] = \sum_{t=1}^n \E\left[\sum_{x \in \cC} p_t(x) f_t(x) - f_t(x_\star)\right] \,.
\label{eq:exp-1d-2}
\end{align}
The question is how to choose loss estimates $\hat s_t$ and exploration distribution $p_t$ so that \cref{eq:exp-1d-2} can be connected to the left-hand side of \cref{eq:exp-1d-1}
and at the same time the right-hand side of \cref{eq:exp-1d-1} is well controlled.

\subsubsection*{Kernel-Based Estimation}
Let us focus on a single round $t$. Let $f \colon K \to [0,1]$ be convex and
$q \in \Delta(\cC)$ and assume that $q(x) > 0$ for all $x \in \cC$. The learner samples $X$ from $p \in \Delta(\cC)$, observes $Y = f(X)$ and uses $p, q, X$ and $Y$ to construct a surrogate estimate 
$\hat s \colon \cC \to \R$ with expectation $s(x) = \E[\hat s(x)]$.
Staring at \cref{eq:exp-1d-1} and \cref{eq:exp-1d-2}, we will be in business 
if we choose $p$ and $\hat s$ in such a way that all of the following hold:
\begin{enumerate}
\item There exist (preferably) small constants $A, B > 0$ such that
\begin{align}
\sum_{x \in \cC} p(x) f(x) - f(x_\star) \leq A\left[ \sum_{x \in \cC} q(x) s(x) - s(x_\star)\right] + B \,.
\label{eq:exp:want}
\end{align}
\item $\E\left[\sum_{x \in \cC} q(x) \hat s(x)^2\right] = \tilde O(1)$.
\item $\hat s(x) \geq 0$ for all $x \in \cC$.
\end{enumerate}

\begin{remark}
The requirement that $\hat s$ is non-negative can be relaxed to $\eta |\hat s(x)| \leq 1$ for all $x$ where $\eta = O(1/\sqrt{n})$ is the learning rate.
See Remark~\ref{rem:exp-discrete}.
\end{remark}

Because $x_\star$ is not known, the most obvious idea is to show that \cref{eq:exp:want} holds for all points in $\cC$, not just $x_\star$. 
An elegant way to construct a surrogate satisfying these properties is by using a kernel. \index{surrogate loss!kernel}\index{probability kernel}
Let $T \colon \cC \times \cC \to \R$ be a function with $x \mapsto T(x | y)$ a probability distribution for all $y$ and define 
\begin{align*}
p(x) &= (Tq)(x) \triangleq \sum_{y \in \cC} T(x | y) q(y) \qquad \text{and} \\
s(y) &= (T^* f)(y) \triangleq \sum_{x \in \cC} T(x | y) f(x) \,.
\end{align*}
It may be helpful to think of $T$ as a $|\cC| \times |\cC|$ matrix and $T^*$ as its transpose (or being fancy, its adjoint). 
Viewing $p, q, f, s$ as vectors in $\R^{|\cC|}$ we have $p = Tq$ and $s = T^*f$.
Because the map $x \mapsto T(x | y)$ is a probability distribution, the surrogate $s$ is some kind of smoothing of $f$. And in fact the surrogates used in earlier chapters
also have this form, though in continuous spaces.
Notice that
\begin{align}
\sum_{y \in \cC} s(y) q(y) 
&= \sum_{y \in \cC} \left(\sum_{x \in K} T(x | y) f(x) \right) q(y) \nonumber \\
&= \sum_{x \in \cC} f(x) \left(\sum_{y\in K} T(x | y) q(y) \right) \nonumber \\
&= \sum_{x \in \cC} f(x) p(x) \,.
\label{eq:exp-kernel}
\end{align}
This is the motivation for choosing $p = Tq$.
Equivalently, using linear algebra notation: $\ip{q, s} = \ip{q, T^*f} = \ip{Tq, f} = \ip{p, f}$.
Given a kernel $T$ we now have a surrogate $s = T^* f$. Next we need a way to estimate this surrogate when the learner samples $X$ from $p$ and observes
$Y = f(X)$. 
Note that $p(x) = 0$ implies that $T(x | y) = 0$ for all $y$, since $q(y) > 0$ for all $y$ by assumption. Then, 
\begin{align*}
s(y) = \sum_{x \in \cC} T(x | y) f(x) 
= \sum_{x \in \cC : p(x) > 0} \frac{T(x | y)}{p(x)} f(x) p(x) \,,
\end{align*}
which shows that when $X$ is sampled from $p$ and $Y = f(X)$, then the surrogate $s$ can be be estimated by
\begin{align*}
\hat s(y) = \frac{T(X | y) Y}{p(X)} \,.
\end{align*}
An important point is that we are allowed to choose $T$ to depend on the exponential weights distribution $q$.
Given $x \in \cC$, let
$\Pi_\cC(x) = \argmin_{y \in \cC, |y| \leq |x|} |x - y|$, which is the $y \in \cC$ closest to $x$ in the direction of the origin.
We also let $\mu = \sum_{y \in \cC} q(y) y$ and $\mu_\pi = \Pi_\cC(\mu)$.
Given $x, y \in \cC$, let $I(x, y) = \{z \in \cC  \colon \min(x, y) \leq z \leq \max(x, y)\}$ 
and
\begin{align}
T(x | y) = \frac{\sind(x \in I(y, \mu_\pi))}{|I(y, \mu_\pi)|}\,. 
\label{eq:exp:T}
\end{align}
That is, $x \mapsto T(x | y)$ is the uniform distribution on $I(y, \mu_\pi) \subset \cC$.\index{uniform measure}
What this means is that the distribution $p$ is obtained from $q$ by spreading the mass $q$ assigns to any point $y$ uniformly between $y$ and the projected mean $\mu_\pi$.
The actions of this kernel are illustrated in \cref{fig:kernel}.
We can think about why this kernel might be useful.
\begin{itemize}
\item  
By \cref{eq:exp-kernel}, $\ip{q, s} = \ip{p, f}$ no matter how the kernel is chosen. If we could have $s(x_\star) = f(x_\star)$, then \cref{eq:exp:want} 
would hold with $A = 1$ and $B = 0$. 
You can actually achieve this by choosing $T(x | y) = \sind(x = y)$. But with this choice the surrogate estimate $\hat s$ generally has an enormous second moment,
which is reduced by smoothing more broadly. 
\item The reason for smoothing towards the mean is that this automatically respects the concentration properties of $q$. That is, if $q$ is concentrated about its mean then 
the loss is also smoothed over a small region about the mean.
\end{itemize}
The following lemma establishes the two essential properties of the kernel, which is that \cref{eq:exp:want} holds with $A = 2$ and $B = \frac{\eps}{2}$ (part~\ref{lem:exp:1-d:cvx}) and
that the second moment of the surrogate loss is well controlled (part~\ref{lem:exp:1-d:stab}). \index{surrogate loss}

\begin{figure}
\includegraphics[width=0.49\textwidth]{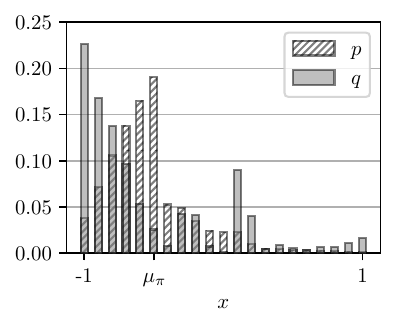}
\includegraphics[width=0.49\textwidth]{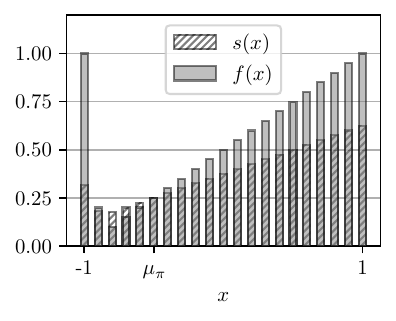}
\caption{The kernel in \cref{eq:exp:T} acting on a loss $s = T^\star f$ and distribution $p = T q$. Notice that $q$ is not particularly smooth while $p$ is unimodal. The surrogate $s$ is neither optimistic nor pessimistic,\index{pessimistic}\index{optimistic} 
in contrast with the surrogates used in \cref{chap:sgd,chap:ftrl}, which are pessimistic and optimistic respectively.}
\label{fig:kernel}
\commentAlt{On the left is a bar chart showing an irregular distribution q while p is unimodal about the mean.
On the right is a bar chart showing the loss and surrogate, which is neither optimistic nor pessimistic.}
\end{figure}

\begin{lemma}\label{lem:exp:1-d}
Let $T$ be defined as in \cref{eq:exp:T} and
$q \in \Delta(\cC)$ be such that $q(x) > 0$ for all $x \in \cC$, $p = Tq$ and $s = T^\star f$. Then
\begin{enumerate}
\item $\sum_{x \in \cC} p(x) f(x) - f(y) \leq 2 \left(\sum_{x \in \cC} q(x) s(x) - s(y)\right) + \frac{\eps}{2}$ for all $y \in \cC$,
\label{lem:exp:1-d:cvx}
\item $\sum_{x \in \cC} \sum_{y \in \cC} p(x) q(y) \left(\frac{T(x | y) f(x)}{p(x)}\right)^2 \leq 2 + \log(n)$.
\label{lem:exp:1-d:stab}
\end{enumerate}
\end{lemma}

\begin{proof} 
By definition, 
\begin{align}
s(y) 
&= (T^* f)(y) \nonumber \\
&= \sum_{x \in \cC} T(x | y) f(x) \nonumber \\
&= \sum_{x \in \cC} \frac{\sind(x \in I(y, \mu_\pi))}{|I(y, \mu_\pi)|} f(x) \nonumber \\
&= \sum_{x \in \cC} \frac{\sind(x \in I(y, \mu_\pi))}{|I(y, \mu_\pi)|} f\left(\left|\frac{x - \mu_\pi}{y - \mu_\pi}\right| y + \left|\frac{y - x}{y - \mu_\pi}\right| \mu_\pi\right) \nonumber \\
\tag{$f$ convex}
&\leq \sum_{x \in \cC} \frac{\sind(x \in I(y, \mu_\pi))}{|I(y, \mu_\pi)|} \left[\left|\frac{x - \mu_\pi}{y - \mu_\pi}\right| f(y) + \left|\frac{y - x}{y - \mu_\pi}\right| f(\mu_\pi)\right] \nonumber \\
&= \frac{1}{2} f(y) + \frac{1}{2} f(\mu_\pi) \,. 
\label{eq:exp:1d-key}
\end{align}
The mean of $x \mapsto T(x | y)$ is $y/2 + \mu_\pi/2$ and therefore
\begin{align}
\sum_{x \in \cC} p(x) x
= \sum_{y \in \cC} q(y) \sum_{x \in \cC} x T(x | y)
= \frac{\mu_\pi}{2} + \frac{1}{2} \sum_{y \in \cC} q(y) y = \frac{\mu_\pi + \mu}{2} \,.
\label{eq:exp:mean}
\end{align}
Because $f$ is convex and bounded in $[0,1]$, 
\begin{align*}
f(\mu_\pi) 
\leq f\left(\frac{\mu_\pi + \mu}{2} \right) + \frac{\eps}{2} 
\explanw{\cref{eq:exp:mean}}
= f\left(\sum_{x \in \cC} p(x) x\right) + \frac{\eps}{2} 
\explanw{$f$ cvx}
\leq \sum_{x \in \cC} p(x) f(x) + \frac{\eps}{2}\,.
\end{align*}
Combining the above display with \cref{eq:exp:1d-key} yields
\begin{align*}
s(y) \leq \frac{1}{2} f(y) + \frac{1}{2} \sum_{x \in \cC} p(x) f(x) + \frac{\eps}{4} \,.
\end{align*}
Lastly, by \cref{eq:exp-kernel}, $\sum_{x \in \cC} p(x) f(x) = \sum_{x \in \cC} q(x) s(x)$ and hence
\begin{align*}
\sum_{x \in \cC} p(x) f(x) - f(y) \leq 2 \left(\sum_{x \in \cC} q(x) s(x) - s(y)\right) + \frac{\eps}{2}\,.
\end{align*}
This establishes \ref{lem:exp:1-d:cvx}. 
Part~\ref{lem:exp:1-d:stab} follows because
\begin{align*}
\sum_{x \in \cC} \sum_{y \in \cC} \frac{q(y) T(x | y)^2 f(x)^2}{p(x)} 
&\leq \sum_{x \in \cC} \sum_{y \in \cC} \frac{\sind(x \in I(y, \mu_\pi))}{|I(y, \mu_\pi)|} \frac{q(y) T(x | y)}{p(x)} \\
&\leq  \sum_{x \in \cC} \frac{1}{|I(x, \mu_\pi)|} \sum_{y \in \cC} \frac{q(y) T(x | y)}{p(x)} \\ 
&= \sum_{x \in \cC} \frac{1}{|I(x, \mu_\pi)|} \\
&\leq 2 + \log(n)\,,
\end{align*}
where in the first inequality we used the fact that $|I(y, \mu_\pi)| \geq |I(x, \mu_\pi)|$ for $x \in I(y, \mu_\pi)$.
The second inequality follows via a standard harmonic sum comparison.
\end{proof}

We are now in a position to use exponential weights for convex bandits using the kernel-based surrogate loss estimators analysed above.
A snapshot of the loss estimators and exponential weights distribution is given in \cref{fig:cont:1-dim}.

\begin{algorithm}[h!]
\begin{algcontents}
\begin{lstlisting}
args: learning rate $\eta > 0$, $\eps = 1/\sqrt{n}$
let $\cC = \left\{\eps k  \colon k \in \mathbb Z, k \eps \in K\right\}$
for $t = 1$ to $n$
  compute $q_t(x) = \frac{\exp\left(-\eta \sum_{u=1}^{t-1} \hat s_u(x)\right)}{\sum_{y \in \cC} \exp\left(-\eta \sum_{u=1}^{t-1} \hat s_u(y)\right)}\,\, \forall x \in \cC$
  let $\mu_t = \Pi_\cC\left(\sum_{x \in \cC} q_t(x) x\right)$ and 
  $T_t(x | y) = \frac{\sind(x \in I(y, \mu_t))}{|I(y, \mu_t)|}$ $\forall x, y \in \cC$
  sample $X_t \sim p_t = T_t q_t$ and observe $Y_t = f_t(X_t)$
  compute $\hat s_t(y) = \frac{T_t(X_t | y) Y_t}{p_t(X_t)}$ $\forall y \in \cC$
\end{lstlisting}
\caption{Exponential weights for convex bandits: $d = 1$}
\label{alg:exp-1-disc}
\end{algcontents}
\end{algorithm}

\FloatBarrier

\subsubsection*{Computation}
\cref{alg:exp-1-disc} is not written in the most efficient way. Naively computing $T_t$ on all inputs would require $|\cC|^2 = \Theta(n)$ time.
This can be improved. Given $\mu_t$ the kernel function $T_t(x | y)$ can be computed in $O(1)$ for any $x, y \in \cC$.
Now $X_t$ can be sampled by first sampling $Z_t$ from $q_t$ and then $X_t$ from $T_t(\cdot | Z_t)$. 
Hence only $T_t(\cdot | Z_t)$ and $T_t(X_t | \cdot)$ need to be computed for all inputs. 
This way the computation per round of \cref{alg:exp-1-disc} is linear in $|\cC| = \Theta(\sqrt{n})$.
Therefore in the worst case the running time over the entire interaction is at most $O(n^{3/2})$.
This is worse than the gradient-based algorithms in \cref{chap:sgd,chap:ftrl} but the regret is smaller.

\begin{theorem}\label{thm:exp-1d}
Suppose that $\eta = n^{-1/2}$.
Under \cref{ass:exp} and with $d = 1$ the regret of \cref{alg:exp-1-disc} is upper-bounded by
\begin{align*}
\E[\Reg_n] \leq 2 \sqrt{n} \log(n) + 7 \sqrt{n}\,. 
\end{align*}
\end{theorem}

\begin{proof} 
Let $x_\star = \argmin_{x \in \cC} \sum_{t=1}^n f_t(x)$. By the definition of $\cC$ and the fact that $(f_t) \in \cF_\pb$ and $\eps = \frac{1}{\sqrt{n}}$,
\begin{align}
\Reg_n &= \sup_{x \in K} \sum_{t=1}^n (f_t(X_t) - f_t(x)) 
\leq \sqrt{n} + \sum_{t=1}^n (f_t(X_t) - f_t(x_\star)) \,.
\label{eq:exp:d1-1}
\end{align}
By \cref{thm:exp-discrete} and \cref{rem:exp-discrete},
\begin{align*}
\E\left[\hReg_n(x_\star)\right]
&\triangleq\E\left[\sum_{t=1}^n \left(\sum_{y \in \cC} q_t(y) \hat s_t(y) - \hat s_t(x_\star)\right)\right] \\
&\leq \frac{\log |\cC|}{\eta} + \frac{\eta}{2} \E\left[\sum_{t=1}^n \sum_{y \in \cC} q_t(y) \hat s_t(y)^2\right] \\
&= \frac{\log |\cC|}{\eta} + \frac{\eta}{2} \E\left[\sum_{t=1}^n \E_{t-1}\left[\sum_{y \in \cC} \frac{q_t(y) T_t(X_t, y)^2 f_t(X_t)^2}{p_t(X_t)^2}\right]\right] \,.
\end{align*}
The inner conditional expectation is bounded using \cref{lem:exp:1-d} by 
\begin{align*}
\E_{t-1}\left[\sum_{y \in \cC} \frac{q_t(y) T_t(X_t, y)^2 f_t(X_t)^2}{p_t(X_t)^2}\right] 
&\leq \log(n) + 2\,.
\end{align*}
Therefore the regret of exponential weights relative to the estimated loss function is bounded by
\begin{align*}
\E\left[\hReg_n(x_\star)\right]
&\leq \frac{\log |\cC|}{\eta} + \frac{\eta n}{2}\left[\log(n) + 2\right]\,.
\end{align*}
The next step is to compare $\E[\hReg_n(x_\star)]$ and $\E[\Reg_n(x_\star)]$.
By \cref{lem:exp:1-d},
\begin{align*}
\E_{t-1}[f_t(X_t)] - f_t(x_\star)
&= \sum_{x \in \cC} p_t(x) f_t(x) - f_t(x_\star) \\
&\leq 2 \left[ \sum_{x \in \cC} q_t(x) s_t(x) - s_t(x_\star)\right] + \frac{1}{2\sqrt{n}}\\
&= 2 \E_{t-1}\left[\sum_{x \in \cC} q_t(x) \hat s_t(x) - \hat s_t(x_\star)\right] + \frac{1}{2\sqrt{n}} \,. 
\end{align*}
Hence,
\begin{align*}
\E\left[\Reg_n(x_\star)\right] \leq 2 \E\big[\hReg_n(x_\star)\big] + \frac{1}{2} \sqrt{n} \leq \frac{2 \log |\cC|}{\eta} + \eta n \left[\log(n) + 2\right] + \frac{1}{2}\sqrt{n} \,.
\end{align*}
The claim follows from the choice of $\eta$, \cref{eq:exp:d1-1}, and the fact that $\log|\cC| \leq \log(1 + 2 \sqrt{n})$ and naive simplification.
\end{proof}

\begin{figure}
\includegraphics[width=10cm]{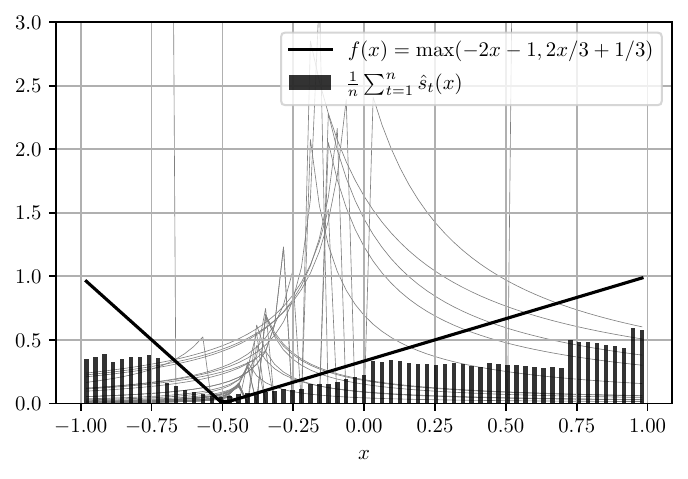}
\caption{The plot shows the estimated losses during the execution of \cref{alg:exp-1-disc}. 
The thin grey lines are the individual loss estimation functions $(\hat s_t)$ while the bar chart is the average. 
The individual loss estimates peak at the point played by the algorithm and decrease away from the mean. 
This is strange -- the loss estimators are quasiconcave while their expectations are convex. 
Mathematically this form is easy to see by studying the kernel. You should think about why it is reasonable intuitively.
}\label{fig:cont:1-dim}
\commentAlt{A bar plot showing a variety of individual estimated losses that are very different from the true loss, while cumulatively they are at least a mediocre approximation.}
\end{figure}

\FloatBarrier

\section[The Bregman Divergence]{The Bregman Divergence} 
\index{Bregman divergence|(}

This section and the next introduce some tools needed in \cref{sec:exp:exp-by-opt};
specifically, the Bregman divergence, projections for Legendre functions and the modern view of exponential weights. 

Bregman divergences are well-defined for any differentiable\index{differentiable} convex function but have particularly nice properties for the class of Legendre convex functions, which we now define.
Our exposition here is probably too brief to really impart the right intuition. There is a beautiful article by \cite{bauschke1997legendre} that explains all the nuances.
Let $R$ be a closed proper convex function.
$R$ is called essentially smooth if it is differentiable on $\interior(\dom(R)) \neq \emptyset$ and $\norm{R'(x_k)} \to \infty$ \index{essentially smooth|textbf} 
whenever $(x_k)$ converges to $\partial \dom(R)$.\index{essentially strictly convex|textbf}
$R$ is called essentially strictly convex if it is strictly convex on all convex $C \subset \dom(\partial R) = \{x \colon \partial R(x) \neq \emptyset\}$. 
$R$ is called Legendre if it is both essentially smooth and essentially strictly convex. \index{Legendre|textbf}
Lastly, the Legendre transform of $R$ is
\begin{align*}
R^\star(u) = \sup_{x \in \R^d} \ip{u, x} - R(x) \,.
\end{align*}
Legendre functions have well-behaved Legendre transforms. In particular, when $R$ is Legendre, than $R'$ is a bijection on suitable domains:

\begin{proposition}[Theorem 26.5, \citealt{Roc70}]\label{prop:breg:legendre}
Suppose that $R$ is Legendre and $T = R^\star$. Then $R' \colon \interior(\dom(R)) \to \interior(\dom(T))$ is a bijection and $(R')^{-1} = T'$.
\end{proposition}

Given a convex function $R \colon \R^d \to \R \cup \{\infty\}$ that is differentiable on $\interior(\dom(R))$, the Bregman divergence with respect to $R$ is 
$\Breg_R \colon \R^d \times \interior(\dom(R)) \to \R \cup \{\infty\}$ defined by 
$\Breg_R(x, y) = R(x) - R(y) - \ip{R'(y), x-y}$.

\begin{proposition}[Generalised Pythagorean Theorem]\label{prop:breg:triangle}
Suppose that $R \colon \R^d \to \R \cup \{\infty\}$ is Legendre and $C \subset \R^d$ is convex and closed with $C \cap \interior(\dom(R)) \neq \emptyset$.
For $x \in \interior(\dom(R))$ the `projection' $\Pi(x) = \argmin_{y \in C} \Breg_R(y, x)$ is unique and for any $z \in C$, \index{projection!Bregman divergence}
\begin{align*}
\Breg_R(z, x) \geq \Breg_R(z, \Pi(x)) + \Breg_R(\Pi(x), x) \,.
\end{align*}
\end{proposition}

\begin{proof}[\Proofskippy]
Let $y = \Pi(x)$.
Under the conditions of the theorem, $y$ exists, is unique and $y \in \interior(\dom(R))$ \citep[Theorem 3.12]{bauschke1997legendre}.
Hence $R$ is differentiable at $y$ and
by the first-order optimality conditions for $\Pi(x)$,
$\ip{R'(y), z - y} - \ip{R'(x), z - y} \geq 0$.
Therefore,
\begin{align*}
\Breg_R(z, y) + \Breg_R(y, x)
&= R(z) - R(x) - \ip{R'(y), z - y} - \ip{R'(x), y - x} \\
&\leq R(z) - R(x) - \ip{R'(x), z - x} \\
&= \Breg_R(z, x)\,.
\qedhere
\end{align*}
\end{proof}

The Bregman divergence with respect to $R$ is related to the Bregman divergence with respect to $R^\star$ as follows:

\begin{proposition}\label{prop:breg:dual}
Suppose that $R$ is Legendre; then for all $x, y \in \interior(\dom(R))$, 
\begin{align*}
\Breg_R(x, y) = \Breg_{R^\star}(R'(y), R'(x)) \,.
\end{align*}
\end{proposition}

\begin{proof}[\Proofskippy]
Let $v = R'(y)$, $u = R'(x)$ and $T = R^\star$ and convince yourself that
$T(v) = \ip{v, y} - R(y)$ and 
$T(u) = \ip{u, x} - R(x)$. Then
\begin{align*}
\Breg_R(x, y) 
&= R(x) - R(y) - \ip{R'(y), x - y} \\
&= [\ip{u, x} - T(u)] - [\ip{v, y} - T(v)] - \ip{v, x - y} \\
&= T(v) - T(u) - \ip{x, v - u} \\
&\stackrel{\star}= T(v) - T(u) - \ip{T'(u), v - u} \\
&= \Breg_{R^\star}(v, u) \\
&= \Breg_{R^\star}(R'(y), R'(x)) \,,
\end{align*}
where $(\star)$ follows from 
\cref{prop:breg:legendre} and the rest by substituting the definitions.
\end{proof}
\index{Bregman divergence|)}

\section[Exponential Weights and Regularisation]{Exponential Weights and Regularisation}
In \cref{chap:lin} we presented the old-school analysis of exponential weights.\index{exponential weights}
Later in this chapter we need the modern viewpoint that exponential weights is the same as follow-the-regularised-leader \index{follow-the-regularised-leader}
with unnormalised negative entropy (negentropy) regularisation.
Unsurprisingly, the analysis is not \textit{really} different. The modern view actually stops before the approximations made in the proof of \cref{thm:exp-discrete}.
Most importantly, by using the language of convex analysis we are able to clarify the duality argument used later.
The only reason not to introduce this approach in \cref{chap:lin} is that it is somewhat less elementary.
\index{unnormalised negative entropy}
Let $R \colon \R^m \to \R \cup \{\infty\}$ be the unnormalised negative entropy function defined by 
\begin{align*}
R(p) = \begin{cases}
\sum_{k=1}^m \left(p_k \log(p_k) - p_k\right) & \text{if } p \in \R^m_+ \\
\infty & \text{otherwise}\,,
\end{cases}
\end{align*}
where we adopt the convention that $0 \log(0) = 0$.
When $p, q \in \Delta_m$, then $\Breg_R(p, q) = \sum_{k=1}^m p_k \log(p_k / q_k)$ is the relative entropy\index{relative entropy} between distributions $p$ and $q$. In the next exercise you will show that $R$ is Legendre, calculate its dual and show that follow-the-regularised-leader
with $R$ is equivalent to exponential weights:

\begin{exer}\label{ex:exp:R}
\faStar \quad
Suppose that $x \in \R^m$ and $q \in (0,\infty)^m$. Show that
\begin{enumerate}
\item $R$ is Legendre and \index{Legendre} $R^\star(x) = \sum_{k=1}^m \exp(x_k)$,
\item $\displaystyle \argmin_{p \in \Delta_m} \Breg_R(p, q) = \frac{q}{\norm{q}_1}$ 
and $\displaystyle \argmin_{p \in \Delta_m} \ip{p, x} + R(p) = \frac{\exp(-x)}{\norm{\exp(-x)}_1}$. 
\end{enumerate}
\end{exer}

\solution{By definition, $\dom(R) = \R_+^m$, which is closed and $R$ is continuous on $\dom(R)$. Hence $R$ is closed.
$R$ is differentiable on $\interior(\dom(R)) = \R_{++}^m$ and has derivative $R'(p) = \log(p)$.
Since $\partial \dom(R)$ is the set of $p$ for which $p_k = 0$ for some $k$, it follows that $\norm{R'(p)} \to \infty$ 
as $p \to \partial(\dom(R))$. Therefore $R$ is essentially smooth.  
$\dom(\partial R) = \R_{++}^m$. Since on this set $R''(p)$ is a diagonal matrix with $1/p$ on the diagonal, which is in $\pd$.
Hence $R$ is strictly convex on $\dom(\partial R)$. Combining the above claims shows that $R$ is Legendre.
By definition
\begin{align*}
R^\star(x) = \sup_{p \in \R^m} \left(\ip{x, p} - R(p) \right) \,.
\end{align*}
Since $R'(p) = \log(p)$, the point $p = \exp(x)$ is a stationary point of $p \mapsto \ip{x, p} - R(p)$.
Therefore $R^\star(x) = \ip{x, \exp(x)} - \sum_{k=1}^m \left(x_k \exp(x_k) - \exp(x_k)\right) = \sum_{k=1}^m \exp(x_k)$.
For the second part, let $q \in \R_{++}^m$ and abbreviate $h(p) = \Breg_R(p, q)$.
Then $h'(p) = R'(p) - R'(q)$.
Note that strict convexity of $R$ ensures that $p$ is unique, if it exists.
Suppose that $r \in \Delta_m$.
Then
\begin{align*}
\ip{h'(p), r - p}
&= \ip{\log(p) - \log(q), r - p} \\
&= \ip{\log(q/\norm{q}_1) - \log(q), r - q / \norm{q}_1} \\
&= \log(1/\norm{q}_1) \ip{\ones, r - q / \norm{q}_1} \\
&= 0 \,.
\end{align*}
Hence $p = q /\norm{q}_1$ satisfies the first-order optimality conditions and is unique by strict convexity of $R$. 
The second follows similarly. Indeed, let $h(p) = \ip{p, x} + R(p)$. Then with $p = \exp(-x) / \norm{\exp(-x)}_1$ and $r \in \Delta_m$,
\begin{align*}
\ip{h'(p), r - p}
&= \ip{x + \log(p), r - p} \\
&= \log(\norm{\exp(-x)}_1) \ip{\ones, r - p} \\
&= 0 \,.
\end{align*}
And by the first-order optimality conditions and strict convexity of $R$ the claim follows.
}

Let $\cC$ be a finite set and $m = |\cC|$.
We identify distributions and functions on $\cC$ by vectors in $\R^m$. Concretely,
for $p \in \Delta(\cC) \equiv \Delta_m$ and $f \colon \cC \to \R$ let $\ip{p, f} = \sum_{x \in \cC} p(x) f(x)$.
Given a learning rate $\eta > 0$ and a sequence of functions $(\hat s_t)_{t=1}^n \colon \cC \to \R$, let 
\begin{align*}
q_t = \argmin_{q \in \Delta_m} \sum_{u=1}^{t-1} \eta \ip{q, \hat s_u} + R(q) \,.
\end{align*}
By \cref{ex:exp:R}, the distribution $q_t$ can be written in its familiar form:
\begin{align}
q_t(x) = \frac{\exp\left(-\eta \sum_{u=1}^{t-1} \hat s_u(x)\right)}{\sum_{y \in \cC} \exp\left(-\eta \sum_{u=1}^{t-1} \hat s_u(y)\right)}\,.
\label{eq:exp:exp}
\end{align}
The following refinement of \cref{thm:exp-discrete} bounds the regret of exponential weights relative to the loss functions $(\hat s_t)$.

\begin{theorem}\label{thm:exp-discrete-2}
Let $(\hat s_t)_{t=1}^n \colon \cC \to \R$ be a sequence of functions, $q_t$ be defined as in \cref{eq:exp:exp} and $\eta > 0$. Then
\begin{align*}
\max_{q_\star \in \Delta_m} \sum_{t=1}^n \ip{q_t - q_\star, \hat s_t}  \leq \frac{\log m}{\eta} + \frac{1}{\eta} \sum_{t=1}^n \cS_{q_t}(\eta \hat s_t)\,,
\end{align*}
where $\cS_q(u) = \Breg_{R^\star}(R'(q) - u, R'(q))$ is called the `stability' function.
\end{theorem}

\begin{proof}[\Proofskippy]
Let $q_\star \in \Delta_m$, $\Phi_t(p) = \sum_{u=1}^t \eta \ip{p, \hat s_u} + R(p)$ and
\begin{align}
\tilde q_{t+1} \triangleq \argmin_{q \in \R^m} \eta \ip{q, \hat s_t} + \Breg_R(q, q_t) = q_t \exp(-\eta \hat s_t) \,,
\label{eq:exp:tq}
\end{align}
where the second equality follows by solving the optimisation problem analytically. 
Note that $q_t \in \Delta_m^+$ holds by \cref{eq:exp:exp}, which means that $\tilde q_{t+1} \in \Delta_m^+$ and 
by \cref{ex:exp:R}, 
\begin{align}
q_{t+1} = \argmin_{q \in \Delta_m} \Breg_R(q, \tilde q_{t+1}) \,.
\label{eq:exp:q}
\end{align}
Repeating the proof of \cref{thm:ftrl},
\begin{align*}
\sum_{t=1}^n \ip{q_t - q_\star, \hat s_t} 
&\leq \frac{R(q_\star) - R(q_1)}{\eta} + \frac{1}{\eta} \sum_{t=1}^n (\Phi_t(q_t) - \Phi_t(q_{t+1})) \\
&\explana= \frac{R(q_\star) - R(q_1)}{\eta} + \frac{1}{\eta} \sum_{t=1}^n \Breg_{\Phi_t}(q_t, q_{t+1}) \\
&\explana= \frac{R(q_\star) - R(q_1)}{\eta} + \frac{1}{\eta} \sum_{t=1}^n \Breg_{R}(q_t, q_{t+1}) \\
&\explana\leq \frac{R(q_\star) - R(q_1)}{\eta} + \frac{1}{\eta} \sum_{t=1}^n \Breg_{R}(q_t, \tilde q_{t+1}) \\
&\explana= \frac{R(q_\star) - R(q_1)}{\eta} + \frac{1}{\eta} \sum_{t=1}^n \Breg_{R^\star}(R'(\tilde q_{t+1}), R'(q_t)) \\
&\explana= \frac{R(q_\star) - R(q_1)}{\eta} + \frac{1}{\eta} \sum_{t=1}^n \Breg_{R^\star}(R'(q_t) - \eta \hat s_t), R'(q_t)) \\ 
&\leq \frac{\log(m)}{\eta} + \frac{1}{\eta} \sum_{t=1}^n \cS_{q_t}(\eta \hat s_t) \,.
\end{align*}
where \explanr{} follows from the definition of the Bregman divergence\index{Bregman divergence} and because by the first-order optimality conditions
$\ip{\Phi_t'(q_{t+1}), q_t - q_{t+1}} = 0$, which holds with equality because $q_t, q_{t+1} \in \Delta_m^+$.
\explanr{} is because $\Breg_{\Phi_t} = \Breg_R$ by virtue of the fact that the Bregman divergence of a linear function vanishes.
\explanr{} follows from \cref{eq:exp:q} and \cref{prop:breg:triangle}, which yields
\begin{align*}
\Breg_R(q_t, \tilde q_{t+1}) &\geq \Breg_R(q_t, q_{t+1}) + \Breg_R(q_{t+1}, \tilde q_{t+1}) \geq \Breg_R(q_t, q_{t+1})\,.
\end{align*}
Finally, \explanr{} follows from \cref{prop:breg:dual}
and \explanr{} from \cref{eq:exp:tq} and because $R'(q) = \log(q)$.
\end{proof}

\begin{remark}

To see why \cref{thm:exp-discrete-2} refines \cref{thm:exp-discrete} (and \cref{rem:exp-discrete}), when $\hat s_t$ is nonnegative, then
\begin{align*}
\cS_{q_t}(\eta \hat s_t)
&=\Breg_{R^\star}(R'(q_t) - \eta \hat s_t, R'(q_t)) \\
&= \sum_{x \in \cC} q_t(x) \exp(-\eta \hat s_t(x)) - 1 + \eta \ip{q_t, \hat s_t} \\
&\leq \sum_{x \in \cC} q_t(x) \left[1 - \eta \hat s_t(x) + \frac{\eta^2 \hat s_t(x)^2}{2}\right] - 1 + \eta \ip{q_t, \hat s_t} \\
&= \frac{\eta^2}{2} \sum_{x \in \cC} q_t(x) \hat s_t(x)^2 \,,
\end{align*}
where we used the inequality $\exp(x) \leq 1 - x + x^2/2$ for $x \geq 0$.
Combining with \cref{thm:exp-discrete-2} recovers the bound in \cref{rem:exp-discrete}. \cref{thm:exp-discrete} is recovered by assuming that
$|\eta \hat s_t(x)| \leq 1$ for all $x$ and using the inequality $\exp(-x) \leq 1 - x + x^2$ for $|x| \leq 1$ instead.
\end{remark}

\section{Exploration by Optimisation}\label{sec:exp:exp-by-opt}\index{exploration by optimisation|(}

In \cref{sec:exp:1d} we showed how to use exponential weights for one-dimensional bandits.
While the algorithm there is efficient and straightforward to implement, much human ingenuity was needed to define the surrogate loss 
and exploration distribution.
There is a principled way to construct both of these objects that eliminates the need for imagination at the cost of computational efficiency.
Note that we are not assuming that $d = 1$ in this section.
The results and concepts in this section and the next are connected in a complicated way. You may find \cref{fig:duality} useful to join the dots.

\begin{assumption}\label{ass:exp:cover}
Throughout this section as well as \cref{sec:exp:bayes,sec:exp:dual} we assume that $\cC \subset K$ is a finite set with $m \triangleq |\cC|$ such that
\begin{enumerate}
\item $\log m \leq d \log(1 + 16 d n^2)$; and
\item for all $f \in \cF_\pb$ there exists an $x \in \cC$ such that $f(x) \leq \inf_{y \in K} f(y) + 1/n$.
\end{enumerate}
\end{assumption}

\begin{exer}
\faStar \quad
Prove a cover satisfying the conditions in \cref{ass:exp:cover} exists.
You may want to combine \cref{prop:cover3}, \cref{prop:lip}, \cref{prop:shrink} and \cref{thm:john}.
\end{exer}

\solution{
Let $f = \frac{1}{n} \sum_{t=1}^n f_t \in \cF_\pb$.
Suppose that $K$ is in John's position so that $\ball_1 \subset K \subset \ball_d$.
Let $K_\eps = (1 - \eps) K$. By \cref{prop:shrink},
$\min_{x \in K_\eps} f(x) \leq \inf_{x \in K} f(x) + \eps$.
Since $\ball_1 \subset K$, $x + \ball_\eps \subset K$ for all $x \in K_\eps$ and hence $\lip_{K_\eps}(f) \leq 1/\eps$.
Let $\cC \subset K_\eps$ be an $\eps^2$-cover of $K_\eps$, which by \cref{prop:cover3} and the fact that $K \subset \ball_d$ can be chosen so that
\begin{align*}
m \leq \left(1 + \frac{4d}{\eps^2}\right)^d
\end{align*}
Given $x \in K$, let $y \in K_\eps$ be such that $f(y) \leq f(x) + \eps$ and $z \in \cC$ be such that $\norm{x - y} \leq \eps^2$.
Then $f(z) \leq f(y) + \eps \leq f(x) + 2\eps$.
The result follows by choosing $\eps = 1/(2n)$.
When $K$ is not in John's position, let $T = \JOHN_K$. Find a cover of $TK$ using the above argument and pullback via $T^{-1}$.
}

Let us start by giving the regret bound for the (abstract) \cref{alg:exp:abstract}.

\begin{theorem}\label{thm:exp:abstract}
Let $x_\star = \argmin_{x \in \cC} \sum_{t=1}^n f_t(x)$ and $p_\star \in \Delta(\cC)$ be a Dirac on $x_\star$.
The expected regret of \cref{alg:exp:abstract} relative to $x_\star$ is bounded by
\begin{align*}
\E[\Reg_n(x_\star)]
&\leq \frac{\log m}{\eta} + \sum_{t=1}^n \E\left[\ip{p_t-p_\star,f_t} + \sip{p_\star - q_t, \hat s_t} + \frac{1}{\eta} \cS_{q_t}(\eta \hat s_t)\right]\,.
\end{align*}
\end{theorem}

\begin{proof}
The proof follows immediately from \cref{thm:exp-discrete-2}:
\begin{align*}
\E[\Reg_n(x_\star)]
&= \sum_{t=1}^n \E\left[ \ip{p_t - p_\star, f_t}\right] \\
&= \sum_{t=1}^n \E\left[ \ip{p_t - p_\star, f_t} + \sip{p_\star - q_t, \hat s_t} + \sip{q_t - p_\star, \hat s_t} \right] \\
&\leq \frac{\log m}{\eta} + \sum_{t=1}^n \E\left[ \ip{p_t - p_\star, f_t} + \sip{p_\star - q_t, \hat s_t} + \frac{1}{\eta}\cS_{q_t}(\eta \hat s_t) \right]  \,,
\end{align*}
where the second equality holds by adding and subtracting $\sip{q_t - p_\star, \hat s_t}$ and the inequality by \cref{thm:exp-discrete-2}.
\end{proof}

Standard methods for analysing concrete instantiations of \cref{alg:exp:abstract} essentially always bound the term inside the expectation
uniformly for all $t$, independently of $f_t$ and $q_t$ and $p_\star$. Note that $f_t$ and $p_\star$ are unknown, while $q_t$ is the exponential weights\index{exponential weights} distribution, 
which is known at the start of round $t$.
In light of this, a natural idea is to choose the distribution $p_t$ and loss estimation function $\hat s_t$ that minimise the upper bound.
This is the idea we execute now.

\subsubsection*{Exploration by Optimisation}
To simplify the notation, let us momentarily drop the time indices and let $q \in \Delta^+_m$.
The learner samples $X$ from some distribution $p \in \Delta_m^+$ and observes $Y = f(X)$.
The estimated surrogate loss $\hat s$ is a vector in $\R^m$ (a function from $\cC$ to $\R$) 
but the learner chooses it based on the observations $X$ and $Y$. So let $\cE$ be the set of all
functions $e \colon \cC \times \R \to \R^m$, with the idea that the estimated surrogate loss will be the function $\hat s = e(X, Y) / p(X)$. \index{surrogate loss!optimisation}
The division by $p(X)$ is a normalisation that makes a certain function defined below convex (see \cref{ex:exp:cvx}) and is also the reason why we insist that $p \in \Delta_m^+$ rather than $\Delta_m$.
The decision for the learner is to choose the exploration distribution $p \in \Delta_m^+$ and exploration function $e \in \cE$. 
The adversary chooses $f \in \cF_\pb$ and $p_\star \in \Delta_m$.
Looking at \cref{thm:exp:abstract}, we hope you agree that the following function may be important:
\newcommand{\LD}{|\!|}
\begin{align*}
\Lambda_{\eta, q}(p, e \LD p_\star, f) =
\frac{1}{\eta} \E\left[\ip{p - p_\star, f} + \ip{p_\star - q, \frac{e(X, Y)}{p(X)}} + \frac{1}{\eta}\cS_q\left(\frac{\eta e(X, Y)}{p(X)}\right)\right] \,,
\end{align*}
where the expectation is over $X \sim p$ and $Y = f(X)$.
You might view $\Lambda_{\eta,q}$ as a function from two pairs of tuples: $(p, e)$ selected by the learner and $(p_\star, f)$ selected by the adversary.
A fundamental quantity that appears throughout the following sections is the minimax value of this game, which is defined in stages 
by the following quantities:
\begin{align*}
\Lambda^\star_{\eta,q}(p, e) &= \sup_{p_\star \in \Delta_m} \sup_{f \in \cF_\pb}  \Lambda_{\eta, q}(p, e \LD p_\star, f) \qquad \text{and}\\ 
\Lambda^\star_{\eta,q} &= \inf_{p \in \Delta_m^+} \inf_{e \in \cE} \Lambda^\star_{\eta,q}(p, e) \qquad \text{and}\\
\Lambda^\star &= \sup_{\eta > 0} \sup_{q \in \Delta_m^+} \Lambda^\star_{\eta,q} \,.
\end{align*}

\begin{exer}\label{ex:exp:cvx}
\faStar \quad
Prove that $(p, e) \mapsto \Lambda_{\eta,q}(p, e \LD p_\star, f)$ is convex for any $p_\star \in \Delta_m$ and $f \in \cF_\pb$.
You may find it useful to expand the expectation and use the fact that $u \mapsto \cS_q(u)$ is convex, which by the perspective construction 
\citep[\S2.3.3]{BV04}
shows that $(u, v) \mapsto v \cS(u/v)$ is convex on $\R^d \times (0,\infty)$. 
\end{exer}

\solution{
Fix $p_\star \in \Delta_m$ and $f \in \cF_\pb$. Expanding the expectation shows that
\begin{align*}
\Lambda_{\eta,q}(p, e \LD p_\star, f) = A(p, e) + B(p, e) + C(p,e)\,,
\end{align*}
where $A(p, e) = \ip{p - p_\star, f} / \eta$ and $B(p, e) = \sum_{x \in \cC} \ip{p_\star - q, e(x, f(x))} / \eta$ and
\begin{align*}
C(p, e) = \frac{1}{\eta^2} \sum_{x \in \cC} p(x) \cS_q\left(\frac{\eta e(x, f(x))}{p(x)}\right) \,.
\end{align*}
Both $A$ and $B$ are linear and hence convex while $C$ is convex by the aforementioned perspective construction.
}

Since the supremum of convex functions is convex it follows from \cref{ex:exp:cvx} that
$(p, e) \mapsto \Lambda^\star_{\eta,q}(p, e)$ is convex.
Note, however, that $\cE$ is infinite-dimensional. So in general it may be non-trivial to efficiently minimise this function.
Despite this, these functions are surprisingly easy to handle mathematically, as we shall see later.
The exploration-by-optimisation algorithm is quite simple:

\begin{algorithm}[h!]
\begin{algcontents}
\begin{lstlisting}[escapeinside={&}{&}]
args: learning rate $\eta > 0$, precision $\eps > 0$, $\cC \subset K$
for $t = 1$ to $n$
  compute distribution $q_t(x) = \frac{\exp\left(-\eta \sum_{u=1}^{t-1} \hat s_u(x)\right)}{\sum_{y \in \cC} \exp\left(-\eta \sum_{u=1}^{t-1} \hat s_u(y)\right)}$
  find distribution $p_t \in \Delta_m^+$ and $e_t \in \cE$ such that &\newline&  $\qquad\qquad\Lambda^\star_{\eta,q_t}(p_t, e_t) \leq \Lambda^\star_{\eta,q_t} + \eps$
  sample $X_t$ from $p_t$ and observe $Y_t = f_t(X_t)$ 
  compute $\hat s_t = e_t(X_t, Y_t) / p_t(X_t)$
\end{lstlisting}
\caption{Exploration by optimisation}\label{alg:exp-by-opt}
\end{algcontents}
\end{algorithm}

\FloatBarrier

The following theorem is almost immediate:

\begin{theorem}\label{thm:exp-by-opt}
Under \cref{ass:exp,ass:exp:cover}, the expected regret of \cref{alg:exp-by-opt} is bounded by 
\begin{align*}
\E[\Reg_n] \leq 1 + \frac{\log m}{\eta} + n \eta \Lambda^\star + n\eta \eps\,.
\end{align*}
\end{theorem}

\begin{proof}
Let $x_\star = \argmin_{x \in \cC} \sum_{t=1}^n f_t(x)$ and
$p_\star \in \Delta(\cC)$ be a Dirac on $x_\star$.
By \cref{ass:exp:cover}, \cref{thm:exp:abstract} and the definition of $\hat s_t$ in \cref{alg:exp-by-opt}, 
\begin{align*}
\E[\Reg_n]
&\leq 1 + \E[\Reg_n(x_\star)] \\
&\leq 1 + \frac{\log m}{\eta} + \sum_{t=1}^n \E\left[\ip{p_t - p_\star, f_t} + \ip{p_\star - q_t, \hat s_t} + \frac{1}{\eta} \cS_{q_t}\left(\eta \hat s_t\right)\right] \\
&\leq 1 + \frac{\log m}{\eta} + n \eta \Lambda^\star + n \eta \eps\,.
\qedhere
\end{align*}
\end{proof}

Combining the definition with \cref{thm:exp-by-opt} immediately yields the following corollary:

\begin{corollary}\label{cor:exp-by-opt}
Under the same conditions as \cref{thm:exp-by-opt} and with
\begin{align*}
\eta = \sqrt{\frac{\log m}{n(\eps + \Lambda^\star)}}\,,
\end{align*}
the regret of \cref{alg:exp-by-opt} is upper-bounded by
\begin{align*}
\E[\Reg_n] \leq 1 + 2 \sqrt{n \Lambda^\star \log m} + 2 \sqrt{n \eps \log m} \,\,\stackrel{\eps \to 0}{\to}\,\, 1 + 2 \sqrt{n \Lambda^\star \log m}\,.
\end{align*}
\end{corollary}

In the next section we explain a connection between $\Lambda^\star$ and a concept used for analysing Bayesian bandit problems called the information 
ratio. This method will eventually show that
\begin{align*}
\Lambda^\star = O\left(d^4 \log(nd)\right) \,.
\end{align*}
Combining this with \cref{ass:exp:cover} to bound $\log m = O(d \log(nd))$ and with \cref{cor:exp-by-opt} shows that the regret of \cref{alg:exp-by-opt}
is bounded by
\begin{align*}
\E[\Reg_n] = O\left( d^{2.5} \sqrt{n} \log(dn)\right) \,.
\end{align*}
Let us emphasise again that this is not much of an algorithm because there is no known computationally efficient method for solving the optimisation problem
that defines $p_t$ and $e_t$.

\index{exploration by optimisation|)}

\section{Bayesian Convex Bandits} \index{Bayesian!convex bandits} \label{sec:exp:bayes}
In the Bayesian version of the convex bandit problem the learner is given a distribution $\xi$ on $\cF_\pb^n$.
The loss functions $(f_t)_{t=1}^n$ are sampled from $\xi$ and the Bayesian regret\index{regret!Bayesian|textbf} of a learning algorithm $\sA$ is
\begin{align*}
\bReg_n(\sA, \xi) = \E\left[\sup_{x \in K} \sum_{t=1}^n (f_t(X_t) - f_t(x))\right]  \,,
\end{align*}
where by `learning algorithm' we simply mean a function that (measurably) maps history sequences to distributions over actions in $\cC$.
Note that here the expectation integrates over the randomness in the loss functions and hence the supremum over $x \in K$ appears inside the expectation.
As in the rest of the book we will not say much about constructing probability spaces and measurability. The measurable space on which $\xi$ is defined sometimes plays
an important technical role. Generally speaking, in what follows it is assumed that the discrete $\sigma$-algebra is used on $\cF_\pb^n$ and that $\xi$ really is a distribution 
in the sense that it is supported on countably many atoms.
$\bReg_n(\sA, \xi)$ is nothing more than the expectation of the standard regret, integrating over the loss functions with respect to the prior $\xi$.
The minimax Bayesian regret is \index{Bayesian!minimax regret}\index{prior}
\begin{align*}
\bReg_n^\star = \sup_{\xi}\,\inf_{\sA}\,\bReg_n(\sA, \xi)\,.
\end{align*}
Compare this to the minimax adversarial regret, which is \index{regret!minimax}
\begin{align*}
\Reg_n^\star = \inf_{\sA} \sup_{(f_t)_{t=1}^n} \Reg_n(\sA, (f_t)_{t=1}^n)\,.
\end{align*}
Whenever you see an expression like this, a minimax theorem should come to mind. Indeed, the minimax adversarial regret can be rewritten as \index{minimax theorem}
\begin{align*}
\Reg_n^\star = \inf_{\sA}\,\sup_\xi\,\bReg_n(\sA, \xi) \,.
\end{align*}
By interpreting an algorithm as a probability measure over deterministic algorithms, both $\sA \mapsto \bReg_n(\sA, \xi)$ and $\xi \mapsto \bReg_n(\sA, \xi)$ are
linear functions and from this one should guess the following theorem as a consequence of some kind of minimax theorem.

\begin{theorem}[\citealt{LS19pminfo}]\label{thm:minimax}
$\Reg_n^\star = \bReg_n^\star$.
\end{theorem}

\cref{thm:minimax} means that one way to bound the adversarial regret is via the Bayesian regret.
One positive aspect of this idea is that the existence of a prior makes the Bayesian setting more approachable.
On the other hand, constructing a prior-dependent algorithm showing that the Bayesian regret is small for any prior does not
give you an algorithm for the adversarial setting.\index{setting!adversarial} The approach is non-constructive.\index{non-constructive}

\begin{remark}
The above setup is the Bayesian version of the adversarial convex bandit problem. 
In the Bayesian version of the stochastic convex bandit problem the prior is on $\cF_{\pb}$ rather than $\cF_{\pb}^n$ and the observation is $f(X_t) + \eps_t$ where $f$ is
sampled at the beginning of the interaction from the prior. Sometimes the noise follows some known distribution. Alternatively, the prior could be over both the loss function
and the noise distribution.
\index{Bayesian!stochastic convex bandits}
\index{noise}\index{prior}
\end{remark}

\section{Duality and the Information Ratio}\label{sec:exp:dual}
\index{information ratio}
We now briefly explain the main tool for bounding the Bayesian regret.\index{regret!Bayesian}
Recall that in the Bayesian setting $f_1,\ldots,f_n$ are sampled from a prior $\xi$ on $\cF_\pb^n$ 
and let $X_\star = \argmin_{x \in \cC} \sum_{t=1}^n f_t(x)$ be the optimal action in $\cC$ in hindsight, which is a random element.
Let $\nu$ be a probability measure on $\Delta_m \times \cF_\pb$ and $p \in \Delta_m$.
Later $\nu$ will be the law of $(p_\star, f_t)$ under the posterior associated with prior $\xi$ and data $X_1,Y_1,\ldots,X_{t-1},Y_{t-1}$
where $p_\star$ is a Dirac on $X_\star$.  \index{posterior}
Define
\begin{align*}
\Delta(p, \nu) = \E[\ip{p - p_\star, f}] \quad \! \text{and} \quad \!
I(p, \nu) = \E\left[\KL(\E[p_\star | X, f(X)], \E[p_\star])\right] \,,
\end{align*}
where $(X, p_\star, f)$ has law $p \otimes \nu$ and $\KL$ is the relative entropy (or Kullback–-Leibler divergence). 
As we remarked already, $\KL = \Breg_R$ where $R$ is the unnormalised negentropy
potential. Intuitively, $\Delta(p, \nu)$ is the expected regret suffered when sampling $X$ from $p$ relative to $p_\star$ on loss $f$ 
sampled from $\nu$, while $I(p, \nu)$ is the information gained about the optimal action when observing $X$ and $f(X)$.\index{information gain}
The information ratio captures the exploration/exploitation trade-off made by a learner and is defined by
\begin{align*}
\Psi(p, \nu) = \frac{\Delta(p, \nu)^2}{I(p, \nu)}\,.
\end{align*}
The information ratio will be small when the regret under $p$ is small relative to the information gained about the optimal action.
The minimax information ratio is 
\begin{align*}
\Psi^\star = \sup_{\nu \in \Delta(\Delta_m \times \cF_\pb)} \min_{p \in \Delta(\cC)} \Psi(p, \nu)\,.
\end{align*}
We can now introduce the information-directed sampling algorithm, which is designed for minimising the Bayesian regret.
In every round the algorithm computes the posterior based on information observed so far and then samples its action $X_t$ from the 
distribution $p_t$ minimising the information ratio.

\begin{algorithm}[h!]
\begin{algcontents}
\begin{lstlisting}
args: prior $\xi$ on $\cF_{\pb}^n$ 
for $t = 1$ to $n$
  compute the posterior $\nu_t = \bbP_{t-1}((p_\star, f_t) \in \cdot)$
  find $p_t = \argmin_{p \in \Delta(\cC)} \Psi(p, \nu_t)$
  sample $X_t$ from $p_t$ and observe $Y_t = f_t(X_t)$
\end{lstlisting}
\caption{Information-directed sampling. The random variable $p_\star$ is a Dirac on $X^\star = \argmin_{x \in \cC} \sum_{t=1}^n f_t(x)$.}
\label{alg:ids}
\end{algcontents}
\end{algorithm}

\FloatBarrier

The next theorem bounds the regret of \cref{alg:ids} in terms of the minimax information ratio.

\begin{theorem}\label{thm:bayes-reg-bound}
Under \cref{ass:exp:cover}, the Bayesian minimax regret is bounded by \index{Bayesian!minimax regret}
\begin{align*}
\bReg_n^\star \leq 1 +  \sqrt{n \Psi^\star \log m} 
\leq 1 +  \sqrt{d n \Psi^\star \log\left(1 + 16dn^2\right)} \,.
\end{align*}
\end{theorem}

\begin{proof}
Let $\xi \in \Delta(\cF_\pb^n)$ be any prior distribution and $\sA$ be information-directed sampling (\cref{alg:ids}).
By \cref{ass:exp:cover}, for all $f \in \cF_\pb$ there exists an $x \in \cC$ with
\begin{align*}
f(x) \leq \inf_{y \in K} f(y) + \frac{1}{n} \,.
\end{align*}
Therefore, with $p_\star$ a Dirac on $X_\star = \argmin_{x \in \cC} \sum_{t=1}^n f_t$, \cref{ass:exp:cover} implies that
\begin{align*}
\bReg_n(\xi, \sA) 
&= \E\left[\sup_{x \in K} \sum_{t=1}^n (f_t(X_t) - f_t(x))\right] 
\leq 1 + \E\left[\sum_{t=1}^n (f_t(X_t) - f_t(X_\star))\right]  \,.
\end{align*}
Let $\nu_t = \bbP_{t-1}((p_\star, f_t) \in \cdot)$, which is a probability measure on $\Delta_m \times \cF_\pb$.
The learner samples $X_t$ from a distribution $p_t \in \Delta(\cC)$ such that $\Psi(p_t, \nu_t) \leq \Psi^\star$, which means that
\begin{align}
\Delta(p_t, \nu_t)^2 \leq I(p_t, \nu_t) \Psi^\star \,.
\label{eq:exp:ids1}
\end{align}
By the tower rule for conditional expectation and the definition of $\Delta(p_t, \nu_t)$,
\begin{align*}
\E\left[\sum_{t=1}^n \left(f_t(X_t) - f_t(X_\star)\right)\right] = \E\left[\sum_{t=1}^n \Delta(p_t, \nu_t)\right] \,.
\end{align*}
By the definition of the information ratio,
\begin{align*}
\E\left[\sum_{t=1}^n \Delta(p_t, \nu_t)\right] 
&\explana\leq \E\left[\sum_{t=1}^n \sqrt{I(p_t, \nu_t) \Psi^\star}\right] \\
&\explana\leq \sqrt{n \Psi^\star \E\left[\sum_{t=1}^n I_t(p_t, \nu_t)\right]} \\
&\explana= \sqrt{n \Psi^\star \E\left[\KL(\bbP_{X_\star|X_1,Y_1,\ldots,X_n,Y_n}, \bbP_{X_\star})\right]} \\
&\explana\leq \sqrt{n \Psi^\star \log(m)} \,,
\end{align*}
where \explanr{} follows from \cref{eq:exp:ids1},
\explanr{} from Cauchy--Schwarz and
\explanr{} from the chain rule for information gain \citep[\S2.5]{CT12}. \explanr{} follows because $X_\star \in \cC$ and the information in any $\cC$-valued random variable
is at most $\log |\cC| = \log m$.
Finally, bound $\log m$ using \cref{ass:exp:cover}.
\end{proof}

There is no obvious reason why $\Psi^\star$ should be well controlled.
However, the information ratio is bounded by the following theorem:

\begin{theorem}[\citealt{Lat20-cvx}]\label{thm:inf-ratio}
$\Psi^\star \leq C d^4 \log(nd)$ where $C > 0$ is an absolute constant.
\end{theorem}

The proof of \cref{thm:inf-ratio} is quite involved and is not included here. 
Even when you can sample from $\nu$, the distribution $p$ witnessing the upper bound 
on $p \mapsto \Psi(p, \nu)$ involves positioning certain convex bodies in minimal surface area position and is not practical to compute.
Combining \cref{thm:inf-ratio} with \cref{thm:bayes-reg-bound} and \cref{thm:minimax} yields the following theorem:

\begin{theorem}
$\Reg_n^\star = \bReg_n^\star \leq C d^{2.5} \sqrt{n} \log(dn)$.
\end{theorem}

The above theorem shows that the adversarial minimax regret and Bayesian minimax regret can be bounded in terms of the minimax information ratio.
The next theorem shows that the adversarial regret of \cref{alg:exp-by-opt} can also be bounded in terms of the minimax information ratio.

\begin{theorem}[\citealt{lattimore2021mirror}]\label{thm:bayes-ig-bound}
$\Lambda^\star \leq \frac{1}{4} \Psi^\star$.
\end{theorem}

Except for the one-dimensional setting (\cref{sec:exp:1d}), the only bounds we have on $\Lambda^\star$ are via \cref{thm:bayes-ig-bound} and bounds on $\Psi^\star$.
Remarkably the constant $\frac{1}{4}$ means the upper bounds obtained by the mirror descent analysis of \cref{alg:exp-by-opt} and information-directed sampling exactly match:
\begin{align*}
\tag*{by \cref{thm:bayes-reg-bound}}
\bReg_n^\star &\leq 1 + \sqrt{n \Psi^\star \log m} \\ 
\tag*{by \cref{cor:exp-by-opt}}
\Reg_n^\star &\leq 1 + 2 \sqrt{n \Lambda^\star \log m} \\
\tag*{by \cref{thm:bayes-ig-bound}}
&\leq 1 + \sqrt{n \Psi^\star \log m} \,. 
\end{align*}

\newcommand{\post}{\text{\textrm\tiny{\scalebox{0.8}{\!po}}}}
\newcommand{\prior}{\text{\textrm\tiny\scalebox{0.8}{\!pr}}}

\begin{proof}[Proof sketch of \cref{thm:bayes-ig-bound}] 
We outline the key ingredients of the argument, ignoring measure-theoretic and topological challenges associated with applying the minimax theory.
Let $A = \Delta_m \times \cF_\pb$ and $\Delta(A)$ be the space of probability distributions on $A$ with the discrete $\sigma$-algebra. 
Given $p \in \Delta_m^+$, $e \in \cE$ and
$\nu \in \Delta(A)$, let
\begin{align}
&\Lambda_{\eta,q}(p, e \LD \nu) 
= \int_A \Lambda_{\eta,q}(p, e \LD p_\star, f) \d{\nu}(p_\star, f) \nonumber \\ 
\label{eq:exp:dual-1}
&\quad= \frac{1}{\eta} \E\left[\ip{p - p_\star, f} + \ip{p_\star - q, \frac{e(X, Y)}{p(X)}} + \frac{1}{\eta} \cS_q\left(\frac{\eta e(X, Y)}{p(X)}\right) \right]\,,
\end{align}
where expectation integrates over $(p_\star, f, X)$ with law $\nu \otimes p$ and $Y = f(X)$.
The expectation $\nu \mapsto \Lambda_{\eta,q}(p, e \LD \nu)$ is linear and \cref{ex:exp:cvx} shows that $(p, e) \mapsto \Lambda_{\eta,q}(p, e \LD \nu)$ is convex.
This hints at the possibility of applying a minimax theorem. Technically you should use Sion's theorem,\index{Sion's minimax theorem} which besides 
some kind of convex/concave structure
also needs compactness. Let us just say now that the following holds:
\begin{align*}
\Lambda_{\eta,q}^\star
&= \inf_{p \in \Delta_m^+} \inf_{e \in \cE} \sup_{\nu \in \Delta(A)} \Lambda_{\eta,q}(p, e \LD \nu) 
= \sup_{\nu \in \Delta(A)} \inf_{p \in \Delta_m^+} \inf_{e \in \cE} \Lambda_{\eta,q}(p, e \LD \nu)\,.
\end{align*}
Hence, the proof will be completed if we can show that for all $\nu \in \Delta(A)$,
\begin{align}
\inf_{p \in \Delta_m^+} \inf_{e \in \cE} \Lambda_{\eta,q}(p, e \LD \nu) \leq \Psi^\star \,.
\label{eq:exp:dual-2}
\end{align}
When $\nu$ is given, it turns out that the estimation function $e$ minimising \cref{eq:exp:dual-1} can be computed by differentiation 
and is
\begin{align*}
e(x, y) = \frac{p(x)}{\eta} \left(R'(q) - R'(\E[p_\star | f(x) = y])\right) \,.
\end{align*}
Two observations: (1) the above is not obvious -- you need to confirm it yourself; (2) the conditional expectation is only well-defined if $x$ is in the support of $p$ and $y$ is
in the support of $f(x)$. When this is not the case you may define $e(x,y)$ in any way you please.
Let $p \in \Delta_m^+$ and
let $p_{\post} = \E[p_\star | X, Y]$ and $p_{\prior} = \E[p_\star]$, which are the posterior and prior distributions, respectively.
Then
\begin{align*}
&\Lambda_{\eta,q}(p, e \LD \nu) 
= \frac{1}{\eta} \E\left[\ip{p - p_\star, f} + \bip{p_\star - q, \frac{e(X, Y)}{p(X)}} + \frac{1}{\eta} \cS_q\left(\frac{\eta e(X, Y)}{p(X)}\right)\right] \\
&\explana= \frac{\Delta(p, \nu)}{\eta} + \frac{1}{\eta^2} \E\left[\ip{p_\star, R'(q) - R'(p_{\post})} - R^\star(R'(q)) + R^\star(R'(p_{\post}))\right] \\
&\explana= \frac{\Delta(p, \nu)}{\eta} + \frac{1}{\eta^2} \E\left[\ip{p_\prior, R'(q)} - \ip{p_\post, R'(p_{\post})} - R^\star(R'(q)) + R^\star(R'(p_{\post}))\right] \\
&\explana= \frac{\Delta(p, \nu)}{\eta} - \frac{1}{\eta^2}\E\left[\Breg_{R^\star}(R'(q), R'(p_{\prior})) + \Breg_{R^\star}(R'(p_{\prior}), R'(p_{\post}))\right] \\
&\explana\leq \frac{\Delta(p, \nu)}{\eta} -\frac{1}{\eta^2}\E\left[\Breg_{R^\star}(R'(p_{\prior}), R'(p_{\post}))\right] \\
&\explana= \frac{\Delta(p, \nu)}{\eta} - \frac{1}{\eta^2}\E\left[\Breg_R(p_{\post}, p_{\prior})\right] \,,
\end{align*}
where \explanr{} follows from the definition of $\Lambda_{\eta,q}(p, e \LD \nu)$ and by substituting the definition of $e$ and using $\cS_q(u) = \Breg_{R^\star}(R'(q) - u, R'(q))$.
\explanr{} follows from the definitions of $p_{\post}$ and $p_{\prior}$;
\explanr{} by the definition of the dual Bregman divergences.\index{Bregman divergence}
\explanr{} follows since Bregman divergences are always non-negative.
\explanr{} follows from duality (\cref{prop:breg:dual}).\index{duality!Bregman divergence}
Ideally we would now choose $p$ to be the minimiser of the information ratio $\Psi(\cdot, \nu)$, but this is generally only supported on two coordinates and hence
not in $\Delta_m^+$ for $m > 2$. Fortunately it is straightforward to show that when $p$ minimises $\Psi(\cdot, \nu)$, then $[0,1) \ni \delta \mapsto \Psi((1-\delta)p + \delta \ones/m, \nu)$
is continuous and hence there exists a $p \in \Delta_m^+$ such that $\Psi(p, \nu) \leq \Psi^\star + \eps$ for any $\eps > 0$.
Let $\eps > 0$ and $p \in \Delta_m^+$ be such that $\Psi(p, \nu) \leq \Psi^\star + \eps$. Then
\begin{align*}
\Lambda_{\eta,q}(p, e \LD \nu)
&\leq \frac{\Delta(p, \nu)}{\eta} - \frac{1}{\eta^2} \E\left[\Breg_R(p_{\post}, p_{\prior})\right] \\
&\leq \frac{1}{\eta} \sqrt{(\Psi^\star + \eps) \E[\Breg_R(p_{\post}, p_{\prior})]} - \frac{1}{\eta^2} \E[\Breg_R(p_{\post}, p_{\prior})] \\
&\leq \sup_{x \geq 0} \left(x \sqrt{\Psi^\star+\eps}- x^2 \right) \\
&= \frac{\Psi^\star+\eps}{4} \,,
\end{align*}
where in the second inequality we used the definition of $p$ and the information ratio.
Since $\eps > 0$ was arbitrary, \cref{eq:exp:dual-2} holds and the proof is complete.
\end{proof}

\begin{exer}
\faStar \faStar \faBook \quad
Make the proof of \cref{thm:bayes-ig-bound} fully rigorous. 
\end{exer}

\newcommand{\TV}{\operatorname{\textsc{tv}}}

\solution{
The main problem working out how to apply Sion's theorem.
To simplify the notation let $q \in \Delta_m^+$ and $\eta > 0$ be fixed and drop them from the subindex notation for the remainder.
Let $u \in \ones/m$ be the uniform distribution in $\Delta_m$ and for $\eps \in (0,1)$ let $\Delta_m^\eps = \eps u + (1 - \eps) \Delta_m \subset
\Delta_m^+$.

\stepsection{Step 1: Bounds on $\Lambda$}
We start by lower bounding $\Lambda(p, e \LD p_\star, f)$.
Abbreviate $g(x) = \eta e(x, f(x)) / p(x)$
\begin{align*}
\Lambda(p, e \LD p_\star, f)
&= \frac{1}{\eta} \left[\ip{p - p_\star, f} + \sum_{x \in \cC} \frac{p(x)}{\eta} \left(\ip{p_\star - q, g(x)} + \cS\left(g(x)\right)\right)\right] \,.
\end{align*}
By definition,
\begin{align*}
&\ip{p_\star - q, g(x)} + \cS(g(x))
= \ip{p_\star - q, g(x)} + \Breg_{R^\star}(R'(q) - g(x), R'(q)) \\
&= \ip{p_\star, g(x)} + R^\star(R'(q) - g(x)) - R^\star(R'(q)) \\
&= -\ip{p_\star, R'(q) - g(x)} + \ip{p_\star, R'(q)} + R^\star(R'(q) - g(x)) - R^\star(R'(q)) \\
&\explana\geq \ip{p_\star, R'(q)} - R^\star(R'(q)) - R(p_\star) \\
&\explana\geq -\norm{R'(q)}_\infty - R^\star(R'(q)) - R(p_\star) \,. 
\end{align*} 
where \explanr{} follows from Fenchel-Young and \explanr{} since $p_\star \in \Delta_m$.
Since $R$ is bounded on $\Delta_m$, it follows that the above quantity can be lower bounded by a constant depending only on $q$ and $\eta$.
And since $f \in \cF_{\pb}$ and $p, p_\star \in \Delta_m$ it follows that the same holds for $\Lambda(p, e \LD p_\star, f)$.
That is, there exists a real-valued $C$ depending only on $q$ and $\eta$ such that
\begin{align*}
\Lambda(p, e \LD p_\star, f) \geq -C \,.
\end{align*}
Moreover, since $\cS(\zeros) = 0$ and $f \in \cF_{\pb}$,
\begin{align*}
\Lambda(p, \zeros \LD p_\star, f) \leq \frac{1}{\eta} \,.
\end{align*}

\stepsection{Step 2: Domain restrictions}
Since $p_\star \mapsto \Lambda(p, e \LD p_\star, f)$ is linear,
\begin{align*}
\Lambda(p, e \LD (1-\eps)p_\star + \eps u, f) 
&= (1-\eps) \Lambda(p, e \LD p_\star, f) + \eps \Lambda(p, e \LD u, f) \\
&\geq (1-\eps) \Lambda(p, e \LD p_\star, f) - \eps C \,.
\end{align*}
Rearranging shows that
\begin{align*}
\Lambda(p, e \LD p_\star, f) \leq \frac{\eps C}{1-\eps} + \frac{\Lambda(p, e \LD (1-\eps)p_\star + \eps u, f)}{1 - \eps}\,.
\end{align*}
Taking the supremum on both sides shows that
\begin{align*}
\Lambda^\star(p, e) \leq \frac{\eps C}{1-\eps} + \frac{\Lambda^\star_\eps(p, e)}{1-\eps}\,,
\end{align*}
where $\Lambda^\star_\eps(p, e) = \sup_{p_\star \in \Delta_m^\eps} \sup_{f \in \cF_\pb} \Lambda(p, e \LD p_\star, f)$.
Hence it suffices to show that $\Lambda^\star_\eps(p, e) \leq \frac{1}{4} \Psi^\star$.
Without loss of generality, let $\eps$ be small enough that $q \in \Delta_m^\eps$ and let $E > 0$ be a constant such that
\begin{align*}
\sup_{p \in \Delta_m, r \in \Delta_m^\eps} \norm{\frac{p(x)}{\eta}(R'(q) - R'(r))}_\infty < E \,.
\end{align*}
The constant $E$ depends on $\eps$ and $q$ and $\eta$ only.

\stepsection{Step 3: Topologies and Sion's theorem}
Let $\cE_\eps$ be the set of functions from $\cC \times \R$ to $[-E, E]^m$.
We endow $\cE_\eps$ with the product topology, which is the coarsest topology such that $(x, y) \mapsto e(x, y)$ is continuous
for all $x \in \cC$ and $y \in \R$.
Tychonoff's theorem says that $\cE_\eps$ with this topology is compact.
The topology on $\Delta_m^\eps$ will be the usual one, under which it is also compact. 
Hence $A = \Delta_m^\eps \times \cE_\eps$ with the product topology is compact.
Next, let $B = \Delta_m^\eps \times \cF_\pb$ with the discrete topology, which means that all functions on $B$ are continuous.
Let $C_b(B)$ be the space of bounded continuous functions from $B \to \R$ with the supremum norm, which is a Banach space.
The topological dual of this space is $C_b(B)^*$, which is the topological vector space consisting of continuous linear functionals on $C_b(B)$.
As usual, when $f \in C_b(B)$ and $\nu \in C_b(B)^*$ we write $\ip{f, \nu} = \nu(f)$.
Let $\Delta(B)$ be the space of discrete probability measures on $B$. That is, distributions in $\Delta(B)$ are supported on countably many
atoms.
Distributions $\nu \in \Delta(B)$ can be viewed as linear functionals on $C_b(B)$ via integration: $\ip{f, \nu} = \int_B f \d{\nu}$.
Since functions $f \in C_b(B)$ are bounded, then $|\ip{f, \nu}| \leq \norm{f}_\infty$, which means that $\ip{\cdot, \nu}$ 
is bounded and hence continuous and in $C_b(B)^*$. 
We take as the topology on $C_b(B)^*$ the weak$^*$ topology, which is the coarsest topology such that $\ip{f, \cdot}$ is continuous
for all $f \in C_b(B)$. By definition, as a function on $A \times B$, $\Lambda$ is bounded: $\norm{\Lambda}_\infty < \infty$.
Therefore $\Lambda(p, e \LD \cdot) \in C_b(B)$ whenever $p, e \in A$.
Hence $\nu \mapsto \ip{\Lambda(p, e \LD \cdot), \nu}$ is continuous
Moreover $(p, e) \mapsto \Lambda(p,e \LD \nu) \triangleq \ip{\Lambda(p, e \LD \cdot), \nu}$ is continuous by the definition of the product topology and convex by \cref{ex:exp:cvx}.
Hence, by Sion's minimax theorem,
\begin{align*}
\min_{p, e \in A} \sup_{\nu \in \Delta(B)} \Lambda(p, e \LD \nu) 
&= \sup_{\nu \in \Delta(B)} \min_{p \in \Delta_m^\eps, e \in \cE_\eps} \Lambda(p, e \LD \nu) \\
&\leq \sup_{\nu \in \Delta(b)} \inf_{p \in \Delta_m^+} \min_{e \in \cE_\eps} \Lambda((1-\eps) p + \eps u, (1 - \eps) e \LD \nu) \\
&\leq (1 - \eps) \sup_{\nu \in \Delta(B)} \inf_{p \in \Delta_m^+} \min_{e \in \cE_\eps} \Lambda(p, e \LD \nu) + \frac{\eps}{\eta} \,.
\end{align*}
Of course the supremum on the left-hand side can be restricted to the Dirac measures so that
\begin{align*}
\min_{p \in \Delta_m^\eps} \min_{e \in \cE_\eps} \Lambda_\eps^\star(p, e) 
\leq (1 - \eps) \sup_{\nu \in \Delta(B)} \inf_{p \in \Delta_m^+} \min_{e \in \cE^\eps} \Lambda(p, e \LD \nu) + \frac{\eps}{\eta} \,.
\end{align*}
\stepsection{Step 4: Finishing}
Suppose that $\nu \in \Delta(B)$ and $p \in \Delta_m^+$ and suppose that $(p_\star, f)$ has law $\nu$.
Let
\begin{align*}
e(x, y) = \frac{p(x)}{\eta} (R'(q) - R'(\E[p_\star|f(x) = y]))\,,
\end{align*}
whenever the conditional expectation is defined and otherwise let $e(x, y) = 0$.
By assumption $p_\star \in \Delta_m^\eps$, which by convexity means that $\E[p_\star|f(x) = y] \in \Delta_m^\eps$ for any $x, y$ such
that $f(x) = y$ has nonzero probability. 
Hence, by definition, $e \in \cE_\eps$.
Repearting now the proof that $p$ can be chosen so that
\begin{align*}
\Lambda(p, e \LD \nu) \leq \frac{\Psi^\star + \eps}{4} \,. 
\end{align*}
Since this holds for any $\nu \in B_\eps$,
\begin{align*}
\inf_{p \in \Delta_m^+} \inf_{e \in \cE} \Lambda^\star(p, e)
&\leq \frac{\eps C}{1-\eps} + \inf_{p \in \Delta_m^+} \inf_{e \in \cE} \frac{\Lambda^\star_\eps(p, e)}{1 - \eps} \\
&= \frac{\eps C}{1-\eps} + \inf_{p \in \Delta_m^+} \inf_{e \in \cE} \sup_{\nu \in \Delta(B)} \frac{\Lambda^\star_\eps(p, e)}{1 - \eps} \\
&\leq \frac{\eps C}{1-\eps} + \inf_{p \in \Delta_m^\eps} \min_{e \in \cE_\eps} \sup_{\nu \in \Delta(B)} \frac{\Lambda^\star_\eps(p, e)}{1 - \eps} \\
&\leq \frac{\eps (C+1/\eta)}{1-\eps} + \sup_{\nu \in \Delta(B)} \inf_{p \in \Delta_m} \min_{e \in \cE_\eps} \Lambda^\star_\eps(p, e) \\
&\leq \frac{\eps (C+1/\eta)}{1-\eps} + \frac{\Psi^\star + \eps}{4} \,.
\end{align*}
Taking the limit as $\eps \to 0$ completes the proof.
}

\begin{figure}[h!]
\caption{The relationship between the results in this chapter, showing 
two ways to bound $\Reg_n^\star$. The first (Bayesian) is completely non-constructive via minimax duality\index{duality!minimax} (\cref{thm:minimax}). 
The second (adversarial) is by bounding the 
regret of \cref{alg:exp-by-opt}. The latter is only to be preferred slightly since \cref{alg:exp-by-opt} has no obvious efficient implementation.
}
\commentAlt{
Two paths for bounding the adversarial regret are shown. One by non-constructively equating the Bayesian and adversarial minimax regrets and then
using the information-theoretic machinery to bound the Bayesian regret. The other using exploration-by-optimisation and the equivalent
between the stability of mirror descent and the minimax information ratio.
}

\label{fig:duality}
\begin{center}
\renewcommand{\arraystretch}{2}
\begin{tabular}{|lccc|}
\hline
& & & $O\left(d^{2.5} \sqrt{n} \log(nd)\right)$ \\
& & & $\stackrel{\text{\cref{thm:inf-ratio}}}{\rotatebox{90}{$\leq$}}$ \\ 
\textsc{bayesian} & $\bReg_n^\star$ & $\stackrel[\text{\cref{alg:ids}}]{\text{\cref{thm:bayes-reg-bound}}}\leq$ & $O\left(\sqrt{nd \Psi^\star \log(nd)}\right)$ \\ 
& $\stackrel{\text{\cref{thm:minimax}}}{\rotatebox{90}{=}}$ & & $\stackrel{\text{\cref{thm:bayes-ig-bound}}}{\rotatebox{90}{$\leq$}}$  \\
\textsc{adversarial} & $\Reg_n^\star$ & $\stackrel[\text{\cref{alg:exp-by-opt}}]{\text{\cref{thm:exp-by-opt}}}\leq$ & $O\left(\sqrt{nd \Lambda^\star \log(nd)}\right)$ \\ \hline 
\end{tabular}
\end{center}
\end{figure}

\FloatBarrier

\section{Notes}

\begin{enumeratenotes}
\item Aside from inconsequential simplifications, the kernel-based method in one dimension was designed by \cite{BEL16}. They extended the general idea to the
higher dimensions to design a polynomial time algorithm with regret $d^{10.5} \sqrt{n}$, which was the first polynomial time algorithm with $\poly(d) \sqrt{n}$ regret in
the adversarial setting. Sadly there are many challenges to generalising \cref{alg:exp-1-disc} and ultimately the higher-dimensional version is not realistically
implementable.
\item Information-directed sampling and the core analysis was introduced by \cite{RV14}. 
The idea has been generalised to frequentist settings, which are explained in depth by
\cite{kirschner2021information} who also details many properties of the information ratio. 
The application to convex bandits to prove bounds non-constructively for \index{non-constructive}\index{information-directed sampling}
adversarial bandit problems is by \cite{BDKP15}, who were the first to show that $\tilde O(\sqrt{n})$ regret is possible for adversarial convex bandits for losses in $\cF_\pb$. 
The extension to higher dimensions is by \cite{BE18} and \cite{Lat20-cvx}. The latter shows that the minimax regret for adversarial bandits is at most $\tilde O(d^{2.5} \sqrt{n})$.
This remains the best bound known, though it is matched in all but logarithmic terms by online Newton step (\cref{chap:ons-adv}).
\item The oldest and most well-known algorithm for Bayesian bandits is Thompson sampling \citep{Tho33},\index{Thompson sampling} 
which in every round samples a loss function from the posterior\index{posterior} and plays
the action that minimises the sampled loss. This algorithm has near-optimal Bayesian regret for some models, including finite-armed bandits and linear bandits \citep{RV16}.\index{bandit!finite-armed}\index{bandit!linear} 
For convex bandits \cite{BLS25} showed that Thompson sampling has $\tilde O(\sqrt{n})$ Bayesian regret in the stochastic setting when $d = 1$.\index{regret!Bayesian}
They also show that for large $d$ there exist priors for which the Bayesian regret of Thompson sampling is exponential in the dimension. 
In the Bayesian adversarial setting with $d = 1$ it is not known
if Thompson sampling has $\tilde O(\sqrt{n})$ regret. \cite{BDKP15} showed that it does not have a bounded information ratio, which means that new proof techniques would be needed
to prove $\tilde O(\sqrt{n})$ regret.\index{prior}
\item The duality between mirror descent and the information ratio was established by \cite{ZL19} and \cite{lattimore2021mirror} with the latter proving the more difficult direction. \index{duality!information ratio}
These connections have led to a beautiful theory on the complexity of sequential decision making in great 
generality \citep{foster2021statistical,foster2022complexity}. In brief, algorithms like exploration-by-optimisation are provably near-optimal in a minimax sense. There
are many subtleties and you should read the aforementioned works. \index{exploration by optimisation}
\item \cite{BLS25} prove that the minimax information ratio satisfies $\Psi^\star = \tilde \Omega(d^2)$, which shows that the best possible bound obtainable via a naive
application of the information-theoretic machinery is $\tilde O(d^{1.5} \sqrt{n})$. 

\begin{exer}
\faStar \faStar \faBook \faQuestion \quad 
Use the arguments by \cite{ZL19} and the lower bound on $\Psi^\star$ by \cite{BLS25} to prove that $\Lambda^\star = \tilde \Omega(d^2)$.
\end{exer}

The exercise shows that no matter how you explore or estimate losses, the classical analysis of exponential weights cannot yield a bound on the regret better than $\tilde O(d^{1.5} \sqrt{n})$.
Importantly, however, lower bounds on the complexity measures do not imply lower bounds\index{lower bound} on the minimax regret. There exist other settings where the minimax regret is better than the upper bound 
in \cref{thm:bayes-reg-bound}. For example, \cite{lattimore2021bandit}
show that
in bandit phase retrieval\index{bandit!phase retrieval} the information-theoretic machinery 
suggests a bound of $\tilde O(d^{1.5}\sqrt{n})$ while the minimax regret is $\tilde \Theta(d \sqrt{n})$. 

\end{enumeratenotes}

\chapter[Cutting Plane Methods]{Cutting Plane Methods\copynotice}\label{chap:ellipsoid}\index{cutting plane methods} 

Like the bisection method\index{bisection method} (\cref{chap:bisection}), cutting plane methods are most naturally suited to the stochastic setting.\index{setting!stochastic}
For the remainder of the chapter we assume the setting is stochastic and the loss function is bounded:

\begin{assumption}\label{ass:cut}
The following hold:
\begin{enumerate}
\item The setting is stochastic, meaning that $f_t = f$ for all $t$; and
\item the loss $f$ is in $\cF_\pb$.
\end{enumerate}
\end{assumption}

The bounds established in this chapter are worse than what we will show for online Newton step\index{online Newton step} in \cref{chap:ons}, but the analysis is considerably more
straightforward and the algorithms are easy to tune.
To keep things simple we study the sample complexity rather than the regret, though most likely the algorithms and analysis can be adapted to the regret setting
without too much difficulty as we discuss briefly in the notes.\index{sample complexity}
This chapter also introduces an algorithm for infinite-armed bandits\index{bandit!infinite-armed} that may be of independent interest. 
The highlight of the chapter is a mechanism for finding
a suitable cutting plane with only noisy zeroth-order access to the loss function (\cref{sec:cut:cut}). 
This is then applied to bound the sample complexity of the centre of gravity method (\cref{sec:cut:cog}) and 
the method of the inscribed ellipsoid (\cref{sec:cut:inscribed-alg}).
In both cases the sample complexity is $\tilde O(d^4 / \eps^2)$ with different computational properties.
The ellipsoid method is discussed in \cref{sec:cut:ellipsoid} and has a moderately worse sample complexity.

\begin{remark}
This chapter is full of half-spaces and ellipsoids, so let us remind you that
$E(x, A) = \{y \in \R^d : \norm{x - y}_{A^{-1}} \leq 1\}$ is an ellipsoid centred at $x$ and $H(x, \eta) = \{y \in \R^d : \ip{y - x, \eta} \leq 0\}$
is a half-space.
\end{remark}

\subsubsection*{Stochastic Oracle Notation}\index{stochastic oracle}
Because of the modular nature of the algorithms in this chapter, it is not practical to keep track of the round of interaction. This necessitates 
a new notation for the interaction protocol.
When we write $y \sim f(x)$ in an algorithm it means that $y = f(x) + \eps$ where $\eps$ is $1$-subgaussian
conditioned on the history (all previous queries).\index{subgaussian}

\section{High-Level Idea}
The basic idea is to let $S$ be a subset of $K$ with non-negligible volume on which the loss $f$ is nearly minimised.
Then initialise $K_1 = K$ and recursively compute a decreasing sequence $(K_k)$ of subsets such that at least one of the following holds:
\begin{itemize}
\item The `centre' $x_k$ of $K_k$ is a near-minimiser of $f$.
\item $K_{k+1} \subset K_k$ contains $S$.
\end{itemize}
Note there are many definitions of the centre of a convex body, as we will soon see.
The largest $k_{\max}$ such that $K_{k_{\max}}$ contains $S$ can be bounded as a function of how fast $k \mapsto \vol(K_k)$ decreases. Combining this with the above requirements on $(K_k)$ and $(x_k)$ yields a bound on how many iterations $k_{\max}$ are needed
before there exists a near-minimiser among $x_1,\ldots,x_{k_{\max}}$. Generally speaking $k_{\max} = \poly(d)$ and the final step is to query the loss on $x_1,\ldots,x_{k_{\max}}$
to identify a near-minimiser among them (\cref{sec:cut:bai}).
In the one-dimensional problem there is limited scope for imagination but in high dimensions there are several intersecting complexities, namely, what geometric
procedure will reduce the volume sufficiently fast? Can it be computed efficiently and how does it interact with convexity? Standard methods are:
\begin{itemize}
\item the ellipsoid method \citep{shor1977cut,judin1977evaluation,yudin1976informational};
\item the centre of gravity method \citep{New65,levin1965algorithm};
\item Vaidya's method \citep{vaidya1996new} and its refinement by \cite{lee2015faster};\index{Vaidya's method}
\item the analytic centre method \citep{nesterov1995cutting,atkinson1995cutting};
\item the method of the inscribed ellipsoid \citep{tarasov1988method}.
\end{itemize}
We focus on the centre of gravity method, ellipsoid method and method of the inscribed ellipsoid. Let us make the considerations above a little more concrete.
By \cref{lem:cut:vol} below there exists a (usually non-regular) 
simplex\index{simplex} $S \subset K$ such that the loss $f$ is near-optimal for all $x\in S$ and $\log(\vol(K)/\vol(S)) = \tilde O(d)$.
Remember a half-space is a set $H = \{y \colon \ip{y - x, \eta} \leq 0\} \triangleq H(x, \eta)$ for nonzero direction $\eta \in \R^d$ and point $x \in \R^d$.
We will study three methods, which classically operate as follows (see also \cref{fig:cut}):
\begin{itemize}
\item The centre of gravity method\index{centre of gravity method} starts with $K_1 = K$ and iteratively updates $K_{k+1} = K_k \cap H_k$ where $H_k$ is a half-space with boundary $\partial H_k$ passing close to the 
centre of mass $x_k$ of $K_k$.
We additionally insist that either $x_k$ is near-optimal or $H_k$ contains $S$. 
A generalisation of Gr\"unbaum's inequality\index{Gr\"unbaum's inequality} shows that $\log \vol(K_{k+1}) \leq \log\vol(K) - ck$ for 
some universal constant $c>0$. Suppose now that $x_1,\ldots,x_k$ are not near-optimal. Then by induction $S \subset K_{k+1}$.
But this implies that $\vol(K_{k+1}) \geq \vol(S)$ and this is only possible for $k = \tilde O(d)$. Consequentially, if $k = \tilde \Theta(d)$, then one of $x_1,\ldots,x_k$ is near-optimal.
With deterministic zeroth-order access to the loss function the learner can simply return the $\argmin \{f(x) \colon x \in x_1,\ldots,x_k\}$ while with noise it can treat the $k$ candidates as a finite-armed bandit\index{bandit!finite-armed} and use an elementary pure exploration bandit algorithm to approximate the $\argmin$ as explained in Section~\ref{sec:cut:bai}. In \cref{rem:cut:simplex} we explain
why $S$ is chosen to be a simplex.

\item The method of the inscribed ellipsoid is the same as the centre of gravity method, but rather than using centre of mass it uses the centre of the largest ellipsoid contained in $K_k$. 
It can be shown that $\log \vol(K_{k+1}) \leq \log\vol(K) + \tilde O(d) - c k$ for some universal constant $c > 0$. Hence, like the centre of gravity method, the number of iterations where $S \subset K_{k+1}$ is at most $\tilde O(d)$.
\item The ellipsoid method \index{ellipsoid method} starts with an ellipsoid $E_1$ such that $K \subset E_1$. The ellipsoid is updated by finding a half-space $H_k$ with boundary passing close to the centre of
$E_k$ and such that $S \subset H_k$ and $E_{k+1}$ is calculated as the smallest ellipsoid containing $E_k \cap H_k$. 
The classical theory of the shallow cut ellipsoid method shows that $\log\vol(E_{k+1}) \leq \log\vol(E_1) - \frac{ck}{d}$ for some universal constant $c > 0$.
Hence, provided that $\log\vol(E_1) \leq \log \vol(K) + \tilde O(d)$, the number of iterations is at most $\tilde O(d^2)$.
\end{itemize}
The big question is how to find the half-spaces passing close to the relevant centre that contains the near-optimal simplex with high probability.
Besides this there are many details to be sorted out. Most notably, how close to the centre of mass or centre of ellipsoid do we need the half-space to be?

\begin{remark}
The centre of gravity and inscribed ellipsoid methods require only $\tilde O(d)$ iterations while the ellipsoid method needs $\tilde O(d^2)$.
The advantage of the latter is the remarkable fact that the smallest ellipsoid containing $E \cap H$ for ellipsoid $E$ and half-space $H$ has a closed-formed expression,
while estimating the centre of mass or finding the maximum-volume inscribed ellipsoid is less elementary. 
\end{remark}

\begin{figure}[h!]
\centering
\begin{subfigure}{0.3\textwidth}
\includegraphics[width=3cm]{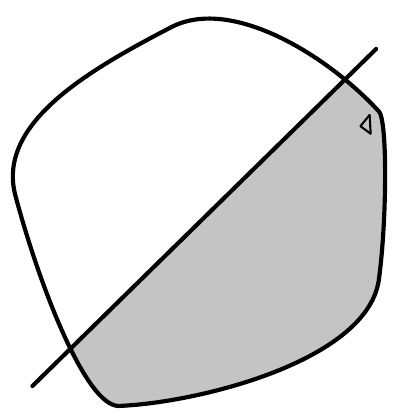}
\caption{Centre of gravity}
\end{subfigure}
\begin{subfigure}{0.3\textwidth}
\includegraphics[width=3cm]{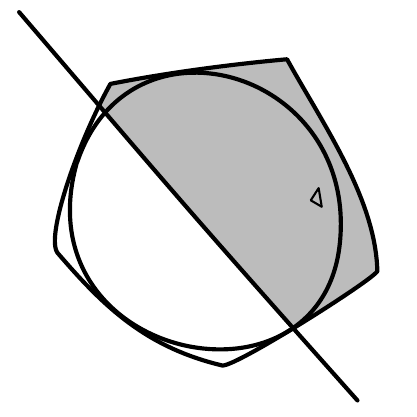}
\caption{Inscribed ellipsoid}
\end{subfigure}
\begin{subfigure}{0.3\textwidth}
\includegraphics[width=3cm]{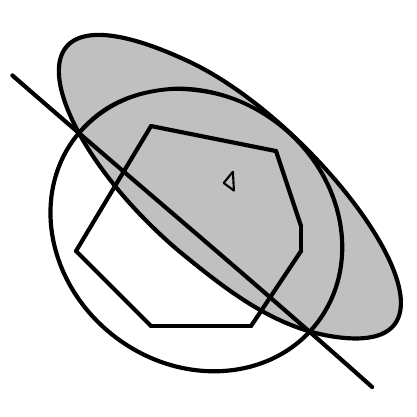}
\caption{Ellipsoid method}
\end{subfigure}
\caption{An illustration of one iteration of the centre of gravity, inscribed ellipsoid and ellipsoid methods.
}
\label{fig:cut}
\commentAlt{Three two-dimensional convex bodies are shown. The left-most has a half-space through its centre of mass illustrating that about half the volume
is in one side relative to the other. The second has a half-space through the centre of its maximum volume enclosed ellipsoid. The third has a half-space
through the centre of an ellipsoid containing it and a new ellipsoid formed as the minimum volume ellipsoid of the intersection of the half-space and original ellipsoid.}
\end{figure}

We finish with the promised lemma establishing the existence of a near-optimal simplex $S$.

\begin{lemma}\label{lem:cut:vol}
Suppose that $K$ is a convex body and $f \in \sF_{\pb}$. Then, for any $\eps \in (0,1)$, 
there exists a simplex\index{simplex} $S = \conv(x_1,\ldots,x_{d+1}) \subset K$ such that
$f(x) \leq \inf_{y \in K} f(y) + \eps$ for all $x \in S$ and 
\begin{align*}
\vol(S) \geq \left(\frac{\eps}{2}\right)^d \frac{\vol(K)}{d! \vol(\ball_{d})} \,.
\end{align*}
In particular, by \cref{prop:vol}\ref{prop:vol:ball}, $\log(\vol(S)) = \log(\vol(K)) - O\left(d \log(d/\eps)\right)$.
\end{lemma}

\begin{proof}
The construction is illustrated in \cref{fig:cut:vol}.
Let $f_\star = \inf_{x \in K} f(x)$.
Suppose that $K$ is in John's position so that by \cref{thm:john}, $\ball_1 \subset K \subset \ball_{d}$.
Letting $e_1,\ldots,e_d$ be the standard basis vectors,
then obviously $T = \conv(\zeros, e_1,\ldots,e_d) \subset K$. 
Let $x$ be some point such that $f(x) \leq f_\star + \frac{\eps}{2}$ and
$S = \{(1-\frac{\eps}{2}) x + \frac{\eps}{2} y \colon y \in T\}$. 
Then, for any $z \in S$ there exists a $y  \in T$ such that $z = (1 - \frac{\eps}{2}) x + \frac{\eps}{2} y$ and since $f \in \cF_{\pb}$ is bounded and convex, $f(z) \leq (1 - \frac{\eps}{2}) f(x) + \frac{\eps}{2} f(y) \leq f(x) + \eps/2 \leq f_\star + \eps$.
Furthermore, $\vol(T) = 1/d!$ and 
\begin{align*}
\vol(S) 
&= \left(\frac{\eps}{2}\right)^d \vol(T) 
= \left(\frac{\eps}{2}\right)^d \frac{1}{d!} \,. 
\end{align*}
Meanwhile, $\vol(K) \leq \vol(\ball_{d})$, which implies that
\begin{align*}
\vol(S) \geq \left(\frac{\eps}{2}\right)^d \frac{\vol(K)}{d! \vol(\ball_d)} \,. 
\end{align*}
The result for general $K$ follows via an affine map\index{affine!map} into John's position.
\end{proof}

\begin{remark}\label{rem:cut:ratio}
Lemma~\ref{lem:cut:vol} is moderately crude. For example, you can improve the result by taking $T$ to be the regular simplex inside $\ball_1$.
Alternatively one may try to avoid using John's theorem\index{John's theorem}, letting $T$ be the simplex of largest volume contained in $K$. For this simplex it is known that for all
convex bodies $K$, $(\vol(T)/\vol(K))^{1/d} \geq c/\sqrt{d}$ for some absolute constant $c > 0$ and this is not improvable \citep{galicer2019minimal}.
For our purpose, however, these refinements make only negligible differences to the constants in our regret bounds.
\end{remark}

\begin{remark}\label{rem:cut:simplex}
You might wonder why $S$ is a simplex rather than just the level set $\{x \colon f(x) \leq f_\star + \eps\}$, which contains $S$.
The reason is that later we will want to prove that $S$ is in some randomised half-space with high probability and for this it suffices to show that the vertices of $S$ are contained
in the half-space with high probability, which involves a union bound over the $d+1$ vertices. Any convex shape with $\poly(d)$ vertices would be sufficient.
But there is no hope to improve the bounds in this chapter by modifying this construction except for miniscule constants.
\end{remark}

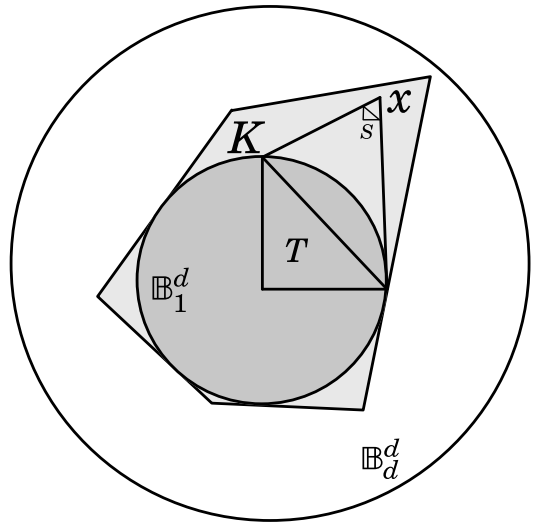
\begin{figure}[h!]
\begin{tikzpicture}[thick]
\draw (0,0) circle (4);
\draw[fill=grayone] (-2.3,0) -- (-1.45,-1.5) -- (-0.3, -2.4) -- (1.9,-1.3) -- (2.3,3) -- (-0.6,3) -- (-2.3,0);
\draw[fill=graytwo] (0,0) circle (2);
\draw[fill=graythree] (0,0) -- (2,0) -- (0,1.98) -- (0,0);
\node at (0.66, 0.66) {$T$};
\node at (1,-3) {$\ball_d$};
\node at (-1,0) {$\ball_1$};
\node at (0,2.5) {$K$};
\draw (2,0) -- (1.8,2.8) -- (0,2);
\node[anchor=west] at (1.8,2.8) {$x$};
\draw[fill=graythree] (1.82, 2.52) -- (1.62, 2.705) -- (1.62, 2.52) -- (1.82, 2.52);
\node[anchor=north] at (1.62, 2.52) {\small $S$};
\end{tikzpicture}
\caption{The construction used in the proof of \cref{lem:cut:vol}.}\label{fig:cut:vol}
\commentAlt{A two-dimensional convex body K in John's position with the simplex T spanned by the basis vectors and the origin. A point x is near the corner of
K and S is constructed as the convex combination of T and x.}
\end{figure}
\FloatBarrier

\section{Infinite-Armed Bandits}\label{sec:cut:inf} \index{bandit!infinite-armed|(}

We make a brief aside to study a kind of infinite-armed bandit problem.
Suppose that $\rho$ is a probability measure on $K$ and $h \colon K \to (-\infty,1]$ is an unknown function such that $\E[h(X)] \geq 0$ when $X$ has law $\rho$.
In contrast to the rest of the book, we are looking for a procedure for (crudely) maximising $h$.
We suppose a learner can sample from $\rho$ and for any $x \in K$ the learner can obtain an unbiased estimate of $h(x)$ with subgaussian\index{subgaussian} tails (just as in the standard convex bandit problem).
We formalise this with the following assumption:

\begin{assumption}\label{ass:cut:inf}
The learner can sample from a probability measure $\rho$ on $K$ and $h \colon K \to (-\infty, 1]$ is a function such that
\begin{align*}
\int_K h(x) \d{\rho}(x) \geq 0 \,.
\end{align*}
For any $x \in K$ the learner can sample from $\varkappa(\cdot | x)$ 
where $\varkappa$ is a probability kernel from $K$ to $\R$ such that for all $x \in K$ \index{probability kernel}
\begin{enumerate}
\item the mean of $\varkappa(\cdot|x)$ is $h(x)$; and
\item the probability measure $\varkappa(\cdot|x)$ is $\sigma$-subgaussian: 
\begin{align*}
\int_{\R} \exp((y/\sigma)^2) \varkappa(\d{y}|x) \leq 2 \,.
\end{align*}
\end{enumerate}
\end{assumption}

There is no geometry in this problem. The function $h$ needs to be measurable\index{measurable} but besides this there is no requirement for 
continuity or any other structural properties beyond the semi-boundedness. 
Of course, in such circumstances there is no hope whatsoever to actually maximise $h$: for example, when it is the indicator function of some singleton and $\rho$ is continuous.
What can be achieved depends on the law of $h(X)$ when $X$ has law $\rho$. Since we have assumed that $\E[h(X)] \geq 0$, it seems reasonable that we can find a point $x$ such that $h(x)$ is close to $0$.
Given an $\eps \in (0,1)$ and $\delta \in (0,1)$, 
we construct an algorithm that with probability at least $1 - \delta$ returns an $x$ such that $h(x) \geq -\eps$
and in expectation takes at most
\begin{align*}
O\left(\frac{\sigma^2 \log(1/\delta)}{\eps^2}\right)
\end{align*}
samples from the probability kernel $\varkappa$.

\begin{remark}
The expected number of times \cref{alg:cut:inf}, presented below, samples from $\rho$ is $O(1/\eps)$.
In our application, samples from $\varkappa(\cdot|x)$ correspond to querying the loss function
while $\rho$ is an explicit distribution on $K$ that is easy to sample from.
\end{remark}

The basic idea is to sequentially sample points $x$ from $\rho$ and then take samples from $\varkappa(\cdot | x)$ to test whether or not $h(x)$ is suitably large.
The big question is how many samples to take from $\varkappa(\cdot | x)$. And in fact the algorithm will vary the number of samples it takes, for reasons we now explain.
Remember we assumed that $\E[h(X)] \geq 0$. There are multiple ways this can happen. Here are two extreme examples:
\begin{itemize}
\item $h(X) = 2\eps$ with probability $1/2$ and $h(X) = -2\eps$ otherwise
\item $h(X) = 1$ with probability $2\eps/(1 + 2\eps)$ and $h(X) = -2\eps$ otherwise
\end{itemize}
And of course there are many intermediate options and mixtures. Suppose we want to design an algorithm that finds an $x$ with $h(x) \geq -\eps$ for distributions of the first kind.
Then we should repeatedly sample points $x$ from $\rho$ and then query $\varkappa(\cdot|x)$ roughly $\tilde O(\sigma^2/\eps^2)$ times 
until we can statistically prove that $h(x)$ is large enough.
Since by assumption $h(X) \geq -\eps$ with constant probability, in expectation this requires only $\tilde O(\sigma^2/\eps^2)$ queries.
On the other hand, for distributions of the second kind the algorithm needs in expectation $O(1/\eps)$ queries to $\rho$ to find an $x$ with $h(x) = 1$.
But to identify when this has happened only requires $\tilde O(\sigma^2)$ queries to $\varkappa(\cdot | x)$ because the margin is large.
Algorithm~\ref{alg:cut:inf} essentially dovetails these two algorithms and all intermediaries on a carefully chosen grid.
Concretely, the inner loop samples $O(1/\eps)$ points from $\rho$ and queries the kernel with sample sizes ranging from $\tilde O(\sigma^2/\eps^2)$ to $\tilde O(\sigma^2)$.
We will prove that with constant probability this inner loop succeeds in identifying an $x$ such that $h(x) \geq -\eps$.
The outer loop simply repeats the inner loop to boost the probability of success. The quantity $b_{m,k}$ is chosen so that $\sum_{k=1}^\infty \sum_{m=1}^\infty b_{m,k}^{-1} = 1$ and
scales the confidence level to permit a union bound over all $m, k$.

\begin{algorithm}[h!]
\begin{algcontents}
\begin{lstlisting}
def $\INF$($\eps \in (0,1)$, $\delta \in (0,1)$, $\rho$, $\varkappa$)
  for $m = 1$ to $\infty$:
    for $k = 1$ to $k_{\max} \triangleq 2 + \floor{\frac{2}{\eps}}$:
      $b_{m,k} = [k(k+1)][m(m+1)]$
      $n_{m,k} = \ceil{\frac{64 \sigma^2 \log\left(2b_{m,k}/\delta\right)}{\eps^2 k^2}}$ and $C_{m,k} = 2 \sigma \sqrt{\frac{\log\left(2b_{m,k}/\delta\right)}{n_{m,k}}}$
      sample $x_{m,k}$ from $\rho$
      sample $H_1,\ldots,H_{n_{m,k}}$ from $\varkappa(\cdot|x_{m,k})$
      compute $\widehat h(x_{m,k}) = \frac{1}{n_{m,k}} \sum_{t=1}^{n_{m,k}} H_t$
      if $\widehat h(x_{m,k}) - C_{m,k} \geq -\eps$: return $x_{m,k}$
\end{lstlisting}
\caption{Infinite-armed bandit algorithm}
\label{alg:cut:inf}
\end{algcontents}
\end{algorithm}
\FloatBarrier

\begin{theorem}\label{thm:cut:inf}
Suppose that \cref{alg:cut:inf} is run with inputs $\eps \in (0,1)$, $\delta \in (0,1)$ and probability measure $\rho$ and $\varkappa$
satisfying \cref{ass:cut:inf}. Then the following hold:
\begin{enumerate}
\item In expectation \cref{alg:cut:inf} takes at most $O\left(\frac{\sigma^2 \log(1/\delta)}{\eps^2}\right)$ samples from the probability kernel $\varkappa$. \label{thm:cut:inf:sc}
\item With probability at least $1-\delta$ \cref{alg:cut:inf} returns an $x$ such that $h(x) \geq -\eps$. \label{thm:cut:inf:correct}
\end{enumerate}
\end{theorem}

\begin{proof}[\Proofskippy] 
We proceed in two steps.

\stepsection{Step 1: Setup, concentration and correctness}\index{concentration}
Let $\bbP_m$ be the probability measure obtained by conditioning on all data obtained during the first $m$ outer loops, let
$G_m$ be the event defined by
\begin{align*}
G_m = \bigcap_{k=1}^{k_{\max}} \left\{\left|\widehat h(x_{m,k}) - h(x_{m,k})\right| \leq C_{m,k} \right\}
\end{align*}
and let $G = \cap_{m=1}^\infty G_m$, which are events that certain estimates lie close to the truth.
We also define an event $V_m$ that holds in outer iteration $m$ when the algorithm samples a point $x_{m,k}$ with suitably large $h(x_{m,k})$, which
is defined in terms of its complement by
\begin{align*}
V_m^c = \bigcap_{k=1}^{k_{\max}} \left\{h(x_{m,k}) < \frac{(k-2) \eps}{2}\right\} \,.
\end{align*}
Note that $(x_{m,k})_{m,k}$ are independent and identically distributed samples from $\rho$.
A union bound combined with \cref{thm:hoeffding} shows that $\bbP(G) \geq 1 - \delta$.
The same argument along with naive simplification shows that 
\begin{align}
\bbP_{m-1}(G_m) \geq 1/2 \,.
\label{eq:cut:inf:12}
\end{align}
Suppose the algorithm halts and $G$ holds.
Then
\begin{align*}
h(x_{k,m}) \geq \widehat h(x_{m,k}) - C_{m,k} \geq -\eps \,,
\end{align*}
which, since $\bbP(G) \geq 1 - \delta$, establishes correctness (part~\ref{thm:cut:inf:correct}).

\stepsection{Step 2: Bounding stopping time}
Let $M$ be the smallest $m$ such that the algorithm halts:
\begin{align*}
M = \min\left\{m \colon \max_{1 \leq k \leq k_{\max}} \left( \widehat h(x_{m,k}) - C_{m,k} \right) \geq -\eps\right\} \,.
\end{align*}
Suppose that $V_m$ and $G_m$ both hold; then there exists a $k \in \{1,\ldots,k_{\max}\}$ such that
\begin{align*}
\widehat h(x_{m,k}) - C_{m,k} \geq h(x_{m,k}) - 2C_{m,k} \geq \frac{(k-2) \eps}{2} - 2C_{m,k} \geq -\eps \,,
\end{align*}
which by construction means the algorithm halts and consequently $M \leq m$.
Suppose $X$ has law $\rho$. Then
\begin{align*}
\bbP_{m-1}(V_m^c)
&= \bbP\left(\bigcap_{k=1}^{k_{\max}} \left\{h(x_{m,k}) < \frac{(k-2) \eps}{2}\right\}\right) \\ 
&= \prod_{k=1}^\infty \bbP\left(h(X) < \frac{(k-2) \eps}{2} \right) \\
&= \exp\left(\sum_{k=1}^\infty \log \left[1 - \bbP\left(h(X) \geq \frac{(k-2) \eps}{2}\right)\right]\right) \\
&\leq \exp\left(-\sum_{k=1}^\infty \bbP\left(h(X) \geq \frac{(k-2) \eps}{2}\right)\right) \\
&\leq \exp\left(-1\right) 
\end{align*}
where for the second last inequality we used the fact that $\log(1+x) \leq x$ for all $x$. The final inequality follows because
\begin{align*}
\eps 
&\leq \E[h(X) + \eps] \\
&\leq \int_0^\infty \bbP(h(X) \geq t - \eps) \d{t} \\
&\leq \frac{\eps}{2} + \frac{\eps}{2} \sum_{k=1}^\infty \bbP\left(h(X) \geq \frac{(k-2)\eps}{2}\right) \,,
\end{align*}
where the first inequality holds because $\E[h(X)] \geq 0$ and the last 
by comparing the integral to the sum (\cref{fig:integral}).
Hence, $\bbP_{m-1}(V_m^c) \leq \exp(-1)$ and by the definitions, a union bound and \cref{eq:cut:inf:12},
\begin{align*}
\sind(M \geq m) \bbP_{m-1}(M > m)
&\leq \bbP_{m-1}(V_m^c) + \bbP_{m-1}(G_m^c) 
\leq \exp\left(-1\right) + \frac{1}{2} \leq \frac{9}{10} \,.
\end{align*}
Therefore $\bbP(M > m) \leq (9/10)^m$ by induction.
Part~\ref{thm:cut:inf:sc} follows because
\begin{align*}
\sum_{k=1}^{k_{\max}} n_{m,k} = O\left(\frac{\sigma^2 \log(m/\delta)}{\eps^2}\right) 
\end{align*}
and $\sum_{m=1}^\infty (9/10)^{m-1} \log(m/\delta) = O(\log(1/\delta))$.
\end{proof}

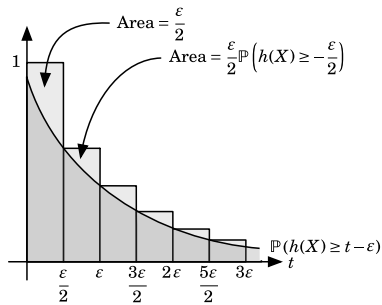
\begin{figure}[h!]
\begin{tikzpicture}[thick]
\begin{axis}[clip=false,xmin=0,xmax=5,width=7cm,height=5cm,ymax=1.2,ymin=0,xtick=\empty,ytick={0,1},legend style={at={(1.1,0.2)}, anchor=west}]
\addplot[
    draw=none, 
    fill=graytwo, 
    domain={0:5}, 
] {1/(1.1+x)} \closedcycle;

\draw[fill=graytwo,opacity=0.5] (axis cs:0,1) -- (axis cs:1,1) -- (axis cs:1,0) -- (axis cs:0,0) -- cycle;
\draw[fill=graytwo,opacity=0.5] (axis cs:1,0.47) -- (axis cs:2,0.47) -- (axis cs:2,0) -- (axis cs:1,0) -- cycle;
\draw[fill=graytwo,opacity=0.5] (axis cs:2,0.322) -- (axis cs:3,0.322) -- (axis cs:3,0) -- (axis cs:2,0) -- cycle;
\draw[fill=graytwo,opacity=0.5] (axis cs:3,0.243) -- (axis cs:4,0.243) -- (axis cs:4,0) -- (axis cs:3,0) -- cycle;
\draw[fill=graytwo,opacity=0.5] (axis cs:4,0.196) -- (axis cs:5,0.196) -- (axis cs:5,0) -- (axis cs:4,0) -- cycle;
\addplot+[mark=none,black,thick,samples=50,domain={0:5}] {1/(1.1+x)};
\addlegendentry{$\bbP(h(X) \geq t - \eps)$};
\node[anchor=north] at (axis cs:1,0) {$\eps/2$};
\node[anchor=north] at (axis cs:2,0) {$\vphantom{3/2}\eps$};
\node[anchor=north] at (axis cs:3,0) {$3\eps/2$};
\node[anchor=north] at (axis cs:4,0) {$2\eps$};
\node[anchor=north] at (axis cs:5,0) {$5\eps/2$};
\node[anchor=west] (a) at (axis cs:1.3,1.1) {Area $= \frac{\eps}{2}$};
\node[anchor=west] (b) at (axis cs:1.6,0.78) {Area $= \frac{\eps}{2} \bbP\left(h(X) \geq -\frac{\eps}{2}\right)$};
\draw[thick] (a.west) edge[-latex,out=180,in=90] (axis cs:0.5,0.8);
\draw[thick] (b.west) edge[-latex,out=180,in=90] (axis cs:1.5,0.42);
\end{axis}
\end{tikzpicture}
\caption{Integral approximation in the proof of \cref{thm:cut:inf}.}\label{fig:integral}
\commentAlt{The Riemann approximation of the integral is shown as a bar chart.}
\end{figure}

\index{bandit!infinite-armed|)}
\section{Best Arm Identification}\label{sec:cut:bai} \index{best arm identification|(}

The methods proposed in \cref{sec:cut:cog,sec:cut:inscribed-alg} effectively return a short list of candidates for near-minimisers.
We need a simple subroutine to identify a near-minimiser from this list.

\begin{algorithm}[h!]
\begin{algcontents}
\begin{lstlisting}
def $\BAI$($\eps, \delta, x_1,\ldots,x_m$)
  $n = \ceil{\frac{16\log(2m/\delta)}{\eps^2}}$
  for $k = 1$ to $m$:
    sample $Y_1 \sim f(x_k),\ldots,Y_n \sim f(x_k)$
    compute $\widehat f(x_k) = \frac{1}{n} \sum_{t=1}^n Y_t$
  return $\argmin \{\widehat f(x) \colon x \in x_1,\ldots,x_m\}$
\end{lstlisting}
\caption{Best arm identification}
\label{alg:cut:bai}
\end{algcontents}
\end{algorithm}

\begin{theorem}\label{thm:cut:bai}
With probability at least $1 - \delta$, \cref{alg:cut:bai} returns an $x \in \{x_1,\ldots,x_m\}$ such that 
\begin{align*}
f(x) \leq \min_{y \in \{x_1,\ldots,x_m\}} f(y) + \eps \,.
\end{align*}
Furthermore, it queries the loss function $f$ at most $m\ceil{16 \log(2m/\delta)/\eps^2}$ times.
\end{theorem}

\begin{proof}
By \cref{thm:hoeffding}\index{Hoeffding's inequality} and a union bound, with probability at least $1-\delta$,
\begin{align*}
|\widehat f(x_k) - f(x_k)| \leq \frac{\eps}{2} \text{ for all } 1 \leq k \leq m\,.
\end{align*}
Therefore, for the $x$ returned by the algorithm and any $y \in \{x_1,\ldots,x_m\}$,
\begin{align*}
f(x) &\leq \widehat f(x) + \frac{\eps}{2} \leq \widehat f(y) + \frac{\eps}{2} \leq f(y) + \eps\,.
\qedhere
\end{align*}
\end{proof}

\index{best arm identification|)}

\section{Finding a Cutting Plane}\label{sec:cut:cut}

Let $f \in \sF_{\pb}$ with $f_\star = \inf_{x \in K} f(x)$.
All cutting plane methods abstract away the construction of the half-space. 
The methods assume a process that given a `centre' $x$ of a convex set $K$ return a half-space $H$ such that at least one of the following holds:
\begin{itemize}
\item $x$ is close to $\partial H$ and $H$ contains all near-minimisers of the loss function.
\item $x$ is a near-minimiser of the loss function.
\end{itemize}
When the learner has gradient access to the loss function $f$, then the half-space $H = H(x, f'(x)) \triangleq \{y \in \R^d \colon \sip{f'(x), y - x} \leq 0\}$ satisfies the conditions.
To see why, suppose that $f(y) \leq f(x)$. Then by convexity
$f(y) \geq f(x) + \sip{f'(x), y - x} \geq f(y) + \sip{f'(x), y - x}$, which implies that $y \in H$.
Hence, either $f(x) \leq f_\star + \eps$ and $x$ is a near-minimiser or $f(x) \geq f_\star + \eps$ and any $y$ with $f(y) \leq f_\star + \eps$ is in $H$.

In the bandit setting we do not have access to the gradient, so another procedure is needed to find $H$. Furthermore, the condition that $H$ contains all near-minimisers of the
loss function will be relaxed to require that any specific near-minimiser is in $H$ with high probability.
We have been a bit vague about how close $x$ needs to be to the half-space. This depends on which cutting plane method is used but fortunately can be abstracted away
in the analysis of each method by using a change of coordinates. Concretely, it suffices to make the following assumption:

\begin{assumption}\label{ass:cut-local}
$K$ is a convex body such that $\ball_1 \subset K \subset \ball_r$ for some $r > 1$. 
\end{assumption}

The `centre' is now taken to be $\zeros$. 
The mission in this section is to design a randomised algorithm that queries the loss function as few times as possible and returns a half-space $H$ such that
\begin{flalign}
&
\begin{minipage}{9cm}
\begin{itemize}
\item the boundary of the half-space intersects $\ball_1$: $\partial H \cap \ball_1 \neq \emptyset$, and
\item $\bbP(y \in H) \geq 1 - \delta$ for any $y \in K$ with $f(y) \leq f(\zeros) - \eps$.
\end{itemize}
\end{minipage}
&&
\label{eq:cut:conds}
\end{flalign}

\begin{remark}
The second item above implies that either $f(\zeros) \leq f_\star + 2\eps$ or for any $y$ with $f(y) \leq f_\star + \eps$, $\bbP(y \in H) \geq 1 - \delta$.
\end{remark}

\subsubsection*{Approach}
A natural idea is to try and use the gradient of a smoothed surrogate. \index{surrogate loss!ellipsoidal}
Let $U$ have law $\cU(\ball_1)$ and
\begin{align}
s(x) = \E[2 f(U/4+x/2) - f(U/2)]\,,
\label{eq:cut:s}
\end{align}
which is the same surrogate that appeared in \cref{chap:ftrl}.
We saw there that $s$ is differentiable,\index{differentiable} convex and optimistic:\index{optimistic} $s(x) \leq f(x)$ for all $x$.
Generally speaking the half-space $H = \{y \colon \sip{s'(x), y - x} \leq 0 \}$ need not contain the minimisers of $f$.
But when $s(x)$ is sufficiently large, then $H$ does contain near-minimisers of $f$.

\begin{lemma}\label{lem:cut:opt}
Let $H = H(x, s'(x))$. Then $y \in H$ for all $y$ with $f(y) \leq s(x)$.
\end{lemma}

\begin{proof}
Suppose that $f(y) \leq s(x)$. Then, by convexity of $s$,
\begin{align*}
f(y) \geq s(y) \geq s(x) + \ip{s'(x), y - x} \geq f(y) + \ip{s'(x), y - x}\,,
\end{align*}
where in the first inequality we used the fact that $s \leq f$ is optimistic (\cref{lem:ellipsoid-smooth}). The second follows from convexity of $s$ and the third by the assumption that $s(x) \geq f(y)$.
Rearranging shows that $\sip{s'(x), y - x} \leq 0$ and therefore $y \in H$.
\end{proof}

\cref{lem:cut:opt} suggests a simple plan:
\begin{itemize}
\item Find an $x \in \ball_1$ such that $s(x)$ is nearly as large as $f(\zeros)$.
\item Find an estimate $\hat s'(x)$ of the gradient of $s$ at $x$.
\item Propose the half-space $H(x, \hat s'(x))$.
\end{itemize}
But why should there exist a point $x \in \ball_1$ where $s(x)$ is large?

\begin{lemma}\label{lem:cut:large}
Suppose that $U,V$ are independent random vectors with law $\cU(\ball_1)$ and $h(x) = s(x) - f(\zeros)$.
Then $\E[h(V)] \geq 0$.
\end{lemma}

\begin{proof}
By convexity of $f$,
\begin{align*}
\E[s(V)]
&= \E[2 f(V/2 + U/4) - f(U/2)] \\
\tag*{Jensen's inequality}\index{Jensen's inequality}
&\geq \E[2 f(V/2) - f(U/2)] \\
\tag*{$V \stackrel{d}= U$}
&= \E[f(U/2)] \\
\tag*{Jensen's inequality}
&\geq f(\zeros)\,.
\end{align*}
Rearranging completes the proof.
\end{proof}

Since the maximum is larger than the expectation, \cref{lem:cut:large} shows
there exists an $x \in \ball_1$ such that $s(x) \geq f(\zeros)$.
Moreover, a point nearly satisfying this can be found using Algorithm~\ref{alg:cut:inf}, as we now explain. 
To implement this plan we need an oracle that can provide estimates of $h(x) = s(x) - f(\zeros)$ for any $x \in \ball_1$. And at the same time we provide an estimator for
the gradients of $s$, namely
\begin{align}
s'(x) = \frac{4d}{\vol(\sphere_1)} \int_{\sphere_1} f(x/2 + u/4) u \d{u} \,.
\label{eq:cut:grad}
\end{align}

\begin{algorithm}[h!]
\begin{algcontents}
\begin{lstlisting}
def $\SAMPLE$($x$)
  sample $U_1, U_2$ from $\cU(\ball_1)$
  let $X_1 = U_1/4+x/2$ and observe $Y_1 \sim f(X_1)$
  let $X_2 = U_2/2$ and observe $Y_2 \sim f(X_2)$
  let $X_3 = \zeros$ and observe $Y_3 \sim f(X_3)$
  return $2 Y_1 - Y_2 - Y_3$
\end{lstlisting}
\caption{Returns an unbiased estimate of $h(x) = s(x) - f(\zeros)$}
\label{alg:cut:oracle}
\end{algcontents}
\end{algorithm}

The following simple lemma shows that \cref{alg:cut:oracle} returns an unbiased estimate of $s(x) - f(\zeros)$ that is also subgaussian:\index{subgaussian}

\begin{lemma}\label{lem:cut:h}
Suppose that \cref{alg:cut:oracle} is run with input $x \in \ball_1$ and has output $Y$.
Then $\E[Y] = s(x) - f(\zeros)$ and $\norm{Y - \E[Y]}_{\psi_2} \leq 4$.
\end{lemma}

\begin{proof}
That $\E[Y] = s(x) - f(\zeros)$ is immediate from the definition of the surrogate \cref{eq:cut:s} and the construction of the algorithm.
The bound on the Orlicz norm follows from the triangle inequality for Orlicz norms (\cref{fact:orlicz}) 
and because the noise is subgaussian (\cref{ass:noise}). 
\end{proof}

\begin{algorithm}[h!]
\begin{algcontents}
\begin{lstlisting}
def $\ESTIMATEG$($\eps, \delta, x$):
  $n = \frac{484 d \log(2/\delta)}{\eps^2}$
  for $t = 1$ to $n$:
    sample $U_t$ from $\cU(\sphere_1)$ and $X_t = \frac{x}{2} + \frac{U_t}{4}$
    let $Y_t \sim f(X_t)$
    $\hat s'(x) \triangleq \frac{4d}{n} \sum_{t=1}^n U_t Y_t$
  return $\hat s'(x)$
\end{lstlisting}
\caption{Returns an unbiased estimate of $s'(x)$}
\label{alg:cut:estimate-g}
\end{algcontents}
\end{algorithm}

\FloatBarrier

\cref{alg:cut:estimate-g} returns an estimate of the gradient of $s'(x)$.
The following lemma provides a high-probability bound on the quality of this estimate.

\begin{lemma}\label{lem:cut:gradient-simple}
Suppose that $\hat s'(x)$ is the output of \cref{alg:cut:estimate-g} given input $x \in \ball_1$.
Then, for any $\eta \in \sphere_1$, with probability at least $1 - \delta$,
\begin{align*}
\left|\ip{\hat s'(x) - s'(x), \eta}\right| \leq \eps \,.
\end{align*}
\end{lemma}

\begin{proof}
Since $Y_t = f(X_t) + \eps_t$ where $\norm{\eps_t}_{\psi_2} \leq 1$, it follows from the assumption that $f \in \sF_{\pb}$, the triangle inequality 
for the norm $\norm{\cdot}_{\psi_2}$ (\cref{fact:orlicz}) and
\cref{lem:orlicz-bound} that
\begin{align*}
\norm{Y_t}_{\psi_2} \leq \norm{\eps_t}_{\psi_2} + \norm{f(X_t)}_{\psi_2} \leq 1 + \frac{1}{\sqrt{\log(2)}} \,.
\end{align*}
Combining this with \cref{prop:orlicz:product,prop:orlicz:sphere} shows that
\begin{align*}
\snorm{\sip{U_t, \eta} Y_t}_{\psi_1}
&\leq \snorm{\sip{U_t, \eta}}_{\psi_2} \snorm{Y_t}_{\psi_2} \\
&\leq \snorm{\eta}\sqrt{\frac{4}{3(d+1)}} \left(1 + \frac{1}{\sqrt{\log(2)}}\right) 
\leq \frac{2.6}{\sqrt{d+1}}\,.
\end{align*}
Let $Z_t = \ip{4 d U_t Y_t - s'(x), \eta}$. The random variables $Z_1,\ldots,Z_n$ are independent and $\E[Z_t] = 0$ 
for all $1 \leq t \leq n$ because by \cref{eq:cut:grad}, $4dU_t Y_t$ is 
an unbiased estimator of $s'(x)$. By the above display and the fact that $\E[4d U_t Y_t] = s_t'(x)$, 
\begin{align*}
\snorm{Z_t}_{\psi_1} 
&= \norm{\sip{4d U_t Y_t, \eta} - \E[\sip{4d U_t Y_t, \eta}]}_{\psi_1} \\
\tag*{\cref{lem:orlicz-center}}
&\leq \left(1 + \frac{1}{\log(2)}\right) \snorm{4d\sip{U_t, \eta} Y_t}_{\psi_1}  \\
\tag*{\cref{fact:orlicz}}
&= 4d\left(1 + \frac{1}{\log(2)}\right) \snorm{\sip{U_t, \eta} Y_t}_{\psi_1} \\
&\leq 11 \sqrt{d}\,.
\end{align*}
Therefore, by Bernstein's inequality\index{Bernstein's inequality} (\cref{thm:bernstein}) with probability at least $1 - \delta$,
\begin{align*}
\left|\ip{\hat s'(x) - s'(x), \eta}\right| 
&= \left|\frac{1}{n} \sum_{t=1}^n Z_t\right| \\
&\leq 11\sqrt{d} \max\left(\sqrt{\frac{4 \log(2/\delta)}{n}}, \frac{2 \log(2/\delta)}{n}\right) \\
&\leq \eps \,,
\end{align*}
where the last inequality follows from the definition of $n$.
\end{proof}

At last we are in a position to use \cref{alg:cut:inf} to find a suitable point to cut and \cref{alg:cut:estimate-g} to 
estimate the gradient of the surrogate loss.

\begin{algorithm}[h!]
\begin{algcontents}
\begin{lstlisting}
def $\CUT$($f, \eps, \delta, r$):
  let $\rho$ be the uniform measure on $\ball_1\index{uniform measure}$ 
  $x = \INF(\frac{\eps}{2}, \frac{\delta}{2}, \rho, \SAMPLE)$      &\Comment{\cref{alg:cut:inf}}&
  $g = \ESTIMATEG(\frac{\eps}{4r}, \frac{\delta}{2}, x)$           &\Comment{\cref{alg:cut:estimate-g}}&
  return $H(x,g)$
\end{lstlisting}
\caption{Returns a half-space satisfying \cref{eq:cut:conds} under \cref{ass:cut-local}}
\label{alg:cut:cut}
\end{algcontents}
\end{algorithm}
\FloatBarrier

\begin{theorem}\label{thm:cut:cut}
Suppose that $f(y) \leq f(\zeros) - \eps$. 
The following hold:
\begin{enumerate}
\item With probability at least $1 - \delta$ \cref{alg:cut:cut} returns a half-space $H$ with $y \in H$ and $\partial H \cap \ball_1 \neq \emptyset$.
\item In expectation \cref{alg:cut:cut} makes at most $O\left(\frac{d r^2 \log(1/\delta)}{\eps^2}\right)$ queries to the loss function.
\end{enumerate}
\end{theorem}

\begin{proof}
Recall that $h(x) = s(x) - f(\zeros)$. By \cref{lem:cut:large},
\begin{align*}
\int_{\ball_1} h(x) \d{\rho}(x) \geq 0 \,.
\end{align*}
Moreover, \cref{lem:cut:h} shows that \cref{alg:cut:oracle} supplies the kind of stochastic oracle for $h$ needed by
\cref{alg:cut:inf}.
Hence, by \cref{thm:cut:inf} ,with probability at least $1 - \delta/2$, the call to \cref{alg:cut:inf} returns an $x \in \ball_1$ such that
$h(x) \geq -\frac{\eps}{2}$, which implies that $s(x) - f(\zeros) \geq -\frac{\eps}{2}$.
By assumption $f(\zeros) \geq f(y) + \eps$.
Combining this with a union bound and \cref{lem:cut:gradient-simple} shows that with probability at least $1 - \delta$,
\begin{align*}
f(y) 
&\geq s(y) \\
&\geq s(x) + \ip{s'(x), y - x} \\
&\geq f(\zeros) - \frac{\eps}{2} + \ip{s'(x), y - x} \\
&\geq f(y) + \frac{\eps}{2} + \ip{s'(x), y - x} \\
&\geq f(y) + \ip{\hat s'(x), y - x} \,.
\end{align*}
Therefore $\ip{\hat s(x), y - x} \leq 0$, which shows that $y \in H$ as required.
Since $x \in \ball_1$ and $x \in \partial H(x, \hat s'(x))$ it follows trivially that $\partial H \cap \ball_1 \neq \emptyset$.
The bound on the number of queries follows from \cref{thm:cut:inf} and the construction of \cref{alg:cut:estimate-g}.
\end{proof}

\section{Centre of Gravity Method}\label{sec:cut:cog} \index{centre of gravity method|(}

Let $K \subset \R^d$ be a convex body and $H$ be a half-space. 
We are interested in the conditions on $H$ and $K$ such that $\vol(K \cap H) \leq \gamma \vol(K)$ for some constant $\gamma$.
Gr\"unbaum's inequality provides such a result for half-spaces $H$ passing through the centre of mass: \index{Gr\"unbaum's inequality}

\begin{theorem}[\citealt{Gru60}]
Let $K$ be a convex body, $y = \frac{1}{\vol(K)} \int_K x \d{x}$ and $H = H(y, \eta)$ for any direction $\eta \neq \zeros$. Then
\begin{align*}
\frac{\vol(K\cap H)}{\vol(K)} \leq \left(\frac{d}{d+1}\right)^d \leq 1 - \frac{1}{e} \,.
\end{align*}
\end{theorem}

The half-space need not pass through exactly the centre:

\begin{theorem}[\citealt{bertsimas2004solving}]\label{thm:Gr-approx}
Let $K$ be a convex body in isotropic position and $H = \{x \colon \sip{x - y, \eta}\}$ for unit vector $\eta$.
Then
\begin{align*}
\frac{\vol(K \cap H)}{\vol(K)} \leq 1 + \norm{y} - \frac{1}{e} \,.
\end{align*}
\end{theorem}

By \cref{thm:isotropic}, when $K$ is isotropic, then $\ball_1 \subset K \subset \ball_{1+d}$ and this is more or less tight 
for the simplex\index{simplex} in isotropic position. \cref{thm:Gr-approx} shows that cutting $K$ anywhere inside $\smash{\ball_{1/(2e)}} \subset K$ will
divide the set into two nearly equal pieces. 
Given a convex body $K$, recall from the discussion prior to \cref{thm:isotropic} that $\ISO_K \colon \R^d \to \R^d$ is the affine map\index{affine!map} such that $\ISO_K(K)$ is isotropic.
Note that $\ISO_K$ is generally non-trivial to compute exactly, but can be approximated with reasonable precision using sampling
as explained in \cref{sec:reg:rounding}. 
\cref{thm:Gr-approx} can be combined with \cref{alg:cut:cut} to obtain a simple algorithm for bandit convex optimisation with near-optimal sample
complexity guarantees. Keeping things simple, we assume exact computation and ask you handle the approximation errors in \cref{ex:cut:approx}.

\begin{algorithm}[h!]
\begin{algcontents}
\begin{lstlisting}
args: $\eps \in (0,1)$, $\delta \in (0,1)$
$k_{\max} = 1 + \ceil{\log\left(\frac{(\eps/4)^d}{d! \vol(\ball_d)}\right) \bigg/ \log(1-1/(2e))}$ and $K_1 = K$
for $k = 1$ to $k_{\max}$:
  $T_k = (2e) \ISO_{K_k}$ and $f_k = f \circ T_k^{-1}$ and $x_k = T_k^{-1}(\zeros)$
  $H_k = \CUT(f_k, \frac{\eps}{4}, \frac{\delta}{2(d+1) k_{\max}}, 4ed)$ &\Comment{\cref{alg:cut:cut}}&
  update $K_{k+1} = K_k \cap T_k^{-1}(H_k)$
return $\BAI(\frac{\eps}{2}, \frac{\delta}{2}, x_1,\ldots,x_{k_{\max}}) \index{best arm identification}$ &\Comment{\cref{alg:cut:bai}}&
\end{lstlisting}
\caption{Centre of gravity for convex bandits}
\label{alg:cut:cog}
\end{algcontents}
\end{algorithm}

\FloatBarrier

\begin{remark}
You might wonder if the final call to $\BAI$ is necessary. Maybe the algorithm should simply return $x_{k_{\max}}$. 
This does not work. 
There is no particular reason to expect that $k \mapsto f(x_k)$ should be decreasing. Indeed, $x_1$ could by chance be the minimiser of $f$.
\end{remark}

\begin{theorem}\label{thm:cut:cog}
Under \cref{ass:cut} the following hold:
\begin{enumerate}
\item With probability at least $1 - \delta$ Algorithm~\ref{alg:cut:cog} outputs an $x$ such that $f(x) \leq \inf_{y \in K} f(y) + \eps$. 
\item The expected number of queries to the loss function made by Algorithm~\ref{alg:cut:cog} is at most
\begin{align*}
O\left(\frac{d^4}{\eps^2} \log\left(\frac{d}{\eps}\right) \log\left(\frac{d \log(1/\eps)}{\delta}\right)\right) \,.
\end{align*}
\end{enumerate}
\end{theorem}

\begin{proof}
Let $f_\star = \inf_{x \in K} f(x)$.
We establish the claim in three steps, beginning with a proof that with high probability the algorithm indeed returns a point that is near-optimal.
The second step proves the key lemma used in the first. The last step bounds the sample complexity.\index{sample complexity}

\stepsection{Step 1: Correctness}
Suppose that
\begin{align}
\bbP\left(\min \{f(x_k) \colon 1 \leq k \leq k_{\max}\} \leq f_\star + \frac{\eps}{2}\right) \geq 1 - \frac{\delta}{2} \,.
\label{eq:cut:correct}
\end{align}
By construction the algorithm returns $\BAI(\frac{\eps}{2}, \frac{\delta}{2}, x_1,\ldots,x_{k_{\max}})$ and by \cref{thm:cut:bai},
this subroutine returns an $x \in \{x_1,\ldots,x_k\}$ such that
\begin{align*}
\bbP\left(f(x) \leq \min_{1 \leq k \leq k_{\max}} f(x_k) + \frac{\eps}{2}\right) \geq 1 - \frac{\delta}{2} \,.
\end{align*}
A union bound combining the above display and \cref{eq:cut:correct} shows that
with probability at least $1 - \delta$ the algorithm returns an $x$ such that $f(x) \leq f_\star + \eps$ as required.
The remainder of this step is devoted to establishing \cref{eq:cut:correct}.
Let $S \subset K$ be a simplex such that $f(x) \leq f_\star + \eps/2$ for all $x \in S$ and
\begin{align}
\vol(S) \geq \left(\frac{\eps}{4}\right)^d \frac{\vol(K)}{d! \vol(\ball_{d})} \,,
\label{eq:cut:S}
\end{align}
which exists by \cref{lem:cut:vol}. We also let $S_k = T_k(S)$, which is a simplex in $J_k = T_k(K_k)$.
We prove the following lemma in the next step:

\begin{lemma}\label{lem:cut:induction}
Suppose that $S \subset K_k$ and $f(x_k) > f_\star + \eps/2$. Then
\begin{align*}
\bbP_{k-1}\left(S \subset K_{k+1} \right) \geq 1 - \frac{\delta}{2k_{\max}}
\end{align*}
where $\bbP_{k-1}$ is the probability measured conditioned on all information available at the end of iteration $k-1$.
\end{lemma}

By induction, a union bound over $1 \leq k \leq k_{\max}$ and \cref{lem:cut:induction}, with probability at least $1 - \frac{\delta}{2}$ at least one of the following holds:
\begin{enumerate}
\item There exists a $1 \leq k \leq k_{\max}$ such that $f(x_k) \leq f_\star + \eps/2$; or \label{eq:cut:cog:opt}
\item $S \subset K_{k_{\max} + 1}$. \label{eq:cut:cog:vol}
\end{enumerate}
In a moment we show that
\begin{align*}
\vol(K_{k_{\max}+1}) \leq \left(1 - \frac{1}{2e}\right)^{k_{\max}} \vol(K) < \vol(S)\,,
\end{align*}
which contradicts \ref{eq:cut:cog:vol} and therefore \ref{eq:cut:cog:opt} occurs with probability at least $1 - \delta$. 
The second inequality in the above display follows from the definition of $k_{\max}$ and \cref{eq:cut:S}.
For the first, by definition $T_k = (2e) \ISO_{K_k}$, which means that $\frac{1}{2e} J_k$ is in isotropic position. Furthermore, by the definition of the algorithm 
$\partial H_k \cap \ball_1 \neq \emptyset$ and hence $\frac{1}{2e} \partial H_k \cap \ball_{1/2e} \neq \emptyset$. Then
\begin{align*}
\frac{\vol(K_{k+1})}{\vol(K_k)}
= \frac{\vol(\textstyle{\frac{1}{2e}} J_{k+1})}{\vol(\textstyle{\frac{1}{2e}}J_k)} 
= \frac{\vol(\textstyle{\frac{1}{2e}} J_k \cap \textstyle{\frac{1}{2e}} H_k)}{\vol(\textstyle{\frac{1}{2e}}J_k)} 
\leq 1 - \frac{1}{2e} \,,
\end{align*}
where the last inequality follows from \cref{thm:Gr-approx}.

\stepsection{Step 2: Proof of \cref{lem:cut:induction}}
Suppose that $y$ is a vertex of $S_k = T_k(S)$. Then $f_k(y) \leq f_\star + \frac{\eps}{4}$ and $f_k(\zeros) > f_\star + \frac{\eps}{2}$, which means that
\begin{align*}
f_k(y) < f_k(\zeros) - \frac{\eps}{4}\,.
\end{align*}
Hence, by \cref{thm:cut:cut},
\begin{align*}
\bbP_{k-1}\left(y \in H_k\right) \geq 1 - \frac{\delta}{2(d+1) k_{\max}}\,.
\end{align*}
Since $H_k$ is convex, if all vertices of $S_k$ are in $H_k$ it follows that $S_k \subset H_k$.
Hence, a union bound over the $d+1$ vertices of $S_k$ combined with the above display shows that
\begin{align*}
\bbP_{k-1}\left(S_k \subset H_k\right) \geq 1 - \frac{\delta}{2 k_{\max}} \,.
\end{align*}
Therefore
\begin{align*}
\bbP_{k-1}\left(S \subset K_{k+1}\right) 
\tag*{Definition of $K_{k+1}$}
&= \bbP_{k-1}\left(S \subset K_k \cap T_k^{-1}(H_k)\right) \\
\tag*{Since $S \subset K_k$}
&= \bbP_{k-1}\left(S \subset T_k^{-1}(H_k) \right) \\
\tag*{Since $T_k$ is invertible}
&= \bbP_{k-1}\left(S_k \subset H_k\right) \\
&\geq 1- \frac{\delta}{2 k_{\max}}
\end{align*}
as required.

\stepsection{Step 3: Sample complexity}
There are $k_{\max}$ iterations and 
\begin{align*}
k_{\max} = \Theta\left(d \log(d/\eps)\right) \,.
\end{align*}
The algorithm makes $k_{\max}$ calls to $\CUT$ with radius bound $r = 4ed$, precision $\frac{\eps}{2}$ and confidence $\bar \delta$ with
\begin{align*}
\bar \delta = \frac{\delta}{2 (d+1) k_{\max}} \,.
\end{align*}
By \cref{thm:cut:cut}, the expected number of queries of the loss used by each call to $\CUT$ is
\begin{align*}
O\left(\frac{dr^2 \log(1 / \bar \delta)}{\eps^2}\right) = O\left(\frac{d^3}{\eps^2} \log\left(\frac{d k_{\max}}{\delta}\right)\right)\,.
\end{align*}
Therefore the total number of queries to the loss function used by all calls to $\CUT$ is bounded in expectation by
\begin{align*}
O\left(\frac{d^3 k_{\max}}{\eps^2} \log\left(\frac{d k_{\max}}{\delta}\right)\right) 
= O\left(\frac{d^4}{\eps^2} \log\left(\frac{d \log(d/\eps)}{\delta}\right) \log\left(\frac{d}{\eps}\right)\right) \,.
\end{align*}
Finally, the call to $\BAI$ uses just
\begin{align*}
O\left(\frac{k_{\max}}{\eps^2} \log\left(\frac{1}{\delta}\right)\right)
\end{align*}
queries.
The result follows by combining the previous two displays.
\end{proof}

The main problem with the centre of gravity method is the need to find the affine maps\index{affine!map} $T_k$.
As we discussed in \cref{sec:reg:rounding}, when $K$ has a reasonable representation this can be done approximately by sampling.
But this is a heavy procedure that one would like to avoid.

\section{Method of the Inscribed Ellipsoid}\label{sec:cut:inscribed-alg}\index{method of inscribed ellipsoid}

When $K$ is a polytope,\index{polytope} the method of the inscribed ellipsoid provides a more computationally efficient cutting plane mechanism than the centre of gravity method.
Given a convex body $K \subset \R^d$, let $\MVIE(K)$ be the ellipsoid of largest volume contained in $K$, which is unique.
Furthermore, we have the following analogue of Gr\"unbaum's inequality:

\begin{proposition}[\citealt{khachiyan1990inequality}]\label{prop:cut:inscribed}
Let $K$ be a convex body, $E = \MVIE(K)$ and $H$ be a half-space with $\partial H$ intersecting the centre of $E$.
Then $\vol(\MVIE(K \cap H)) \leq 0.85 \vol(E)$.
\end{proposition}

The standard method of inscribed ellipsoids initialises $K_1 = K$ and subsequently computes $K_{k+1}$ from $K_k$ as follows:
\begin{itemize}
\item Find $E_k = \MVIE(K_k)$.
\item Let $x_k$ be the centre of $E_k$ and $H_k = H(x_k, g_k)$ where $g_k \in \partial f(x_k)$. 
\item Update $K_{k+1} = K_k \cap H_k$.
\end{itemize}
Of course, this is only feasible with access to subgradients\index{subgradient} of the loss $f$. We will use \cref{alg:cut:cut} to find $H_k$ instead,
but for this we can only guarantee that $\partial H_k$ passes through a point close to $x_k$ and therefore need a refinement of \cref{prop:cut:inscribed}.

\begin{proposition}\label{prop:cut:inscribed-gen}
Let $K$ be a convex body, $E = \MVIE(K)$ and $H$ be a half-space such that $\partial H \cap E(\frac{1}{2}) \neq \emptyset$, where $E(\frac{1}{2})$ is the same ellipsoid with half the radius. 
Then $\vol(\MVIE(K \cap H)) \leq 0.97 \vol(E)$.
\end{proposition}

As far as we know this result is new, though our proof -- deferred to \cref{sec:cut:inscribed} -- follows that by \cite{khachiyan1990inequality} 
in almost every detail. 

\begin{algorithm}[h!]
\begin{algcontents}
\begin{lstlisting}
args: $\eps \in (0,1)$, $\delta \in (0,1)$
$k_{\max} = 1 + \ceil{\log\left(\frac{(\eps/4)^d}{d! \vol(\ball_d)}\right) \bigg/ \log(0.97)}$ and $K_1 = K$
for $k = 1$ to $k_{\max}$:
  $E_k = \MVIE(K_k) = E(x_k, A_k)$
  $T_k(x) = 2 A_k^{-1/2}(x - x_k)$ and $f_k = f \circ T_k^{-1}$ 
  $H_k = \CUT(f_k, \frac{\eps}{4}, \frac{\delta}{2(d+1) k_{\max}}, 2d)$ &\Comment{\cref{alg:cut:cut}}&
  update $K_{k+1} = K_k \cap T_k^{-1}(H_k)$
return $\BAI(\frac{\eps}{2}, \frac{\delta}{2}, x_1,\ldots,x_{k_{\max}}) \index{best arm identification}$ &\Comment{\cref{alg:cut:bai}}&
\end{lstlisting}
\caption{Method of the inscribed ellipsoid for convex bandits}
\label{alg:cut:inscribed}
\end{algcontents}
\end{algorithm}

\FloatBarrier

\subsubsection*{Computation}
Finding the maximum-volume enclosed ellipsoid is the most computationally heavy part of \cref{alg:cut:inscribed}, but this only needs to be done
$\tilde O(d)$ cumulatively.
Suppose that $P = \{x \colon Cx \leq b\}$ is a convex body with $C \in \R^{m \times d}$. Since $P$ is a convex body, $m \geq d + 1$.
\cite{KT93} show that $\MVIE(P)$ can be computed to extreme precision in $\tilde O(m^{3.5})$ time.
Their algorithm is based on interior point methods.\index{interior point methods}
More generally, the problem is a semidefinite program and is practically solvable using modern solvers \citep[\S8.4.2]{BV04}.
The complexity per round is dominated by evaluating the affine map\index{affine!map} in the definition of $f_k$, which is $O(d^2)$.

\begin{theorem}\label{thm:cut:inscribed}
Under \cref{ass:cut} the following hold:
\begin{enumerate}
\item With probability at least $1 - \delta$ \cref{alg:cut:inscribed} outputs an $x$ such that $f(x) \leq \inf_{y \in K} f(y) + \eps$.
\item The expected number of queries to the loss function made by \cref{alg:cut:inscribed} is at most
\begin{align*}
O\left(\frac{d^4}{\eps^2} \log\left(\frac{d}{\eps}\right) \log\left(\frac{d \log(1/\eps)}{\delta}\right)\right) \,.
\end{align*}
\end{enumerate}
\end{theorem}

\begin{proof}
The argument is the same as the proof of \cref{thm:cut:cog}.
The only difference is volume calculation.
Recall that $S \subset K$ is a simplex such that
\begin{align*}
\vol(S) \geq \left(\frac{\eps}{4}\right)^d \frac{\vol(K)}{d! \vol(\ball_d)} \,,
\end{align*}
which means that $\log(\vol(S)) = \Omega\left(d \log(d/\eps)\right)$.
Then
\begin{align*}
\log(\vol(K_{k+1})) 
&\leq \log(\vol(E_{k+1}(d)))  \\
&=d \log(d) + \log(\vol(E_{k+1}) ) \\
&\leq d \log(d) + \log(\vol(E_1)) + k \log(0.97) \\
&\leq d \log(d) + \log(\vol(K)) + k \log(0.97)\,, 
\end{align*}
where in the first inequality we used the fact that $K_{k+1} \subset E_{k+1}(d)$, which follows from John's theorem (\cref{thm:john}).\index{John's theorem}
The second inequality follows from \cref{prop:cut:inscribed-gen} and the last since $E_1 \subset K$ by definition. 
Hence $\vol(K_{k+1}) \leq \vol(S)$ once $k \geq k_{\max} = \Theta(d \log(d/\eps))$ iterations, just as for the centre of gravity method.
\end{proof}

\index{centre of gravity method|)}

\section{Ellipsoid Method}\label{sec:cut:ellipsoid}\index{ellipsoid method|(}

The ellipsoid method is an alternative cutting plane method that is statistically less efficient than both the method of the inscribed
ellipsoid and the centre of gravity method. The advantage is that it can be implemented using a separation oracle \index{separation oracle}
only and relatively efficiently.
Given a convex body $K$ let $\MVEE(K)$ be the ellipsoid of minimum volume containing $K$.
The classical ellipsoid method starts with an ellipsoid $E_1$ containing $K$.
\begin{itemize}
\item Let $x_k$ be the centre of $E_k$.
\item If $x_k \in K$, then let $H_k = \{x \colon \ip{x - x_k, g_k} \leq 0\}$ with $g_k \in \partial f(x_k)$. Otherwise let $H_k = H(x_k, \SEP_K(x_k))$.
\item Let $E_{k+1} = \MVEE(E_k \cap H_k)$.
\end{itemize}
We give the formula for $E_{k+1}$ as well as references for the following claims in Note~\ref{note:ellipsoid}.
The beauty of the ellipsoid method is that there is a closed-form expression for $\MVEE(E \cap H)$ when $E$ is an ellipsoid and $H$ is a half-space.
Furthermore, $\vol(E_{k+1}) = O(1-1/d) \vol(E_k)$.
By construction we have $K_k \subset E_k$. 
There is no need for $x_k$ to be the centre of $E_k$. In fact, when $x_k \in E_k(\frac{1}{2d})$ it holds 
that $\vol(E_{k+1}) \leq (1 - \frac{1}{20d}) \vol(E_k)$.
Based on this, one might try to implement the ellipsoid method by letting $E_k = E(x_k, A_k)$, $T_k = 2d A_k^{-1/2}(x - x_k)$ and $f_k = f \circ T_k^{-1}$, which are
chosen so that
\begin{align*}
T_k(E_k) = \ball_{2d} \quad\text{and}\quad
T_k(E_k(1/(2d))) = \ball_1 \,.
\end{align*}
We want to run \cref{alg:cut:cut} on $f_k$. The problem is that there is no reason why $E_k(1/(2d)) \subset K$ should hold. Equivalently, it can happen that
$\ball_1 \not\subset \dom(f_k)$.
There are two ways to remedy this.
\begin{enumerate}
\item There exists a modification of the ellipsoid method that guarantees $E_k(1/r) \subset K_k \subset E_k$ with $r = O(d^{3/2})$.
Then you can let $T_k = r A_k^{-1/2}(x - x_k)$ which means that
\begin{align*}
T_k(E_k) = \ball_r \quad \text{and} \quad 
T_k(E_k(1/r)) = \ball_1  \,.
\end{align*}
This ensures that $\ball_1 \subset \dom(f_k)$ but the price is that $r = O(d^{3/2})$ and this leads to a final sample complexity of 
\begin{align*}
O\left(\frac{d^6 \polylog(d, 1/\delta, 1/\eps)}{\eps^2} \right)\,.
\end{align*}
\item Use the extension defined in \cref{prop:reg:bandit-extension-eps} in place of $f$ so that $\dom(f_k) = \R^d$. \index{extension}
The challenge is that $f_k$ may not be bounded anymore, so the concentration analysis in the proof of \cref{lem:cut:gradient-simple} is no longer valid. 
To handle this, note that if the algorithm ever queries a point $f_k$ at some $y$ for which $T_k^{-1}(y) \notin K$, then you can use the separation oracle\index{separation oracle} to define
the cutting plane. On the other hand, if the probability of querying $f_k$ at such a point is very low, then the additional variance introduced by using the surrogate
loss is minimal and some version of \cref{lem:cut:gradient-simple} continues to hold. After bashing out all the details you will eventually arrive at a sample
complexity bound of
\begin{align*}
O\left(\frac{d^5 \polylog(d, 1/\delta, 1/\eps)}{\eps^2}\right)\,.
\end{align*}
\end{enumerate}

\index{ellipsoid method|)}

\section[Proof of Proposition~\ref{prop:cut:inscribed-gen}]{Proof of Proposition~\ref{prop:cut:inscribed-gen} ($\skippy$)}\label{sec:cut:inscribed} 

We start with a lemma:

\begin{lemma}[Lemma 1, \citealt{khachiyan1990inequality}]\label{lem:cut:inscribed-lem}
Let $E_\star = E(x_\star, A_\star) = \MVIE(K)$ and $E = E(x, A) \subset K$.
Then
\begin{align*}
\frac{\vol(E)}{\vol(E_\star)} 
\leq \min_{\chi \in I} \chi \exp(1 - \chi) 
\,,\qquad I = \left[\min_{\eta \in \sphere_1} \frac{\norm{A^{-1/2} \eta}}{\norm{A_\star^{-1/2} \eta}}, \max_{\eta \in \sphere_1} \frac{\norm{A^{-1/2} \eta}}{\norm{A_\star^{-1/2} \eta}}\right]\,.
\end{align*}
\end{lemma}

\begin{proof}[Proof of \cref{prop:cut:inscribed-gen}]
Assume by means of a coordinate change that $E_\star = \MVIE(K) = E(x, D^2)$ and $G_\star = \MVIE(K \cap H) = E(-x, D^{-2})$ for some diagonal matrix $D$ 
with eigenvalues $\lambda_1 \geq \cdots \geq\lambda_d$ and $x \in \R^d$. This has been chosen so that $E_\star = x + D(\ball_1)$, 
which means that $\vol(E_\star) = \det(D) \vol(\ball_1)$.
We claim that $E = E(\zeros, \id+xx^\top) \subset K$.
Recall that the support function of a convex set $A \subset \R^d$ is $h_A(u) = \sup_{x \in A} \ip{u, x}$.\index{support function}
Suppose that $u \in \R^d$. Then
\begin{align*}
h_{E_\star}(u) &= \ip{u, x} + \snorm{D u} \leq h_K(u) \\
h_{G_\star}(u) &= -\ip{u, x} + \snorm{D^{-1} u} \leq h_K(u) \,.
\end{align*}
Multiplying these inequalities shows that 
\begin{align*}
\norm{u}^2 
&\leq \snorm{Du} \snorm{D^{-1} u} \leq (h_K(u) - \ip{u, x}) (h_K(u) + \ip{u, x}) \\
&= h_K(u)^2 - \ip{u, x}^2 \,,
\end{align*}
where the first inequality follows from Cauchy--Schwarz.
Rearranging shows that for any $u \in \R^d$,
\begin{align*}
h_E(u) = \norm{\sqrt{\id + xx^\top} u} = \sqrt{\norm{u}^2 + \ip{u, x}^2} \leq h_K(u)\,.
\end{align*}
Therefore $E \subset K$.
Next, with $A^{-1/2} = \sqrt{\id + xx^\top}$,
\begin{align*}
\min_{\eta \in \sphere_1} \frac{\norm{A^{-1/2} \eta}}{\norm{D \eta}} \leq \frac{\sqrt{1 + \norm{x}^2}}{\lambda_1}
\leq \max_{\eta \in \sphere_1} \frac{\norm{A^{-1/2} \eta}}{\norm{D \eta}}\,.
\end{align*}
Hence, by \cref{lem:cut:inscribed-lem},
\begin{align*}
\frac{\sqrt{1 + \norm{x}^2}}{\lambda_1 \cdots \lambda_d} = \frac{\vol(E)}{\vol(E_\star)} \leq \frac{\sqrt{1+ \norm{x}^2}}{\lambda_1} \exp\left(1 - \frac{\sqrt{1 + \norm{x}^2}}{\lambda_1}\right)
\end{align*}
and therefore
\begin{align*}
\frac{\vol(G_\star)}{\vol(E_\star)} &= \frac{1}{\lambda_1^2 \cdots \lambda_d^2} \leq \frac{1}{\lambda_1^2} \exp\left(2 - \frac{2 \sqrt{1 + \norm{x}^2}}{\lambda_1}\right)\,.
\end{align*}
By assumption there exists a point $y \in \partial H \cap E_\star(\frac{1}{2})$.
By definition $y \notin \interior(G_\star)$. Since $y \in E_\star(\frac{1}{2})$ there exists an $\eta \in \ball_{1/2}$ such that
$y = x + D \eta$. Therefore
\begin{align*}
1 \leq \norm{y - (-x)}_{D^2} = \norm{2 x + D\eta }_{D^2} \leq 2\lambda_1 \norm{x} + \frac{\lambda_1^2}{2} \,, 
\end{align*}
using in the first inequality the fact that $y \notin \interior(G_\star) = \{z \colon \norm{z + x}_{D^2} < 1\}$.
Rearranging shows that
\begin{align*}
\norm{x}^2 \geq \frac{\max(0, 1 - \lambda_1^2/2)^2}{4\lambda_1^2} \,.
\end{align*}
Therefore
\begin{align*}
\frac{\vol(G_\star)}{\vol(E_\star)} 
&\leq \sup_{\lambda_1 > 0} \frac{1}{\lambda_1^2} \exp\left(2 - \frac{2 \sqrt{1 + \frac{\max(0, 1 - \lambda_1^2/2)^2}{4\lambda_1^2}}}{\lambda_1}\right) \leq 0.97 \,.
\qedhere
\end{align*}
\end{proof}

\section{Notes}

\begin{enumeratenotes}
\item \cref{alg:cut:bai} is due to \cite{EMM06}. For our application it makes no difference, but you may be interested to 
know this algorithm is not quite optimal in terms of the logarithmic factors. The optimal algorithm is called median elimination and is 
also due to \cite{EMM06}.
\item Infinite-armed bandits of the kind studied in \cref{sec:cut:inf} are still a little niche and go back to \cite{berry1997bandit}.\index{bandit!infinite-armed}
Generally speaking the objective is to find a near-optimal arm and various assumptions are made on $h$ and $\rho$.
\cite{carpentier2024simple} essentially introduced the much weaker objective of finding an action that is close to the mean of $h$ under $\rho$ and 
provided a slightly more complicated algorithm than analysed here but based on the same principles.

\item The algorithm for finding a suitable cutting plane is inspired by \cite{LG21a} and \cite{carpentier2024simple}. The former paper uses an optimistic\index{optimistic} surrogate
combined with an overly complicated method for finding the cutting point. The latter uses the pessimistic\index{pessimistic} surrogate from \cref{chap:sgd} along with a beautiful
argument about when its gradient is nevertheless useful to define a cutting plane at certain points. 
She uses an infinite-armed bandit to find suitable points, but
this is combined with a more complicated recursive argument. The rates we obtain here are the same up to logarithmic factors 
as those obtained by \cite{carpentier2024simple}.

\item
The algorithms presented in this chapter do not have well-controlled regret.
What is missing is a degree of adaptivity within the mechanism for finding a cutting plane that stops early when there is a large margin.
This idea was used by \cite{LG21a} and can probably be adapted to the more refined algorithms in this chapter.

\item 
Vaidya's method \citep{vaidya1996new} and its refinement by \cite{lee2015faster} provide even faster cutting plane methods for polytopes. \index{Vaidya's method}\index{polytope}
What is needed to use these methods in combination with \cref{alg:cut:cut} is to prove that an inexact centre suffices to drive the algorithm, just
as we did for the method of inscribed ellipsoids with \cref{prop:cut:inscribed-gen}.

\item 
\label{note:ellipsoid} \index{ellipsoid method}
All of the assertions about the ellipsoid method in \cref{sec:cut:ellipsoid} are explained and proven in Chapters 3 and 4
of the wonderful book by \cite{MLA12}. When the half-space is not at the centre, the method is referred to as the shallow-cut ellipsoid method. 

\begin{theorem}\label{thm:ellipsoid}
Suppose the dimension $d \geq 2$.
Given an ellipsoid $E = E(a, A)$ and half-space $H = H(x, g)$ let $E(b, B) = \MVEE(E \cap H)$. 
Let $\lambda = \frac{\ip{g,a-x}}{\norm{g}_A}$. Then
\begin{enumerate}
\item $H$ intersects $E$ if and only if $\lambda \in [-1,1]$.
Moreover, if $\lambda \in [-1,-1/d]$, then $E(b, B) = E(a, A)$.
\item If $\lambda \in [-1/d,1)$, then with $\eta = \frac{A g}{\norm{g}_A}$,
\begin{align*}
b = a - \frac{1+d\lambda}{d+1} \eta \quad \text{ and } \quad 
B = \frac{d^2(1 - \lambda^2)}{d^2-1}\left[A - \frac{2(1 + d\lambda)}{(d+1)(1+\lambda)} \eta \eta^\top\right] \,.
\end{align*}
\item If $\lambda \in (-1/d,1/d)$, then 
$\displaystyle \vol(\MVEE(E \cap H)) \leq \vol(E) \exp\left(-\frac{(1 - d \lambda)^2}{5d}\right)$.
\end{enumerate}
\end{theorem}

The proof of \cref{thm:ellipsoid} is given by \citet[Chapter 3]{MLA12}.
The advantage of the ellipsoid method is that only a separation oracle is needed.\index{separation oracle} 
The downside is that the sample complexity is a factor of $d$ worse than what we obtained using the centre of gravity or inscribed ellipsoid methods.\index{sample complexity}

\begin{exer}
\faStar\faStar\faStar \faQuestionCircle \quad
Suppose $K$ is represented by a separation oracle.
Use the separation oracle to construct a polytope $P$ such that $K \subset P$. Use the extension in \cref{prop:reg:bandit-extension} to extend the loss from a shrunk subset of $K$ to $P$.
Use the ideas in \cref{sec:cut:ellipsoid} to implement the method of the inscribed ellipsoid and obtain a polynomial time algorithm with $\tilde O(d^4/\eps^2)$ sample complexity.
\end{exer}

\item 
In our analysis of \cref{alg:cut:cog} and \cref{alg:cut:inscribed} we assumed exact computation of isotropic position and the maximum-volume inscribed ellipsoid.
The following exercise asks you to prove these algorithms are robust to approximations:

\begin{exer}\label{ex:cut:approx}
\faStar\faStar\faStar\faBook \quad 
Suppose that $K$ is represented by a separation oracle. Show
that \cref{alg:cut:cog} is robust to approximation of isotropic position and give a complexity bound on all of the following:
\begin{itemizeinner}
\item Number of queries to the loss function
\item Number of calls to the separation oracle
\item Number of arithmetic operations 
\end{itemizeinner}
\end{exer}

\begin{exer}\label{ex:cut:approx2}
\faStar\faStar\faStar\faBook \quad 
Suppose that $K$ is represented by an intersection of half-spaces. Show that \cref{alg:cut:inscribed} is robust to approximation of the maximum-volume inscribed ellipsoid and
give a complexity bound on the number of queries to the loss function and the number of arithmetic operations.
\end{exer}

In both cases all complexities should be $\poly(d, m, 1/\eps, \log(R/r))$ where $m = 1$ for \cref{ex:cut:approx} and the number of constraints defining $K$ in \cref{ex:cut:approx2}, and
it is assumed that $\ball_r \subset K \subset \ball_R$ for known $0 < r \leq R$.
You will need to combine results from many sources. A good place to start would be the references in \cref{sec:reg:compute}.
\end{enumeratenotes}

\chapter[Online Newton Step]{Online Newton Step\copynotice}\label{chap:ons}

\newcommand{\pipm}{v}
\newcommand{\pipmr}{v_\varrho}

We can now present a simple method for obtaining $\tilde O(d^{1.5}\sqrt{n})$ regret for losses in $\cF_\pb$ with the limitation that 
the analysis only works in the stochastic setting where $f_t = f$ for all rounds.\index{setting!stochastic}

\begin{assumption}\label{ass:ons}
The following hold:
\begin{enumerate}
\item The setting is stochastic: $f_t = f$ for all $t$.
\item The loss is bounded: $f \in \cF_\pb$.
\item The constraint set is rounded: $\ball_1 \subset K \subset \ball_{2d}$.
\end{enumerate}
\end{assumption}

The assumption that $K$ is rounded is not restrictive, since the constraint set can be repositioned as explained in \cref{sec:reg:rounding}.
The bandit algorithm presented here is based on online Newton step, which is a second-order online learning algorithm.\index{online Newton step}
Compared to cutting plane methods in \cref{chap:ellipsoid}, the method here has an improved dimension-dependence and can be generalised to the adversarial setting (\cref{chap:ons-adv}).
On the negative side, the analysis is quite involved and the algorithm is hard to tune.
We start the chapter with an intuitive argument about the role of curvature in bandit convex optimisation.
There follows an introduction to online Newton step in the full information setting\index{setting!full information} and a brief explanation of some concepts in convex geometry.
The algorithm and its analysis are presented at the end. 
To ease the presentation and analysis we let $\delta \in (0,1)$ be a small positive constant and
\begin{align*}
L = C \log(1/\delta)\,,
\end{align*}
where $C > 0$ is a sufficiently large universal constant.
We will prove a bound on the regret that holds with probability at least $1 - \delta$ but at various points we implicitly
assume that $\delta \leq \poly(1/n, 1/d)$.

\section{The Blessing and Curse of Curvature}\index{curvature}

\begin{wrapfigure}[5]{r}{4cm}
\vspace{-0.6cm}
\includegraphics[width=4cm]{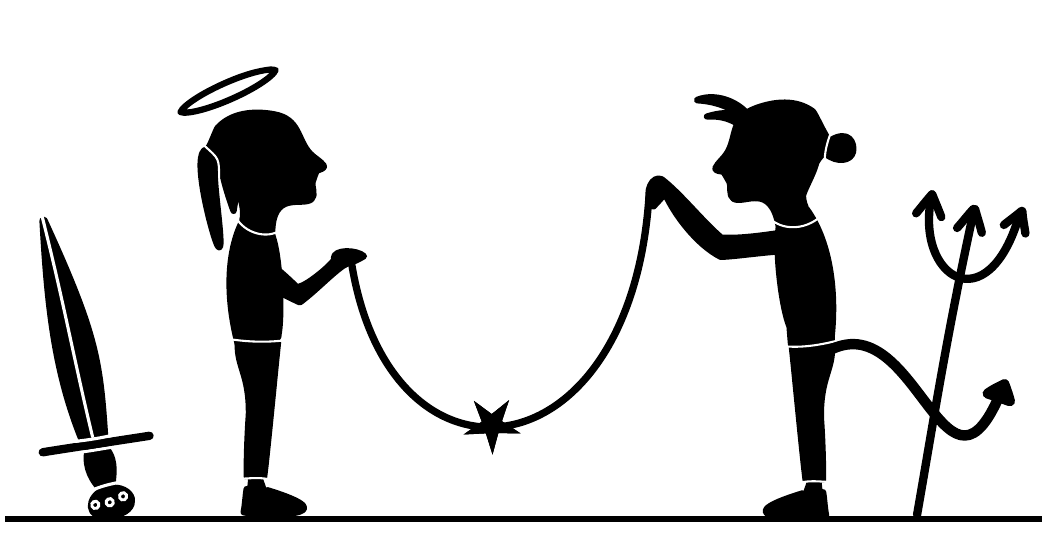}
\end{wrapfigure}
The presence of curvature in bandit convex optimisation is both a blessing and a curse.
The key to obtaining optimal regret is to make sure you exploit the positive aspects while taking care to control the negative ones.
It helps to think about the case where the loss is nearly quadratic in the sense that it has a nearly constant Hessian.
The main implications of high curvature are the following:
\begin{itemize}
\item Smoothing should be done on a smaller radius to maintain a suitably small approximation error. This increases the variance of the gradient estimator, which makes gradient-based
online learning algorithms unstable.
\item The variance of the gradient estimator is modulated by the regularisation of the algorithm, which suggests that the amount of regularisation should increase with the curvature.
But adding more regularisation means the algorithm moves more slowly, which normally increases its regret.
The saving grace is that in the presence of curvature the regret decreases quadratically as the iterate approaches the minimiser.
\end{itemize}
We saw this behaviour already in \cref{sec:ftrl:sm-sc} where the loss was assumed to be in $\cF_{\pb,\psm,\psc}$. 
Provided that $\beta/\alpha$ is not too large, such losses are nearly quadratic. 
\cref{alg:ftrl:basic-sc} uses a self-concordant barrier\index{self-concordant barrier} for regularisation with an additional quadratic that depends on $\alpha$ and uses a gradient estimate that integrates
over a region that is small for large $\alpha$. That is, more curvature implies more regularisation and less smoothing. 
There are multiple challenges when generalising this approach to the setting where the loss is only assumed to be in $\cF_{\pb}$:
\begin{itemize}
\item The amount of curvature is not known.
\item Even if the loss is approximately quadratic, the curvature can be large in some directions and small in others.
\item The loss may not even be differentiable. For example it could be piecewise linear. How should we understand the role of curvature in these situations?
\end{itemize}
The plan is to use a surrogate loss function that does so much smoothing that it is nearly quadratic on a region containing both the current iterate $\mu_t$ and the minimiser of the loss.
The curvature of this surrogate can then be estimated and used in online Newton step, which we explain next.

\section{Introducing Online Newton Step}\label{sec:ons:ons}
\index{online Newton step}
Let $\hat q_1,\ldots,\hat q_n \colon \R^d \to \R$ be a sequence of (possibly non-convex)\index{non-convex} quadratic functions
and consider the full information setting\index{setting!full information}
where in round $t$ the learner proposes $\mu_t \in K$ and observes the entire function $\hat q_t$, and the regret relative to $x \in K$ is 
\begin{align*}
\hqReg_n(x) = \sum_{t=1}^n \left(\hat q_t(\mu_t) - \hat q_t(x)\right) \,. 
\end{align*}
Online Newton step is a second-order method summarised in \cref{alg:ons:abstract}.

\begin{algorithm}[h!]
\begin{algcontents}
\begin{lstlisting}
args: $\eta > 0$, $\Sigma_1^{-1} \in \pd$ and $\mu_1 \in K$
for $t = 1$ to $n$
  let $g_t = \hat q_t'(\mu_t)$ and $H_t = \hat q_t''(\mu_t)$
  update $\Sigma_{t+1}^{-1} = \Sigma_t^{-1} + \eta H_t$
  update $\mu_{t+1} = \argmin_{x \in K} \norm{x - \left[\mu_t - \eta \Sigma_{t+1} g_t\right]}_{\Sigma_{t+1}^{-1}}^2$
\end{lstlisting}
\caption{Online Newton step for quadratic losses}
\label{alg:ons:abstract}
\end{algcontents}
\end{algorithm}

\FloatBarrier

Before the analysis, let us make some connections between online Newton method and other techniques in online learning.

\subsubsection*{Connection to Exponential Weights}
A little notation is needed. Let $\Theta = \R^d \times \pd$ and $\Theta_K = K \times \pd$. 
Let $\mu_\theta$ and $\Sigma_\theta$ be the obvious projections from $\Theta$ into $\R^d$ and $\pd$ respectively
and abbreviate $\cN(\theta) = \cN(\mu_\theta, \Sigma_\theta)$. 
Lastly, let $\KL(\theta, \vartheta)$ be the relative entropy between Gaussian distributions $\cN(\theta)$ and $\cN(\vartheta)$, which has an explicit form:
\begin{align*}
\KL(\theta, \vartheta) = \frac{1}{2} \left[\log \det\left(\Sigma^{\vphantom{-1}}_\vartheta \Sigma_\theta^{-1}\right) + \tr(\Sigma^{\vphantom{-1}}_\theta \Sigma_\vartheta^{-1}) 
+ \norm{\mu_\theta - \mu_\vartheta}^2_{\Sigma_\vartheta^{-1}} - d\right] \,.
\end{align*}
Assume that $\frac{1}{2} \norm{\cdot}^2_{\Sigma_1} + \sum_{s=1}^t \hat q_s$ is convex for all $t$.
Suppose that $p_1$ is the density of $\cN(\mu_1, \Sigma_1)$ and $(p_t)_{t=1}^n$ are Gaussians with parameters $(\theta_t)_{t=1}^n \in \Theta_K$ 
defined inductively as follows. Given $p_t$, define $\tilde \theta_{t+1}$ as the parameters of the Gaussian with density
$\tilde p_{t+1}$ given by
\begin{align*}
\tilde p_{t+1}(x) &= \frac{p_t(x) \exp\left(-\eta \hat q_t(x)\right)}{\int_{\R^d} \exp\left(-\eta \hat q_t(y)\right) p_t(y) \d{y}} \,.
\end{align*}
Then let $\theta_{t+1} = \argmin_{\theta \in \Theta_K} \KL(\theta, \tilde \theta_{t+1})$ and $p_{t+1} = \cN(\theta_{t+1})$.
\index{exponential weights}\index{projection!relative entropy}\index{relative entropy}
The mean $\mu_t$ and covariance $\Sigma_t$ that define $\theta_t$ are exactly the iterates produced by online Newton step \citep{HEK18}.

\subsubsection*{Classical Newton}
The classical Newton method for unconstrained minimisation of a loss function $f \colon \R^d \to \R$ starts with $x_1 \in \R^d$ and uses the update
rule $x_{t+1} = x_t - f''(x_t)^{-1} f'(x_t)$, which corresponds to minimising the quadratic approximation of $f$ at $x_t$.
Online Newton step looks superficially similar, but the preconditioning matrix is based on the accumulated curvature\index{curvature} rather than the local curvature; and
there is the learning rate, which further slows the algorithm.
\index{Newton's method}

\subsubsection*{High-Level Behaviour}
Let us suppose for a moment that $d = 1$ and $f \colon [-1,1] \to [0,1]$ is convex and minimised at $0$. 
In our application, the $\hat q_t$ will be estimates of a quadratic approximation of an extension of $f$. But to simplify our thinking let us suppose that $f$ is quadratic
and $\hat q_t = f$, which means that $g_t = f'(\mu_t)$ and $H_t = f''(\mu_t) \triangleq H$. 
Suppose that $\Sigma_1^{-1} = 1$. By construction, $\Sigma_t^{-1} = 1 + \eta t H$, which means that 
\begin{align*}
\mu_{t+1} = \mu_t - \frac{\eta}{1 + \eta t H} g_t  \,. 
\end{align*}
In our application $\eta = \Theta(\sqrt{1/n})$, which means that online Newton step moves very slowly as $t$ grows unless there is very little curvature.
The corresponding flow in continuous time is \index{continuous time}
\begin{align*}
\d {\mu}(t) = -\frac{\eta}{1 + \eta t H} f'(\mu(t)) \d{t} = -\frac{\eta H \mu(t)}{1 + \eta t H} \d{t} \,, 
\end{align*}
which has a closed-form solution $\mu(t) = \frac{\mu(0)}{1+\eta tH}$, and
the regret is
\begin{align*}
\int_0^n (f(\mu(t)) - f(0)) \d{t} = \frac{H}{2} \int_0^n \mu(t)^2 \d{t} = \frac{H}{2} \int_0^n \left(\frac{\mu(0)}{1 + \eta t H}\right)^2 \d{t} = O(\sqrt{n}) \,.
\end{align*}
Of course, the regret of the gradient flow would be greatly reduced by increasing $\eta$. But in the bandit setting $g_t$ and $H_t$ need to be estimated 
and the increased regularisation is needed to control the variance. What the argument above shows is that despite the slow progress of the algorithm when $\eta = \Theta(n^{-1/2})$,
a regret of $O(\sqrt{n})$ is nevertheless achievable.

\subsubsection*{Analysis}
Moving now to the analysis of online Newton step, which mirrors that of other gradient-based algorithms, 
we have the following theorem:

\begin{theorem}\label{thm:ons}
Suppose that $\Sigma_t^{-1} \in \pd$ for all $1 \leq t \leq n+1$; then for any $x \in K$,
\begin{align}
\frac{1}{2} \norm{\mu_{n+1} - x}^2_{\Sigma_{n+1}^{-1}} \leq \frac{1}{2} \norm{\mu_1 - x}^2_{\Sigma_1^{-1}} + \frac{\eta^2}{2} \sum_{t=1}^n \norm{g_t}^2_{\Sigma_{t+1}} - \eta \hqReg_n(x) \,.
\label{eq:ons:ons-bound}
\end{align}
\end{theorem}

\begin{remark}
The condition that the inverse covariance matrices are positive definite corresponds to assuming that 
\begin{align*}
\frac{1}{2} \norm{\cdot}^2_{\Sigma_1^{-1}} + \sum_{s=1}^{t-1} \hat q_s
\end{align*}
is convex for all $1 \leq t \leq n+1$.
In most applications the first term in \cref{eq:ons:ons-bound} is dropped and the regret is moved to the left-hand side. An interesting feature of our application of
this result is that we use it to simultaneously bound the regret and $\norm{\mu_{n+1} - x}_{\Sigma_{n+1}^{-1}}$.
\end{remark}

\begin{proof}[Proof of \cref{thm:ons}]
By definition, for any $x \in K$,
\begin{align}
\frac{1}{2} \norm{\mu_{t+1} - x}^2_{\Sigma_{t+1}^{-1}}
&\explana\leq \frac{1}{2} \norm{\mu_t - x - \eta \Sigma_{t+1} g_t}^2_{\Sigma_{t+1}^{-1}} \nonumber \\
&\explana= \frac{1}{2} \norm{\mu_t - x}^2_{\Sigma_{t+1}^{-1}} - \eta \ip{g_t, \mu_t - x} + \frac{\eta^2}{2} \norm{g_t}^2_{\Sigma_{t+1}} \nonumber \\
&\explana= \frac{1}{2} \norm{\mu_t - x}^2_{\Sigma_{t}^{-1}} - \eta (\hat q_t(\mu_t) - \hat q_t(x)) + \frac{\eta^2}{2} \norm{g_t}^2_{\Sigma_{t+1}}, 
\label{eq:ons:ons}
\end{align}
where in \explanr{} we used the assumption that $\Sigma_{t+1}^{-1}$ is positive definite so that $\norm{\cdot}_{\Sigma_{t+1}^{-1}}$ is a norm, the fact that $x \in K$ and the definition
\begin{align*}
\mu_{t+1} = \argmin_{\mu \in K} \norm{\mu - [\mu_t - \eta \Sigma_{t+1} g_t]}_{\Sigma_{t+1}^{-1}} \,.
\end{align*}
For the equalities, \explanr{} is obtained by expanding the square, and \explanr{} since 
\begin{align*}
\frac{1}{2} \norm{\mu_t - x}_{\Sigma_{t+1}^{-1}}^2 
&= \frac{1}{2} \norm{\mu_t - x}_{\Sigma_t^{-1}}^2 + \frac{\eta}{2} (\mu_t - x)^\top H_t (\mu_t - x) \\
&= \frac{1}{2}\norm{\mu_t - x}_{\Sigma_t^{-1}}^2 - \eta (\hat q_t(\mu_t) - \hat q_t(x)) + \eta \ip{g_t, \mu_t - x} \,.
\end{align*}
The proof is completed by summing the inequality in \cref{eq:ons:ons} over $t$ from $1$ to $n$.
\end{proof}

\section{Regularity}\label{sec:ons:reg}
In our application of online Newton step to bandits it will be important that the constraint set $K$ is suitably rounded.\index{rounded}
We explained the basics already in \cref{sec:reg:rounding} but here we introduce a more subtle concept based on the mean width of the polar body.
Under Assumption~\ref{ass:ons},
\begin{align}
\ball_1 \subset K \subset \ball_{2d} \,.
\label{eq:ons:rounded}
\end{align}
Note, any improvements in the constant $2$ would only lead to minor constant-factor improvements in the regret.
Let $\pi$ be the Minkowski functional of $K$ (\cref{sec:regularity:minkowski}) and
\begin{align*}
K_\eps = \{x \in K \colon \pi(x) \leq 1 - \eps\} = (1 - \eps) K \,.
\end{align*}
Let $X$ be uniformly distributed on $\sphere_1$ and define \index{uniform measure} $M(K) = \E[\pi(X)]$.
The Minkowski functional is the support function of the polar\index{polar}\index{support function} of $K$ (\cref{lem:reg:mink}\ref{lem:reg:mink:support}), 
which means that for $x \in \sphere_1$, $\pi(x) + \pi(-x)$ is
the width of the polar $K^\circ$ in direction $x$. Hence $M(K)$ is half the mean width of $K^\circ$ as illustrated in Figure~\ref{fig:width}.\index{mean width}
For our application it is best if $K$ is positioned so that $M(K)$ is small.
Thanks to our assumption that $\ball_1 \subset K \subset \ball_{2d}$, $\frac{1}{2d} \norm{x} \leq \pi(x) \leq \norm{x}$ 
and hence $M(K) \in [\frac{1}{2d}, 1]$.
This estimate is often loose, however.
\cref{tab:widths} gives bounds on the inner and outer radii and $M(K)$ for $K$ in various classical positions.
While L\"owner's position yields the strongest bound, all rows are relevant when computation is important, with the best
position depending on how $K$ is represented and what computational resources are available.

\begin{remark}
It is not obvious that there exists an affine transformation $T$ such that $(TK)^\circ$ is in isotropic position, much less that $T$ can be approximately computed
efficiently. More details can be found in Note~\ref{note:ons:iso}.
\end{remark}

\begin{figure}
\centering
\begin{tikzpicture}[thick]
\draw[fill=grayone] (0,0) -- (1,2) -- (3,3) -- (5,2.5) -- (6,1.5) -- (5.5, 0.5) -- (4,-1) -- (1,-1.5) -- (0, 0);
\node at (3,0.5) {$K^\circ$};
\draw[-latex] (0,0) -- (-1,0);
\draw[-latex] (6,1.5) -- (7,1.5);
\draw[dotted] (0,0) -- (1.3,0);
\draw[dotted] (6,1.5) -- (1.3,1.5);
\draw (1.3,-2) -- (1.3,3);
\node[anchor=west] at (1.3,-0.5) {$\zeros$};
\draw[fill=black] (1.3,-0.5) circle (1pt);
\node[anchor=west] at (7,1.5) {$x$};
\node[anchor=east] at (-1,0) {$-x$};
\node[anchor=south] at (3.65,1.5) {$\pi(x)$};
\node[anchor=south] at (0.65,0) {$\pi(-x)$};
\end{tikzpicture}

\caption{The width of $K^\circ$ in direction $x$ is $\pi(x) + \pi(-x)$. The quantity $M(K)$ is obtained by integrating the width 
uniformly over all directions in $x \in \sphere_1$ and dividing by two.}
\label{fig:width}
\commentAlt{A two-dimensional depiction of the polar of K and the supporting hyperplanes with normals x and minus x,
which are parallel and separated by the sum of the Minkowski functionals of x and minus x.}
\end{figure}
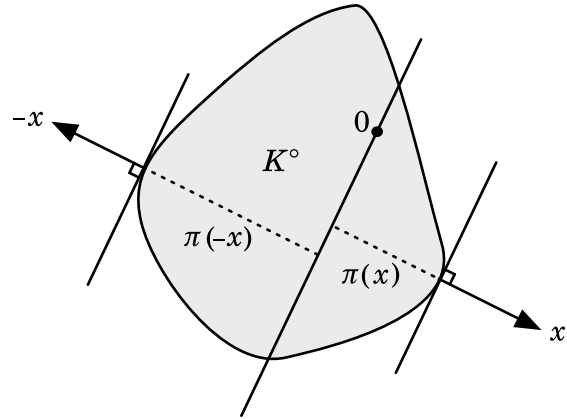

\begin{table}[h!]
\renewcommand{\arraystretch}{1.6}
\centering
\small
\caption{Classical positions and bounds on the inner and outer radii and $M(K)$.
More discussion and references appear in the notes of this chapter.
}\label{tab:widths}

\begin{tabular}{|Nlll|}
\hline 
\multicolumn{2}{|c}{\textsc{conditions}} & \textsc{bounds} & $M(K)$ \\ \hline
\label{tab:widths:low} & $\frac{1}{d} K$ in L\"owner's position & $\ball_1 \subset K \subset \ball_d$ & $\tilde O(d^{-1/2})$ \index{L\"owner's position}\\
\label{tab:widths:iso} & $K$ isotropic & $\ball_1 \subset K \subset \ball_{d+1}$ & $\leq 1$ \index{isotropic position} \\
\label{tab:widths:iso-sym} & $K$ symmetric, isotropic & $\ball_1 \subset K \subset \ball_{d+1}$ & $\tilde O(d^{-1/10})$ \\
\label{tab:widths:john} & $K$ in John's position & $\ball_1 \subset K \subset \ball_d$ & $\leq 1$ \index{John's position} \\
\label{tab:widths:polar-iso} & $(d+1)K^\circ$ isotropic & $\ball_1 \subset K \subset \ball_{d+1}$ & $\tilde O(d^{-1/4})$ \\ 
\label{tab:widths:john-sym} & $K$ symmetric, $\frac{1}{\sqrt{d}}K$ in John's position & $\ball_{\sqrt{d}} \subset K \subset \ball_d$ & $\leq 1/\sqrt{d}$ \\
\label{tab:widths:polar-iso-sym} & $K$ symmetric, $(d+1)K^\circ$ isotropic & $\ball_1 \subset K \subset \ball_{d+1}$ & $\tilde O(d^{-1/2})$ \\ \hline
\end{tabular}
\end{table}

\begin{proof}[Proof of claims in \cref{tab:widths}]
\begin{enumerate}
\item[(\ref{tab:widths:low})]
Let $w(K)$ be the half mean width of a convex body $K$, which is defined by $w(K) = \E[h_K(\theta)]$ where $\theta$ has law $\cU(\sphere_1)$ and $h_K$ is the
support function of $K$.\index{support function}\index{mean width}
Suppose that $\frac{1}{d} K$ is in L\"owner's position. By definition $(K/d)^\circ = d K^\circ$ is in John's position. 
Combining John's theorem and the fact that polarity reverses inclusion again shows that $\ball_1 \subset K \subset \ball_d$.\index{John's theorem}
Let $S$ be a regular simplex in John's position.
\cite{barthe1998extremal} proved that the mean width of a convex body in John's position is maximised by $S$ and 
a result of \cite{finch2011mean} shows that $w(S) = O(\sqrt{d \log(d)})$. 
Therefore $M(K) = w(K^\circ) = \frac{1}{d} w(d K^\circ) \leq \frac{1}{d} w(S) = O(\sqrt{\log(d)/d})$.
\item[(\ref{tab:widths:iso},\ref{tab:widths:john})]
Well-roundedness follows from \cref{thm:isotropic} and \cref{thm:john}, respectively.
Since in both positions $\ball_1 \subset K$, it follows that $K^\circ \subset \ball_1$ and hence $M(K) \leq 1$.
Note, this claim is not improvable for John's position as explained in Note~\ref{note:ons:johns}.
\item[(\ref{tab:widths:iso-sym})] 
Well-roundedness follows as in \ref{tab:widths:iso}. The bound on $M(K)$ is supplied by \cite{giannopoulos2014m} and is most likely conservative.
\item[(\ref{tab:widths:polar-iso})]
That $\ball_1 \subset K \subset \ball_{d+1}$ follows because polarity reverses inclusion and by \cref{thm:isotropic}.
The bound on $M(K)$ is due to \cite{pivovarov2010volume}.
\item[(\ref{tab:widths:john-sym})] Suppose that $K$ is symmetric and in John's position. Then by John's theorem \citep[Theorem 2.1.3]{ASG15}, $\ball_1 \subset K \subset \ball_{\sqrt{d}}$. Then use the scaling and repeat the argument for \ref{tab:widths:john}.
\item[(\ref{tab:widths:polar-iso-sym})]
As for (\ref{tab:widths:polar-iso}) but use the result of \cite{milman2015mean} to bound $M(K)$.
\end{enumerate}
\end{proof}

An essential ingredient in many previous regret analyses is a bound on the magnitude of the observed losses. 
The algorithm in this chapter replaces the real loss with the extended loss function defined in \cref{sec:reg:extension} and
the magnitude of this loss depends on the Minkowski functional.
For this reason it is essential to have a good understanding of the law of $\pi(X)$ when $X$ is Gaussian, which the next two lemmas provide.

\begin{lemma}\label{lem:ons:mink}
Suppose that $X$ has law $\cN(\mu, \Sigma)$ with $\mu \in K$ and $\Sigma \preceq \sigma^2 \id$. Then
\begin{enumerate}
\item $\E[\pi(X)] \leq 1 + \sigma M(K) \sqrt{d}$; and \label{lem:ons:mink:E}
\item $\bbP\left(\left|\pi(X) - \E[\pi(X)]\right| \geq \sqrt{2 \sigma^2 \log(2/\delta)}\right) \leq \delta$ for all $\delta \in (0,1)$.
\label{lem:ons:mink:psi}
\end{enumerate}
\end{lemma}

\begin{proof}
Since $\pi$ is sub-additive (\cref{lem:reg:mink}),
\begin{align}
\E[\pi(X)] 
&= \E[\pi(X - \mu + \mu)] \nonumber \\
&\leq \E[\pi(X - \mu)] + \pi(\mu) \nonumber \\ 
&\leq 1 + \E[\pi(X - \mu)] \,. 
\label{eq:ons:mink:1}
\end{align}
By assumption $\sigma^2 \id - \Sigma$ is positive semidefinite.
Let $W$ and $U$ be independent of $X$ and have laws $\cN(\zeros, \sigma^2 \id)$ and $\cN(\zeros, \sigma^2 \id - \Sigma)$ respectively.
Note that $U + X - \mu$ has the same law as $W$ and by Jensen's inequality\index{Jensen's inequality} $\E[\pi(X - \mu)] \leq \E[\pi(X - \mu + U)] = \E[\pi(W)]$.
Furthermore, since $W/\norm{W}$ and $\norm{W}$ are independent,
\begin{align*}
\E[\pi(X - \mu)] \leq \E[\pi(W)] 
= \E\left[\pi\left(\frac{W}{\norm{W}}\right)\right] \E[\norm{W}]
&\leq \sigma M(K) \sqrt{d}\,,
\end{align*}
where we used the facts that $\E[\norm{W}] \leq \E[\norm{W}^2]^{1/2} = \sigma \sqrt{d}$ and $W/\norm{W}$ is uniformly distributed on $\sphere_1$.
Combining this with \cref{eq:ons:mink:1} completes the proof of part~\ref{lem:ons:mink:E}.
For part~\ref{lem:ons:mink:psi}, by \cref{lem:reg:mink}\ref{lem:reg:mink:lip} and the assumption that $\ball_1 \subset K$
we have $\lip(\pi) \leq 1$. The result follows from \cref{thm:orlicz-lip}.
\end{proof}

As mentioned, we are planning to use the extension introduced in \cref{sec:reg:extension} where the functions $\pip(x) = \max(1, \pi(x)/(1-\eps))$ and $\pipm(x) = \pip(x) - 1$ 
appear with $\eps \in (0, 1/2)$. The following bound will be useful:

\begin{lemma}\label{lem:ons:v}
Let $\eps \in (0,1/2)$, $\pip(x) = \max(1, \pi(x)/(1-\eps))$ and $\pipm(x) = \pip(x) - 1$.
Suppose that $X$ has law $\cN(\mu, \Sigma)$ with $\mu \in K_\eps$ and $\norm{\Sigma} \leq \frac{1}{\delta}$. Then 
\begin{align*}
\E[\pipm(X)^2] \leq L \left(1 + \max\left(1, M(K) \sqrt{d}\right) \sqrt{\norm{\Sigma}}\right) \left[\tr\left(\Sigma \lim_{\varrho \to 0} \E[\pipmr''(X)]\right) + \delta\right]\,,
\end{align*}
where $\pipmr = \pipm * \phi_\varrho$ with $\phi_\varrho$ the smoothing kernel defined in \cref{sec:reg:smooth}.
\end{lemma}

\begin{proof}[\Proofskippy] 
By \cref{lem:orlicz-var},
\begin{align}
\E[\pipm(X)^2] &\leq \norm{\pipm(X)}_{\psi_1}\left( \E[\pipm(X)] \left(1 + \log\left(\frac{\E[\pipm(X)^2]}{\delta^2}\right)\right) + \delta\right) \,.
\label{eq:ons:lem:v}
\end{align}
By convexity and the fact that $\lip(v) < \infty$,
\begin{align*}
\E[\pipm(X)] 
&= \lim_{\varrho \to 0} \E[\pipmr(X)] \\
\tag*{Convexity}
&\leq \lim_{\varrho \to 0} \left[\pipmr(\mu) + \E\left[\ip{\pipmr'(X), X - \mu}\right]\right] \\
&= \lim_{\varrho \to 0} \E\left[\ip{\pipmr'(X), X - \mu}\right] \\
&= \lim_{\varrho \to 0} \E\left[\ip{\Sigma \pipmr'(X), \Sigma^{-1}(X - \mu)}\right] \\
\tag*{Integration by parts}
&= \tr\left(\Sigma \lim_{\varrho \to 0} \E[\pipmr''(X)]\right)\,.
\end{align*}
By the assumption that $\eps < 1/2$ and because $\lip(\pi) \leq 1$, it follows that $\lip(v) \leq 2$.
Hence, using \cref{thm:orlicz-lip} and \cref{lem:ons:mink}\ref{lem:ons:mink:E},
\begin{align*}
\norm{\pipm(X)}_{\psi_1}
&\leq \norm{\pipm(X) - \E[\pipm(X)]}_{\psi_1} + \norm{\E[\pipm(X)]}_{\psi_1} \\ 
&\leq 2\sqrt{6\norm{\Sigma}} + \frac{\E[\pipm(X)]}{\log(2)}  \\
&\leq 2\sqrt{6\norm{\Sigma}} +  \frac{2(1 + M(K) \sqrt{d \norm{\Sigma}})}{\log(2)} \,.
\end{align*}
The second moment $\E[\pipm(X)^2]$ that appears in \cref{eq:ons:lem:v} can now be bounded using \cref{lem:orlicz-moment} and the result follows by naive simplification.
\end{proof}

\section{Extension and Surrogate Losses} \index{surrogate loss!quadratic}\label{sec:ons:surrogate}
As in \cref{chap:ftrl} we use a quadratic surrogate. Unlike that chapter, however, the curvature\index{curvature} of the surrogate now depends on the loss function
and needs to be estimated. The other distinction is that in \cref{chap:ftrl} the actions were sampled 
from a scaled Dikin ellipsoid,\index{Dikin ellipsoid} which is guaranteed to be contained
in $K$. By contrast, the algorithm presented in this chapter will sample from a Gaussian, which means its actions may lie outside of $K$, possibly
with high probability. This problem is handled by making use of the extension in \cref{sec:reg:extension}. \index{extension}
Let $\pi$ be the Minkowski functional of $K$ and for $\eps \in (0,1/2)$ let $K_\eps = (1-\eps)K$, 
$\pip(x) = \max(1, \pi(x)/(1-\eps))$ and
\begin{align}
\ext(x) = \pip(x) f\left(\frac{x}{\pip(x)}\right) + \frac{2 (\pip(x)-1)}{\eps} \,. 
\label{eq:ons:ext}
\end{align}
Note that $\pip(x) - 1 = 0$ and $\ext(x) = f(x)$ for $x \in K_\eps$ 
and \cref{prop:reg:bandit-extension-eps} ensures that $x \mapsto \pip(x) f(x/\pip(x)) + \frac{1-\eps}{\eps}(\pip(x)-1)$ is convex.
The additional factor in the second term of the above display ensures that $\ext$ has slightly more curvature, which gives the algorithm an additional nudge to play inside $K_\eps$.
The extension in \cref{eq:ons:ext} is defined on all of $\R^d$ and can be queried at any $x \in \R^d$ by evaluating the real loss $f$ at $x / \pip(x)$. 
To simplify the notation it is convenient to abstract away this reduction by redefining the meaning of the actions $X_t$ and observed losses $Y_t$.
In round $t$ the algorithm samples $X_t$ from a Gaussian $\cN(\mu_t, \Sigma_t)$ but actually plays $X_t / \pip(X_t) \in K_\eps$,
observes $f(X_t / \pip(X_t)) + \eps_t$ and computes the (noisy) loss relative to the extension as 
\begin{align*}
Y_t = \ext(X_t) + \pip(X_t) \eps_t = \pip(X_t)\left[f\left(\frac{X_t}{\pip(X_t)}\right) + \eps_t\right] + \frac{2 (\pip(X_t)-1)}{\eps}\,.
\end{align*}
Note that the noise term is now effectively $\pip(X_t) \eps_t$, which conditioned on $X_t$ can have variance as large as $\pip(X_t)^2$.
With the new meaning of $X_t$, the regret is
\begin{align*}
\Reg_n = \sup_{x \in K} \sum_{t=1}^n \left(f\left(\frac{X_t}{\pip(X_t)}\right) - f(x)\right) \,.
\end{align*}
By \cref{prop:reg:bandit-extension-eps}, $f(x) = e(x/\pip(x)) \leq e(x)$ for all $x \in \R^d$ and $e(x) = f(x)$ for all $x \in K_\eps$, which when
combined with \cref{prop:shrink} shows that
\begin{align*}
\Reg_n 
&\leq n\eps + \max_{x \in K_\eps} \sum_{t=1}^n \left(f\left(\frac{X_t}{\pip(X_t)}\right) - f(x)\right) \\
&\leq n\eps + \max_{x \in K_\eps} \sum_{t=1}^n \left(e(X_t) - e(x)\right) \,.
\end{align*} 
Therefore the true regret is upper bounded in terms of the regret relative to the extension.

\subsubsection*{Surrogate Losses}
The algorithm makes use of a quadratic surrogate loss
\begin{align}
q_t(x) = \sip{s_t'(\mu), x - \mu_t} + \frac{1}{4} \norm{x - \mu_t}^2_{s''(\mu_t)}\,,
\label{eq:ons:q}
\end{align}
where $s_t \colon \R^d \to \R$ is the convex surrogate defined by \index{surrogate loss!Gaussian}
\begin{align*}
s_t(x) = \E_{t-1}\left[\left(1 - \frac{1}{\lambda}\right) \ext(X_t) + \frac{1}{\lambda} \ext((1 - \lambda)X_t + \lambda x)\right]
\end{align*}
with $\lambda \in (0,\frac{1}{d+1})$ a tuning parameter that determines the amount of smoothing.
We spend all of \cref{chap:opt} on the intuitions and analysis of this surrogate loss. You can skip ahead to that chapter now or accept the following
properties as gospel.   

\begin{proposition}\label{prop:ons:q}
Let $q_t$ be the function defined in \cref{eq:ons:q} and suppose that $x \in \R^d$ satisfies $\lambda \norm{x - \mu_t}_{\Sigma_t^{-1}} \leq \frac{1}{\sqrt{L}}$
and $\lambda \leq \frac{1}{d L^2}$. Then
\begin{align*}
\E_{t-1}[\ext(X_t)] - \ext(x) \leq q_t(\mu_t) - q_t(x) + \frac{4}{\lambda} \tr\left(q_t''(\mu_t) \Sigma_t\right) + \frac{1}{n} \,.
\end{align*}
\end{proposition}

\begin{proof}[\Proofskippy] 
By \cref{cor:gauss:lower},
\begin{align*}
\E_{t-1}[\ext(X_t)] - \ext(x) 
&\leq q_t(\mu_t) - q_t(x) + \frac{2}{\lambda} \tr(s_t''(\mu_t)\Sigma_t) + \delta\left[\frac{2d}{\lambda} + \frac{1}{\lambda^2}\right] \\
&\leq q_t(\mu_t) - q_t(x) + \frac{4}{\lambda} \tr(q_t''(\mu_t)\Sigma_t) + \frac{1}{n} \,,
\end{align*}
where the second inequality follows because 
$q_t''(\mu_t) = s_t''(\mu_t)/2$ and by naively bounding the constants. 
\end{proof}

\subsubsection*{Estimation}
The function $q_t$ cannot be reconstructed from $X_t$ and $Y_t$ alone.
But it can be estimated by
\begin{align*}
\hat q_t(x) = \ip{g_t, x - \mu_t} + \frac{1}{2} (x - \mu_t)^\top H_t (x - \mu_t) \,,
\end{align*}
where $g_t$ and $H_t$ are defined by
\begin{align*}
g_t &= \frac{R_t Y_t\Sigma_t^{-1} (X_t - \mu_t)}{(1-\lambda)^2} \quad \text{ and }\\
H_t &= \frac{\lambda R_t Y_t}{2(1-\lambda)^2} \left[\frac{\Sigma_t^{-1}(X_t - \mu_t)(X_t - \mu_t)^\top \Sigma_t^{-1}}{(1 - \lambda)^2} - \Sigma_t^{-1}\right] \,,
\end{align*}
with $p_t$ the density of $\cN(\mu_t, \Sigma_t)$ and
\begin{align*}
R_t = \frac{p_t\left(\frac{X_t -\lambda \mu}{1-\lambda}\right)}{(1 - \lambda)^d p_t(X_t)} \,.
\end{align*}
By \cref{prop:opt:bias}, which follows from a change of measure and Stein's lemma,
\begin{align*}
\E_{t-1}[g_t] = s_t'(\mu_t) = q_t'(\mu_t) \quad \text{ and } \quad \E_{t-1}[H_t] = \frac{1}{2} s_t''(\mu_t) = q_t''(\mu_t) \,.
\end{align*}

\section{Algorithm and Analysis}\label{sec:ons:analysis}

The algorithm combines online Newton step with the quadratic surrogate estimates from \cref{sec:ons:surrogate}.

\begin{algorithm}[h!]
\begin{algcontents}
\begin{lstlisting}
args: $\eta$, $\lambda$, $\sigma^2$, $\eps$
$\mu_1 = \zeros$, $\Sigma_1 = \sigma^2 \id$
for $t = 1$ to $n$
  sample $X_t$ from $\cN(\mu_t, \Sigma_t)$ with density $p_t$
  observe $Y_t = \pip(X_t) \left(f\left(\frac{X_t}{\pip(X_t)}\right) + \eps_t\right) + \frac{2 (\pip(X_t)-1)}{\eps}$ $\label{line:ons:mink}$
  let $R_t = \frac{p_t\left(\frac{X_t - \lambda \mu_t}{1 - \lambda}\right)}{(1 - \lambda)^d p_t(X_t)}$
  compute $g_t = \frac{R_t Y_t \Sigma_t^{-1} (X_t - \mu_t)}{(1-\lambda)^2}$
  compute $H_t = \frac{\lambda R_t Y_t}{2(1 -\lambda)^2} \left[\frac{\Sigma_t^{-1}(X_t - \mu_t)(X_t - \mu_t)^\top \Sigma_t^{-1}}{(1-\lambda)^2} - \Sigma_t^{-1}\right]$
  compute $\Sigma_{t+1}^{-1} = \Sigma_t^{-1} + \eta H_t$ 
  compute $\mu_{t+1} = \argmin_{\mu \in K_\eps} \norm{\mu - [\mu_t - \eta \Sigma_{t+1} g_t]}_{\Sigma_{t+1}^{-1}}$ $\label{line:ons:update}$
\end{lstlisting}
\caption{Online Newton step for convex bandits}
\label{alg:ons:bandit}
\end{algcontents}
\end{algorithm}

\FloatBarrier

\subsubsection*{Computation}
\cref{alg:ons:bandit} is straightforward to implement and relatively efficient. 
The main computational bottlenecks are as follows:
\begin{itemize}
\item \textit{Gaussian sampling}:\,\, The algorithm samples from a Gaussian with mean $\mu_t$ and covariance $\Sigma_t$. Given access to standard Gaussian noise, this more or less corresponds to computing
an eigenvalue decomposition of the covariance matrix $\Sigma_t$, which at least naively requires $O(d^3)$ operations per round.
Probably this can be improved with a careful incremental implementation. See Note~\ref{note:ons:inc}.
\item \textit{Minkowski functional}:\,\, Remember that 
\begin{align*}
\pip(X_t) = \max(1, \pi(X_t)/(1-\eps)) \,.
\end{align*}
This can be approximated to accuracy $1/n^2$ using 
bisection search and only logarithmically many queries to a membership oracle for $K$.
The increase in regret due to the approximation is negligible.
\item \textit{Projections}:\,\, The projection in Line~\ref{line:ons:update} is a convex optimisation problem and the hardness depends 
on how $K$ is represented (\cref{tab:comp}).
Note that the projection is only needed in rounds $t$ where $\mu_t - \eta \Sigma_{t+1} g_t \notin K_\eps$.\index{projection}
\end{itemize}
The algorithm needs $\tilde O(d^2)$ memory\index{memory} to store the covariance matrix. There is also the initial rounding procedure to put $K$ into a good position, which only
needs to be done once and is discussed in
\cref{sec:reg:rounding} and in the notes for this chapter.

\begin{theorem}\label{thm:ons:bandit}
Let $M = \max(d^{-1/2}, M(K))$ and
\begin{align*}
\sigma &= \frac{1}{M \sqrt{2d}} &
\lambda &= \frac{1}{4d^{1.5} M L^2} &
\eta &= \frac{dM}{3} \sqrt{\frac{1}{n}} &
\eps &= \frac{240 d^2 ML^4}{\sqrt{n}} \,.
\end{align*}
Under \cref{ass:ons}, with probability at least $1 - \delta$ 
the regret of \cref{alg:ons:bandit} is bounded by
\begin{align*}
\Reg_n \leq 480 d^2 M L^4 \sqrt{n} \,.
\end{align*}
\end{theorem}

According to \cref{tab:widths}, if $K/d$ is in L\"owner's position, then $M = \tilde O(d^{-1/2})$ and the regret is $\tilde O(d^{1.5} \sqrt{n})$.
The detailed proof of \cref{thm:ons:bandit} is deferred to \cref{sec:thm:ons:bandit}, but we give an outline below.

\begin{proof}[Proof outline]
The rigorous proof depends on a relatively intricate concentration analysis, which we brush over for now. 

\stepsection{Step 1: Regret comparison}
To begin, let $x_\star = \argmin_{x \in K_\eps} f(x)$ and
\begin{align*}
\eReg_n(x_\star) = \sum_{t=1}^n \left(\E_{t-1}[e(X_t)] - e(x_\star)\right) \,.
\end{align*}
The regret with respect to the quadratic surrogates and its estimates are
\begin{align*}
\qReg_n(x_\star) &= \sum_{t=1}^n \left(q_t(\mu_t) - q_t(x_\star)\right) \quad\text{and} \\
\hqReg_n(x_\star) &= \sum_{t=1}^n \left(\hat q_t(\mu_t) - \hat q_t(x_\star)\right) \,. 
\end{align*}
We start by comparing the true regret to the regret relative to the extended losses:
\begin{align}
\Reg_n 
\tag*{\cref{prop:shrink}}
&\leq n\eps + \sum_{t=1}^n \left(f(X_t/\pip(X_t)) - f(x_\star)\right) \nonumber \\
\tag*{Concentration}
&\explan{whp}\lesssim n\eps + \sum_{t=1}^n \left(\E_{t-1}[f(X_t/\pip(X_t))] - f(x_\star)\right) \nonumber \\
\tag*{\cref{prop:reg:bandit-extension-eps}\ref{prop:reg:bandit-extension-eps:leq}}
&\leq n\eps + \sum_{t=1}^n \left(\E_{t-1}[e(X_t)] - f(x_\star)\right) \nonumber \\
\tag*{\cref{prop:reg:bandit-extension-eps}\ref{prop:reg:bandit-extension-eps:equal}}
&= n\eps + \sum_{t=1}^n \left(\E_{t-1}[e(X_t)] - e(x_\star)\right) \nonumber \\
&= \tilde O\left(M d^2 \sqrt{n}\right) + \eReg_n(x_\star)\,,
\label{eq:ons:sketch:ereg}
\end{align}
where the final line follows from the definition of $\eps$.
Therefore it suffices to bound $\eReg_n(x_\star)$.
The plan is to use \cref{prop:ons:q} to compare the regrets relative to the extension and the quadratic surrogates.
A serious issue is that the quadratic surrogate is only well-behaved on an ellipsoid about $\mu_t$. Concretely, we will need to show that with high probability
\begin{align*}
F_t \triangleq \frac{1}{2} \norm{\mu_t - x_\star}^2_{\Sigma_t^{-1}} = \tilde O(1/\lambda^2) \text{ for all } 1 \leq t \leq n \,.
\end{align*}
Let $\tau$ be the first round where $F_{\tau+1}$ is not $\tilde O(1/\lambda^2)$, with $\tau$ defined to be $n$ if no such round exists.
By \cref{prop:ons:q}, for any $t \leq \tau$
\begin{align*}
\E_{t-1}[e(X_t)] - e(x_\star) \lesssim q_t(\mu_t) - q_t(x_\star) + \frac{4}{\lambda} \tr(q''_t(\mu_t) \Sigma_t)\,,
\end{align*}
where we ignored the final miniscule term in \cref{prop:ons:q}.
Hence,
\begin{align}
\eReg_\tau(x_\star) 
&\lesssim \qReg_\tau(x_\star) + \frac{4}{\lambda} \sum_{t=1}^n \tr(q''_t(\mu_t) \Sigma_t) \nonumber \\
&\explan{whp}\lesssim \hqReg_\tau(x_\star) + \frac{4}{\lambda} \sum_{t=1}^n \tr(q''_t(\mu_t) \Sigma_t) \,, \label{eq:ons:e-q-reg}
\end{align}
where in the second inequality we used the fact that with high probability the estimated quadratic losses are close cumulatively to the unobserved quadratic losses.
By \cref{thm:ons},
\begin{align}
F_{\tau+1}
&\leq F_1 + \frac{\eta^2}{2} \sum_{t=1}^\tau \norm{g_t}^2_{\Sigma_t} - \eta \hqReg_\tau(x_\star) \nonumber \\
&\leq \frac{2d^2}{\sigma^2} + \frac{\eta^2}{2} \sum_{t=1}^\tau \norm{g_t}^2_{\Sigma_t} - \eta \hqReg_\tau(x_\star) \nonumber \\
&\explan{whp}\lesssim \frac{2d^2}{\sigma^2} + \frac{\eta^2}{2} \sum_{t=1}^\tau \norm{g_t}^2_{\Sigma_t} + \frac{4\eta}{\lambda} \sum_{t=1}^\tau \tr(q_t''(\mu_t) \Sigma_t) - \eta \eReg_\tau(x_\star) \,,
\label{eq:ons:sketch}
\end{align}
where in the second inequality we used the assumption that $K \subset \ball_{2d}$ and $\mu_1 = \zeros$ to bound $F_1 = \frac{1}{2\sigma^2} \norm{x - \mu_1}^2\leq 2d^2/\sigma^2$.
The third inequality follows from \cref{eq:ons:e-q-reg}.
There are two terms in the right-hand side that need a little more manipulation, which we break out into two additional steps.

\stepsection{Step 2: Bounding gradient norms}
By definition
\begin{align*}
\norm{g_t}^2_{\Sigma_t} 
= \frac{R_t^2 Y_t^2}{(1 - \lambda)^4} \norm{X_t - \mu_t}^2_{\Sigma_t^{-1}}
\explan{whp}= \tilde O\left(d Y_t^2\right)\,, 
\end{align*}
where we used the facts that $0 \leq R_t \leq 3$ (\cref{lem:measure}) and $\lambda \in (0,1/2)$ and that $\Sigma_t^{-1/2}(X_t - \mu_t)$ is
a standard Gaussian.
Summing shows that
\begin{align*}
\frac{\eta^2}{2} \sum_{t=1}^\tau \norm{g_t}^2_{\Sigma_t} \explan{whp}= \tilde O\left(d \eta^2 \sum_{t=1}^\tau Y_t^2\right) \,.
\end{align*}
We saw expressions like this in many other analyses and simply bounded $Y_t^2 \leq 1$. But because we have used the extended loss this is not possible anymore.
Fortunately, one can prove that the algorithm mostly plays close to $K_\eps$ where the extended loss and true losses are equal.
A complex argument eventually shows that with high probability $\sum_{t=1}^\tau Y_t^2 = \tilde O(n)$.

\stepsection{Step 3: Bounding the trace term}
We claim the sum of traces in \cref{eq:ons:sketch} is bounded by
\begin{align}
\frac{4}{\lambda} \sum_{t=1}^\tau \tr(q_t''(\mu_t) \Sigma_t) = \tilde O\left(\frac{d}{\lambda \eta}\right)\,.
\label{eq:ons:trace}
\end{align}
By definition $\E_{s-1}[H_s] = q''_s(\mu_s)$ and therefore it is plausible (and true) that with high probability 
\begin{align*}
\Sigma_t^{-1} = \Sigma_1^{-1} + \eta \sum_{s=1}^{t-1} H_s \approx \Sigma_1^{-1} + \eta \sum_{s=1}^{t-1} q_s''(\mu_s)\,.
\end{align*}
Let us make a continuous-time\index{continuous time} approximation, which often provides a good ansatz for such problems. 
Let $\bar \Sigma(t)^{-1} = \Sigma_1^{-1} + \eta \int_1^t \bar H_{\floor{s}} \d{s}$ where $\bar H_s = q''_s(\mu_s)$.
Then
\begin{align*}
\sum_{t=1}^\tau \tr(\bar H_t \Sigma_t) 
&\approx \frac{1}{\eta} \int_{1}^\tau \tr\left(\frac{\d{}}{\d{t}}(\bar \Sigma(t)^{-1}) \bar \Sigma(t)\right) \d{t} \\
&= \frac{1}{\eta} \int_{1}^\tau \frac{\d{}}{\d{t}} [-\log \det(\bar \Sigma(t))] \d{t} \\
&= \frac{1}{\eta} \log \det(\bar \Sigma(1) \bar \Sigma(\tau)^{-1}) \\
&\explan{$(\star)$} \leq \frac{d}{\eta} \log\left(\frac{\tr(\bar \Sigma(1) \bar \Sigma(\tau)^{-1})}{d}\right) \\
&\explan{whp}= \tilde O\left(\frac{d}{\eta}\right) \,,
\end{align*}
where $(\star)$ follows from the arithmetic--geometric mean inequality and
the final inequality follows by proving that $\snorm{\bar \Sigma(\tau)^{-1}} \leq \poly(n, d)$ with high probability.
This justifies \cref{eq:ons:trace}.

\stepsection{Step 4: Combining}
By the previous two steps and \cref{eq:ons:sketch},
\begin{align}
F_{\tau+1} 
&\explan{whp}\lesssim \frac{2d^2}{\sigma^2} + \tilde O(\eta^2 n d) + \tilde O\left(\frac{d}{\lambda}\right) - \eta \eReg_\tau(x_\star) \nonumber \\
&= \tilde O(1/\lambda^2) - \eta \eReg_\tau(x_\star) \,.
\label{eq:ons:sketch:final}
\end{align}
Since $\eReg_\tau(x_\star) \geq 0$, this shows that $F_{\tau+1} = \tilde O(1/\lambda^2)$ and by the definition of $\tau$ this means that $\tau = n$. 
\cref{eq:ons:sketch:final} also shows that 
\begin{align*}
\eReg_n(x_\star) = \tilde O(1/(\eta \lambda^2)) = \tilde O(M d^2 \sqrt{n}) \,,
\end{align*}
which when combined with \cref{eq:ons:sketch:ereg} completes the argument.
\end{proof}

\begin{remark}
Looking at \cref{eq:ons:sketch:final}, you may wonder: why not choose $\sigma$ to be very large?
The reason is hidden in the calculations needed in the second step. Consider the case that $d = 1$ and $K = [-1,1]$.
In the very first round the algorithm samples $X_t \sim \cN(\zeros, \sigma^2 \id)$ and therefore
$\E[|X_1|] \approx \sigma$.
Hence, if $\sigma \gg 1$, then by the definition of the extended loss
\begin{align*}
\E[\ext(X_1)] = \Omega\left(\frac{\sigma}{\eps}\right) \approx \sigma \sqrt{n} \,.
\end{align*}
So the regret in even a single round is $\Omega(\sigma \sqrt{n})$.
\end{remark}

\section{Proof of Theorem~\ref{thm:ons:bandit}}\label{sec:thm:ons:bandit}

At various points we need that $\eps \in (0,1/2)$, which only holds for sufficiently large $n$.
Note, however, that if $\eps \geq 1/2$, then the regret bound in the theorem is implied by the assumption that $f \in \cF_{\pb}$ so that $\Reg_n \leq n$.
Hence, for the remainder we assume that $\eps \in (0,1/2)$.
The main complication is that the conclusion of \cref{prop:ons:q} only holds for some $x$ and the losses of the extension of $f$ are very much
not bounded in $[0,1]$.
At various points in the analysis we refer to various relations between the constants. These are collected in \cref{sec:ons:constraints}.
Let 
\begin{align*}
F_t = \frac{1}{2} \norm{\mu_t - x_\star}^2_{\Sigma_t^{-1}} \,.
\end{align*}
In order to make our analysis go through we need to argue that $F_t \leq \frac{1}{2 L \lambda^2}$ for all $t$ with high probability.
There are a few other complications. Most notably, the algorithm is not properly defined if $\Sigma_t$ fails to be positive definite. Hence we need
to prove that this occurs with low probability. 
Note that $\E_{t-1}[H_t]$ is the Hessian of a convex function and hence positive semidefinite. Thus we will use concentration-of-measure to show that $\Sigma_t$
indeed stays positive definite with high probability. Define the following quantities:
\begin{align*}
S_t = \sum_{u=1}^t H_u  \quad\text{and}\quad
\bar S_t = \sum_{u=1}^t \E_{u-1}[H_u] \quad\text{and}\quad
\bar \Sigma_t^{-1} = \Sigma_1^{-1} + \eta \bar S_{t-1} \,.
\end{align*}
We also let
\begin{align*}
\qReg_\tau(x) = \sum_{t=1}^\tau (q_t(\mu_t) - q_t(x)) \quad\text{and}\quad
\hqReg_\tau(x) = \sum_{t=1}^\tau (\hat q_t(\mu_t) - \hat q_t(x)) \,.
\end{align*}

\begin{definition}\label{def:ons:tau}
Let $\tau$ be the first round when one of the following does \textit{not} hold:
\begin{enumerate}
\item $F_{\tau+1} \leq \frac{1}{2L \lambda^2}$. \label{def:ons:tau:F} 
\item $\Sigma_{\tau+1}$ is positive definite. \label{def:ons:tau:P}
\item $\delta \id \preceq \frac{1}{2} \bar \Sigma_{\tau+1}^{-1} \preceq \Sigma_{\tau+1}^{-1} \preceq \frac{3}{2} \bar \Sigma_{\tau+1}^{-1} \preceq \delta^{-1} \id$. \label{def:ons:tau:S}
\end{enumerate}
In case no such round exists, then $\tau$ is defined to be $n$.
\end{definition}

Note that $F_{t+1}$ and $\Sigma_{t+1}$ are measurable\index{measurable} with respect to $\sF_t$, which means that $\tau$ is a stopping time\index{stopping time} with respect to the filtration $(\sF_t)_{t=1}^n$.\index{filtration}
A simple consequence of the definition of $\tau$ is that for any $t \leq \tau$,
\begin{align}
\snorm{\Sigma_t} \leq 2 \snorm{\bar \Sigma_t} = 2 \norm{\left(\frac{1}{\sigma^2} \id + \eta \bar S_t\right)^{-1}}
\leq 2\sigma^2 \,,
\label{eq:ons:sigma}
\end{align}
where the last step follows because $\bar S_t \in \psd$.

\stepsection{Step 1: Regret relative to the extension}
Most of our analysis bounds the regret with respect to the extension of $f$, which is only meaningful because the regret relative to the extension is nearly an upper
bound on the regret relative to the real loss.
Let $x_\star = \argmin_{x \in K_\eps} e(x)$ and
\begin{align*}
\beReg_n(x_\star) &= \sum_{t=1}^n \E_{t-1}[\ext(X_t) - \ext(x_\star)] \,. 
\end{align*}
Let \eventAzuma{} be the event that
\begin{align*}
 \sum_{t=1}^n f(X_t / \pip(X_t))
&\leq \sum_{t=1}^n \E_{t-1}[f(X_t / \pip(X_t))] + \sqrt{2n \log(7/\delta)} \,.
\end{align*}
Since $X_t / \pip(X_t) \in K_\eps \subset K$ and $f \in \cF_\pb$, by Azuma's inequality (\cref{thm:azuma}) $\bbP(\eventAzuma) \geq 1 - \delta/7$.
\index{Azuma's inequality}
Suppose the above high-probability event occurs; then
by \cref{prop:shrink} and \cref{prop:reg:bandit-extension-eps},
\begin{align}
\Reg_n 
&\leq n \eps + \max_{x \in K_\eps} \Reg_n(x) \nonumber \\
&\leq n \eps + \sqrt{2n \log(1/\delta)} + \max_{x \in K_\eps} \sum_{t=1}^n \E_{t-1}[f(X_t / \pip(X_t)) - f(x)] \nonumber \\
&\leq n \eps + \sqrt{2n \log(1/\delta)} + \max_{x \in K_\eps} \sum_{t=1}^n \E_{t-1}[\ext(X_t) - \ext(x)] \nonumber \\
&= n \eps + \sqrt{2n \log(1/\delta)} + \beReg_n(x_\star) \,.
\label{eq:ons:ereg}
\end{align}
Therefore it suffices to bound $\beReg_n(x_\star)$.

\stepsection{Step 2: Concentration}\index{concentration}
We need to show that the behaviour of the various estimators used by \cref{alg:ons:bandit} is suitably regular with
high probability.
Define an event \eventNoise{} by
\begin{align*}
\tag{\eventNoise}
\eventNoise = \left\{\max_{1 \leq t \leq \tau} |\eps_t| \leq \sqrt{\log(14n/\delta)}\right\} \,.
\end{align*}
By \cref{lem:orlicz-tail} and a union bound, $\bbP(\eventNoise) \geq 1 - \delta / 7$.
The magnitude of the observed losses depends heavily on $\pi(X_t)$.
Define an event \eventPi{} by
\begin{align*}
\tag{\eventPi}
\eventPi = \left\{\max_{1 \leq t \leq \tau} \pi(X_t) \leq \sqrt{L} \right\} \,.
\end{align*}
\begin{lemma}
$\bbP(\eventPi) \geq 1 - \delta / 7$.
\end{lemma}

\begin{proof}
By \cref{lem:ons:mink} and a union bound, with probability at least $1 - \delta/7$ for all $t \leq \tau$,
\begin{align*}
\pi(X_t) 
\tag*{by \cref{lem:ons:mink}\ref{lem:ons:mink:psi}}
&\leq \E_{t-1}[\pi(X_t)] + \sqrt{2 \norm{\Sigma_t} \log(14n/\delta)} \\
\tag*{by \cref{lem:ons:mink}\ref{lem:ons:mink:E}}
&\leq 1 + M \sqrt{d\norm{\Sigma_t}} + \sqrt{2 \norm{\Sigma_t} \log(14n/\delta)} \\
\tag*{by \cref{eq:ons:sigma}}
&\leq 1 +  \sigma M \sqrt{2d} + 2 \sigma \sqrt{\log(14n/\delta)} \\
&\leq \sqrt{L} \,.
\end{align*}
where the final inequality holds by the definition of the constants, which satisfy $\sigma \leq \sigma M \sqrt{2d} \leq 1$ (\cref{tab:ons:sigma} in \cref{sec:ons:constraints}).
\end{proof}

We also need to control the magnitude of $\snorm{X_t - \mu_t}_{\Sigma_t^{-1}}$. Define an event $\eventGauss{}$ by
\begin{align*}
\tag{\eventGauss}
\eventGauss = \left\{\max_{1 \leq t \leq \tau} \snorm{X_t - \mu_t}_{\Sigma_t^{-1}} \leq \sqrt{\frac{8d}{3} \log(14n/\delta)}\right\} \,. 
\end{align*}
By \cref{lem:orlicz-tail,prop:orlicz}, $\bbP(\eventGauss) \geq 1 - \delta / 7$.
The next lemma bounds the sum $\sum_{t=1}^\tau \E_{t-1}[Y_t^2]$. Because the extended loss matches the true loss on $K_\eps$ and the latter is bounded in $[0,1]$, we should 
expect that when $X_t \in K_\eps$, then $Y_t = \tilde O(1)$ with high probability. In other words, provided the algorithm is playing mostly in $K_\eps$, then 
we should hope that $\sum_{t=1}^\tau \E_{t-1}[Y_t^2] = \tilde O(n)$.
This is exactly what the following lemma says. The proof is deferred to \cref{sec:lem:ons:bounds} but should not be skipped.

\begin{lemma}\label{lem:ons:bounds}
Let $\Ymax = \max_{1 \leq t \leq \tau} \left(|Y_t| + \E_{t-1}[|Y_t|]\right)$.
On $\eventNoise \cap \eventPi \cap \eventGauss$ the following hold:
\begin{enumerate}
\item $\Ymax \leq \frac{L}{\eps}$. \label{lem:ons:bounds:Ymax}
\item $\sum_{t=1}^\tau \E_{t-1}[Y_t^2] \leq 10n$. \label{lem:ons:bounds:Y}
\end{enumerate}
\end{lemma}

We also need to control $\sum_{t=1}^\tau Y_t^2$.
Let $\eventY{}$ be the event defined by
\begin{align*}
\tag{\eventY}
\eventY = \left\{\sum_{t=1}^\tau Y_t^2 \leq 21n \right\} \,.
\end{align*}

\begin{lemma}\label{lem:ons:eventY}
$\bbP(\eventY \cup (\eventNoise \cap \eventPi \cap \eventGauss)^c) \geq 1 - \delta / 7$.
\end{lemma}

\begin{proof}
Let $E_t = \{|Y_t| \leq L / \eps\}$. 
By \cref{thm:conc:unnormalised2}, with probability at least $1 - \delta / 7$,
\begin{align*}
\sum_{t=1}^\tau \sind_{E_t} Y_t^2 
&\explana\leq 2\sum_{t=1}^\tau \E_{t-1}[\sind_{E_t} Y_t^2] + \frac{L^2 \log(7/\delta)}{\eps^2} \\
&\explana\leq 2\sum_{t=1}^\tau \E_{t-1}[Y_t^2] + \frac{L^3}{\eps^2} \\
&\explana\leq n + 2 \sum_{t=1}^\tau \E_{t-1}[Y_t^2] \,. 
\end{align*}
where \explanr{} follows from \cref{thm:conc:unnormalised2}, \explanr{} by the definition of $L$ and \explanr{} by \cref{tab:ons:eps}.
By \cref{lem:ons:bounds}, on $\eventNoise \cap \eventPi \cap \eventGauss$, $E_t$ holds for all $t \leq \tau$ and in this case
$\sum_{t=1}^\tau \sind_{E_t} Y_t^2 = \sum_{t=1}^\tau Y_t^2$ and $\sum_{t=1}^\tau \E_{t-1}[Y_t^2] \leq 10n$. 
Hence, $\bbP(\eventY \cup (\eventNoise \cap \eventPi \cap \eventGauss)^c) \geq 1 - \delta / 7$.
\end{proof}

The last two events control the concentration of the estimated quadratic surrogate about its mean 
at the optimal point and the concentration of the Hessian estimates.
Let $\eventQ$ be the event that
\begin{align}
\tag{\eventQ}
\eventQ = \left\{
\qReg_\tau(x_\star) \leq \hqReg_\tau(x_\star) + \frac{4 \sqrt{nL}}{\lambda} \right\} \,.
\label{eq:ons:conc:q}
\end{align}

\begin{lemma}
$\bbP(\eventQ \cup (\eventNoise \cap \eventPi \cap \eventGauss \cap \eventY)^c) \geq 1 - \delta / 7$.
\end{lemma}

\begin{proof}
By the definition of $\tau$, for all $t \leq \tau$,
\begin{align*}
\lambda \snorm{\mu_t - x_\star}_{\Sigma_t^{-1}} \leq \frac{1}{\sqrt{L}}\,.
\end{align*}
Hence, by \cref{prop:conc-q}\ref{prop:conc-q:non-uniform}, with probability at least $1 - \delta / 7$,
\begin{align}
\sum_{t=1}^\tau (\hat q_t(x_\star) - q_t(x_\star)) 
&\leq 1 + \frac{1}{\lambda}\left[\sqrt{\sum_{t=1}^\tau \E_{t-1}[Y_t^2] L} + \Ymax L\right]  \,.
\label{eq:ons:conc:q-2}
\end{align}
On this event and $\eventNoise \cap \eventPi \cap \eventGauss$, 
\begin{align*}
\qReg_\tau(x_\star) - \hqReg_\tau(x_\star)
&= \sum_{t=1}^\tau (q_t(\mu_t) - q_t(x_\star)) - \sum_{t=1}^\tau (\hat q_t(\mu_t) - \hat q_t(x_\star)) \\
&\explana= \sum_{t=1}^\tau (\hat q_t(x_\star) - q_t(x_\star)) \\
&\explana\leq 1 + \frac{1}{\lambda} \left[\sqrt{\sum_{t=1}^\tau \E_{t-1}[Y_t^2] L} + \Ymax L\right] \\
&\explana\leq 1 + \frac{1}{\lambda} \left[\sqrt{10n L} + \frac{L^2}{\eps}\right] \\
&\explana\leq \frac{4}{\lambda} \sqrt{n L} \,,
\end{align*}
where \explanr{} holds because $q_t(\mu_t) = \hat q_t(\mu_t) = 0$,
\explanr{} from \cref{eq:ons:conc:q-2},
\explanr{} from \cref{lem:ons:bounds}\ref{lem:ons:bounds:Ymax}\ref{lem:ons:bounds:Y} and
\explanr{} by naive simplification.
\end{proof}

Finally, let $\eventS{}$ be the event that
\begin{align*}
\tag{\eventS}
\eventS = \left\{-6 \lambda L^2 \sqrt{dn} \bar \Sigma_\tau^{-1} \preceq S_\tau - \bar S_\tau \preceq 6 \lambda L^2 \sqrt{dn} \bar \Sigma_\tau^{-1}\right\} \,.
\end{align*}

\begin{lemma}
$\bbP(\eventS \cup (\eventNoise \cap \eventPi \cap \eventGauss \cap \eventY)^c) \geq 1 - \delta / 7$.
\end{lemma}

\begin{proof}
By \cref{prop:conc-hessian-path} with $\Sigma^{-1} = \frac{3}{2} \bar \Sigma_\tau^{-1}$, with probability at least $1 - \delta / 7$,
\begin{align*}
\bar S_\tau - S_\tau &\preceq \lambda L^2 \left[1 + \sqrt{d \sum_{t=1}^\tau \E_{t-1}[Y_t^2]} + d^2 \Ymax \right] \frac{3}{2} \bar \Sigma_\tau^{-1} \quad \text{and}\\ 
S_\tau - \bar S_\tau &\preceq \lambda L^2 \left[1 + \sqrt{d \sum_{t=1}^\tau \E_{t-1}[Y_t^2]} + d^2 \Ymax \right] \frac{3}{2}\bar \Sigma_\tau^{-1} \,.
\end{align*}
As before, the claim follows from \cref{lem:ons:bounds}\ref{lem:ons:bounds:Ymax}\ref{lem:ons:bounds:Y} and 
\cref{tab:ons:eps2}, which says that
$\frac{d^2 L}{\eps} \leq \frac{1}{L}\sqrt{d n}$.
\end{proof}

Let $E = \eventAzuma \cap \eventNoise \cap \eventPi \cap \eventGauss \cap \eventY \cap \eventQ \cap \eventS$ be the intersection of all these high-probability events.
A union bound over all the calculations above shows that $\bbP(E) \geq 1 - \delta$.
For the remainder of the proof we bound the regret on $E$.

\stepsection{Step 3: Simple bounds}
We can now make some elementary conclusions that hold on the intersection of all the high probability events outlined
in the previous step.
To begin, by \cref{lem:ons:bounds}\ref{lem:ons:bounds:Y},
\begin{align}
\sum_{t=1}^\tau \norm{g_t}^2_{\Sigma_t}
&= \sum_{t=1}^\tau \norm{\frac{R_t Y_t \Sigma_t^{-1}(X_t - \mu_t)}{(1-\lambda)^2}}^2_{\Sigma_t} \nonumber \\
&\leq \left(\sum_{t=1}^\tau Y_t^2\right) \max_{1 \leq t \leq \tau} \left(\frac{R_t}{(1 - \lambda)^2}\right)^2 \norm{X_t - \mu_t}^2_{\Sigma_t^{-1}} \nonumber \\
&\leq d n L\,,
\label{eq:ons:conc:g}
\end{align}
where in the last inequality we combined \eventGauss{} to bound the norm and \eventY{} to bound the sum of squared losses 
with \cref{lem:measure} to bound $0 \leq R_t \leq 3$ and used the fact that $\lambda \leq \frac{1}{2}$.
By the definition of \eventS{},
\begin{align*}
\Sigma_{\tau+1}^{-1} 
&= \Sigma_1^{-1} + \eta S_\tau 
\preceq \Sigma_1^{-1} + \eta \bar S_\tau + 6\eta \lambda L^2 \sqrt{dn}\bar \Sigma_{\tau}^{-1}
\preceq \frac{3}{2} \bar \Sigma_{\tau+1}^{-1}\,,
\end{align*}
where the final inequality follows from the definitions of $\eta$ and $\lambda$ (\cref{tab:ons:Sigma-conf}) and because $\bar \Sigma_{\tau}^{-1} \preceq \bar \Sigma_{\tau+1}^{-1}$.
Similarly,
\begin{align*}
\Sigma_{\tau+1}^{-1} 
&\succeq \Sigma_1^{-1} + \eta \bar S_\tau - \frac{1}{2} \bar \Sigma_{\tau}^{-1}
= \bar \Sigma_{\tau+1}^{-1} - \frac{1}{2} \bar \Sigma_\tau^{-1} 
\succeq \frac{1}{2} \bar \Sigma_{\tau+1}^{-1}\,.
\end{align*}
Combining shows that
\begin{align}
\frac{1}{2} \bar \Sigma_{\tau+1}^{-1} \preceq \Sigma_{\tau+1}^{-1} \preceq \frac{3}{2} \bar \Sigma_{\tau+1}^{-1} \,.
\label{eq:ons:S}
\end{align}
We also want to show that $2\delta \id \preceq \bar \Sigma_{\tau+1}^{-1} \preceq \frac{2}{3 \delta}\id$.
The left-hand inequality is immediate because $\bar \Sigma_{\tau+1} \preceq \bar \Sigma_1 = \sigma^2 \id \preceq \frac{1}{2 \delta} \id$. 
By \cref{prop:opt:H-upper}, for any $t \leq \tau$,
\begin{align*}
\norm{\E_{t-1}[H_t]} 
&= \frac{1}{2} \norm{s''_t(\mu_t)} \\ 
&\leq \frac{\lambda \lip(e)}{2(1 - \lambda)} \sqrt{d \norm{\Sigma_t^{-1}}} \\
&\leq \frac{\lambda \lip(e)}{2(1-\lambda)} \sqrt{\frac{3d}{2} \norm{\bar \Sigma_t^{-1}}} \\
&\leq \frac{\lambda \lip(e)}{2(1 - \lambda)} \sqrt{\frac{d}{\delta}}\,,
\end{align*}
where the final inequality follows from \cref{def:ons:tau} and because $t \leq \tau$.
Therefore, by ensuring that $\delta = O(1/\poly(n, d))$ is small enough and bounding $\lip(e) = O(1/\eps)$ using \cref{prop:reg:bandit-extension-eps},
\begin{align}
\norm{\bar \Sigma_{\tau+1}} 
= \norm{\bar \Sigma_1^{-1} + \eta \sum_{u=1}^\tau \bar H_u} 
\leq \norm{\bar \Sigma_1^{-1}} + \frac{\eta n \lambda \lip(e)}{2(1-\lambda)} \sqrt{\frac{d}{\delta}}
\leq \frac{1}{\delta}\,.
\label{eq:ons:S-bound}
\end{align}
Therefore both \cref{def:ons:tau}\ref{def:ons:tau:P} and \cref{def:ons:tau}\ref{def:ons:tau:S} hold and so the only way that 
$\tau \neq n$ is if $F_{\tau+1} > \frac{1}{2L \lambda^2}$.

\stepsection{Step 4: Trace/Logdet inequalities}
Online Newton step bounds the regret relative to the estimated quadratic surrogate losses, which are close to the true quadratic losses. The regret relative to the extension
can be bounded in terms of the regret relative to quadratic surrogates:

\begin{lemma}\label{lem:ons:trace-log}
On event $E$, $\displaystyle \beReg_\tau(x_\star) \leq \qReg_\tau(x_\star) + \frac{d L}{\lambda \eta}$.
\end{lemma}

\begin{proof}
By the definition of $\tau$, for $t \leq \tau$ it holds that $\lambda \norm{\mu_t - x_\star}_{\Sigma_t^{-1}} \leq \frac{1}{\sqrt{L}}$. Hence,
by \cref{prop:ons:q},
\begin{align}
\beReg_\tau(x_\star) \leq \qReg_\tau(x_\star) + \frac{4}{\lambda} \sum_{t=1}^\tau \tr(q''_t(\mu_t) \Sigma_t) + 1 \,.
\label{eq:ons:trace-log}
\end{align}
By \cref{prop:opt:H-upper} and \cref{prop:reg:bandit-extension-eps}, for any $t \leq \tau$,
\begin{align}
\eta \norm{\Sigma_t^{1/2} q_t''(\mu_t) \Sigma_t^{1/2}}
\leq \frac{\eta \lambda \lip(e)}{2(1 - \lambda)} \sqrt{d \norm{\Sigma_t}} 
\leq \frac{2\eta \lambda \sigma \sqrt{2d}}{\eps (1 - \lambda)} 
\leq 1 \,,
\label{eq:ons:trace-log-2}
\end{align}
where in the second last inequality we used \cref{eq:ons:sigma} to bound $\snorm{\Sigma_t} \leq 2\sigma^2$ and \cref{prop:reg:bandit-extension-eps} and the assumption 
that $\ball_1 \subset K$ to bound $\lip(e) \leq \frac{2}{\eps(1-\eps)} \leq \frac{4}{\eps}$.
The final inequality holds because of the choice of $\lambda$, $\eta$ and $\eps$ (\cref{tab:ons:logdet}).
Hence, by \cref{lem:tr-logdet},
\begin{align}
\frac{4}{\lambda} \sum_{t=1}^\tau \tr(q''_t(\mu_t) \Sigma_t) 
&\explana\leq \frac{16}{\lambda \eta} \sum_{t=1}^\tau \log\det\left(\id + \frac{\eta q''_t(\mu_t) \Sigma_t}{2}\right) \nonumber \\ 
&\explana\leq \frac{16}{\lambda \eta} \sum_{t=1}^\tau \log \det\left(\id + \eta q''_t(\mu_t) \bar \Sigma_t\right) \nonumber \\
&= \frac{16}{\lambda \eta} \sum_{t=1}^\tau \log \det\left(\bar \Sigma_t^{\vphantom{-1}} \bar \Sigma_{t+1}^{-1}\right) \nonumber \\
&= \frac{16}{\lambda \eta} \log \det \left(\bar \Sigma_1^{\vphantom{-1}} \bar \Sigma_{\tau+1}^{-1}\right) \nonumber \\
&\explana\leq \frac{16}{\lambda \eta} \log \det \left(\frac{\sigma^2}{\delta} \id\right) \nonumber \\
&\explana\leq \frac{dL}{\lambda \eta} - 1\,, 
\label{eq:ons:logdet}
\end{align}
where \explanr{} follows from \cref{lem:tr-logdet} and \cref{eq:ons:trace-log-2},
\explanr{} because for $t \leq \tau$, $\Sigma_t \leq 2 \bar \Sigma_t$, and
\explanr{} follows from \cref{eq:ons:S-bound}.
Lastly, \explanr{} holds by the definition of $L$.
The claim of the lemma now follows from \cref{eq:ons:trace-log}.
\end{proof}

\stepsection{Step 5: Regret}
By \cref{eq:ons:S}, for any $t \leq \tau$,
\begin{align}
\norm{g_t}^2_{\Sigma_{t+1}}
&\leq 2 \norm{g_t}^2_{\bar \Sigma_{t+1}}
\leq 2 \norm{g_t}^2_{\bar \Sigma_t}
\leq 3 \norm{g_t}^2_{\Sigma_t}\,.
\label{eq:ons:g}
\end{align}
By \cref{thm:ons} and the bounds in \cref{eq:ons:conc:g,eq:ons:conc:q},
\begin{align*}
F_{\tau+1}
&=\frac{1}{2} \norm{x_\star - \mu_{\tau+1}}^2_{\Sigma_{\tau+1}^{-1}} \\
&\explana\leq \frac{\norm{x_\star}^2}{2 \sigma^2} + \frac{\eta^2}{2} \sum_{t=1}^\tau \norm{g_t}^2_{\Sigma_{t+1}} - \eta \hqReg_\tau(x_\star) \\
&\explana\leq \frac{\norm{x_\star}^2}{2\sigma^2} + \frac{3\eta^2 n d L}{2} - \eta \hqReg_\tau(x_\star) \\
&\explana\leq \frac{\norm{x_\star}^2}{2\sigma^2} + \frac{3\eta^2 n d L}{2} + \frac{4\eta  \sqrt{nL}}{\lambda}  - \eta \qReg_\tau(x_\star) \\
&\explana\leq \frac{\norm{x_\star}^2}{2\sigma^2} + \frac{3\eta^2 n d L}{2} + \frac{d L}{\lambda} + \frac{4\eta \sqrt{nL}}{\lambda} - \eta \beReg_\tau(x_\star) \\
&\explana\leq \frac{2d^2}{\sigma^2} + \frac{3\eta^2 n d L}{2} + \frac{d L}{\lambda} + \frac{4\eta \sqrt{nL}}{\lambda} - \eta \beReg_\tau(x_\star) \\
&\explana\leq \frac{1}{2L\lambda^2} - \eta \beReg_\tau(x_\star)\,,
\end{align*}
where \explanr{} follows from \cref{thm:ons}, \explanr{} from \cref{eq:ons:conc:g} and \cref{eq:ons:g}, \explanr{} holds on event \ref{eq:ons:conc:q}, and
\explanr{} follows from \cref{lem:ons:trace-log}.
\explanr{} follows because we assumed that $K \subset \ball_{2d}$. 
Lastly, \explanr{} follows by substituting the definition of the constants \cref{tab:ons:main}.
Therefore all of the following hold:
\begin{enumerate}
\item $F_{\tau+1} \leq \frac{1}{2 L\lambda^2}$.
\item $\beReg_\tau(x_\star) \leq \frac{1}{2 L \eta \lambda^2}$.
\item $\Sigma_{\tau+1}$ is positive definite and $\delta \id \preceq \frac{1}{2} \bar \Sigma_{\tau+1}^{-1} \preceq \Sigma_{\tau+1}^{-1} \preceq \frac{3}{2} \bar \Sigma_{\tau+1}^{-1} \preceq \delta^{-1} \id$.
\end{enumerate}
By \cref{def:ons:tau}, on this event we have $\tau = n$ and hence $\beReg_n(x_\star) \leq \frac{1}{2L \eta \lambda^2}$.
At long last, \cref{thm:ons:bandit} now follows from the definitions of $\eta$ and $\lambda$ and \cref{eq:ons:ereg}.

\section{Proof of Lemma~\ref{lem:ons:bounds}}\label{sec:lem:ons:bounds} 
Staring with part~\ref{lem:ons:bounds:Ymax}, let $t \leq \tau$. Then, on \eventNoise{} and \eventPi{},
\begin{align}
|Y_t| 
&= \left|\ext(X_t) + \pip(X_t) \eps_t\right| \nonumber \\
&= \left|\pip(X_t) f\left(\frac{X_t}{\pip(X_t)}\right) + \frac{2(\pip(X_t) - 1)}{\eps} + \pip(X_t) \eps_t\right| \nonumber \\
&\leq \frac{L}{2\eps}\,,
\label{eq:lem:ons:bounds:1}
\end{align}
where in the final inequality we used the definitions of $\eventNoise$ and $\eventPi$ and the definition $\pip(X_t) = \max(1, \pi(X_t)/(1+\eps))$.
The expectation is bounded by
\begin{align}
\E_{t-1}[|Y_t|] 
&= \E_{t-1}\left[\left|\pip(X_t) f\left(\frac{X_t}{\pip(X_t)}\right) + \frac{2(\pip(X_t) - 1)}{\eps} + \pip(X_t) \eps_t\right|\right] \nonumber \\
&\explana\leq \frac{2}{\eps} \E_{t-1}\left[\pip(X_t)\right] + \E_{t-1}[|\pip(X_t) \eps_t|]  \nonumber \\
&\explana\leq \frac{3}{\eps} \E_{t-1}\left[\pip(X_t)\right] \nonumber \\
&\explana\leq \frac{3}{\eps}\left(1 + \sigma M \sqrt{2d}\right) \nonumber \\
&\explana\leq \frac{L}{2\eps}\,,
\label{eq:lem:ons:bounds:2}
\end{align}
where \explanr{} follows because $f \in \cF_{\pb}$ is bounded;
\explanr{} holds because $\E_{t-1}[|\eps_t||X_t] \leq 1$ and by naively bounding $1 \leq 1/\eps$; and
\explanr{} follows from \cref{lem:ons:mink} and because $\norm{\Sigma_t} \leq 2\sigma^2$ by \cref{eq:ons:sigma}.
Lastly, \explanr{} is true because $\sigma M \sqrt{2d} \leq 1$ (\cref{tab:ons:sigma}).
Combining \cref{eq:lem:ons:bounds:1} and \cref{eq:lem:ons:bounds:2} completes the proof of part~\ref{lem:ons:bounds:Ymax}.
Part~\ref{lem:ons:bounds:Y} is more interesting.
The main point is that under $\eventNoise, \eventPi, \eventGauss$ the only way that $Y_t$ can be large is if $X_t$ is far outside of $K_\eps$.
By the definition of the algorithm $\mu_t \in K_\eps$, which means that for $X_t$ to be large the covariance $\Sigma_t$ must also be relatively large.
But because $\Sigma_t^{-1}$ increases with curvature\index{curvature} and the extended loss function has considerable curvature near $\partial K_\eps$, we should expect that
as the algorithm plays outside $K_\eps$ the covariance $\Sigma_t$ will decrease and hence $X_t$ gets closer to $K_\eps$ and $Y_t = \tilde O(1)$.
Recall that $v(x) = \pip(x) - 1$ and $v_\varrho = v * \phi_\varrho$ where $\phi_\varrho$ is the smoothing kernel from \cref{sec:reg:smooth}. Then
\begin{align*}
&\sum_{t=1}^\tau \E_{t-1}[Y_t^2] 
= \sum_{t=1}^\tau \E_{t-1}\left[\left(\pip(X_t) \left(f(X_t/\pip(X_t)) + \eps_t\right) + \frac{2 \pipm(X_t)}{\eps}\right)^2\right] \\
&\quad\explana\leq \sum_{t=1}^\tau \E_{t-1}\left[4 \pip(X_t)^2 + \frac{8 \pipm(X_t)^2}{\eps^2}\right] \\
&\quad\explana\leq \sum_{t=1}^\tau \E_{t-1}\left[8 + \frac{10 \pipm(X_t)^2}{\eps^2} \right] \\
&\quad\explana\leq 8n + \frac{10}{\eps^2} \sum_{t=1}^\tau L\left(1 + M \sqrt{d \snorm{\Sigma_t}}\right) \left[\tr\left(\Sigma_t \lim_{\varrho \to 0} \E_{t-1}[\pipmr''(X_t)]\right) + \delta\right] \\
&\quad\explana\leq 8n + \frac{20 L}{\eps^2} \sum_{t=1}^\tau \left[\tr\left(\Sigma_t \lim_{\varrho \to 0} \E_{t-1}[\pipmr''(X_t)]\right) + \delta\right] \\
&\quad\explana\leq 9n + \frac{80 L}{\lambda \eps} \sum_{t=1}^\tau \tr\left(\Sigma_t q_t''(\mu_t)\right) 
\explana\leq 9n + \frac{20 dL^2}{\lambda \eta \eps} 
\explana\leq 10n 
\end{align*}
where \explanr{} uses that $(a+b)^2 \leq 2a^2 + 2b^2$ 
and the assumption that $\eps_t$ subgaussian\index{subgaussian} under $\bbP_{t-1}(\cdot|X_t)$ so that $\E_{t-1}[(f(X_t/\pip(X_t)) + \eps_t)^2|X_t] \leq 2$.
\explanr{} uses that $\eps \leq 1/2$ and again that $(a+b)^2 \leq 2a^2 + 2b^2$;
\explanr{} follows from \cref{lem:ons:v} and because $M = \max(M(K), 1/\sqrt{d})$;
\explanr{} follows from \cref{eq:ons:sigma} to bound $\norm{\Sigma_t} \leq 2 \sigma^2$
and the definition of the constraints (\cref{tab:ons:sigma});
\explanr{6} follows from \cref{eq:ons:logdet}; and \explanr{7} uses the definition of
the constants again (\cref{tab:ons:bounds}).
It remains to justify \explanr{5}.
Recall that
\begin{align*}
\ext(x) 
&= \left[\pip(x) f\left(\frac{x}{\pip(x)}\right) + \frac{\pipm(x)}{\eps}\right] + \frac{\pipm(x)}{\eps} 
\triangleq h(x) + \frac{\pipm(x)}{\eps}\,.
\end{align*}
Now $h$ is convex by \cref{prop:reg:bandit-extension-eps} 
and $\lip(h) < \infty$, and hence by \cref{prop:reg:smooth}, Proposition \ref{prop:gauss:mean-hess} and \cref{tab:ons:lambda},
\begin{align*}
\tr(\Sigma_t q_t''(\mu_t))
&= \frac{1}{2} \tr(\Sigma_t s_t''(\mu_t)) \\
&\geq \frac{\lambda}{4} \tr\left(\Sigma_t \lim_{\varrho \to 0} \E_{t-1}\left[e''_\varrho(X_t)\right]\right) - \frac{\delta d}{2} \\
&\geq \frac{\lambda}{4 \eps} \tr\left(\Sigma_t \lim_{\varrho \to 0} \E_{t-1}\left[\pipmr''(X_t)\right]\right) - \frac{\delta d}{2} \,,
\end{align*}
where the last inequality holds because $e''_\varrho(x) = h''_\varrho(x) + \frac{1}{\eps} \pipmr''(x) \succeq \frac{1}{\eps} v_\varrho''(x)$.
Rearranging shows that
\begin{align*}
\tr\left(\Sigma_t \lim_{\varrho \to 0} \E_{t-1}\left[\pipmr''(X_t)\right]\right) + \delta 
\leq \frac{4 \eps}{\lambda} \tr\left(\Sigma_t q_t''(\mu_t)\right) + \delta\left[1 + \frac{2 \eps d}{\lambda}\right]\,,
\end{align*}
which by naively bounding the terms involving $\delta$ suffices to establish \explanr{5}.

\section{Constraints}\label{sec:ons:constraints} 

Most bandit algorithms are fairly straightforward to tune as we saw in \cref{chap:sgd,chap:ftrl,chap:lin,chap:exp}.
Regrettably the parameters of \cref{alg:ons:bandit} interact in a more complicated way. 
The constraints needed for the analysis are listed in \cref{tab:ons:constraints}.

\begin{table}[h!]
\renewcommand{\arraystretch}{1.7}
\caption{Constraints on the constants needed in the analysis of \cref{alg:ons:bandit}.}
\label{tab:ons:constraints}
\begin{tabular}{|Np{6cm}|}
\hline
\multicolumn{1}{|c}{} & \multicolumn{1}{p{5cm}|}{\textsc{constraint}} \\ \hline
\label{tab:ons:eps} & $\frac{L^3}{\eps^2} \leq n$ \\ 
\label{tab:ons:eps2} & $\frac{d^2 L}{\eps} \leq \frac{1}{L}\sqrt{dn}$ \\
\label{tab:ons:Sigma-conf} & $6 \eta \lambda L^2 \sqrt{dn} \leq \frac{1}{2}$ \\
\label{tab:ons:logdet} & $\frac{2 \sigma \eta \lambda}{\eps(1 - \lambda)} \sqrt{2d} \leq 1$ \\
\label{tab:ons:main} & $\frac{2d^2}{\sigma^2} + \frac{3 \eta^2 n d L}{2} + \frac{d L}{\lambda} + \frac{4\eta \sqrt{nL}}{\lambda} \leq \frac{1}{2 L \lambda^2}$ \\
\label{tab:ons:sigma} & $\sigma M \sqrt{2d} \leq 1$ \\
\label{tab:ons:bounds} & $\frac{20dL^2}{\lambda \eta \eps} \leq n$ \\
\label{tab:ons:lambda} & $\lambda \leq \frac{1}{d L^2}$ \\
\hline
\end{tabular}
\end{table}

You can check (laboriously) that the constants defined by
\begin{align*}
\sigma &= \frac{1}{M \sqrt{2d}} &
\lambda &= \frac{1}{4d^{1.5} M L^2} &
\eta &= \frac{dM}{3} \sqrt{\frac{1}{n}} &
\eps &= \frac{240 d^2 ML^4}{\sqrt{n}}
\end{align*}
satisfy all of the above constraints.

\section{Notes}

\begin{enumeratenotes}

\item Let us collect some notes on the trade-offs between the various positions.
\begin{itemizeinner}
\item L\"owner's position yields a regret of $\tilde O(d^{1.5} \sqrt{n})$ for any $K$. But in general it can only be computed efficiently when $K$
is the convex hull of a small number of vertices.
\item Isotropic position yields $\tilde O(d^2 \sqrt{n})$ regret for any $K$ and can be computed relatively efficiently when $K$ is given by a separation or membership oracle.
The bound improves to $\tilde O(d^{1.9} \sqrt{n})$ when $K$ is symmetric.
\item John's position yields $\tilde O(d^2 \sqrt{n})$ regret for any $K$ and can be computed efficiently when $K$ is a polytope represented by the intersection of half-spaces.
When $K$ is symmetric, the regret improves to $\tilde O(d^{1.5} \sqrt{n})$.
\item When the polar of $K$ is isotropic, then the regret is $\tilde O(d^{1.75} \sqrt{n})$ in general and $\tilde O(d^{1.5} \sqrt{n})$ when $K$ is symmetric. 
Positioning $K$ such that the polar is isotropic is quite delicate and is discussed in detail in \ref{note:ons:iso} below.
\end{itemizeinner}

\item \label{note:ons:johns} No scaling of John's position yields the same uniform bound on $M(K)$ as L\"owner's position.\index{convex body}
Let $K$ be such that $K^\circ = \ball_1 \cap \{u \colon u_1 \geq -1/d\}$, which is the convex body formed as the intersection of the unit ball and a half-space.
\cref{thm:ellipsoid} shows that $K^\circ$ is in L\"owner's position and therefore $K$ is in John's position.
But $M(K) \geq \frac{1}{2}$ is obvious and yet $-d e_1 \in K$ holds since $\sup_{u \in K^\circ} \ip{-d e_1, u} = 1$.
Therefore $\max_{x \in K} \norm{x} \geq d$.
Since for $\gamma > 0$, $M(\gamma K) = \frac{1}{\gamma} M(K)$,
\begin{align*}
M(\gamma K) \max_{x \in \gamma K} \norm{x} \geq \frac{d}{2} \text{ for all } \gamma \in (0, \infty) \,.
\end{align*}
Therefore every scaling of $K$ such that $\diam(K) = O(d)$ has $M(K) = \Omega(1)$.
This suggests the following problem:

\begin{exer}
\faStar \faStar \faStar \faBook \faQuestion \quad
Does there exist a polynomial-time algorithm for positioning the constraint set such that $\ball_1 \subset K \subset \ball_{2d}$
and $M(K) = \tilde O(d^{-1/2})$ when
\begin{itemizeinner}
\item $K = \{x \colon Ax \leq b\}$ is a polytope? \index{polytope}
\item $K$ is given by a separation or membership oracle?\index{separation oracle}
\end{itemizeinner}
\end{exer}

\newcommand{\cen}{\operatorname{cen}}
\newcommand{\cov}{\operatorname{cov}}
\item \label{note:ons:iso}
The existence and computation of an affine map $T$ such that $(TK)^\circ$ is isotropic is quite interesting.
Let 
\begin{align*}
\cen(K) &= \frac{1}{\vol(K)} \int_K x \d{x} \qquad \text{ and }  \\
\cov(K) &= \frac{1}{\vol(K)} \int_K (x - \cen(K))(x - \cen(K))^\top \d{x} \,.
\end{align*}
We want to find an affine map $T$ such that $\cen((TK)^\circ) = \zeros$ and $\cov((TK)^\circ) = \id$.
An elementary calculation shows that $(AK)^\circ = A^{-1} K^\circ$ for positive definite $A$.
The behaviour of $x \mapsto (K-x)^\circ$, however, is more complicated.
Nevertheless, there exists an $s \in \interior(K)$ called the Sant\'alo point such that
$\cen((K-s)^\circ) = \zeros$, as explained by \cite{schneider2013convex}. 
\cite{meyer1998santalo} show that $x \mapsto \vol((K-x)^\circ)$ is strictly convex on the interior of $K$ and the minimiser of this function is
the Sant\'alo point \citep{santalo1949invariante}.
Hence, letting $s$ be the Sant\'alo point of $K$, $A = \cov((K - s)^\circ)^{1/2}$ and $Tx = Ax - As$, it follows that $(TK)^\circ$ is isotropic.
Regarding computation, given $x \in \interior(K)$ it is possible in principle to estimate $\vol((K-x)^\circ)$ by sampling and hence use zeroth-order convex
optimisation to find the Sant\'alo point. Once $s$ has been found, the matrix $A$ can be estimated by sampling from $(K-s)^\circ$. 
A reasonable guess is that a suitable approximation of $T$ can be
found in polynomial time using this procedure, but the devil is in the details of the approximation errors:

\begin{exer}
\faStar \faStar \faStar \faBook \faQuestion \quad
Suppose that $K$ is a polytope or represented by a separation oracle. Does there exist a polynomial time algorithm to position $K$ so that $K^\circ$ is approximately isotropic
with errors small enough that the relevant results in \cref{tab:widths} hold (approximately)?
The problem is studied when $d = 2$ by \cite{kaiser1993santalo}.
\end{exer}

\item Online Newton step was originally designed for full-information\index{setting!full information} online convex optimisation \citep{HAK07}. 
Its use as an algorithm for driving bandit convex optimisation methods has been developed by a number of authors. 
\cite{SRN21,suggala2024second} studied the quadratic and near-quadratic settings, while \cite{LG23} considered Lipschitz convex functions in the unconstrained setting. 
The algorithm and analysis in this chapter follows the work by \cite{LFMV24}. 

\item You should be a little unhappy about the non-specific constants in \cref{thm:ons:bandit}.
How can you run the algorithm if the constants depend on universal constants and unspecified logarithmic factors?
The problem is that the theory is overly conservative. In practice both $\eta$ and $\lambda$ should be much larger than the theory would suggest, even if you
tracked the constants through the analysis quite carefully. This is quite a standard phenomenon, and usually not a problem. Here there is a little twist, however.
If you choose $\eta$ or $\lambda$ too large, then the algorithm can explode with non-negligible probability. For example, the covariance matrix might
stop being positive definite at some point, or $F_t$ could grow too large and the algorithm may move slowly relative to the regret suffered.
Hopefully this issue can be resolved in the future but for now you should be cautious.

\item \label{note:ons:noise}
An interesting question is whether or not \cref{alg:ons:bandit} can be adapted to exploit low-variance noise, possibly using the idea in \cref{sec:sc:stoch}.
The principal challenge is the huge range of the extension, which you could mitigate by assuming the loss is Lipschitz and using \cref{prop:reg:bandit-extension}.
Or, even better, by developing some new techniques:

\begin{exer}
\faStar \faStar \faStar \faQuestion \quad
Analyse the regret of \cref{alg:ons:bandit} or some modification thereof under \cref{ass:sgd-stoch}.
\end{exer}

\item  \label{note:ons:inc}
There are several ways to improve the computational complexity of \cref{alg:ons:bandit}. 
At the moment the complexity when $K$ is a polytope is $O(d^3 + m d^2)$
with the former term due to computing the eigenvalue decomposition of the covariance matrix and the latter from the projection (\cref{tab:comp}).
Note that $m \geq d+1$ is needed for $K$ to be a convex body, so the second term always dominates.

\begin{exer}
\label{ex:ons:inc}
\faStar \faStar \faBook \faQuestion \quad
Suppose that $\Sigma^{1/2}$ is stored in memory\index{memory} and $u \in \R^d$. 
\begin{enumerate}[labelindent=0pt]
\item Find a high-quality approximation of $(\Sigma^{-1} + uu^\top)^{-1/2}$ that can be computed using $\tilde O(d^2)$ arithmetic operations. \label{ex:ons:inc:inc} 
\item Show how to use \ref{ex:ons:inc:inc} to reduce the complexity of the Gaussian sampling in \cref{alg:ons:bandit} to $\tilde O(d^2)$ arithmetic operations per round.
\end{enumerate}
\end{exer}

For \ref{ex:ons:inc:inc} you may find the results by \cite{hale2008computing} useful.
The next exercise is more speculative:

\begin{exer}
\faStar \faStar \faStar \faQuestion \quad 
Modify \cref{alg:ons:bandit} by removing the projection onto $K_\eps$ and prove whether or not the regret bound in \cref{thm:ons:bandit} still holds.
We only used that $\mu_t \in K_\eps$ in the proof of \cref{lem:ons:bounds}, which uses \cref{lem:ons:mink}.
Intuitively, even without the projection the algorithm should keep $\mu_t$ close to $K_\eps$ since the extended loss grows rapidly outside $K_\eps$.
Alternatively, you may argue that the projections happen rarely, possibly after modifying the algorithm in some way.
\end{exer}

\end{enumeratenotes}

\chapter[Online Newton Step for Adversarial Losses]{Online Newton Step for Adversarial Losses\copynotice}\label{chap:ons-adv}

Because the optimistic\index{optimistic} Gaussian surrogate \index{surrogate loss!Gaussian} is only well-behaved on a shrinking ellipsoidal focus region,\index{focus region} algorithms that use it 
are most naturally analysed in the stochastic setting, where it is already a challenge to prove that the optimal
action remains in the focus region. In the adversarial setting\index{setting!adversarial} there is limited hope to ensure the optimal action in hindsight
remains in the focus region. The plan is to use a mechanism that detects when the minimiser leaves the focus region and restarts the algorithm.\index{restart}
This is combined with an argument that the regret is negative whenever a restart occurs.
The formal setting studied in this chapter is characterised by the following assumption:

\begin{assumption}\label{ass:ons-adv}
The following hold:
\begin{enumerate}
\item The losses are bounded: $f_t \in \cF_{\pb}$ for all $t$.
\item The constraint set is rounded: $\ball_1 \subset K \subset \ball_{2d}$. \index{rounded}
\end{enumerate}
\end{assumption}

This is the same assumption as in \cref{chap:ons} except that the setting is now adversarial.
The highlight is a computationally efficient algorithm and regret bound of $\tilde O(d^{2.5} \sqrt{n})$ under \cref{ass:ons-adv}.
As in \cref{chap:ons} we assume that $\delta \in (0,1)$ is a small user-defined constant that satisfies $\delta \leq \poly(1/d, 1/n)$ and 
\begin{align*}
L = C \log(1/\delta)
\end{align*}
where $C > 0$ is a universal constant.
The analysis in this chapter builds on and refers to the arguments in \cref{chap:ons}, which should be read first.

\section[Approximate Convex Minimisation]{Approximate Convex Minimisation ($\skippy$)}\index{gradient descent!for non-convex optimisation|(}
The version of online Newton step for adversarial convex bandits makes use of a subroutine for approximately minimising a nearly convex function.

\begin{assumption}\label{ass:ons-adv:sgd}
$K \subset \R^d$ is a convex body and $h \colon K \to \R$, $\hat h \colon K \to \R$ and $\hat h' \colon K \to \R^d$ are functions such that
\begin{enumerate}
\item $h$ is convex and differentiable; 
\item $|h(x) - \hat h(x)| \leq \eps_0$ for all $x \in K$; and
\item $\sip{h'(x) - \hat h'(x), x-y} \leq \eps_1$ for all $x, y \in K$.
\end{enumerate}
\end{assumption}

Note, in spite of the notation, there is no need for $\hat h'$ to be the gradient of $\hat h$.
The objective is to find a procedure that finds a near-minimiser of $\hat h$, which is an approximately convex function.
This can be accomplished in many ways, but the most straightforward idea is to use gradient descent with the `gradients' provided by $\hat h'$.

\begin{algorithm}[h!]
\begin{algcontents}
\begin{lstlisting}
args: $\eta > 0$, $A \in \pd$
let $x_1 \in K$
for $t = 1$ to $n$
  $x_{t+1} = \argmin_{x \in K} \snorm{x_t - \eta A \hat h'(x_t) - x}_{A^{-1}}$
return $\frac{1}{n} \sum_{t=1}^n x_t$
\end{lstlisting}
\caption{Approximate gradient descent}
\label{alg:ons-adv:sgd}
\end{algcontents}
\end{algorithm}

\begin{remark}
The matrix $A$ accepted as input by \cref{alg:ons-adv:sgd} corresponds to a change of coordinates, which you will discover is needed in our application because gradient descent is not
equivariant under coordinate changes.
\end{remark}

\begin{theorem}\label{thm:ons-adv:sgd}
Suppose that $\snorm{x - y}^2_{A^{-1}} \leq 1$ for all $x, y \in K$ and that $\snorm{\hat h'(x)}_A \leq G$ for all $x \in K$.
Then, under \cref{ass:ons-adv:sgd},
the output $y$ of \cref{alg:ons-adv:sgd} with $\eta = G / \sqrt{n}$ satisfies
\begin{align*}
\hat h(y) \leq \inf_{x \in K} \hat h(x) + \eps_1 + 2\eps_0 + \frac{G}{\sqrt{n}}\,.
\end{align*}
\end{theorem}

\begin{proof}
We follow the standard analysis of gradient descent. Let $x \in K$ be arbitrary. Then
\begin{align*}
\frac{1}{2} \snorm{x_{t+1} - x}^2_{A^{-1}} 
&\leq \frac{1}{2}\snorm{x_t - \eta A \hat h'(x_t) - x}^2_{A^{-1}} \\
&= \frac{1}{2} \snorm{x_t - x}^2_{A^{-1}} - \eta \sip{\hat h'(x_t), x_t - x} + \frac{\eta^2}{2} \snorm{\hat h'(x_t)}^2_A \\
&\leq \frac{1}{2} \snorm{x_t - x}^2_{A^{-1}} - \eta \sip{h'(x_t), x_t - x} + \eta \eps_1 + \frac{\eta^2}{2} \snorm{\hat h'(x_t)}^2_A \\ 
&\leq \frac{1}{2} \snorm{x_t - x}^2_{A^{-1}} - \eta \sip{h'(x_t), x_t - x} + \eta \eps_1 + \frac{G^2 \eta^2}{2}   \,.
\end{align*}
Summing and telescoping shows that
\begin{align}
\sum_{t=1}^n \left(h(x_t) - h(x)\right)
\tag*{$h$ convex}
&\leq \sum_{t=1}^n \ip{h'(x_t), x_t - x} \nonumber \\
&\leq n \eps_1 + \frac{n \eta G^2}{2} + \frac{1}{2 \eta} \nonumber \\
&= n \eps_1 + G \sqrt{n} \,. \label{eq:ons-adv:cvx-approx}
\end{align}
By the definition of the algorithm $y = \frac{1}{n} \sum_{t=1}^n x_t$ is the average of the iterates. 
Then, by convexity,
\begin{align*}
\hat h(y) - \hat h(x) 
\tag*{By assumption}
&\leq 2\eps_0 + h\left(y\right) - h(x) \\
\tag*{$h$ convex}
&\leq 2\eps_0 + \frac{1}{n} \sum_{t=1}^n \left(h(x_t) - h(x)\right) \\
\tag*{By \cref{eq:ons-adv:cvx-approx}}
&\leq 2\eps_0 + \eps_1 + \frac{G}{\sqrt{n}}\,.
\end{align*}
The result follows since $x \in K$ was arbitrary.
\end{proof}

\begin{corollary}\label{cor:ons-adv:sgd}
Under the same conditions as \cref{thm:ons-adv:sgd}, running \cref{alg:ons-adv:sgd} for $n = \ceil{\frac{G^2}{(2\eps_0 + \eps_1)^2}}$ iterations yields
a point $y$ such that
\begin{align*}
\hat h(y) \leq \inf_{x \in K} \hat h(x) + 4\eps_0 + 2 \eps_1\,.
\end{align*}
\end{corollary}

A theoretically faster but less practical solution is to use the ellipsoid method, as we explain in Note~\ref{note:ons-adv:ellipsoid}.

\index{gradient descent!for non-convex optimisation|)}

\section{Decaying Online Newton Step}\index{decaying online Newton step}
We introduce a modification of online Newton step that decays the covariance matrix.
As in \cref{sec:ons:ons}, let $(\hat q_t)_{t=1}^n$ be a sequence of quadratic functions and $(K_t)_{t=0}^n$ be a sequence of nonempty compact convex
sets such that $K_{t+1} \subset K_t$ for all $t$.
Decaying online Newton step produces a sequence of iterates $(\mu_t)$ and covariances $(\Sigma_t)$ such that $\mu_{t+1} \in K_t$.

\begin{algorithm}[h!]
\begin{algcontents}
\begin{lstlisting}
args: $\eta > 0$, $\mu_1 \in K_0$ and $\Sigma_1 \in \psd$
for $t = 1$ to $n$
  compute $\gamma_t \in (0,1]$ and $K_t \subset K_{t-1}$ in some way
  let $g_t = \hat q_t(\mu_t)$ and $H_t = \hat q_t''(\mu_t)$
  update $\Sigma_{t+1}^{-1} = \gamma_t \Sigma_t^{-1} + \eta H_t$
  update $\mu_{t+1} = \argmin_{\mu \in K_t} \norm{\mu - (\mu_t - \eta \Sigma_{t+1} g_t)}^2_{\Sigma_{t+1}^{-1}}$
\end{lstlisting}
\caption{Decaying online Newton step}
\label{alg:ons-adv:ons-decay}
\end{algcontents}
\end{algorithm}

\FloatBarrier

\begin{theorem}\label{thm:ons-adv:ons}
Suppose that \cref{alg:ons-adv:ons-decay} is run on a sequence of quadratics $(\hat q_t)_{t=1}^n$ and produces iterates $(\mu_t)_{t=1}^{n+1}$ and covariances $(\Sigma_t)_{t=1}^{n+1}$.
Then, provided that $\Sigma_{1}^{-1},\ldots,\Sigma_{n+1}^{-1} \in \pd$, for any $x \in K_n$
\begin{align*}
\frac{1}{2\eta} \norm{\mu_{n+1} - x}^2_{\Sigma_{n+1}^{-1}}
&\leq \frac{1}{2\eta} \norm{\mu_1 - x}^2_{\Sigma_1^{-1}} + \frac{\eta}{2} \sum_{t=1}^n \norm{g_t}^2_{\Sigma_{t+1}} - \frac{\Gamma_n(x)}{\eta} - \hqReg_n(x) \,,
\end{align*}
where $\hqReg_n(x) = \sum_{t=1}^n (\hat q_t(\mu_t) - \hat q_t(x))$ and $\Gamma_n(x) = \frac{1}{2} \sum_{t=1}^n (1 - \gamma_t) \snorm{x - \mu_t}_{\Sigma_t^{-1}}$.
\end{theorem}

\begin{remark}
The sets $(K_t)_{t=0}^n$ and decay factors $(\gamma_t)_{t=1}^n$ can be data-dependent. In our application $(K_t)$ will be defined as the intersection of ellipsoidal focus regions\index{focus region} on which the surrogate loss is well-behaved.
\end{remark}

Before the proof, let us say something about why the regret bound is useful.
What is new compared to the standard version of online Newton step (\cref{thm:ons}) is
the negative terms appearing on the right-hand side. These can be considerable when $x$ is far from the centre of
the ellipsoid $E(\mu_t, \Sigma_t)$. 
This is precisely the focus region where our surrogate is well-behaved, which means that once a comparator leaves the focus region the 
regret with respect to the estimated quadratic surrogates will be negative, at least when the parameters are tuned correctly.

\begin{proof}[Proof of \cref{thm:ons-adv:ons}]
Let $x \in K_n$.
By definition,
\begin{align*}
&\frac{1}{2} \norm{\mu_{t+1} - x}^2_{\Sigma_{t+1}^{-1}} 
\leq \frac{1}{2} \norm{\mu_t - \eta \Sigma_{t+1} g_t - x}^2_{\Sigma_{t+1}^{-1}} \\
&\quad= \frac{1}{2} \norm{\mu_t - x}^2_{\Sigma_{t+1}^{-1}} + \frac{\eta^2}{2} \norm{g_t}^2_{\Sigma_{t+1}} - \eta \ip{g_t, \mu_t - x} \\
&\quad= \frac{1}{2} \norm{\mu_t - x}^2_{\gamma_t \Sigma_t^{-1}} + \frac{\eta^2}{2} \norm{g_t}^2_{\Sigma_{t+1}} - \eta (\hat q_t(\mu_t) - \hat q_t(x)) \\
&\quad= \frac{1}{2} \norm{\mu_t - x}^2_{\Sigma_t^{-1}} + \frac{\eta^2}{2} \norm{g_t}^2_{\Sigma_{t+1}} - \eta (\hat q_t(\mu_t) - \hat q_t(x)) - \frac{1-\gamma_t}{2} \norm{\mu_t - x}^2_{\Sigma_t^{-1}}\,,
\end{align*}
where in the inequality we used the fact that $x \in K_n \subset K_t$ and the assumption that $\Sigma_{t+1}^{-1}$ is positive definite.
Summing, telescoping and rearranging completes the claim.
\end{proof}

\section{Regularity and Extensions}

Remember that $M(K) = \E[\pi(X)]$,
where $\pi$ is the Minkowski functional associated with $K$ and $X$ is uniformly distributed on $\sphere_1$.\index{mean width}
Under \cref{ass:ons-adv}, $\ball_1 \subset K$ and therefore $M(K) \leq 1$.
In contrast to \cref{chap:ons}, there is nothing to be gained here by improving $M(K)$ from $O(1)$ to $\tilde O(d^{-1/2})$.
Given an $\eps \in (0,1)$ to be tuned later,
the extension of the loss function $f_t$ is
\begin{align*}
e_t(x) = \pip(x) f_t\left(\frac{x}{\pip(x)}\right) + \frac{2 (\pip(x)-1)}{\eps}\,,
\end{align*}
where $\pip(x) = \max(1, \pi(x)/(1-\eps))$.
Remember that $K_\eps = (1 - \eps) K$.
As in \cref{chap:ons} we abuse notation by saying that the algorithm samples $X_t$ from $\cN(\mu_t, \Sigma_t)$, plays
$X_t / \pip(X_t)$ and observes 
\begin{align*}
Y_t = \pip(X_t)\left(f_t\left(\frac{X_t}{\pip(X_t)}\right) + \eps_t\right) + \frac{2(\pip(X_t)-1)}{\eps}\,.
\end{align*}
The surrogate loss in round $t$ is defined by
\begin{align*}
s_t(x) = \E_{t-1}\left[\left(1 - \frac{1}{\lambda}\right) e_t(X_t) + \frac{1}{\lambda} e_t((1 - \lambda)X_t + \lambda x)\right]
\end{align*}
and its quadratic approximation is $q_t(x) = \sip{s_t'(\mu_t), x - \mu_t} + \frac{1}{4} \norm{x - \mu_t}^2_{s''_t(\mu_t)}$.
Let $p_t$ be the density of $\cN(\mu_t, \Sigma_t)$. 
The surrogate and its gradient and Hessian are estimated by
\begin{align}
\hat s_t(x) &= \left(1 - \frac{1}{\lambda} + \frac{\bar r_t(x)}{\lambda}\right) Y_t \quad \text{and} \nonumber \\ 
\label{eq:ons-adv:grad-est}
\hat s_t'(x) &= \frac{\bar r_t(x) Y_t}{1-\lambda} \Sigma_t^{-1} \left(\frac{X_t - \lambda x}{1 - \lambda} - \mu_t\right) \quad \text{and} \\
\label{eq:ons-adv:hess-est}
\hat s_t''(x) &= \frac{\lambda \bar r_t(x) Y_t}{(1 - \lambda)^2} \left[\Sigma_t^{-1}\left(\frac{X_t - \lambda x}{1-\lambda} - \mu_t\right)\left(\frac{X_t - \lambda x}{1 - \lambda} - \mu_t\right)^\top \Sigma_t^{-1} - \Sigma_t^{-1}\right] 
\end{align}
with
\begin{align*}
\bar r_t(x) = \min\left(\frac{p_t\left(\frac{X_t - \lambda x}{1 - \lambda}\right)}{(1 - \lambda)^d p_t(X_t)},\, \exp(2)\right)\,.
\end{align*}
Finally,
$\hat q_t(x) = \sip{\hat s_t'(\mu_t), x - \mu_t} + \frac{1}{4} \norm{x - \mu_t}^2_{\hat s''_t(\mu_t)}$ is the estimate of the quadratic
surrogate. These are the same estimators that appear in \cref{chap:ons,chap:opt}.

\section{Algorithm}
The algorithm is a modification of \cref{alg:ons:bandit} for the stochastic setting.
The main differences are
that decaying online Newton step replaces the classical version and
a gadget is introduced to restart the algorithm if negative regret is detected.\index{restart}
The estimated regret with respect to the estimated cumulative surrogate losses is
$\hsReg_t(x) = \sum_{u=1}^t \left(\hat s_u(\mu_u) - \hat s_u(x)\right)$.

\begin{algorithm}[h!]
\begin{algcontents}
\begin{lstlisting}
args: $\eta$, $\lambda$, $\sigma^2$, $\eps$, $\gamma$, $\rho$
$\mu_1 = \zeros$, $\Sigma_1 = \sigma^2 \id$ and $K_0 = K_\eps$
for $t = 1$ to $n$
  compute $K_t = \left\{x \in K_{t-1} \colon \lambda \snorm{x - \mu_t}_{\Sigma_t^{-1}} \leq \frac{1}{\sqrt{2L}}\right\}$ 
  sample $X_t$ from $\cN(\mu_t, \Sigma_t)$ 
  observe $Y_t = \pip(X_t) \left(f_t\left(\frac{X_t}{\pip(X_t)}\right) + \eps_t\right) + \frac{2 (\pip(X_t)-1)}{\eps}$ 
  compute $g_t = \hat s_t'(\mu_t)$ and $H_t = \frac{1}{2} \hat s_t''(\mu_t)$ using $\cref{eq:ons-adv:grad-est}, \cref{eq:ons-adv:hess-est}$
  compute $\displaystyle z_{t-1} = \argmin_{z \in \R^d} \left[\Gamma_{t-1}(z) \triangleq \sum_{s=1}^{t-1} (1-\gamma_s) \snorm{z - \mu_s}^2_{\Sigma_s^{-1}}\right] \label{line:ons-adv:Gamma}$
  $\displaystyle \gamma_t = \begin{cases}
  1 & \text{if } \Gamma_{t-1}(z_{t-1}) \geq 3 \rho \\
  \gamma & \text{if }  \Sigma_t^{-1} \not\preceq \sum_{s=1}^{t-1} \sind(\gamma_s \neq 1) \Sigma_s^{-1} \\
  \gamma & \text{if } \snorm{\mu_t - z_{t-1}}^2_{\Sigma_t^{-1}} \geq \frac{1}{8 L \lambda^2} \\
  1 & \text{otherwise}
  \end{cases}$
  compute $\Sigma_{t+1}^{-1} = \gamma_t \Sigma_t^{-1} + \eta H_t$ 
  compute $\mu_{t+1} = \argmin_{\mu \in K_t} \norm{\mu - [\mu_t - \eta \Sigma_{t+1} g_t]}_{\Sigma_{t+1}^{-1}}$  $\label{line:ons-adv:mu}$
  find $\!y_t \in K_t$ such that $\eta \hsReg_t(y_t) \geq \max_{y \in K_t} \eta \hsReg_t(y) - \rho$ $\label{line:ons-adv:y}$
  if $\hsReg_t(y_t) \leq -2\rho$ then: restart
\end{lstlisting}
\caption{Online Newton step for adversarial convex bandits}
\label{alg:ons-adv:bandit}
\end{algcontents}
\end{algorithm}

\FloatBarrier

\subsubsection*{Computation}
\cref{alg:ons-adv:bandit} can be implemented in polynomial time provided that $K$ is suitably represented. The difficult steps are:
\begin{itemize}
\item Sampling from the Gaussian, which naively requires an eigenvalue decomposition of the covariance matrix.
\item The computation of $z_{t-1}$ in Line~\ref{line:ons-adv:Gamma}, which is an unconstrained convex quadratic minimisation problem and therefore has a closed-form solution:
\begin{align*}
z_{t-1} = \left(\sum_{s=1}^{t-1} (1 - \gamma_s) \Sigma_s^{-1}\right)^{-1} \sum_{s=1}^{t-1} (1 - \gamma_s) \Sigma_s^{-1} \mu_s \,,
\end{align*}
with the convention that $z_{t-1} = \zeros$ when $t = 1$.
\item The projection in Line~\ref{line:ons-adv:mu} depends on the representation of $K$. Note that the projection is onto $K_t$, which is the intersection of $K$ and the ellipsoidal
focus regions.\index{focus region} When $K$ is a polytope,\index{polytope} then this is a quadratic program\index{quadratic programming} and can be solved efficiently using interior point methods.\index{interior point methods} Note, however, there are $O(t)$ quadratic constraints due
to the intersecting ellipsoidal focus regions.\index{projection}
\item The most challenging problem is the non-convex\index{non-convex} problem in Line~\ref{line:ons-adv:y}, for which you can use \cref{alg:ons-adv:sgd} or the ellipsoid method
as explained in \cref{sec:ons-adv:approx}.
\end{itemize}

\subsubsection*{Explanation of the Parameters}
\cref{alg:ons-adv:bandit} has many tuning parameters, including two new ones compared to the stochastic setting.
The role of the parameters and their values up to logarithmic factors are given in \cref{tab:ons-adv:tuning}.

\begin{table}[h!]
\caption{Table of tuning constants. Constants and logarithmic factors are omitted, but are given in \cref{thm:ons-adv:bandit}.}
\label{tab:ons-adv:tuning}
\centering
\renewcommand{\arraystretch}{1.6}
\begin{tabular}{|lll|}
\hline
        & \textsc{parameter} & \textsc{approx.\ value} \\ \hline
$\eta$  & learning rate & $\sqrt{d/n}$ \\
$\lambda$ & smoothing parameter for surrogate & $\frac{1}{d^2}$ \\
$\sigma^2$ & defines the initial covariance & $\frac{1}{d}$ \\
$\eps$ & defines the set on which the extension is defined & $\sqrt{d^5/n}$ \\
$\gamma$ & determines the decay of the covariance & $1 - \frac{1}{d}$ \\
$\rho$ & margin that triggers restart condition & $d^3$ \\ \hline
\end{tabular}
\end{table}

\section{Analysis}

\newcommand{\xs}{x^{\scriptsize s}}
\newcommand{\yhs}{y^{\scriptsize \hat s}}
\newcommand{\xe}{x^{\scriptsize e}}
\newcommand{\ys}{y^{\scriptsize s}}
\newcommand{\xq}{x^{\scriptsize q}}
\newcommand{\yq}{y^{\scriptsize q}}

\begin{theorem}\label{thm:ons-adv:bandit}
Suppose \cref{alg:ons-adv:bandit} is run with parameters
\begin{align*}
\lambda &= \frac{1}{C d^2 L^3 } &
\eta &= \frac{\sqrt{d/n}}{CL} &
\gamma &= 2^{-\frac{1}{1+dL}} \\
\sigma^2 &= \frac{1}{dL^4} &
\rho &= 2C d^3 L^4 &
\eps &= \frac{d^{2.5} L^5}{\sqrt{n}} \,.
\end{align*}
Then, under \cref{ass:ons-adv}, with probability at least $1 - \delta$, the regret of \cref{alg:ons-adv:bandit} is bounded by
\begin{align*}
\Reg_n = O\left(d^{2.5} L^5 \sqrt{n}\right) \,.
\end{align*}
\end{theorem}

\begin{proof}
The argument largely follows the analysis in \cref{sec:ons:analysis}, but with several additional steps.
Most importantly, we need to show the following hold with high probability:
\begin{enumerate}
\item If the algorithm restarts, then the regret is negative.\index{restart}
\item If the minimiser of the surrogate moves outside the focus region,\index{focus region} then the algorithm restarts.
\end{enumerate}
Let $\bar H_t = \E_{t-1}[H_t]$ and define $\bar \Sigma_t^{-1}$ inductively by $\bar \Sigma_1^{-1} = \Sigma_1^{-1}$ and
\begin{align*}
\bar \Sigma_{t+1}^{-1} = \gamma_t \bar \Sigma_t^{-1} + \eta \bar H_t \,.
\end{align*}
Define
\begin{align*}
\xe_\tau &= \argmin_{x \in K_\eps} \sum_{t=1}^\tau e_t(x) &
\xs_\tau &= \argmin_{x \in K_\eps} \sum_{t=1}^\tau s_t(x) \\ 
\ys_\tau &= \argmin_{y \in K_\tau} \sum_{t=1}^\tau s_t(y) & 
\yhs_\tau &= \argmin_{y \in K_\tau} \sum_{t=1}^\tau \hat s_t(x) \,.
\end{align*}
The superscript indicates which functions are being minimised. Note also the different domains, with $\xe_\tau$ and $\xs_\tau$ in $K_\eps$ and $\ys_\tau$ and $\yhs_\tau$ in $K_\tau$.
As in \cref{sec:ons:analysis}, we define a stopping time.\index{stopping time}

\begin{definition}\label{def:ons-adv:tau}
Let $\tau$ be the first round when one of the following does \textit{not} hold:
\begin{enumerate}
\item $\xs_\tau \in K_{\tau+1}$.
\item $\Sigma_{\tau+1}$ is positive definite.\label{def:ons-adv:tau:P}
\item $\delta \id \preceq \frac{1}{2} \bar \Sigma_{\tau+1}^{-1} \preceq \Sigma_{\tau+1}^{-1} \preceq 2 \bar \Sigma_{\tau+1}^{-1} \preceq \frac{1}{\delta} \id$. \label{def:ons-adv:tau:S}
\item The algorithm does not restart at the end of round $\tau$.
\end{enumerate}
If the conditions hold for all rounds, then $\tau$ is defined to be $n$.
\end{definition}

Note that $\tau$ is a stopping time with respect to $(\sF_t)_{t=1}^n$ because $\xs_t$, $K_{t+1}$ and $\Sigma_{t+1}$ are $\sF_t$-measurable.\index{measurable}

\stepsection{Step 1: Regret relative to extension}
Let
\begin{align*}
\eReg_n(\xe_n) = \sum_{t=1}^n \left(\E_{t-1}[e_t(X_t)] - e_t(\xe_n)\right) \,.
\end{align*}
Repeating the analysis in the proof of \cref{thm:ons:bandit} shows that with probability at least $1 - \delta / 7$,
\begin{align}
\tag{\eventAzuma}
\Reg_n \leq n \eps + \sqrt{2n \log(7/\delta)} + \eReg_n(\xe_n) \,.
\label{eq:adv-ons:r-e}
\end{align}
Hence for the remainder we focus on bounding $\eReg_n(\xe_n)$ with high probability.

\stepsection{Step 2: Concentration}\index{concentration}

Define events
\begin{align*}
\tag{\eventNoise}
\eventNoise &= \left\{\max_{1 \leq t \leq \tau} |\eps_t| \leq \sqrt{\log(14 n/\delta)}\right\} \\ 
\tag{\eventPi}
\eventPi &= \left\{\max_{1 \leq t \leq \tau} \pi(X_t) \leq \sqrt{L} \right\} \\
\tag{\eventGauss}
\eventGauss &= \left\{\max_{1 \leq t \leq \tau} \norm{X_t - \mu_t}_{\Sigma_t^{-1}} \leq \sqrt{\frac{8d}{3} \log(14 n /\delta)}\right\} \,. 
\end{align*}
Repeating the arguments in the proof of \cref{thm:ons:bandit} and using \cref{tab:ons-adv:sigma} (\cref{sec:ons-adv:constraints}) 
shows that $\bbP(\eventNoise \cap \eventPi \cap \eventGauss) \geq 1 - 3\delta/7$.
The next lemma bounds the sum $\sum_{t=1}^\tau \E_{t-1}[Y_t^2]$. 

\begin{lemma}\label{lem:ons-adv:bounds}
Let $\Ymax = \max_{1 \leq t \leq \tau} \left(|Y_t| + \E_{t-1}[|Y_t|]\right)$.
On $\eventNoise \cap \eventPi \cap \eventGauss$ the following hold:
\begin{enumerate}
\item $\Ymax \leq \frac{L}{\eps}$. \label{lem:ons-adv:bounds:Ymax}
\item $\sum_{t=1}^\tau \E_{t-1}[Y_t^2] \leq 10n$. \label{lem:ons-adv:bounds:Y}
\end{enumerate}
\end{lemma}

\begin{proof}
See the proof of \cref{lem:ons:bounds} except that \cref{eq:ons:sigma} is replaced by \cref{eq:ons-adv:sigma} and
use \cref{tab:ons-adv:bounds}.
\end{proof}

We also need to bound $\sum_{t=1}^\tau Y_t^2$ with high probability.
Let $\eventY{}$ be the event defined by
\begin{align*}
\tag{\eventY}
\eventY = \left\{\sum_{t=1}^\tau Y_t^2 \leq 21n \right\} \,.
\end{align*}

\begin{lemma}\label{lem:ons-adv:eventY}
$\bbP(\eventY \cup (\eventNoise \cap \eventPi \cap \eventGauss)^c) \geq 1 - \delta/7$.
\end{lemma}

\begin{proof}
See the proof of \cref{lem:ons:eventY}.
\end{proof}

The last two events control the concentration of the estimated quadratic surrogate about its mean 
at the optimal point and the concentration of the Hessian estimates.
Let $\cnf = \frac{3 L^2 \sqrt{dn}}{\lambda}$ and $\eventQ$ be the event that all of the following hold for all $t \leq \tau$: 
\begin{enumerate}
\item $\left|\sum_{u=1}^t (\hat q_u(x) - q_u(x))\right| \leq \cnf$  for all $x \in K_t$.
\item $\left|\sum_{u=1}^t (\hat s_u(x) - s_u(x))\right| \leq \cnf$ for all $x \in K_t$.
\item $\left|\sum_{u=1}^t \sip{\hat s'_u(x) - s'_u(x), x - y}\right| \leq \cnf$ for all $x,y \in K_t$.
\end{enumerate}

\begin{lemma}
$\bbP(\eventQ \cup (\eventNoise \cap \eventPi \cap \eventGauss \cap \eventY)^c) \geq 1 - \delta/7$.
\end{lemma}

Note, the dimension-dependence in $\cnf$ appears because the concentration bound needs to hold uniformly for all $x \in K_\tau$, which is accomplished by a covering
argument and union bound. By contrast, in the stochastic setting\index{setting!stochastic} bounds of this kind were only needed at the minimiser of the loss.

\begin{proof}
Use \cref{prop:conc-q}\ref{prop:conc-q:uniform}, \cref{prop:conc-s}\ref{prop:conc-s:uniform} and
\cref{prop:opt:conc-gradient}\ref{prop:opt:conc-gradient:uniform} in combination with \cref{lem:ons-adv:bounds}.
\end{proof}

Comparing $\bar \Sigma_t$ and $\Sigma_t$ is slightly more delicate thanks to the decay.
Let $w_0 = 1$ and $w_t = \prod_{s=1}^t \gamma_s$.
A simple induction shows that
\begin{align*}
\Sigma_{t+1}^{-1} &= w_t \left[ \frac{1}{\sigma^2} \id + \eta \sum_{s=1}^t \frac{H_s}{w_s}\right] \triangleq w_t \left[\frac{1}{\sigma^2} \id + \eta S_t\right] 
\quad \text{and} \\
\bar \Sigma_{t+1}^{-1} &= w_t \left[ \frac{1}{\sigma^2} \id + \eta \sum_{s=1}^t \frac{\bar H_s}{w_s}\right] \triangleq w_t \left[\frac{1}{\sigma^2} \id + \eta \bar S_t\right]\,.
\end{align*}
The next lemma characterises the important properties of the weights $(w_t)_{t=1}^\tau$.
The proof is deferred to \cref{sec:ons-adv:decay}.

\begin{lemma}\label{lem:ons-adv:weights}
The following hold:
\begin{enumerate}
\item $\sum_{t=1}^\tau \sind(\gamma_t \neq 1) \leq 1 + dL$. \label{lem:ons-adv:weights:sum}
\item $w_t \in [1/2, 1]$ for all $t \leq \tau$. \label{lem:ons-adv:weights:bounded}
\item $w_t$ is $\sF_{t-1}$-measurable for all $t$. \label{lem:ons-adv:weights:measurable}\index{measurable}
\end{enumerate}
\end{lemma}

Finally, let $\eventS{}$ be the event that 
\begin{align*}
\tag{\eventS}
\eventS = \left\{-3\lambda L^2 \sqrt{dn} \bar \Sigma_\tau^{-1} \preceq \bar S_\tau - S_\tau \preceq 3 \lambda L^2 \sqrt{dn} \bar \Sigma_\tau^{-1} \right\}\,.
\end{align*}

\begin{lemma}
$\bbP(\eventS \cup (\eventNoise \cap \eventPi \cap \eventGauss \cap \eventY)^c) \geq 1 - \delta/7$.
\end{lemma}

\begin{proof}
By \cref{prop:conc-hessian-path} (and \cref{rem:conc-hessian-path} and \cref{lem:ons-adv:weights}) with $\Sigma^{-1} = \frac{3}{2} \bar \Sigma_\tau^{-1}$, with probability at least $1 - \delta$, 
\begin{align*}
&- \lambda L^2 \left[1 + \sqrt{d \sum_{t=1}^\tau \E_{t-1}[Y_t^2]} + d^2 \Ymax \right] \frac{3}{2} \bar \Sigma_\tau^{-1} 
\preceq \bar S_\tau - S_\tau \\
&\qquad\qquad\preceq \lambda L^2 \left[1 + \sqrt{d \sum_{t=1}^\tau \E_{t-1}[Y_t^2]} + d^2 \Ymax \right] \frac{3}{2} \bar \Sigma_\tau^{-1} \,.
\end{align*}
The claim now follows from \cref{lem:ons-adv:bounds}.
\end{proof}

Let $E = \eventAzuma \cap \eventNoise \cap \eventPi \cap \eventGauss \cap \eventY \cap \eventQ \cap \eventS$ 
be the intersection of all these high-probability events.
A union bound over the preceding lemmas shows that $\bbP(E) \geq 1 - \delta$.
For the remainder of the proof we bound the regret on $E$.

\stepsection{Step 3: Simple bounds}
We can now make some elementary conclusions that hold on the intersection of all the high-probability events outlined
in the previous step.
Repeating the calculation used to derive \cref{eq:ons:conc:g} but using
\cref{lem:ons-adv:bounds}\ref{lem:ons-adv:bounds:Y},
\begin{align}
\sum_{t=1}^\tau \norm{g_t}^2_{\Sigma_t}
&\leq d n L\,.
\label{eq:ons-adv:conc:g}
\end{align}
By the definition of \eventS{},
\begin{align*}
\Sigma_{\tau+1}^{-1} 
&= w_\tau\left[\Sigma_1^{-1} + \eta S_\tau\right] \\
&\preceq w_\tau\left[\Sigma_1^{-1} + \eta \bar S_\tau + 3\eta \lambda L^2 \sqrt{dn}\bar \Sigma_{\tau}^{-1}\right] \\
\tag*{by \cref{tab:ons-adv:Sigma-conf}}
&\preceq w_\tau\left[\Sigma_1^{-1} + \eta \bar S_\tau + \frac{1}{2} \bar \Sigma_{\tau}^{-1}\right] \\
&= \bar \Sigma_{\tau+1}^{-1} + \frac{w_\tau}{2} \bar \Sigma_\tau^{-1} \\
&\preceq \frac{3}{2} \bar \Sigma_{\tau+1}^{-1}\,,
\end{align*}
where the final inequality holds because
$\bar \Sigma_{\tau+1}^{-1} = \gamma_t \bar \Sigma_\tau^{-1} + \eta \bar H_\tau \succeq \gamma_\tau \bar \Sigma_\tau^{-1} \succeq w_\tau \bar \Sigma_\tau^{-1}$.
Similarly,
\begin{align*}
\Sigma_{\tau+1}^{-1} 
&\succeq w_\tau[\Sigma_1^{-1} + \eta \bar S_\tau] - \frac{w_\tau}{2} \bar \Sigma_{\tau}^{-1}
= \bar \Sigma_{\tau+1}^{-1} - \frac{w_\tau}{2} \bar \Sigma_\tau^{-1} 
\succeq \frac{1}{2} \bar \Sigma_{\tau+1}^{-1}\,.
\end{align*}
Combining shows that
\begin{align*}
\frac{1}{2} \bar \Sigma_{\tau+1}^{-1} \preceq \Sigma_{\tau+1}^{-1} \preceq \frac{3}{2} \bar \Sigma_{\tau+1}^{-1} \,.
\end{align*}
We also need to show that 
\begin{align}
2\delta \id \preceq \bar \Sigma_{\tau+1}^{-1} \preceq \frac{1}{2 \delta}\id\,,
\label{eq:ons-adv:Sigma-delta}
\end{align}
which follows from exactly the same argument as in the proof of \cref{thm:ons:bandit}.
Therefore both \cref{def:ons-adv:tau}\ref{def:ons-adv:tau:P} and \cref{def:ons-adv:tau}\ref{def:ons-adv:tau:S} hold and so the only way that 
$\tau \neq n$ is if $\xs_\tau \notin K_{\tau+1}$ or the algorithm restarts at the end of round $\tau$.
The map $t \mapsto \Sigma_t$ is nearly non-increasing in the following sense. Given $s \leq t \leq \tau$,
\begin{align}
\Sigma_t 
\preceq 2 \bar \Sigma_t 
&= \frac{2}{w_{t-1}} \left(\Sigma_1^{-1} + \eta \bar S_t\right)^{-1} \nonumber \\
&\preceq \frac{2}{w_{t-1}} \left(\Sigma_1^{-1} + \eta \bar S_s\right)^{-1} \nonumber \\
&= \frac{2 w_{s-1}}{w_{t-1}} \bar \Sigma_s \nonumber \\
&\preceq 4 \bar \Sigma_s 
\preceq 8 \Sigma_s \,.
\label{eq:ons-adv:sigma}
\end{align}
The decay of the inverse covariance matrix has an important implication.
Recall the definition of $(\Gamma_t)$ in Line~\ref{line:ons-adv:Gamma} of \cref{alg:ons-adv:bandit}.

\begin{lemma}\label{lem:ons-adv:bonus}
Suppose that $x \in \partial K_\tau$. Then $\Gamma_\tau(x) \geq 3\rho$.
\end{lemma}

The proof is deferred to \cref{sec:ons-adv:decay}.
We also need an elementary bound on the magnitude of the surrogate losses.

\begin{lemma}\label{lem:ons-adv:s-bound}
Suppose that $x \in K_\eps$. Then
$-\frac{10}{\lambda\eps} \leq s_t(x) \leq 1$.
\end{lemma}

\begin{exer}
\faStar \quad
Prove \cref{lem:ons-adv:s-bound} using the following steps:
\begin{enumerate}
\item For the upper bound, combine the properties of the extension (\cref{prop:reg:bandit-extension-eps}) and \cref{lem:basic}\ref{lem:basic:opt}.
\item For the lower bound, use the definition of $s_t$, non-negativity of $e_t$ and its definition, 
\cref{lem:ons:mink}\ref{lem:ons:mink:E} and the fact that $M(K) \sqrt{d \sigma^2} \leq 1$.
\end{enumerate}
\end{exer}

\solution{
The upper bound follows because $e_t = f_t$ on $K_\eps$ (\cref{prop:reg:bandit-extension-eps}) and since $s_t \leq e_t$ everywhere (\cref{lem:basic}\ref{lem:basic:opt}). Combining shows that
$s_t(x) \leq e_t(x) = f_t(x) \leq 1$ with the final inequality following from \cref{ass:ons-adv} that $f_t \in \cF_\pb$.
For the lower bound,
\begin{align*}
s_t(x) 
\tag*{Definition}
&= \E_{t-1}\left[\left(1 - \frac{1}{\lambda}\right) e_t(X_t) + \frac{1}{\lambda} e_t((1 - \lambda) X_t + \lambda x)\right] \\
&\geq -\frac{1}{\lambda} \E_{t-1}[e_t(X_t)] \\
\tag*{Definition}
&= -\frac{1}{\lambda} \E_{t-1}\left[\pip(X_t) \left(f_t\left(\frac{X_t}{\pip(X_t)}\right) + \eps_t\right) + \frac{2(\pip(X_t) - 1)}{\eps}\right] \\
\tag*{$\E_{t-1}[\eps_t|X_t] = 0$}
&= -\frac{1}{\lambda} \E_{t-1}\left[\pip(X_t) f_t\left(\frac{X_t}{\pip(X_t)}\right) + \frac{2(\pip(X_t) - 1)}{\eps}\right] \\
&\geq -\frac{1}{\lambda} \E_{t-1}\left[1 + \pi(X_t) + \frac{2\pi(X_t)}{\eps}\right] \\
&\geq -\frac{1}{\lambda \eps} \left(4 + 3M \sqrt{d \norm{\Sigma_t}}\right) \\
\tag*{by \cref{eq:ons-adv:sigma}}
&\geq -\frac{1}{\lambda \eps} \left(4 + 6M \sqrt{d \sigma^2}\right) \\
&\geq -\frac{10}{\lambda \eps} \,.
\end{align*}
where in the inequalities we used the fact that $e_t(y) \geq f_t(y/\pip(y)) \geq 0$ for all $y \in \R^d$, 
that $\pip(y) \leq 1 + \pi(y)$ and \cref{ass:ons-adv} that $f_t \in \cF_\pb$, 
\cref{lem:ons:mink}\ref{lem:ons:mink:E} and finally that $M \sqrt{d \sigma^2} \leq 1$.
}

\stepsection{Step 4: Trace/logdet inequalities}
Again, we repeat the corresponding argument in the proof of \cref{thm:ons:bandit}.
The argument is made slightly more complicated by the decaying covariance matrices.

\begin{lemma}\label{lem:ons-adv:logdet}
The following holds:
\begin{align*}
\frac{4}{\lambda} \sum_{t=1}^\tau\tr(\bar H_t \Sigma_t) \leq \frac{d L}{\lambda \eta} \,.
\end{align*}
\end{lemma}

\begin{proof}
By \cref{prop:opt:H-upper} and \cref{prop:reg:bandit-extension-eps}, for any $t \leq \tau$,
\begin{align}
\eta \norm{\Sigma_t^{1/2} q_t''(\mu_t) \Sigma_t^{1/2}}
\leq \frac{\eta \lambda \lip(e_t)}{4(1 - \lambda)} \sqrt{d \norm{\Sigma_t}} 
\leq \frac{\eta \lambda \sigma \sqrt{d}}{\eps (1 - \lambda)} 
\leq 1 \,,
\label{eq:ons-adv:spec}
\end{align}
where we used \cref{eq:ons-adv:sigma} to bound $\norm{\Sigma_t} \leq 4 \sigma^2$. 
The final inequality follows from \cref{tab:ons-adv:logdet}.
Note that
\begin{align*}
\bar \Sigma_{t+1}^{-1} = \gamma_t \bar \Sigma_t^{-1} + \eta \bar H_t \succeq \gamma_t \bar \Sigma_t^{-1} \,.
\end{align*}
By \cref{lem:tech:log-sub},
\begin{align}
\log \det(\id + \eta \bar H_t \bar \Sigma_t)
&= \log \det(\gamma_t \id + \eta \bar H_t \bar \Sigma_t + (1 - \gamma_t) \id) \nonumber \\
&\leq \log \det(\gamma_t \id + \eta \bar H_t \bar \Sigma_t) + \frac{d(1 - \gamma_t)}{\gamma_t} \nonumber \\
&= \log \det(\bar \Sigma_t \bar \Sigma_{t+1}^{-1}) + \frac{d(1 - \gamma_t)}{\gamma_t} \nonumber \\
&\leq \log \det(\bar \Sigma_t \bar \Sigma_{t+1}^{-1}) + \frac{\sind(\gamma_t \neq 1)}{\sqrt{L}}\,,
\label{eq:ons-adv:tr1}
\end{align}
where the final inequality follows because either $\gamma_t = 1$ or $\gamma = 2^{-\frac{1}{1 + dL}} \approx 1 - \frac{\log(2)}{1 + dL}$.
Combining \cref{eq:ons-adv:spec}, \cref{eq:ons-adv:tr1} and \cref{lem:tr-logdet},
\begin{align*}
\frac{4}{\lambda} \sum_{t=1}^\tau \tr(\bar H_t \Sigma_t) 
&\explana\leq \frac{16}{\lambda \eta} \sum_{t=1}^\tau \log\det\left(\id + \frac{\eta \bar H_t \Sigma_t}{2}\right)  \\ 
&\explana\leq \frac{16}{\lambda \eta} \sum_{t=1}^\tau \log \det\left(\id + \eta \bar H_t \bar \Sigma_t\right) \\
&\explana\leq \frac{16}{\lambda \eta \sqrt{L}} \sum_{t=1}^\tau \sind(\gamma_t \neq 1) + \frac{16}{\lambda \eta} \sum_{t=1}^\tau \log \det\left(\bar \Sigma_t^{\vphantom{-1}} \bar \Sigma_{t+1}^{-1}\right) 
\\
&\explana\leq \frac{dL}{2\lambda \eta} + \frac{16}{\lambda \eta} \log \det \left(\id + \eta \sigma^2 \bar S_\tau\right)  \\
&\explana\leq \frac{dL}{2\lambda \eta} + \frac{16}{\lambda \eta} \log \det \left(\id + \frac{\sigma^2 \id}{\delta}\right) \\
&\explana\leq \frac{dL}{\lambda \eta} \,, 
\end{align*}
where \explanr{} follows from \cref{lem:tr-logdet},
\explanr{} since $\Sigma_t \preceq 2 \bar \Sigma_t$,
\explanr{} by \cref{eq:ons-adv:tr1},
\explanr{} from \cref{lem:ons-adv:weights}\ref{lem:ons-adv:weights:sum} and by telescoping the sum of log-determinants, and
\explanr{} by bounding $\eta \bar S_\tau \preceq \frac{1}{w_\tau} \bar \Sigma_{\tau+1}^{-1} \preceq \frac{1}{\delta} \id$ using \cref{eq:ons-adv:Sigma-delta}
and \cref{lem:ons-adv:weights} to bound $w_\tau \geq 1/2$.
\explanr{} follows by naive simplification.
\end{proof}

A simple consequence is a bound on the regret relative to the extension in terms of the regret relative to the quadratic surrogates.

\begin{lemma}\label{lem:ons-adv:e-q}
The following hold:
\begin{enumerate}
\item $\eReg(\xe_\tau) \leq \sReg(\xs_\tau) + 1 + \frac{d L}{\eta\lambda}$. 
\item $\eReg(\xe_\tau) \leq \qReg(\xs_\tau) + 2 + \frac{d L}{\eta \lambda}$ whenever $\xs_\tau \in K_\tau$. 
\end{enumerate}
\end{lemma}

\begin{proof}
We have
\begin{align*}
\eReg(\xe_\tau)
\tag*{by \cref{prop:lower}}
&\leq \sReg(\xe_\tau) + \frac{2}{\lambda} \sum_{t=1}^\tau \tr(s_t''(\mu_t) \Sigma_t) + 1 \\
\tag*{by def.\ $\xs_\tau$}
&\leq \sReg(\xs_\tau) + \frac{2}{\lambda} \sum_{t=1}^\tau \tr(s_t''(\mu_t) \Sigma_t) + 1 \\
\tag*{by def.\ $\bar H_t$}
&= \sReg(\xs_\tau) + \frac{4}{\lambda} \sum_{t=1}^\tau \tr(\bar H_t \Sigma_t) + 1 \\
\tag*{by \cref{lem:ons-adv:logdet}}
&\leq \sReg(\xs_\tau) + \frac{dL}{\eta \lambda} + 1 \,. 
\end{align*}
The second part follows from \cref{prop:ons:q-lower}, the definition of $K_\tau$ and naively bounding constants.
\end{proof}

\stepsection{Next steps}
So far the analysis has largely followed that in the stochastic setting, but now there is a serious deviation.
Based on the arguments so far, it has been established that with high probability the only way $\tau \neq n$ is if $\xs_\tau \notin K_{\tau+1}$. 
While in the stochastic setting it was possible to prove that the minimiser of the loss stays in the focus region,\index{focus region} this is no longer the case.
Instead, it is necessary to consider a case-by-case analysis and handle the restarts:
\begin{itemize}
\item When $\xs_\tau \in K_\tau$ it can be shown that $\eReg_\tau(\xe_\tau)$ is small and simultaneously that $\xs_\tau \in K_{\tau+1}$ with the latter
showing that $\tau = n$.
\item When $\xs_\tau \notin K_\tau$, then the algorithm restarts at the end of round $\tau$.
\item Whenever the algorithm restarts, then the regret relative to the extension is negative.
\end{itemize}

\stepsection{Step 5: Regret}
Suppose that $\xs_\tau \in K_\tau$ and the algorithm has not restarted at the end of round $\tau$; then by \cref{thm:ons-adv:ons}
\begin{align*}
\frac{1}{2} \snorm{\xs_\tau - \mu_{\tau+1}}^2_{\Sigma_{\tau+1}^{-1}} 
&\explana\leq \frac{\snorm{\xs_\tau}^2}{2\sigma^2} + \frac{\eta^2}{2} \sum_{t=1}^\tau \snorm{g_t}^2_{\Sigma_{t+1}}  - \eta \hqReg(\xs_\tau) - \Gamma_\tau(\xs_\tau) \\
&\explana\leq \frac{2 d^2}{\sigma^2} + 4\eta^2 dnL - \eta \hqReg(\xs_\tau) - \Gamma_\tau(\xs_\tau) \\
&\explana\leq \frac{2 d^2}{\sigma^2} + 4\eta^2 dnL + \eta \cnf - \eta \qReg(\xs_\tau) - \Gamma_\tau(\xs_\tau) \\
&\explana\leq \frac{2 d^2}{\sigma^2} + 4\eta^2 dnL + \eta \cnf + 2 \eta + \frac{dL}{\lambda} - \eta \eReg(\xe_\tau)  \\
&\explana\leq \rho - \eta \eReg(\xe_\tau) \,,
\end{align*}
where \explanr{} follows from \cref{thm:ons-adv:ons} and the assumption that $\xs_\tau \in K_\tau$, and
\explanr{} by the assumption that $K \subset \ball_{2d}$ and by \cref{eq:ons-adv:conc:g} and \cref{eq:ons-adv:sigma} to bound 
$\Sigma_{t+1} \preceq 8 \Sigma_t$.
\explanr{} holds on \eventQ{}.
\explanr{} follows from \cref{lem:ons-adv:e-q} and because $\xs_\tau \in K_\tau$; and also
because $\Gamma_\tau(\xs_\tau) \geq 0$.
\explanr{} follows from the definition of the constants (\cref{tab:ons-adv:main}).
Rearranging shows that
\begin{align*}
\eReg(\xe_\tau) \leq \frac{\rho}{\eta} \,.
\end{align*}
Furthermore, since the algorithm has not restarted in round $\tau$ it holds that
\begin{align*}
&\frac{1}{2} \snorm{\xs_\tau - \mu_{\tau+1}}^2_{\Sigma_{\tau+1}^{-1}} 
\leq \frac{2d^2}{\sigma^2} + 4\eta^2 dnL + \eta \cnf - \eta \qReg_\tau(\xs_\tau) \\
\tag*{by \cref{prop:ons:q-lower}}
&\qquad\leq \frac{2d^2}{\sigma^2} + 4\eta^2 dnL + \eta \cnf + 1 - \eta \sReg_\tau(\xs_\tau) \\
&\qquad\leq \frac{2d^2}{\sigma^2} + 4\eta^2 dnL + \eta \cnf + 1 - \eta \sReg_\tau(y_\tau) \\
\tag*{on \eventQ{}}
&\qquad\leq \frac{2d^2}{\sigma^2} + 4\eta^2 dnL + 2 \eta \cnf + 1 - \eta \hsReg_\tau(y_\tau) \\
&\qquad\leq 3 \rho \leq \frac{1}{4L \lambda^2}\,, 
\end{align*}
where the second-last inequality follows because the algorithm did not restart, so that $\eta \hsReg_\tau(y_\tau) \geq -2\rho$.
The last inequality follows from \cref{tab:ons-adv:decay}.
Rearranging shows that
$\lambda \snorm{\xs_\tau - \mu_{\tau+1}}_{\Sigma_{\tau+1}^{-1}} \leq 1/\sqrt{2L}$,
which when combined with the assumption that $\xs_\tau \in K_\tau$ shows that $\xs_\tau \in K_{\tau+1}$.
Hence by \cref{def:ons-adv:tau}, $\tau = n$ and 
we have successfully bounded $\eReg(\xe_n) \leq \frac{\rho}{\eta}$.

\stepsection{Step 6: Restart analysis}\index{restart}
Suppose at the end of round $\tau$ that the algorithm restarts, which means that
\begin{align}
\eta \hsReg_\tau(\yhs_\tau) = \max_{y \in K_\tau} \eta \hsReg_\tau(y) \leq \eta \hsReg_\tau(y_\tau) + \rho \leq -\rho\,,
\label{eq:ons-adv:negative}
\end{align}
where in the first inequality we used the definition of $y_\tau$ in Line~\ref{line:ons-adv:y} of \cref{alg:ons-adv:bandit} 
and in the second we used the fact that a restart is triggered when $\eta \hsReg_\tau(y_\tau) \leq -2\rho$.
Then
\begin{align*}
\eta \eReg_\tau(\xe_\tau)
&\explana\leq 1 + \frac{d L}{\lambda} + \eta\sReg_\tau(\xe_\tau) \\
&\explana\leq 1 + \frac{d L}{\lambda} + \eta\sReg_\tau(\xs_\tau) \\
&\explana\leq 1 + \frac{d L}{\lambda} + \eta\sReg_\tau(\xs_{\tau-1}) + \eta\left(s_\tau(\xs_{\tau-1}) - s_\tau(\xs_\tau)\right) \\
&\explana\leq 1 + \frac{d L}{\lambda} + \eta \cnf + \eta\left(1 + \frac{10}{\eps\lambda}\right) + \eta \hsReg_\tau(\xs_{\tau-1}) \\
&\explana\leq 1 + \frac{d L}{\lambda} + \eta \cnf + \eta\left(1 + \frac{10}{\eps\lambda}\right) + \eta \hsReg_\tau(\yhs_\tau) \\
&\explana\leq \rho + \eta \hsReg_\tau(\yhs_\tau) \\
&\explana\leq 0 \,,
\end{align*}
where \explanr{} follows from \cref{lem:ons-adv:e-q},
\explanr{} by the definition of $\xs_\tau$; and
\explanr{} by the definition of $\xs_{\tau-1}$.
\explanr{} holds on \eventQ{} and by \cref{lem:ons-adv:s-bound}.
\explanr{} follows from the definition of $\tau$ so that $x^s_{\tau-1} \in K_\tau$ and
the definition of $\yhs_\tau$ as the maximiser of $\hsReg_\tau$ over $K_\tau$.
\explanr{} follows from the definition of the constants (\cref{tab:ons-adv:negative}) and 
\explanr{} from \cref{eq:ons-adv:negative}.
Hence, whenever the algorithm restarts the regret with respect to the extension is negative.
On the other hand, if $\xs_\tau$ leaves $K_\tau$, then $\ys_\tau \in \partial K_\tau$ and 
\begin{align*}
\norm{\ys_\tau - \mu_{\tau+1}}^2_{\Sigma_{\tau+1}^{-1}}
&\explana\leq \frac{2d^2}{\sigma^2} + 4\eta^2dnL+ \eta \cnf - \eta \qReg_\tau(\ys_\tau) - \Gamma_\tau(\ys_\tau) \\
&\explana\leq 1 + \frac{2d^2}{\sigma^2} + 4\eta^2 dnL + \eta \cnf - \eta \sReg_\tau(\ys_\tau) - \Gamma_\tau(\ys_\tau) \\
&\explana\leq 1 + \frac{2d^2}{\sigma^2} + 4\eta^2 dnL + \eta \cnf - \eta \sReg_\tau(\yhs_\tau) - \Gamma_\tau(\ys_\tau) \\
&\explana\leq 1 + \frac{2d^2}{\sigma^2} + 4\eta^2 dnL + 2 \eta \cnf - \eta \hsReg_\tau(\yhs_\tau) - \Gamma_\tau(\ys_\tau) \\
&\explana\leq 1 + \frac{2d^2}{\sigma^2} + 4\eta^2 dnL + 2 \eta \cnf - \eta \hsReg_\tau(\yhs_\tau) - 3 \rho \\
&\explana\leq -2\rho - \eta \hsReg_\tau(\yhs_\tau) \,,
\end{align*}
where \explanr{} follows from the same calculations as in step 4,
\explanr{} from \cref{prop:ons:q-lower}, and
\explanr{} by the definition of $\ys_\tau$ and $\yhs_\tau$.
\explanr{} holds on \eventQ{},
\explanr{} from \cref{lem:ons-adv:bonus} and
\explanr{} by the definition of the constants (\cref{tab:ons-adv:restart}).
Therefore $\eta \hsReg_\tau(\yhs_\tau) \leq -2\rho$. Hence, by Line~\ref{line:ons-adv:y} of \cref{alg:ons-adv:bandit},
\begin{align*}
\eta \hsReg_\tau(y_\tau) \leq \eta \hsReg_\tau(\yhs_\tau) \leq -2\rho
\end{align*}
and a restart is triggered.
\end{proof}

\section{Decay Analysis}\label{sec:ons-adv:decay}

The purpose of this section is to prove \cref{lem:ons-adv:weights,lem:ons-adv:bonus}, both of which are related to the decaying covariance matrix.
Recall that
\begin{align*}
z_{t-1} = \argmin_{z \in \R^d} \left(\Gamma_{t-1}(z) \triangleq \sum_{s=1}^{t-1} (1-\gamma_s) \snorm{z - \mu_s}^2_{\Sigma_s^{-1}}\right\} 
\end{align*}
and with $D_t = \sum_{s=1}^t \sind(\gamma_s \neq 1) \Sigma_s^{-1}$,
\begin{align*}
\gamma_t = \begin{cases}
  1 & \text{if } \Gamma_{t-1}(z_{t-1}) \geq 3 \rho \\
  \gamma & \text{if }  \Sigma_t^{-1} \not\preceq D_{t-1} \\ 
  \gamma & \text{if } \snorm{\mu_t - z_{t-1}}^2_{\Sigma_t^{-1}} \geq \frac{1}{8 L \lambda^2} \\
  1 & \text{otherwise}\,.
  \end{cases}
\end{align*}
Remember also that $w_t = \prod_{s=1}^t \gamma_s$.

\begin{proof}[Proof of \cref{lem:ons-adv:weights}]
We start with part~\ref{lem:ons-adv:weights:sum}, which is the only difficult part.
By definition, $\sind(\gamma_t \neq 1) = A_t + B_t$ where
\begin{align*}
A_t &= \sind\left(\Sigma_t^{-1} \not\preceq D_{t-1} \text{ and } \Gamma_{t-1}(z_{t-1}) < 3 \rho\right) \qquad \text{and}\\ 
B_t &= \sind\left(A_t = 0 \text{ and } \snorm{\mu_t - z_{t-1}}^2_{\Sigma_t^{-1}} \geq \frac{1}{8 L \lambda^2}\text{ and } \Gamma_{t-1}(z_{t-1}) < 3\rho\right) \,.
\end{align*}
By \cref{def:ons-adv:tau}, for $s \leq \tau$, $\Sigma_s^{-1}$ is positive definite
and therefore $D_{t-1} \preceq D_t$ for all $t \leq \tau$.
Furthermore, in rounds $t$ where $A_t = 1$ it holds that $D_t = D_{t-1} + \Sigma_t^{-1}$ and $\Sigma_t^{-1} \not \preceq D_{t-1}$, which means that
$\id \preceq D_{t-1}^{-1/2} D_t D_{t-1}^{-1/2} = \id + D_{t-1}^{-1/2} \Sigma_t^{-1} D_{t-1}^{-1/2} \not \preceq 2 \id$ and therefore
\begin{align*}
\log \det D_{t-1}^{-1} D_t = \log \det D_{t-1}^{-1/2} D_t D_{t-1}^{-1/2} \geq \log(2)\,.
\end{align*}
Hence,
\begin{align*}
\sum_{t=1}^\tau A_t \log(2) 
&\leq\sum_{t=1}^\tau \log \det D_{t-1}^{-1} D_t \\
&=\log \det D_1^{-1} D_\tau \\
&= \log \det\left(\Sigma_1 \sum_{t=1}^\tau \sind(\gamma_t \neq 1) \Sigma_t^{-1}\right) \\
&\leq d \log \left(\frac{n \sigma^2}{\delta}\right) \leq d L \log(2)\,,
\end{align*}
where the last inequality follows from \cref{def:ons-adv:tau}.
Rearranging shows that
\begin{align*}
\sum_{t=1}^\tau A_t \leq d L\,.
\end{align*}
Moving now to bound $\sum_{t=1}^\tau B_t$, recall that $\Gamma_t$ is a convex quadratic minimised at $z_t \in \R^d$.
Let $\Gamma_t^\star = \Gamma_t(z_t)$.
Note that $t \mapsto \Gamma_t^\star$ is non-decreasing by the definition of $\Gamma_t$.
A simple calculation shows that when $\gamma_t \neq 1$, then in rounds $t$ where $B_t = 1$, 
\begin{align*}
\Gamma_t^\star 
&\explana= \Gamma^\star_{t-1} + (1 - \gamma) \snorm{z_{t-1} - \mu_t}^2_{\Sigma_t^{-1} - \Sigma_t^{-1} D_t^{-1} \Sigma_t^{-1}} \\
&\explana\geq \Gamma^\star_{t-1} + \frac{1 - \gamma}{2} \snorm{z_{t-1} - \mu_t}^2_{\Sigma_t^{-1}} \\
&\explana\geq \Gamma^\star_{t-1} + \frac{1 - \gamma}{16 L \lambda^2} \\
&\explana\geq \Gamma^\star_{t-1} + 3 \rho \,,
\end{align*}
where in \explanr{} we used \cref{lem:tech:quadratic} with $A = \Sigma_t^{-1}$ and $B = D_{t-1}$.
\explanr{} follows because $B_t = 1$ implies that $D_t = D_{t-1} + \Sigma_t^{-1} \succeq 2 \Sigma_t^{-1}$,
which shows that
$\Sigma_t^{-1} D_t^{-1} \Sigma_t^{-1} \preceq \frac{1}{2} \Sigma_t^{-1}$.
\explanr{} follows from the definition of $B_t$ and \explanr{} from the definition of the constants (\cref{tab:ons-adv:decay}).
Since $B_t = 1$ can only happen if $\Gamma^\star_{t-1} < 3 \rho$, it follows that $\sum_{t=1}^\tau B_t \leq 1$.
Combining the two parts shows that $\sum_{t=1}^\tau \sind(\gamma_t \neq 1) = \sum_{t=1}^\tau A_t + \sum_{t=1}^\tau B_t \leq 1 + d L$, 
which establishes part~\ref{lem:ons-adv:weights:sum}.
Part~\ref{lem:ons-adv:weights:bounded} follows from part~\ref{lem:ons-adv:weights:sum} since
\begin{align*}
w_t &= \prod_{s=1}^t \gamma_s \geq \gamma^{1 + dL} = \frac{1}{2} \,.
\end{align*}
Measurability of the weights (part~\ref{lem:ons-adv:weights:measurable}) follows from the construction of the algorithm.
\end{proof}

\begin{proof}[Proof of \cref{lem:ons-adv:bonus}]
Let $x \notin K_\tau$. By the definition of $K_\tau$ there exists a $t \leq \tau$ such that
\begin{align}
\snorm{x - \mu_t}^2_{\Sigma_t^{-1}} \geq \frac{1}{2 L \lambda^2} \,.
\label{eq:ons-adv:decay:ellipsoid} 
\end{align}
Since $\Gamma_\tau(x) \geq \Gamma_t(x)$ it suffices to show that $\Gamma_t(x) \geq 3 \rho$.
Suppose that $\gamma_t = \gamma$; then 
\begin{align*}
\Gamma_t(x) = \sum_{s=1}^t (1 - \gamma_s) \snorm{x - \mu_s}^2_{\Sigma_s^{-1}}
\geq (1 - \gamma) \snorm{x - \mu_t}^2_{\Sigma_t^{-1}} 
\explana\geq \frac{1 - \gamma}{2 L \lambda^2} 
\explana\geq 3\rho\,,
\end{align*}
where \explanr{} follows from \cref{eq:ons-adv:decay:ellipsoid} and \explanr{} from the definition of the constants (\cref{tab:ons-adv:decay}).
For the remainder assume that $\gamma_t = 1$. According to the definition of $\gamma_t$ there are two ways this can happen.
On the one hand, if $\Gamma_{t-1}(z_{t-1}) \geq 3\rho$, then trivially $\Gamma_t(x) = \Gamma_{t-1}(x) \geq \Gamma_{t-1}(z_{t-1}) \geq 3\rho$.
On the other hand, if $\Gamma_{t-1}(z_{t-1}) < 3\rho$ and $\gamma_t = 1$, then
\begin{align}
\Sigma_t^{-1} &\preceq \sum_{s=1}^{t-1} \sind(\gamma_s \neq 1) \Sigma_s^{-1} \quad \text{ and } \label{eq:ons-adv:bonus:1} \\
\snorm{\mu_t - z_{t-1}}^2_{\Sigma_t^{-1}} &\leq \frac{1}{8 L \lambda^2} \,. \label{eq:ons-adv:bonus:2}
\end{align}
Therefore,
\begin{align*}
\Gamma_t(x) 
&= \Gamma_{t-1}(x) \\
&= \sum_{s=1}^{t-1} (1 - \gamma_s) \snorm{x - \mu_s}^2_{\Sigma_s^{-1}} \\
&= (1 - \gamma) \sum_{s=1}^{t-1} \sind(\gamma_s \neq 1) \snorm{x - \mu_s}^2_{\Sigma_s^{-1}} \\
&\explana\geq (1 - \gamma) \sum_{s=1}^{t-1} \sind(\gamma_s \neq 1) \left[\frac{1}{2} \snorm{x - z_{t-1}}^2_{\Sigma_s^{-1}} - \snorm{z_{t-1} - \mu_s}^2_{\Sigma_s^{-1}}\right] \\
&= \frac{1 - \gamma}{2} \sum_{s=1}^{t-1} \sind(\gamma_s \neq 1) \snorm{x - z_{t-1}}^2_{\Sigma_s^{-1}} - \Gamma_{t-1}(z_{t-1}) \\
&\explana\geq \frac{1 - \gamma}{2} \snorm{x - z_{t-1}}^2_{\Sigma_t^{-1}} - 3\rho \\
&\explana\geq \frac{1 - \gamma}{2} \left[\frac{1}{2} \snorm{x - \mu_t}^2_{\Sigma_t^{-1}} - \snorm{\mu_t - z_{t-1}}^2_{\Sigma_t^{-1}}\right] - 3\rho \\
&\explana\geq \frac{1 - \gamma}{16\lambda^2 L} - 3\rho \\
&\explana\geq 3 \rho \,,
\end{align*}
where \explanr{} follows from the inequality $\snorm{a + b}^2 \leq 2 \snorm{a}^2 + 2 \snorm{b}^2$,
\explanr{} by the assumption that $\Gamma_{t-1}(z_{t_1}) \leq 3 \rho$ and \cref{eq:ons-adv:bonus:1},
\explanr{} by the same inequality as \explanr{1}.
\explanr{} by \cref{eq:ons-adv:decay:ellipsoid} and \cref{eq:ons-adv:bonus:2} and 
\explanr{} by the definition of the constants (\cref{tab:ons-adv:decay}).
\end{proof}

\section{Approximate Optimisation}\label{sec:ons-adv:approx}

We need to explain how \cref{alg:ons-adv:bandit} might implement the optimisation problem in Line~\ref{line:ons-adv:y} to find
a point $x \in K_t$ such that
\begin{align}
\eta \hsReg_t(x) \geq \max_{x \in K_t} \eta \hsReg_t(x) - \rho\,,
\label{eq:ons-adv:approx}
\end{align}
which is equivalent to finding an $x \in K_t$ such that
\begin{align*}
\eta \sum_{u=1}^t \hat s_u(x) \leq \min_{y \in K_t} \eta \sum_{u=1}^t \hat s_u(y) + \rho\,.
\end{align*}
The plan is to use gradient descent (\cref{alg:ons-adv:sgd}). To this end, let $t \leq \tau$ be fixed for the remainder of the section and
\begin{align*}
\hat h(x) = \eta \sum_{u=1}^t \hat s_u(x) \quad \text{and} \quad 
h(x) = \eta \sum_{u=1}^t s_u(x) \quad \text{and} \quad
\hat h'(x) = \eta \sum_{u=1}^t \hat s'_u(x) \,.
\end{align*}
The following is needed in order to apply \cref{cor:ons-adv:sgd}:

\begin{lemma}\label{lem:ons-adv:approx}
Given any $t \leq \tau$, the following hold:
\begin{enumerate}
\item \textit{Bounded gradients:} $\max_{x \in K_t} \snorm{\hat h'(x)}_{\smash{\Sigma_t^{-1}}} = O(\poly(d, n))$. \label{lem:ons-adv:approx:bg}
\item \textit{Approximate values:} $\max_{x \in K_t} |h(x) - \hat h(x)| \leq \frac{\rho}{6}$. \label{lem:ons-adv:approx:av}
\item \textit{Approximate gradients:} $\max_{x, y \in K_t} \ip{h'(x) - \hat h'(x), x - y} \leq \frac{\rho}{6}$. \label{lem:ons-adv:approx:ag}
\end{enumerate}
\end{lemma}

\cref{lem:ons-adv:approx} when combined with \cref{cor:ons-adv:sgd} shows that 
\cref{alg:ons-adv:sgd} when run with $A = \Sigma_t^{-1}$, $K = K_t$, gradient function $\hat h'(x) = \eta \sum_{u=1}^t \hat s_u(x)$
and $n = O(\poly(n, d))$
returns a point $x$ satisfying \cref{eq:ons-adv:approx}.

\begin{exer}
\faStar \faStar \quad
Prove Lemma~\ref{lem:ons-adv:approx}.
You may find the entries in \cref{tab:ons-adv:opt} and \ref{tab:ons-adv:cnf} useful. 
\end{exer}

\solution{
Starting with part~\ref{lem:ons-adv:approx:bg},
\begin{align*}
G 
&= \max_{x \in K_t} \snorm{\hat h'(x)}_{\Sigma_t} \\
&= \eta \max_{x \in K_t} \norm{\sum_{u=1}^t \hat s'_u(x)}_{\Sigma_t}  \\
&\leq \eta  \max_{x \in K_t} \sum_{u=1}^t \snorm{\hat s'_u(x)}_{\Sigma_t} \,.
\end{align*}
For any $x \in K_t$ and $1 \leq u \leq t$,
\begin{align*}
\snorm{\hat s'_u(x)}_{\Sigma_t} 
&\leq 4 \snorm{\hat s'_u(x)}_{\Sigma_u} \\
&= 4 \left|\frac{\bar r_u(x) Y_u}{1-\lambda}\right| \norm{\frac{X_u - \lambda x}{1-\lambda} - \mu_u}_{\Sigma_u^{-1}} \\
&\leq 16 \exp(2) Y_u \left[\snorm{X_u - \mu_u}_{\Sigma_u^{-1}} + \lambda \snorm{x - \mu_u}_{\Sigma_u^{-1}}\right] \\ 
&\leq Y_u \sqrt{d L} \,.
\end{align*}
where the first inequality follows from \cref{eq:ons-adv:sigma},
the second follows from the definition of $\hat s'_u(x)$, 
the third since $\lambda \leq 1/2$ and $\bar r_u(x) \leq \exp(2)$ by definition.
The last inequality follows on \eventGauss{} and the assumption that $x \in K_t \subset K_u$.
Hence, on \eventY{},
\begin{align*}
G \leq \eta \sum_{u=1}^t Y_u \sqrt{d L} \leq \eta \sqrt{nd L \sum_{u=1}^t Y_t^2} \leq \eta n \sqrt{3 d L^3} = \tilde O(\poly(n, d)) \,.
\end{align*}
Moving now to part~\ref{lem:ons-adv:approx:av}, on \eventQ{}
\begin{align*}
\eta \left|\sum_{u=1}^t \left(s_u(x) - \hat s_u(x)\right)\right| \leq \eta \cnf = \frac{3 \eta L^2 \sqrt{dn}}{\lambda} \leq \frac{\rho}{6} \,.
\end{align*}
Lastly, for part~\ref{lem:ons-adv:approx:ag}, also on \eventQ{},
\begin{align*}
\eta \left|\sum_{u=1}^t \ip{s'_u(x) - \hat s'_u(x), y - x}\right| \leq \eta \cnf \leq \frac{\rho}{6} \,. 
\end{align*}
}

\section{Constraints}\label{sec:ons-adv:constraints}

As in \cref{chap:ons}, the analysis in this chapter depends on a complicated set of constraints on the parameters,
which are given in \cref{tab:ons-adv}.

\renewcommand{\arraystretch}{1.6}
\begin{table}[h!]
\caption{Constraints on the parameters used in the analysis of \cref{alg:ons-adv:bandit}}\label{tab:ons-adv}
\begin{tabular}{|Np{5cm}|}
\hline
\multicolumn{1}{|c}{} & \multicolumn{1}{p{5cm}|}{\textsc{constraint}} \\ \hline
\label{tab:ons-adv:sigma} & $\sigma \sqrt{d} \leq 1$ \\ 
\label{tab:ons-adv:bounds} & $\frac{40 d L^2}{\eps \eta \lambda} \leq 2n$ \\
\label{tab:ons-adv:main} & $\frac{2d^2}{\sigma^2} + 4\eta^2 dnL + \eta \cnf + 2 \eta + \frac{d L}{\lambda} \leq \rho$ \\
\label{tab:ons-adv:Sigma-conf} & $3 \eta \lambda L^2 \sqrt{dn} \leq \frac{1}{2}$ \\ 
\label{tab:ons-adv:negative} & $\frac{d L}{\lambda} + \eta \cnf + \eta\left(1 + \frac{10}{\eps \lambda}\right) \leq \rho$ \\
\label{tab:ons-adv:logdet} & $\frac{\eta \lambda \sigma \sqrt{d}}{\eps(1 - \lambda)} \leq 1$ \\
\label{tab:ons-adv:restart} & $1 + \frac{2d^2}{\sigma^2} + 4 \eta^2 dnL + 2 \eta \cnf \leq \rho$ \\
\label{tab:ons-adv:decay} & $\rho \leq \frac{1-\gamma}{96 L \lambda^2}$ \\
\label{tab:ons-adv:opt} & $\eta \cnf \leq \frac{\rho}{6}$ \\
\label{tab:ons-adv:cnf} & $\cnf = \frac{3 L^2 \sqrt{dn}}{\lambda}$ \\
\hline
\end{tabular}
\end{table}

You can check that the constraints are satisfied when 
\begin{align*}
\lambda &= \frac{2}{C d^2 L^3 } &
\eta &= \frac{\sqrt{d/n}}{CL} &
\gamma &= 2^{-\frac{1}{1+dL}} \\ 
\sigma^2 &= \frac{1}{d L^4} &
\rho &= C d^3 L^4 &
\eps &= \frac{d^{2.5} C^3 L^6}{\sqrt{n}} \,,
\end{align*}
where $C > 0$ is a suitably large absolute constant.

\begin{remark}
Let us comment on where there might be room for improvement and on the tightness of the choices of the parameters.
First, it seems nearly essential that $1 - \gamma = \tilde O(1/d)$. But
satisfying \cref{tab:ons-adv:main} ensures that $\frac{d}{\lambda} = \tilde O(\rho)$ and satisfying 
\cref{tab:ons-adv:decay} ensures that $(1 - \gamma) / \lambda^2 = \tilde \Omega(\rho)$,
which means that $\lambda = \tilde O(\frac{1 - \gamma}{d}) = \tilde O(1/d^2)$.
But $\cnf = \tilde \Omega(\frac{1}{\lambda} \sqrt{dn})$, so the regret with this choice of $\lambda$ is at least $\tilde \Omega(\cnf) = \tilde \Omega(d^{2.5} \sqrt{n})$, which is the rate achieved.
The $\sqrt{d}$ in $\cnf$ arises from a union bound that may be loose. If this could be improved to $\cnf = \frac{1}{\lambda} \sqrt{n}$, then the regret would become $d^2 \sqrt{n}$.
\end{remark}

\section{Notes}

\begin{enumeratenotes}
\item The algorithm and analysis here are refined versions of those proposed by \cite{LFMV24}.
The restarting has been used to handle adversarial losses by a number of authors \citep{HL16,BEL16,SRN21}.\index{restart} 
The gadget used here most closely resembles that by \cite{SRN21}. The main difference is the mechanisms for deciding
when to decay the inverse covariance. 
At a high level both decay the inverse covariance when the focus region\index{focus region} changes too much. They use an argument based on
reduction of volume, which is less computationally efficient than the more algebraic calculations used in \cref{alg:ons-adv:bandit}.

\item We mentioned that the ellipsoid method can replace gradient descent for approximately minimising a near-convex function. \index{ellipsoid method!for non-convex optimisation}
We adopt the notation and assumptions in \cref{thm:ons-adv:sgd}. Let us additionally assume access to a separation oracle\index{separation oracle} for $K$ and
extend $\hat h' \colon K \to \R^d$ to $\hat h' \colon \R^d \to \R^d$ by defining $\hat h'(x)$ to be the output of the separation oracle for $x \notin K$.
Let $E_1$ be an ellipsoid such that $K \subset E_1$ and define $(E_k)$ centred at $(x_k)$ inductively by
$E_{k+1} = \MVEE(E_k \cap \{x \colon \sip{\hat h'(x_k), x - x_k} \leq 0\})$.
Let $x \in K$ be arbitrary and suppose that $\hat h(x) \leq \hat h(x_k) - 2\eps_0 - \eps_1$. 
Then
\begin{align*}
\sip{\hat h'(x_k), x - x_k} 
&\leq \sip{h'(x_k), x - x_k} + \eps_1 \\
&\leq h(x) - h(x_k) + \eps_1 \\
&\leq \hat h(x) - \hat h(x_k) + 2\eps_0 + \eps_1 \\
&\leq 0\,,
\end{align*}
which means that $x \in K_{k+1}$.
Following the standard argument of the ellipsoid method shows that with $m = O(d^2 \log(G/\max(\eps_0, \eps_1))$ it holds that
\begin{align*}
\min_{k \leq m} \hat h(x_k) \leq \inf_{x \in K} \hat h(x) + 4 \eps_0 + 2 \eps_1 \,.
\end{align*}
\label{note:ons-adv:ellipsoid}
\item The running time per round of \cref{alg:ons-adv:bandit} depends polynomially on $t$. The reason is two-fold:
\texttt{(1)} the focus region $K_t$ has at least $t$ quadratic constraints, which means the projection in Line~\ref{line:ons-adv:mu} involves a large number
of constraints;
\texttt{(2)} approximately maximising the empirical regret in Line~\ref{line:ons-adv:y} requires storing and accumulating all previous data during the
approximate convex optimisation procedure.
The following exercise is quite speculative:

\begin{exer}
\faStar \faStar \faStar \faQuestion \quad
Suppose that $K$ is represented as a polytope\index{polytope} or via a separation oracle.
Modify \cref{alg:ons-adv:bandit} to have $O(\poly(d, \log(n)))$ running time per round. The following is a suggestion only: 
\begin{enumerate}
\item Show that the focus region can be updated only $\tilde O(d)$ times when the ellipsoid $E(\mu_t, \Sigma_t)$ changes dramatically.
\item Show that the optimisation procedure in Line~\ref{line:ons-adv:y} can be warm-started or implemented in a streaming fashion to reduce
the complexity per round.
\end{enumerate}
\end{exer}

\item As with \cref{alg:ons:bandit} from the previous chapter, the analysis of \cref{alg:ons-adv:bandit} relies on complex and moderately non-explicit parameters.
In principle you can calculate the constants explicitly, but in practice the resulting choices will be overly conservative. And the problem that poor approximations of the optimal
constants may lead to linear regret is even worse here than in \cref{chap:ons}, thanks to the additional constants that define the decay of the inverse covariance and the restart condition.
\end{enumeratenotes}

\chapter[Gaussian Optimistic Smoothing]{Gaussian Optimistic Smoothing\copynotice}\label{chap:opt}\index{optimistic}

The purpose of this chapter is to introduce and analyse the surrogate loss functions used in \cref{chap:ons,chap:ons-adv}. \index{surrogate loss!Gaussian}
The results are stated in as much generality as possible to facilitate their use in future applications.
In case you want a quick summary of the results, read this introductory section
for the basic definitions and then head directly to \cref{sec:summary}.

Suppose that $f \colon \R^d \to \R$ is convex and $X$ is a random vector in $\R^d$. 
We are interested in the problem of estimating the entire function $f$ from a single 
observation $Y = f(X) + \eps$ where $\E[\eps|X] = 0$ and $\E[\exp(\eps^2)|X] \leq 2$.
Given a parameter $\lambda \in (0,1)$, define the surrogate by
\begin{align}
s(x) = \E\left[\left(1 - \frac{1}{\lambda}\right) f(X) + \frac{1}{\lambda} f((1 - \lambda)X + \lambda x)\right]\,.
\label{eq:s}
\end{align}
A geometric intuition for this surrogate is shown in \cref{fig:gauss:int} while the surrogate loss itself is
plotted in \cref{fig:optimistic,fig:optimistic-lambda}.
We saw this surrogate in \cref{chap:ftrl,chap:ellipsoid} with $\lambda = 1/2$ and where $X$ was supported on an ellipsoid.
For the remainder we assume that the law of $X$ is Gaussian with mean $\mu$ and covariance $\Sigma$.
The density of $X$ with respect to the Lebesgue
measure is 
\begin{align*}
p(x) = \left(\frac{1}{2\pi}\right)^{d/2} \sqrt{\det \Sigma^{-1}} \exp\left(-\frac{1}{2} \snorm{x - \mu}^2_{\Sigma^{-1}}\right)\,. 
\end{align*}
We also make use of a quadratic approximation of the surrogate defined by
\begin{align*}
q(x) = \ip{s'(\mu), x - \mu} + \frac{1}{4} \snorm{x - \mu}^2_{s''(\mu)}\,,
\end{align*}
which is related to the second-order expansion of $s$ at $\mu$ except that the zeroth-order term is dropped 
and the leading constant of the quadratic term is $\frac{1}{4}$ rather than $\frac{1}{2}$.
Since we dropped the zeroth-order term you should not expect that $q(x) \approx s(x)$. Rather we will see that $q(x) - q(\mu)$ is comparable to $s(x) - s(\mu)$ for suitable
$x$.

\begin{figure}[h!]
\begin{tikzpicture}
\begin{axis}[xmin=-1.2,xmax=3.5,ymin=-1,ymax=13,clip=false,xtick=\empty,ytick=\empty,width=10cm,height=6cm]
\addplot+[domain={-1.2:3.5},mark=none,black] {x^2};
\addplot+[domain={-1.2:3.5},mark=none,black] {1 + (x + 1)};
\node at (axis cs:-1,-2.5) {$X$};
\node at (axis cs:2,-2.5) {$(1 - \lambda) X + \lambda x$};
\node at (axis cs:3,-2.5) {$x$};
\draw (axis cs:-1,-1) -- (axis cs:-1,1);
\draw (axis cs:2,-1) -- (axis cs:2,4);
\draw (axis cs:3,-1) -- (axis cs:3,5);
\node[anchor=west] (a) at (axis cs:-1.2,11) {$(1 - \frac{1}{\lambda}) f(X) + \frac{1}{\lambda} f((1 - \lambda) X + \lambda x)$};
\node[anchor=west] at (axis cs:3,12) {$f$};
\draw[thick] (a.east) edge[-latex,out=0,in=90,shorten >= 0.1cm] (axis cs:3,5);
\end{axis}
\end{tikzpicture}
\caption{The plot shows a lower bound on $f(x)$ obtained by the linear approximation of $f$ using points $X$ and $(1 - \lambda)X + \lambda x$.
The optimistic surrogate is obtained by averaging over all such approximations according to the law of $X$.}
\label{fig:gauss:int}
\end{figure}

\FloatBarrier

\subsubsection*{Assumptions and Logarithmic Factors}
Because the analysis is quite intricate and we are not so concerned about constants and logarithmic factors, we make the following
assumption:

\begin{assumption}\label{ass:gauss}
The following hold:
\begin{enumerate}
\item \textit{Convexity:} $f \colon \R^d \to \R$ is convex and $K$ is a convex body.
\item \textit{Gaussian iterates:} $X$ has law $\cN(\mu, \Sigma)$ and $\mu \in K$. 
\item \textit{Subgaussian responses:} $Y = f(X) + \eps$ with 
\begin{align*}
\E[\eps|X] = 0 \quad\text{ and }\quad \E[\exp(\eps^2)|X] \leq 2 \,.
\end{align*}
\item \textit{Boundedness:} 
$\delta \in (0,1)$ is a constant such that 
\begin{align*}
\max\left(d, \lip(f), \sup_{x \in K} |f(x)|, \snorm{\Sigma}, \snorm{\Sigma^{-1}}, 1/\lambda\right) \leq \frac{1}{\delta}\,;
\end{align*}
moreover, $\lambda \leq \frac{1}{d+1}$.
\end{enumerate}
\end{assumption}

We let $L$ be a logarithmic constant:
\begin{align*}
L = C \log(1/\delta) 
\end{align*}
where $C > 0$ is a large non-specified universal positive constant. 
We also let $(C_k)$ be a collection of $k$-dependent universal positive constants.

\begin{figure}[h!]
\centering
\includegraphics[width=0.75\textwidth]{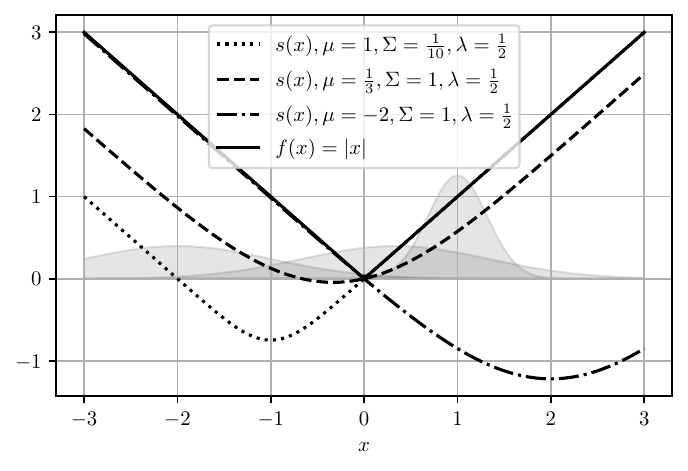}
\caption{
The surrogate for different choices of $\mu$ and $\Sigma$ with $\lambda = \frac{1}{2}$ in all cases.
Notice that the surrogate is always optimistic in the sense that $s(x) \leq f(x)$ for all $x$.
Moreover, the quality of the approximation depends on whether or not $x$ is in the region where the relevant Gaussian is well-concentrated
and the amount of curvature of $f$ in that region.\index{curvature}
}\label{fig:optimistic}
\commentAlt{A plot showing the absolute value loss function alongside the surrogate for different choices of lambda and sigma.}
\end{figure}

\begin{figure}[h!]
\centering
\includegraphics[width=0.75\textwidth]{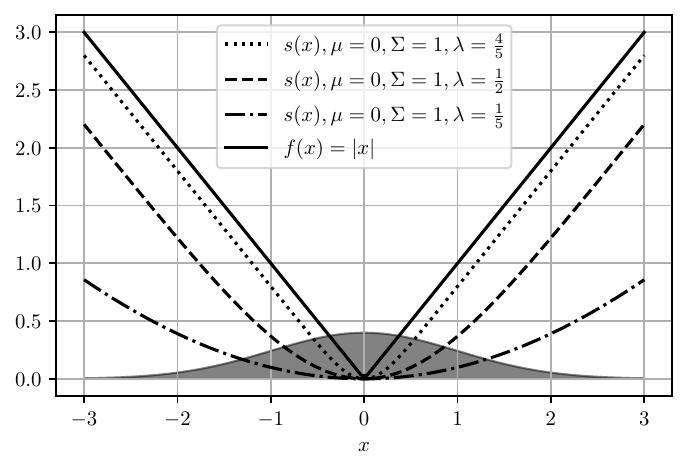}
\caption{
The surrogate for $\mu$ and $\Sigma$ constant and different choices of $\lambda$.
You can see that smaller $\lambda$ yields a smoother surrogate, but also one that has more approximation error.
}
\label{fig:optimistic-lambda}
\commentAlt{A plot showing the absolute value loss function alongside the surrogate with fixed mean and sigma and varying choices of lambda.
For smaller values of lambda the approximation is less good but the surrogate is smoother.}
\end{figure}

\FloatBarrier

\section{Smoothing}\label{sec:ons:smooth}
Our analysis would often be made considerably easier if $f \in \cF_\psm$.
Let $\varrho = \exp(-L/200)$, which is a miniscule constant, and let 
\begin{align*}
f_\varrho = f \star \phi_{\varrho}  
\end{align*}
where $\phi_{\varrho}$ is the smoothing kernel defined in \cref{sec:reg:smooth}.
The following lemma is nothing but a rewriting of \cref{prop:reg:smooth}.

\begin{lemma}\label{lem:opt:smooth}
The following hold:
\begin{enumerate}
\item $f_\varrho$ is twice differentiable and $\beta$-smooth with $\beta = (d+1)(d+6) \lip(f) / \varrho$.
\item $\snorm{f_\varrho - f}_\infty \leq \varrho \lip(f)$.
\end{enumerate}
\end{lemma}

The surrogate loss associated with the smoothed loss $f_\varrho$ is
\begin{align*}
s_{\varrho}(x) = \E\left[\left(1 - \frac{1}{\lambda}\right) f_\varrho(X) + \frac{1}{\lambda} f_\varrho((1 - \lambda)X + \lambda x)\right] \,.
\end{align*}
By definition, $\varrho$ is tiny, which means that $s_{\varrho}$ may not be \textit{that} smooth, but it is just enough for our purposes. 

\section{Elementary Properties}
An immediate consequence of the definitions is that $s$ is convex, Lipschitz and a lower bound on $f$.

\begin{lemma}\label{lem:basic}
The function $s$ in \cref{eq:s} is well-defined, infinitely differentiable and 
\begin{enumerate}
\item $s$ is convex; \label{lem:basic:cvx}
\item $s(x) \leq f(x)$ for all $x \in \R^d$; and \label{lem:basic:opt}
\item $\lip(s) \leq \lip(f)$. \label{lem:basic:lip}
\end{enumerate}
\end{lemma}

\begin{proof}
That $s$ is well-defined, infinitely differentiable and Lipschitz is left as an exercise. 
Part \ref{lem:basic:cvx} is immediate from the convexity of $f$. Part~\ref{lem:basic:opt} also uses convexity of $f$ and Jensen's inequality:\index{Jensen's inequality}
\begin{align*}
s(x) 
&= \E\left[\left(1 - \frac{1}{\lambda}\right) f(X) + \frac{1}{\lambda} f((1-\lambda)X + \lambda x)\right]  \\
&\leq \E\left[\left(1 - \frac{1}{\lambda}\right) f(X) + \frac{1}{\lambda} \left[(1-\lambda) f(X) + \lambda f(x)\right]\right] \\
&= f(x) \,.
\qedhere
\end{align*}
\end{proof}

\begin{exer}
\faStar \quad
Prove the omitted parts of \cref{lem:basic}.
\end{exer}

Perhaps the most important property of $s$ is that it is not too far below $f$ on an ellipsoidal region about $\mu$. Establishing this is quite involved, however,
and relies on a better understanding of the Hessian of $s$.

\section{Properties of the Hessian}

The next important property is a kind of continuity of the Hessian.

\begin{proposition}\label{prop:s-sc}
If $\lambda \norm{x - y}_{\Sigma^{-1}} \leq L^{-1/2}$, then $s''(x) \preceq 2 s''(y) + \delta \Sigma^{-1}$.
\end{proposition}

\begin{proof}
The interesting part is to establish a version of the claim for the smoothed surrogate loss, which is followed by a mundane comparison.

\stepsection{Step 1: Smoothed analysis}
Let $\eps = \frac{\lambda(y-x)}{1-\lambda}$ and assume by changing coordinates that $\mu = \zeros$.
By definition,
\begin{align*}
s_\varrho''(x) 
&= \lambda \E\left[f_\varrho''((1 - \lambda) X + \lambda x)\right] \\
&= \lambda \int_{\R^d} f_\varrho''((1 - \lambda) z + \lambda x) p(z) \d{z} \\ 
&= \lambda \int_{\R^d} f_\varrho''((1 - \lambda) w + \lambda y) p\left(w + \eps\right) \d{w} \,, 
\end{align*}
where the exchange of integral and derivatives is justified by the assumption that $\lip(f) < \infty$.
The last equality holds by a change of coordinates.
Given a set $B \subset \R^d$ let
\begin{align*}
I(B) &= \lambda \int_{B} f_\varrho''((1 - \lambda) w + \lambda y) p\left(w + \eps\right) \d{w} \,.
\end{align*}
The plan is to construct a set $A$ for which $\bbP(X \notin A)$ is negligible and $I(A) \preceq 2s''_\varrho(y)$ and
then argue that $I(A^c)$ is negligible.
Consider the density ratio
\begin{align*}
\frac{p(w + \eps)}{p(w)}
&= \exp\left(-\frac{1}{2} \norm{w + \eps}^2_{\Sigma^{-1}} + \frac{1}{2} \norm{w}^2_{\Sigma^{-1}}\right) \\
&= \exp\left(-\frac{1}{2} \norm{\eps}^2_{\Sigma^{-1}} - \ip{w, \eps}_{\Sigma^{-1}}\right) \,.
\end{align*}
Next, let $A = \left\{w \colon -\frac{1}{2} \norm{\eps}^2_{\Sigma^{-1}} - \ip{w, \eps}_{\Sigma^{-1}} \leq \log(2) \right\}$,
which is chosen so that
\begin{align*}
I(A) 
&=\lambda \int_A f_\varrho''((1 - \lambda) w + \lambda y) p(w + \eps) \d{w} \\ 
&\preceq 2 \lambda \int_A f_\varrho''((1 - \lambda)w + \lambda y) p(w) \d{w} \\
&\preceq 2 \lambda \int_{\R^d} f_\varrho''((1 -\lambda)w + \lambda y) p(w) \d{w} \\
&= 2 s_\varrho''(y)\,,
\end{align*}
where the first inequality uses the fact that $f_\varrho$ is convex so that $f_\varrho'' \succeq \zeros$ and the definition of $A$.
The second follows from convexity of $f_\varrho$.
Moving now to bound the integral over $A^c$, recall the definition of $\beta$ in \cref{lem:opt:smooth}. By convexity of the spectral norm,
\begin{align*}
\norm{I(A^c)}
&=\norm{\lambda \int_{A^c} f_\varrho''((1 - \lambda) w + \lambda y) p(w+\eps) \d{w}} \\
&\leq \lambda \beta \int_{A^c} p(w+\eps) \d{w} \\
&= \lambda \beta \bbP(X - \eps \in A^c) \\
&= \lambda \beta \bbP\left(-\frac{1}{2} \norm{\eps}^2_{\Sigma^{-1}} - \ip{X - \eps, \Sigma^{-1} \eps} > \log(2)\right) \\
&\leq \lambda \beta \exp\left(-\frac{\left(\log(2) - \frac{1}{2} \norm{\eps}^2_{\Sigma^{-1}}\right)^2}{2\norm{\eps}^2_{\Sigma^{-1}}}\right) \,,
\end{align*}
where in the final inequality we used \cref{thm:conc:gaussian} and the fact that $\ip{X, \Sigma^{-1} \eps}$ has law $\cN(0, \snorm{\eps}^2_{\Sigma^{-1}})$,
as well as the fact that $\frac{1}{2} \snorm{\eps}^2_{\Sigma^{-1}} \leq \log(2)$, which holds for suitably large $L$ by the assumptions in the proposition statement.
Therefore,
\begin{align}
s_\varrho''(x)
&\preceq 2 s_\varrho''(y) + \lambda \beta \exp\left(-\frac{\left(\log(2) - \frac{1}{2} \norm{\eps}^2_{\Sigma^{-1}}\right)^2}{2\norm{\eps}^2_{\Sigma^{-1}}}\right) \id \nonumber \\
&\preceq 2 s_\varrho''(y) + \lambda \beta \exp\left(-\frac{L}{100}\right) \id \,, \label{eq:gauss:hess-1}
\end{align}
where the last inequality follows because $\lambda \leq 1/2$ by \cref{ass:gauss} and the conditions in the statement that $\lambda \norm{x - y}_{\Sigma^{-1}} \leq 1/ \sqrt{L}$, so that
\begin{align*}
\norm{\eps}^2_{\Sigma^{-1}} = \left(\frac{\lambda}{1-\lambda}\right)^2 \norm{x - y}^2_{\Sigma^{-1}} \leq \frac{1}{L(1-\lambda)^2} \leq \frac{4}{L} \leq \log(2) \,.
\end{align*}

\stepsection{Step 2: Comparison}
We now compare $s''$ and $s''_\varrho$. Let
\begin{align*}
M(z) = \Sigma^{-1/2} zz^\top \Sigma^{-1/2} - \id \,.
\end{align*}
Then, using convexity of the spectral norm and \cref{lem:opt:smooth},
\begin{align}
\norm{\Sigma^{1/2} (s''(x) - s''_\varrho(x) ) \Sigma^{1/2}}
&\explana= \frac{\lambda}{(1 - \lambda)^2} \norm{\int_{\R^d} (f - f_\varrho) ((1 - \lambda) z + \lambda x) M(z) p(z) \d{z}} \nonumber \\
&\explana\leq \frac{\lambda \varrho \lip(f)}{(1 - \lambda)^2} \int_{\R^d} \norm{M(z)} p(z) \d{z} \nonumber \\
&\explana= \frac{\lambda \varrho \lip(f)}{(1 - \lambda)^2} \int_{\R^d} \norm{\Sigma^{-1/2} zz^\top \Sigma^{-1/2} - \id} p(z) \d{z} \nonumber \\
&\explana\leq \frac{(d+1) \lambda \varrho \lip(f)}{(1 - \lambda)^2} \,,  \label{eq:gauss:hess-2}
\end{align}
where \explanr{} follows by (twice) integrating by parts, \explanr{} by \cref{lem:opt:smooth}, \explanr{} by substituting the definition of $M$
and \explanr{} since $\Sigma^{-1/2} z$ under $p(z)$ is a standard Gaussian and for $W \sim \cN(\zeros, \id)$, $\E[\snorm{WW^\top - \id}] \leq \E[\snorm{WW^\top}] + 1 = \E[\snorm{W}^2] + 1 = d+1$.
Therefore,
\begin{align*}
s''(x)
&\explana\preceq s''_\varrho(x) + \frac{(d+1) \lambda \varrho \lip(f)}{(1 - \lambda)^2} \Sigma^{-1} \\
&\explana\preceq 2 s''_\varrho(y) + \lambda \beta \exp\left(-\frac{L}{100}\right) \id + \frac{(d+1) \lambda \varrho \lip(f)}{(1 - \lambda)^2} \Sigma^{-1} \\
&\explana\preceq 2 s''(y) + \lambda \beta \exp\left(-\frac{L}{100}\right) \id + \frac{2(d+1) \lambda \varrho \lip(f)}{(1 - \lambda)^2} \Sigma^{-1} \\
&\explana\preceq 2 s''(y) + \delta \Sigma^{-1} \,,
\end{align*}
where \explanr{} follows from \cref{eq:gauss:hess-2},
\explanr{} from \cref{eq:gauss:hess-1} and
\explanr{} from \cref{eq:gauss:hess-2} again. Lastly,
\explanr{}
follows from the definitions of $\varrho = \exp(-L/200)$ and $\beta = (d+1)(d+6) \lip(f) / \varrho$ in \cref{sec:ons:smooth} 
and \cref{ass:gauss} that $\delta \id \preceq \Sigma^{-1}$.
\end{proof}

Since $f$ is Lipschitz and $s$ is a smoothing of $f$, the Hessian of $s$ cannot be too large relative to $\snorm{\Sigma^{-1}}$,
as the next proposition shows.

\begin{proposition}\label{prop:opt:H-upper}
For any $z \in \R^d$:
\begin{enumerate}
\item $\norm{s''(z)} \leq \frac{\lambda \lip(f)}{1 - \lambda} \sqrt{d \snorm{\Sigma^{-1}}}$;
\item $\norm{\Sigma^{1/2} s''(z) \Sigma^{1/2}} \leq \frac{\lambda \lip(f)}{1 - \lambda} \sqrt{d \norm{\Sigma}}$.
\end{enumerate}
\end{proposition}

\begin{proof}
Assume for a moment that $f$ is twice differentiable. Then, exchanging derivatives and the expectation and integrating by parts shows that
for any $\eta \in \sphere_1$,
\begin{align*}
\eta^\top s''(z)\eta 
&= \lambda \E[\eta^\top f''((1 - \lambda)X + \lambda z) \eta] \\
&= \frac{\lambda}{1 - \lambda} \E\left[\sip{\eta, f'((1 - \lambda)X + \lambda z)} \sip{\eta, \Sigma^{-1}(X -\mu)}\right] \\
&\leq \frac{\lambda \lip(f)}{1-\lambda} \E\left[\snorm{\Sigma^{-1}(X - \mu)}\right] \\
&\leq \frac{\lambda \lip(f)}{1-\lambda} \snorm{\Sigma^{-1}}^{1/2} \E\left[\snorm{X - \mu}_{\Sigma^{-1}}\right] \\
&\leq \frac{\lambda \lip(f)}{1 - \lambda} \sqrt{d \snorm{\Sigma^{-1}}} \,.
\end{align*}
Then use the fact that $\norm{s''(z)} = \max_{\eta \in \sphere_1} \eta^\top s''(z) \eta$.
The second part follows from the same argument.
In case $f$ is not twice-differentiable, apply the above argument to $s_\varrho$ and $f_\varrho$ and pass to the limit as $\varrho \to 0$. Alternatively, use
direct means to justify the second equality above with $\ip{\cdot, f'(\cdot)}$ replaced with the directional derivative $Df(\cdot)[\cdot]$.
The second part follows from the same argument and is left as an exercise.
\end{proof}

\begin{exer}
\faStar \quad
Prove the second part of \cref{prop:opt:H-upper}.
\end{exer}

Lastly, we compare the Hessian of the surrogate to the mean Hessian of the loss $f$.

\begin{proposition}\label{prop:gauss:mean-hess}
Suppose that $\lambda \leq \frac{1}{dL^2}$. Then 
\begin{align*}
\lim_{\varrho \to 0} \E[f''_\varrho(X)] \preceq \frac{1}{\lambda} \left[2 s''(\mu) + 2 \delta \Sigma^{-1}\right] \,.
\end{align*}
\end{proposition}

Note, the limit of the smoothing is used because $f$ may not be twice differentiable.
This corresponds to viewing the Hessian of $f$ as an operator on suitable distributions.

\begin{proof}
Let $Z$ have law $\cN(\mu, \frac{2-\lambda}{\lambda} \Sigma)$, which is chosen so that
$(1 - \lambda) X + \lambda Z$ has the same law as $X$. 
Therefore,
\begin{align*}
\E[s_\varrho''(Z)]
= \lambda \E[f_\varrho''((1 - \lambda)X + \lambda Z)]
= \lambda \E[f_\varrho''(X)]\,.
\end{align*}
Passing to the limit shows that
\begin{align*}
\E[s''(Z)] = \lambda \lim_{\varrho \to 0} \E[f_\varrho''(X)] \,.
\end{align*}
Define event $A = \{\lambda \snorm{Z - \mu}_{\Sigma^{-1}} \leq L^{-1/2}\}$.
By \cref{prop:s-sc}, 
\begin{align}
\sind_A s''(Z) \preceq 2 s''(\mu) + \delta \Sigma^{-1}\,.
\label{eq:gauss:dist-1}
\end{align}
By the definition of $Z$, $\sqrt{\frac{\lambda}{2-\lambda}} \Sigma^{-1/2}(Z - \mu)$ has law $\cN(\zeros, \id)$.
Therefore, by \cref{prop:orlicz} and \cref{lem:orlicz-tail},
\begin{align*}
\bbP\left(A^c\right)
&= \bbP\left(\lambda \norm{Z - \mu}_{\Sigma^{-1}} > L^{-1/2}\right) \\ 
&= \bbP\left(\norm{\sqrt{\frac{\lambda}{2-\lambda}} \Sigma^{-1/2}(Z - \mu)}^2 > \frac{1}{\lambda(2 - \lambda)L}\right) \\
&\leq 2\exp\left(-\frac{3}{8d  \lambda(2 - \lambda)L}\right) \\
\tag*{since $\lambda \leq \frac{1}{dL^2}$}
&\leq 2\exp\left(-\frac{3L}{16}\right)\,. 
\end{align*}
Combining the above display with \cref{eq:gauss:dist-1} and \cref{prop:opt:H-upper} shows that
\begin{align*}
\E[s''(Z)] 
&= \E[\sind_A s''(Z) + \sind_{A^c}s''(Z)] \\
&\preceq 2 s''(\mu) + \delta \Sigma^{-1} + \bbP(A^c) \frac{\lambda \lip(f)}{1 - \lambda} \sqrt{d \snorm{\Sigma^{-1}}} \id \\
&\preceq 2 s''(\mu) + \delta \Sigma^{-1} + 2 \exp\left(-\frac{3L}{16}\right) \frac{\lambda \lip(f)}{1 - \lambda} \sqrt{d \snorm{\Sigma^{-1}}} \id \\
&\preceq 2 s''(\mu) + 2\delta \Sigma^{-1} \,,
\end{align*}
where the final inequality follows from the definition of $L$ and by \cref{ass:gauss}.
\end{proof}

\section{Properties of the Quadratic Surrogate}
Recall that the quadratic surrogate is
\begin{align*}
q(x) = \ip{s'(\mu), x - \mu} + \frac{1}{4} \norm{x - \mu}^2_{s''(\mu)} \,.
\end{align*}
Obviously $q$ inherits convexity from $s$. 
By \cref{prop:s-sc}, $s$ has a nearly constant Hessian in a region about $\mu$ from which it follows
that $q(x) - q(\mu) \lesssim s(x) - s(\mu)$ on a region about $\mu$, as the following proposition shows:

\begin{proposition}\label{prop:ons:q-lower}
Suppose that $\lambda \snorm{x - \mu}_{\Sigma^{-1}} \leq \frac{1}{\sqrt{L}}$; then 
\begin{align*}
s(\mu) - s(x) \leq q(\mu) - q(x) + \frac{\delta}{\lambda^2} \,.
\end{align*}
\end{proposition}

\begin{proof}
By \cref{prop:s-sc}, for any $y \in [\mu, x]$, 
\begin{align*}
s''(y) \succeq \frac{1}{2} \left[s''(\mu) - \delta \Sigma^{-1}\right]\,.
\end{align*}
By Taylor's theorem there exists a $y \in [\mu, x]$ such that
\begin{align*}
s(x) &= s(\mu) + \ip{s'(\mu), x - \mu} + \frac{1}{2} \norm{x - \mu}^2_{s''(y)}  \\
&\geq s(\mu) + \ip{s'(\mu), x - \mu} + \frac{1}{4} \norm{x - \mu}^2_{s''(\mu)} - \frac{\delta}{4} \norm{x - \mu}^2_{\Sigma^{-1}} \\
&\geq s(\mu) + \ip{s'(\mu), x - \mu} + \frac{1}{4} \norm{x - \mu}^2_{s''(\mu)} - \frac{\delta}{4\lambda^2 L} \\
&\geq s(\mu) + q(x) - \frac{\delta}{\lambda^2} \,.
\end{align*}
The result follows by rearranging and because $q(\mu) = 0$.
\end{proof}

\section{Lower Bound on the Surrogate}
We have shown that $s \leq f$ holds everywhere. 
In general there is no uniform upper bound on the entire function $f - s$, but
$f(\mu) - s(\mu)$ can be upper-bounded in terms of the curvature of $s$ as $\mu$.\index{curvature}

\begin{proposition}\label{prop:lower}
Provided that $\lambda \leq \frac{1}{d L^2}$,
\begin{align*}
f(\mu) \leq \E[f(X)] \leq s(\mu) + \frac{2}{\lambda} \tr(s''(\mu) \Sigma) + \frac{2 \delta d}{\lambda} \,.
\end{align*}
\end{proposition}

\begin{proof}
The first inequality is immediate from Jensen's inequality\index{Jensen's inequality} and because $\E[X] = \mu$.
Let $Z$ be a random variable that is independent of $X$ and has law $\cN(\mu, \rho^2 \Sigma)$
where $\rho^2 = \frac{2 - \lambda}{\lambda}$ is chosen so that $(1-\lambda)^2 + \lambda^2 \rho^2 = 1$.
Then
\begin{align*}
\E[s(Z)]
&= \E\left[\left(1 - \frac{1}{\lambda}\right) f(X) + \frac{1}{\lambda} f((1 - \lambda)X+\lambda Z)\right] \\
&= \E\left[\left(1 - \frac{1}{\lambda}\right) f(X) + \frac{1}{\lambda} f(X)\right] \\
&= \E\left[f(X)\right]\,,
\end{align*}
where we used the fact that $(1 - \lambda) X + \lambda Z$ has the same law as $X$.
Let us now compare $\E[s(Z)]$ to $s(\mu)$. By Taylor's theorem, for every $z \in \R^d$ there exists a $\xi_z \in [\mu, z]$ such that
$s(z) = s(\mu) + s'(\mu)^\top (z - \mu) + \frac{1}{2} \snorm{z - \mu}^2_{s''(\xi_z)}$.
By \cref{prop:s-sc}, if $z$ is close enough to $\mu$, then $s''(z)$ is close to $s''(\mu)$.
Define
\begin{align*}
A &=\left\{z \in \R^d \colon \lambda \norm{z - \mu}_{\Sigma^{-1}} \leq L^{-1/2} \right\} \,.
\end{align*}
Note that $\mu \in A$ and that $A$ is convex. 
Hence, if $z \in A$, then $\xi_z \in A$ and by \cref{prop:s-sc}, $s''(\xi_z) \leq 2 s''(\mu) + \delta \Sigma^{-1}$.
Then
\begin{align*}
\E[s(Z)] 
&= \E\left[s(\mu) + s'(\mu)^\top (Z - \mu) + \frac{1}{2} \norm{Z - \mu}^2_{s''(\xi_Z)}\right] \\
&= s(\mu) + \underbracket{\E\left[\frac{\sind_A(Z)}{2} \norm{Z - \mu}^2_{s''(\xi_Z)}\right]}_{D} + \underbracket{\E\left[\frac{\sind_{A^c}(Z)}{2}\norm{Z - \mu}^2_{s''(\xi_Z)}\right]}_{E} \,.
\end{align*}
The dominant term $D$ is bounded using \cref{prop:s-sc} and the definition of $A$:
\begin{align*}
D 
&= \E\left[\frac{\sind_A(Z)}{2} \norm{Z - \mu}^2_{s''(\xi_Z)}\right] \\
\tag*{by \cref{prop:s-sc}}
&\leq \E\left[\norm{Z - \mu}^2_{s''(\mu)} + \frac{\delta}{2} \norm{Z - \mu}^2_{\Sigma^{-1}}\right] \\
&= \tr\left(s''(\mu) \E\left[(Z - \mu)(Z - \mu)^\top\right]\right) + \frac{\delta}{2} \E\left[\norm{Z - \mu}^2_{\Sigma^{-1}}\right] \\
&= \rho^2 \tr\left(s''(\mu) \Sigma\right) + \frac{\delta d \rho^2}{2} \\
\tag*{since $\rho^2 \leq \frac{2}{\lambda}$}
&\leq \frac{2}{\lambda} \tr\left(s''(\mu) \Sigma\right) + \frac{\delta d}{\lambda} \,.
\end{align*}
Collecting the results shows that
\begin{align*}
\E[f(X)] = \E[s(Z)] \leq s(\mu) + \frac{2}{\lambda} \tr\left(s''(\mu) \Sigma\right) + \frac{\delta d}{\lambda} + E \,.
\end{align*}
All that remains is to bound the error term, which follows by showing that $Z \in A^c$ holds with vanishingly small probability.

\stepsection{Bounding the error term}
Let
\begin{align}
M = \sup_{z \in \R^d} \norm{\Sigma^{1/2} s''(z) \Sigma^{1/2}} 
\leq \frac{\lambda \lip(f)}{1 - \lambda} \sqrt{\frac{d}{\delta}} \,,
\label{eq:gauss:lower:1}
\end{align}
where the inequality follows from \cref{prop:opt:H-upper} and the assumption that $\norm{\Sigma} \leq \frac{1}{\delta}$.
Let $W$ have law $\cN(\zeros, \id)$ and note that $\frac{1}{\rho} \Sigma^{-1/2}(Z - \mu)$ also has law $\cN(\zeros, \id)$.
Then, 
\begin{align}
E
&=\frac{1}{2} \E\left[\norm{Z - \mu}^2_{s''(\xi_Z)} \sind_{A^c}(Z)\right] \nonumber \\
&\leq \frac{M}{2} \E\left[\norm{Z - \mu}^2_{\Sigma^{-1}} \sind_{A^c}(Z)\right] \nonumber \\
&= \frac{M \rho^2}{2} \E\left[\norm{W}^2 \sind_{A^c}(Z)\right] \nonumber \\
&\leq \frac{M \rho^2}{2} \sqrt{\E\left[\norm{W}^4_{\Sigma^{-1}}\right] \bbP(Z \notin A)} \nonumber \\
&= \frac{M \rho^2}{2} \sqrt{(d^2+2d) \bbP\left(\norm{W}^2 \geq \frac{1}{\lambda^2 \rho^2 L} \right)} \nonumber \\
&\leq M d \rho^2 \sqrt{\bbP\left(\norm{W}^2 \geq \frac{1}{\lambda^2 \rho^2 L}\right)} \,,
\label{eq:gauss:lower:2}
\end{align}
where we used the definition of $M$, Cauchy--Schwarz and \cref{prop:gaussian}.
By \cref{prop:orlicz}, $\norm{\norm{W}^2}_{\psi_1} \leq 3d$ and by \cref{lem:orlicz-tail} and the assumption that $\lambda \leq \frac{1}{dL^2}$,
\begin{align*}
\bbP\left(\norm{W}^2 \geq \frac{1}{\lambda^2 \rho^2 L}\right)
&\leq 2\exp\left(-\frac{1}{3d \lambda^2 \rho^2 \logs}\right)
\leq 2\exp\left(-\frac{1}{6d \lambda L}\right) 
\leq 2\exp\left(-\frac{L}{6}\right)\,.
\end{align*}
Combining the above with \cref{eq:gauss:lower:2} and \cref{eq:gauss:lower:1} and naively simplifying the constants by ensuring $L$ is large enough
shows that $E \leq \frac{\delta d}{\lambda}$.
\end{proof}

\begin{corollary}\label{cor:gauss:lower}
Suppose that $\lambda \leq \frac{1}{d L^2}$ and $\lambda \norm{x - \mu}_{\Sigma^{-1}} \leq \frac{1}{\sqrt{L}}$. Then
\begin{align*}
\E[f(X)] - f(x) \leq q(\mu) - q(x) + \frac{2}{\lambda} \tr(s''(\mu) \Sigma) + \delta\left[\frac{2d}{\lambda} + \frac{1}{\lambda^2}\right] \,.
\end{align*}
\end{corollary}

\begin{proof}
By \cref{lem:basic}\ref{lem:basic:opt} and \cref{prop:lower} and \cref{prop:ons:q-lower},
\begin{align*}
\E[f(X)] - f(x) 
&\leq s(\mu) - s(x) + \frac{2}{\lambda} \tr(s''(\mu) \Sigma) + \frac{2 \delta d}{\lambda} \\ 
&\leq q(\mu) - q(x) + \frac{2}{\lambda} \tr(s''(\mu) \Sigma) + \delta\left[\frac{2d}{\lambda} + \frac{1}{\lambda^2}\right] \,.
\qedhere
\end{align*}
\end{proof}

\section{Estimation}
The surrogate loss function can be estimated from $X$ and $Y$ using a change of measure.
Precisely,
\begin{align}
s(z) 
&= \int_{\R^d} \left(\left(1 - \frac{1}{\lambda}\right) f(x) + \frac{1}{\lambda} f((1 - \lambda) x + \lambda z)\right) p(x) \d{x} \nonumber \\ 
&= \int_{\R^d} \left(1 - \frac{1}{\lambda} + \frac{p\left(\frac{x - \lambda z}{1-\lambda}\right)}{\lambda (1 - \lambda)^d p(x)}\right) f(x) p(x) \d{x} \nonumber \\
&= \int_{\R^d} \left(1 - \frac{1}{\lambda} + \frac{r(x, z)}{\lambda}\right) f(x) p(x) \d{x}\,,
\label{eq:s:imp}
\end{align}
where $r(x, z)$ is the change of measure defined by
\begin{align}
r(x, z) = \left(\frac{1}{1-\lambda}\right)^d \frac{p\left(\frac{x - \lambda z}{1-\lambda}\right)}{p(x)} \,,
\label{eq:opt:measure}
\end{align}
which satisfies
\begin{align*}
\frac{\partial r(x, z)}{\partial z}  &= \frac{\lambda r(x, z)}{1-\lambda} \Sigma^{-1} \left(\frac{x - \lambda z}{1-\lambda} - \mu\right) \\
\frac{\partial^2 r(x,z)}{\partial z^2}  &= \frac{\lambda^2 r(x, z)}{(1 - \lambda)^2} \left[\Sigma^{-1}\left(\frac{X - \lambda z}{1 - \lambda} - \mu\right)\left(\frac{X - \lambda z}{1-\lambda} - \mu\right)^\top \Sigma^{-1} - \Sigma^{-1}\right] \,.
\end{align*}
Looking at \cref{eq:s:imp} and exchanging derivatives and expectations, we might estimate $s$ and its derivatives by
\begin{align*}
\tilde s(z) &= \left(1 - \frac{1}{\lambda} + \frac{r(X, z)}{\lambda}\right) Y \\
\tilde s'(z) &= \frac{r(X, z) Y}{1-\lambda} \Sigma^{-1} \left(\frac{X - \lambda z}{1-\lambda} - \mu\right) \\
\tilde s''(z) &= \frac{\lambda r(X, z) Y}{(1-\lambda)^2} \left[\Sigma^{-1}\left(\frac{X - \lambda z}{1-\lambda} - \mu\right)\left(\frac{X - \lambda z}{1 - \lambda} - \mu\right)^\top \Sigma^{-1} - \Sigma^{-1}\right]\,.
\end{align*}
And indeed, these are unbiased estimators of $s(z)$, $s'(z)$ and $s''(z)$, respectively.

\begin{exer}\label{ex:gauss:unbiased}
\faStar \quad
Show that $\E[\tilde s(z)] = s(z)$, $\E[\tilde s'(z)] = s'(z)$ and $\E[\tilde s''(z)] = s''(z)$ for all $z \in \R^d$.
\end{exer}

The quantity $r(X, z)$, however, is not especially well behaved and for this reason we let
$\bar r(x, z) = \min(\exp(2), r(x, z))$ and define estimators
\begin{align*}
\hat s(z) &= \left(1 - \frac{1}{\lambda} + \frac{\bar r(X, z)}{\lambda}\right) Y \\
\hat s'(z) &= \frac{\bar r(X, z) Y}{1-\lambda} \Sigma^{-1} \left(\frac{X - \lambda z}{1-\lambda} - \mu\right) \\
\hat s''(z) &= \frac{\lambda \bar r(X, z) Y}{(1-\lambda)^2} \left[\Sigma^{-1}\left(\frac{X - \lambda z}{1-\lambda} - \mu\right)\left(\frac{X - \lambda z}{1 - \lambda} - \mu\right)^\top \Sigma^{-1} - \Sigma^{-1}\right]\,.
\end{align*}

\begin{remark}
Our notation for these estimators is a little clumsy because $\hat s'(z)$ and $\hat s''(z)$ are not the derivatives of $\hat s(z)$.
\end{remark}

Note that while $s$ is convex, in general neither $x \mapsto \hat s(x)$ nor $x \mapsto \tilde s(x)$ are (see \cref{fig:s-est} in \cref{sec:conc-seq}).
Mostly we are interested in estimating gradients and Hessians of the surrogate at $\mu$, which satisfy
\begin{align*}
\hat s'(\mu) &= \frac{r(X, \mu) Y \Sigma^{-1}(X - \mu)}{(1 - \lambda)^2}  \,.\\
\hat s''(\mu) &= \frac{\lambda r(X, \mu) Y}{(1 - \lambda)^2}\left[\frac{\Sigma^{-1}(X - \mu)(X - \mu)^\top \Sigma^{-1}}{(1 - \lambda)^2} - \Sigma^{-1}\right] \,.
\end{align*}
Note that $r(x, \mu) = \bar r(x, \mu)$ for all $x$ thanks to \cref{lem:measure} in the next section.

\begin{proposition}\label{prop:opt:bias}
Provided that $\lambda \snorm{z - \mu}_{\Sigma^{-1}} \leq \sqrt{\frac{1}{L}}$,
the following hold:
\begin{enumerate}
\item $|\E[\hat s(z)] - s(z)| \leq \delta$; \label{prop:opt:bias:0}
\item $\norm{\E[\hat s'(z)] - s'(z)} \leq \delta$; \label{prop:opt:bias:1}
\item $\norm{\E[\hat s''(z)] - s''(z)} \leq \delta$. \label{prop:opt:bias:2}
\end{enumerate}
Moreover, $\E[\hat s(\mu)] = s(\mu)$, $\E[\hat s'(\mu)] = s'(\mu)$ and $\E[\hat s''(\mu)] = s''(\mu)$.
\end{proposition}

\begin{proof}
Let $E$ be the event that $r(X, z) > \exp(2)$. 
Then
\begin{align*}
|r(X, z) - \bar r(X, z)|
&\leq \sind_E r(X, z) \,.
\end{align*}
By \cref{lem:measure},
\begin{align*}
r(X, z) 
&\leq \exp\left(1 + \frac{\lambda}{(1-\lambda)^2} \ip{X - \mu, z - \mu}_{\Sigma^{-1}}\right) \,.
\end{align*}
Therefore, using the definition of $E$ and the fact that $\lambda \ip{X - \mu, z - \mu}_{\Sigma^{-1}}$ has law $\cN(\zeros, \lambda^2 \snorm{z - \mu}^2_{\Sigma^{-1}})$,
\begin{align*}
\bbP(E) 
&= \bbP(r(X, z) > \exp(2)) \\
&\leq \bbP\left(\lambda \ip{X - \mu, z - \mu}_{\Sigma^{-1}} > (1 - \lambda)^2\right) \\
\tag*{by \cref{thm:conc:gaussian}}
&\leq \exp\left(-\frac{(1 - \lambda)^4}{2 \lambda^2 \snorm{z - \mu}^2_{\Sigma^{-1}}}\right) \\
&\leq \exp\left(-\frac{(1 - \lambda)^4 L}{2}\right) \,.
\end{align*}
You showed in \cref{ex:gauss:unbiased} that $\tilde s(z)$ is an unbiased estimator of $s(z)$ and therefore
\begin{align*}
|\E[\hat s(z)] - s(z)|
&= |\E[\hat s(z) - \tilde s(z)]| \\
&= \left|\E\left[\frac{Y(\bar r(X, z) - r(X, z))}{\lambda}\right]\right| \\
&\leq \E\left[\frac{\sind_E |Y| r(X, z)}{\lambda}\right] \\
&\leq \frac{1}{\lambda} \E[Y^4]^{\frac{1}{4}} \E[r(X, z)^4]^{\frac{1}{4}} \sqrt{\bbP(E)} \\
&\leq \delta\,,
\end{align*}
where we used \cref{lem:opt:r-moment} and \cref{lem:opt:f-moment} below.
The proofs of parts~\ref{prop:opt:bias:1} and \ref{prop:opt:bias:2} follow the same argument.
That the estimators are unbiased when $z = \mu$ follows from the observation that $\bar r(x, \mu) = r(x, \mu)$ for all $x$
and \cref{ex:gauss:unbiased}.
\end{proof}

\begin{exer}
\faStar \quad
Prove \cref{prop:opt:bias}\ref{prop:opt:bias:1} and \ref{prop:opt:bias:2}.
\end{exer}

\section[Concentration]{Concentration (\skippy)}\index{concentration!of surrogate|(}
In this section we explore the tail behaviour of the estimators in the previous section.
Almost all of the results here are only used as technical lemmas in the previous and next sections.
Recall that $r$ is the change of measure function defined by
\begin{align*}
r(x, z) = \frac{p\left(\frac{x - \lambda z}{1 - \lambda}\right)}{(1 - \lambda)^d p(x)}
\quad \text{and}\quad
\bar r(x, z) = \min(\exp(2), r(x, z)) \,, 
\end{align*}
where $p$ is the density of the $\cN(\mu, \Sigma)$.
The next few lemmas bound the magnitude, gradients and moments of $r$.

\begin{lemma}\label{lem:measure}
For all $x, z \in \R^d$,
\begin{align*}
r(x, z) 
&\leq \exp\left(1 + \frac{\lambda}{(1-\lambda)^2} \ip{x - \mu, z - \mu}_{\Sigma^{-1}}\right) \,. 
\end{align*}
\end{lemma}

\begin{proof}
Let us assume without loss of generality that $\mu = \zeros$. By definition,
\begin{align*}
r(x, z) 
&= \frac{p\left(\frac{x - \lambda z}{1 - \lambda}\right)}{(1 - \lambda)^d p(x)} \\
&= \frac{1}{(1 - \lambda)^d}\exp\left(-\frac{1}{2} \norm{\frac{x - \lambda z}{1-\lambda}}^2_{\Sigma^{-1}} + \frac{1}{2} \norm{x}^2_{\Sigma^{-1}}\right) \\ 
&\leq \frac{1}{(1 - \lambda)^d} \exp\left(\frac{\lambda \ip{x, z}_{\Sigma^{-1}}}{(1 - \lambda)^2}\right) \\ 
&\leq \exp\left(1 + \frac{\lambda \ip{x, z}_{\Sigma^{-1}}}{(1 - \lambda)^2}\right) \,, 
\end{align*}
where in the final inequality we used \cref{ass:gauss} that $\lambda \leq \frac{1}{d+1}$ so that
$(1 - \lambda)^{-d} \leq (1+\frac{1}{d})^d \leq \exp(1)$.
\end{proof}

The next lemma loosely bounds the Lipschitz constant of $z \mapsto \bar r(x, z)$.

\begin{lemma}\label{lem:opt:measure-lip}
Suppose that $A = \{z \colon \lambda \norm{z - \mu}_{\Sigma^{-1}} \leq 1\}$ and $x \in \R^d$. Then 
\begin{align*}
\lip_A(\bar r(x, \cdot)) \leq \frac{1}{\delta} \,. 
\end{align*}
\end{lemma}

\begin{exer}
\faStar \quad
Prove \cref{lem:opt:measure-lip}.
\end{exer}

\solution{
Abbreviate $\bar r(z) = \bar r(x, z)$ and $r(z) = \bar r(x, z)$.
By the definition of $\bar r$ it suffices to bound $\norm{r'(z)}$ for $z$ with $r(z) < \exp(2)$.
Suppose that $r(z) < \exp(2)$ and consider two cases.

\begin{enumerate}
\item[\texttt{\uline{Case 1}:}] $\norm{x}_{\Sigma^{-1}} \leq \frac{4}{\lambda}$.
Then we have
\begin{align*}
\norm{r'(z)}
&= \frac{\lambda r(z)}{(1 - \lambda)^2} \norm{\Sigma^{-1}(x - \lambda z)} \\
&\leq \frac{\lambda \exp(2)}{(1 - \lambda)^2} \snorm{\Sigma^{-1/2}} \norm{x - \lambda z}_{\Sigma^{-1}} \\
&\leq \frac{\lambda \exp(2)}{(1 - \lambda)^2} \snorm{\Sigma^{-1/2}} \left[\frac{4}{\lambda} + 1\right] \\
&\leq \frac{1}{\delta} \,.
\end{align*}
\item[\texttt{\uline{Case 2}:}] $\norm{x}_{\Sigma^{-1}} > \frac{4}{\lambda}$.
Then we have
\begin{align*}
\norm{r'(z)}
&= \frac{\lambda r(z)}{\sqrt{\delta}(1 - \lambda)^2} \norm{x - \lambda z}_{\Sigma^{-1}} 
\leq \frac{2\lambda r(z)}{\sqrt{\delta}(1 - \lambda)^2} \norm{x}_{\Sigma^{-1}} \,.
\end{align*}
Moreover,
\begin{align*}
r(z) 
&\leq \exp\left(1 - \lambda \norm{x}^2_{\Sigma^{-1}} + 2 \norm{x}_{\Sigma^{-1}}\right)
\leq \exp\left(1 - \frac{\lambda}{2} \norm{x}^2_{\Sigma^{-1}}\right)\,.
\end{align*}
Therefore 
\begin{align*}
\norm{r'(z)} 
&\leq \frac{2 \lambda}{\sqrt{\delta}(1 - \lambda)^2} \norm{x}_{\Sigma^{-1}} \exp\left(1 - \frac{\lambda}{2} \norm{x}^2_{\Sigma^{-1}}\right) 
\leq \frac{1}{\delta}\,.
\end{align*}
\end{enumerate}
}

Later we need some conservative upper bounds on the moments of the various estimators.
The easiest way to obtain these is to bound the moments of the constituent parts and combine
them using H\"older's inequality.
Remember in what follows that $X$ has law $\cN(\mu, \Sigma)$.

\begin{lemma}\label{lem:opt:r-moment}
For any $k \geq 1$,
$\E[r(X, x)^k] \leq C_k \exp\left(C_k \lambda^2 \norm{x - \mu}^2_{\Sigma^{-1}}\right)$.
\end{lemma}

\begin{proof}
By \cref{lem:measure} and \cref{prop:gaussian},
\begin{align*}
\E[r(X, x)^k]
&\leq \exp(k)  \E\left[\exp\left(\frac{\lambda k}{(1 - \lambda)^2} \ip{X - \mu, x - \mu}_{\Sigma^{-1}}\right)\right] \\
&= \exp(k) \exp\left(\frac{\lambda^2 k^2}{2(1 - \lambda)^4} \norm{x - \mu}^2_{\Sigma^{-1}}\right) \\
&\leq C_k \exp\left(C_k \lambda^2 \norm{x - \mu}^2_{\Sigma^{-1}}\right)\,.
\qedhere
\end{align*}
\end{proof}

\begin{lemma}\label{lem:opt:f-moment}
Suppose that $\mu \in K$. Then, for any $k \geq 1$,
\begin{align*}
\E[|f(X)|^k] \leq C_k\delta^{-2k} \qquad \text{and} \qquad
\E[|Y|^k] \leq C_k \delta^{-2k}\,.
\end{align*}
\end{lemma}

\begin{proof}
Since $\mu \in K$, by \cref{ass:gauss}, $|f(\mu)| \leq 1/\delta$, so
\begin{align*}
\E[|f(X)|^k]
&= \E\left[|f(X) - f(\mu) + f(\mu)|^k\right] \\
&\leq 2^{k-1} \delta^{-k} + 2^{k-1} \lip(f)^k \E\left[\norm{X - \mu}^k\right] \\
&\leq 2^{k-1} \delta^{-k} + 2^{k-1} \lip(f)^k \norm{\Sigma}^{\frac{k}{2}} \E\left[\norm{X - \mu}^k_{\Sigma^{-1}}\right] \\
&\leq C_k \delta^{-2k} \,,
\end{align*}
where we used the fact that $(a+b)^k \leq 2^{k-1} [a^k + b^k]$, \cref{prop:orlicz,lem:orlicz-moment} to bound the moments of $\norm{X - \mu}_{\Sigma^{-1}}$ and \cref{ass:gauss}.
And of course we made sure to choose $C_k$ as a suitably large $k$-dependent constant.
The second part follows from the first using the fact that $Y = f(X) + \eps$ where $\eps$ is conditionally subgaussian (\cref{ass:gauss}).
\end{proof}

\section{Sequential Concentration}\label{sec:conc-seq}
We now focus on the sequential aspects of concentration. Let $f_1,\ldots,f_n \colon \R^d \to \R$ be a sequence of convex functions.
Assume that $X_1,Y_1,\ldots,X_n,Y_n$ is the sequence of actions and losses generated by an algorithm interacting with a convex bandit,
which is adapted as usual to the filtration
$(\sF_t)_{t=1}^n$. Let $\tau$ be a stopping time with respect to the filtration $(\sF_t)$.\index{filtration}\index{stopping time}
As usual, we let $\bbP_t = \bbP(\cdot | \sF_t)$ and $\E_t$ be the corresponding expectation operator.
We let $\norm{\cdot}_{t,\psi_k}$ be the $k$th Orlicz norm with respect to $\bbP_t$.
The next assumption generalises \cref{ass:gauss} to the sequential setting:

\begin{assumption}\label{ass:opt:conc}
The following hold almost surely for all $1 \leq t \leq \tau$:
\begin{enumerate}
\item \textit{Convexity:} $K$ is a convex body and $f_t$ is convex.
\item \textit{Gaussian iterates:} $X_t$ has law $\cN(\mu_t, \Sigma_t)$ under $\bbP_{t-1}$ and $\mu_t \in K$.
\item \textit{Subgaussian responses:} $Y_t = f_t(X_t) + \eps_t$ where 
\begin{align*}
\E_{t-1}[\eps_t|X_t] = 0 \quad \text{ and } \quad \E_{t-1}[\exp(\eps_t^2)|X_t] \leq 2\,.
\end{align*}
\item \textit{Boundedness:} $\lambda \leq \frac{1}{d+1}$ and 
\begin{align*}
\max\left(n, d, \lip(f_t), \sup_{x \in K} |f_t(x)|, 1/\lambda, \snorm{\Sigma_t}, \snorm{\Sigma_t^{-1}}\right) \leq \frac{1}{\delta} \,.
\end{align*}
\end{enumerate}
\end{assumption}

As before, we let $L$ be a logarithmic factor:
\begin{align*}
L = C \log (1/\delta)\,,
\end{align*}
where $C > 0$ is a universal constant.
The surrogate function and its quadratic approximation now change from round to round and are given by
\begin{align*}
s_t(z) &= \E_{t-1}\left[\left(1 - \frac{1}{\lambda}\right) f_t(X_t) + \frac{1}{\lambda} f_t((1 - \lambda) X_t + \lambda z)\right] \qquad \text{and} \\
q_t(z) &= \ip{s'_t(\mu_t), z - \mu_t} + \frac{1}{4} \norm{z - \mu_t}^2_{s''_t(\mu_t)} \,.
\end{align*}
Note that even when $f_t = f$ is unchanging, the surrogate depends on $\mu_t$ and $\Sigma_t$ and may still change from round to round.
Let $p_t$ be the density of $\cN(\mu_t, \Sigma_t)$ and
\begin{align*}
\bar r_t(x, z) = \min\left(\exp(2),\, \frac{p_t\left(\frac{x - \lambda z}{1 - \lambda}\right)}{(1 - \lambda)^d p_t(x)}\right) \,.
\end{align*}
We abbreviate $\bar r_t(z) = \bar r_t(X_t, z)$.
The estimators of $s_t$ and its derivatives are
\begin{align*}
\hat s_t(z) &= \left(1 - \frac{1}{\lambda} + \frac{\bar r_t(z)}{\lambda}\right) Y_t \\
\hat s_t'(z) &= \frac{\bar r_t(z) Y_t}{1-\lambda} \Sigma_t^{-1} \left(\frac{X_t - \lambda z}{1-\lambda} - \mu_t\right) \\
\hat s_t''(z) &= \frac{\lambda \bar r_t(z) Y_t}{(1-\lambda)^2} \left[\Sigma_t^{-1}\left(\frac{X_t - \lambda z}{1-\lambda} - \mu_t\right)\left(\frac{X_t - \lambda z}{1 - \lambda} - \mu_t\right)^\top \Sigma_t^{-1} - \Sigma_t^{-1}\right]\,.
\end{align*}
Throughout we let $g_t = \hat s'_t(\mu_t)$, $H_t = \frac{1}{2} \hat s''_t(\mu_t)$, $\bar g_t = s_t'(\mu_t)$ and $\bar H_t = \frac{1}{2} s_t''(\mu_t)$,
which means an estimator of the quadratic surrogate is
\begin{align*}
\hat q_t(x) = \ip{g_t, x - \mu_t} + \frac{1}{2} (x - \mu_t)^\top H_t (x - \mu_t) 
\end{align*}
and the actual quadratic surrogate is
\begin{align*}
q_t(x) = \ip{\bar g_t, x - \mu_t} + \frac{1}{2} (x - \mu_t)^\top \bar H_t (x - \mu_t) \,.
\end{align*}

\begin{figure}
\includegraphics[width=0.9\textwidth]{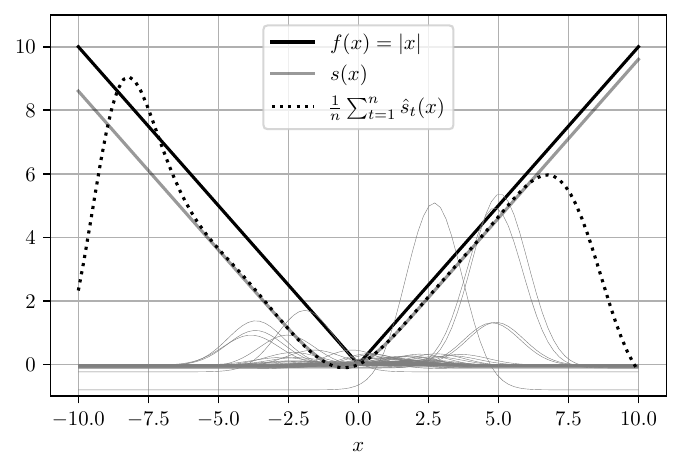}
\caption{The concentration of $\sum_{t=1}^n \hat s_t(x) / n$ with $f_t = f = |\cdot|$, $n = 10^5$, $\mu = 1/2$, $\Sigma = 1$ and $\lambda = 1/2$. The thin lines correspond
to the first one hundred estimated surrogates.
The estimate is very close to the real surrogate on an interval around $\mu$ but can be extremely poorly behaved far away, even with $n$ so large.
Note also that the estimated surrogate is convex near $\mu$ but not everywhere.
}\label{fig:s-est}
\commentAlt{A plot showing the absolute loss along with the surrogate and the estimated surrogate, which is a near-perfect match to the surrogate close to the mean and completely
wrong far away. Small grey lines show the individual loss estimates that are neither convex nor concave and look like shifted Gaussians.}
\end{figure}

\subsubsection*{Objectives and Plan}
The questions in this section concern concentration of quantities like $\sum_{t=1}^\tau (\hat s_t - s_t)$.
This is an entire function, so we need to be precise about what is meant by concentration.
Typical results show that functions like this are small at a specific $x$ or for all $x$ in some set.
The magnitude of the errors generally depends on some kind of cumulative predictable variation and our bounds reflect that.
The change of measure $\bar r_t(x)$ that appears in the definition of the estimators is well-behaved when $x$ is close enough to $\mu$.
Because of this most of the concentration bounds that follow only hold on a subset of $\R^d$.
An illustrative experiment is given in \cref{fig:s-est}.
Given $r > 0$, let 
\begin{align*}
K_\tau(r) = \left\{x \in \R^d \colon \max_{1 \leq t \leq \tau} \lambda \norm{x - \mu_t}_{\Sigma_t^{-1}} \leq r\right\}\,,
\end{align*}
which is an intersection of ellipsoids and hence convex. The set $K_\tau(r)$ is often referred to as the focus region.\index{focus region}
The general flavour of the results is as follows:
\begin{itemize}
\item Given a deterministic $x$, $\sum_{t=1}^\tau (\hat s_t(x) - s_t(x))$ is well-concentrated about zero provided that $x \in K_\tau(1/\sqrt{2L})$ almost surely.
\item The function $\sum_{t=1}^\tau (\hat s_t(x)-s_t(x))$ is well-concentrated about zero for all $x \in K_\tau(1/\sqrt{2L})$, with a slightly
wider confidence interval\index{confidence interval} than the case above.
\end{itemize}
The predictable variation of the estimators is mostly caused by the variance in the losses.
Let $V_\tau = \sum_{t=1}^\tau \E_{t-1}[Y_t^2]$ and $\Ymax = \max_{1 \leq t \leq \tau} |Y_t| + \E_{t-1}[|Y_t|]$.
Generally speaking, in applications the losses $(Y_t)$ will be bounded in $O(1)$ with high probability and in this case $V_\tau = O(n)$.
Our concentration bounds will be established using a martingale version\index{martingale} of Bernstein's inequality (\cref{thm:freedman}),\index{Freedman's inequality}\index{Bernstein's inequality} 
which is a variance-aware concentration
inequality.

\subsubsection*{Concentration Bounds}
We start with a naive bound on $V_\tau$ and $\Ymax$, which is only used to bound these quantities when they appear in logarithmic terms.

\begin{proposition}\label{prop:gauss:conc:naive}
With probability at least $1 - \delta / 2$,
\begin{align*}
\max(\Ymax, V_\tau) \leq \poly(1/\delta)\,.
\end{align*}
\end{proposition}

\begin{proof}
Combine Markov's inequality, a union bound and \cref{lem:opt:f-moment}.
\end{proof}

The first significant result is a
Bernstein-like\index{Bernstein's inequality} concentration bound for the sum of the surrogate loss estimators:

\begin{proposition}\label{prop:conc-s}
Under \cref{ass:opt:conc},
the following hold:
\begin{enumerate}
\item Let $x \in \R^d$ be a non-random vector such that $x \in K_\tau(1/\sqrt{L})$ almost surely. 
Then, with probability at least $1 - \delta$, \label{prop:conc-s:single}
\begin{align*}
\left|\sum_{t=1}^\tau (\hat s_t(x) - s_t(x))\right| \leq 1 + \frac{1}{\lambda} \left[\sqrt{L V_\tau} + L \Ymax\right] \,.
\end{align*}
\item 
With probability at least $1 - \delta$, \label{prop:conc-s:uniform}
\begin{align*}
\max_{x \in K_\tau(1/\sqrt{2L})} \left|\sum_{t=1}^\tau (\hat s_t(x) - s_t(x))\right| \leq 2 + \frac{1}{\lambda} \sqrt{d L V_\tau} + \frac{d L \Ymax}{\lambda} \,.
\end{align*}
\end{enumerate}
\end{proposition}

You should view these bounds as a kind of Bernstein inequality\index{Bernstein's inequality} with the term involving $\Ymax$ the lower-order term and $V_\tau = O(n)$ 
with high probability. 

\begin{proof}
Starting with part~\ref{prop:conc-s:single},
let $x \in \R^d$ be such that $x \in K_\tau(1/\sqrt{L})$ almost surely and recall that
\begin{align*}
\hat s_t(x) = \left(1 + \frac{\bar r_t(x) - 1}{\lambda}\right) Y_t \,.
\end{align*}
Let $\Delta_t = \hat s_t(x) - \E_{t-1}[\hat s_t(x)]$.
A martingale\index{martingale} version of Bernstein's inequality (\cref{thm:freedman}) \index{Freedman's inequality} 
applied to the sequence $(\Delta_t)_{t=1}^\tau$ says that with probability at least $1 - \delta/2$,
\begin{align}
\left|\sum_{t=1}^\tau \Delta_t\right| \leq 3 \sqrt{M \log\left(\frac{4 \max(B, \sqrt{M})}{\delta}\right)} + 2 B \log\left(\frac{4 \max (B, \sqrt{M})}{\delta}\right)\,,
\label{eq:os:conc-s:bern}
\end{align}
where $M = \sum_{t=1}^\tau \E_{t-1}[\Delta_t^2]$ and $B = \max(1, \max_{1 \leq t \leq \tau} |\Delta_t|)$.
We now bound the random variables $M$ and $B$. 
By definition,
\begin{align*}
M 
\leq \sum_{t=1}^\tau \E_{t-1}[\hat s_t(x)^2] 
\leq \frac{\exp(4)}{\lambda^2} \sum_{t=1}^\tau \E_{t-1}[Y_t^2] 
= \frac{\exp(4)}{\lambda^2} V_\tau \,,
\end{align*}
where we used the fact that $\bar r_t(x) \leq \exp(2)$.
Additionally,
\begin{align*}
B &= \max\left(1, \max_{1 \leq t \leq \tau} |\Delta_t|\right) \\
&\leq 1 + \max_{1 \leq t \leq \tau} \left|\hat s_t(x) - \E_{t-1}[\hat s_t(x)]\right| \\
&\leq 1 + \max_{1 \leq t \leq \tau} \left(\left|\hat s_t(x)\right| + \left|\E_{t-1}[\hat s_t(x)]\right| \right) \\
&\leq 1 + \frac{\exp(2)}{\lambda} \max_{1 \leq t \leq \tau} \left(|Y_t| + \E_{t-1}[|Y_t|]\right) \\
&= 1 + \frac{\exp(2) \Ymax}{\lambda}\,.
\end{align*}
By \cref{prop:gauss:conc:naive}, with probability at least $1 - \delta/2$,
\begin{align*}
\max(B, \sqrt{M}) \leq \poly(1/\delta)\,.
\end{align*}
Combining with \cref{eq:os:conc-s:bern} shows that with probability at least $1 - \delta$,
\begin{align*}
\left|\sum_{t=1}^\tau \Delta_t\right| \leq \frac{1}{\lambda} \left[\sqrt{V_\tau L} + L \Ymax\right] \,.
\end{align*}
Note that so far we have not used the fact that $x \in K_\tau(1/\sqrt{L})$.
The argument above shows that $\sum_{t=1}^\tau (\hat s_t(x) - \E_{t-1}[\hat s_t(x)])$ concentrates well for any $x \in \R^d$.
All that remains is to argue that $\E_{t-1}[\hat s_t(x)]$ is close to $s_t(x)$.
Since $x \in K_\tau(1/\sqrt{L})$, by \cref{prop:opt:bias},
\begin{align*}
\left|\sum_{t=1}^\tau (\hat s_t(x) - s_t(x))\right|
&\leq n \delta + \left|\sum_{t=1}^\tau (\hat s_t(x) - \E_{t-1}[\hat s_t(x)])\right| \\
&\leq 1 + \left|\sum_{t=1}^\tau \Delta_t\right| \\
&\leq 1 + \frac{1}{\lambda} \left[\sqrt{V_\tau L} + L \Ymax\right] \,.
\end{align*}
This completes the proof of Part~\ref{prop:conc-s:single}.
Moving now to part~\ref{prop:conc-s:uniform},
abbreviate $K_\tau = K_\tau(1/\sqrt{2L})$. By \cref{ass:opt:conc}, $\Sigma \preceq \frac{1}{\delta} \id$ and $\lambda \geq \delta$,
and by the definition of $K_\tau$,
\begin{align}
K_\tau 
\subset \left\{x \colon \lambda \norm{x - \mu_1}_{\Sigma_1^{-1}} \leq 1/\sqrt{2L}\right\} 
\subset \left\{x \colon \norm{x - \mu_1} \leq \delta^{-3/2} \right\} \triangleq J\,.
\label{eq:prop:conc-s:J}
\end{align}
The argument follows along the same lines as part~\ref{prop:conc-s:single} but now we need an additional covering and Lipschitz argument.
Let $\cC$ be a finite cover of $J$ such that for all $y \in J$ there exists an $x \in \cC$ such that $\norm{x - y} \leq \eps$ with 
\begin{align}
\eps = \poly(\delta) \,.
\label{eq:opt:epsilon-s}
\end{align}
\citet[Corollary 4.2.13]{Ver18} shows that $\cC$ can be chosen so that 
\begin{align*}
|\cC| \leq \left(\frac{2 \diam(J)}{\eps} + 1\right)^d\,.
\end{align*}
By \cref{eq:prop:conc-s:J}, $\diam(J) \leq 2 \delta^{-3/2}$. Hence, by the definition of $L = C \log(1/\delta)$ for suitably large universal
constant $C$, it follows that $\log |\cC| \leq d L$. 
Repeating the argument in Part~\ref{prop:conc-s:single} along with a union bound over $\cC$ shows that with probability at least $1 - \delta$ the
following both hold:
\begin{enumeraterom}
\item $\max_{x \in \cC} \left|\sum_{t=1}^{\tau} (\hat s_t(x) - \E_{t-1}[\hat s_t(x)])\right| \leq \frac{1}{\lambda}\left[\sqrt{d V_\tau L} + d L \Ymax\right]$;
\item $\max(\Ymax, V_\tau) \leq \poly(1/\delta)$.
\end{enumeraterom}
For the remainder we assume these events occur.
Let $y \in K_\tau$.
By the construction of $\cC$ there exists an $x \in \cC$ such that $\norm{x - y} \leq \eps$.
Since $y \in K_\tau$ for any $t \leq \tau$, 
\begin{align*}
\lambda \norm{x - \mu_t}_{\Sigma_t^{-1}} 
&\leq \lambda \norm{x - y}_{\Sigma_t^{-1}} + \lambda \norm{y - \mu_t}_{\Sigma_t^{-1}}  
\leq \frac{\lambda \eps}{\sqrt{\delta}} + \frac{1}{\sqrt{2L}} 
\leq \frac{1}{\sqrt{L}}\,,
\end{align*}
where we used the definition of $K_\tau$, the triangle inequality, the definition of $\eps$ in \cref{eq:opt:epsilon-s} and naive bounding.
Therefore $x \in K_\tau(1/\sqrt{L})$ and for any $t \leq \tau$, 
\begin{align*}
\left|\hat s_t(x) - \hat s_t(y)\right|
= \frac{\left|\bar r_t(x) - \bar r_t(y)\right| |Y_t| }{\lambda} 
\leq \frac{|Y_t|}{\lambda \delta} \norm{x - y} 
\leq \frac{\eps \Ymax}{\lambda \delta} 
\leq \frac{1}{2n} \,.
\end{align*}
where we used \cref{lem:opt:measure-lip} and the assumption that $\norm{x - y} \leq \eps$.
By \cref{lem:basic}\ref{lem:basic:lip}, 
\begin{align*}
\lip(s_t) \leq \lip(f_t) \leq \frac{1}{\delta} \,.
\end{align*}
Hence $|s_t(x) - s_t(y)| \leq \frac{\eps}{\delta} \leq \frac{1}{2n}$.
Therefore
\begin{align*}
\left|\sum_{t=1}^\tau (\hat s_t(y) - s_t(y))\right| 
&\leq 1 + \left|\sum_{t=1}^\tau (\hat s_t(x) - s_t(x))\right| \\ 
\tag*{by \cref{prop:opt:bias}}
&\leq 2 + \left|\sum_{t=1}^\tau (\hat s_t(x) - \E_{t-1}[\hat s_t(x)])\right| \\
&\leq 2 + \frac{1}{\lambda}\left[\sqrt{d V_\tau L} + d L \Ymax\right] \,,
\end{align*}
which completes the proof.
\end{proof}

More or less the same result holds for the quadratic surrogates estimates.

\begin{proposition}\label{prop:conc-q}
Under \cref{ass:opt:conc},
the following hold: 
\begin{enumerate}
\item Let $x \in \R^d$ be a non-random vector such that $x \in K_\tau(1/\sqrt{L})$ almost surely. 
Then, with probability at least $1 - \delta$,  \label{prop:conc-q:non-uniform}
\begin{align*}
\left|\sum_{t=1}^\tau (q_t(x) - \hat q_t(x))\right| \leq 1 + \frac{1}{\lambda} \left[\sqrt{V_\tau L} + \Ymax L\right] \,.
\end{align*}
\item With probability at least $1 - \delta$, \label{prop:conc-q:uniform}
\begin{align*}
\max_{x \in K_\tau(1/\sqrt{2L})} \left|\sum_{t=1}^\tau (q_t(x) - \hat q_t(x))\right| \leq 1 + \frac{1}{\lambda} \left[\sqrt{d V_\tau L} + d\Ymax L\right] \,. 
\end{align*}
\end{enumerate}
\end{proposition}

\begin{proof}
Let $x \in \R^d$ be a non-random vector such that $x \in K_\tau(1/\sqrt{L})$ almost surely and $\Delta_t = \hat q_t(x) - q_t(x)$, which
satisfies $\E_{t-1}[\Delta_t] = 0$ by \cref{prop:opt:bias} and the definition of $\hat q_t$.
By \cref{thm:freedman},\index{Freedman's inequality} with probability at least $1 - \delta/4$,
\begin{align}
\left|\sum_{t=1}^\tau \Delta_t\right| \leq 3 \sqrt{M \log\left(\frac{8 \max(B, \sqrt{M})}{\delta}\right)} + 2 B \log\left(\frac{8 \max(B, \sqrt{M})}{\delta}\right) \,,
\label{eq:opt:q:main}
\end{align}
where $M = \sum_{t=1}^\tau \E_{t-1}[\Delta_t^2]$ and $B = \max_{1 \leq t \leq \tau} |\Delta_t|$.
Moreover, by \cref{prop:gauss:conc:naive}, with probability at least $1 - \delta/2$,
\begin{align}
\max(\Ymax, V_\tau) \leq \poly(1/\delta) \,.
\label{eq:opt:q:naive}
\end{align}
Since $\E_{t-1}[\hat q_t(x)] = q_t(x)$ we have
\begin{align*}
\E_{t-1}[\Delta_t^2] \leq \E_{t-1}[\hat q_t(x)^2] 
\quad \text{ and } \quad
|\Delta_t|  \leq |\hat q_t(x)| + \E_{t-1}[|\hat q_t(x)|] \,.
\end{align*}
Let $U_t = \ip{x - \mu_t, X_t - \mu_t}_{\Sigma_t^{-1}}$. Substituting the definitions of $g_t$ and $H_t$ yields
\begin{align*}
\hat q_t(x) 
&= \ip{g_t, x - \mu_t} + \frac{1}{2} (x - \mu_t)^\top H_t (x - \mu_t) \\
&= \bar r_t(\mu_t) Y_t \left[\frac{U_t}{1-\lambda} + \frac{\lambda U_t^2}{4(1 - \lambda)^4} 
- \frac{\lambda \norm{x - \mu_t}^2_{\Sigma_t^{-1}}}{4(1-\lambda)^2}\right] \,.
\end{align*}
Using the fact from \cref{lem:measure} that $\bar r_t(\mu_t) \in [0,\exp(1)]$, $\lambda \leq 1/2$ and the condition that $x \in K_\tau(1/\sqrt{L})$, it follows that
\begin{align}
|\hat q_t(x)| \leq |Y_t| \exp(1) \left[2|U_t| + 4\lambda |U_t|^2 + \frac{1}{\lambda L}\right] \,.
\label{eq:opt:q:1}
\end{align}
The law of $U_t$ under $\bbP_{t-1}$ is $\cN(0, \snorm{x - \mu_t}_{\Sigma_t^{-1}})$. 

\begin{exer}\label{ex:opt:q}
\faStar \quad
Complete the following steps:
\begin{enumeraterom}
\item Show that $\snorm{U_t^2}_{t-1,\psi_1} \leq \frac{8 }{3L\lambda^2}$. \label{ex:opt:q:norm}
\item Use \cref{eq:opt:q:1}, part~\ref{ex:opt:q:norm}, \cref{ass:opt:conc} and \cref{lem:orlicz-var} to show that 
\begin{align*}
\E_{t-1}[\hat q_t(x)^2] \leq \delta + \frac{C \E_{t-1}[Y_t^2]}{\lambda^2} \qquad \! 
\E_{t-1}[|\hat q_t(x)|] \leq \delta + \frac{C \E_{t-1}[|Y_t|]}{\lambda} \,.
\end{align*}
\item Use \cref{lem:orlicz-tail} to show that with probability at least $1 - \delta/4$ 
\begin{align}
|\hat q_t(x)| \leq \frac{C |Y_t|}{\lambda} \text{ for all } 1 \leq t \leq \tau\,.
\label{eq:opt:q:hp}
\end{align}
\end{enumeraterom}
\end{exer}

By a union bound, with probability at least $1 - \delta$ the high-probability events in \cref{eq:opt:q:main}, \cref{eq:opt:q:naive} and 
\cref{eq:opt:q:hp} all hold. For the remainder assume these events hold. 
Your solution to \cref{ex:opt:q} shows that 
\begin{align*}
M 
&= \sum_{t=1}^\tau \E_{t-1}[\Delta_t^2]
\leq n \delta + \frac{C V_\tau}{\lambda^2} \,. 
\end{align*}
Moreover,
\begin{align*}
B &= \max_{1 \leq t \leq \tau} |\Delta_t| 
\leq \max_{1 \leq t \leq \tau} \left(|\hat q_t(x)| + \E_{t-1}[|\hat q_t(x)|]\right)
\leq \delta + \frac{C \Ymax}{\lambda} \,.
\end{align*}
Combining these with \cref{eq:opt:q:main} shows that 
\begin{align*}
\left|\sum_{t=1}^\tau (q_t(x) - \hat q_t(x))\right| \leq 1 + \frac{1}{\lambda}\left[\sqrt{V_\tau L} + L \Ymax\right]\,.
\end{align*}
Part~\ref{prop:conc-q:uniform} follows by 
repeating more or less the same argument as the covering argument in \cref{prop:conc-s}.
\end{proof}

\begin{exer}
\faStar \quad
Prove \cref{prop:conc-q}\ref{prop:conc-q:uniform}.
\end{exer}

The next proposition shows that the accumulation of second derivatives along the sequence $\mu_1,\ldots,\mu_\tau$ is well-concentrated.

\begin{proposition}\label{prop:conc-hessian-path}
Let $\sS \subset \psd$ be the (random) set of positive definite matrices such that $\Sigma_t^{-1} \preceq \Sigma^{-1}$ for all $t \leq \tau$
and $S_\tau = \sum_{t=1}^\tau \hat s_t''(\mu_t)$ and $\bar S_\tau = \sum_{t=1}^\tau s_t''(\mu_t)$.
With probability at least $1 - \delta$, for all $\Sigma^{-1} \in \sS$,
\begin{align*}
 - \lambda L^2 \left[1 + \sqrt{d V_\tau} + d^2 \Ymax \right] \Sigma^{-1} 
&\preceq S_\tau - \bar S_\tau  
\preceq \lambda L^2 \left[1 + \sqrt{d V_\tau} + d^2 \Ymax\right] \Sigma^{-1} \,.
\end{align*}
\end{proposition}

\begin{proof}
By definition of the L\"owner order, the claim is equivalent to showing that with probability at least $1 - \delta$ for all $u \in \sphere_1$ and
$\Sigma \in \cS$,
\begin{align}
\left|u^\top (S_\tau - \bar S_\tau) u\right| \leq \lambda L^2 \left[1 + \sqrt{d V_\tau} + d^2 \Ymax\right] \norm{u}^2_{\Sigma^{-1}} \,.
\label{eq:prop:conc-h:1}
\end{align}
The approach followed will be the usual one:
\begin{itemize}
\item Prove that \cref{eq:prop:conc-h:1} holds for all $u$ in a finite cover of $\sphere_1$ with slightly smaller constants.
\item Extend to all $u \in \sphere_1$ by a Lipschitz argument.
\end{itemize}
Let $\cCS$ be a cover of $\sphere_1$ such that for all $u \in \sphere_1$ there exists a $v \in \cCS$ such that $\norm{u - v} \leq \eps$
where $\eps = \poly(\delta)$ is a small constant. \citet[Corollary 4.2.13]{Ver18} shows that $\cCS$ can be chosen so that
\begin{align*}
\log |\cCS| \leq d \log \left(\frac{2}{\eps} + 1\right) \leq d L \,.
\end{align*}
Let $W_t = \Sigma_t^{-1/2}(X_t - \mu_t)$, which is a standard Gaussian under $\bbP_{t-1}$. 
By definition,
\begin{align*}
\hat s''_t(\mu_t) 
&= \frac{\lambda \bar r_t(\mu_t) Y_t}{(1 - \lambda)^2} \left[ \frac{\Sigma_t^{-1} (X_t - \mu_t)(X_t - \mu_t)^\top \Sigma_t^{-1}}{(1-\lambda)^2} - \Sigma_t^{-1}\right] \\
&= \frac{\lambda \bar r_t(\mu_t) Y_t}{(1 - \lambda)^2} \left[\frac{\Sigma_t^{-1/2} W_t W_t^\top \Sigma_t^{-1/2}}{(1-\lambda)^2} - \Sigma_t^{-1}\right]  \,.
\end{align*}
Let $u \in \cCS$ and define
\begin{align*}
Q_{t,u} = \frac{\lambda \bar r_t(\mu_t)}{(1 - \lambda)^2} \left[\frac{\sip{\Sigma_t^{-1/2} u, W_t}^2}{(1-\lambda)^2} - \norm{u}^2_{\Sigma_t^{-1}}\right]\,, 
\end{align*}
which is chosen so that $u^\top \hat s_t''(\mu_t) u = Y_t Q_{t,u}$.
By \cref{prop:opt:bias} $\E_{t-1}[\hat s_t''(\mu_t)] = s_t''(\mu_t)$. Hence 
\begin{align*}
\Delta_{t,u} = Y_t Q_{t,u} - \E_{t-1}[Y_t Q_{t,u}] = u^\top \hat s_t''(\mu_t) u - u^\top s_t''(\mu_t) u \,. 
\end{align*}
By \cref{lem:measure}, $\bar r_t(\mu_t) \leq \exp(1)$ and hence by \cref{eq:orlicz-1} and \cref{prop:orlicz:product,prop:orlicz} and naively simplifying 
constants,
\begin{align}
\snorm{Q_{t,u}}_{t-1,\psi_1} \leq 150 \lambda \norm{u}^2_{\Sigma_t^{-1}} \,.
\label{eq:conc:hessian-path}
\end{align}
Therefore, since the variance of a random variable is upper-bounded by its second moment and by \cref{lem:orlicz-var} with $k = 2$,
\begin{align*}
\E_{t-1}[\Delta_{t,u}^2]
&\leq \E_{t-1}[Y_t^2 Q_{t,u}^2] \\
&\leq \lambda^2 \left(\delta + L^2 \E_{t-1}[Y_t^2]\right) \norm{u}^4_{\Sigma_t^{-1}} \\
&\leq \poly(1/\delta)\left(1 + \E_{t-1}[Y_t^2]\right)\,,
\end{align*}
where the final inequality crudely uses the fact that $u \in \sphere_1$ and the assumption that $\snorm{\Sigma_t^{-1}} \leq 1/\delta$.
A union bound combined with \cref{eq:conc:hessian-path} and \cref{lem:orlicz-tail} and \cref{lem:orlicz-var} with $k = 1$ shows that with probability at least $1 - \delta/4$ for all $u \in \cCS$ and $1 \leq t \leq \tau$,
\begin{align*}
\left|\Delta_{t,u}\right|
\leq \lambda d (\Ymax+\delta) L \norm{u}^2_{\Sigma_t^{-1}} 
\leq \poly(1/\delta)\left(1 + \Ymax\right) \,.
\end{align*}
By \cref{prop:gauss:conc:naive}, with probability at least $1 - \delta/2$, $\max (\Ymax, V_\tau) \leq \poly(1/\delta)$.
Combining these high-probability bounds and another union bound over $\cCS$ with \cref{thm:freedman}\index{Freedman's inequality} 
shows that with probability at least $1 - \delta$ for all $u \in \cCS$ and any $\Sigma \in \cS$,
\begin{align*}
\left|\sum_{t=1}^\tau \Delta_{t,u} \right| 
&\leq \sqrt{d L \sum_{t=1}^\tau \E_{t-1}[\Delta_{t,u}^2]} + dL \max_{1 \leq t \leq \tau} |\Delta_{t,u}| \\
&\leq \lambda \norm{u}^2_{\Sigma^{-1}} L^2 \left[1 + \sqrt{d \sum_{t=1}^\tau \E_{t-1}[Y_t^2]} + d^2 \Ymax\right]\,, 
\end{align*}
where we used the assumption that $\Sigma_t^{-1} \preceq \Sigma^{-1}$ for all $t \leq \tau$.
The claim is finished by a Lipschitz argument:

\begin{exer}
\faStar \quad
Complete the proof by showing that with high-probability $u \mapsto \Delta_{t,u}$ is suitably Lipschitz and using the properties of the cover
$\cC$.
\end{exer}
\end{proof}

\begin{remark}\label{rem:conc-hessian-path}
Suppose that $(a_t)_{t=1}^n$ is a sequence such that $a_t \in [0,C]$ almost surely and $a_t$ is $\sF_{t-1}$-measurable for all $t$.
Redefine 
\begin{align*}
S_\tau = \sum_{t=1}^\tau a_t \hat s''_t(\mu_t) \quad \text{ and } \quad 
\bar S_\tau = \sum_{t=1}^\tau a_t s''_t(\mu_t) \,.
\end{align*}
Then \cref{prop:conc-hessian-path} continues to hold with essentially the same proof.
\end{remark}

Finally, the gradient estimates of the surrogate loss also concentrate.

\begin{proposition}\label{prop:opt:conc-gradient}
Let $\sS \subset \psd$ be the (random) set of positive definite matrices such that $\Sigma_t^{-1} \preceq \Sigma^{-1}$ for all $t \leq \tau$.
The following hold:
\begin{enumerate}
\item Suppose that $x \in K_\tau(1/\sqrt{2L})$ almost surely. Then, for any $u \in \R^d$,
with probability at least $1 - \delta$, for any $\Sigma \in \cS$, \label{prop:opt:conc-gradient:single}
\begin{align*}
\left|\sum_{t=1}^\tau \ip{\hat s_t'(x) - s_t'(x), u}\right| \leq \norm{u}_{\Sigma^{-1}} L \left[1 + \sqrt{V_\tau} + \Ymax\right] \,.
\end{align*}
\item With probability at least $1 - \delta$ for all $x \in K_\tau(1/\sqrt{2L})$, $u \in \R^d$ and $\Sigma \in \cS$,
\label{prop:opt:conc-gradient:uniform}
\begin{align*}
\left|\sum_{t=1}^\tau \ip{\hat s_t'(x) - s_t'(x), u}\right| \leq \norm{u}_{\Sigma^{-1}} L \left[1 + \sqrt{d V_\tau} + d^2 \Ymax\right] \,.
\end{align*}
\end{enumerate}
\end{proposition}

\begin{proof}
Let $u \in \R^d$ and $x$ be such that $x \in K_\tau(1/\sqrt{2L})$ almost surely.
By definition,
\begin{align*}
\ip{\hat s_t(x), u} = \frac{\bar r_t(x) Y_t}{1-\lambda} \ip{u, \Sigma_t^{-1} \left(\frac{X_t - \lambda x}{1-\lambda} - \mu_t\right)} \,.
\end{align*}
Let $\Delta_t = \ip{\hat s_t(x), u} - \E_{t-1}[\ip{\hat s_t(x), u}]$. As usual, we need to bound $\E_{t-1}[\Delta_t^2]$ and $\max_{1 \leq t \leq \tau} |\Delta_t|$. 
Let
\begin{align*}
Q_t = \frac{\bar r_t(x)}{1-\lambda} \ip{u, \Sigma_t^{-1} \left(\frac{X_t - \lambda x}{1-\lambda} - \mu_t\right)}\,,
\end{align*}
which is chosen so that $\Delta_t = Y_t Q_t$.
By definition $\bar r_t(x) \leq \exp(2)$ and therefore
\begin{align*}
\norm{Q_t}_{t-1,\psi_2} \leq C \norm{u}_{\Sigma^{-1}} \left[1 + \lambda \norm{x - \mu_t}_{\Sigma_t^{-1}} \right] 
\leq 2 C \norm{u}_{\Sigma_t^{-1}} \,.
\end{align*}
Hence, by \cref{lem:orlicz-var},
\begin{align*}
\E_{t-1}[\Delta_t^2] \leq \norm{u}_{\Sigma_t^{-1}}^2 \left(\delta + \E_{t-1}[Y_t^2]\right) L \,.
\end{align*}
Also by \cref{lem:orlicz-var} in combination with \cref{lem:orlicz-tail} and a union bound, with probability at least $1 - \delta/2$, for all $t \leq \tau$
\begin{align*}
|\Delta_t| \leq \norm{u}_{\Sigma_t^{-1}} \left(\delta + \Ymax\right) \sqrt{L} \,.
\end{align*}
Part~\ref{prop:opt:conc-gradient:single} now follows from \cref{thm:freedman}
and \cref{prop:opt:bias}. \index{Freedman's inequality}
For part~\ref{prop:opt:conc-gradient:uniform}, combine the above with the covering argument in \cref{prop:conc-s}, covering 
both $K_1(1/\sqrt{2L})$ and $\sphere_1$. 
\end{proof}

\begin{exer}
\faStar \quad
Prove \cref{prop:opt:conc-gradient}\ref{prop:opt:conc-gradient:uniform}.
\end{exer}

\index{concentration!of surrogate|)}

\section{Summary}\label{sec:summary}

Let us summarise what has been shown.
The surrogate loss function is convex (\cref{lem:basic}\ref{lem:basic:cvx}) and optimistic:
\begin{align*}
\tag*{\cref{lem:basic}\ref{lem:basic:opt}}
s(x) \leq f(x) \text{ for all } x \in \R^d\,.
\end{align*}
On the other hand, the surrogate evaluated at $\mu$ is relatively close to $f(\mu)$:
\begin{align*}
\tag*{\cref{prop:lower}}
\E[f(X)] \leq s(\mu) + \frac{2}{\lambda} \tr(s''(\mu) \Sigma) + \frac{2 \delta d}{\lambda} \,.
\end{align*}
Furthermore, the quadratic surrogate offers the same benefits on the focus region.\index{focus region} Specifically, 
\cref{cor:gauss:lower} shows that for any $x$ such that $\lambda \norm{x - \mu}_{\Sigma^{-1}} \leq \frac{1}{L}$,
\begin{align*}
\E[f(X)] - f(x) \leq q(\mu) - q(x) + \frac{2}{\lambda} \tr(s''(\mu) \Sigma) + \delta\left[\frac{2d}{\lambda} + \frac{1}{\lambda^2}\right] \,.
\end{align*}
The effectiveness of the quadratic surrogate arises from the fact that $s$ is nearly quadratic on the focus region. 
\cref{prop:s-sc} shows that provided $\lambda \norm{x - y}_{\Sigma^{-1}} \leq L^{-1/2}$, then $s''(x) \preceq 2 s''(y) + \delta \Sigma^{-1}$.

\subsubsection*{Sequential Concentration}
Recall the notation of the sequential setting explained in \cref{sec:conc-seq}; particularly, that
\begin{align*}
K_\tau(r) = \left\{x \in K \colon \max_{t \leq \tau} \lambda \norm{x - \mu_t}_{\Sigma_t^{-1}} \leq r\right\}\,.
\end{align*}
Remember also that $V_\tau = \sum_{t=1}^\tau \E_{t-1}[Y_t^2]$ and $\Ymax = \max_{1 \leq t \leq \tau} (|Y_t| + \E_{t-1}[|Y_t|])$.
The following results hold under \cref{ass:opt:conc}.
The surrogate is well-concentrated in the sense that by \cref{prop:conc-s},
\begin{enumerate}
\item for $x \in \R^d$ such that $x \in K_\tau(1/\sqrt{L})$ almost surely, with probability at least $1 - \delta$,
\begin{align*}
\left|\sum_{t=1}^\tau (\hat s_t(x) - s_t(x))\right| \geq 1 + \frac{1}{\lambda} \left[\sqrt{L V_\tau} + L \Ymax\right] \,;
\end{align*}
\item with probability at least $1 - \delta$,
\begin{align*}
\sup_{x \in K_\tau(1/\sqrt{2L})} \left|\sum_{t=1}^\tau (\hat s_t(x) - s_t(x))\right| \leq 1 + \frac{1}{\lambda}\left[\sqrt{d L V_\tau} + d L \Ymax\right] \,.
\end{align*}
\end{enumerate}
Similar results hold for the quadratic surrogate. Precisely, by \cref{prop:conc-q}, 
\begin{enumerate}
\item given any $x \in \R^d$ such that $x \in K_\tau(1/\sqrt{L})$ almost surely, with probability at least $1 - \delta$,
\begin{align*}
\left|\sum_{t=1}^\tau \hat q_t(x) - q_t(x)\right| \leq 1 + \frac{1}{\lambda}\left[\sqrt{V_\tau L} + \Ymax L\right] \,;
\end{align*}
\item with probability at least $1 - \delta$,
\begin{align*}
\sup_{x \in K_\tau(1/\sqrt{2L})} \left|\sum_{t=1}^\tau \hat q_t(x) - q_t(x)\right| \leq 1 + \frac{1}{\lambda}\left[\sqrt{d V_\tau L} + d \Ymax L\right] \,.
\end{align*}
\end{enumerate}
The Hessian estimates are also reasonably well-behaved. Recall that $\cS$ is the random set of matrices $\Sigma^{-1}$ such that $\Sigma_t^{-1} \preceq \Sigma^{-1}$
for all $t \leq \tau$. 
Then, with probability at least $1 - \delta$, for all $\Sigma \in \cS$,
\begin{align*}
- \lambda L^2 \left[1 + \sqrt{d V_\tau} + d^2 \Ymax\right] \Sigma^{-1} 
\preceq S_\tau - \bar S_\tau 
\preceq \lambda L^2 \left[1 + \sqrt{d V_\tau} + d^2 \Ymax\right] \Sigma^{-1}\,,
\end{align*}
where $S_\tau = \sum_{t=1}^\tau \hat s''_t(\mu_t)$ and $\bar S_\tau = \sum_{t=1}^\tau s''_t(\mu_t)$.
Lastly, the gradient estimates concentrate.
Let $\cS$ be as above. The following hold:
\begin{enumerate}
\item For $x \in \R^d$ with $x \in K_\tau(1/\sqrt{2L})$ almost surely and $u \in \R^d$, with probability at least $1 - \delta$, for any $\Sigma \in \cS$,
\begin{align*}
\left|\sum_{t=1}^\tau \ip{\hat s_t'(x) - s_t'(x), u}\right| \leq \norm{u}_{\Sigma^{-1}} L \left[1 + \sqrt{V_\tau} + \Ymax\right] \,.
\end{align*}
\item With probability at least $1 - \delta$ for all $x \in K_\tau(1/\sqrt{2L})$ and all $u \in \R^d$ and any $\Sigma \in \cS$,
\begin{align*}
\left|\sum_{t=1}^\tau \ip{\hat s_t'(x) - s_t'(x), u}\right| \leq \norm{u}_{\Sigma^{-1}} L \left[1 + \sqrt{d V_\tau} + d^2 \Ymax\right] \,.
\end{align*}
\end{enumerate}

\section{Notes}

\begin{enumeratenotes}
\item The optimistic surrogate was introduced in a slightly different form by \cite{BEL16} and in the present form by \cite{LG21a}. 
The quadratic approximation was first used by \cite{LG23}, who proved most of the results in this chapter or variants thereof.
\item The parameter $\lambda$ determines the amount of smoothing. The change of measure in \cref{eq:opt:measure} blows up as $\lambda \geq 1/d$.
Meanwhile, for $\lambda \in (0, 1/d)$ there are trade-offs: 
\begin{itemizeinner}
\item A large value of $\lambda$ increases the power of the lower bound of \cref{prop:lower}, showing that $s$ is not too far below $f$.
\item A large value of $\lambda$ decreases the focus region\index{focus region} on which the quadratic surrogate is close to the non-quadratic surrogate and where
the concentration properties of the estimators are well-behaved.
\end{itemizeinner}
\end{enumeratenotes}

\chapter[Submodular Minimisation]{Submodular Minimisation\copynotice}\label{chap:submod}\index{bandit!submodular minimisation}

Let $[d] = \{1,\ldots,d\}$ for some integer $d$ and $\sP$ its powerset.
A function $f \colon \sP \to [0,1]$ is submodular if \index{submodular function}
for all $X \subset Y \subset [d]$ and $x \in [d] \setminus Y$, 
\begin{align*}
f(X \cup \{x\}) - f(X) \geq f(Y \cup \{x\}) - f(Y) \,.
\end{align*}
Submodular functions are sometimes viewed as a combinatorial analogue of convexity
via a gadget called the Lov\'asz extension that we explain in a moment.
In bandit submodular minimisation the adversary secretly chooses a sequence $(f_t)_{t=1}^n \colon \sP \to [0,1]$ of submodular functions.
Then, in each round $t$, the learner chooses a set $X_t \in \sP$ and observes $Y_t = f_t(X_t) + \eps_t$.
The optimal set is 
\begin{align*}
X_\star = \argmin_{X \in \sP} \sum_{t=1}^n f_t(X)
\end{align*}
and the regret definition is unchanged.
As usual, one can consider the stochastic version of the problem, where $f_t = f$ for some unknown $f$ and all $t$.
The raison d'\^etre of this chapter is to explain how convex bandit algorithms can be used for bandit submodular minimisation.
In particular, $\tilde O(d^{1.5} \sqrt{n})$ regret is possible in the stochastic setting by combining \cref{alg:ons:bandit} with the Lov\'asz extension;\index{setting!stochastic}
and in the adversarial setting\index{setting!adversarial} \cref{alg:ons-adv:bandit} yields a regret bound of $\tilde O(d^{2.5} \sqrt{n})$.
Besides this we explain how the special structure of bandit submodular minimisation means that the algorithm in \cref{chap:sgd} can be improved
to have regret $O(d n^{2/3})$ and $O(d^3/\eps^2)$ sample complexity.

\begin{remark}
The classical optimisation problem of finding the minimum of a submodular function $f \colon \sP \to [0,1]$ is quite interesting and we give some pointers
in Note~\ref{note:sub:opt}.
\end{remark}

Many problems in economics and operations research have some kind of submodularity based on the principles of diminishing returns or economies of scale.
Consider the following toy example. A specialty chocolate manufacturer is considering offering a subset of $[d]$ items on their website.
The expected earnings when offering item $k$ is some unknown $p(k)$ and for $X \subset [d]$ let $c(X)$ be the cost of offering subset $X$.
The expected loss (costs minus earnings) when offering $X$ is 
\begin{align*}
f(X) = c(X) - \sum_{k \in X} p(k) \,.
\end{align*}
Thanks to economies of scale one might expect that when $k \notin Y \supset X$ the cost of adding $k$ to $Y$ may be less than adding it to $X$: 
\begin{align*}
c(X \cup \{k\}) - c(X) \geq c(Y \cup \{k\}) - c(Y) \,,
\end{align*}
which implies that $f$ is submodular.
You can find many applications of submodular function minimisation in the survey by \cite{mccormick2005submodular}.

\section{Lov\'asz Extension}\index{Lov\'asz extension}
Let $f \colon \sP \to [0,1]$ be a submodular function. 
We can and will identify $\sP$ with $\{0,1\}^d$ in terms of the indicator function, so that $(1,1,\ldots,1) \equiv [d]$, $(0,0,\ldots,0) \equiv \emptyset$,
$(1,0,1,0,0,\ldots,0) = \{1,3\}$ and so on.
The Lov\'asz extension is a function $g \colon [0,1]^d \to [0,1]$ defined by
\begin{align}
g(x) = \int_0^1 f(\{i \colon x_i \geq u\}) \d{u} \,.
\label{eq:sub:int}
\end{align}
An illustration of the Lov\'asz extension is shown in \cref{fig:submod} and its integral representation above is shown in \cref{fig:submod-int}.
There are many ways to represent the Lov\'asz extension. A simple one is given in the following exercise:

\begin{exer}\label{ex:sub:mean}
\faStar \quad
Suppose that $U$ is uniformly distributed on $[0,1]$ and $S = \{i \colon x_i \geq U\}$. Show that \index{uniform measure}
$\E[f(S)] = g(x)$.
\end{exer}

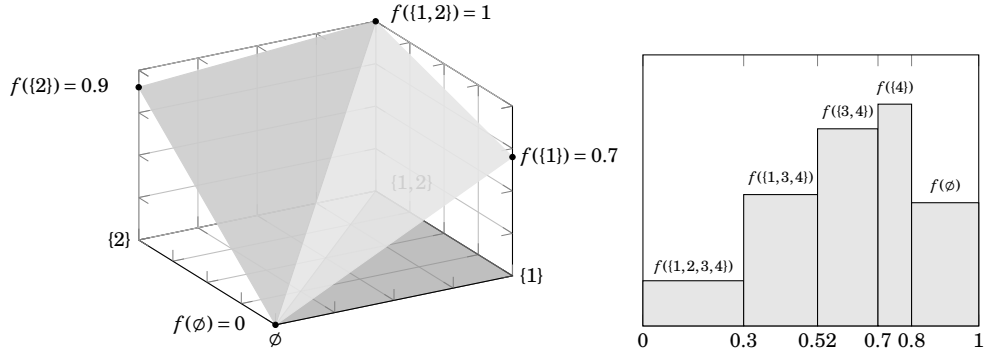
\begin{figure}
\centering
\begin{tikzpicture} 
\begin{axis}[
            width=7cm,
            grid=major,
            clip=false,
            xmin=0,xmax=1,
            ymin=0,ymax=1,
            zmin=0,zmax=1,
            xtick={0,0.25,0.5,0.75,1},
            ytick={0,0.25,0.5,0.75,1},
            xticklabels={,,,,},
            yticklabels={,,,,},
            ztick={0,0.25,0.5,0.75,1},
            zticklabels={,,,,},
            colormap={gray}{color(0cm)=(black); color(1cm)=(white)},
            view={-30}{30}, 
]
    \draw[fill=graythree,opacity=0.5] (axis cs:0,0,0) -- (axis cs:1,0,0) -- (axis cs:1,1,0) --cycle;
    \node at (axis cs:1,1,0) [right,xshift=0.1cm,yshift=0.1cm] {$\{1,2\}$};
    \addplot3 [
        patch,
        color=grayone,
        shader=flat,
        opacity=0.9,
        patch type=triangle] coordinates {
          (0,0,0) (1,0,0.7) (1,1,1)};
    \addplot3 [
        patch,
        color=graytwo,
        opacity=0.8,
        shader=flat,
        patch type=triangle] coordinates {
          (0,0,0) (0,1,0.9) (1,1,1) };
    \node at (axis cs:1, 1, 1) [right,xshift=0.1cm,yshift=0.1cm] {$f(\{1,2\}) = 1$};
    \node at (axis cs:1, 0, 0.7) [right] {$f(\{1\}) = 0.7$};
    \node at (axis cs:0, 1, 0.9) [xshift=-0.3cm,left] {$f(\{2\}) = 0.9$};
    \node at (axis cs:0, 0, 0) [left,xshift=-0.3cm] {$f(\emptyset) = 0$};
    \node at (axis cs:0, 0, 0) [below] {$\emptyset$};
    \node at (axis cs:1, 0, 0) [right] {$\{1\}$};
    \node at (axis cs:0, 1, 0) [left] {$\{2\}$};
    \draw[fill=black] (axis cs:0,0,0) circle[radius=1pt];
    \draw[fill=black] (axis cs:1,0,0.7) circle[radius=1pt];
    \draw[fill=black] (axis cs:0,1,0.9) circle[radius=1pt];
    \draw[fill=black] (axis cs:1,1,1) circle[radius=1pt];
\end{axis}
\end{tikzpicture}
\caption{The Lov\'asz extension for the submodular function defined by $f(\emptyset) = 0, f(\{1\}) = 0.7, f(\{2\}) = 0.9, f(\{1,2\}) = 1$.
The Lov\'asz extension is piecewise linear with each piece corresponding to a permutation $\sigma$ of $[d]$. The piece associated with permutation $\sigma$
is the simplex spanned by the sets $\emptyset, \{\sigma(1)\}, \{\sigma(1), \sigma(2)\}, \cdots, \{\sigma(1), \cdots, \sigma(d)\}$ where a set is associated with a corner
of the binary hypercube by its indicator function. In the figure $d = 2$ so there are only $d! = 2$ pieces.
}\label{fig:submod}
\commentAlt{A submodular function on a set of size 2 is shown as a 3-dimensional plot with the four possible sets on the corners of a hypercube
on the base and the Lovasz extension is piecewise linear with two pieces.}
\end{figure}

You should be able to see that if $x \in \{0,1\}^d$, then $g(x) = f(x)$, where we have abused notation by using the identification between $\{0,1\}^d$ and $\sP$.
The following classical theorem is what makes this chapter possible:

\begin{figure}
\begin{tikzpicture}[scale=1]
\begin{axis}[width=11.5cm,height=5cm,xmin=0,xmax=1,ymin=0,ymax=1.1,xtick={0,0.3,0.52,0.7,0.8,1},ytick=\empty,clip=false,xticklabel style={/pgf/number format/.cd,fixed,precision=3}]
\draw[fill=grayone] (axis cs:0,0) rectangle (axis cs:0.3, 0.1833);
\draw[fill=grayone] (axis cs:0.3,0) rectangle (axis cs:0.52,0.533);
\draw[fill=grayone] (axis cs:0.52,0) rectangle (axis cs:0.7,0.8);
\draw[fill=grayone] (axis cs:0.7,0) rectangle (axis cs:0.8,0.9);
\draw[fill=grayone] (axis cs:0.8,0) rectangle (axis cs:1,0.5);
\node[anchor=south] at (axis cs:0.15, 0.1833) {$f(\{1,2,3,4\})$};
\node[anchor=south,yshift=0.03cm] at (axis cs:0.41, 0.533) {$f(\{1,3,4\})$};
\node[anchor=south,yshift=0.03cm] at (axis cs:0.61, 0.8) {$f(\{3,4\})$};
\node[anchor=south,yshift=0.03cm] at (axis cs:0.75, 0.9) {$f(\{4\})$};
\node[anchor=south,yshift=0.03cm] at (axis cs:0.9, 0.5) {$f(\emptyset)$};
\end{axis}
\end{tikzpicture}
\caption{
An example of the integral computation in \cref{eq:sub:int} with $d = 4$, unspecified $f$ and $x = [0.52, 0.3, 0.7, 0.8]$; $g(x)$ is the area in grey.
}\label{fig:submod-int}
\commentAlt{On the right, is a bar chart showing the integral representation of a submodular function.}
\end{figure}

\begin{theorem}\label{thm:sub:cvx}
Let $g$ be the Lov\'asz extension of $f$. The following hold:
\begin{enumerate}
\item $g$ is convex;  \label{thm:sub:cvx:cvx}
\item $g$ is piecewise linear; and \label{thm:sub:cvx:lin}
\item $g$ is minimised on a vertex of the hypercube. \label{thm:sub:cvx:min}\index{hypercube}
\end{enumerate}
\end{theorem}

\begin{proof}
A sequence $S_1,\ldots,S_m \in \sP$ is a chain if $S_1 \subsetneq S_2 \subsetneq \cdots \subsetneq S_m$.
Remember that we identify $U \in \sP$ with an element in $\{0,1\}^d \subset [0,1]^d$.
The convex closure of $f$ is the function \index{convex closure}
\begin{align*}
h(x) = \min\left(\sum_{U \in \sP} p(U) f(U) \colon \sum_{U \in \sP} p(U) U = x,\, p \in \Delta(\sP)\right) \,,
\end{align*}
which is convex by the theory of linear programming \citep[\S5.2]{BT97}.\index{linear programming}
Suppose that $f$ is submodular. We will show that $g = h$.
Define $\phi(p) = \sum_{U \in \sP} p(U) |U|^2$ and let $x \in [0,1]^d$ be fixed.
By compactness, there exists a $p \in \Delta(\sP)$ such that 
\begin{align*}
h(x) = \sum_{U \in \sP} p(U) f(U) \quad \text{ and } \quad
x = \sum_{U \in \sP} p(U) U \,.
\end{align*}
In case of ties, let $p$ maximise $\phi$.
Suppose that $\{S \colon p(S) > 0\}$ is not a chain; that is, there exist $S, T \in \sP$ with $p(S) \geq p(T) > 0$ and $S \not \subset T$ and $T \not \subset S$.
Consider $q = p - p(T) \sind_T - p(T) \sind_S + p(T) \sind_{S \cap T} + p(T) \sind_{S \cup T}$. Without any assumptions, $\sum_{U \in \sP} q(U) U = x$ and
by submodularity, $f(S \cup T) + f(S \cap T) \leq f(S) + f(T)$, which implies that
$\sum_{U \in \sP} q(U) f(U) \leq \sum_{U \in \sP} p(U) f(U)$.
Furthermore,
\begin{align*}
|S \cup T|^2 + |S \cap T|^2 = |S|^2 + |T|^2 + 2|T \setminus S| |S \setminus T| > |S|^2 + |T|^2 \,.
\end{align*}
But from this it follows that $\phi(q) > \phi(p)$, contradicting our assumption that $p$ maximises $\phi$.
Therefore $\{S \colon p(S) > 0 \}$ is a chain and $\sum_{S \in \sP} p(S) S = x$. We leave you to prove that there is a unique chain satisfying these properties and to conclude
from \cref{eq:sub:int} that $h(x) = g(x)$.
Part~\ref{thm:sub:cvx:lin} follows since $h$ is piecewise linear \citep[\S5.2]{BT97}. 
Part~\ref{thm:sub:cvx:min} is immediate since $g$ is an average of $f$-values and the minimum is never larger than the average.
\end{proof}

\begin{exer}
\faStar \quad Finish the proof of \cref{thm:sub:cvx}\ref{thm:sub:cvx:cvx} as instructed above.
\end{exer}

The Lov\'asz extension also has nice computational properties. Given exact access to $f$ you can compute it by noticing that the integrand in
its definition is piecewise linear with at most $d+1$ pieces. You only need to evaluate $f$ at $d+1$ different sets, all of which are easily found
by sorting the coordinates of $x$.
Given $x \in [0,1]^d$ let $\sigma(\cdot|x) \colon [d] \to [d]$ be a permutation such that $k \mapsto x_{\sigma(k|x)}$ is non-increasing with ties broken arbitrarily.
We adopt the convention that $\sigma(0|x) = 0$, $\sigma(d+1|x) = d+1$, $x_0 = 1$ and $x_{d+1} = 0$.
Let $S(k|x) = \{\sigma(i|x) \colon i \in [k]\}$, which means that
\begin{align*}
\emptyset = S(0|x) \subset S(1|x) \subset \cdots \subset S(d|x) = [d] \,.
\end{align*}
When all coordinates of $x$ are distinct, then $S(k|x)$ contains exactly the coordinates associated with the $k$ largest entries of $x$.
Then, with $g$ the Lov\'asz extension of submodular function $f$,
\begin{align*}
g(x) 
&= \int_0^1 f(\{i \colon x_i \geq u\}) \d{u} \\
&= \sum_{k=0}^d f(S(k|x)) \left(x_{\sigma(k|x)} - x_{\sigma(k+1|x)}\right) \\
&= f(S(0|x)) + \sum_{k=1}^d x_{\sigma(k|x)} \left(f(S(k|x)) - f(S(k-1|x))\right)  \,. 
\end{align*}

There is also a nice expression for the subgradients\index{subgradient!of Lov\'asz extension} of the Lov\'asz extension.
Staring at the above equality yields the following standard proposition:

\begin{proposition}\label{prop:sub:sub}
Let $g$ be the Lov\'asz extension of a submodular function $f$, $x \in [0,1]^d$ and $\sigma^{-1}(\cdot|x)$ be the inverse of the permutation $\sigma(\cdot|x)$. Then 
the vector $u \in \R^d$ with
\begin{align*}
u_k = f(S(\sigma^{-1}(k|x)|x)) - f(S(\sigma^{-1}(k|x)-1|x)) 
\end{align*}
is a subgradient of $g$ at $x$.
\end{proposition}

\begin{exer}
\faStar \quad
Prove \cref{prop:sub:sub}.
\end{exer}

You may wonder what properties the Lov\'asz extension has. For example, is it smooth, strongly convex or Lipschitz?
A look at the definition reveals that it is piecewise linear and hence it cannot be strongly convex and is only smooth in the special case that it is linear.
It is $2$-Lipschitz, however:

\begin{proposition}[Lemma 1, \citealt{jegelka2011online}]
The Lov\'asz extension $g$ of a submodular function $f \colon \sP \to [0,1]$ satisfies $\lip(g) \leq 2$.
\end{proposition}

\begin{proof}
By \cref{thm:sub:cvx}, $g$ is convex
and piecewise linear so that $\lip(g) \leq \sup_x \norm{g'(x)}$ where
the supremum is over all $x$ where $g$ is differentiable.\index{differentiable}
Let $x \in (0,1)^d$ be any point where $g$ is differentiable and abbreviate $S(k) = S(k|x)$ and $\sigma(k) = \sigma(k|x)$. By \cref{prop:sub:sub},
\begin{align*}
g'_k(x) = f(S(\sigma^{-1}(k))) - f(S(\sigma^{-1}(k)-1)) \,.
\end{align*}
Let $P = \{k \colon g'_k(x) > 0\}$; then
\begin{align}
\norm{g'(x)} \leq \norm{g'(x)}_1 
&= \sum_{k \in P} g'_k(x) - \sum_{k \in [d] \setminus P} g'_k(x) \nonumber \\
&= 2\sum_{k \in P} g'_k(x) - \sum_{k \in [d]} g'_k(x) \,. 
\label{eq:submod:lip}
\end{align}
The last term telescopes:
\begin{align}
-\sum_{k \in [d]} g'_k(x) = f(\emptyset) - f([d]) \,.
\label{eq:submod:lip2}
\end{align}
For the other term in \cref{eq:submod:lip}, reorder the terms in the sum so that
\begin{align*}
2 \sum_{k \in P} g'_k(x) = 2 \sum_{m=1}^{|P|} (f(S(\pi(m))) - f(S(\pi(m) - 1)))\,,
\end{align*}
where $m \mapsto |S(\pi(m))|$ is increasing. This sum does not obviously telescope. Let us now make use of submodularity.
Let $a_m = S(\pi(m)) \setminus S(\pi(m)-1)$ and $A_m = \{a_1,\ldots,a_m\}$.
Since $\emptyset \subset S(1) \subset \cdots \subset S(d) = [d]$ is a chain, we have $A_m \subset S(\pi(m))$ and therefore
by submodularity,
\begin{align*}
f(A_m) - f(A_{m-1}) \geq f(S(\pi(m))) - f(S(\pi(m) - 1))\,.
\end{align*}
Therefore
\begin{align}
2 \sum_{m=1}^{|P|} (f(S(\pi(m))) - f(S(\pi(m) - 1))) 
&\leq 2 \sum_{m=1}^{|P|} (f(A_m) - f(A_{m-1})) \\
&= 2 f(A_{|P|}) - 2 f(\emptyset)\,.
\label{eq:submod:lip3}
\end{align}
Combining \cref{eq:submod:lip2} and \cref{eq:submod:lip3} with \cref{eq:submod:lip} shows that
\begin{align*}
\norm{g'(x)} &\leq 2 f(A_{|P|}) - f(\emptyset) - f([d]) \leq 2\,. 
\qedhere
\end{align*}
\end{proof}

\section{Bandit Submodular Minimisation}

We now explain how to use the Lov\'asz extension as a bridge between bandit convex optimisation on the hypercube and bandit submodular minimisation.
Let $g_t$ be the Lov\'asz extension of $f_t$ and $K = [0,1]^d$ the hypercube.\index{hypercube} 

\begin{exer}\label{ex:sub:linear}
\faStar \quad
Show that $f = \sum_{t=1}^n f_t$ is submodular and the Lov\'asz extension $f$ is $\sum_{t=1}^n g_t$.
\end{exer}

Your solution to \cref{ex:sub:linear} combined with \cref{thm:sub:cvx}\ref{thm:sub:cvx:min} shows that $\sum_{t=1}^n g_t$ is minimised at some $x_\star \in \{0,1\}^d$.
Consider a bandit convex optimisation algorithm playing on $K$ and let $(X^K_t)_{t=1}^n$ be actions in $K$ proposed by the bandit algorithms.
We need a way to select the real actions $X_t \subset [d]$ and return losses to the algorithm.
This is done by sampling $\lambda_t$ uniformly from $[0,1]$ and letting $X_t = \{i \colon (X^K_t)_i \geq \lambda_t\}$. Then the loss is $Y_t = f_t(X_t) + \eps_t$.
By \cref{ex:sub:mean}, $\E_{t-1}[Y_t] = g_t(X_t^K)$ so this procedure is equivalent to the learner interacting with the Lov\'asz extension sequence.
The following proposition shows that the regret of iterates $(X^K_t)$ with respect to the Lov\'asz extension implies a regret bound for the original bandit submodular optimisation
problem.

\begin{proposition}
Let $(X^K_t)_{t=1}^n \in K$ and let $(X_t)_{t=1}^n$ and $(Y_t)_{t=1}^n$ be defined as above. Then, with probability at least $1 - \delta$,
\begin{align*}
\Reg_n \triangleq \sum_{t=1}^n \left(f_t(X_t) - f_t(X_\star)\right)
&\leq \gReg_n + \sqrt{2n\log(1/\delta)}\,,
\end{align*}
where $\gReg_n = \sum_{t=1}^n \left(g_t(X_t^K) - g_t(x_\star)\right)$.
\end{proposition}

\begin{proof}
As we argued above, $\sum_{t=1}^n g_t(x_\star) = \sum_{t=1}^n f_t(X_\star)$.
Then
\begin{align*}
\gReg_n 
&= \sum_{t=1}^n \left(g_t(X_t^K) - g_t(x_\star)\right) \\
&= \sum_{t=1}^n \left(g_t(X_t^K) - f_t(X_\star)\right) \\
&= \sum_{t=1}^n \left(f_t(X_t) - f_t(X_\star)\right) + \sum_{t=1}^n \left(g_t(X_t^K) - f_t(X_t)\right) \\
&= \Reg_n + \sum_{t=1}^n \left(g_t(X_t^K) - f_t(X_t)\right)\,.
\end{align*}
By definition,
\begin{align*}
g_t(X_t^K) = \E_{t-1}[f_t(X_t)|X_t^K]\,.
\end{align*}
Therefore the sum is a martingale\index{martingale} with increments bounded in $[-1,1]$ and by Azuma's inequality (\cref{thm:azuma}), with probability at least $1 - \delta$,
\index{Azuma's inequality}
\begin{align*}
\Reg_n &\leq \gReg_n + \sqrt{2n \log(1/\delta)}\,.
\qedhere
\end{align*}
\end{proof}

Consequently, any algorithm for bandit convex optimisation can be used for submodular minimisation with very little overhead.
When looking at the complete list of algorithms in \cref{sec:table}, remember that for the hypercube we have $\diam(K) = \sqrt{d}$ and the self-concordance parameter is $\vartheta = \Theta(d)$. Moreover, except for scaling $K$ is already in L\"owner's position.\index{L\"owner's position}\index{hypercube}
The special structure of the Lov\'asz extension allows for at least one new idea, which we explain in \cref{sec:sub:sgd}.
\cref{tab:sub} more or less summarises the state of the art in bandit submodular optimisation. 

\begin{table}[h!]
\caption{Regret bounds for various algorithms for bandit submodular minimisation}
\label{tab:sub}
\renewcommand{\arraystretch}{1.6}
\centering \setlength\tabcolsep{4pt}
\begin{tabular}{|llll|}
\hline
\textsc{author} & \textsc{regret/comp} & \textsc{compute} & \textsc{notes} \\ \hline
\cite{hazan2012online} & \hyperref[thm:sub:sgd]{$d n^{2/3}$} & $O(d \log d)$ & \textendash \\
\cite{LFMV24} & \hyperref[thm:ons:bandit]{$d^{1.5} \sqrt{n}$} & $O(d^2) + \Pi + \textsc{svd}$ & stochastic only \\ 
\cite{LFMV24} & \hyperref[thm:ons-adv:bandit]{$d^{2.5} \sqrt{n}$} & $\poly(d, n)$ & \textendash \\ 
This book & \hyperref[thm:sub:sgd]{$\frac{d^3}{\eps^2}$} & $O(d \log(d))$ & stochastic only  \\ \hline
\end{tabular}
\end{table}

\FloatBarrier

\section{Gradient Descent for Submodular Minimisation}\label{sec:sub:sgd}\index{gradient descent!for submodular bandits|(}
By applying the algorithm in \cref{chap:sgd} you can immediately obtain a regret of $\E[\Reg_n] = O(d^{3/4} n^{3/4})$.
It is instructive to revisit gradient descent for submodular minimisation via the Lov\'asz extension because the special structure of the problem leads to a significant
improvement with almost no additional work.
At the same time,
in \cref{chap:overview} we promised to explain the phenomenon noted by \cite{Sha13} that for quadratic bandits\index{bandit!quadratic} the simple regret is $\Theta(d^2/n)$ while
the cumulative regret is $\Omega(d \sqrt{n})$. This curious situation is a consequence of two factors:
\begin{itemize}
\item in the simple regret setting the learner can afford to play actions far from the minimiser $x_\star$; and
\item in many parametric settings these actions can be far more informative than playing actions close to $x_\star$ because they allow the learner to reduce variance.
\end{itemize}
This situation arises in submodular bandits, for which a very simple algorithm has $\sReg_n = O(d^{1.5}/\sqrt{n})$ simple regret, 
while the algorithms with $\Reg_n = O(d^{1.5}\sqrt{n})$ are all sophisticated second-order methods.
The idea is to play gradient descent on the Lov\'asz extension and estimate the gradient by sampling from a subset of corners of the hypercube.
The cumulative regret incurred with this approach is linear but the variance of the gradient estimate is small, which leads to small simple regret.

\subsubsection*{Gradient Descent}
Algorithm~\ref{alg:sub:sgd} plays gradient descent on the Lov\'asz extension to produce a sequence of iterates $(x_t)_{t=1}^n$ but replaces the spherical smoothing
using in \cref{alg:sgd} with another mechanism for estimating the gradient.
With probability $\gamma \in (0,1)$ the algorithm explores and otherwise it exploits:
\begin{itemize}
\item When exploring, the algorithm samples $k_t$ uniformly on $\{0,1,\ldots,d\}$ and plays $S_t = S(k_t|x_t)$. The result can be used to estimate an element of $\partial g_t(x_t)$ 
using importance-weighting and \cref{prop:sub:sub}.
\item When exploiting, the algorithm samples $\lambda_t$ uniformly from $[0,1]$ and plays $\{k \colon (x_t)_k \geq \lambda_t\}$, which in expectation leads to a loss of $g_t(x_t)$.
\end{itemize}
When minimising the simple regret we choose $\gamma = 1$ so that the algorithm always explores.
Otherwise $\gamma$ is tuned to balance exploration and exploitation.

\begin{algorithm}[h!]
\begin{algcontents}
\begin{lstlisting}
args: $\eta > 0$, $\gamma \in (0,1)$
let $x_1 \in [0,1]^d$
for $t = 1$ to $n$
  sample $\lambda_t$ from $\sU([0,1])$ and let $S_t = \{k \colon (x_t)_k \geq \lambda_t\}$
  sample $k_t$ uniformly from $\{0,1,\ldots,d\}$ 
  let $E_t = \begin{cases}
    1 & \text{with prob. } \gamma \\
    0 & \text{with prob. } 1 - \gamma
  \end{cases}$ and $X_t = \begin{cases} 
  S(k_t|x_t) & \text{if } E_t = 1 \\
  S_t & \text{if } E_t = 0 
  \end{cases}$
  observe $Y_t = f(X_t) + \eps_t$
  let $\hat v_t = \frac{(d+1)Y_t E_t}{\gamma}\left[\sind(k_t \neq 0) e_{\sigma(k_t|x_t)} - \sind(k_t \neq d) e_{\sigma(k_t + 1|x_t)}\right]$
  update $x_{t+1} = \argmin_{x \in [0,1]^d} \norm{x - [x_t - \eta \hat v_t]}$
let $\widehat X_n = \frac{1}{n} \sum_{t=1}^n x_t$
sample $U$ from $\sU([0,1])$ 
return $\widehat S_n = \left\{k \colon (\widehat X_n)_k \geq U\right\}$
\end{lstlisting}
\caption{Gradient descent for bandit submodular simple regret minimisation}
\label{alg:sub:sgd}
\end{algcontents}
\end{algorithm}

\FloatBarrier

\begin{theorem}\label{thm:sub:sgd}
The following hold for \cref{alg:sub:sgd}:
\begin{enumerate}
\item Suppose that $\gamma = (1+d) n^{-1/3}$ and $\eta = \frac{1}{2} n^{-2/3}$. \label{alg:sub:sgd:adv}
Then the cumulative regret in the adversarial setting is bounded by
\begin{align*}
\E[\Reg_n] \leq 3(d+1) n^{2/3} = O(d n^{2/3})\,.
\end{align*}
\item Suppose that $\gamma = 1$ and $\eta = \frac{1}{2 + 2d} \sqrt{\frac{d}{n}}$. Then
in the stochastic setting\index{setting!stochastic} the simple regret of \cref{alg:sub:sgd} is bounded by \index{regret!simple}\label{alg:sub:sgd:stoch}
\begin{align*}
\E\left[\sReg_n\right] \leq 2(1+d) \sqrt{\frac{d}{n}} = O\left(\frac{d^{1.5}}{\sqrt{n}}\right) \,.
\end{align*}
\end{enumerate}
\end{theorem}

\begin{proof}
Since $\Reg_n \leq n$ for any algorithm, we assume for the remainder that $\gamma \in (0,1]$ since otherwise the claimed regret bound in Part~\ref{alg:sub:sgd:adv} holds
for any algorithm.
Let $g_t$ be the Lov\'asz extension of $f_t$ and
$v_t \in \partial g_t(x_t)$ be the subgradient\index{subgradient!of Lov\'asz extension} defined in \cref{prop:sub:sub}.

\begin{exer}\label{ex:sub:unbiased}
\faStar \quad
Show that $\E_{t-1}[\hat v_t] = v_t$.
\end{exer}

\solution{
Using the definitions
\begin{align*}
v_t 
&= \sum_{k=1}^d e_k \left[f(S(\sigma^{-1}(k|x)|x)) - f(S(\sigma^{-1}(k|x) - 1|x))\right] \\
&= (d+1) \sum_{k=1}^d e_k \E\left[Y_t\left(\sind(k_t = \sigma^{-1}(k|x)) - \sind(k_t = \sigma^{-1}(k|x)-1)\right)\right] \\
&= (d+1) \E_{t-1}\left[Y_t\left(\sind(k_t \neq 0) e_{\sigma(k_t|x)} - \sind(k_t \neq d)e_{\sigma(k_t+1|x)}\right)\right] \\
&= \E_{t-1}\left[\frac{(d+1)E_t Y_t}{\gamma} \left(\sind(k_t \neq 0)e_{\sigma(k_t|x)} - \sind(k_t \neq d)e_{\sigma(k_t+1|x)}\right)\right] \\
&= \E_{t-1}[\hat v_t] \,.
\end{align*}
}

We are now in a position to bound the regret of the iterates:
\begin{align}
\E\left[\sum_{t=1}^n (g_t(x_t) - g_t(x_\star))\right]
&\explana\leq \E\left[\sum_{t=1}^n \ip{v_t, x_t - x_\star}\right] \nonumber \\
&\explana\leq \frac{\diam(K)^2}{2\eta} + \frac{\eta}{2} \sum_{t=1}^n \E\left[\norm{\hat v_t}^2\right] \nonumber \\
&\explana= \frac{d}{2\eta} + \frac{\eta}{2} \sum_{t=1}^n \E\left[\norm{\hat v_t}^2\right] \nonumber \\
&\explana\leq \frac{d}{2\eta} + \frac{2 \eta n(d+1)^2}{\gamma} \,,
\label{eq:sub:sgd}
\end{align}
where \explanr{} follows from convexity of $g_t$,
\explanr{} by \cref{ex:sub:unbiased} and \cref{thm:sgd:abstract} and
\explanr{} since $\diam(K) = \sqrt{d}$.
\explanr{} follows because
\begin{align*}
\E\left[\norm{\hat v_t}^2\right]
\leq 2 \E\left[E_t \left(\frac{(d+1)Y_t}{\gamma}\right)^2\right] 
\leq 4 \E\left[E_t \left(\frac{d+1}{\gamma}\right)^2\right] 
\leq \frac{4(d+1)^2}{\gamma} \,,
\end{align*}
where we used \cref{eq:noise-var} and the assumption that $f_t$ is bounded in $[0,1]$ to bound
$\E_{t-1}[Y_t^2|E_t] = \E_{t-1}[(f_t(X_t) + \eps_t)^2|E_t] \leq 2$.
By definition the regret satisfies
\begin{align*}
\E[\Reg_n]
&= \E\left[\sum_{t=1}^n f_t(X_t) - f_t(x_\star)\right] \\
&= \E\left[\sum_{t=1}^n f_t(X_t) - g_t(x_t)\right] + \E\left[\sum_{t=1}^n g_t(x_t) - g_t(x_\star)\right] \\
&\leq n \gamma + \frac{d}{2\eta} + \frac{2\eta n(d+1)^2}{\gamma} \,,
\end{align*}
where in the final inequality we used the fact that
\begin{align*}
\E_{t-1}[f_t(X_t)]
&= \E_{t-1}[E_t f_t(X_t)] + \E_{t-1}[(1 - E_t) f_t(X_t)] \\
&= \E_{t-1}[E_t f_t(X_t)] + \E_{t-1}[(1 - E_t) f_t(S_t)] \\
&= \E_{t-1}[E_t f_t(X_t)] + (1 - \gamma) g_t(x_t) \\
&\leq \gamma + g_t(x_t) \,.
\end{align*}
The bound on the adversarial regret now follows by substituting the definition of the parameters.
To bound the simple regret in the stochastic setting where $f_t = f$ and $g = g_t$ is the Lov\'asz extension of $f$,
by substituting the parameters into \cref{eq:sub:sgd},
\begin{align}
\E\left[\sum_{t=1}^n (g(x_t) - g(x_\star))\right] \leq 2 (d+1) \sqrt{dn} \,.
\label{eq:sub:sgd-simple}
\end{align}
Then, using the definitions at the end of the algorithm and convexity of the Lov\'asz extension,
\begin{align*}
\E\left[f(\widehat S_n)\right]
\explanw{Ex.~\ref{ex:sub:mean}}
= \E\left[g(\widehat X_n)\right] 
\explanw{$g$ cvx}
\leq \E\left[\frac{1}{n} \sum_{t=1}^n g(x_t)\right] 
\explanw{by (\ref{eq:sub:sgd-simple})}
\leq g(x_\star) + 2 (d+1) \sqrt{\frac{d}{n}}\,.
\end{align*}
Since $g(x_\star) = f(X_\star)$ it follows that
\begin{align*}
\E[\sReg_n] &\leq 2 (d+1) \sqrt{\frac{d}{n}} \,. \qedhere
\end{align*}
\end{proof}

So what has been achieved?
\cref{alg:sub:sgd} is computationally practical and has excellent simple regret when $\gamma = 1$. 
On the other hand, the dependence of its regret on $n$ is suboptimal, though in some regimes it is theoretically superior to \cref{alg:ons-adv:bandit} in the adversarial setting
and its analysis and implementation are much simpler than \cref{alg:ons:bandit} in the stochastic setting.

\index{gradient descent!for submodular bandits|)}

\section{Notes}

\begin{enumeratenotes}

\item There are many resources for studying submodular functions and optimisation; for example, the book by \cite{bach2013learning} 
or the wonderful short survey by \cite{bilmes2022submodularity}.
\item \cref{alg:sub:sgd} is due to \cite{hazan2012online}, though the elementary simple regret analysis is new.
\item Bandit submodular \textit{maximisation}\index{bandit!submodular maximisation} is another topic altogether and has its own rich literature
\citep{gabillon2013adaptive,zhang2019online,foster2021submodular,chen2017interactive,takemori2020submodular,tajdini2024nearly,niazadeh2021online}.
Even in the classical optimisation setting without noise, exact submodular maximisation is computationally intractable. Approximately maximising submodular functions is often possible, however,
at least provided the constraints are reasonably well behaved. Because of this the standard approach in bandit submodular maximisation is to prove that
\begin{align*}
\sum_{t=1}^n f_t(X_t) \geq \alpha \max_{X \in \cC} \sum_{t=1}^n f_t(X) + o(n)\,,
\end{align*}
where $\alpha \in (0,1)$ is the approximation ratio,\index{approximation ratio} which depends on the constraints $\cC$ and assumptions on the functions $(f_t)$.
For example, when the functions $(f_t)$ are assumed to be submodular and monotone and $\cC = \{X \subset [d] \colon |X| \leq k\}$, then
\cite{niazadeh2021online} design an efficient algorithm such that
\begin{align*}
\left(1 - \frac{1}{e}\right) \max_{X \in \cC} \sum_{t=1}^n f_t(X) 
- \E\left[\sum_{t=1}^n f_t(X_t)\right] 
= \tilde O(kd^{2/3}n^{2/3}) \,.
\end{align*}
In the stochastic setting\index{setting!stochastic} the regret can be improved to $\tilde O(k d^{1/3} n^{2/3})$, which is essentially the best achievable \citep{tajdini2024nearly}.

\item The Lov\'asz extension is due to \cite{lovasz1983submodular}\index{Lov\'asz extension} 
and provides the interface between submodular bandits and convex bandits but introduces additional noise.
This is one justification for ensuring your algorithm can handle additional noise, even in the adversarial setting.\index{setting!adversarial}\index{noise}

\item \label{note:sub:opt} 
The complexity of minimising a submodular functions $f$ without noise is still an active area of research.
Let $g$ be the Lov\'asz extension of $f$ and $f_\star = \min_{x \in \{0,1\}^d} g(x)$.
By \cref{prop:sub:sub} a subgradient\index{subgradient!of Lov\'asz extension} of $g$ can be computed with $O(d)$ queries to the $f$.
Combining this with cutting plane methods \citep{bach2013learning,Bub15} shows that with $O(d^2 \log(d/\eps))$ queries to $f$ one can find
an $\widehat x \in [0,1]^d$ such that $g(\widehat x) \leq f_\star + \eps$. Then let
\begin{align*}
\widehat S = \argmin\{S(k|\widehat x) \colon 0 \leq k \leq d\}\,,
\end{align*}
which can be evaluated with another $O(d)$ queries to the loss function $f$ and satisfies $f(\widehat S) \leq f_\star + \eps$.
The discrete structure of submodular optimisation means you can even achieve exact minimisation \citep{jiang2022minimizing}.
\end{enumeratenotes}

\chapter[Outlook]{Outlook\copynotice}

The tool-chest for convex bandits and zeroth-order optimisation has been steadily growing in recent decades. 
Nevertheless, there are many intriguing open questions, both theoretical and practical. The purpose of this short chapter is to highlight
some of the most important (in the author's view, of course) open problems. 

\begin{enumeratenotes}
\item The most fundamental problem is to understand the minimax regret for $\cF_\pb$. The lower bound is $\Omega(d \sqrt{n})$ and the upper bound is $\tilde O(d^{1.5} \sqrt{n})$ 
in the stochastic setting\index{setting!stochastic} and $\tilde O(d^{2.5} \sqrt{n})§$ in the adversarial setting.  
\item From a practical perspective the situation is still relatively dire for $d > 1$. 
The algorithms that are simple and efficient to implement have slow convergence rates without smoothness and strong convexity.  
Algorithms with fast convergence rates have awkward assumptions. For example, online Newton step\index{online Newton step} learns fast for $\cF_\pb^\ps$ and
is difficult to tune.
Is there a simple algorithm that works well in practice without too much tuning and obtains the fast rate?
\item Algorithms that manage $\tilde O(\poly(d)\sqrt{n})$ regret without smoothness and strong convexity all estimate some kind of curvature\index{curvature} or use continuous exponential weights.\index{continuous exponential weights}
In particular, they use $\Omega(d^2)$ memory.\index{memory} Can you prove this is necessary?
\item 
In the stochastic setting both the range of the loss function and the Orlicz norm (or maybe variance) of the noise should appear in the regret. This is hidden throughout because
in most places we (questionably) opted to fix the range to $[0,1]$ and the Orlicz norm of the noise to $1$.
If you repeat the analysis naively for most settings you will find that for losses bounded in $[0,B]$ and Orlicz norm bound of $\sigma$, the leading term in the regret is $B + \sigma$.
For example, for online Newton step with $K$ in L\"owner's position we would have a regret of $\tilde O((B+\sigma) d^{1.5} \sqrt{n})$.
Really, however, the range of the losses should appear in a lower-order term, since over time an algorithm with sublinear regret will play in a region 
where the range of the losses is small (otherwise 
it would have high regret). Handling this properly in the analysis is probably quite complicated and maybe not fundamentally that interesting. But 
it could lead to practical improvements in many
cases. Note, we did things properly in \cref{chap:bisection} where the loss was not assumed to be bounded and also in \cref{sec:sc:stoch}.
The techniques developed there may be adaptable to other algorithms, including cutting plane methods and online Newton step.

\item More adaptive algorithms are needed.\index{adaptive} We have seen a plethora of results for this algorithm or that with such-and-such assumptions. But what if you don't know
what class the losses lie in. Can you adapt? What is the price? Very few works consider this or have serious limitations. A good place to start is the paper by \cite{LZZ22}.
We have also assumed that $n$ is known upfront and used this to tune learning rates or other parameters. You can always use a doubling trick \citep{BK18},\index{doubling trick} but generally speaking you would
expect better practical performance by using a decaying learning rate. We would not expect to encounter too many challenges implementing this idea, but the devil may be in the details.
For example, the set $K_\eps$ usually depends on $n$ via $\eps$ in many algorithms. And in the analysis of online Newton step the definition of the extended loss would become time-dependent.

\item There is scope to refine the algorithms and analysis in this text to the non-convex case;\index{non-convex} 
of course, proving bounds relative to a stationary point rather than a global minimum.
Someone should push this program through. Alternatively, one may still focus on finding the global minimum but with weaker assumptions such as quasi-convexity \index{quasiconvex}
or losses that satisfy the Polyak--Lojasiewicz condition \citep{polyak1963gradient,karimi2016linear,akhavan2024gradient}. \index{Polyak--Lojasiewicz condition}
Note that these function classes are not closed under addition and hence are not amenable to the adversarial setting.

\item Almost all of the properties we proved for the optimistic surrogate relied on Gaussianity of the exploration distribution.
Two properties that do not rely on this, however, are optimism and convexity. This leaves hope that something may be possible using an exponential weights distribution
rather than a Gaussian and this may interact better with the geometry of the constraint set. This seems to have been the 
original plan of \cite{BEL16} before they resorted to approximating exponential weights distributions by Gaussians. Perhaps you can make it work.

\item \cite{SRN21} and \cite{BEL16} both handle adversarial problems by some sort of test to see if the adversary is moving the minimum and proving that if this occurs, then
the regret must be negative and it is safe to restart the algorithm.\index{restart} One might wonder if there is some black-box procedure to implement this program so that any algorithm
designed for the stochastic setting can be used in the adversarial setting.
\item It would be fascinating to gain a better understanding of \cref{alg:exp-by-opt}. What loss estimators does it use? Maybe you can somehow implement this
algorithm when $d = 1$ or derive analytically what the estimators look like for special cases.

\item There is potential to unify the algorithms and analysis from stochastic approximation\index{stochastic approximation} and bandit convex optimisation.
The former are generally focused on precise asymptotics while the latter concentrate on minimax finite-time regret. At the moment there is very little integration between
these fields, even though many of the ideas are the same.

\item You could spend a huge amount of time generalising the conditions on the noise; for example, to heavy-tailed distributions. This has been widely 
explored in the multi-armed bandit setting \citep[and citations to/from]{BCL13}.
Probably this should only be done if you have a particular application in mind.
Alternatively, you could investigate heteroscedastic or multiplicative noise, with an initial foray by \cite{zhan2025regularizedonlinenewtonmethod}. \index{noise!heteroscedastic} \index{noise!multiplicative}

\item When $d = 1$ the best bound on the cumulative regret for losses in $\cF_\pb$ is $O(\sqrt{n} \log(n))$ while the lower bound is $\Omega(\sqrt{n})$.
Maybe this setting is the best place to start trying to understand whether or not the logarithmic factors are necessary.

\item This entire book is about bandit convex optimisation on subsets of euclidean space.
Convexity and the convex bandit problem can be generalised to Riemannian manifolds. 
\cite{ao2025riemannian} start this program by constructing the Riemannian analogue of \cref{alg:ftrl:basic} and proving that under appropriate smoothness conditions
its regret is $O(n^{2/3})$. The extent to which other algorithms and analyses in this book can be generalised to the non-euclidean setting is probably quite a fascinating
question.

\end{enumeratenotes}

\appendix

\chapter{Miscellaneous}

\section{Identities}

\begin{proposition}[\S4.9.1.4, \citealt{CRC18}]\label{prop:vol}
Let $\Gamma(\cdot)$ be Euler's gamma function. Then:
\begin{enumerate}
\item $\vol(\ball_r) = \frac{r^d \pi^{d/2}}{\Gamma\left(\frac{d}{2} + 1\right)}$. \label{prop:vol:ball} 
\item $\vol(\sphere_r) = \frac{d}{r} \vol(\ball_r)$. \label{prop:vol:sphere}
\end{enumerate}
\end{proposition}

The next proposition is the standard formula for integration in spherical coordinates, which follows from the co-area formula
\citep[Appendix C]{evans2018measure} or by direct proof.

\begin{proposition}\label{prop:radial-int}
Suppose that $f \colon \R^d \to \R$ satisfies $f(x) = g(\norm{x})$ for some $g \colon \R^d \to \R$.
Then, as long as either the left-hand or right-hand side below is well-defined,
\begin{align*}
\int_{\ball_r} f(x) \d{x}
&= d \vol(\ball_1) \int_0^r s^{d-1} g(s) \d{s} \,. 
\end{align*}
\end{proposition}

\begin{proposition}\label{prop:gaussian}
Suppose that $W$ has law $\cN(\zeros, \Sigma)$. Then:
\begin{enumerate}
\item $\E[\norm{W}^2] = \tr(\Sigma)$.
\item $\E[\norm{W}^4] = \tr(\Sigma)^2 + 2 \tr(\Sigma^2)$.
\item $\E[\exp(\ip{W, a})] = \exp\left(\frac{1}{2} \norm{a}^2_{\Sigma}\right)$.
\end{enumerate}
\end{proposition}

\begin{proof}
Let $Z$ have law $\cN(\zeros, \id)$. Then $W$ and $\Sigma^{1/2}Z$ have the same law. Therefore
$\E[\snorm{W}^2] = \E[\snorm{\Sigma^{1/2}Z}^2] = \E[Z^\top \Sigma Z] = \E[\tr(ZZ^\top \Sigma)] = \tr(\Sigma)$.
For the second part, since the euclidean norm is rotationally invariant we can assume without loss of generality that $\Sigma$ is diagonal.
Then write $\norm{W}^2$ as a sum, expand the square $\E[\norm{W}^4] = \E[\norm{W}^2 \norm{W}^2]$ and use the fact that $\E[Z_k^4] = 3$.
Finally, note that $X = \ip{W, a}$ has law $\cN(\zeros, \norm{a}^2_{\Sigma})$, which has moment-generating function 
$M_X(\lambda) \triangleq \E[\exp(\lambda X)] = \exp(\lambda^2 \norm{a}^2_{\Sigma}/2)$;
the result follows by substituting $\lambda = 1$.
\end{proof}

\begin{proposition}\label{prop:ball-mean}
Suppose that $X$ has law $\cU(\ball_1)$. Then $\E[\norm{X}] = \frac{d}{d+1}$.
\end{proposition}

\begin{proof}
By \cref{prop:radial-int},
\begin{align*}
\E[\norm{X}] 
&= \frac{1}{\vol(\ball_1)} \int_{\ball_1} \norm{x} \d{x}
= d \int_0^1 r^d \d{r} 
= \frac{d}{d+1} \,.
\qedhere
\end{align*}
\end{proof}

\section{Moore--Penrose Pseudoinverse}\label{sec:pinv}\index{Moore--Penrose pseudoinverse}

Given a matrix $A \in \R^{m \times n}$, the Moore--Penrose pseudoinverse is a matrix $A^+ \in \R^{n \times m}$ such that all of the following hold:
\begin{enumerate}
\item $AA^+A = A$.
\item $A^+ A A^+ = A^+$.
\item $(AA^+)^\top = AA^+$.
\item $(A^+A)^\top = A^+ A$.
\end{enumerate}

\begin{proposition}
The Moore--Penrose pseudoinverse of any $A \in \R^{m \times n}$ exists and is unique.
\end{proposition}

\begin{proof}
Let $A = U D V^\top$ be the singular value decomposition of $A$,\index{singular value decomposition} which means that $D$ is diagonal and $U$ and $V$ have orthonormal columns.
A straightforward calculation shows that if $D$ has diagonal $\lambda_1,\ldots,\lambda_k$, then $D^+$ is the diagonal matrix
with diagonal $\rho_1,\ldots,\rho_k$ and $\rho_i = 1/\lambda_i$ for $\lambda_i \neq \zeros$ and $0$ otherwise.
Then $A^+ = V D^+ U^\top$ straightforwardly satisfies the conditions of being a pseudoinverse.
For uniqueness, let $B$ and $C$ be two matrices satisfying the conditions.
Then
\begin{align*}
A B 
= A C A B 
= (AC)^\top (AB)^\top
= C^\top A^\top B^\top A^\top
= C^\top (A B A)^\top
= C^\top A^\top
= AC \,.
\end{align*}
Similarly, $BA = CA$.
Therefore 
$B = BAB = BAC = CAC = C$.
\end{proof}

\begin{fact}\label{fact:pinv}
Suppose that $A \in \R^{m \times n}$ is a matrix and $y \in \im(A^\top)$ and $\theta \in \R^n$.
Then $\ip{y, A^+ A \theta} = \ip{y, \theta}$.
\end{fact}

\begin{proof}
By assumption there exists a $w$ such that $y = A^\top w$ and so
$\ip{y, A^+A \theta} = \ip{A^\top w, A^+ A \theta} = \ip{w, A A^+ A \theta} = \ip{w, A \theta} = \ip{A^\top w, \theta} = \ip{y, \theta}$.
\end{proof}

\begin{fact}\label{fact:pinv-lowner}
Suppose that $A, B$ are positive semidefinite and $A \succeq B$. Then $A^+ \preceq B^+$ if and only if $\ker(A) = \ker(B)$.
\end{fact}

\begin{fact}\label{fact:pinv-trace}
Suppose that $A \in \R^{m \times n}$. Then $\tr(AA^+) = \tr(A^+ A) = \rank(A)$.
\end{fact}

\begin{proof}
Let $A = U D V^\top$ be the singular value decomposition of $A$ so that $U$ and $V$ have orthonormal columns and 
$D$ is diagonal with $\rank(A)$ nonzero entries.
Then
\begin{align*}
\tr(AA^+) = \tr(U DD^+ U^\top) = \tr(DD^+) = \rank(A) \,.
\end{align*}
Moreover, $\tr(A^+A) = \tr(D^+ D) = \tr(DD^+) = \rank(A)$.
\end{proof}

\section{Technical Inequalities}

\begin{lemma}\label{lem:tr-logdet}
Suppose that $A$ is positive definite and $A \preceq \ones$. Then $\tr(A) \leq 2 \log \det (\id + A)$.
\end{lemma}

\begin{proof}
Use the facts that the trace is the sum of the eigenvalues and the determinant is the product, and that $x \leq 2 \log(1+x)$ for $x \in [0,1]$.
\end{proof}

\begin{lemma}\label{lem:tech:log-sub}
Suppose $x,y  > 0$ and $x \id \preceq A \in \psd$. Then $\log \det(A + y \id) \leq \log \det(A) + \frac{d y}{x}$.
\end{lemma}

\begin{proof}
Let $(\lambda_k)_{k=1}^d$ be the eigenvalues of $A$.
The eigenvalues of $A + y \id$ are $(\lambda_k + y)_{k=1}^d$ and by concavity of the logarithm,
\begin{align*}
\log \det(A + y \id) 
&= \sum_{k=1}^d \log(\lambda_k + y \id) \\
&\leq \sum_{k=1}^d \left(\log(\lambda_k) + \frac{y}{\lambda_k}\right) \\
&\leq \sum_{k=1}^d \log(\lambda_k) + \frac{d y}{x} \\
&= \log \det(A) + \frac{d y}{x}\,.
\qedhere
\end{align*}
\end{proof}

\begin{lemma}\label{lem:tech:quadratic}
Suppose that $f(x) = \norm{x - y}^2_A$ and $g(x) = \norm{x - z}^2_B$ with $A, B \in \pd$. 
Then
\begin{align*}
\min_{x \in \R^d} (f(x) + g(x)) = \norm{z - y}^2_{A - A(A+B)^{-1} A}  \,.
\end{align*}
\end{lemma}

\begin{proof}
Let $h(x) = f(x) + g(x)$.
Then $h$ is quadratic and strictly convex and hence has a unique minimiser at $x \in \R^d$ with 
\begin{align*}
\zeros = h'(x) = 2A(x - y) + 2B(x - z) \,.
\end{align*}
Solving shows that $x = (A + B)^{-1} (Ay + Bz)$ and therefore
\begin{align*}
\min_{x \in \R^d} h(x)
&= \norm{(A+B)^{-1}(Ay + Bz) - y}^2_A + \norm{(A+B)^{-1}(Ay + Bz) - z}^2_B \\
&= \norm{(A+B)^{-1}(Bz - By)}^2_A + \norm{(A+B)^{-1}(Ay - Az)}^2_B \\
&= \norm{z - y}^2_{H} \,,
\end{align*}
where
\begin{align*}
H &= B (A+B)^{-1} A (A+B)^{-1} B + A (A+B)^{-1} B (A+B)^{-1} A \\
& = A (A+B)^{-1} B + A (A+B)^{-1}\left[-A (A+B)^{-1} B + B (A+B)^{-1} A\right] \\
&= A (A+B)^{-1} B \\
&= A - A(A+B)^{-1} A\,.
\qedhere
\end{align*}
\end{proof}

\chapter{Concentration}\label{app:conc}

\section{Orlicz Norms}\label{app:orlicz}\index{Orlicz norm|textbf}

Given a random variable $X$ and $k \in \{1, 2\}$ let \label{def:orlicz}
\begin{align*}
\norm{X}_{\psi_k} = \inf\left\{t > 0 \colon \E\left[\exp\left(|X/t|^k \right)\right] \leq 2 \right\}\,.
\end{align*}
The random variable $X$ is called subgaussian\index{subgaussian} if $\norm{X}_{\psi_2} < \infty$ and subexponential\index{subexponential} if $\norm{X}_{\psi_1} < \infty$.
As explained in detail by \cite{Ver18}, this definition is equivalent except for universal constants to the definitions
based on moments or the moment generating function, which appear, for example, in the book by \cite{BLM13}.
See also \cref{prop:subgauss-equiv}.

\begin{fact}\label{fact:orlicz}
For $k \in \{1,2\}$, $\norm{\cdot}_{\psi_k}$ is a norm on the corresponding Orlicz space, which is the subset of measurable functions $f$ such that $\norm{f}_{\psi_k} < \infty$ and
where functions that agree $\bbP$-almost everywhere are identified.\index{measurable}
In particular:
\begin{enumerate}
\item $\norm{X + Y}_{\psi_k} \leq \norm{X}_{\psi_k} + \norm{Y}_{\psi_k}$.
\item $\norm{aX}_{\psi_k} = a \norm{X}_{\psi_k}$ for all $a \geq 0$.
\item $\norm{X}_{\psi_k} = 0$ implies that $X = 0$ with probability $1$.
\end{enumerate}
\end{fact}

Regrettably, the Orlicz norms of constant functions do not behave exactly as you might expect: 
\begin{align}
\norm{1}_{\psi_1} = \frac{1}{\log(2)} \qquad \text{ and } \qquad \norm{1}_{\psi_2} = \frac{1}{\sqrt{\log(2)}} \,.
\label{eq:orlicz-1}
\end{align}
More generally, for bounded random variables:

\begin{lemma}\label{lem:orlicz-bound}
Suppose that $|X| \leq B$. Then 
\begin{enumerate}
\item $\norm{X}_{\psi_1} \leq \frac{B}{\log(2)}$,
\item $\norm{X}_{\psi_2} \leq \frac{B}{\sqrt{\log(2)}}$.
\end{enumerate}
\end{lemma}

\begin{proof}
Substitute the definitions.
\end{proof}

\begin{lemma}\label{lem:orlicz-tail}
Given any random variable $X$ and $t > 0$,
\begin{enumerate}
\item $\bbP(|X| \geq t) \leq 2\exp\left(-\frac{t}{\norm{X}_{\psi_1}}\right)$,\label{lem:orlicz-tail:exp}
\item $\bbP(|X| \geq t) \leq 2\exp\left(-\frac{t^2}{\norm{X}^2_{\psi_2}}\right)$. \label{lem:orlicz-tail:gauss}
\end{enumerate}
\end{lemma}

\begin{proof}
Both results follow from a standard method. For \ref{lem:orlicz-tail:exp},
\begin{align*}
\bbP(|X| \geq t)
&= \bbP\left(\exp\left(\frac{|X|}{\norm{X}_{\psi_1}}\right) \geq \exp\left(\frac{t}{\norm{X}_{\psi_1}}\right)\right) \\
\tag*{Markov's inequality}
&\leq 2 \exp\left(-\frac{t}{\norm{X}_{\psi_1}}\right)\,.
\end{align*}
Part~\ref{lem:orlicz-tail:gauss} is left as an exercise.
\end{proof}

\begin{lemma}\label{lem:orlicz-moment}
Let $\Gamma(\cdot)$ be the Gamma function.
Given any random variable $X$ and $k \geq 1$,
\begin{enumerate}
\item $\E[|X|^k] \leq 2\Gamma(1+k) \norm{X}_{\psi_1}^k$, \label{lem:orlicz-moment:1}
\item $\E[|X|^k] \leq 2\Gamma(1+k/2) \norm{X}_{\psi_2}^k$. \label{lem:orlicz-moment:2}
\end{enumerate}
\end{lemma}

\begin{proof}
Since $|X|$ is non-negative,
\begin{align*}
\E[|X|^k]
&= \int_0^\infty \bbP(|X|^k \geq t) \d{t} \\
&= \int_0^\infty \bbP(|X| \geq t^{1/k}) \d{t} \\
&\leq \int_0^\infty 2\exp\left(-\frac{t^{1/k}}{\norm{X}_{\psi_1}}\right) \d{t} \\ 
&= 2\Gamma(1+k) \norm{X}_{\psi_1}^k \,.
\end{align*}
Part~\ref{lem:orlicz-moment:2} follows by the same argument.
\end{proof}

Lemma~\ref{lem:orlicz-moment} can be refined when $k = 1$.

\begin{lemma}\label{lem:orlicz-mean}
Given any random variable $X$,
\begin{enumerate}
\item $\E[|X|] \leq \norm{X}_{\psi_1}$, \label{lem:orlicz-mean:1}
\item $\E[|X|] \leq \sqrt{\frac{1}{\log(2)}} \norm{X}_{\psi_2}$. \label{lem:orlicz-mean:2}
\end{enumerate}
\end{lemma}

\begin{proof}
For \ref{lem:orlicz-mean:1}
assume without loss of generality that $\norm{X}_{\psi_1} = 1$.
Using the inequality $x \leq \exp(x) - 1$,
\begin{align*}
\E[|X|] \leq \E[\exp(|X|)] - 1 \leq 1 \,.
\end{align*}
For \ref{lem:orlicz-mean:2} assume without loss of generality that $\norm{X}_{\psi_2} = 1$.
Then, by \cref{prop:orlicz:product} and \cref{eq:orlicz-1},
\begin{align*}
\E[|X|] &\leq \norm{X}_{\psi_1} \leq \norm{X}_{\psi_2} \norm{1}_{\psi_2} = \frac{1}{\sqrt{\log(2)}} \approx 1.2011 \ldots \,.
\qedhere
\end{align*}
\end{proof}

\begin{lemma}\label{lem:orlicz-center}
Given a random variable $X$,
\begin{enumerate}
\item $\norm{X - \E[X]}_{\psi_1} \leq (1 + \frac{1}{\log(2)}) \norm{X}_{\psi_1}$, \label{lem:orlicz-center:1}
\item $\norm{X - \E[X]}_{\psi_2} \leq (1 + \frac{1}{\log(2)}) \norm{X}_{\psi_2}$. \label{lem:orlicz-center:2}
\end{enumerate}
\end{lemma}

\begin{proof}
By the triangle inequality (\cref{fact:orlicz}), 
\begin{align*}
\norm{X - \E[X]}_{\psi_1} 
&\leq \norm{X}_{\psi_1} + \norm{\E[X]}_{\psi_1} \\
\tag*{\cref{lem:orlicz-bound}}
&\leq \norm{X}_{\psi_1} + \frac{1}{\log(2)} |\E[X]| \\
\tag*{Jensen's inequality}\index{Jensen's inequality}
&\leq \norm{X}_{\psi_1} + \frac{1}{\log(2)} \E[|X|] \\
\tag*{By \cref{lem:orlicz-mean}}
&\leq \norm{X}_{\psi_1} + \frac{1}{\log(2)} \norm{X}_{\psi_1} \,.
\end{align*}
Part~\ref{lem:orlicz-center:2} follows using the same argument.
\end{proof}

\begin{lemma}\label{lem:ent}
Suppose that $X$ is a non-negative random variable with $\E[X] = 1$. Then $\E[X \log X] \leq \log \E[X^2]$.
\end{lemma}

\begin{proof}
Let $\rho$ be the law of $X$ and define a measure $\nu$ by
$\nu(A) = \E[X \sind_A(X)]$, which is a probability measure by the assumption that $\E[X] = 1$.
By construction $\frac{\d{\nu}}{\d{\rho}}(x) = x$ so that
\begin{align*}
\int_{\R} x \log(x) \d{\rho}(x)
&= \int_{\R} \log(x) \d{\nu}(x) \\
&\leq \log\left(\int_\R x \d{\nu}(x)\right) \\
&= \log\left(\int_\R x^2 \d{\rho}(x)\right) \\
&= \log \E[X^2] \,,
\end{align*}
where the inequality follows from concavity of the logarithm.
\end{proof}

\begin{lemma}\label{lem:orlicz-var}
Let $\log_+(x) = \log(\max(1, x))$.
There exist absolute constants $(C_k)_{k=1}^\infty$ depending only on $k$ and with $C_1 = 1$ such that
for any non-negative random variables $X$ and $Y$ with $\E[X^2] < \infty$ and $\snorm{Y^{1/k}}_{\psi_1} < \infty$, it holds that  
\begin{align*}
\E[XY] 
&\leq C_k \E[X] \snorm{Y^{1/k}}_{\psi_1}^k \left(1 + \left(\log\frac{\E[X^2]}{\E[X]^2}\right)^k\right) \qquad \text{and}\\
E[XY] &\leq C_k \snorm{Y^{1/k}}_{\psi_1}^k \left(\E[X] + \E[X] \left(\log_+\frac{\E[X^2]}{\xi^2}\right)^k + \xi\right) \text{ for all } \xi > 0\,. 
\end{align*}
\end{lemma}

\begin{proof}
Suppose that $k = 1$.
The result is immediate if $\E[X]$ or $\E[Y]$ vanish. Hence, by normalising it suffices to consider the case where $\E[X] = 1$ and $\norm{Y}_{\psi_1} = 1$.
Let $f(y) = \exp(y)$ which has convex conjugate $f^*(x) = x \log x - x$.
By the Fenchel--Young inequality,\index{Fenchel--Young inequality}
\begin{align*}
\E[XY] 
\tag*{Fenchel--Young}
&\leq \E[f^*(X) + f(Y)] \\ 
&= \E[X \log X + \exp(Y)] - 1 \\
\tag*{since $\snorm{Y}_{\psi_1} = 1$} 
&\leq \E[X \log X] + 1 \\
\tag*{by \cref{lem:ent}}
&\leq 1 + \log \E[X^2] \,. 
\end{align*}
The second part follows since
\begin{align*}
\E[X] \log \left(\frac{\E[X^2]}{\E[X]^2}\right) 
&= \E[X] \log\left(\frac{\E[X^2]}{\xi^2}\right) + \E[X] \log\left(\frac{\xi^2}{\E[X]^2}\right) \\
&\leq \E[X] \log\left(\frac{\E[X^2]}{\xi^2}\right) + \xi \,,
\end{align*}
where the inequality follows because $\log(x) \leq \sqrt{x}$ for all $x > 0$.
The argument for $k > 1$ follows from the same high-level idea using the (non-convex)\index{non-convex} function $f(y) = \exp(y^{1/k})$
and is left as an exercise.
\end{proof}

\begin{exer}
\faStar\quad Prove \cref{lem:orlicz-var} with $k > 1$.
\end{exer}

\solution{%
Let $f(u) = \exp(u^{1/k})$. A tedious calculation shows that for $x \geq 0$,
\begin{align*}
f^*(x) = \sup_{u \geq 0} \left(x u - f(u)\right) \leq x\left(2 \log(1+x) + 2k^{e/(e-1)}\right)^k \triangleq x g(x)\,.
\end{align*}
By the Fenchel-Young inequality,
\begin{align*}
\E[XY] 
&\leq \E[f^*(X) + f(Y)] \leq \E[f^*(X)] + 2 \\
&\leq \E[X g(X)] + 2 \\
&\leq g(\E[X^2]) + 2 \\
&= \left(2 \log(1 + \E[X^2]) + 2k^{e/(e-1)}\right)^k + 2\,.
\end{align*}
The first part follows by naively simplifying and introducing the absolute constants.
For the second part use the fact that for $x \leq a$
$x \log(a/x)^k \leq C (x (\log_+(a/\xi))^k + \xi)$ for suitably large constant $C$ depending only on $k$.
}

The definition of subgaussianity based on moment generating functions is as follows.
Given a random variable $X$, let $M_X(\lambda) = \E[\exp(\lambda X)]$ be its moment generating function.
The set of subgaussian random variables with variance proxy $\sigma^2$ is 
\begin{align*}
\sG(\sigma) = \{X \colon M_X(\lambda) \leq \exp(\sigma^2 \lambda^2/2) \text{ for all } \lambda \in \R \} \,. 
\end{align*}
The next proposition connects $\sG(\sigma)$ to $\{X \colon \E[X] = 0, \norm{X}_{\psi_2} \leq \sigma\}$.
Similar results with slightly larger constants are given by \cite{BLM13,Ver18,zhang2020concentration}.

\begin{proposition}\label{prop:subgauss-equiv}
Suppose that $\E[X] = 0$. Then:
\begin{enumerate}
\item If $X \in \sG(\sigma)$, then $\norm{X}_{\psi_2} \leq \sqrt{8/3} \sigma$.
\label{prop:subgauss-equiv:1}
\item If $\norm{X}_{\psi_2} \leq \sigma$, then $X \in \sG(\sqrt{2} \sigma)$.
\label{prop:subgauss-equiv:2}
\end{enumerate}
\end{proposition}

\begin{remark}
When $X$ has law $\cN(0, 1)$, then \cref{prop:subgauss-equiv}\ref{prop:subgauss-equiv:1} holds with equality.
\end{remark}

\begin{proof}
In both parts assume without loss of generality that $\sigma = 1$.
For \ref{prop:subgauss-equiv:1}, let $a = \sqrt{8/3}$. Then
\begin{align*}
\E[\exp((X/a)^2)]
&= \E\left[\frac{1}{\sqrt{\pi}} \int_{-\infty}^\infty \exp\left(-t^2 + \frac{2 tX}{a}\right) \d{t} \right] \\
&= \frac{1}{\sqrt{\pi}} \int_{-\infty}^\infty \E\left[\exp\left(-t^2 + \frac{2 tX}{a}\right)\right] \d{t} \\
&\leq \frac{1}{\sqrt{\pi}} \int_{-\infty}^\infty \exp\left(-t^2 + \frac{2t^2}{a^2}\right) \d{t} \\
&= 2 \,.
\end{align*}
Therefore $\norm{X}_{\psi_2} \leq a$.
Moving now to \ref{prop:subgauss-equiv:2}, suppose that $|\lambda| \leq 1$. Then, 
\begin{align*}
\E[\exp(\lambda X)]
&\explana\leq \E[\exp(\lambda^2 X^2) + \lambda X] \\
&\explana\leq \lambda^2 \E[\exp(X^2)] + 1 - \lambda^2 \\
&\explana\leq 1 + \lambda^2 \\
&\explana\leq \exp(\lambda^2) \,, 
\end{align*}
where \explanr{} follows from the fact that $\exp(x) \leq x + \exp(x^2)$ for all $x \in \R$,
\explanr{} since $\E[X] = 0$ and by convexity of $\lambda^2 \mapsto \exp(x^2 \lambda^2)$,
\explanr{} since by assumption $\E[\exp(X^2)] \leq 2$ and
\explanr{} since $1 + x \leq \exp(x)$ for all $x \in \R$.
On the other hand, if $|\lambda| \geq 1$, then
\begin{align*}
\E[\exp(\lambda X)]
&= \E\left[\exp\left(\frac{\lambda}{\sqrt{2}} \sqrt{2} X\right)\right] \\
&\explana\leq \E\left[\exp\left(\frac{\lambda^2}{8} + X^2\right)\right] \\
&\explana\leq 2 \exp(\lambda^2/8) \\ 
&\explana\leq \exp(\lambda^2) \,,
\end{align*}
where \explanr{} follows from the Fenchel--Young inequality: $xy \leq x^2/2 + y^2/2$ for all $x, y \in \R$,
\explanr{} since $\E[\exp(X^2)] \leq 2$ by assumption and \explanr{} from the assumption that $|\lambda| \geq 1$.
\end{proof}

\begin{proposition}[Lemma 2.7.7, \citealt{Ver18}]\label{prop:orlicz:product}
Let $X$ and $Y$ be any random variables. Then $\norm{XY}_{\psi_1} \leq \norm{X}_{\psi_2} \norm{Y}_{\psi_2}$.
\end{proposition}

All the results in the next proposition can be found somewhere in the book by \cite{Ver18} but with non-explicit constants. The referenced paper gives
the explicit constant but is probably not the first to do so.

\begin{proposition}[\citealt{LG23}]\label{prop:orlicz}
Suppose that $W$ is a standard Gaussian random variable in $\R^d$. Then:
\begin{enumerate}
\item $\norm{\ip{x, W}}_{\psi_2} = 2 \sqrt{2/3} \norm{x}$.
\item $\norm{\tr(A W W^\top)}_{\psi_1} \leq 3 \tr(A)$.
\item $\norm{\norm{W}^2}_{\psi_1} \leq 8d/3$.
\item $\norm{\norm{WW^\top - \id}}_{\psi_1} \leq 5d$.
\end{enumerate}
\end{proposition}

Lastly we give a bound on the Orlicz norm of $\ip{U, \eta}$ where $U$ is uniformly distributed on\index{uniform measure}
the sphere. Morally this is comparable to the case where $U$ is $\cN(\zeros, \frac{1}{d} \id)$, as the bound shows.

\begin{proposition}\label{prop:orlicz:sphere}
Let $X = \ip{U, \eta}$ where $U$ has law $\cU(\sphere_1)$ and $\eta \in \R^d$.
Then
\begin{align*}
\norm{X}_{\psi_2} \leq \snorm{\eta}\sqrt{\frac{4}{3(d+1)}}\,.
\end{align*}
\end{proposition}

\begin{proof} 
Assume without loss of generality that $\norm{\eta} = 1$.
Let $Y$ be beta distributed with parameters $\alpha = \beta = d/2$, which has mean $1/2$.
Then $X/2+1/2$ has the same law as $Y$ and \cite{marchal2017sub} prove that for all $\lambda \in \R$,
\begin{align*}
\E[\exp(\lambda (Y - \E[Y]))] \leq \exp\left(-\frac{\lambda^2}{8(d + 1)}\right) \,.
\end{align*}
Therefore
\begin{align*}
\E[\exp(\lambda X)]
= \E[\exp((2\lambda) (Y - \E[Y]))]
\leq \exp\left(-\frac{\lambda^2}{4(d+1)}\right)
= \exp\left(-\frac{\lambda^2 \sigma^2}{2}\right)\,,
\end{align*}
where $\sigma^2 = \frac{1}{2(d+1)}$.
The result follows from \cref{prop:subgauss-equiv}.
\end{proof}

\section{Concentration}\index{concentration}

The following are classical:

\begin{theorem}[\citealt{duembgen2010bounding}]\label{thm:conc:gaussian}
Suppose that $X$ has law $\cN(0, 1)$. Then, for any $x \geq 0$,
\begin{align*}
\bbP(X \geq x) \leq \frac{1}{2} \exp\left(-\frac{x^2}{2}\right)\,.
\end{align*}
\end{theorem}

\begin{theorem}[\citealt{BLM13}, Theorem 5.6]\label{thm:orlicz-lip}
Suppose that $f \colon \R^d \to \R$ and $X$ has law $\cN(\mu, \Sigma)$. Then, for any $\delta \in (0,1)$,
\begin{align*}
\bbP\left(\left|\E[f(X)] - f(X)\right| \geq \lip(f) \sqrt{2 \norm{\Sigma} \log(2/\delta)}\right) \leq \delta\,.
\end{align*}
Furthermore, $\norm{f(X) - \E[f(X)]}_{\psi_2} \leq \lip(f) \sqrt{6 \norm{\Sigma}}$.
\end{theorem}

The next two theorems are versions of Hoeffding's and Bernstein's inequalities in terms of Orlicz norms. 
The constant $4$ appearing in Hoeffding's inequality follows by combining the bound using the definition of subgaussianity based on the moment generating function \citep{BLM13}
and \cref{prop:subgauss-equiv}.

\index{Hoeffding's inequality|textbf}
\begin{theorem}[Hoeffding's inequality]\label{thm:hoeffding}
Let $X_1,\ldots,X_n$ be a sequence of independent random variables with $\norm{X_t}_{\psi_2} \leq \sigma$ and $\E[X_t] = 0$. Then, for any $\delta \in (0,1)$,
\begin{align*}
\bbP\left(\left|\frac{1}{n} \sum_{t=1}^n X_t\right| \geq \sigma \sqrt{\frac{4\log(2/\delta)}{n}}\right) \leq \delta \,.
\end{align*}
\end{theorem}

\index{Bernstein's inequality|textbf}

\begin{theorem}[Bernstein's inequality, \citealt{pinelis2022improved}]\label{thm:bernstein}
Let $X_1,\ldots,X_n$ be a sequence of independent random variables with $\norm{X_t}_{\psi_1} \leq \sigma$ and $\E[X_t] = 0$ for all $1 \leq t \leq n$. Then,
for any $\delta \in (0,1)$,
\begin{align*}
\bbP\left(\left|\frac{1}{n} \sum_{t=1}^n X_t\right| \geq \max\left(\sqrt{\frac{4 \sigma^2 \log(2/\delta)}{n}},\, \frac{2 \sigma \log(2/\delta)}{n}\right)\right)
\leq \delta \,.
\end{align*}
\end{theorem}

\index{Azuma's inequality|textbf}
\begin{theorem}[Azuma's inequality]\label{thm:azuma}
Let $X_0,X_1,\ldots,X_n$ be a martingale\index{martingale} with $|X_t - X_{t-1}| \leq c_t$ almost surely. Then, with probability at least $1 - \delta$,
\begin{align*}
X_n \leq X_0 + \sqrt{2 \sum_{t=1}^n c_t^2 \log(1/\delta)}\,. 
\end{align*}
\end{theorem}

The next theorem is a strengthened version of Freedman's inequality due to \cite{zimmert22b}.
\index{Freedman's inequality|textbf}

\begin{theorem} \label{thm:freedman}
Let $X_1,\ldots,X_n$ be a sequence of random variables adapted to filtration\index{filtration} $(\sF_t)$ and $\tau$ be a stopping time\index{stopping time} with respect to $(\sF_t)_{t=1}^n$ with 
$\tau \leq n$ almost surely. Let $\E_t[\cdot] = \E[\cdot|\sF_t]$ and assume that $\E_{t-1}[X_t] = 0$ almost surely for all $t \leq \tau$. Then, with probability at least $1 - \delta$,
\begin{align*}
\left|\sum_{t=1}^\tau X_t \right| 
\leq 3 \sqrt{V_\tau \log\left(\frac{2 \max(B, \sqrt{V_\tau})}{\delta}\right)} + 2B \log\left(\frac{2\max(B, \sqrt{V_\tau})}{\delta}\right)\,,
\end{align*}
where $V_\tau = \sum_{t=1}^\tau \E_{t-1}[X_t^2]$ is the sum of the predictable variations and $B = \max(1, \max_{1 \leq t \leq \tau} |X_t|)$.
\end{theorem}

The next theorem is a folklore result. A version without stopping times appears as Exercise 5.15 in the book by \cite{LS20book}. 

\begin{theorem}\label{thm:conc:unnormalised}
Let $X_1,\ldots,X_n$ be a sequence of random variables adapted to filtration $(\sF_t)_{t=1}^n$ and $\tau$ be a stopping time with respect to the same filtration.
Suppose that $\alpha |X_t| \leq 1$ almost surely for all $t \leq \tau$ with $\alpha \geq 0$. Then, with probability at least $1 - \delta$,
\begin{align*}
\sum_{t=1}^\tau \left(X_t - \E_{t-1}[X_t]\right) \leq \alpha \sum_{t=1}^\tau \E_{t-1}[X_t^2] + \frac{\log(1/\delta)}{\alpha}\,,
\end{align*}
where $\E_t[\cdot] = \E[\cdot|\sF_t]$.
\end{theorem}

\begin{proof}
Let $\Delta_t = X_t - \E_{t-1}[X_t]$ and $S_\tau = \sum_{t=1}^\tau \Delta_t$ and $V_\tau = \sum_{t=1}^\tau \E_{t-1}[X_t^2]$.
By Markov's inequality,
\begin{align*}
\bbP\left(S_\tau \geq \alpha V_\tau + \frac{\log(1/\delta)}{\alpha} \right)
&= \bbP\left(\exp\left(\alpha S_\tau - \alpha^2 V_\tau\right) \geq \frac{1}{\delta}\right) \\
&\leq \delta \E\left[\exp\left(\alpha S_\tau - \alpha^2 V_\tau\right)\right] \\
&= \delta \E\left[M_\tau\right]\,, 
\end{align*}
where $M_t = \exp\left(\alpha S_t - \alpha^2 V_t\right)$.
Suppose that $t \leq \tau$. Then
\begin{align*}
\E_{t-1}[M_t] 
&= M_{t-1} \E_{t-1}\left[\exp\left(\alpha \Delta_t - \alpha^2 \E_{t-1}[X_t^2]\right)\right] \\ 
&= M_{t-1} \exp\left(-\E_{t-1}[\alpha^2 X_t^2 + \alpha X_t]\right) \E_{t-1}\left[\exp\left(\alpha X_t\right)\right] \\
&\leq M_{t-1} \exp\left(-\E_{t-1}[\alpha^2 X_t^2 + \alpha X_t]\right) \E_{t-1}\left[1 + \alpha X_t + \alpha^2 X_t^2\right] \\ 
&\leq M_{t-1} \,,
\end{align*}
where in the first inequality we used the fact that for $|x| \leq 1$, $\exp(x) \leq 1 + x + x^2$
and in the last that $1+x \leq \exp(x)$ for all $x$.
Hence $(M_t)$ is a supermartingale\index{martingale} and $M_0 = 1$ is immediate from the definition. Since $\tau$ is almost surely bounded, by the optional stopping theorem $\E[M_\tau] \leq M_0 = 1$.
Therefore 
\index{optional stopping theorem}
\begin{align*}
\bbP\left(S_\tau \geq \alpha V_\tau + \frac{\log(1/\delta)}{\alpha} \right)
&\leq \delta \,.
\qedhere
\end{align*}
\end{proof}

The following is an elementary corollary of \cref{thm:conc:unnormalised}.

\begin{theorem}\label{thm:conc:unnormalised2}
Let $X_1,\ldots,X_n$ be a sequence of non-negative random variables adapted to filtration $(\sF_t)_{t=1}^n$ and $\tau$ be a stopping time with respect to the same filtration.
Suppose that $\alpha X_t \leq 1$ almost surely for all $t \leq \tau$ with $\alpha \geq 0$. Then, with probability at least $1 - \delta$,
\begin{align*}
\sum_{t=1}^\tau X_t \leq 2\sum_{t=1}^\tau \E_{t-1}[X_t] + \frac{\log(1/\delta)}{\alpha}\,,
\end{align*}
where $\E_t[\cdot] = \E[\cdot|\sF_t]$.
\end{theorem}

\begin{proof}
Apply \cref{thm:conc:unnormalised} and bound $\alpha \E_{t-1}[X_t^2] \leq \E_{t-1}[X_t]$.
\end{proof}

\chapter{Notation}\label{app:notation}

{
\fontsize{7.5}{8}\selectfont \centering \setlength\tabcolsep{10pt}
\renewcommand{\arraystretch}{1.9}

\begin{longtable}{|lp{7cm}p{1cm}|}
\hline
$A + B$ & Minkowski addition, $A + B = \{a + b \colon a \in A, b \in B\}$ & \\
$-A$ & $-A = \{-a \colon a \in A\}$ & \\
$A - B$ & Minkowski subtraction, $A - B = A + (-B)$ & \\
$x + A$ & abbreviation for $\{x\} + A$ & \\
$u A$ & Minkowski multiplication, $uA = \{ua \colon a \in A\}$, $u \in \R$ & \\
\hline
$\subset$, $\supset$ & subset/superset, possibly equal & \\
$\succeq$, $\preceq$, $\succ$, $\prec$ & L\"oewner order operators & \\
$\ceil{\cdot}$, $\floor{\cdot}$ & ceiling/floor functions & \\
$f * g$ & convolution of $f$ and $g$ & \\
$\ip{x, y}$ & standard euclidean inner product & \\
$\ip{x, y}_A$ & inner product $x^\top A y$ for $A \in \pd$ & \\
\hline 
$\R$, $\mathbb Z$ & reals, integers & \\
$\R_{++}$, $\R_{+}$ & $(0,\infty)$ and $[0, \infty)$ & \\
$\pd$, $\psd$ & positive definite and positive semidefinite matrices acting on $\R^d$ & \\
$\ball_r$ & euclidean ball of radius $r$ in $\R^d$ & \\
$\sphere_r$ & euclidean sphere of radius $r$ embedded in $\R^d$ & \\
$E(x, A)$ & the ellipsoid $\{y \colon \norm{x - y}_{A^{-1}} \leq 1\}$ & \\
$H(x, \eta)$ & the half-space $\{y \colon \ip{y - x, \eta} \leq 0\}$ & \\
$[x, y]$ & convex chord $\{\lambda x + (1 - \lambda) y \colon x \in [0,1]\}$ & \\
\hline
$\sB(K)$ & Borel $\sigma$-algebra on $K$ & \\
$\Delta(K)$ & probability measures on $(K, \sB(K))$ & \\
$\Delta_m$ & probability distribution on $\{1,\ldots,m\}$ & \\
$\Delta_m^+$ & $\Delta_m \cap \R_{++}^m$ & \\
$\sU(A)$ & uniform probability measure on $A \subset \R^d$ & \index{uniform measure} \\
$\sF_t$ & natural filtration $\sigma(X_1, Y_1,\ldots,X_t,Y_t)$ & \cpageref{page:filtration} \\
$\bbP_t$ & conditional probability measure $\bbP(\cdot | \sF_t)$ & \cpageref{page:cond-measure} \\
$\E_t$ & conditional expectation $\E[\cdot|\sF_t]$ & \cpageref{page:cond-expect} \\
\hline
$\norm{\cdot}$ & euclidean/spectral norm of vector/matrix & \\
$\norm{\cdot}_1$ & $1$-norm, $\norm{x}_1 = \sum_k |x_k|$ & \\
$\norm{\cdot}_\infty$ & $\infty$-norm, $\norm{x}_\infty = \max_k |x_k|$ & \\
$\norm{\cdot}_{\psi_k}$ & Orlicz norms, $k \in \{1, 2\}$ & \cpageref{def:orlicz} \\
$\norm{\cdot}_A$ & $\norm{x}_A = \sqrt{x^\top A x}$ for positive semidefinite $A$ & \\
$\norm{\cdot}_K$ & $\norm{\cdot}_K = \pi_K(\cdot)$ when $K$ is a symmetric convex body & \cpageref{page:K-norm} \\
\hline
$K$ & normally a convex set & \\
$K^\circ$ & the polar of a convex set $K$ & \cpageref{page:polar} \\
$\pi_K$ & Minkowski functional of $K$, usually abbreviated to $\pi$ & \cpageref{page:mink} \\
$h_K$ & support function of $K$, usually abbreviated to $h$ & \cpageref{page:support} \\
$\interior(K)$ & interior of $K$ & \cpageref{page:interior} \\
$\ri(K)$ & relative interior of $K$ & \cpageref{page:ri} \\
$\partial K$ & boundary of $K$ & \cpageref{page:boundary} \\
$\diam(K)$ & diameter of $K$, $\sup_{x, y \in K} \norm{x - y}$ & \cpageref{page:diam} \\
$\vol(K)$ & volume of $K$ & \\
$\conv(A)$ & convex hull of $A$ & \\
$N(A, B)$, $\bar N(A, B)$ & external/internal covering numbers & \cpageref{page:cover} \\
$\ISO_K$ & affine map\index{affine!map} such that $\ISO_K(K)$ is isotropic & \cpageref{page:isotropic} \\
$\JOHN_K$ & affine map such that $\JOHN_K(K)$ is in John's position & \cpageref{page:john} \\
$\SEP_K$ & separation oracle for $K$ & \cpageref{page:separation} \\
$\MEM_K$ & membership oracle for $K$ & \cpageref{page:membership} \\
\hline
$f'$ / $f''$ & gradient/Hessian of $f \colon \R^d \to \R$ & \\
$\partial f(x)$ & subgradients\index{subgradient} of $f$ at $x$ & \\
$Df(x)[u]$ & directional derivative of $f$ at $x$ in direction $u$ & \\
$D^2f(x)[u,v]$ & second-order directional derivative in directions $u, v$ & \\ 
$D^3f(k)[u,v,w]$ & third-order directional derivative in directions $u,v,w$ & \\
$\dom(f)$ & domain of convex $f \colon \R^d \to \R \cup \{\infty\}$, $\{x \in \R^d \colon f(x) < \infty\}$ & \\
$\lip_K(f)$ & Lipschitz constant, $\sup\left\{\frac{f(x) - f(y)}{\norm{x - y}} \colon x, y \in K, x \neq y \right\}$ & \\
$\lip(f)$ & $\lip_{\dom(f)}(f)$ & \\ 
\hline
$f_1,\ldots,f_n$ & convex/submodular loss functions in rounds $t \in 1\ldots n$ & \\
$\eps_1,\ldots,\eps_n$ & noise random variables & \cpageref{page:noise} \\
$X_1,\ldots,X_n$ & iterates played by algorithm & \cpageref{page:iterates} \\
$Y_1,\ldots,Y_n$ & observed losses & \cpageref{page:responses} \\
$\alpha$ and $\beta$ & smoothness and strong convexity parameters & \cpageref{page:smooth} \\
\hline
\end{longtable}
}

\endappendix

\addtocontents{toc}{\vspace{\baselineskip}}

\bibliographystyle{cambridgeauthordate}

\bibliography{all}

@book{Ces06,
  author =        {Cesa-Bianchi, N. and Lugosi, G.},
  publisher =     {Cambridge University Press},
  title =         {Prediction, learning, and games},
  year =          {2006},
}

@article{Haz16,
  author =        {Hazan, E.},
  journal =       {Foundations and Trends{\textregistered} in
                   Optimization},
  number =        {3--4},
  pages =         {157--325},
  publisher =     {Now Publishers, Inc.},
  title =         {Introduction to online convex optimization},
  volume =        {2},
  year =          {2016},
}

@misc{Ora19,
  author =        {F. Orabona},
  note =          {arXiv:1912.13213},
  title =         {A modern introduction to online learning},
  year =          {2019},
}

@article{BC12,
  author =        {S. Bubeck and N. Cesa-Bianchi},
  journal =       {Foundations and Trends{\textregistered} in Machine Learning},
  number =        {1},
  pages =         {1-122},
  title =         {Regret analysis of stochastic and nonstochastic
                   multi-armed bandit problems},
  volume =        {5},
  year =          {2012},
}

@article{Sli18,
  author =        {A. Slivkins},
  journal =       {Foundations and Trends{\textregistered} in Machine Learning},
  number =        {1--2},
  pages =         {1--286},
  title =         {Introduction to multi-armed bandits},
  volume =        {12},
  year =          {2019},
  issn =          {1935-8237},
}

@book{LS20book,
  author =        {T. Lattimore and Cs. Szepesv{\'a}ri},
  publisher =     {Cambridge University Press},
  title =         {Bandit algorithms},
  year =          {2020},
}

@book{nemirovski96,
  author =        {A. Nemirovski},
  title =         {Lecture notes: Interior-point polynomial time methods
                   for convex programming},
  publisher= {Georgia Institute of Technology},
  year =          {1996},
}

@book{Ver18,
  author =        {R. Vershynin},
  publisher =     {Cambridge University Press},
  title =         {High-dimensional probability: An introduction with
                   applications in data science},
  year =          {2018},
}

@book{BLM13,
  author =        {Boucheron, S. and Lugosi, G. and Massart, P.},
  publisher =     {Oxford University Press},
  title =         {Concentration inequalities: A nonasymptotic theory of
                   independence},
  year =          {2013},
}

@book{Roc70,
  author =        {Rockafellar, R. T.},
  publisher =     {Princeton University Press},
  title =         {Convex analysis},
  year =          {1970},
}

@book{ASG15,
  author =        {S. Artstein-Avidan and A. Giannopoulos and
                   V. D. Milman},
  publisher =     {American Mathematical Society},
  title =         {Asymptotic geometric analysis, {P}art {I}},
  year =          {2015},
}

@book{nesterov2018lectures,
  author =        {Y. Nesterov},
  publisher =     {Springer},
  title =         {Lectures on convex optimization},
  year =          {2018},
}

@article{polyak1990optimal,
  author =        {Polyak, B. T. and Tsybakov, A. B.},
  journal =       {Problemy Peredachi Informatsii},
  number =        {2},
  volume = {26},
  pages =         {45--53},
  title =         {Optimal accuracy orders of stochastic approximation
                   algorithms},
  year =          {1990},
}

@inproceedings{bach2016highly,
  author =        {F. Bach and V. Perchet},
  booktitle =     {Proceedings of the 29th Conference on Learning Theory},
  title =         {Highly-smooth zero-th order online optimization},
  year =          {2016},
}

@inproceedings{akhavan2020exploiting,
  author =        {A. Akhavan and M. Pontil and A. Tsybakov},
  booktitle =       {Advances in Neural Information Processing Systems},
  title =         {Exploiting higher order smoothness in derivative-free
                   optimization and continuous bandits},
  year =          {2020},
}

@article{akhavan2024gradient,
  author =        {A. Akhavan and E. Chzhen and M. Pontil and
                   A. Tsybakov},
  journal =       {Journal of Machine Learning Research},
  number =        {370},
  pages =         {1--50},
  title =         {Gradient-free optimization of highly smooth
                   functions: Improved analysis and a new algorithm},
  volume =        {25},
  year =          {2024},
}

@inproceedings{ADX10,
  author =        {A. Agarwal and O. Dekel and L. Xiao},
  booktitle =     {Proceedings of the 23rd Conference on Learning Theory},
  title =         {Optimal algorithms for online convex optimization
                   with multi-point bandit feedback},
  year =          {2010},
}

@article{nesterov2017random,
  author =        {Y. Nesterov and V. Spokoiny},
  journal =       {Foundations of Computational Mathematics},
  pages =         {527--566},
  publisher =     {Springer},
  title =         {Random gradient-free minimization of convex
                   functions},
  volume =        {17},
  year =          {2017},
}

@article{duchi2015optimal,
  author =        {J. Duchi and M. Jordan and M. Wainwright and
                   A. Wibisono},
  journal =       {IEEE Transactions on Information Theory},
  number =        {5},
  pages =         {2788--2806},
  publisher =     {IEEE},
  title =         {Optimal rates for zero-order convex optimization: The
                   power of two function evaluations},
  volume =        {61},
  year =          {2015},
}

@inproceedings{AFHKR11,
  author =        {Agarwal, A. and Foster, D. P. and Hsu, D. J. and
                   Kakade, S. M. and Rakhlin, A.},
  booktitle =     {Advances in Neural Information Processing Systems},
  title =         {Stochastic convex optimization with bandit feedback},
  year =          {2011},
}

@inproceedings{LG21a,
  author =        {T. Lattimore and A. Gy\"orgy},
  booktitle =     {Proceedings of the 34th Conference on Learning
                   Theory},
  title =         {Improved regret for zeroth-order stochastic convex
                   bandits},
  year =          {2021},
}

@article{carpentier2024simple,
  title={A simple and improved algorithm for noisy, convex, zeroth-order optimisation},
  author={Carpentier, Alexandra},
  journal={Mathematical Statistics and Learning},
  number = {3--4},
  pages = {165--192},
  volume = {8},
  year={2025}
}

@inproceedings{Kle04,
  author =        {Kleinberg, R.},
  booktitle =     {Advances in Neural Information Processing Systems},
  title =         {Nearly tight bounds for the continuum-armed bandit
                   problem},
  year =          {2005},
}

@inproceedings{FK05,
  author =        {Flaxman, A. and Kalai, A. and McMahan, H.B.},
  booktitle =     {Proceedings of the 16th Annual ACM-SIAM
                   Symposium on Discrete Algorithms},
  title =         {Online convex optimization in the bandit setting:
                   Gradient descent without a gradient},
  year =          {2005},
}

@inproceedings{Sah11,
  author =        {A. Saha and A. Tewari},
  booktitle =     {Proceedings of the 14th International
                   Conference on Artificial Intelligence and Statistics},
  title =         {Improved regret guarantees for online smooth convex
                   optimization with bandit feedback},
  year =          {2011},
}

@inproceedings{HL14,
  author =        {E. Hazan and K. Levy},
  booktitle =     {Advances in Neural Information Processing Systems},
  title =         {Bandit convex optimization: {T}owards tight bounds},
  year =          {2014},
}

@inproceedings{SRN21,
  author =        {A. Suggala and P. Ravikumar and P. Netrapalli},
  booktitle =     {Proceedings of the 34th Conference on Learning
                   Theory},
  title =         {Efficient bandit convex optimization: {B}eyond linear
                   losses},
  year =          {2021},
}

@inproceedings{LG23,
  author =        {T. Lattimore and A. Gy{\"o}rgy},
  booktitle =     {Proceedings of the 36th Conference on Learning
                   Theory},
  title =         {A second-order method for stochastic bandit convex
                   optimisation},
  year =          {2023},
}

@inproceedings{suggala2024second,
  author =        {A. Suggala and J. Sun and P. Netrapalli and E. Hazan},
  booktitle =     {Proceedings of the 37th Conference on Learning
                   Theory},
  title =         {Second order methods for bandit optimization and
                   control},
  year =          {2024},
}

@inproceedings{LFMV24,
  author =        {H. Fokkema and D. van der Hoeven and T. Lattimore and
                   J. Mayo},
  booktitle =     {Proceedings of the 37th Conference on Learning Theory},
  title =         {Online {N}ewton method for bandit convex
                   optimisation},
  year =          {2024},
}

@inproceedings{BEL16,
  author =        {Bubeck, S. and Lee, Y.-T. and Eldan, R.},
  booktitle =     {Proceedings of the 49th Annual ACM SIGACT Symposium
                   on Theory of Computing},
  title =         {Kernel-based methods for bandit convex optimization},
  year =          {2017},
}

@inproceedings{RV14,
  author =        {Russo, D. and {Van Roy}, B.},
  booktitle =     {Advances in Neural Information Processing Systems},
  title =         {Learning to optimize via information-directed
                   sampling},
  year =          {2014},
}

@inproceedings{BDKP15,
  author =        {Bubeck, S. and Dekel, O. and Koren, T. and Peres, Y.},
  booktitle =     {Proceedings of the 28th Conference on Learning
                   Theory},
  title =         {Bandit convex optimization: $\sqrt{T}$ regret in one
                   dimension},
  year =          {2015},
}

@article{BE18,
  author =        {S. Bubeck and R. Eldan},
  journal =       {Mathematical Statistics and Learning},
  number =        {1},
  pages =         {73--100},
  title =         {Exploratory distributions for convex functions},
  volume =        {1},
  year =          {2018},
}

@article{Lat20-cvx,
  author =        {T. Lattimore},
  journal =       {Mathematical Statistics and Learning},
  number =        {3--4},
  pages =         {311--334},
  title =         {Improved regret for zeroth-order adversarial bandit
                   convex optimisation},
  volume =        {2},
  year =          {2020},
}

@article{AFHK13,
  author =        {A. Agarwal and D. P. Foster and D. Hsu and
                   S. M. Kakade and A. Rakhlin},
  journal =       {SIAM Journal on Optimization},
  number =        {1},
  pages =         {213--240},
  title =         {Stochastic convex optimization with bandit feedback},
  volume =        {23},
  year =          {2013},
}

@book{NY83,
  author =        {A. Nemirovsky and D. Yudin},
  publisher =     {John Wiley \& Sons},
  title =         {Problem complexity and method efficiency in
                   optimization},
  year =          {1983},
}

@inproceedings{lattimore2021mirror,
  author =        {T. Lattimore and A. Gy\"orgy},
  booktitle =     {Proceedings of the 34th Conference on Learning Theory},
  title =         {Mirror descent and the information ratio},
  year =          {2021},
}

@misc{HL16,
  author =        {E. Hazan and Y. Li},
  note =       {arXiv:1603.04350},
  title =         {An optimal algorithm for bandit convex optimization},
  year =          {2016},
}

@article{KW52,
  author =        {J. Kiefer and J. Wolfowitz},
  journal =       {The Annals of Mathematical Statistics},
  number =        {3},
  pages =         {462--466},
  publisher =     {Institute of Mathematical Statistics},
  title =         {Stochastic estimation of the maximum of a regression
                   function},
  volume =        {23},
  year =          {1952},
}

@article{RM51,
  author =        {H. Robbins and S. Monro},
  journal =       {The Annals of Mathematical Statistics},
  pages =         {400--407},
  volume = {22},
  number = {3},
  title =         {A stochastic approximation method},
  year =          {1951},
}

@article{Blu54,
  author =        {J. R. Blum},
  journal =       {The Annals of Mathematical Statistics},
  number =        {4},
  volume =        {25},
  pages =         {737--744},
  title =         {Multidimensional stochastic approximation methods},
  year =          {1954},
}

@inproceedings{spall1994developments,
  author =        {J. C. Spall},
  booktitle =     {Proceedings of Winter Simulation Conference},
  title =         {Developments in stochastic optimization algorithms
                   with gradient approximations based on function
                   measurements},
  year =          {1994},
}

@article{prashanth2025gradient,
  author =        {L. A. Prashanth and S. Bhatnagar},
  journal =       {Foundations and Trends{\textregistered} in
                   Optimization},
  number =        {1--3},
  pages =         {1--332},
  publisher =     {Now Publishers, Inc.},
  title =         {Gradient-based algorithms for zeroth-order
                   optimization},
  volume =        {8},
  year =          {2025},
}

@article{spall1992multivariate,
  author =        {J. C. Spall},
  journal =       {IEEE Transactions on Automatic Control},
  number =        {3},
  pages =         {332--341},
  title =         {Multivariate stochastic approximation using a
                   simultaneous perturbation gradient approximation},
  volume =        {37},
  year =          {1992},
}

@article{ghadimi2013stochastic,
  author =        {S. Ghadimi and G. Lan},
  journal =       {SIAM Journal on Optimization},
  number =        {4},
  pages =         {2341--2368},
  title =         {Stochastic first- and zeroth-order methods for
                   nonconvex stochastic programming},
  volume =        {23},
  year =          {2013},
}

@article{NJL09,
  author =        {Nemirovski, A. and Juditsky, A. and Lan, G. and
                   Shapiro, A.},
  journal =       {SIAM Journal on Optimization},
  number =        {4},
  pages =         {1574-1609},
  title =         {Robust stochastic approximation approach to
                   stochastic programming},
  volume =        {19},
  year =          {2009},
}

@inproceedings{DHK08,
  author =        {Dani, V. and Hayes, T. P. and Kakade, S. M.},
  booktitle =     {Proceedings of the 21st Conference on Learning
                   Theory},
  title =         {Stochastic linear optimization under bandit feedback},
  year =          {2008},
}

@inproceedings{Sha13,
  author =        {O. Shamir},
  booktitle =     {Proceedings of the 26th Conference on Learning
                   Theory},
  title =         {On the complexity of bandit and derivative-free
                   stochastic convex optimization},
  year =          {2013},
}

@inproceedings{shamir2015complexity,
  author =        {O. Shamir},
  booktitle =     {Proceedings of the 28th Conference on Learning Theory},
  title =         {On the complexity of bandit linear optimization},
  year =          {2015},
}

@inproceedings{HPGySz16:BCO,
  author =        {Hu, X. and {Prashanth L. A.} and Gy{\"o}rgy, A. and
                   {Sz}epesv{\'a}ri, {Cs}.},
  booktitle =     {Proceedings of the 19th International Conference on Artificial Intelligence and Statistics},
  title =         {({B}andit) convex optimization with biased noisy
                   gradient oracles},
  year =          {2016},
}

@inproceedings{BCK12,
  author =        {Bubeck, S. and Cesa-Bianchi, N. and Kakade, S.},
  booktitle =     {Proceedings of the 25th Conference on Learning
                   Theory},
  title =         {Towards minimax policies for online linear
                   optimization with bandit feedback},
  year =          {2012},
}

@inproceedings{BLN15,
  author =        {A. Belloni and T. Liang and H. Narayanan and
                   A. Rakhlin},
  booktitle =     {Proceedings of the 28th Conference on Learning Theory},
  title =         {Escaping the local minima via simulated annealing:
                   Optimization of approximately convex functions},
  year =          {2015},
}

@article{BEL18,
  author =        {Bubeck, S. and Eldan, R. and Lehec, J.},
  journal =       {Discrete {\&} Computational Geometry},
  volume = {59},
  pages = {757--783},
  title =         {Sampling from a log-concave distribution with
                   projected {L}angevin {M}onte {C}arlo},
  year =          {2018},
}

@inproceedings{Ito20,
  author =        {S. Ito},
  booktitle =     {Proceedings of the 23rd International
                   Conference on Artificial Intelligence and Statistics},
  title =         {An optimal algorithm for bandit convex optimization
                   with strongly-convex and smooth loss},
  year =          {2020},
}

@article{larson2019derivative,
  author =        {J. Larson and M. Menickelly and S. Wild},
  journal =       {Acta Numerica},
  pages =         {287--404},
  publisher =     {Cambridge University Press},
  title =         {Derivative-free optimization methods},
  volume =        {28},
  year =          {2019},
}

@book{conn2009introduction,
  author =        {A. Conn and K. Scheinberg and L. Vicente},
  publisher =     {Society for Industrial and Applied Mathematics},
  title =         {Introduction to derivative-free optimization},
  year =          {2009},
}

@book{BPP12,
  author =        {S. Bhatnagar and H. L. Prasad and L. A. Prashanth},
  publisher =     {Springer},
  series =        {Lecture Notes in Control and Information Sciences},
  title =         {Stochastic recursive algorithms for optimization:
                   Simultaneous perturbation methods},
  year =          {2012},
}

@article{liu2020primer,
  author =        {S. Liu and P.-Y. Chen and B. Kailkhura and G. Zhang and
                   A. III Hero and P. Varshney},
  journal =       {IEEE Signal Processing Magazine},
  number =        {5},
  pages =         {43--54},
  publisher =     {IEEE},
  title =         {A primer on zeroth-order optimization in signal
                   processing and machine learning: Principles, recent
                   advances, and applications},
  volume =        {37},
  year =          {2020},
}

@article{balasubramanian2022zeroth,
  author =        {K. Balasubramanian and S. Ghadimi},
  journal =       {Foundations of Computational Mathematics},
  pages =         {1--42},
  volume =        {22},
  number =        {2},
  publisher =     {Springer},
  title =         {Zeroth-order nonconvex stochastic optimization:
                   Handling constraints, high dimensionality, and saddle
                   points},
  year =          {2022},
}

@article{zhao2021bandit,
  author =        {P. Zhao and G. Wang and L. Zhang and Z.-H. Zhou},
  journal =       {Journal of Machine Learning Research},
  number =        {1},
  pages =         {5562--5606},
  title =         {Bandit convex optimization in non-stationary
                   environments},
  volume =        {22},
  year =          {2021},
}

@inproceedings{LZZ22,
  author =        {H. Luo and M. Zhang and P. Zhao},
  booktitle =     {Proceedings of the 25th Conference on Learning
                   Theory},
  title =         {Adaptive bandit convex optimization with
                   heterogeneous curvature},
  year =          {2022},
}

@article{wang2023adaptivity,
  author =        {Y. Wang},
  journal =       {Operations Research},
  publisher =     {INFORMS},
  volume = {73},
  number = {2},
  pages = {819--828},
  title =         {On adaptivity in nonstationary stochastic
                   optimization with bandit feedback},
  year =          {2023},
}

@misc{LBZ25,
  author =        {X. Liu and D. Baudry and J. Zimmert and P. Rebeschini and
                   A. Akhavan},
  note =       {arXiv:2506.02980},
  title =         {Non-stationary bandit convex optimization: A
                   comprehensive study},
  year =          {2025},
}

@inproceedings{jamieson2012query,
  author =        {K. Jamieson and R. Nowak and B. Recht},
  booktitle =       {Advances in Neural Information Processing Systems},
  title =         {Query complexity of derivative-free optimization},
  year =          {2012},
}

@misc{liang2014zeroth,
  author =        {T. Liang and H. Narayanan and A. Rakhlin},
  note =       {arXiv:1402.2667},
  title =         {On zeroth-order stochastic convex optimization via
                   random walks},
  year =          {2014},
}

@article{kannan1995isoperimetric,
  author =        {R. Kannan and L. Lov{\'a}sz and M. Simonovits},
  journal =       {Discrete \& Computational Geometry},
  pages =         {541--559},
  publisher =     {Springer},
  title =         {Isoperimetric problems for convex bodies and a
                   localization lemma},
  volume =        {13},
  year =          {1995},
}

@misc{klartag2024affirmative,
  author =        {B. Klartag and J. Lehec},
  note =       {arXiv:2412.15044},
  title =         {Affirmative resolution of {B}ourgain's slicing problem
                   using {G}uan's bound},
  year =          {2024},
}

@article{lovasz2006simulated,
  author =        {L. Lov{\'a}sz and S. Vempala},
  journal =       {Journal of Computer and System Sciences},
  number =        {2},
  pages =         {392--417},
  publisher =     {Elsevier},
  title =         {Simulated annealing in convex bodies and an {$O^*(n^4)$}
                   volume algorithm},
  volume =        {72},
  year =          {2006},
}

@book{BV04,
  author =        {Boyd, S. and Vandenberghe, L.},
  publisher =     {Cambridge University Press},
  title =         {Convex optimization},
  year =          {2004},
}

@inproceedings{jiang2020faster,
  author = {S. Jiang and Z. Song and O. Weinstein and H. Zhang},
  title = {A faster algorithm for solving general {LP}s},
  year = {2021},
  booktitle = {Proceedings of the 53rd Annual ACM SIGACT Symposium on Theory of Computing},
}

@inproceedings{lee2018efficient,
  author =        {Y.-T. Lee and A. Sidford and S. Vempala},
  booktitle =     {Proceedings of the 31st Conference on Learning Theory},
  title =         {Efficient convex optimization with membership
                   oracles},
  year =          {2018},
}

@inproceedings{lee2015faster,
  author =        {Y.-T. Lee and A. Sidford and S. Wong},
  booktitle =     {IEEE 56th Annual Symposium on Foundations of
                   Computer Science},
  title =         {A faster cutting plane method and its implications
                   for combinatorial and convex optimization},
  year =          {2015},
}

@article{KT93,
  author =        {L. Khachiyan and M. Todd},
  journal =       {Mathematical Programming},
  month =         {08},
  pages =         {137-159},
  title =         {On the complexity of approximating the maximal
                   inscribed ellipsoid for a polytope},
  volume =        {61},
  year =          {1993},
}

@misc{drori2018properties,
  author =        {Y. Drori},
  note =          {arXiv:1812.02419},
  title =         {On the properties of convex functions over open sets},
  year =          {2018},
}

@inproceedings{mhammedi2022efficient,
  author =        {Z. Mhammedi},
  booktitle =     {Proceedings of the 35th Conference on Learning Theory},
  title =         {Efficient projection-free online convex optimization
                   with membership oracle},
  year =          {2022},
}

@misc{orseau2023line,
  author =        {L. Orseau and M. Hutter},
  note =          {arXiv:2307.16560},
  title =         {Line search for convex minimization},
  year =          {2023},
}

@article{kiefer1953sequential,
  author =        {J. Kiefer},
  journal =       {Proceedings of the American Mathematical Society},
  number =        {3},
  pages =         {502--506},
  title =         {Sequential minimax search for a maximum},
  volume =        {4},
  year =          {1953},
}

@article{BCCP22,
  title={A theoretical framework for zeroth-order budget convex optimization},
  author =        {F. Bachoc and T. Cesari and R. Colomboni and
                   A. Paudice},
  journal={Transactions on Machine Learning Research},
  year={2024}
}

@misc{BCCP22b,
  author =        {F. Bachoc and T. Cesari and R. Colomboni and
                   A. Paudice},
  note =          {arXiv:2209.00885},
  title =         {Regret analysis of dyadic search},
  year =          {2022},
}

@inproceedings{Zin03,
  author =        {Zinkevich, M.},
  booktitle =     {Proceedings of the 20th International Conference on
                   Machine Learning},
  title =         {Online convex programming and generalized
                   infinitesimal gradient ascent},
  year =          {2003},
}

@inproceedings{garber2022new,
  author =        {D. Garber and B. Kretzu},
  booktitle =     {Proceedings of the 35th Conference on Learning Theory},
  title =         {New projection-free algorithms for online convex
                   optimization with adaptive regret guarantees},
  year =          {2022},
}

@article{nesterov1988polynomial,
  author =        {Y. Nesterov},
  journal =       {Izvestija AN SSSR, Tekhnitcheskaya kibernetika},
  number =        {3},
  title =         {Polynomial time methods in linear and quadratic
                   programming},
  year =          {1988},
}

@article{yudin1976informational,
  author =        {D. Yudin and A. Nemirovskii},
  journal =       {Matekon},
  number =        {2},
  pages =         {22--45},
  title =         {Informational complexity and efficient methods for
                   the solution of convex extremal problems},
  volume =        {13},
  year =          {1976},
}

@article{protasov1996algorithms,
  author =        {V. Protasov},
  journal =       {Mathematical Notes},
  number =        {1},
  pages =         {69--74},
  publisher =     {Springer},
  title =         {Algorithms for approximate calculation of the minimum
                   of a convex function from its values},
  volume =        {59},
  year =          {1996},
}

@article{NN89,
  author =        {Y. Nesterov and A. Nemirovsky},
  journal =       {USSR Academy of Sciences Central Economic \& Mathematical
                   Institute, Moscow},
  title =         {Self-concordant functions and polynomial time methods
                   in convex programming},
  year =          {1989},
}

@inproceedings{AHR08,
  author =        {Abernethy, J. D. and Hazan, E. and Rakhlin, A.},
  booktitle =     {Proceedings of the 21st Conference on Learning
                   Theory},
  title =         {Competing in the dark: An efficient algorithm for
                   bandit linear optimization},
  year =          {2008},
}

@misc{akhavan2024contextual,
  author =        {A. Akhavan and K. Lounici and M. Pontil and
                   A. Tsybakov},
  note =          {arXiv:2406.05714},
  title =         {Contextual continuum bandits: Static versus dynamic
                   regret},
  year =          {2024},
}

@inproceedings{chewi2023entropic,
  author =        {S. Chewi},
  booktitle =     {Geometric Aspects of Functional Analysis: Israel
                   Seminar},
  title =         {The entropic barrier is {$n$}-self-concordant},
  year =          {2023},
}

@inproceedings{bubeck2014entropic,
  author =        {Bubeck, S. and Eldan, R.},
  booktitle =     {Proceedings of the 28th Conference on Learning
                   Theory},
  title =         {The entropic barrier: A simple and optimal universal
                   self-concordant barrier},
  year =          {2015},
}

@book{nesterov1994interior,
  author =        {Y. Nesterov and A. Nemirovski},
  publisher =     {Society for Industrial and Applied Mathematics},
  title =         {Interior-point polynomial algorithms in convex
                   programming},
  year =          {1994},
}

@article{lee2021universal,
  author =        {Y.-T. Lee and M--C Yue},
  journal =       {Mathematics of Operations Research},
  number =        {3},
  pages =         {1129--1148},
  publisher =     {INFORMS},
  title =         {Universal barrier is {$n$}-self-concordant},
  volume =        {46},
  year =          {2021},
}

@misc{novitskii2021improved,
  author =        {V. Novitskii and A. Gasnikov},
  note =          {arXiv:2101.03821},
  title =         {Improved exploiting higher order smoothness in
                   derivative-free optimization and continuous bandit},
  year =          {2021},
}

@article{KW60,
  author =        {Kiefer, J. and Wolfowitz, J.},
  journal =       {Canadian Journal of Mathematics},
  number =        {5},
  pages =         {363--365},
  title =         {The equivalence of two extremum problems},
  volume =        {12},
  year =          {1960},
}

@book{chewi24,
  author =        {S. Chewi},
  title =         {Log-concave sampling},
  year =          {2024},
}

@article{motzkin1965maxima,
  author =        {T. Motzkin and G. Straus},
  journal =       {Canadian Journal of Mathematics},
  pages =         {533--540},
  publisher =     {Cambridge University Press},
  title =         {Maxima for graphs and a new proof of a theorem of
                   {T}ur{\'a}n},
  volume =        {17},
  year =          {1965},
}

@incollection{karp1972reducibility,
  author={R. Karp},
  title={Reducibility among combinatorial problems},
  booktitle={Proceedings of a Symposium on the Complexity of Computer Computations},
  year={1972},
}

@inproceedings{chatterji2019online,
  author =        {N. Chatterji and A. Pacchiano and P. Bartlett},
  booktitle =     {Proceedings of the 36th International Conference on Machine Learning},
  pages =         {971--980},
  title =         {Online learning with kernel losses},
  year =          {2019},
}

@article{HKM16,
  author =        {Hazan, E. and Karnin, Z. and Meka, R.},
  journal =       {Journal of Machine Learning Research},
  number =        {119},
  pages =         {1--34},
  title =         {Volumetric spanners: An efficient exploration basis
                   for learning},
  volume =        {17},
  year =          {2016},
}

@article{bauschke1997legendre,
  author =        {H. Bauschke and J. Borwein},
  journal =       {Journal of Convex Analysis},
  number =        {1},
  pages =         {27--67},
  title =         {Legendre functions and the method of random Bregman
                   projections},
  volume =        {4},
  year =          {1997},
}

@inproceedings{LS19pminfo,
  author =        {T. Lattimore and Cs. Szepesv{\'a}ri},
  booktitle =     {Proceedings of the 32nd Conference on Learning
                   Theory},
  title =         {An information-theoretic approach to minimax regret
                   in partial monitoring},
  year =          {2019},
}

@book{CT12,
  author =        {Cover, T. M. and Thomas, J. A.},
  publisher =     {John Wiley \& Sons},
  title =         {Elements of information theory},
  year =          {2012},
}

@phdthesis{kirschner2021information,
  author =        {J. Kirschner},
  school =        {ETH Zurich},
  title =         {Information-directed sampling -- frequentist analysis
                   and applications},
  year =          {2021},
}

@article{Tho33,
  author =        {Thompson, W.},
  journal =       {Biometrika},
  number =        {3--4},
  pages =         {285--294},
  publisher =     {JSTOR},
  title =         {On the likelihood that one unknown probability
                   exceeds another in view of the evidence of two
                   samples},
  volume =        {25},
  year =          {1933},
}

@article{RV16,
  author =        {Russo, D. and {Van Roy}, B.},
  journal =       {Journal of Machine Learning Research},
  number =        {1},
  pages =         {2442--2471},
  title =         {An information-theoretic analysis of {T}hompson
                   sampling},
  volume =        {17},
  year =          {2016},
  issn =          {1532-4435},
}

@inproceedings{BLS25,
  author =        {A. Bakhtiari and T. Lattimore and Cs. Szepesv{\'a}ri},
  booktitle =     {Proceedings of the 38th Conference on Learning
                   Theory},
  title =         {Thompson sampling for bandit convex optimisation},
  year =          {2025},
}

@inproceedings{ZL19,
  author =        {J. Zimmert and T. Lattimore},
  booktitle =     {Advances in Neural Information Processing Systems},
  title =         {Connections between mirror descent, {T}hompson
                   sampling and the information ratio},
  year =          {2019},
}

@misc{foster2021statistical,
  author =        {D. J. Foster and S. Kakade and J. Qian and
                   A. Rakhlin},
  note =       {arXiv:2112.13487},
  title =         {The statistical complexity of interactive decision
                   making},
  year =          {2021},
}

@inproceedings{foster2022complexity,
  booktitle =       {Advances in Neural Information Processing Systems},
  author =        {D. J. Foster and A. Rakhlin and A. Sekhari and
                   K. Sridharan},
  title =         {On the complexity of adversarial decision making},
  year =          {2022},
}

@inproceedings{lattimore2021bandit,
  author =        {T. Lattimore and B. Hao},
  booktitle =       {Advances in Neural Information Processing Systems},
  title =         {Bandit phase retrieval},
  year =          {2021},
}

@article{shor1977cut,
  author =        {N. Shor},
  journal =       {Cybernetics},
  number =        {1},
  pages =         {94--96},
  publisher =     {Springer},
  title =         {Cut-off method with space extension in convex
                   programming problems},
  volume =        {13},
  year =          {1977},
}

@article{judin1977evaluation,
  author =        {D. Yudin and A. Nemirovskii},
  journal =       {Matekon},
  number =        {2},
  pages =         {3--25},
  title =         {Evaluation of the informational complexity of
                   mathematical programming problems},
  volume =        {13},
  year =          {1977},
}

@article{New65,
  address =       {New York, NY, USA},
  author =        {Newman, D. J.},
  journal =       {Journal of the ACM},
  number =        {3},
  pages =         {395–398},
  publisher =     {Association for Computing Machinery},
  title =         {Location of the maximum on unimodal surfaces},
  volume =        {12},
  year =          {1965},
}

@inproceedings{levin1965algorithm,
  author =        {A. Levin},
  booktitle =     {Doklady Akademii Nauk},
  number =        {6},
  organization =  {Russian Academy of Sciences},
  pages =         {1244--1247},
  title =         {An algorithm for minimizing convex functions},
  volume =        {160},
  year =          {1965},
}

@article{vaidya1996new,
  author =        {P. M. Vaidya},
  journal =       {Mathematical Programming},
  number =        {3},
  pages =         {291--341},
  publisher =     {Springer},
  title =         {A new algorithm for minimizing convex functions over
                   convex sets},
  volume =        {73},
  year =          {1996},
}

@article{nesterov1995cutting,
  author =        {Y. Nesterov},
  journal =       {Mathematical Programming},
  number =        {1},
  pages =         {149--176},
  title =         {Cutting plane algorithms from analytic centers:
                   efficiency estimates},
  volume =        {69},
  year =          {1995},
}

@article{atkinson1995cutting,
  author =        {D. S. Atkinson and P. M. Vaidya},
  journal =       {Mathematical Programming},
  number =        {1},
  pages =         {1--43},
  publisher =     {Springer},
  title =         {A cutting plane algorithm for convex programming that
                   uses analytic centers},
  volume =        {69},
  year =          {1995},
}

@inproceedings{tarasov1988method,
  author =        {S. Tarasov and L. G. Khachiyan and I. I. Erlich},
  booktitle =     {Soviet Mathematics-Doklady},
  number =        {1},
  pages =         {226--230},
  title =         {The method of inscribed ellipsoids},
  volume =        {37},
  year =          {1988},
}

@article{galicer2019minimal,
  author =        {D. Galicer and M. Merzbacher and D. Pinasco},
  journal =       {Journal of Geometric Analysis},
  pages =         {717--732},
  publisher =     {Springer},
  title =         {The minimal volume of simplices containing a convex
                   body},
  volume =        {29},
  year =          {2019},
}

@article{Gru60,
  author =        {B. Gr{\"u}nbaum},
  journal =       {Pacific Journal of Mathematics},
  number =        {4},
  pages =         {1257--1261},
  publisher =     {Pacific Journal of Mathematics, A Non-profit
                   Corporation},
  title =         {Partitions of mass-distributions and of convex
                   bodies by hyperplanes},
  volume =        {10},
  year =          {1960},
}

@article{bertsimas2004solving,
  author =        {D. Bertsimas and S. Vempala},
  journal =       {Journal of the ACM (JACM)},
  number =        {4},
  pages =         {540--556},
  publisher =     {ACM New York, NY, USA},
  title =         {Solving convex programs by random walks},
  volume =        {51},
  year =          {2004},
}

@article{khachiyan1990inequality,
  author =        {L. Khachiyan},
  journal =       {Discrete \& Computational Geometry},
  pages =         {219--222},
  publisher =     {Springer},
  title =         {An inequality for the volume of inscribed ellipsoids},
  volume =        {5},
  year =          {1990},
}

@article{EMM06,
  author =        {Even-Dar, E. and Mannor, S. and Mansour, Y.},
  journal =       {Journal of Machine Learning Research},
  pages =         {1079--1105},
  title =         {Action elimination and stopping conditions for the
                   multi-armed bandit and reinforcement learning
                   problems},
  volume =        {7},
  year =          {2006},
}

@article{berry1997bandit,
  author =        {D. Berry and R. Chen and A. Zame and D. Heath and
                   L. Shepp},
  journal =       {The Annals of Statistics},
  number =        {5},
  pages =         {2103--2116},
  publisher =     {Institute of Mathematical Statistics},
  title =         {Bandit problems with infinitely many arms},
  volume =        {25},
  year =          {1997},
}

@book{MLA12,
  author =        {Gr{\"o}tschel, M. and Lov{\'a}sz, L. and
                   Schrijver, A.},
  publisher =     {Springer Science \& Business Media},
  title =         {Geometric algorithms and combinatorial optimization},
  year =          {2012},
}

@inproceedings{HEK18,
  author =        {D. van der Hoeven and T. van Erven and
                   W. Kot{\l}owski},
  booktitle =     {Proceedings of the 31st Conference on Learning
                   Theory},
  title =         {The many faces of exponential weights in online
                   learning},
  year =          {2018},
}

@article{barthe1998extremal,
  author =        {F. Barthe},
  journal =       {Mathematische Annalen},
  pages =         {685--693},
  publisher =     {Springer},
  title =         {An extremal property of the mean width of the
                   simplex},
  volume =        {310},
  year =          {1998},
}

@misc{finch2011mean,
  author =        {S. Finch},
  note =          {arXiv:1111.4976},
  title =         {Mean width of a regular simplex},
  year =          {2011},
}

@inproceedings{giannopoulos2014m,
  author =        {A. Giannopoulos and E. Milman},
  booktitle =     {Geometric Aspects of Functional Analysis: Israel
                   Seminar},
  title =         {{$M$}-estimates for isotropic convex bodies and their {$L_q$}-centroid bodies},
  year =          {2014},
}

@article{pivovarov2010volume,
  author =        {P. Pivovarov},
  journal =     {Mathematical Proceedings of the Cambridge
                   Philosophical Society},
  number =        {2},
  pages =         {317--331},
  title =         {On the volume of caps and bounding the mean-width of
                   an isotropic convex body},
  volume =        {149},
  year =          {2010},
}

@article{milman2015mean,
  author =        {E. Milman},
  journal =       {International Mathematics Research Notices},
  number =        {11},
  pages =         {3408--3423},
  publisher =     {Oxford University Press},
  title =         {On the mean-width of isotropic convex bodies and
                   their associated L p-centroid bodies},
  volume =        {2015},
  year =          {2015},
}

@book{schneider2013convex,
  author =        {R. Schneider},
  publisher =     {Cambridge University Press},
  title =         {Convex bodies: The Brunn--Minkowski theory},
  year =          {2013},
}

@article{meyer1998santalo,
  author =        {M. Meyer and E. Werner},
  journal =       {Transactions of the American Mathematical Society},
  number =        {11},
  pages =         {4569--4591},
  title =         {The {S}antal{\'o}-regions of a convex body},
  volume =        {350},
  year =          {1998},
}

@article{santalo1949invariante,
  author =        {L. A. Santal{\'o}},
  journal =       {Portugaliae Mathematica},
  pages =         {155--161},
  title =         {Un invariante afin para los cuerpos convexos del
                   espacio de {$n$} dimensiones},
  volume =        {8},
  year =          {1949},
}

@article{kaiser1993santalo,
  author =        {M. J. Kaiser},
  journal =       {Applied Mathematics Letters},
  number =        {2},
  pages =         {47--53},
  publisher =     {Elsevier},
  title =         {The {S}antal{\'o} point of a planar convex set},
  volume =        {6},
  year =          {1993},
}

@article{HAK07,
  author =        {E. Hazan and A. Agarwal and S. Kale},
  journal =       {Machine Learning},
  pages =         {169--192},
  publisher =     {Springer},
  title =         {Logarithmic regret algorithms for online convex
                   optimization},
  volume =        {69},
  year =          {2007},
}

@article{hale2008computing,
  author =        {N. Hale and N. Higham and L. Trefethen},
  journal =       {SIAM Journal on Numerical Analysis},
  number =        {5},
  pages =         {2505--2523},
  title =         {Computing {$A^\alpha$}, {$\log(A)$} and
                   related matrix functions by contour integrals},
  volume =        {46},
  year =          {2008},
}

@article{mccormick2005submodular,
  author =        {S. McCormick},
  journal =       {Handbooks in Operations Research and Management
                   Science},
  pages =         {321--391},
  publisher =     {Elsevier},
  title =         {Submodular function minimization},
  volume =        {12},
  year =          {2005},
}

@book{BT97,
  author =        {Bertsimas, D. and Tsitsiklis, J. N.},
  publisher =     {Athena Scientific},
  title =         {Introduction to linear optimization},
  year =          {1997},
}

@inproceedings{jegelka2011online,
  author =        {S. Jegelka and J. Bilmes},
  booktitle =     {Proceedings of the 28th International Conference on Machine Learning},
  title =         {Online submodular minimization for combinatorial
                   structures},
  year =          {2011},
}

@article{hazan2012online,
  author =        {E. Hazan and S. Kale},
  journal =       {Journal of Machine Learning Research},
  number =        {10},
  title =         {Online submodular minimization},
  volume =        {13},
  year =          {2012},
}

@article{bach2013learning,
  author =        {F. Bach},
  journal =       {Foundations and Trends{\textregistered} in Machine
                   Learning},
  number =        {2--3},
  pages =         {145--373},
  publisher =     {Now Publishers, Inc.},
  title =         {Learning with submodular functions: A convex
                   optimization perspective},
  volume =        {6},
  year =          {2013},
}

@misc{bilmes2022submodularity,
  author =        {J. Bilmes},
  note =          {arXiv:2202.00132},
  title =         {Submodularity in machine learning and artificial
                   intelligence},
  year =          {2022},
}

@inproceedings{gabillon2013adaptive,
  author =        {V. Gabillon and B. Kveton and Z. Wen and B. Eriksson and
                   S. Muthukrishnan},
  booktitle =       {Advances in Neural Information Processing Systems},
  title =         {Adaptive submodular maximization in bandit setting},
  year =          {2013},
}

@inproceedings{zhang2019online,
  author =        {M. Zhang and L. Chen and H. Hassani and A. Karbasi},
  booktitle =       {Advances in Neural Information Processing Systems},
  title =         {Online continuous submodular maximization: From
                   full-information to bandit feedback},
  year =          {2019},
}

@misc{foster2021submodular,
  author =        {D. P. Foster and A. Rakhlin},
  note =       {arXiv:2112.02165},
  title =         {On submodular contextual bandits},
  year =          {2021},
}

@inproceedings{chen2017interactive,
  author =        {L. Chen and A. Krause and A. Karbasi},
  booktitle =       {Advances in Neural Information Processing Systems},
  title =         {Interactive submodular bandit},
  year =          {2017},
}

@inproceedings{takemori2020submodular,
  author =        {S. Takemori and M. Sato and T. Sonoda and J. Singh and
                   T. Ohkuma},
  booktitle =     {Proceedings of the 36th Conference on Uncertainty in Artificial Intelligence},
  title =         {Submodular bandit problem under multiple constraints},
  year =          {2020},
}

@inproceedings{tajdini2024nearly,
  author =        {A. Tajdini and L. Jain and K. Jamieson},
  booktitle =       {Advances in Neural Information Processing Systems},
  title =         {Nearly minimax optimal submodular maximization with
                   bandit feedback},
  year =          {2024},
}

@inproceedings{niazadeh2021online,
  author =        {R. Niazadeh and N. Golrezaei and J. Wang and F. Susan and
                   A. Badanidiyuru},
  booktitle =     {Proceedings of the 22nd ACM Conference on Economics
                   and Computation},
  title =         {Online learning via offline greedy algorithms:
                   Applications in market design and optimization},
  year =          {2021},
}

@inproceedings{lovasz1983submodular,
  author =        {L. Lov{\'a}sz},
  booktitle =       {Mathematical Programming: The State of the Art: Bonn 1982},
  pages =         {235--257},
  title =         {Submodular functions and convexity},
  year =          {1983},
}

@article{Bub15,
  author =        {S. Bubeck},
  journal =       {Foundations and Trends{\textregistered} in Machine
                   Learning},
  number =        {3--4},
  pages =         {231--357},
  publisher =     {Now Publishers, Inc.},
  title =         {Convex optimization: Algorithms and complexity},
  volume =        {8},
  year =          {2015},
}

@article{jiang2022minimizing,
  author =        {H. Jiang},
  journal =       {Journal of the ACM},
  number =        {1},
  pages =         {1--27},
  publisher =     {ACM New York, NY},
  title =         {Minimizing convex functions with rational minimizers},
  volume =        {70},
  year =          {2022},
}

@misc{BK18,
  author =        {L. Besson and E. Kaufmann},
  note =  {arXiv:1803.06971},
  title =         {What doubling tricks can and can't do for multi-armed
                   bandits},
  year =          {2018},
}

@article{polyak1963gradient,
  author =        {B. Polyak},
  journal =       {Zhurnal Vychislitel'noi Matematiki I Matematicheskoi
                   Fiziki},
  number =        {4},
  pages =         {643--653},
  publisher =     {Russian Academy of Sciences, Branch of Mathematical
                   Sciences},
  title =         {Gradient methods for minimizing functionals},
  volume =        {3},
  year =          {1963},
}

@inproceedings{karimi2016linear,
  author =        {H. Karimi and J. Nutini and M. Schmidt},
  booktitle =     {Proceedings of the Joint European Conference on Machine Learning and
                   Knowledge Discovery in Databases},
  title =         {Linear convergence of gradient and proximal-gradient
                   methods under the {P}olyak-{\L}ojasiewicz condition},
  year =          {2016},
}

@article{BCL13,
  author =        {Bubeck, S. and Cesa-Bianchi, N. and Lugosi, G.},
  journal =       {IEEE Transactions on Information Theory},
  number =        {11},
  pages =         {7711--7717},
  publisher =     {IEEE},
  title =         {Bandits with heavy tail},
  volume =        {59},
  year =          {2013},
}

@misc{zhan2025regularizedonlinenewtonmethod,
  author =        {Z. Jingxin and X. Yuchen and J. Kaicheng and
                   Z. Zhihua},
  note =       {arXiv:2501.11127},
  title =         {A regularized online Newton method for stochastic
                   convex bandits with linear vanishing noise},
  year =          {2025},
}

@misc{ao2025riemannian,
  author =        {R. Ao and H. Hu and D. Simchi-Levi},
  note =          {Available at SSRN 5250625},
  title =         {Riemannian online convex optimization with
                   self-concordant barrier},
  year =          {2025},
}

@book{CRC18,
  author =        {D. Zwillinger},
  edition =       {33rd},
  publisher =     {Chapman and Hall/CRC},
  title =         {CRC standard mathematical tables and formulae},
  year =          {2018},
}

@book{evans2018measure,
  author =        {L. Evans},
  publisher =     {Routledge},
  title =         {Measure theory and fine properties of functions},
  year =          {2018},
}

@article{zhang2020concentration,
author =        {H. Zhang and S. Chen},
title =         {Concentration inequalities for statistical inference},
journal = {Communications in Mathematical Research },
year = {2021},
volume = {37},
number = {1},
pages = {1--85},
}

@article{marchal2017sub,
  author =        {O. Marchal and J. Arbel},
  journal =       {Electronic Communications in Probability},
  number =        {54},
  pages =         {1--14},
  title =         {On the sub-Gaussianity of the Beta and Dirichlet
                   distributions},
  volume =        {22},
  year =          {2017},
}

@misc{duembgen2010bounding,
  author =        {L. Duembgen},
  note =       {arXiv:1012.2063},
  title =         {Bounding standard gaussian tail probabilities},
  year =          {2010},
}

@article{pinelis2022improved,
  author =        {I. Pinelis},
  journal =       {Statistics \& Probability Letters},
  pages =         {109666},
  publisher =     {Elsevier},
  title =         {Improved concentration bounds for sums of independent
                   sub-exponential random variables},
  volume =        {191},
  year =          {2022},
}

@inproceedings{zimmert22b,
  author =        {J. Zimmert and T. Lattimore},
  booktitle =     {Proceedings of the 35th Conference on Learning
                   Theory},
  title =         {Return of the bias: Almost minimax optimal high
                   probability bounds for adversarial linear bandits},
  year =          {2022},
}

\cleardoublepage
\phantomsection

\printindex

\end{document}